%% file: Zeta-Memoire_V4-Hubert_.tex
\documentclass[a4paper, 12pt,multicol]{amsart}

\usepackage{lscape}
\usepackage{amsmath}
\usepackage{mathtools}
\usepackage[all]{xy}
\usepackage[latin1]{inputenc}
\usepackage{varioref}
\usepackage{amsfonts}
\usepackage{amssymb}
\usepackage{bbm}
\usepackage{graphicx}
\usepackage{mathrsfs}
\usepackage{amsmath}
\usepackage{mathtools}
\usepackage[all]{xy}
\usepackage[latin1]{inputenc}
\usepackage{varioref}
\usepackage{amsfonts}
\usepackage{amssymb}
\usepackage{bbm}
\usepackage{graphicx}
\usepackage{mathrsfs}
\usepackage[hypertexnames=false,backref=page,pdftex,
 	pdfpagemode=UseNone,
 	breaklinks=true,
 	extension=pdf,
 	colorlinks=true,
 	linkcolor=blue,
 	citecolor=red,
 	urlcolor=blue,
 ]{hyperref}

\usepackage[hypertexnames=false,backref=page,pdftex,
 	pdfpagemode=UseNone,
 	breaklinks=true,
 	extension=pdf,
 	colorlinks=true,
 	linkcolor=blue,
 	citecolor=red,
 	urlcolor=blue,
 ]{hyperref}

\usepackage{amsmath, multicol,a4}
\usepackage[all]{xy}
 \usepackage{imakeidx}
 \makeindex
 
\newtheorem{theorem}{\bf Theorem}[subsection]

\newtheorem{prop}[theorem]{\bf Proposition}
\newtheorem{cor}[theorem]{\bf Corollary}
\newtheorem{lemme}[theorem]{\bf Lemma}
 
\newtheorem{definition}[theorem]{\bf Definition}

\theoremstyle{remark}
\newtheorem{example}[theorem]{\bf Exemple}
\theoremstyle{remark}
\newtheorem{rem}[theorem]{\bf Remark}
\theoremstyle{remark}
\newtheorem{notation}[theorem]{\bf Notation}

 \numberwithin{equation}{subsection}

\newcommand{\resp}{{\it resp.\;}}

\newcommand{\sorth}{\begin{picture}(15,10)(-3,-10)
\put(0,-9){\line(1,0){10}}
\put(4,-9){\line(0,1){7}}
\put(6,-9){\line(0,1){7}}
   \end{picture}}

\newcommand{\ad}{\operatorname{ad}}

\newcommand{\tr}{\operatorname{tr}}

\newcommand{\codim}{\operatorname{codim}}

\def\al{\alpha}
\def\be{\beta}

\def\De{\Delta}
\def\de{\delta}
\def\ep{\varepsilon}

\def\ga{\gamma}

\def\la{\lambda}
\def\om{\omega}
\def\Om{\Omega}

\def\si{\sigma}
\def\Si{\Sigma}

\def\te{\theta}
\def\i{\iota}

\def\plus{\mathop{\hbox{$\oplus$}}}

\def\go{\mathfrak}
\def\bb{\mathbb}
\def\C{\bb C}
\def\Z{\bb Z}
\def\N{\bb N}
\def\R{\bb R}
\def\Q{\bb Q}
\def\H{\bb H}
\def\cal{\mathcal}
\def\cqfd{\qed}
\def\plus{\mathop{\hbox{$\oplus$}}}

\def\me{\medskip}
\def\no{\noindent}

 \def\beq{\begin{equation}}
\def\eeq{\end{equation}}

\newenvironment{res}
              {\begin{equation}\begin{minipage}{0.85\textwidth}}
               {\end{minipage}\end{equation}}
\def\ber{\begin{res}}
\def\eer{\end{res}}

\def\qed{{\null\hfill\ \raise3pt\hbox{\framebox[0.1in]{}}\break\null}}

\def\bb{\mathbb}
\def\ad{{\rm{ad}}}

\def\det{{\rm{det}}}

\def\cal{\mathcal }
\def\cA{{\mathcal A}}
\def\cB{{\mathcal B}}

\def\cE{{\mathcal E}}
\def\cF{{\mathcal F}}

\def\cH{{\mathcal H}}

\def\cJ{{\mathcal J}}
\def\cK{{\mathcal K}}
\def\cL{{\mathcal L}}
\def\cM{{\mathcal M}}
\def\cN{{\mathcal N}}
\def\cO{{\mathcal O}}
\def\cP{{\mathcal P}}

\def\cR{{\mathcal R}}
\def\cS{{\mathcal S}}
\def\cT{{\mathcal T}}

\def\cV{{\mathcal V}}

\def\cZ{{\mathcal Z}}
\DeclareFontFamily{OT1}{rsfs}{}
\DeclareFontShape{OT1}{rsfs}{n}{it}{<-> rsfs10}{}
\DeclareMathAlphabet{\mathscr}{OT1}{rsfs}{n}{it}

\def\Cr{{\mathscr C}}

\def\Or{{\mathscr O}} 
\def\Sr{{\mathscr S}}

\def\C{\mathbb C}

\def\Q{\mathbb Q}
\def\R{\mathbb R}

\def\N{\mathbb N}
\def\P{\mathbb P}
\def\Q{\mathbb Q}
\def\R{\mathbb R}
\def\X{\mathbb X}
\def\Z{\mathbb Z}

\def\bC{\mathbf C}

\def\bP{\mathbf P}

\def\bU{\mathbf U}

 \def\adots{\mathinner{\mkern2mu\raise1pt\hbox{.}
\mkern3mu\raise4pt\hbox{.}\mkern1mu\raise7pt\hbox{.}}}

 \usepackage[left=2cm,right=2cm,top=2cm,bottom=2cm]{geometry}

\author{ Pascale Harinck}   

\author{Hubert Rubenthaler}

\address
{Pascale.~Harinck, CMLS, CNRS, \'Ecole polytechnique, Institut Polytechnique de Paris,   91128 Palaiseau Cedex, France.\\
E-mail: {\tt pascale.harinck@polytechnique.edu}}

\address{Hubert Rubenthaler\\ Institut de Recherche Math\'ematique Avanc\'ee\\
Universit\'e de Strasbourg et CNRS\\
7 rue Ren\'e Descartes\\
67084 Strasbourg Cedex\\ France\\
E-mail: {\tt rubenth@math.unistra.fr}}

\begin{document}

 \parindent=0pt
\title[Classification, Structure, and associated Local Zeta Functions]{Classification, Structure, and associated Local Zeta Functions for a class of p-adic symmetric spaces }



  \maketitle
  
  \vskip 30pt  
 \hskip 300pt {A la M\'emoire d'Anne Fricker}
 \centerline{\bf Part I: Structure and Orbits}
  
 \vskip 100pt 
 
 {\bf Abstract:}\,

 {\bf  Part I}:\,\, {\it   Let  $F$ be  a p-adic field of characteristic $0$.  Let  $\widetilde{\go g}$ be  a reductive Lie algebra over $F$ which is endowed with a short $\Z$-grading:  $\widetilde{\go g}=\widetilde{\go g}_{-1}\oplus \widetilde{\go g}_{0}\oplus \widetilde{\go g}_{1}$. It is known that the representation $(\widetilde{\go g}_{0},\widetilde{\go g}_{1})$ is always prehomogeneous. Under some additional conditions we classify these objects using weighted Satake-Tits diagrams.
  Moreover we  study  the orbits of $\widetilde {G}_{0}$ in $\widetilde{\go g}_{1}$, where $\widetilde {G}_{0}$ is an algebraic group defined over $F$,  whose Lie algebra is $\widetilde{\go g}_{0}$. It turns out that the open orbits are symmetric spaces. We then investigate the $P$-orbits in $\widetilde{\go g}_{1}$, where $P$ is a minimal $\sigma$-split parabolic  subgroup of $\widetilde {G}_{0}$ ($\sigma$ being the involution defining the  symmetric spaces).}

 {\bf  Part II}:\,\, {\it In the second part we define and study the zeta functions associated to the minimal spherical principal  series of representations (of $\widetilde {G}_{0}$) for  the reductive $p$-adic symmetric 
 spaces described  in Part I.    We prove that these zeta functions satisfy a functional equation which is given explicitly (see Theorem \ref{thm-principal} and Theorem \ref{thm-mainbis}).
 Moreover, for a subclass of these spaces, we define $L$-functions and $\ep$-factors associated to the representations.}
 
This paper is essentially the synthesis of two preceding papers by the authors (arXiv:2003.05764 and arXiv:2407.06667)
 with some additions and some corrections.
 
\vfill\eject
 {\small \def\contentsname{Table of Contents}
\tableofcontents}

\medskip\medskip\medskip\medskip\medskip\medskip
\newpage

 \vskip 50pt

\section*{Introduction}

 \hskip 10pt The  purpose of this paper is to   define  local Zeta functions, for a class of reductive p-adic symmetric spaces  and attached to the   representations of the $\sigma$-minimal series,  and to prove their explicit functional equation. This goal will be reached in the second part.
 
  \vskip 10pt
 \hskip 10pt This paper follows the same lines as the corresponding paper which dealt with the real case by Nicole Bopp and the second author (\cite{B-R}). But of course, due to the big   differences between  the structures of the base fields, the proofs (as well as the definition of the group which acts), are often rather different.

 \vskip 30pt

\centerline{\bf Part I}
\vskip 30pt

In the  first part, we are only concerned with classification and structure theory for these symmetric spaces.

\hskip 10pt
\hskip 10pt Let $F$ be a p-adic field of characteristic $0$ whose residue class field has characteristic $\neq 2$. The main object under consideration is a reductive  Lie algebra $\widetilde{ {\go g}}$ endowed with a short $\Z$-grading 
$$\widetilde{ {\mathfrak g}}=V^-\oplus {\mathfrak g}\oplus V^+\quad (V^+\not = \{0\}).$$
 \vskip 10pt
 
\hskip 10pt This means that $[\go{g}, \go{g}]\subset \go{g}$, $[\go{g},V^\pm]\subset V^\pm$, $[V^+,V^+]=[V^-,V^-]=\{0\}$. We also suppose that the corresponding representation of $\go{g}$ on $V^+$ is absolutely irreducible.  Then $(\go{g},V^+)$ is an infinitesimal prehomogeneous vector space. It is well known that such a grading is defined by a grading element $H_{0}$. We normalize it in such a way that $\ad(H_{0})$ has eigenvalues $-2$,$0$, $2$ on $V^-$, $\go{g}$, $V^+$, respectively.

We will call such an object a {\it graded Lie algebra}.

  \vskip 10pt

\hskip 10pt We introduce a natural algebraic subgroup $G$ of the group of automorphisms of $\widetilde{\go{g}}$ whose Lie algebra is $\go{g}$ (which was first used by Iris Muller (\cite {Mu97})). This group is defined  at the beginning of section \ref{sub-section-generic-strongly-orth} as the centralizer of $H_{0}$ in the group of automorphisms of  $\widetilde{\go{g}}$ which become elementary over a field extension.  Then $(G, V^+)$ is a prehomogeneous vector space.   
 \vskip 10pt
\hskip 10pt Although a part of the structure theory can be carried out with  no further assumption (section \ref{sub-section-definitions} up to section \ref{sub-section-generic-strongly-orth}), we need to introduce the so-called {\it regularity condition}, and we will essentially   consider only regular graded Lie algebras in this paper. 
A graded Lie algebra is said to be regular (Definition \ref{def-regulier}) if the grading element $H_{0}$   is the semi-simple element of an $\go{sl}_{2}$-triple. This condition is also equivalent to the existence of a non-trivial relative invariant polynomial $\Delta_{0}$ of the prehomogeneous space $(G, V^+)$, and also, as we will see in section \ref{section-G/H}, to the fact that the various open $G$-orbits in $V^+$ are symmetric spaces. 
 \vskip 10pt
\hskip 10pt These    open orbits of $G$ in $V^+$ (and in $V^-$) are precisely the symmetric spaces we are interested in.

 \vskip 10pt
 \hskip 10pt  One  can always suppose that  $ {\widetilde{\go{g}}}$ is semi-simple. Then the  assumptions on the grading imply that $\overline{\go{g}}+ \overline{V^+}$  (where the overline stands for scalar extension to an algebraic closure) is a maximal parabolic subalgebra of $\overline{\widetilde{\go{g}}}$, and hence it is defined by the single root which is removed from the root basis of  $\overline{\widetilde{\go{g}}}$ to obtain the root system of $\go{g}$. It can be shown that this single root is a ``white'' root in the Satake-Tits diagram of $\widetilde{\go{g}}$. Therefore the gradings we are interested in are in one to one correspondence with ``weighted'' Satake-Tits diagrams where one ``white'' root is circled. The classification is done in section  \ref{section-classification} and the list of the allowed diagrams is given in Table 1 in section \ref{sec:section-Table}.
 \vskip 10pt
 \hskip 10pt  A key tool in the orbital descripion of $(G,V^+)$ is a kind of {\it principal diagonal} 
 $$\widetilde{\go{g}}^{\lambda_{0}}\oplus\ldots\oplus \widetilde{\go{g}}^{\lambda_{k}}\subset V^+$$
 where $\lambda_{0},\lambda_{1},\ldots,\lambda_{k}$ is a maximal subset of strongly orthogonal roots living in $V^+$. This sequence starts with the root $\lambda_{0}$, the root defining the above mentioned parabolic, and  is obtained by an induction process which we call ``descent'' (see sections \ref{section-descente1} and \ref{sub-section-descente2}). 
  \vskip 10pt
 \hskip 10pt
 A first step in the classification of the orbits is to prove that any $G$-orbits meets this principal diagonal. This is done in Theorem \ref{th-V+caplambda}. Another step is the study of the so-called rank 1 case in section \ref{subsectionk=0} (Theorem \ref{th-k=0}).    

 \vskip 10pt
 \hskip 10pt

 Finally, in order to classify the orbits, we need to distinguish  three Types (see Definition \ref{def-type}) and the full classification of the orbits is obtained in Theorem   \ref{thm-orbites-e04} for Type $I$, Theorem \ref{thm-orbites-e1} and Theorem \ref{thm-orbites-e2} for Type $II$, and Theorem  \ref{th-d=3} for Type $III$. \\
 \vskip 10pt
 \hskip 10pt Section \ref{section-G/H}  is devoted to the study of the associated symmetric spaces.  Let $\Omega_{1},\ldots,\Omega_{r_0}$ be the open orbits in $V^+$. As we said before these open orbits are symmetric spaces. More precisely this means that if we choose elements $I_{j}^+\in \Omega_{j}$, and if $H_{j}={ Z}_{G}(I_{j}^+)$  then for all $j$  there exists an involution $\sigma_{j}$ of $\go{g}$ (in fact the restriction of an  involution of $\widetilde{\go{g}}$ which stabilizes $\go{g}$) such that $H_{j}$ is an open subgroup of the fixed point group $G^{\sigma_{j}}$. Therefore $\Omega_{j}\simeq G/H_{j}$ can be viewed as a symmetric space. A striking fact is that all these $G$-symmetric spaces have the same minimal $\sigma_{j}$-split parabolic subgroup $P$  (i.e. $\sigma_j(P)$ is the opposite parabolic of $P$  and $P$ is minimal for this property). Moreover $(P,V^+)$ is again a prehomogeneous space. We define and study a family of polynomials $\Delta_{0},\ldots,\Delta_{k}$ which are the fundamental relative invariants of $(P,V^+)$. We also determine the open orbits of $(P, V^+)$  in terms of the values of the $\Delta_{j}$ (Theorem \ref{th-Porbites}). Finally we introduce an involution $\gamma$ of $\widetilde{\go{g}}$ which exchanges $V^+$ and $V^-$ and which allows to define the fundamental relative invariants of $(P,V^-)$ and to determine its open orbits.

  \vskip 10pt
 \hskip 10pt We must also mention that the study of graded algebras over a $p$-adic field was initiated by Iris Muller in a series of paper (\cite{Mu98}, \cite{Mu97}, \cite{Mu02}, \cite{Mu}),  in a more general context (the so-called quasi-commutative case), but her results seem to us less precise and sometimes weaker than  ours. Moreover she never considered the symmetric space   aspects of these spaces.
 
  \vskip 10pt
 \hskip 10pt Finally it should be noticed that, via the Kantor-Koecher-Tits construction (which is still valid over a p-adic field), there is a bijection between the regular graded Lie algebras we consider here and absolutely simple Jordan algebra structures on $V^+$ (see \cite{FK}, and the references there). And, probably,   the group $G$ which is used in this paper is very closed to the  ``structure group'' for the $p$-adic Jordan algebra $V^+$.

 \vskip 30pt

\centerline{\bf Part II}
\vskip 30pt

In this second part  we define and study the local zeta functions associated to  representations of the minimal spherical principal series for the class of symmetric spaces introduced  in the first part.  We obtain explicit functional equation for these zeta functions.

  It is worth noticing that Wen-Wei Li was interested in similar questions for prehomogeneous spaces $(G,X)$ where $G$ is a reductive group over a local field of characteristic zero whose open orbits in $X$ are spherical varieties (see section 6.3 in \cite{Li}).\\
 Further developments of our second part could  certainly  be related to the work of Wen-Wei Li. 
 \vskip 5pt
Let us first  describe, formally,   the  general setting where this paper takes place. Let $G$ be a connected reductive algebraic group over a p-adic local field $F$ of characteristic zero. Suppose we are given an irreducible regular prehomogeneous vector space $(G,V)$ defined over $F$ (see \cite{Sa} for example) and denote by $\Delta_{0}$ the fundamental relative invariant. Then the dual representation $(G,V^*)$ is still a prehomogeneous vector space of the same type. We denote by $\Delta_{0}^*$ its fundamental relative invariant.

Suppose,  for sake of simplicity, only in this introduction, that $(G,V)$ has only one open orbit $\Omega$. Let $H$ be the isotropy subgroup of an element $I\in \Omega$. Moreover, suppose that $G/H\simeq \Omega$ is a symmetric space corresponding to an involution $\sigma$ of $G$.
 
 Then the dual space $(G,V^*)$ has the same property. More precisely the open orbit $\Omega^*\subset V^*$ contains the element $I^*= \frac{d \Delta_{0}}{\Delta_{0}}(I)$ (see \cite{Sa}) and the isotropy subgroup of $I^*$ is still $H$.

Consider now a minimal $\sigma$-split parabolic subgroup $P$ of $G$ (this means that $P$ and $\si(P)$ are opposite and that  $P$ is minimal for this property) and denote by $L=P\cap \si(P)$ its $\si$-stable Levi subgroup.  Again for sake of simplicity,  we suppose that $P$ also has a unique open orbit in $V$.  It is well known (see \cite{B-D}) that if $\chi$ is a character of $L$ which is trivial on $L\cap H$, the induced representation $\pi_{\chi}= \text{Ind}_{P}^{G}(\chi)$ is generically $H$- distinguished (or $H$-spherical).  This means that for ``almost all''  characters $\chi$ the dual space $I_{\chi}^*$ of the space $I_{\chi}$ of $\pi_{\chi}$ contains a nonzero  $H$-invariant vector $\xi$. Therefore if $w\in I_{\chi}$ the coefficient $\langle \pi_{\chi}^*(g)\xi,w\rangle $ is right $H$-invariant and hence can be considered as a function on $\Omega$ (or $\Omega^*$). The minimal spherical series for $G/H$ is the set of the representations $\pi_{\chi}$. 

  Let $\cS(V)$ (respectively $\cS(V^*)$) be the spaces of locally constant functions with compact support on $V$ (respectively $V^*$). Let also $d^{*}X$ (resp. $d^*Y$) be the $G$-invariant measure on $V$ (resp. $V^*$). 
  For $\Phi\in \cS(V)$ and $\Psi\in \cS(V^*)$, and $s\in \C$, let us define (formally!) the following local zeta functions:
  
  $$Z(\Phi,s,\xi,w)=\int_{\Omega}\Phi(X) |\Delta_{0}(X)|^s\langle \pi_{\chi}^*(X)\xi,w\rangle d^*X,$$
  
  $$Z^*(\Psi,s,\xi,w)=\int_{\Omega^*}\Psi(Y) |\Delta_{0}^*(Y)|^s\langle \pi_{\chi}^*(Y)\xi,w\rangle d^*Y.$$
  
  It is expected    that these zeta functions can be correctly defined (via absolute convergence and meromorphic continuation)   and that they should verify a functional of the type:
  
  $$Z^*(\cF(\Phi), m-s, \xi, w)= \gamma(s, \chi)Z(\Phi,s, \xi,w), $$
  where $\cF:\cS(V)\longrightarrow \cS(V^*)$ is the Fourier transform, $m$ is a suitable ``shift'', and $\gamma(s, \chi)$ a meromorphic function.
  
  The aim of this part is to perform this program, including an explicit form of $\gamma(s, \chi)$ in terms of local Tate factors and Weil constants related to the Fourier transform of some quadratic characters. This is done  even in the case where $G$ has several open orbits, each of it being a symmetric space, for a class of $p$-adic prehomogeneous vector spaces (which is essentially the class described in the first part). See Theorem \ref{thm-principal} and \ref{thm-mainbis} below.
 
 Our results contain, as a particular case, the case of $GL_{n}(F)\times GL_{n}(F)$ acting on the space $M_{n}(F)$ of $n$ by $n$ matrices, which gives rise to the Godement-Jacquet zeta function for the principal minimal series for $GL_{n}(F)$.
 
 Moreover, in the case where $G$ and $P$ have a unique open orbit in $V^+$,  the space of $H$-invariant linear forms on $I_\chi$ is $1$-dimensional. In that case we define and prove the existence of $L$-functions which describe the poles of the zeta functions $Z(\Phi,s,\xi,w)$ and $Z^*(\Psi,s,\xi,w)$ for all $\Phi\in\cS(V^+)$, $\Psi\in \cS(V^-)$ and $w\in I_\chi$. We define  also  the corresponding $\ep$-factors. This  generalizes the  results of Godement-Jacquet  for the principal minimal series for $GL_{n}(F)$. See Proposition \ref{prop-L(z,pi)} and Theorem \ref{thm-analogueGJ} below.\me

 Let us now describe   the content of the second part.
 
 $\bullet$ In section \ref{sec:Preliminaries} we summarize briefly the main properies of the graded Lie algebras which were obtained in Part I.  
    
     For technical reasons we also introduce a subgroup $\tilde P$ of the minimal $\sigma$-split parabolic $P$ which will play an important role. The representations $(P,V^{\pm})$ and $(\tilde P, V^{\pm})$ are prehomogeneous and the open orbits are described.
 \me

 $\bullet  $ Section \ref{sec:moyennes-Weil} is devoted to fix some  imported tools. First of all we define the Fourier transform 
$\cF: \cS(V^{+})\longrightarrow \cS(V^{-})$ and prove that there exists a unique pair of measures on $V^{+}\times V^{-}$ which are dual for $\cF $ and verify some  additional compatibility condition. 
 We also define and study the  so-called {\it mean functions} (see definition \ref{def-moyennes}) which correspond, roughly speaking, to integration on an $\cN$ orbit where $\cN$ is the nilradical of a $\sigma$-split parabolic, non necessarily minimal. We also normalize the measures on various subspaces of $\go g$. These normalizations are necessary to compute precisely the factors in the final functional equations. In this section these normalizations are   also needed to compute precisely the   Weil constants corresponding to some quadratic forms occurring in the classification of the orbits (see Proposition \ref{prop-calcul-gamma}).
  \me

 $\bullet  $ As said before the representations $(\tilde P, V^{\pm})$ are prehomogeneous and regular. But there are several fundamental relative invariants named $\Delta_{0},\Delta_{1},\ldots,\Delta_{k}$. In this situation one can classically define   zeta functions associated to this prehomogeneous space. Roughly speaking they are of the form $\cK(f,s)=\int_{V^+}f(X)|\Delta(X)|^{s}dX$ where $f\in \cS(V^+)$, $s=(s_{0},\ldots,s_{k})\in \C^{k+1}$ and where $|\Delta(X)|^{s}=|\Delta_{0}(X)|^{s_{0}}\ldots|\Delta_{k}(X)|^{s_{k}}$. See Definition \ref{def-zetaP}. In such a situation it is known from the work of F. Sato (\cite{Sa}) that there exists a functional equation if the prehomogeneous spaces satisfies a certain condition  ($A2'$).   In section \ref{sec:zetaPH} we prove that the spaces $(\tilde P, V^{\pm})$ satisfy this condition   (Theorem \ref{ConditionA2}) and compute explicitly the constants in the functional equation (Theorem \ref{EF-(P,V^+)}).
  \me

 $\bullet$  Section \ref{sec:zeta} contains the main results.  The symmetric spaces  $G/H_{i}$ ($i=1,\ldots, r_0$) we consider are the open $G$-orbits in $V^{+}$. Let $\sigma_{i}$ be the corresponding involution of $G$. The key point here is the fact that the parabolic subgroup $P$ defined previously is minimal $\sigma_{i}$-split for all $\sigma_{i}$.  Therefore these symmetric spaces have the same minimal spherical series. 
 
 The zeta functions associated to such representations are defined in \ref{def-pi(X)-fonction-zeta} as integrals  depending on several complex parameters $\mu_{j}$ and $z$.  We prove that they are rational functions in the variables $q^{\pm \mu_{j}}, q^{\pm z }$ ($q$ is the  cardinal of the residue field of $F$), that they satisfy a functional equation which is explicitly computed (see Theorem \ref{thm-principal}). The main ingredient of the proofs is the work of P. Blanc and P. Delorme (\cite {B-D})  and  the close relation, via the Poisson kernel,  between these zeta functions associated to representations and the zeta functions of the prehomogeneous space $(P,V^+)$ studied in section 3.
We also give a second version of the main Theorem, which modulo the introduction of an  operator valued ``gamma'' factor, has a very simple form (see Theorem \ref{thm-mainbis}).
Finally,  in the case where  $G$ and $P$ have a unique open orbit in $V^+$,  we   define and prove the existence of $L$-functions and  $\ep$-factors associated to these representations (Proposition \ref{prop-L(z,pi)} and Theorem \ref{thm-analogueGJ}).
\me
\medskip

${\bf Acknowledgments}$

During the writing of this paper we have greatly benefitted from help trough discussions or  mails from several colleagues. We would like to thank  Giuseppe Ancona, Henri Carayol, Jan Denef, Rutger Noot, Guy Rousseau, Torsten Schoeneberg and Marcus Slupinski.

   \newpage
   
  {\,}
 \vskip 200pt

  \part{} { \centering{\Huge {\bf Structure and Orbits}}}
 \newpage
 
 \vskip 20pt
 \section{3-graded Lie algebras}\label{section-PH}
 \vskip 20pt
 \subsection{A class of graded algebras}\label{sub-section-definitions}\hfill
 
 In this paper the ground field $F$ is a $p$-adic field of characteristic $0$, i.e. a finite extension of $\Q_{p}$. Moreover we will always suppose that the residue class field has characteristic $\neq 2$ (non dyadic case). We will denote by $\overline{F}$ an algebraic closure of $F$. In the sequel, if $U$ is a $F$-vector space, we will set 
 $$\overline U=U\otimes_{F}\overline{F}$$.

 \begin{definition}\label{def-alg-grad-1}
 
 Throughout this paper a reductive  Lie algebra $\widetilde{{\mathfrak g}}$ over $F$  \index{g@$\widetilde{{\mathfrak g}}$} satisfying the following two  hypothesis will be called {\it a {graded Lie algebra}\index{graded Lie algebra}}:
 
 \begin{itemize}
 \item[$(\bf H_1)$]\label{H1}{\it  There exists an  element $H_0\in \widetilde{ {\mathfrak g}}$ such that 
$\ad H_0$ defines a   ${\mathbb  Z}$-grading  of the form }
$$\widetilde{ {\mathfrak g}}=V^-\oplus {\go{g}}\oplus V^+\quad (V^+\not = \{0\})\ ,$$
{\it where } $[H_0,X]=\begin{cases}
 0&\hbox{  for }X\in {\mathfrak g}\ ;\\
         2X&\hbox{  for }X\in V^+\ ;\\
        -2X&\hbox{  for }X\in V^-\ .
\end{cases}$\index{g@{$\go{g}$}} \index{Vplus@$V^+$}\index{Vmoins@$V^-$}\index{H0@$H_{0}$}

 ({\it Therefore, in fact,  } $H_{0}\in \go{g}$)
\item[$(\bf H_2)$]\label{H2}{\it  The $($bracket$)$ representation  of ${\mathfrak g}$ on  $\overline {V^+}$ is irreducible. $($In other words, the representation $({\mathfrak g},V^+)$ is absolutely irreducible$)$}
\end{itemize}

\end{definition}

\vskip 5pt 

\noindent The following relations are trivial consequences from $(\bf H_1)$ :
$$[{\mathfrak g},V^+]\subset V^+\ ;\ [{\mathfrak g},V^-]\subset V^-\ ;\ [{\mathfrak g},{\mathfrak g}]\subset {\mathfrak g}\
;\ [V^+,V^-]\subset {\mathfrak g}\ ;\ [V^+,V^+]=[V^-,V^-]=\{0\} .$$

\vskip 20pt

\subsection{The restricted root system}\label{subsec:racines}\hfill

There exists a maximal split abelian Lie subalgebra  $\go{a}$\index{a@$\go{a}$} of $\go{g}$ containing $H_{0}$. Then $\go{a}$ is also maximal split abelian in $\widetilde{\go g}$.

Denote by  $\widetilde \Sigma$ \index{Sigmatilde@$\widetilde \Sigma$}the roots of $(\widetilde{\go g},\go{a})$ and by  $\Sigma$ \index{Sigma@$\Sigma$}the roots of $(\go{g},\go{a})$. (These are effectively root systems: see \cite{Seligman}, p.10).

 \vskip 5pt

 Let $H_\al$ be the coroot of  $\al\in\tilde{\Si}$, and for $\al,\be\in \tilde{\Si}$, we set as usual $n(\al,\be)=\al(H_\be)$. For $\mu\in\go a^*$, we denote by $\tilde{\go g}^\mu$ the subspace of weight $\mu$ of $\tilde{\go g}$. 
 \vskip 5pt
  \me
  
  We will denote by $\widetilde{W}$ and  $W$ \index{Wtilde@$\widetilde{W}$}\index{W@$W$} the Weyl groups of $\widetilde{\Sigma}$   and $\Sigma$, respectively. $W$ is the subgroup of  $\widetilde{W}$ generated by the reflections with respect to the roots in $\Sigma$. In particular $H_{0}$ is fixed by each element of  $W$.

\vskip 5pt
Let $\text{Aut}_{e}(\go{g})$\index{Aut@$\text{Aut}_{e}(\go{g})$} denote the group of elementary automorphisms of $\go{g}$ (\cite{Bou2} VII, \S3, $n^\circ 1$).

Two maximal split abelian subalgebras of $\go{g}$ are conjugated by  $\text{Aut}_{e}(\go{g})$ (\cite{Seligman}, Theorem 2, page 27, or  \cite{Schoeneberg}, Theorem 3.1.16 p. 27)

\vskip 10pt

\begin{theorem}\label{thbasepi} {\rm (Cf. Th.  1.2 p.10 of  \cite{B-R})}\- 

{\rm (1)} There exists a system of simple roots   $\widetilde{ \Pi }$\label{Pitilde}\index{Pitilde@$\widetilde{ \Pi }$} in $\widetilde{ \Sigma }$
such that
$${\nu}\in \widetilde{ \Pi }\Longrightarrow  \nu (H_0)=0 \hbox{  or }2\ .$$

{\rm (2)} There exists an unique root \label{lambda0} $\lambda _0\in \widetilde{ \Pi }$ \index{lambda0@$\lambda _0$}such that $\lambda _0(H_0)=2$.

{\rm (3)} If the decomposition of a  positive root $\lambda\in \widetilde{ \Sigma }$ in the basis
$\widetilde{
\Pi }$ is given by
$$\lambda =m_0\lambda _0+\sum_{\nu \in \widetilde{ \Pi } \backslash \{\lambda _0\}} m_\nu \nu \
,\ m_0\in {\mathbb Z}^+, m_\nu \in {\mathbb Z}^+$$
then $m_0=0$ or $m_0=1$. Moreover $\lambda $ belongs to $\Sigma $ if and only if $m_0=0$.

\end{theorem}

\begin{proof}

Let $S$ be the subset of $\widetilde{ \Sigma }$ given by
$$S=\{\lambda \in \widetilde{ \Sigma }\mid \lambda (H_0)=0\ or\   2\}\ .$$
It is easily seen from $(\bf H_1)$  that $S$ is a parabolic subset of $\widetilde{ \Sigma }$,
i.e.
${\mathfrak g}\oplus V^+$ is a parabolic subalgebra of $\widetilde{ {\mathfrak g}}$. It is well known 
(see \cite{Bou1} chap. 6 Prop. 20) that there exists an order on $\widetilde{ \Sigma }$ such that,
 if  a   root 
$\lambda\in \widetilde{ \Sigma } $ is positive for this order, then $\widetilde{ {\mathfrak
g}}^\lambda
$ is a subspace of    ${\mathfrak g}\oplus V^+$. If $\widetilde{ \Pi }$ denotes the set of simple roots
of
$\widetilde{ \Sigma }$  corresponding  to this order, then $\widetilde{ \Pi }$ satisfies (1).
\vskip 5pt 

There exists at least  one root $\lambda _0\in \widetilde{ \Pi }$ such that $\lambda _0(H_0)=2$
because  $V^+\not =\{ 0\}$. Moreover the commutativity of  
$  V^+$ which is the nilradical of the parabolic algebra $\go
g\oplus V^+$ implies  (3).
\vskip 5pt 
 Let us suppose  that there exists in $\widetilde{ \Pi }$ a root $\lambda _1\not = \lambda
_0$ such that $\lambda _1(H_0)=2$. Let $V_0$ be  the sum of the root spaces $\widetilde{ {\mathfrak
g}}^\lambda $ for the roots $\lambda $ of the form 
$$ \lambda =\lambda _0+\sum_{\nu \in \widetilde{ \Pi },\nu (H_0)=0}m_\nu \nu
\quad (m_\nu \in {\mathbb Z}^+)\ .$$
Since $V_0$ does not contain  $ \widetilde {\go g}^{\lambda_1}$, it
is  a non-trivial  subspace of $V^+$ which is invariant under the action of ${\mathfrak g}$.
This gives a contradiction with $(\bf H_2)$ and the uniqueness of $\lambda _0$ such that $\lambda
_0(H_0)=2$ follows, hence  (2) is proved.

 \end{proof}
Let us fix once and for all such a set of simple roots $\widetilde{ \Pi }$ of $\widetilde{ \Sigma }$. Then 
$$\widetilde{ \Pi }=\Pi \cup \{\lambda _0\}\ ,$$ \index{Pi@$\Pi$}
where $\Pi=\{\nu \in \widetilde{ \Pi }\,|\, \nu(H_{0})=0\}$  is a set of simple roots of  $ \Sigma $. We denote by  $\widetilde{\Sigma}^+$ (resp. $\Sigma^+$) the set of positive roots of $\widetilde{\Sigma}$ (resp. $\Sigma$) for the order defined by $\widetilde{ \Pi }$ (resp. $\Pi$).

  \vskip 5pt

Then we have the following
characterization of  $\lambda _0\ $:
\vskip 5pt 

\begin{cor}\label{corlambda0} The root $\lambda _0$ is the unique  root in $\widetilde{
\Sigma }$ such that 
\begin{align*} 
\bullet \hskip 15pt &\lambda _0(H_0)=2\ ;\cr
   \bullet \hskip 15pt & \lambda \in \Sigma ^+\Longrightarrow \lambda _0-\lambda \notin
\widetilde{ \Sigma }\ . 
\end{align*}
 \end{cor} 

 \begin{proof}  It is clear that $\lambda_{0}$ satisfies the two properties. Let $\mu\in \widetilde{
\Sigma }$ a root satisfying the same properties. Then the first property implies  that $\mu\in \widetilde{\Sigma }^+$. Suppose that $\mu\neq\lambda_{0}$. Then $\mu=\lambda_{0}+\Sigma_{\nu \in A} m_{\nu}\nu$ with $A\subset \Pi $, $A\neq \emptyset$,  and  $m_{\nu}\neq0$. This implies that $\mu=\nu_{1}+\dots+\nu_{k}+\lambda_{0}+\nu_{k+1}+\dots+\nu_{p}$, where each partial sum is a root. Then either $\mu-\nu_{p}\in \widetilde{
\Sigma }$, or (if $\nu_{k+1}=\dots=\nu_{p}=0$) one has  $\mu - (\nu_{1}+\dots+\nu_{k})=\lambda_{0}\in \widetilde{
\Sigma }$. In both cases  the second property would not be verified.

 \end{proof}
 
 \begin{rem}\label{rem-precisions}
 
 Let  $\go{m}$ \index{mgothique@$\go{m}$} be the centralizer of  $\go{a}$ in $\go{g}$ (which is also the centralizer of $\go{a}$ in $\widetilde {\go g}$). We have then the following decompositions:
 $$\widetilde{\go{g}}=\go{m}\oplus \sum_{\lambda\in \widetilde{\Sigma}}\widetilde{\go{g}}^\lambda,\hskip 10pt\go{g}=\go{m}\oplus \sum_{\lambda\in  {\Sigma}} {\go{g}}^\lambda, \hskip 10pt\ V^+=\sum_{\lambda\in \widetilde{\Sigma}^+\setminus \Sigma^+}\widetilde{\go{g}}^\lambda.$$
 The algebra  $[\go{m},\go{m}]$ is {\it anisotropic} (i.e., its unique  split abelian subalgebra  is $\{0\}$), and is called the    {\it anisotropic kernel} \index{anisotropic kernel} of  $\widetilde{\go{g}}$.
 \end{rem}

 \vskip 20pt
 \subsection{Extension to the algebraic closure}\label{section1-extension}\hfill
 \vskip 10pt
 Let us fix an algebraic closure  $\overline F$ \index{Fbarre@$\overline F$} of $F$. Remember that for each vector space  $U$ over  $F$ we  note $\overline U $ the vector space  obtained by extension of the scalars:
$$\overline U =U\otimes _{F} \overline F.$$
  We have then the decomposition 
 $$\overline{\widetilde {\go g}}=\overline{V^-}\oplus \overline{\go g}\oplus \overline{V^+}$$\index{ggothiquetilde@$\overline{\widetilde {\go g}}$}\index{Vbarre@$\overline{V^\pm}$}\index{ggothiquebarre@$\overline{\go g}$}
 and according to   $(\bf{H_{2}})$, the representation $(\overline{\go g},\overline{V^+})$ is irreducible.
 Let $\go j$ \index{jgothique@$\go j$} be a  Cartan subalgebra of  $ \widetilde {\go g}$ which contains  $\go a$. Then  $\go j\subset \go{g}$ and ${\go j}$ is also a Cartan subalgebra of  ${\go g}$. This implies that  $\overline {\go j}$ \index{jgothiquebarre@$\overline {\go j}$} is a Cartan subalgebra of $\overline{\widetilde {\go g}}$ (and also of $\overline{\go g}$) (see \cite{Bou2}, chap.VII, \S 2, Prop.3)

 Let  $\widetilde{\cal R}$ \index{Rtilde@$\widetilde{\cal R}$} (resp. ${\cal R}$) \index{R@${\cal R} $ (Part I)} be the roots of the pair $(\overline{\widetilde {\go g}}, \overline {\go j})$ (resp.   $(\overline{\go g}, \overline {\go j})$) 	and let  $\overline{\widetilde {\go g}}^{^{\alpha}}$ (resp. ${\overline{\go g}}^{^{\alpha}}$) be the corresponding root spaces.

 Let $X\in \overline{\widetilde {\go g}}^{^{\alpha}}$. Let us write $X=\sum a_{i}X_{i}$ where  $a_{i}\in \overline{F}^{^*}$ and where the  elements  $X_{i}\in \widetilde {\go g}$ are $\overline{F}$-free eigenvectors of $\go{a}$. Then for $H\in \go{a}$, we have  $[H,X]=\alpha(H)X=\sum a_{i}[H,X_{i}]=\sum a_{i}\alpha(H)X_{i}$. Hence $[H,X_{i}]=\alpha(H)X_{i}$. If $X_{i}=\sum_{\lambda\in \Sigma \cup \{0\}}X_{\lambda}$, we obtain $[H,X_{i}]=\sum \lambda(H)X_{\lambda}=\alpha(H)\sum X_{\lambda}$. Therefore  $\alpha(H)\in F$, in other words the  restrictions to $\go{a}$ of roots belonging to $\widetilde{\cal R}$,  take values in  $F$. 
 
  We denote by $\rho: {\overline {\go j}}^*\longrightarrow {\go a}^*$ \index{rho@$\rho$} the restriction morphism.   
  One sees easily that  $\rho(\widetilde{\cal R})= \widetilde{\Sigma}\cup\{0\}$ and that $\rho({\cal R})=\Sigma\cup \{0\}$.
 
 Let us recall the following well known result:
 
 \begin{lemme}\label{lemme-espacesradiciels-cloture}\hfill
 
 Let  $\lambda\in \widetilde{\Sigma}$. Let $S_{\lambda}=\{\alpha\in \widetilde{\cal R},\,\rho(\alpha)=\lambda\}$. Then we have:
 $$\overline{\widetilde{\go{g}}^\lambda}=\sum_{\alpha\in S_{\lambda}}\overline{\widetilde {\go g}}^{^{\alpha}}$$
 \end{lemme}
 
 \begin{proof} Let  $\alpha\in S_{\lambda}$ and $X\in \overline{\widetilde {\go g}}^{^{\alpha}}$. The element $X$ can be written
$X=\sum_{i=1}^{n} a_{i}X_{i}$ where $a_{i}\in \overline{F}^*$ and where the elements  $X_{i}\in \widetilde{\go{g}}$ are free over  $\overline{F}$ and are eigenvectors $\go{a}$ ($\go{a}$ is split). Then for $H\in \go{a}$ we have:

$[H,X]=\sum_{i=1}^{n}a_{i}[H,X_{i}]= \sum_{i=1}^{n}a_{i}\gamma_{i}(H)X_{i}=\alpha(H)X=\lambda(H)X= \sum_{i=1}^{n}   a_{i}\lambda(H)X_{i}$. Therefore for each $H\in \go{a}$,  $\gamma_{i}(H)=\lambda(H)$.This implies the inclusion $\sum_{\alpha\in S_{\lambda}}\overline{\widetilde {\go g}}^{^{\alpha}}\subset \overline{\widetilde{\go{g}}^\lambda}$. Conversely let $X\in \overline{\widetilde{\go{g}}^\lambda}$ whose root space decomposition in $\overline{\widetilde{\go{g}}}$ is given by  $X= \sum _{\beta\in \widetilde{\cal R}\cup\{0\}}X_{\beta}$. For $H\in \go{a}$ we have: 
$$[H,X]=\lambda(H)X=\sum \lambda(H)X_{\beta}=\sum \beta(H)X_{\beta},$$
and hence $\rho(\beta)=\lambda$.
 \end{proof}

\vskip 3pt
 
Set $\widetilde P=\{\alpha\in \widetilde{\cal R},\, \rho(\alpha)\in \widetilde{\Sigma}^+\cup \{0\}\}$. One shows easily that  $\widetilde P$ is a parabolic subset $\widetilde{\cal R}$. Therefore there is an order on  $\widetilde{\cal R}$ such that if $\widetilde{\cal R}^+$ is the set of positive roots for this order  one has  $\widetilde{\cal R}^+\subset \widetilde P$ (\cite{Bou1}, chap. VI, \S1, $n^\circ 7$, Prop. 20). Then:

 $$\rho(\widetilde{\cal R}^+)= \widetilde{\Sigma}^+\cup \{0\}$$
 and hence
 $$\rho({\cal R}^+)= {\Sigma}^+\cup\{0\}.$$
 We will denote by  $\widetilde{\Psi}$ the set of simple roots of  $\widetilde{\cal R}$ corresponding to  $\widetilde{\cal R}^+$.
 
 \begin{prop}\label{prop.alpha0}\hfill
 
  There is a unique simple root  $\alpha_{0}\in \widetilde{\Psi}$ \index{alpha0@$\alpha_{0}$} such that  $\rho(\alpha_{0})= \lambda_{0}$.
 \end{prop}
 
 \begin{proof} Suppose that there are two distinct simple roots $\alpha_{0},\beta_{0}$ such that  $\rho(\alpha_{0})=\rho(\beta_{0})= \lambda_{0}$, and suppose that these  two roots belong to the same irreducible component of   $\widetilde{\cal R}$.
 
  Then the highest root in this irreducible component will be of the form:
 $$\gamma=m_{0}\alpha_{0}+m_{1}\beta_{0} +\sum_{\alpha\in \widetilde{\Psi}\setminus\{\alpha_{0},\beta_{0}\}}m_{\alpha}\alpha,\,\,\, m_{0},m_{1}\geq1, \,m_{\alpha}\geq 0,$$
 and then
 $$\rho(\gamma)=(m_{0}+m_{1})\lambda_{0} +\sum_{\nu \in \Pi}m_{\nu}\nu.$$
 And this is impossible according to Theorem \ref{thbasepi}. 
 
 Consequently each of the roots $\alpha_{0}$ and $\beta_{0}$ belong to a different irreducible component of  $\widetilde{\cal R}$. Let $\omega_{0}$ (resp.  $\omega_{1}$) the highest root of the irreducible  component of $\widetilde{\cal R}$ containing  $\alpha_{0}$ (resp. $\beta_{0}$). As $\omega_{i}(H_{0})=2$ (because $\alpha_{0}(H_{0})=\beta_{0}(H_{0})=2$ and $\omega_{i}(H_{0})=-2,0, 2$) the linear forms   $\omega_{i}$ are dominant weights of the irreducible representation  $(\overline{\go g},\overline{V^+})$. This implies the result.

 
 
  \end{proof}
  
  \begin{rem}\label{simplicite-eventuelle} If we had supposed that $[\widetilde {\go g},\widetilde {\go g}]$ was absolutely simple, then the second part of the proof would have been superfluous. 
  \end{rem}

  \vskip 20pt
 \subsection{The highest root in  $\widetilde{\Sigma}$}\hfill
 \vskip 10pt
 \begin{prop}\label{prop.plusgranderacine}\hfill
 
There is a unique root $\lambda^0 \in \widetilde{\Sigma}$ \index{lambda0@$\lambda^0$} such that
 \begin{align*}
\bullet\hskip 15pt &\lambda ^0(H_0)=2\ ;\\
\bullet \hskip 15pt &\forall  \lambda \in  \Sigma ^+,  \lambda ^0+\lambda \notin \widetilde{
\Sigma }\ .
\end{align*}
This root  $\lambda^0$ is the highest root of the irreducible component of $\widetilde{\Sigma}$ which contains  $\lambda_{0}$.

 \end{prop}
 
 \begin{proof} Let $\omega$ be the highest weight of the representation $(\go{g}, \overline{V^+})$. We will show that the restriction of $\omega$ to $\go{a}$   is the unique root in $\widetilde{\Sigma}$ satisfying the conditions of the proposition.

 Define  $\lambda^0=\rho(\omega)$. Then $\lambda^0\in \widetilde{\Sigma}^+$ and  $\lambda^0(H_{0})=\omega(H_{0})=2$.  Let $\lambda\in \Sigma^+$. If $\lambda+\lambda^0\in \widetilde{\Sigma}$ there exist two roots $\alpha, \beta\in \widetilde{\cal R}$ such that $\rho(\alpha)=\lambda$ and $\rho(\beta)=\lambda+\lambda^0$. As  ${\overline{\go{g}}}^\beta\subset\overline{V^+}$, the root $\beta$ will be a weight of  $(\overline{\go{g}},\overline{V^+})$ and therefore can be written :
 $\beta=\omega-\sum_{\gamma\in \widetilde{\Psi}}m_{\gamma}\gamma$ ($m_{\gamma}\in \N$).
 Restricting this equality to  $\go{a}$ one gets:
 $$\rho(\beta)=\lambda+\lambda^0=\rho(\omega)-\sum_{\gamma\in \widetilde{\Psi}}m_{\gamma}\rho(\gamma)=\lambda^0- \sum_{\gamma\in \widetilde{\Psi}}m_{\gamma}\rho(\gamma).$$
 Hence
 $$\lambda=- \sum_{\gamma\in \widetilde{\Psi}}m_{\gamma}\rho(\gamma).\eqno{(*)}$$
 
 But  $\lambda\in \Sigma^+\subset \widetilde\Sigma^+$ and from the hypothesis  $\rho(\gamma)\in \widetilde\Sigma^+\cup\{0\}$; the preceding equation $(*)$ is therefore  impossible and we have showed that $\lambda^0=\rho(w)$ satisfies the required properties.

Let $\lambda^1$ be another root in $ \widetilde\Sigma$ having these properties. In fact $\lambda^1\in \widetilde\Sigma^+$. Set
$$S=\{\alpha\in \widetilde{\cal R}^+\,|\, \rho(\alpha)=\lambda^1\}.$$
Each element in  $S$ is a weight   of $(\go{g},\overline{V^+})$ (as $\alpha(H_{0})=\rho(\alpha)(H_{0})=\lambda^1(H_{0})=2$, we have $\overline{\widetilde {\go g}}^{^{\alpha}}\in \overline{V^+}$). Let $\omega^1$ be a maximal element in $S$ for the order induced by $\widetilde{\cal R}^+$.

\hskip 10pt $\bullet $ if  $\beta\in \widetilde{\cal R}^+$ is such that $\rho(\beta)=0$, then $\omega^1+\beta$ is not a root, because we would have  $\omega^1+\beta\in S$ and  this contradicts  the maximality of $\omega^1$.

\hskip 10pt $\bullet $ If $\beta\in \widetilde{\cal R}^+$ is such that $\rho(\beta)\neq 0$, then  $\omega^1+\beta$ is not a root, because in that case $\rho(\omega^1+\beta)=\lambda^1+\rho(\beta)$ would be a root of $ \widetilde\Sigma^+$ and this contradicts the second property.

Therefore  $\omega^1$ is a highest weight of   $(\overline{\go{g}},\overline{V^+})$, hence $\omega^1=\omega$. This implies $\lambda^1=\lambda^0$.

The commutativity of  $V^+$ implies that  $\lambda^0+\lambda\notin \widetilde{\Sigma}$ for $\lambda\in \widetilde{\Sigma}^+\setminus {\Sigma}^+$. From the obtained characterisation of  $\lambda^0$ we obtain the last assertion.
 
 \end{proof}

  \vskip 20pt
 \subsection{The first step in the descent}\label{section-descente1}\index{descent}\hfill
  \vskip 10pt
  
  Let  $\widetilde{ {\go l}}_0$ \index{ltilde0@$\widetilde{ {\go l}}_0$} be the algebra generated by the root spaces ${\widetilde{\go{g}}^{\lambda_0}}$ and ${\widetilde{\go{g}}^{-\lambda_{0}}}$.
  One has:
  $$\widetilde{ {\go l}}_0={\widetilde{\go{g}}^{-\lambda_{0}}}\oplus[{\widetilde{\go{g}}^{-\lambda_{0}}},{\widetilde{\go{g}}^{\lambda_{0}}}]\oplus {\widetilde{\go{g}}^{\lambda_{0}}}.$$
  (Just remark that ${\widetilde{\go{g}}^{-\lambda_{0}}}\oplus[{\widetilde{\go{g}}^{-\lambda_{0}}},{\widetilde{\go{g}}^{\lambda_{0}}}]\oplus {\widetilde{\go{g}}^{\lambda_{0}}}$ is a Lie algebra). This algebra is graded by the  coroot  $H_{\lambda_{0}}\in \go{a}$  of $\la_0$.
 
 We will need the following result:
 
 \begin{lemme}\label{lemme-alg.simple} \hfill

 Let $\go{u}=\go{u}_{-1}\oplus \go{u}_{0}\oplus \go{u}_{1}$ be a semi-simple graded Lie algebra over $F$. Suppose that $\go{u}_{1}$ is an absolutely simple   $\go{u}_{0}$-module. Then the Lie algebra  $\go{u}'$ generated by $\go{u}_{1} $ and  $\go{u}_{-1}$ is absolutely simple.

 \end{lemme}
 
 \begin{proof} The algebra $\overline{\go{u}}$ is again graded and semi-simple:
 $$\overline{\go{u}}=\overline{\go{u}_{-1}}\oplus \overline{\go{u}_{0}}\oplus \overline{\go{u}_{1}}$$
  and from the hypothesis  $\overline{\go{u}_{1}}$ is a simple $\overline{\go{u}_{0}}$-module. As before for $\widetilde{ {\go l}}_0$, one has  $\go{u}'=\go{u}_{-1}\oplus  [\go{u}_{-1}, \go{u}_{1}]\oplus \go{u}_{1}$, and one verifies easily that  $\go{u}'=\go{u}_{-1}\oplus  [\go{u}_{-1}, \go{u}_{1}]\oplus \go{u}_{1}$ is an ideal of $\go{u}$, and therefore semi-simple. Then it is enough to prove that $\overline{\go{u}'}=\overline{\go{u}_{-1}}\oplus  [\overline{\go{u}_{-1}},\overline{ \go{u}_{1}}]\oplus \overline{\go{u}_{1}}$ is a simple algebra over  $\overline{F}.$ Let $J$ be an ideal of $\overline{\go{u}'}$. We will show that $J=\{0\}$ or  $J=\overline{\go{u}'}$. 
  
  Note first that the ideal $\overline{\go{u}'}^{\perp}=\overline{\go{u}}''$, orthogonal of $\overline{\go{u}'}$ for the Killing form of  $\overline{\go{u}}$ , is a subset of  $\overline{\go{u}_{0}}$. Hence
  $$\overline{\go{u}}=\overline{\go{u}_{-1}}\oplus ( [\overline{\go{u}_{-1}},\overline{ \go{u}_{1}}]\oplus\overline{\go{u}_{0}}'')\oplus \overline{\go{u}_{1}}$$
 As $J$ is an ideal of $\overline{\go{u}'}$, the space  $J\cap  \overline{\go{u}_{1}}$ is stable  under $[\overline{\go{u}_{-1}},\overline{ \go{u}_{1}}] $. On the other hand  $J\cap  \overline{\go{u}_{1}}$ is also stable under  $\overline{\go{u}_{0}}''$ because  $[\overline{\go{u}_{0}}'',\overline{\go{u}_{1}}]=\{0\}$ (In a semi-simple Lie algebras orthogonal ideals commute). Therefore  $J\cap  \overline{\go{u}_{1}}$ is a sub-$\overline{\go{u}_{0}}$-module of $\overline{ \go{u}_{1}}$. From the hypothesis, either $J\cap  \overline{\go{u}_{1}}=   \overline{\go{u}_{1}}$ or   $J\cap  \overline{\go{u}_{1}}=\{0\}$.
  
  $\bullet$ If $J\cap  \overline{\go{u}_{1}}=   \overline{\go{u}_{1}}$, then  $[\overline{\go{u}_{-1}},\overline{ \go{u}_{1}}]\subset J$. Let  $U_{0}$ be the grading element  (which always exists). From the semi-simplicity of $\go{u}'$, one gets $U_{0}\in [\go{u}_{-1}, \go{u}_{1}]$. Hence $U_{0}$  is in $J$. This implies that $\overline{\go{u}_{-1}}\subset J$, and finally  $J=\overline{\go{u}'}$.
  
   $\bullet$ If $J\cap  \overline{\go{u}_{1}}=   \{0\}$, then if $J^{\perp}$ is the ideal of $\overline{\go{u}'}$ orthogonal to $J$ for the Killing form, we get $J^{\perp}=\overline{\go{u}'}$, hence $J=\{0\}$.
   
   This proves that $\overline{\go{u}'}$ is simple.
 
 \end{proof}
 
 \begin{prop}\label{ell-0-simple}\hfill
 
The representation  $([{\widetilde{\go{g}}^{-\lambda_{0}}},{\widetilde{\go{g}}^{\lambda_{0}}}], {\widetilde{\go{g}}^{\lambda_{0}}})$ is absolutely simple and the algebra  $\widetilde{ {\go l}}_0={\widetilde{\go{g}}^{-\lambda_{0}}}\oplus[{\widetilde{\go{g}}^{-\lambda_{0}}},{\widetilde{\go{g}}^{\lambda_{0}}}]\oplus {\widetilde{\go{g}}^{\lambda_{0}}}$ is absolutely simple of split rank  $1$.
 \end{prop}
 
 \begin{proof} Let $\go{m}={\cal Z}_{\go{g}}(\go{a})$ be the centralizer of $\go{a}$ in  $\go{g}$. Consider the graded algebra
 $$\go{u}= {{\widetilde{\go{g}}}^{-\lambda_{0}}\oplus\go{m}\oplus{\widetilde{\go{g}}}^{\lambda_{0}}}$$
From Lemma  \ref{lemme-alg.simple}, to prove that $\widetilde{ {\go l}}_0$ is absolutely simple, it is enough to show that  the representation $(\go{m}, {\widetilde{\go{g}}^{\lambda_{0}}})$ is absolutely simple. One has:
 $$\overline{\go{m}}=\overline{\go{j}}\oplus\sum_{\{\alpha\in \widetilde{\cal R}\,|\, \rho(\alpha)=0\}}{\overline{\widetilde{\go g}}}^{^{\alpha}}.$$
 
The algebra $\overline{\go{m}}$ is reductive, and its root system for the Cartan subalgebra $\overline{\go{j}}$ is  $\widetilde{\cal R}_{0}=\{\alpha\in \widetilde{\cal R}\,|\, \rho(\alpha)=0\}$. We put the order induced by $\widetilde{\cal R}^+$ on $\widetilde{\cal R}_{0}$.
    
   We have to show that the module $(\overline{\go{m}},\overline{{\widetilde{\go{g}}}^{\lambda_{0}})}$ is  simple. If it would not, this  module would have a lowest weight $\alpha_{1}$  distinct from $\alpha_{0}$ (see  Proposition \ref{prop.alpha0}). But from hypothesis  (${\bf {H_{2}}}$), the module ($\overline{\go{g}},\overline {V^+}$) is simple, and its lowest weight (with respect to ${\cal R}^+=\widetilde{\cal R}^+\cap {\cal R}$) is  $\alpha_{0}$. There exists then a sequence  $\beta_{1},\dots,\beta_{k}$ of simple roots in $\Psi=\widetilde{\Psi}\setminus \{\alpha_{0}\}$ such that $\alpha_{1}=\alpha_{0}+\beta_{1}+\dots+\beta_{k}$, and such that each partial sum is a root. But as  $\rho(\alpha_{0})=\rho(\alpha_{1})=\lambda_{0}$, and as the roots $\beta_{i}$ are in  $\widetilde{\cal R}^+$, we obtain  $\rho(\beta_{i})=0$, for $i=1,\dots,k$. Hence $\beta_{i}\in \widetilde{\cal R}_{0}^+$ and  $\alpha_{1}-\beta_{k}=\alpha_{0}+\beta_{1}+\dots+\beta_{k-1}$ is a root. As $\alpha_{1}$ is a lowest weight, this is impossible. 
   
  The fact that this algebra is of split rank one is easy. 
 
  \end{proof}
  
  Let  $\widetilde{\go{g}}_{1}\index{ggothique1@$\widetilde{\go{g}}_{1}$}={\cal Z}_{\widetilde{\go{g}}}(\widetilde{ {\go l}}_0)$ be the  centralizer of $\widetilde{ {\go l}}_0$ in $\widetilde{\go{g}}$. This is a reductive subalgebra (see \cite{Bou2},chap.VII, \S 1, $n^\circ 5$, Prop.13). The same is then true for $\overline{\widetilde{\go{g}}_{1}}=\overline{{\cal Z}_{\widetilde{\go{g}}}(\widetilde{ {\go l}}_0)}={\cal Z}_{\overline{\widetilde {\go g}}}(\overline{\widetilde{ {\go l}}_0})$ (\cite{Bou3}, Chap. I, \S 6, $n^\circ 10$). (We will show the last equality in the proof of the next proposition).

  \begin{prop}\label{prop-gtilde(1)}\hfill
  
  If $\widetilde{\go{g}}_{1}\cap V^+\neq\{0\}$, then  $\widetilde{\go{g}}_{1}$ satisfies the hypothesis $({\bf H_{1}})$ and $({\bf H_{2}})$, where the grading is defined by the element  $H_{1}= H_{0}-H_{\lambda_{0}}$\index{H1@$H_{1}$}. More precisely the grading is given by
$$\widetilde{\go{g}}_{1}= V_{1}^-\oplus \go{g}_{1}\oplus V_{1}^+$$ \index{V1@$V_{1}^{\pm}$} \index{ggothique1@$\go{g}_{1}$}
where  $\go{g}_{1}=\go{g}\cap \widetilde{\go{g}}_{1}$, $V_{1}^{\pm}=V^{\pm}\cap \widetilde{\go{g}}_{1}$.
  \end{prop}
    
  \begin{proof} It is clear that  $H_{1}\in \widetilde{\go{g}}_{1}$. As $H_{\lambda_{0}}\in \widetilde{ {\go l}}_0$, the actions of  $\ad H_{0}$ and of  $\ad H_{1}$ on $\widetilde{\go{g}}_{1}$ are the same. As $V^+_{1}= \widetilde{\go{g}}_{1}\cap V^+\neq\{0\}$, the eigenvalues of $\ad H_{1}$ on  $\widetilde{\go{g}}_{1}$ are $-2,0,2$. Therefore the hypothesis ${\bf{H_{1}}}$ is satisfied.
  
Note that $\overline{V_{1}^+}=\overline{V^+}\cap \overline{\widetilde{\go{g}}_{1}}$. It remains to show that   the representation  $(\go{g}_{1},\overline{V_{1}^+})$ (or $(\overline{\go{g}_{1}},\overline{V_{1}^+})$) is irreducible over $\overline{F}$.

 From Lemma \ref{lemme-espacesradiciels-cloture} one has  $\overline{\widetilde{ {\go l}}_0}= \overline{\widetilde{\go{g}}}^{-S_{\lambda_{0}}}\oplus [\overline{\widetilde{\go{g}}}^{-S_{\lambda_{0}}},\overline{\widetilde{\go{g}}}^{S_{\lambda_{0}}}]\oplus \overline{\widetilde{\go{g}}}^{S_{\lambda_{0}}}$ where $\overline{\widetilde{\go{g}}}^{S_{\lambda_{0}}}= \overline{\widetilde{\go{g}}^{\lambda_{0}}}=\sum_{\alpha\in S_{\lambda_{0}}}\overline{\widetilde {\go g}}^{^{\alpha}}$.
  Let us first describe  $\overline{\widetilde{\go{g}}_{1}}= \overline{{\cal Z}_{\widetilde{\go{g}}}(\widetilde{ {\go l}}_0)}$. It is easy to see that  $\overline{{\cal Z}_{\widetilde{\go{g}}}(\widetilde{ {\go l}}_0)}\subset {\cal Z}_{\overline{\widetilde {\go g}}}(\overline{\widetilde{ {\go l}}_0})$. Let  $X\in {\cal Z}_{\overline{\widetilde {\go g}}}(\overline{\widetilde{ {\go l}}_0})$. Let us write
  $X=H+\sum_{\alpha\in \widetilde{\cal R}}X_{\alpha} $ whith $H\in \overline{\go{j}}$ and with $X_{\alpha}\in {\overline{\widetilde{\go g}}}^{^{\alpha}}$. As $[X, X_{\beta}]=0$ for  $\beta\in \pm S_{\lambda_{0}}$, we obtain that   $X\in \overline{\go{j}}_{1}\oplus \sum_{\beta\in \widetilde{ {\cal R}}_1}\overline{\widetilde {\go g}}^{^{\beta}}$, where 
   $$\overline{\go{j}}_{1}=\{H\in \overline{\go {j}}\,|\, \alpha(H)=0\,, \forall \alpha\in S_{\lambda_{0}}\}$$
   
   $$ \widetilde{ {\cal R}}_1=\{\beta \in \widetilde{ {\cal R}}\mid  \ \beta \sorth
\alpha, \, \forall \alpha \in S_{\lambda_{0}} \}.$$ \index{Rtilde1@ $\widetilde{ {\cal R}}_1$}
We will now show that $\overline{\go{j}}_{1}\oplus \sum_{\beta\in \widetilde{ {\cal R}}_1}\overline{\widetilde {\go g}}^{^{\beta}}\subset \overline{{\cal Z}_{\widetilde{\go{g}}}(\widetilde{ {\go l}}_0)}$.

$\bullet$ \hskip 5pt Consider first $\overline{\widetilde {\go g}}^{^{\beta}}$ for  $\beta\in \widetilde{ {\cal R}}_1$. Any element  $X\in \overline{\widetilde {\go g}}^{^{\beta}}\subset \widetilde{\go{g}}$ can be written $X= \sum_{i=1}^{n}e_{i}X_{i}$ where $X_{i}\in \widetilde{\go{g}}$ and where the $e_{i}$'s are elements of $\overline{F}$ which are free over  $F$. As $\beta \in \widetilde{ {\cal R}}_1$ one has  $[X, \overline{\widetilde {\go g}}^{^{\beta}}]=\{0\}$ for each $\gamma\in S_{\lambda_{0}}$. As $\widetilde{\go{g}}^{\lambda_{0}}\subset\overline{\widetilde{\go{g}}}^{S_{\lambda_{0}}}\subset \sum_{\alpha\in S_{\lambda_{0}}}\overline{\widetilde {\go g}}^{^{\alpha}}$, one has $[X, \widetilde{\go{g}}^{\lambda_{0}}]=\{0\}=\sum_{i=1}^{n}e_{i}[X_{i},\widetilde{\go{g}}^{\lambda_{0}}]$. As $[X_{i},\widetilde{\go{g}}^{\lambda_{0}}]\subset\widetilde{\go{g}}$ and as the  $e_{i}$'s are free over  $F$, we obtain that  $ [X_{i},\widetilde{\go{g}}^{\lambda_{0}}]=\{0\}$.

One would similarly prove that $ [X_{i},\widetilde{\go{g}}^{-\lambda_{0}}]=\{0\}$. Hence $X\in \overline{{\cal Z}_{\widetilde{\go{g}}}(\widetilde{ {\go l}}_0)}$.  

$\bullet$ \hskip 5pt Consider now  $\overline{\go{j}}_{1}=\{H\in \overline{\go {j}}\,|\, \alpha(H)=0\,, \forall \alpha\in S_{\lambda_{0}}\}$. An element  $u\in\overline{\go{j}}_{1}$ can be written  $u=\sum_{i=1}^{n}a_{i}u_{i}$ where the elements $a_{i}\in \overline{F}$ are free over $F$ and where $u_{i}\in \go{j}$. Then for  $\alpha\in S_{\lambda_{0}}$ one has: $\alpha(u)=0=\sum_{i=1}^{n}a_{i}\alpha(u_{i})$. This implies that $\alpha(u_{i})=0$, for any  $\alpha\in S_{\lambda_{0}}$ and any   $i$. Therefore the elements $u_{i}$ belong to $\go{j}_{1}=\{u\in \go{j}, \alpha(u)=0,\,\forall \alpha\in S_{\lambda_{0}}\}$. But then  $u=\sum_{i=1}^{n}a_{i}u_{i}\in \overline{(\go{j}_{1})}$ and $[u_{i},  \widetilde{\go{g}}^{\pm\lambda_{0}}]\subset [u_{i}, \overline{\widetilde{\go{g}}}^{\pm S_{\lambda_{0}}}]=\{0\}$ (from the definition of  $\go{j}_{1}$). Hence $u_{i}\in {\cal Z}_{\widetilde{\go{g}}}(\widetilde{ {\go l}}_0)$, and therefore $u\in \overline{{\cal Z}_{\widetilde{\go{g}}}(\widetilde{ {\go l}}_0)}$.

Finally we have proved that 
$$\overline{\widetilde{\go{g}}_{1}}=\overline{{\cal Z}_{\widetilde{\go{g}}}(\widetilde{ {\go l}}_0)}={\cal Z}_{\overline{\widetilde {\go g}}}(\overline{\widetilde{ {\go l}}_0})=\overline{\go{j}}_{1}\oplus \sum_{\beta\in \widetilde{ {\cal R}}_1}\overline{\widetilde {\go g}}^{^{\beta}}.$$

If $\beta\in  \widetilde{ {\cal R}}_1$, then  $\alpha(H_{\beta})=0$ for all root $\alpha\in S_{\lambda_{0}}$. Hence $H_{\beta}\in \overline{\go{j}}_{1}$, and the restriction of $\beta $ to $\overline{\go{j}}_{1}$ is non zero. This implies that $\overline{\go{j}}_{1}$ is a Cartan subalgebra of the reductive algebra $\overline{\widetilde{\go{g}}_{1}}$ and $\widetilde{ {\cal R}}_1$ can be seen as the root system of the pair  $(\overline{\widetilde{\go{g}}_{1}},\overline{\go{j}}_{1})$. The order on  $\widetilde{\cal R}$ defined by  $\widetilde{\cal R}^+$ induces an order on  $\widetilde{ {\cal R}}_1$ and on the root system  $ {\cal R}_1= \widetilde{ {\cal R}}_1\cap  {\cal R}$ \index{R1@${\cal R}_1$}of the pair  $(\overline{\go{g}_{1}},\overline{\go{j}}_{1})$ by setting:
$$\widetilde{ {\cal R}}_1^+=\widetilde{ {\cal R}}_1\cap \widetilde{ {\cal R}}^+,\,\, {\cal R}_1^+={\cal R}_{1}\cap  \widetilde{ {\cal R}}^+=\widetilde{ {\cal R}}_1\cap {\cal R}^+.$$

Let  $\omega_{1}$ be the highest weight, for the preceding order, of one of the irreducible components of the representation $(\overline{\go{g}_{1}},\overline{V_{1}^+})$. $\omega_{1}$ is a root of $\widetilde{ {\cal R}}_1^+$ and we will show  that it is also    the highest weight of the representation   $(\overline{\go{g}}, \overline{V^+})$ which is irreducible from ${\bf (H_{2})}$. This will show that  $(\overline{\go{g}_{1}},\overline{V_{1}^+})$  is irreducible. To do this we will show  that if $\beta\in {\cal R}^+$ then $\omega_{1}+\beta\notin \widetilde{\cal R}$.
\vskip 5pt

\hskip 10pt $(1)$ If  $\beta\in {\cal R}_{1}^+$, then from the definition of  $\omega_{1}$ we get $[\overline{\widetilde {\go g}}^{^{\omega_{1}}},\overline{\widetilde {\go g}}^{^{\beta}}]$=\{0\}. Hence $\omega_{1}+\beta\notin \widetilde{\cal R}$.
\vskip 5pt
\hskip 10pt $(2)$ If $\beta\notin {\cal R}_{1}^+$, there exists  a root $\alpha\in S_{\lambda_{0}}$ such that $\alpha$ and  $\beta$ are not strongly orthogonal. We will show that in that case there is a root $\gamma\in S_{\lambda_{0}}$ such that 
$$(*)\hskip 15pt \gamma+\beta\in \widetilde{\cal R}\text{  and  }\gamma-\beta\notin \widetilde{\cal R}.$$

\vskip 5pt
\hskip 25pt $(2.1)$ If $\rho(\beta)\neq 0$, then $\alpha-\beta$ is not a root. If it would be the case, we would have 
$$\rho(\alpha-\beta)=\rho(\alpha)-\rho(\beta)=\lambda_{0}-\rho(\beta).$$
But  $\rho(\beta)\neq \lambda_{0}$ as $\beta\in \Sigma^+$, so  $\lambda_{0}-\rho(\beta)$ would be a root, and this is impossible as  $\lambda_{0}$ is a simple root in  $\widetilde{\Sigma}^+$. In this case the root  $\alpha$ satisfies $(*)$.

\vskip 5pt
\hskip 25pt $(2.2)$ If  $\rho(\beta)=0$, consider the root  $\beta$-string  through $\alpha$. These roots are in  $S_{\lambda_{0}}$ and as $\alpha$ and  $\beta$ are not strongly orthogonal, this string contains at least two roots. This implies that there exists a root  $\gamma$ verifying  $(*)$.

Therefore, from  \cite{Bou1} (chap. 6, \S 1, Prop. 9) we obtain that  $(\beta,\gamma)<0$. As $\omega_{1}$ and  $\gamma$ are strongly orthogonal, we have    $(\omega_{1},\gamma)=0$ ($\omega_{1}\in \widetilde{ {\cal R}}_1^+$ and  $\gamma\in S_{\lambda_{0}}$).

Now if  $\omega_{1}+\beta$ is a root, we would have:
$$(\omega_{1}+\beta, \gamma)=(\omega_{1},\gamma)+(\beta,\gamma)=(\beta,\gamma)<0.$$
Then $\omega_{1}+\beta+\gamma$ is a root and  $(\omega_{1}+\beta+\gamma)(H_{0})=4 $ (as $\omega_{1}(H_{0})=2, \beta(H_{0})=0, \gamma(H_{0})=2$). Hence $\omega_{1}+\beta$ is not a root.

This means that  $\omega_{1}$ is a highest weight of  $(\overline{\go{g}}, \overline{V^+})$ which is irreducible from  ${\bf (H_{2})}$. Therefore $(\go{g}_{1},\overline{V_{1}^+})$ is irreducible (and  $\omega_{1}=\omega$).
\vskip 5pt
 
  \end{proof}
  
  \vskip 10pt
  
  Set  $\go{a}_{1}=\go{a}\cap \widetilde{\go{g}}_{1}$. Then  $\go{a}_{1}$ is a maximal split abelian  subalgebra of    $\widetilde{\go{g}}_{1} $  included in $\go{g}_{1}$ (because the actions of the maximal split abelian subalgebras can be diagonalized in all finite dimensional representation).
  
  \begin{prop}\label{racines-g1}\hfill

  $1)$ The root system $\widetilde{\Sigma}_{1}$ \index{Sigma1tilde@$\widetilde{\Sigma}_{1}$}of the pair $(\widetilde{\go{g}}_{1},\go{a}_{1})$ is
  $$\widetilde{\Sigma}_{1}=\{\lambda\in  \widetilde{\Sigma}\,,\,\lambda\sorth \lambda_{0}\}=\{\lambda\in  \widetilde{\Sigma}\,,\,\lambda\perp \lambda_{0}\}.$$
  (where $\perp$ means "orthogonal" and where  $\sorth$ means  "strongly  orthogonal".)
  \vskip 3pt
  
  $2)$ Consider the order on $\widetilde{\Sigma}_{1}$ defined by 
  $$\widetilde{\Sigma}_{1}^+=\widetilde{\Sigma}^+\cap \widetilde{\Sigma}_{1}. $$
  This order satisfies the properties of    Theorem \ref{thbasepi} for the graded algebra $\widetilde{\go{g}}_{1}$. 
   \vskip 3pt
  
  $3)$ The set of simple roots $\widetilde{\Pi}_{1}$ \index{Pitilde1@$\widetilde{\Pi}_{1}$}defined by  $\widetilde{\Sigma}_{1}^+$ is given by:
  $$\widetilde{\Pi}_{1}=({\Pi}\cap\widetilde{\Sigma}_{1})\cup\{\lambda_{1}\}$$
  where  $\lambda_{1}$ \index{lambda1@$\lambda_{1}$}is the unique root of  $\widetilde{\Pi}_{1}$ such that $\lambda_{1}(H_{0}-H_{\lambda_{0}})=\lambda_{1}(H_{0})=2$.

  \end{prop}
  
  \begin{proof}\hfill
  
 $1)$ Let us first show that a root  $\lambda\in \widetilde{\Sigma}$ is strongly orthogonal to  $\lambda_{0}$ if and only if it is orthogonal to $\lambda_{0}$.

   Let   $\lambda\in \widetilde{\Sigma}$ be a root orthogonal to  $\lambda_{0}$. Let
   $$\lambda-q\lambda_{0},\dots,\lambda-\lambda_{0},\lambda, \lambda+\lambda_{0},\dots,\lambda+p\lambda_{0}$$
   be the  $\lambda_{0}$-string of roots through $\lambda$.
   From    \cite{Bou1} (Chap. VI, \S1, $n^\circ 3$, Prop. 9) one has  $p-q= -2\frac{(\lambda,\lambda_{0})}{(\lambda_{0},\lambda_{0})}=0$. Hence $p=q$, and the string is symmetric.

  \vskip 5pt
   
  $a)$ If $\lambda(H_{0})=2$, then as  $\lambda+\lambda_{0}\notin \widetilde{\Sigma}$  (from ${\bf(H_{1})}$),  we get  $\lambda-\lambda_{0}\notin \widetilde{\Sigma}$
  
  \vskip 5pt
  
 $b)$ If $\lambda(H_{0})=-2$, then $\lambda-\lambda_{0}\notin \widetilde{\Sigma}$,  and the same proof shows that  $\lambda+\lambda_{0}\notin\widetilde{\Sigma}$.    
     \vskip 5pt
  
    $c) $ If  $\lambda(H_{0})=0$ then $\lambda\in \Sigma$. Recall that  $\widetilde{\Pi}= \Pi\cup \{\lambda_{0}\}$. If $\lambda\in \Sigma^+$ then $\lambda-\lambda_{0}\notin \widetilde{\Sigma}$, and if $\lambda\in \Sigma^-$ then $\lambda+\lambda_{0}\notin \widetilde{\Sigma}$.  Again the same proof shows that, respectively,  $\lambda+\lambda_{0}\notin \widetilde{\Sigma}$ and  $\lambda-\lambda_{0}\notin \widetilde{\Sigma}$.
    
      \vskip 5pt
      In all cases we have showed that $\lambda\sorth \lambda_{0}$.
      
      For  $\lambda \in \{\lambda\in  \widetilde{\Sigma}\,,\,\lambda\sorth \lambda_{0}\}$, it is clear that  ${\widetilde{\go{g}}}^{\lambda}\subset \widetilde{\go{g}}_{1}={\cal Z}_{\widetilde{\go{g}}}(\widetilde{ {\go l}}_0)$, that   its coroot $H_{\lambda}$ belongs to  $\go{a}_{1}=\go{a}\cap \widetilde{\go{g}}_{1}$ and that  $\lambda_{|_{\go{a}_{1}}}$ is a root of the pair  $(\widetilde{\go{g}}_{1},\go{a}_{1})$. Conversely any root of the pair  $(\widetilde{\go{g}}_{1},\go{a}_{1})$ can be extended to a linear form on $\go{a}$ by setting  $\lambda(H_{0})=0$, and this extension  is a root orthogonal to $\lambda_{0}$, and hence strongly orthogonal to $\lambda_{0}$ from above, and finally this extension is in $\widetilde{\Sigma}_{1}$ .
      
      \vskip 5pt
      2) The set  $\widetilde{\Sigma}_{1}^+=\widetilde{\Sigma}^+\cap \widetilde{\Sigma}_{1}      $ defines an order on $\widetilde{\Sigma}$. For $\lambda\in \widetilde{\Sigma}_{1}^+$ one has  $\lambda(H_{1})=\lambda(H_{0})= 0 \text{ or }2$, and this gives 2).
      
       \vskip 5pt
       
       3) Let  $\widetilde{\Pi}_{1}$ be the set of simple roots in $\widetilde{\Sigma}_{1}^+$. From Theorem \ref{thbasepi}, $\widetilde{\Pi}_{1}=\Pi_{1}\cup \{\lambda_{1}\}$ where $\Pi_{1}=\{\lambda\in \widetilde{\Pi}_{1}\,,\, \lambda(H_{1})=0\}$\index{Pi1@$\Pi_{1}$}.  If $\lambda\in {\Pi}\cap\widetilde{\Sigma}_{1}$ then  $\lambda\in \widetilde{\Pi}_{1}$ and $\lambda(H_{1})=\lambda(H_{0})-\lambda(H_{\lambda_{0}})=0$.  Therefore $ {\Pi}\cap\widetilde{\Sigma}_{1}\subset \Pi_{1}$. Conversely let  $\mu\in \Pi_{1}$. Then  $\mu\in \Sigma^+$. Hence $\mu$ can be written  $\mu=\sum_{\nu \in \Pi}m_{\nu}\nu$, with $m_{\nu}\in \N$. One has  $(\mu,\lambda_{0})=0$ and   $(\nu, \lambda_{0})\leq 0$ (because $\Pi\cup \{\lambda_{0}\}= \widetilde{\Pi}$ is a set of simple roots). Therefore  if $m_{\nu}\neq 0$, one has $(\nu, \lambda_{0})=0$, and hence $\nu\in \widetilde{\Sigma}_{1}$. Finally we have proved that  $\Pi_{1}={\Pi}\cap\widetilde{\Sigma}_{1}$.

    \end{proof}
    
    \begin{rem} \label{rem-unicite-lambda0}\hfill
    
    One may remark that  $\Pi_{1}\subset \Pi$, but  $\widetilde{\Pi}_{1}$ is not a subset of $\widetilde{\Pi}$. Indeed $\lambda_{1}(H_{0})=2$, but there is only one root $\lambda$ in $\widetilde{\Pi}$ such that $\lambda(H_{0})=2$, namely $\lambda_{0}$. Hence $\lambda_{1}\notin \widetilde{\Pi}$.
    \end{rem}

  \vskip 20pt
 \subsection{The descent}\label{sub-section-descente2}\hfill
  \vskip 10pt
  
  \begin{theorem}\label{th-descente}\hfill 
  
 There exists a unique sequence of strongly orthogonal roots in  $\widetilde{\Sigma}^+\setminus \Sigma^+$, denoted by  $\lambda_{0},\lambda_{1},\dots,\lambda_{k}$ \index{lambdaj@$\lambda_{j}$}and a sequence of reductive Lie algebras  $\widetilde{\go{g}}\supset \widetilde{\go{g}}_{1}\supset \dots \supset \widetilde{\go{g}}_{k}$ \index{ggothiquetilde1@$\widetilde{\go{g}}_{j}$}such that 
  \vskip 5pt
  
  \hskip 10pt $(1)$ $\widetilde{\go{g}}_{j}={\cal Z}_{\widetilde{\go{g}}}(\widetilde{ {\go l}}_0\oplus \widetilde{ {\go l}}_1\oplus \dots\oplus \widetilde{ {\go l}}_{j-1})$ where $\widetilde{ {\go l}}_i={\widetilde{\go{g}}^{-\lambda_{i}}}\oplus[{\widetilde{\go{g}}^{-\lambda_{i}}},{\widetilde{\go{g}}^{\lambda_{i}}}]\oplus {\widetilde{\go{g}}^{\lambda_{i}}}$ is the subalgebra generated by  $\widetilde{\go{g}}^{\lambda_{i}}$ and $\widetilde{\go{g}}^{-\lambda_{i}}$.
  
 \vskip 5pt
 
 \hskip 10pt $(2)$ The algebra $\widetilde{\go{g}}_{j}$ is a graded Lie algebra verifying the hypothesis ${\bf (H_{1})}$ and ${\bf (H_{2})}$ with the grading element $H_{j}=H_{0}-H_{\lambda_{0}}-\dots-H_{\lambda_{j-1}}$  (where $H_{\lambda_s}$ is the coroot of $\la_s$, $s=0,\ldots, k$).  More precisely
 $$\widetilde{\go{g}}_{j}=V_{j}^-\oplus \go{g}_{j}\oplus V_{j}^+$$ \index{Vj@$V_{j}^\pm$} \index{ggothique@$\go{g}_{j}$}
 where $V_{j}^{\pm}= V^{\pm}\cap \widetilde{\go{g}}_{j}=V_{j-1}^{\pm}\cap  {\cal Z}_{\widetilde{\go{g} }_{j-1}}(  \widetilde{ {\go l}}_{j-1})$  and   $\widetilde{\go{g}}_{j}={\cal Z}_{\widetilde{\go{g}}}(\widetilde{ {\go l}}_0\oplus \widetilde{ {\go l}}_1\oplus \dots\oplus \widetilde{ {\go l}}_{j-1})={\cal Z}_{\widetilde{\go{g}}_{j-1}}( \widetilde{ {\go l}}_{j-1})$.

  \vskip 5pt
 
 \hskip 10pt $(3)$  $V^+\cap {\cal Z}_{\widetilde{\go{g}}}(\widetilde{ {\go l}}_0\oplus \widetilde{ {\go l}}_1\oplus \dots\oplus \widetilde{ {\go l}}_{k})=\{0\}$.
  \end{theorem}
  
  \begin{proof} The proof is done by induction on  $j$, starting from Proposition \ref{racines-g1}, and using the fact that $\widetilde{\go{g}}_{j}$ is the centralizer of  $\widetilde{ {\go l}}_{j-1}$ in  $\widetilde{\go{g}}_{j-1}$. It is worth noting that the construction stops for the index  $k$ such that $V^+\cap {\cal Z}_{\widetilde{\go{g}}_{k}}(\widetilde{ {\go l}}_{k})=\{0\}$,  which amounts to saying that  ${\cal Z}_{V^+}(\widetilde{ {\go l}}_0\oplus \widetilde{ {\go l}}_1\oplus \dots\oplus \widetilde{ {\go l}}_{k})=\{0\}$.
  
  \end{proof}

\begin{definition}\label{def-rang} The number $k+1$ \index{k@$k+1$}of strongly orthogonal roots  appearing in the preceding  Theorem  will be called the   {\bf rank} \index{rank} of the graded algebra  $\widetilde{\go{g}}$.
  \end{definition}
  
 \begin{notation}\label{notationsgj} Everything above also applies to the graded algebra $\widetilde{\go{g}}_{j}=V_{j}^-\oplus \go{g}_{j}\oplus V_{j}^+$  which is graded by  $H_{j}=H_{0}-H_{\lambda_{0}}-\dots-H_{\lambda_{j-1}}$. The algebra  $\go{a}_{j}=\go{a}\cap \widetilde{\go{g}}_{j}$ is a maximal split abelian subalgebra of $\widetilde{\go{g}}_{j}$ contained in  $\go{g}_{j}$. The set  $\widetilde{\Sigma}_{j}$ of roots of the pair  $(\widetilde{\go{g}}_{j}, \go{a}_{j})$ is the set of roots in $\widetilde{\Sigma}$ which are strongly orthogonal to $\lambda_{0},\lambda_{1},\dots,\lambda_{j-1}$. We put an order on  $\widetilde{\Sigma}_{j}$ by setting $\widetilde{\Sigma}_{j}^+= \widetilde{\Sigma}^+\cap \widetilde{\Sigma}_{j}$. The corresponding set of simple roots $\widetilde{\Pi}_{j}$ is given by  $\widetilde{\Pi}_{j}=\Pi_{j}\cup\{\lambda_{j}\}$ where $\Pi_{j}$ is the set of simple roots of $\Sigma_{j}$ defined by  $\Sigma_{j}^+=\Sigma_{j}\cap \Sigma^+$ (where $\Sigma_{j}= \Sigma\cap \widetilde{\Sigma}_{j}$). 
 
 \end{notation}
 
 \vskip 20pt
 
 \subsection{Generic elements in $V^+$ and maximal systems of long strongly orthogonal roots}\label{sub-section-generic-strongly-orth}\hfill
  \vskip 10pt
  
  We  define now the groups we will use. If  $\go{k}$ is a Lie algebra over  $F$, we will denote by  $\text{Aut}(\go{k})$ the group of automorphisms of $\go{k}$. The map 
  $g\longmapsto g\otimes1$ is an injective homomorphism from  $\text{Aut}(\go{k})$ in   $\text{Aut}(\go{k}\otimes_{F}\overline{F})= \text{Aut}(\overline{\go{k}})$. This allows to consider $\text{Aut}(\go{k})$ as a subgroup of    $\text{Aut}(\overline{\go{k}})$. 
  From now on   $\go{k}$ will be  reductive.
  
   We will denote by  $\text{Aut}_{e}(\go{k})$ \index{Aute@$\text{Aut}_{e}(\,.\,)$}the subgroup of elementary automorphisms of $\go{k}$, that is the  automorphisms which are finite products of automorphisms of the form $e^{\ad x}$, where $\ad(x)$ is nilpotent in $\go{k}$. If $g\in \text{Aut}_{e}(\go{k})$, then $g$ fixes pointwise the elements of   the center of $\go{k}$ and therefore  $\text{Aut}_{e}(\go{k})$ can be identified with $\text{Aut}_{e}([\go{k},\go{k}])$. 
   
   Let us set
   $\text{Aut}_{0}(\go{k})= \text{Aut}(\go{k})\cap \text{Aut}_{e}(\overline{\go{k}})$\index{Aut0@$\text{Aut}_{0}(\,.\,)$}. The elements of  $\text{Aut}_{0}(\go{k})$ are the automorphisms of $\go{k}$ which become elementary  after extension from  $F$ to  $\overline{F}$. Again $\text{Aut}_{0}(\go{k})$ can be identified with  $\text{Aut}_{0}([\go{k},\go{k}])$. 
   Finally we have the following inclusions:
   $$\text{Aut}_{e}(\go{k})\subset \text{Aut}_{0}(\go{k})\subset\text{Aut}([\go{k},\go{k}])\subset \text{Aut}(\go{k}).$$
   From \cite{Bou2} (Chap. VIII, \S 8, $n^\circ 4$, Corollaire de la Proposition 6, p.145), $\text{Aut}_{0}(\go{k})$ is open and closed in  $\text{Aut}([\go{k},\go{k}])$, and  $\text{Aut}([\go{k},\go{k}])$ is closed in  $\text{End}([\go{k},\go{k}])$ for the Zariski topology (\cite{Bou2}, Chap. VIII, \S5, $n^\circ 4$, Prop. 8 p.111), therefore $\text{Aut}_{0}(\go{k})$ is an algebraic group. As  $\text{Aut}_{0}(\go{k})$ is   open in $\text{Aut}([\go{k},\go{k}])$, its Lie algebra is the same as the Lie algebra of  $\text{Aut}(\go{k})$, namely $[\go{k},\go{k}]$. We know also that   $\text{Aut}_{0}(\overline{\go{k}})=\text{Aut}_{e}(\overline{\go{k}})$  is the connected component of the neutral element of  $\text{Aut}([\overline{\go{k}},\overline{\go{k}}])$ (\cite{Bou2}, Chap. VIII,\S5, $n^\circ 5$, Prop. 11, p. 113). 
   
 \vskip 10pt
 
 Define first:
\vskip 10pt

 $\bullet $    $ \widetilde{G}={\rm Aut}_0( \widetilde{\go g}).$\index{Gtilde@$\widetilde{G}$}
\vskip 10pt

 The group  $G$ we consider here is the following  (this group was first introduced by Iris Muller in \cite{Mu97} and \cite{Mu98}):  
 
\vskip 10pt

 $\bullet $   $G= {\cal Z}_{\text{Aut}_{0}(\widetilde{\go{g}})}(H_{0})=\{g\in \text{Aut}_{0}(\widetilde{\go{g}}),\, g.H_{0}=H_{0}\}$ is the centralizer of $H_{0}$ in  $\text{Aut}_{0}(\widetilde{\go{g}})$. \index{G@$G$}
\vskip 10pt

 The Lie algebra of  $G$ is then ${\cal Z}_{ [\widetilde{\go{g}}, \widetilde{\go{g}}]}(H_{0})=\go{g}\cap [\widetilde{\go{g}}, \widetilde{\go{g}}]=[ \go{g},  \go{g}]+[V^+, V^-]\supset F H_{0}\oplus[\go{g},  \go{g}]$.

As the elements of  $G$ fix $H_{0}$, we obtain  that $V^+$ (which is the eigenspace of $\ad H_{0}$ for the eigenvalue $2$) is stable under the action of $G$. Of course the same is true for  $V^-$.  

\vskip 5pt
The representation  $(G, V^+)$ is a prehomogeneous vector space, or more precisely it is an $F$-form of a prehomogeneous vector space. This means just that this representation has a Zariski-open orbit.  In fact  $G$  can be viewed as  the $F$-points of an algebraic group  (also noted $G$). Then if  $\overline{G}= G(\overline{F})$ stands for the points over  $\overline{F}$ of $G$, $( \overline{G}, \overline{V^+})$ is a prehomogeneous vector space   from a well known result of Vinberg (\cite{Vinberg}). As $\overline{F}$ is algebraically closed, this prehomogeneous vector space has of course only one open orbit. It is well known that then  the representation   $(G, V^+)$ has only a finite number of open orbits. This is a consequence of a result of Serre (see \cite{Sa}, p.469, or \cite{Ig1},  p. 1).

\vskip 5pt \vskip 5pt


One has  $G(\overline{F})={\cal Z}_{\text{Aut}_{0}(\overline{\widetilde{\go{g}}})}(H_{0})={\cal Z}_{\text{Aut}_{e}(\overline{\widetilde{\go{g}}})}(H_{0})$. (Because over an algebraically closed field one has $\text{Aut}_{e}(\overline{\widetilde{\go{g}}})= \text{Aut}_{0}(\overline{\widetilde{\go{g}}})$). 
Moreover, as we mentioned before, $\text{Aut}_{0}(\widetilde{\go{g}})$ and $\text{Aut}_{0}(\overline{\widetilde{\go{g}}})$ are closed and are the connected components of the neutral element  in $\text{Aut}(\widetilde{\go{g}})$ and  $\text{Aut}(\overline{\widetilde{\go{g}}})$ respectively (for the Zariski topology of   $\text{End}([\widetilde{\go{g}},\widetilde{\go{g}}])$ and $\text{End}([\overline{\widetilde{\go{g}}},\overline{\widetilde{\go{g}}}])$. 

The Lie subalgebra $\go{t}=FH_{0}$ is algebraic. Let us denote by $T$ the corresponding one dimensional torus.
Then as $G(\overline{F})={\cal Z}_{\text{Aut}_{0}(\overline{\widetilde{\go{g}}})}(H_{0})= {\cal Z}_{\text{Aut}_{0}(\overline{\widetilde{\go{g}}})}(T)$, we obtain   that $G(\overline{F})$ is connected (\cite{Hum} Theorem 22.3 p.140).

\vskip 5pt \vskip 5pt

\begin{definition}\label{def-elementsgeneriques}  An element $X\in V^+$ is called generic \index{generic} if it satisfies one of the following equivalent conditions:

$(i)$ The $G$-orbit of $X$ in $V^+$ is open.

$(ii)$ $\ad(X): \go{g}\longrightarrow V^+$ is surjective.
\end{definition}
\vskip 5pt
\begin{lemme}\label{lem-sl2-generique} \hfill

Let $X\in V^+$. If there exists $Y\in V^-$ such that  $(Y,H_{0},X)$ is an $\go{sl}_{2}$-triple, then $X$ is generic in $V^+$.
\end{lemme}

\begin{proof}
For $v\in V^+$ one has  $2v=[H_{0},v]=\ad([Y,X])v=-\ad(X)\ad(Y)v$. Hence $\ad(X)$ is surjective from  $\go{g}$ onto   $V^+$.

\end{proof}

{\bf Note}: Although, in the preceding proof we used the commutativity of  $V^+$, the result is in fact true for any $\Z$-graded algebra.

\vskip 5pt

For $i=0,\dots,k$, let us  choose once and for all $X_{i}\in \widetilde{\go{g}}^{\lambda_{i}}$ and  $Y_{i}\in \widetilde{\go{g}}^{-\lambda_{i}}$, such that  $(Y_{i},H_{\lambda_{i}},X_{i}) $ is an  $\go{sl}_{2}$-triple. This is always possible  (see for example \cite{Seligman}, Corollary of Lemma 6, p.6, or \cite{Schoeneberg}, Proposition 3.1.9 p.23).

\begin{lemme}\label{Xk-generique}\hfill

The element $X_{k}$ is generic in $V_{k}^+ $.
\end{lemme}

\begin{proof} The Lie algebra $\widetilde{\go{g}}_{k}$ is graded by  $H_{k}=H_{0}-H_{\lambda_{0}}-\dots-H_{\lambda_{k-1}}$ and not by $H_{\lambda_{k}}$, therefore we cannot use Lemma \ref{lem-sl2-generique}.

We will show that  $\ad(X_{k}):\go{g}_{k}\longrightarrow V_{k}^+$ is surjective (Cf.  definition \ref{def-elementsgeneriques}). Let  $\lambda$ be a root such that $ \widetilde{\go{g}}^{\lambda}\subset V_{k}^+$ and let $z\in \widetilde{\go{g}}^{\lambda}$. Then 
$$\ad(X_{k})\ad(Y_{k})z=-\ad(H_{k})z=-\lambda(H_{k})z.$$
 As $\ad(Y_{k})z\in \go{g}_{k}$, it is enough to show that  $\lambda(H_{k})\neq0$. If $\lambda=\lambda_{k}$, then of course  $\lambda(H_{k})=2$. If $\lambda\neq\lambda_{k}$ and  $\lambda(H_{k})=0$, then $\lambda\perp \lambda_{k}$, and hence  $\lambda\sorth \lambda_{k}$, by Proposition \ref{racines-g1} 1). Then $\widetilde{\go{g}}^{\lambda}\subset V^+\cap {\cal Z}_{\widetilde{\go{g}}}(\widetilde{ {\go l}}_0\oplus \widetilde{ {\go l}}_1\oplus \dots\oplus \widetilde{ {\go l}}_{k})=\{0\}$. Contradiction.

\end{proof}

\begin{lemme}\label{recurrence-genericite}\hfill

If $X$ is generic in  $V^+_{j}$, then $X_{0}+X_{1}+\dots+X_{j-1}+X$ is generic in  $V^+$.
\end{lemme}

\begin{proof}By induction we must just prove the Lemma  for $j=1$. Let $X$ be be generic in  $V_{1}^+$. We will prove that  $[\go{g},X_{0}+X]=V^+$. 

$\bullet $ One has $[\go{g}_{1}, X_{0}+X]=[\go{g}_{1}, X]= V^+_{1}$, from the definition of  $\go{g}_{1}$ and because $X$ is generic in $V_{1}^+$.

$\bullet$ If $z\in {\widetilde{\go{g}}}^{\lambda_{0}}$, then $[Y_{0},z]\in \go{g}\cap  \widetilde{ {\go l}}_0$ and hence  $[Y_{0},z]$ commutes with $X\in V^+_{1}$. But then  $\ad(X_{0}+X)[Y_{0},z]=\ad(X_{0})[Y_{0},z]= [-H_{0},z]=-2z$.

$\bullet$   If  $z\in {\widetilde{\go{g}}}^{\lambda}$ with $\lambda\neq\lambda_{0}$ and  $z\notin V^+_{1}$, then $\mu=\lambda-\lambda_{0}$ is a root (as $\lambda$ is not strongly orthogonal to  $\lambda_{0}$) which is positive, hence in  $\Sigma^+$ (because  $\lambda=\lambda_{0}+\dots$). Therefore  $\mu-\lambda_{0}$ is not a root. As $\mu+\lambda_{0}$ is a root, one has $(\mu,\lambda_{0})<0$ (see \cite{Bou1} (Chap. VI, \S1, $n^\circ 3$, Prop. 9)).

Suppose first that  $[\go{g}^{\mu},V_{1}^+]\neq\{0\}$. Then there exists a root  $\nu\in \widetilde{\Sigma}^+$ such  $\widetilde{\go{g}}^{\nu}\subset V^+_{1}$ and  $\mu+\nu\in \widetilde{\Sigma}^+$. One would have   $(\mu+\nu, \lambda_{0})= (\mu,\lambda_{0})<0$. Hence  $\mu+\nu+\lambda_{0}$ would be a root, such that 
$$(\mu+\nu+\lambda_{0})(H_{0})=\mu(H_{0})+\nu(H_{0})+\lambda_{0}
(H_{0})=0+2+2=4$$
and this is impossible from  the hypothesis ${\bf (H_{1})}$

Therefore  $[\go{g}^{\mu},V_{1}^+]=\{0\}$ and then, as  $U=\ad( Y_{0})z\in \go{g}^{\mu}$,  we have  $\ad(X_{0}+X)U=\ad(X_{0})U=\ad(X_{0})\ad(Y_{0})z=-\ad(H_{0})z=-2z$.

\end{proof}

\vskip 5pt
 
The two preceding results imply:
 
 \vskip 5pt
 
 \begin{prop}\label{X0+...+Xkgenerique}
  An element of the form  $X_{0}+X_{1}+\dots+X_{k}$, where $X_{i}\in \widetilde{\go{g}}^{\lambda_{i}}\setminus \{0\}$, is generic in  $V^+$.
 \end{prop}

   \begin{prop}\label{prop-conjdeslambda(i)}\hfill
 
For $j=0,\dots,k$ let us denote by   $W_{j}$ the Weyl group of the pair  $(\go{g}_{j},\go{a}_{j})$   (where $(\go a_0, W_{0})=(\go a, W)$). Let $w_{j}$ be the unique element of  $W_{j}$ such that $w_{j}(\Sigma_{j}^+)=\Sigma_{j}^-$. Then $w_{j}(\lambda_{j})=\lambda^0$ and the roots  $\lambda_{j}$ are long roots in  $\widetilde{\Sigma}$.

\end{prop}

\begin{proof}\hfill

One has   $w_{0}\in W_{0}\subset{\cal Z}_{\text{Aut}_{e}(\widetilde{\go{g}})}(H_{0})\subset G={\cal Z}_{\text{Aut}_{0}(\widetilde{\go{g}})}(H_{0})$. Hence  $$w_{0}(\lambda_{0})(H_{0})=\lambda_{0}(w_{0}(H_{0}))=\lambda_{0}(H_{0})=2.$$

On the other hand, from Corollary  \ref{corlambda0}, one has $\lambda_{0}-\lambda\notin \widetilde{\Sigma}$ for $\lambda\in \Sigma^+$. Therefore $w_{0}(\lambda_{0})-w_{0}(\lambda)\notin \widetilde{\Sigma}$. But  $w_{0}(\lambda)$ takes all values in $-\Sigma^+$ when $\lambda$ varies in $  \Sigma^+$. Therefore  $w_{0}(\lambda_{0})+\lambda\notin \widetilde{\Sigma}$ if $\lambda\in \Sigma^+$.  Proposition \ref{prop.plusgranderacine} implies then that  $w_{0}(\lambda_{0})=\lambda^0$. It is easy to see that  $\lambda^0$ is also the highest root of  the root systems  $\widetilde{\Sigma}_{j}$.   As $\lambda_{j}$ is the analogue of  $\lambda_{0}$ in $\widetilde{\go{g}}_{j}$, we obtain that $w_{j}(\lambda_{j})=\lambda^0$. As the highest root  $\lambda^{0}$ is a long root (see \cite{Bou1}, Chap. VI,\S1 $n^\circ 1$, Proposition 25 p.165), the roots $\lambda_{j}$ are all long.

\end{proof}

\begin{prop}\label{prop-systememax}\hfill

$(1)$ The set  $\lambda_{0},\dots,\lambda_{k}$ is a maximal system of strongly orthogonal long roots in  $\widetilde{\Sigma}^+\setminus \Sigma^+  $.

$(2)$ If  $\beta_{0},\dots,\beta_{m}$ is another maximal system of strongly orthogonal long roots in $\widetilde{\Sigma}^+\setminus \Sigma^+  $ then  $m=k$ and there exists  $w\in W$ such that  $w(\beta_{j})=\lambda_{j}$ for  $j=0,\dots,k$.

\end{prop}

\begin{proof} The following proof is adapted from  Th\'eor\`eme 2.12 in \cite{M-R-S} which concerns the complex case.

$(1)$ We have already seen that the roots $\lambda_{j}$ are long and strongly orthogonal. Suppose that there exists a long root  $\lambda_{k+1}$which is strongly orthogonal to each root $\lambda_{j}$ ($j=0,\ldots,k$). Then $\widetilde{\go{g}}^{\lambda_{k+1}}$ would be included in ${\cal Z}_{V^+}(\widetilde{ {\go l}}_0\oplus \widetilde{ {\go l}}_1\oplus \dots\oplus \widetilde{ {\go l}}_{k})=\{0\}$ (Theorem \ref{th-descente}). Contradiction.

$(2)$ Set  $\Phi=\{\gamma\in \Sigma\,|\, \gamma-\beta_{0}\notin \widetilde{\Sigma}\}$. Let us show that $\Phi$ is a parabolic subset of  $\Sigma$.

- We first show that $\Phi$ is a closed subset. Let $\gamma_{1},\gamma_{2}\in \Phi$ such that  $\gamma_{1}+\gamma_{2}\in \Sigma$. Then $[\go{g}^{\gamma_{1}},\go{g}^{\gamma_{2}}]\neq\{0\}$.  Let  $X_{\gamma_{i}}\in \go{g}^{\gamma_{i}}\setminus \{0\}$ ($i=1,2$) such that  $X_{\gamma_{1}+\gamma_{2}}=[X_{\gamma_{1}},X_{\gamma_{2}}]\neq0$. The Jacobi identity implies that  $[X_{\gamma_{1}+\gamma_{2}},X_{-\beta_{0}}]=0$. Hence $\gamma_{1}+\gamma_{2}-\beta_{0}\notin \Phi$. And hence $\Phi$ is closed.

-  It remains to show that  $\Phi\cup(-\Phi)=\Sigma$. If this is not the case, it exists  $\gamma_{0}\in \Sigma$ such that  $\gamma_{0}-\beta_{0}\in \widetilde{\Sigma}$ and $\gamma_{0}+\beta_{0}\in \widetilde{\Sigma}$. Therefore the $\beta_{0}$-string of roots through $\gamma_{0}$ can be written:
$$\gamma_{0}-\beta_{0},\gamma_{0},\gamma_{0}+\beta_{0}$$
(remember that  $V^\pm$ are commutative). From \cite{Bou1}(chap. VI, \S1, $n^\circ 3$, Corollaire de la proposition 9 p.149) one has
$$n(\gamma_{0}-\beta_{0}, \beta_{0})=-2 \eqno{(*)}.$$

But from  (\cite{Bou1} Chap. VI, p.148) this is only  possible, as $\beta_{0}$ is long, if  $\gamma_{0}-\beta_{0}=-\beta_{0}$, in other words if $\gamma_{0}=0$. Contradiction. Hence $\Phi$ is parabolic in $\Sigma$

Then from \cite{Bou1} (Chap. VI, \S 1, $n^\circ 7$, Prop. 20, p.161) there exists a basis $\Pi'$ of  $\Sigma$ such that  $\Pi'\subset \Phi$. Hence it exists  $w_{0}\in W$ such that  $w_{0}(\Pi')= \Pi$. Then  $w_{0}(\beta_{0})(H_{0})=\beta_{0}(w_{0}(H_{0}))=\beta_{0}(H_{0})=2$. One also has $\lambda-w_{0}(\beta_{0})\notin \widetilde{\Sigma}$ for  $\lambda\in \Sigma^+$. If one would have $\lambda-w_{0}(\beta_{0})\in \widetilde{\Sigma}$, then as  $\lambda=w_{0}\lambda'$ (with $\lambda'$ positive for $\Pi'$), then $\omega(\lambda')-w_{0}(\beta_{0})\in \widetilde{\Sigma}$, and hence $\lambda'-\beta_{0}\in \widetilde{\Sigma}$, this is impossible from the definition of  $\Phi$.
  But we know from Corollary \ref{corlambda0} that these properties characterize   $\lambda_{0}$. Hence $w_{0}(\beta_{0})=\lambda_{0}$. 

Suppose first  $k=0$. In this case the set  $\lambda_{0}=w_{0}(\beta_{0}), w_{0}(\beta_{1}),\dots,w_{0}(\beta_{m})$ is a set of strongly orthogonal roots. Hence $w_{0}(\beta_{1}),\dots,w_{0}(\beta_{m})\in {\cal Z}_{V^+}(\widetilde{ {\go l}}_0)=\{0\} $  This means that if  $k=0$ then $m=0$  and the assertion $(2)$ is proved in that case.

The general case goes by induction on $k$. Suppose that the result is true when the rank of the graded algebra is  $<k$. In view of the above,  there exists  $w_{0}\in W$ such that $w_{0}(\beta_{0})=\lambda_{0}$. Then $w_{0}(\beta_{1}),\dots,w(\beta_{m})$ is a maximal system of strongly orthogonal long roots in $\widetilde{\Sigma}_{1}^+\setminus \Sigma_{1}^+  $. As the graded algebra $\widetilde{\go{g}}_{1}$ is of rank  $k-1$, we have $m=k $ by induction  and there exists $w_{1}\in W_{1}$ ($W_{1}\subset W$ is the Weyl group of $\Sigma_{1}$) such that $w_{1}(w_{0}(\beta_{i}))=\lambda_{i} $ for $i=1,\dots,k$. The assertion $(2)$ is then proved with  $w=w_{0}w_{1}.$

\end{proof}

 \begin{cor} \label{cor-structureG}{\rm (see \cite{Mu98} Lemme 2.1. p. 166, for the regular case defined below)} \hfill
 
 Let  $\lambda_{0},\lambda_{1},\dots,\lambda_{k}$ be the maximal system of strongly orthogonal long roots obtained from the descent. Let $H_{\lambda_{0}}, H_{\lambda_{1}},\dots,H_{\lambda_{k}}$ be the corresponding co-roots. For $i=0,\dots,k$ we denote by  $G_{H_{\lambda_{i}}}$ the stabilizer of $H_{\lambda_{i}}$ in $G$. Then we have :
 
 $$G= {\rm Aut}_{e}(\go{g}).(\bigcap_{i=0}^k G_{H_{\lambda_{i}}}).$$
  \end{cor}
 \begin{proof}\hfill
 
 Let $g\in G$. The elements  $g.H_{\lambda_{i}}$ belong to a maximal split torus  $\go{a}'$. As the maximal split tori of $\go{g}$ are conjugated under $\text{Aut}_{e}(\go{g})\subset G$ (\cite{Seligman}, Theorem 2, page 27, or \cite{Schoeneberg}, Theorem 3.1.16 p. 27) there exists  $h\in \text{Aut}_{e}(\go{g})$ such that $hg(\go{a})=\go{a}$, hence  $hg\in \text{Aut}(\go{g}, \go{a})$, the group of automorphisms of  $\go{g}$ stabilizing  $\go{a}$. But then $hg$ acts on  $\widetilde{\Sigma}$ and  $hg({\lambda_{0}}),\dots,hg(\lambda_{k})$ is a sequence of strongly orthogonal long roots in  $\widetilde{\Sigma}^+\setminus \Sigma^+  $. From the preceding proposition  \ref{prop-systememax}  , there exists  $w\in W\subset \text{Aut}_{e}(\go{g})$ such that  $whg({\lambda_{i}})=\lambda_{i}$.
Hence $whg\in \bigcap_{i=0}^k G_{H_{\lambda_{i}}}$. It follows that  $g=h^{-1}w^{-1}whg\in {\rm Aut}_{e}(\go{g}).(\bigcap_{i=0}^k G_{H_{\lambda_{i}}})$.

 \end{proof}
  \begin{definition}\label{def-a^0-L}
  We set:
   $$\go{a}^{0}=\oplus_{j=0}^k F H_{\lambda_{j}}\subset \go{a},$$ \index{agothique0@$\go{a}^{0}$}
   and
   $$L=Z_{G}(\go{a}^{0})=\bigcap_{i=0}^k G_{H_{\lambda_{i}}}.$$ \index{L@$L$}
   
 Hence, from the preceding Corollary, we have $G= {\rm Aut}_{e}(\go{g}).L$.

\end{definition}
 \begin{rem} \label{rem-inclusion-groupes}  Let  $j\in\{0,\ldots,k\}$. Let us denote by   $G_j$ the analogue of the group $G$ for the Lie algebra  $\tilde{\go g}_j$ (hence $G_0=G$). As $\overline{G_{j}} \subset {\rm Aut}_e(\tilde{\go g}_j\otimes \overline{F})\subset {\rm Aut}_e(\tilde{\go g}\otimes \overline{F})$, any element $g$ of  $G_j$  extends to  an elementary automorphism of $ \tilde{\go g}\otimes\overline{F}$, which we denote by $ext(g)$ and which acts trivially on $\oplus_{s=0}^{j-1}\tilde{\go l}_s$. Therefore  $ext(g)$ centralizes $H_{0}$ and one has:
$$ {\rm Aut}_e(\go g_j)\subset G_j\subset\overline{G_{j}} \subset \overline{G} ={\mathcal Z}_{{\rm Aut}_e(\tilde{\go g}\otimes \overline{F})}(H_0).$$ 
However, it may happen that, for  $g\in G_j$, the automorphism  $ext(g)$ does not stabilize  $\tilde{\go g}$ and then   $ext (g)$   does not define an automorphism of  $\tilde{\go g}$. For example  (see the proof of  Theorem \ref{thm-orbites-e1}), when   $\tilde{\go g}$ is the symplectic algebra $ {\go sp}(2n,F)$, graded by $H_0=\left(\begin{array}{cc} I_n & 0\\ 0 & -I_n\end{array}\right)$ where  $I_n$ is the identity matrix of size $n$, the group $G$ is the group of elements  ${\rm Ad}(g)$ for  $g=\left(\begin{array}{cc} \mathbf g & 0\\ 0 & \mu\;^{t}{\mathbf g}^{-1}\end{array}\right)$ where $\mathbf g\in GL(n,F)$ and  $\mu\in F^*$ and the group  $G_j$, as a subgroup of  $\overline{G}$, is the subgroup of elements of the form  ${\rm Ad}(g_j)$ where $$g_j=\left(\begin{array}{c|c} \begin{array}{cc}\mathbf g_j & 0\\ 0 & I_{j}\end{array} & 0\\
\hline
 0 &  \begin{array}{cc}\mu\; ^{t}{\mathbf g}_j^{-1} & 0\\ 0 & I_{j}\end{array} \end{array}\right)$$ with $\mathbf g_j\in GL(n-j,F)$ and  $\mu\in F^*$.\\
This shows that $G_j$ is not always included in   $G$. 
\end{rem} 
 \begin{definition}
 \label{def-regulier} A reductive graded Lie algebra $\widetilde{\go g}$ which verifies condition ${\bf (H_{1})}$ and ${\bf (H_{2})}$ is called  {\bf regular} \index{regular} if furthermore it satisfies: \vskip 3pt
  \hskip 15pt  $\bf{(H_{3})}$ There exist $I^+\in V^+$ and $I^-\in V^-$ such that $(I^-,H_{0},I^+)$ is an $\go{sl}_{2}$-triple.  
 \end{definition}
 
  \vskip 5pt
  
  \begin{prop}\label{prop-generiques-cas-regulier}\hfill
  
  In a regular graded Lie algebra $\widetilde{\go g}$, an element  $X\in V^+$ is generic  if and only if it exists $Y\in V^-$, such that  $(Y,H_{0},X)$ is an $\go{sl}_{2}$-triple. Moreover, for a fixed generic element $X$ , the element $Y$   is unique.
  \end{prop}
  
  \begin{proof}
  
  If $X$ can be put in such an  $\go{sl}_{2}$-triple, then $X$ is generic in $V^+$ from Lemma  \ref{lem-sl2-generique}.
  
  Conversely, if $X$ is generic in  $V^+$, then $\ad(X): \go{g}\longrightarrow V^+$ is surjective (cf. Definition  \ref{def-elementsgeneriques}), hence $\ad(X): \overline{\go{g}}\longrightarrow \overline{V^+}$ is also surjective, therefore the  $\overline{G}$-orbit of  $X$ in $\overline{V^+}$ is open  and  $X$ is generic in   $\overline{V^+}$. But there is only one open orbit in  $\overline{V^+}$. Therefore $X$ is in the $\overline{G}$-orbit of  $I^+$ (Definition \ref{def-regulier}). Hence there exists $g\in \overline{G}$ such that  $g.I^+=X$. But then $(g.I^-, g.H_{0}=H_{0}, g.I^+=X)$ is an $\go{sl}_{2}$-triple. A standard tensor product argument (write
  $g.I^-=Y$ under the form $Y=\sum_{i=1}^n a_{i}Y_{i}$, with $a_{1}=1$, $a_{2},\dots,a_{n}$ elements of  $\overline{F}$ free over sur $F$, $Y_{i}\in V^-$), shows that $Y\in V^-$. Uniqueness is classical,(\cite{Bou2}, Chap VII, \S 11, $n^\circ 1$, Lemme 1).
  \end{proof}
  \vskip 10pt
  {\bf From now on we will always suppose that the graded Lie algebra $(\widetilde{\go{g}},H_{0})$ is regular.}
  
   \vskip 20pt
 \subsection{Structure of the regular graded Lie algebra $(\widetilde{\go{g}},H_{0})$}\hfill
  \vskip 10pt
  
  Let $\lambda_{0},\lambda_{1},\dots,\lambda_{k}$ be the sequence of strongly orthogonal  roots defined in Theorem \ref{th-descente}. Let $H_{\lambda_{0}},H_{\lambda_{1}},\dots,H_{\lambda_{k}}$ be the sequence of the corresponding co-roots.
  
 Remember that we have defined:
  $$\go{a}^{0}=\oplus_{j=0}^k F H_{\lambda_{j}}\subset \go{a}.$$

For  $i,j\in
\{0,1,\ldots ,k\}$ and 
$p,q\in {\bb Z}$ we define the subspaces $E_{i,j}(p,q)$ \label{Eij} of
$\widetilde{ {\go g}}$  by setting:
$$ E_{i,j}(p,q)=\left\{ X\in \widetilde{ {\go g}}\ \Bigl|\   [H_{\lambda
_\ell},X]=
\begin{cases}
pX&\hbox{  if }\ell=i\ ;\cr
qX&\hbox{ if }\ell=j\ ;\cr
0&\hbox{ if }\ell\notin \{i,j\}\ .
\end{cases}\right\}.
$$ \index{Eij@$E_{i,j}(p,q)$}
 \vskip 5pt
\begin{theorem}\label{th-decomp-Eij}\hfill

If $(\widetilde{\go{g}}, H_{0})$ is regular then 
$$H_{0}=H_{\lambda_{0}}+H_{\lambda_{1}}+\dots+H_{\lambda_{k}}.$$
Moreover one has the following decompositions:
\begin{align*}
(1) \hskip 25pt {\go g}&={\cal Z}_{\go g}({\go a}^0)\oplus\bigl(\oplus _{i\not =
j}E_{i,j}(1,-1)\bigr)\ ;\cr 
(2) \hskip 15pt V^+&=\bigl(\oplus_{j=0}^k \widetilde{ {\go g}}^{\lambda _j}\bigr)
\oplus\bigl(\oplus _{i<j}E_{i,j}(1,1)\bigr)\ ;\cr
(3) \hskip 15pt V^-&=\bigl(\oplus_{j=0}^k \widetilde{ {\go g}}^{-\lambda _j}\bigr) \oplus\bigl(\oplus
_{i<j}E_{i,j}(-1,-1)\bigr)\ .
\end{align*}
\end{theorem}
\vskip 5pt
\begin{proof}
For  $j=\{0,1,\dots,k\}$, we choose $X_{j}\in  \widetilde{ {\go g}}^{\lambda _j}\setminus \{0\}$. From Proposition \ref{X0+...+Xkgenerique} the element $X=X_{0}+X_{1}+\dots+X_{k} $ is generic. Therefore  $X$ can be put in an  $\go{sl}_{2}$-triple  of the form  $(Y,H_{0},X)$ with  $Y\in V^-$ (Proposition \ref{prop-generiques-cas-regulier}). We choose also  $X_{-j}\in  \widetilde{ {\go g}}^{-\lambda _j}$ such that  $(X_{-j},H_{\lambda_{j}}, X_{j})$ is an $\go{sl}_{2}$-triple, and we set $Y'=X_{-0}+X_{-1}+\dots+X_{-k}$. 

Then  $(Y', H_{\lambda_{0}}+H_{\lambda_{1}}+\dots+H_{\lambda_{k}},X)$ is again an  $\go{sl}_{2}$-triple. On the other hand
$$\ad(X)^2: V^-\longrightarrow V^+$$
is injective  and as  $\ad(X)^2Y=2X= \ad(X)^2Y'$, we obtain that $Y=Y'$. But then $H_{0}=\ad(X)Y=\ad(X)Y'=H_{\lambda_{0}}+H_{\lambda_{1}}+\dots+H_{\lambda_{k}}$. The first assertion is proved.
\vskip 10pt
Let now  $X$ be an element  of an eigenspace of  $\ad(\go{a})$, i.e. either an element of a root space of  $\widetilde{\go{g}}$ or an element of the centralizer  $\go{m}$ of $\go{a}$ in $\widetilde{\go{g}}$ (cf. Remark \ref{rem-precisions}). The representation theory of  $\go{sl}_{2}$ implies the existence, for all $j=0,1,\dots,k$, of an integer  $p_{j}\in \Z$ such that $[H_{\lambda_{j}},X]=p_{j}X$. As 
$H_{0}=H_{\lambda_{0}}+H_{\lambda_{1}}+\dots+H_{\lambda_{k}}$ one has 
$$p_0+p_1+\cdots +p_k=
\begin{cases}
2&\hbox{ if   } X\in V^+\ ;\cr
0&\hbox{  if } X\in {\go g}\ ;\hfill\cr
-2&\hbox{  if }X\in V^-\ .
\end{cases}$$

Define $w_j=e^{\ad X_j}e^{\ad X_{-j}}e^{\ad X_j}$. Hence $w_{j}$ is the unique non trivial element of the Weyl group of the Lie algebra isomorphic  to  $\go{sl}_{2}$ generated by the triple $(X_{-j},H_{\lambda_{{j}}}, X_{j})$. As the  $\lambda_{j}$ are strongly orthogonal  the elements   $w_{j}$ commute and  
$$w_j.H_{\lambda _i}=
\begin{cases}
H_{\lambda _i}\hbox{  for }i\not = j\ ;\cr
-H_{\lambda _j}\hbox{  for }i = j\ .
\end{cases}$$
Let $J\subset \{p_{0},p_{1},\dots,p_{k}\}$ be the subset of $i$'s such that $p_{i}<0$. If we set $w=\prod_{i\in J}w_{i}$, one obtains that the sequence of eigenvalues of $H_{\lambda_{j}}$ on $wX$ is  $|p_{0}|, |p_{1}|, \dots,|p_{k}|$. Hence  $|p_{0}|+|p_{1}|+\dots+|p_{k}|=0 \text{ or }2$. 

The case  $|p_{0}|+|p_{1}|+\dots+|p_{k}|=0$  occurs if and only if $X\in {\cal Z}_{\go g}({\go a}^0)$.  

If   $|p_{0}|+|p_{1}|+\dots+|p_{k}|=2$ then all  $p_{j}$ are zero, except for one which takes the value $\pm2$, or for two among them which take the value $\pm1$.  

 In the first case or if the two nonzero $p_i$'s are equal then $X\in V^+$ or $X\in V^-$.
Otherwise one  $p_{i}$ equals $1$ and the other $-1$,  and then $X\in \go{g}$.

To obtain the announced decompositions it remains to prove that
 $$E_{i,j}(0,2)=\widetilde{ {\go
g}}^{\lambda _j} \text{ and }  E_{i,j}(0,-2)=\widetilde{ {\go
g}}^{-\lambda _j} \text{ for } i\not =
j\ .$$

A root space  $\widetilde{ {\go g}}^{\lambda } $ occurs in  $E_{i,j}(0,2)$ if $\lambda(H_{j})=2$ and  $\lambda(H_{\ell})=0$ for $\ell\neq j$, and this means that  $\lambda\perp \lambda_{\ell}$ if $\ell\neq j$. As $(\lambda+\lambda_{\ell})(H_{0})=4$,   $\lambda+\lambda_{\ell}$ is not a root. If $\lambda-\lambda_{\ell}$ is a root then  $(\lambda,\lambda_{\ell})\neq 0$  (\cite{Bou1} (Chap. VI, \S1, $n^\circ 3$, Prop. 9)). Hence $\lambda\sorth \lambda_{\ell}$ for $\ell \neq j$ and $[\widetilde{ {\go g}}^{\lambda },\widetilde{ {\go g}}^{\pm \lambda_{\ell} }]=0  $ for $\ell \neq j$. In particular 
$\widetilde{ {\go g}}^{\lambda } \in {\cal Z}_{\widetilde{\go{g}}}(\widetilde{ {\go l}}_0\oplus \widetilde{ {\go l}}_1\oplus \dots\oplus \widetilde{ {\go l}}_{j-1})$, and hence  $\lambda\in \widetilde{\Sigma}_{j}$. As $\lambda\geq 0$, one has  $\lambda\in \widetilde{\Sigma}_{j}^+$. But  $\lambda_{j}$ is a simple root in $ \widetilde{\Sigma}_{j}^+$, and therefore $\lambda-\lambda_{j}$ is not a root. If $\lambda\neq\lambda_{j}$, the equality $\lambda(H_{\lambda_{j}})=2$ would imply that  $\lambda-\lambda_{j}$ is a root.
Hence $\lambda=\lambda_{j}$ and  $E_{i,j}(0,2)=\widetilde{ {\go
g}}^{\lambda _j}$. The same proof shows that $E_{i,j}(0,-2)=\widetilde{ {\go
g}}^{-\lambda _j}$.

\end{proof}
\vskip 5pt
\begin{cor}\label{cor-orth=fortementorth}\hfill

Let  $\lambda\in \widetilde{\Sigma}$. Then for $j=0,1,\dots,k$, one has :
$$\lambda\perp \lambda_{j}\Longleftrightarrow \lambda\sorth \lambda_{j}.$$
\end{cor}

\begin{proof} \hfill

Let  $\lambda\perp \lambda_{j}$.

If $\lambda(H_{0})=2$, then $\lambda+\lambda_{j}$ is not a root, and if  $\lambda-\lambda_{j}$ is a root  one would have $(\lambda,\lambda_{j})\neq 0$, see  \cite{Bou1} (Chap. VI, \S1, $n^\circ 3$, Prop. 9)). Therefore  $\lambda\sorth \lambda_{j}$.

If  $\lambda(H_{0})=-2$, the same proof shows that  $\lambda\sorth \lambda_{j}$.

If  $\lambda(H_{0})=0$, then either  $\go{g}^\lambda\subset {\cal Z}_{\go{g}}(\go{a}^0)$ or  $\go{g}^\lambda \subset  E_{r,s}(1,-1)$ for  $r\neq j$, $s\neq j$.

\hskip 15pt a) If  $\go{g}^\lambda\subset {\cal Z}_{\go{g}}(\go{a}^0)$, then $(\lambda+\lambda_{j})(H_{\lambda_{j}})=2$ and if  $\lambda+\lambda_{j}$ is a root the preceding Theorem  \ref{th-decomp-Eij} says that  $\lambda+\lambda_{j}=\lambda_{j}$, this is not possible. Hence  $\lambda+\lambda_{j}$ is not a root , and the same argument as before shows that $\lambda\sorth \lambda_{0}$.

\hskip 15pt b) If  $\go{g}^\lambda \subset  E_{r,s}(1,-1)$ and if  $\lambda+\lambda_{j}$ is a root, then $(\lambda+\lambda_{j})(H_{\lambda_{j}})=2$, $(\lambda+\lambda_{j})(H_{\lambda_{r}})=1$, and  $(\lambda+\lambda_{j})(H_{\lambda_{s}})=-1$, which is impossible. Again the same argument as before shows that $\lambda\sorth \lambda_{0}$.
\end{proof}
\vskip 5pt

\begin{rem}\label{rem-avantages-Eij} The decomposition of  $\widetilde{\go{g}}$ using the subspaces  $E_{i,j}(p,q)$ will be more useful that the root spaces decomposition. The bracket between two such spaces can be easily computed using the Jacobi identity. We will also show   that this decomposition is also  a ``root space decomposition`` with respect to another system of roots than  $\widetilde{\Sigma}$ (see Remark \ref{rem-decomp-racines} and Proposition \ref{prop-inclusionSP}). 

 \end{rem}
\vskip 5pt
The first part of Therorem \ref{th-decomp-Eij} shows that the grading of $\widetilde{\go{g}}_{j}$ is defined by  $H_{\lambda_{j}}+\dots+H_{\lambda_{k}}$. By setting  $I^+_{j}=X_{j}+\dots+X_{k}$, $I^-_{j}=X_{-j}+\dots+X_{-k}$, where the elements  $X_{\pm\ell}\in\widetilde{\go{g}}^{\pm \lambda_{\ell}}$ are chosen such that  $(X_{-\ell},H_{\lambda_{\ell}}, X_{\ell})$ is an ${\go{sl}}_{2}$-triple, one obtains an  ${\go{sl}}_{2}$-triple $(I^-_{j}, H_{\lambda_{j}}+\dots+H_{\lambda_{k}}, I^+_{j})$. Theorem \ref{th-descente} implies then the following decomposition of   $\widetilde{\go{g}}_{j}$.

\begin{cor}\label{cor-decomp-j}\hfill

 For $j=0,\ldots ,k$, the graded algebra  $(\widetilde{ {\go
g}}_j,H_{\lambda _j}+\cdots +H_{\lambda _k})$ satisfies the hypothesis $(\bf H_1)$, $(\bf H_2)$ and $(\bf H_3)$.
One also has the following decompositions:
 
\begin{align*}
(1)& \hskip 20pt {\go g}_j=\Bigl({\go z}_{\go g}({\go a}^0)\cap {\widetilde{\go g}}_j\Bigr)\oplus\Bigl(\oplus
_{r\not = s;j\leq r;j\leq s }E_{r,s}(1,-1)\Bigr)\ ;\cr 
(2)& \hskip 15pt V^+_j=\Bigl(\oplus_{s=j}^k \widetilde{ {\go g}}^{\lambda _s}\Bigr)
\oplus\Bigl(\oplus _{j\leq r<s}E_{r,s}(1,1)\Bigr)\ ;\cr
(3)& \hskip 15pt V^-_j=\Bigl(\oplus_{s=j}^k \widetilde{ {\go g}}^{-\lambda _s}\Bigr) \oplus\Bigl(\oplus
_{j\leq r<s}E_{r,s}(-1,-1)\Bigr)\ .
\end{align*}

\end{cor}

\vskip 5pt

One can also extend the preceding decomposition to a subset $A$ which is different from  $\{j,j+1,\dots,k\}$:
\vskip 5pt
\begin{cor}\label{cor-gA}\hfill

Let $A$ be a non empty subset of  $\{0,1,\dots,k\}$. Define  $H_{A}=\sum_{j\in A}H_{\lambda_{j}}$ and set:
\begin{align*}
{\go g}_A&=\{X\in {\go g}\mid [H_A,X]=0\}\ ;\cr
V^+_A&=\{X\in V^+\mid [H_A,X]=2X\}\ ;\cr
V^-_A&=\{X\in V^{-}\mid [H_A,X]=-2X\}\ .
\end{align*}
Then the graded algebra   $(\widetilde{ {\go g}}_A=V^-_A\oplus{\go g}_A\oplus V^+_A, H_A)$ 
 satisfies  $(\bf H_1)$, $(\bf H_2)$ and  $(\bf H_3)$.
\end{cor}

\begin{proof}  Note that if  $A^c$ is the complementary set of $A$ in  $\{0,1,\dots,k\}$ and if  $H_{A^c}=\sum_{j\in A^c}H_{\lambda_{j}}$,  then  $\widetilde{ {\go g}}_A= {\cal Z}_{\widetilde{ {\go g}}}(H_{A^c})$. This implies that  $\widetilde{ {\go g}}_A$ is reductive (\cite{Bou2},chap.VII, \S 1, $n^\circ 5$, Prop.13). The hypothesis  $(\bf H_1)$ is then clearly verified.  The regularity condition  $(\bf H_3)$ can be proved in the same way as for  $\widetilde{\go{g}}_{j}$.

It remains to show  $(\bf H_2)$, that is that the representation $({\go g}_A, \overline{V^+_A})$ (or $(\overline{{\go g}_A}, \overline{V^+_A})$) is irreducible.  Let  $U$ be a subspace of $ \overline{V^+_A}$ which is invariant by ${\go g}_A$. From Theorem  \ref{th-decomp-Eij} the algebra  ${\go g}$ decomposes as follows:
$${\go g}={\go g}(-1)\oplus {\go g}_A \oplus {\go g}(1)\ ,$$
where  ${\go g}(-1)=\oplus _{i\notin A,j\in A} E_{i,j}(1,-1)$ and   ${\go g}(1)=\oplus _{i\in A
,j\notin A} E_{i,j}(1,-1)$. Remark also that ${\go g}(-1)=\{X\in \go{g}\,,\,[H_{A},X]=-X\}$   and  ${\go g}(1)=\{X\in \go{g}\,,\,[H_{A},X]=X\}$.

\vskip 5pt 
Note also  that $U$ is included in the eigenspace of $\ad(H_{A})$ for the eigenvalue  $2$.
\vskip 3pt
\hskip 15pt $\bullet$ $U_{1}=[\go{g}(-1),U] $ is then included in the eigenspace of $\ad(H_{A})$ for the eigenvalue  $1$.
\vskip 3pt
\hskip 15pt $\bullet$  $U_{0}=[\go{g}(-1),U_{1}] $ is then included in the eigenspace of  $\ad(H_{A})$ for the eigenvalue  $0$.
\vskip 3pt
\hskip 15pt   $\bullet$  One also has $[\go{g}(1),U]=\{0\}$.

\vskip 3pt

We will now show that  $U_{0}\oplus U_{1}\oplus U$ is  $\go{g}$-invariant  in  $ \overline{V^+}$.
\vskip 3pt
\hskip 10pt a) One has : $[\go{g}_{A},U]\subset U$, $[\go{g}(1),U]=\{0\}$, $[\go{g}(-1),U]=U_{1}$. Hence $[\go{g},U]\subset U_{1}\oplus U$.
\vskip 3pt
One also shows easily that:

\hskip 10pt b) $[\go{g},U_{1}]\subset U_{0}\oplus U_{1}\oplus U$,
\vskip 3pt
\hskip 10pt c) $[\go{g},U_{0}]\subset U_{0}\oplus U_{1}\oplus U$.

Therefore, using the hypothesis  ${\bf (H_{2})}$ for $\widetilde{ {\go g}}$, $U_{0}\oplus U_{1}\oplus U= \overline{V^+}$. But $U$ is contained in the eigenspace of  $\ad(H_{A})$ for the eigenvalue $2$, namely $\overline{V^+_A}$, and $U_{1}$ (resp.  $U_{0}$) corresponds to the eigenvalue $1$ (resp. $2$). This implies $U=\overline{V^+_A}$.

\end{proof}
\vskip 5pt

\begin{rem}\label{remgjdifferent}
Define $A_{j}=\{j,j+1,\dots,k\}$. Then the reductive algebras $\widetilde{\go{g}}_{j}$ and $\widetilde{ {\go g}}_{A_{j}}$ are graded by the same element $H_{\lambda_{j}}+\dots+H_{\lambda_{k}}$. But these algebras are not equal. The obvious inclusion $\widetilde{\go{g}}_{j}\subset \widetilde{ {\go g}}_{A_{j}}$ is strict as $H_{\lambda_{0}}\in {\go g}_{A_{j}}\setminus  {\go g}_j$. More precisely one has
\begin{align*}
&V^+_{A_{j}}=\widetilde{ {\go g}}_{A_{j}}\bigcap V^+=\Bigl(\oplus_{s=j}^k \widetilde{ {\go g}}^{\lambda _s}\Bigr)
\oplus\Bigl(\oplus _{j\leq r<s}E_{r,s}(1,1)\Bigr)
=V^+_j\ ;\cr
&{\go g}_{A_{j}}={\cal Z}_{\go g}({\go a}^0)\oplus\Bigl(\oplus _{r \not = s;j\leq r ;j\leq s
}E_{r ,s}(1,-1)\bigr)\oplus\bigl(\oplus _{r \not = s;r <j;s<j} E_{r ,s}(1,-1)\Bigr)\cr
&\supset_{\text{strict}} {\go g}_j={\cal Z}_{\go g}({\go a}^0)\cap {\go g}_j \oplus\Bigl(\oplus _{r \not = s;j\leq r ;j\leq s
}E_{r ,s}(1,-1)\bigr)\ .
\end{align*}

\end{rem}
\vskip 5pt

\begin{prop}\label{prop-lie(V+V-)simple}\hfill

The Lie algebra  $\widetilde{\go{G}}=V^-\oplus [V^-,V^+]\oplus V^+$ \index{Ggothique@$\widetilde{\go{G}}$} generated by $V^+$ and $V^-$ is a regular graded algebra  which is an absolutely simple ideal of $\widetilde{\go{g}}$.
\end{prop}
\begin{proof} As  $\widetilde{\go{g}}$ is regular, the element  $H_{0}\in [V^-,V^+]$, hence  $\bf(H_{1})$ is satisfied. The hypothesis  $\bf(H_{3})$ holds also clearly.

One easily verifies that  $\widetilde{\go{G}}$ is an ideal of  $\widetilde{\go{g}}$. Therefore $[V^-,V^+]$ is the only part of  $\go{g}$ which acts effectively on $V^+$. Therefore the representation $([V^-,V^+],V^+)$ is absolutely simple. In other words  $\bf(H_{2})$ is true.  
 
 The fact that $\widetilde{\go{G}}$ is simple is then a consequence of  Lemma \ref{lemme-alg.simple}.

\end{proof}

\begin{rem}\label{rem-simple} From the preceding Proposition   \ref{prop-lie(V+V-)simple}, we obtain that 
$\tilde{\go g}=\tilde{\go G}\oplus  {\go G}'$ where the subalgebra  $\tilde{\go G}=V^-\oplus [V^-, V^+]\oplus V^+$ is an abolutely simple graded  Lie algebra and where  ${\go G}'$ is the orthogonal of  $\tilde{\go G}$ in  $\tilde{\go g}$ with respect to the form $\tilde{B}$. Moreover, the subalgebra  $\tilde{\go G}$ is an ideal of  $\tilde{\go g}$ and hence  ${\go G}'$ is an ideal of  $\tilde{\go g}$ too. Therefore, if  $X\in {\go G}'$ is nilpotent over an algebraic closure of  $F$,  then  $e^{\ad X}$ acts trivially on  $\tilde{\go G}$. \\
Hence,  in order to classify the orbits of  $G$ in  $V^+$, one can suppose that  $\tilde{\go g}$ is simple. 
\end{rem}

  \vskip 20pt
 \subsection{Properties of the spaces  $E_{i,j}(p,q)$}\hfill
  \vskip 10pt
  
  \begin{prop}\label{proplambda>0} \hfill
  
  Let $\lambda\in \Sigma$ such that $\go{g}^{\lambda}\subset E_{i,j}(1,-1)$. Then $\lambda$ is positive if and only if  $i>j$.
  
  \end{prop}
  
  \begin{proof} Let $\lambda\in \Sigma$ such that  $\lambda(H_{\lambda_{i}})=1$ and  $\lambda(H_{\lambda_{j}})=-1$. From Theorem  \ref{th-decomp-Eij}, $\lambda(H_{\lambda_{s}})=0$ if $s\neq i$ and  $s\neq j$.
  
 If $i> j$, then $\lambda\perp \lambda_{s}$ for $s< j$. But by Corollary \ref{cor-orth=fortementorth}, one has  $\lambda\sorth \lambda_{s}$, for $s< j$. Hence  $\lambda\in \widetilde{\Sigma}_{j}$. But as  $(\lambda,\lambda_{j})<0$ ($\lambda(H_{\lambda_{j}})=-1$),  $\lambda+\lambda_{j}$ is a root. We have also  $\lambda\sorth \lambda_{s}$ for  $s< j$ and    $\lambda_{j}\sorth \lambda_{s}$ for $s< j$. Hence $\lambda+\lambda_{j}\in \widetilde{\Sigma}_{j}$. As $\lambda\in \Sigma_{j}$ and as $\lambda_{j}\in \widetilde{\Pi}_{j}\setminus \Pi_{j}$ (Notation \ref{notationsgj}),we obtain that $\lambda\in \Sigma_{j}^+$. But then $\lambda\in {\Sigma}^+$.
  
 Conversely suppose that  $\lambda\in \Sigma$ and  $\go{g}^{\lambda}\subset E_{i,j}(1,-1)$ with $i< j$. Then  $\go{g}^{-\lambda}\in E_{j,i}(1,-1)$, and from above $-\lambda\in \Sigma^+$, hence $\lambda\in \Sigma^-$. Hence if  $\lambda$ is positive and  $\go{g}^\lambda\in E_{i,j}(1,-1)$, then $i>j$.  
  \end{proof}
  
\begin{prop}\label{propWconjugues}\hfill

           Let  $s_0$  be the unique element of   $W$ sending $\Sigma ^+$ on  $\Sigma ^-$.  Then
           $$s_0.\lambda _j=\lambda _{k-j} \hbox{  for  }j=0,1,\ldots ,k\ .$$
Moreover the roots   $\lambda _i$ and  $\lambda _j$ are conjugated under $W$ for all $i$ and  $j$ in  $\{0,1,\ldots ,k\}$.
\end{prop} 

\begin{proof} Let us first prove that  $s_{0}.\lambda_{0}=\lambda_{k}$.

Recall (Corollary \ref{corlambda0}) that the root   $\lambda _0$ is the unique root in  $\widetilde{\Sigma }$ such that 
$$ (1)\quad\begin{cases}
\bullet&\hskip 5pt \lambda _0 (H_0)=2 \  ;\cr
\bullet& \hskip 5pt  \lambda \in \Sigma ^+\Longrightarrow \lambda _0-\lambda \notin \widetilde{
\Sigma }\ .
\end{cases}$$

Therefore the root  $\mu=s_{0}.\lambda_{0}$ is characterized by the properties:
$$(2)\quad
\begin{cases}
           \bullet&\hskip 5pt \mu  (H_0)=2 \  (\text{since }s_{0}.H_{0}=2);\\
            \bullet& \hskip 5pt \lambda   \in \Sigma ^+\Longrightarrow
            \mu +\lambda \notin \widetilde{\Sigma }\  (\text{since } s_{0}.\Sigma^+=\Sigma^-)\ . 
\end{cases}$$

From Proposition \ref{prop.plusgranderacine} (and its proof) we know that $s_{0}.\lambda_{0}$ is the root $\lambda^{0}$ which is the restriction to  $\go{a}$ of the highest weight  $\omega$ of  $\go{g}$ on  $\overline{V^+}$. From the proof of  Proposition \ref{prop-gtilde(1)} we know also that $\omega$ is the highest weight of  $\go{g}_{1}$ on $\overline{V_{1}^+}$, and by induction  $\omega$ will be the highest weight of $\go{g}_{k}$ on $\overline{V_{k}^+}$. Hence $\omega_{|_{\go{a}}}$ is a root of $\widetilde{\Sigma}$ which is strongly orthogonal to $\lambda_{0},\lambda_{1},\dots,\lambda_{k-1}$. From  Theorem \ref{th-decomp-Eij}, there is only one root having this property, namely $\lambda_{k}$. Hence  $\lambda^0=s_{0}.\lambda_{0}=\lambda_{k}$.
\vskip 5pt
We will now prove that  $s_{0}.\lambda_{1}=\lambda_{k-1}$. We first decide to take  $s_{0}.\widetilde{\Sigma}^+$ as the set of positive roots in $\widetilde{\Sigma}$. The corresponding base will be $s_{0}.\widetilde{\Pi}$. We have  $s_{0}.\Sigma^+=\Sigma^-=-\Sigma^+$. As  $H_{0}$ is fixed by  $W$, the elements  $\lambda\in s_{0}.\widetilde{\Pi}$ still verify $\lambda(H_{0})=0 \text{ or } 2$ (condition (1) in Theorem \ref{thbasepi}). We apply now all we did before  and we obtain by "descent" a sequence $\mu_{0},\mu_{1},\dots,\mu_{k}$ of strongly orthogonal roots.  The root $\mu_{0}$ is the unique root in $s_{0}.\widetilde{\Pi}$, such that  $\mu_{0}(H_{0})=2$. Hence  $\mu_{0}=s_{0}.\lambda_{0}=\lambda_{k}$. 
\vskip 5pt
Now we will prove that  $s_{0}.\lambda_{1}=\lambda_{k-1}$. The centralizer of  $\widetilde{\go{l}}_{k}$ verifies again ${\bf (H_{1})}$ and ${\bf (H_{2})}$ and we will apply the preceding results to ${\cal Z}_{\widetilde{\go{g}}}(\widetilde{\go{l}}_{k})$.  Corollary  \ref{corlambda0} applied to this graded algebra implies that the root $\mu_{1}$ is characterized by
$$  (3)\quad
\begin{cases}
           \bullet & \mu _1(H_0)=2 \ ;\hfill\cr
              \bullet & \mu _1\sorth \lambda _k\ ;\hfill\cr
            \bullet&  \lambda \in \Sigma ^+\hbox{  and }\lambda \sorth\lambda _k
                   \Longrightarrow  \mu_1 +\lambda \notin\widetilde{ \Sigma }\ .  
\end{cases}$$

The same Corollary  applied to the graded algebra $\widetilde{\go{g}}_{1}$ iimplies that the root  $\lambda_{1}$ is characterized by
$$\begin{cases}
              \bullet & \lambda _1(H_0)=2 \ ;\hfill\cr
             \bullet& \lambda _1\sorth\lambda _0\ ;\hfill\cr
              \bullet& \lambda \in \Sigma ^+\hbox{  and }\lambda \sorth \lambda _0
             \Longrightarrow  \lambda _1 -\lambda \notin\widetilde{
               \Sigma }\ . 
\end{cases} $$
\vskip 5pt

As $s_{0}.\Sigma^+=\Sigma^-$ and  $s_{0}.\lambda_{0}=\lambda_{k}$, we get  $\mu_{1}=s_{0}.\lambda_{1}$.
\vskip 5pt

On the other hand the root $\lambda_{k-1}$ appears in  $V^+$ and is strongly orthogonal to $\lambda_{k}$. Let  $\lambda\in \Sigma^+$ be a root strongly orthogonal to  $\lambda_{k}$. If  $\lambda(H_{\lambda_{k-1}})=0$, then $\lambda$ is strongly orthogonal to  $\lambda_{k-1}$ from Corollary \ref{cor-orth=fortementorth}. Hence $\lambda+\lambda_{k-1}$ is not a root. If $\lambda(H_{\lambda_{k-1}})\neq 0$, then by  Proposition \ref{proplambda>0} there exists  $j< k-1$ such that $\go{g}^{\lambda}\subset E_{k-1,j}(1,-1)$. As $(\lambda+\lambda_{k-1})(H_{\lambda_{k-1}})=3$, $\lambda+\lambda_{k-1}$ is not a root. This shows that  $\lambda_{k-1}$ verifies the properties  $(3)$. Hence $\lambda_{k-1}= \mu_{1}=s_{0}.\lambda_{1}$.
\vskip 5pt

The first assertion is then proved by induction on $j$.

For the second assertion one applies the preceding result  to the graded algebras  $\widetilde{\go{g}}_{i}$ and $\widetilde{\go{g}}_{j}$ where $\lambda_{i}$ and $\lambda_{j}$ play the role of  $\lambda_{0}$. There exists an element $s_{i}\in W_{i}$ ($W_{i}$ is the Weyl group of  $(\widetilde{\go{g}}_{i}, \go{a}_{i})$) such that $s_{i}.\lambda_{i}=\lambda_{k}$ and an element $s_{j}\in W_{j}$ such that  $s_{j}.\lambda_{j}=\lambda_{k}$. As $W_{i}$ and  $W_{j}$ are subgroups of $W$, $s=s_{j}^{-1}s_{i}$ is an element of  $W$ which verifies  $s.\lambda_{i}=\lambda_{j}$.

 \end{proof}
 \vskip 5pt
 \begin{prop}\label{prop-memedimension}\-
 
        {\rm (1)} For  $j =0,\ldots ,k$ the root spaces   $\widetilde{ {\go g}}^{\lambda _j}$ have the same dimension.
        
         {\rm (2)} More generally the Lie algebras  $\widetilde{ {\go l}}_i={\widetilde{\go{g}}^{-\lambda_{i}}}\oplus[{\widetilde{\go{g}}^{-\lambda_{i}}},{\widetilde{\go{g}}^{\lambda_{i}}}]\oplus {\widetilde{\go{g}}^{\lambda_{i}}}$ \index{ltildegothiquei@$\widetilde{ {\go l}}_i$} are two by two conjugated by $G$.

        {\rm (3)} For  $i\not = j$, the spaces   $E_{i,j}(1,1),E_{i,j}(-1,-1)$ and $ E_{i,j}(1,-1)$ have the same dimension. This dimension is non zero    and independant of the pair  $\{i,j\} \in \{0,1\ldots ,k\}^2 $.

 \end{prop}
 
 \begin{proof} $(1)$ From  Proposition \ref{propWconjugues} there exists an element $s\in W$ such that  $s.\lambda_{i}=\lambda_{j}$. Let $g$ be an element of  $G$ such that  $g_{|_{\go{a}}}=s$. It is the easy to see that  $g. {\widetilde{\go{g}}^{\lambda_{i}}}={\widetilde{\go{g}}^{\lambda_{j}}}$. Therefore the vector spaces  ${\widetilde{\go{g}}^{\lambda_{i}}}$ and  ${\widetilde{\go{g}}^{\lambda_{j}}}$ are isomorphic.

 $(2)$ Let  $g\in G$  be the preceding element. Then one has also  $g. {\widetilde{\go{g}}^{-\lambda_{i}}}={\widetilde{\go{g}}^{-\lambda_{j}}}$. Hence $g.\widetilde{ {\go l}}_i=\widetilde{ {\go l}}_j$.
 
 $(3)$ Fix a pair $(i,j)$ with $i\neq j$. We choose $(X_{i},Y_{i})$ (resp. $(X_{j},Y_{j})$) in  $\widetilde{\go{g}}^{-\lambda_{i}}\times \widetilde{\go{g}}^{\lambda_{i}}$ (resp. in  $\widetilde{\go{g}}^{-\lambda_{j}}\times \widetilde{\go{g}}^{\lambda_{j}}$) such that$(Y_{i},H_{\lambda_{i}}, X_{i})$ (resp. $(Y_{j},H_{\lambda_{j}}, X_{j}))$ is an  $\go{sl}_{2}$-triple.
 
 Then $\ad X_i: E_{i,j}(-1,-1)\longrightarrow E_{i,j}(1,-1)$ is an isomorphism  whose   inverse is -$\ad Y_i$ (if  $u\in E_{i,j}(-1,-1)$ then $-\ad Y_{i}\ad X_{i}(u)=-[H_{\lambda_{i}},u]+[X_{i},[Y_{i},u]]=u$).
 
 Similarly  $\ad X_j: E_{i,j}(1,-1)\longrightarrow E_{i,j}(1,1)$ is an isomorphism whose inverse is  $-\ad Y_{j}$.

 This implies that the spaces  $E_{i,j}(\pm1,\pm1)$ are isomorphic when $i$ and  $j$ 	are fixed.
 
 In order to prove that the spaces  $E_{i,j}(1,1)$ are isomorphic  for distinct pairs  $(i,j)$   (with $i\neq j$), we will use the elements of the Weyl group which permute the $\lambda_{j}$.
 
 \vskip 5pt
 
 We prove first that  $E_{i,j}(1,1)\simeq E_{k,j}(1,1)$ for $j<i\leq k $ (recall that $k$ is the final index in the descent). Proposition  \ref{propWconjugues} applied to the graded algebra $\widetilde{\go{g}}_{i}$ implies the existence of  $s_{i}\in W_{i}$ which permutes  $\lambda_{i}$ and $\lambda_{k}$. The group  $W_{i}$ is generated by the reflections  defined by the roots $\lambda$ strongly orthogonal to  $\lambda_{0},\lambda_{1},\dots,\lambda_{i-1}$. As $j<i$, $\lambda_{j}$ is invariant under $W_{i}$. Hence 
 $$s_{i}: E_{i,j}(1,1)\longrightarrow E_{k,j}(1,1)$$
  is an isomorphism.
  Indeed if  $u\in E_{i,j}(1,1)$ then $[H_{\lambda_{k}}, s_{i}.u]= s_{i}. [s_{i}^{-1}.H_{\lambda_{k}},u]=s_{i}.[H_{\lambda_{i}},u]=s_{i}.u$ and  
  $[H_{\lambda_{j}}, s_{i}.u]= s_{i}. [s_{i}^{-1}.H_{\lambda_{j}},u]=s_{i}.[H_{\lambda_{j}},u]=s_{i}.u$. Hence $s_{i}.E_{i,j}(1,1)\subset E_{k,j}(1,1)$ and the restriction of $s_{i}^{-1}$ to $E_{k,j}(1,1)$ is the inverse.
  
  \vskip 5pt
  
Applying this to the triple $0<k-j\leq k$ one obtains that  
  $$s_{k-j}: E_{k-j,0}(1,1)\longrightarrow E_{k,0}(1,1)$$
  is an isomorphism.   
  
  On the  other hand, a similar proof  (and Proposition \ref{propWconjugues}) shows that
$$s_{0}:E_{k,j}(1,1)\longrightarrow E_{k-j,0}(1,1)$$
is an isomorphism.

This will imply that all the spaces  $E_{i,j}(1,1)$ are isomorphic. Indeed let us start from  $E_{i,j}(1,1)$ and  $E_{i',j'}(1,1)$ with  $j<i$ and $j'<i'$.

From above we have the following isomorphisms:
  $$E_{i,j}(1,1)\simeq E_{k,j}(1,1)\simeq E_{k-j,0}(1,1)\simeq E_{k,0}(1,1)\simeq E_{k-j',0}(1,1)\simeq E_{k,j'}(1,1)\simeq E_{i',j'}(1,1).$$
  
  It remains to prove that these spaces are not reduced to  $\{0\}$. \\
 If they were trivial,  the spaces $E_{i,j}(\pm1,\pm 1)$ would all be trivial and one would have the following decompositions:
  $$V^+=\oplus _{j=0}^k \widetilde{ {\go g}}^{\lambda _j}\quad\textrm{ and}\quad \go g= {\cal Z}_{\go g}({\go a}^0).$$
  
  But then $\widetilde{\go{g}}^{\lambda_{0}}$ would be invariant under $\go{g}$.  This is impossible by  ${\bf(H_{2})}$.
 
 \end{proof}
 \vskip 5pt
 \begin{notation}\label{notdle} In the rest of the paper we will use the following notations:
\medskip
           \begin{align*}
           \ell&=\dim \widetilde{ {\go g}}^{\lambda _j} \hbox{  for }j=0,\ldots ,k\ ;\cr
           d&=\dim E_{i,j}(\pm 1,\pm 1) \hbox{  for }i\not = j\in \{0,\ldots ,k\}\;(k\geq1)\cr
           e&=\dim \tilde{\go g}^{(\lambda_i\pm\lambda_j)/2}  \hbox{  for }i\not = j\in \{0,\ldots ,k\}\ (k\geq1)(e \text{ may be equal to } 0).\
                \end{align*} \index{lcal@$\ell$}\index{d@$d$}\index{e@$e$}
\end{notation} 
\vskip 5pt

From  Theorem \ref{th-decomp-Eij} giving the decomposition of $\dim V^+$, we obtain the following relation between, $k$, $d$ and $\ell$.

\begin{prop}\label{propdimV}
$$\dim V^+=(k+1)\left( \ell+\frac{kd}{2}\right)\ .$$
\end{prop}

  \vskip 20pt
 \subsection {Normalization of the Killing form} \label{sec:NKF}\hfill
  \vskip 10pt
  
  Let   $\widetilde{B}$ be a non degenerate  extension to $\widetilde{\go{g}}$ of the Killing form of  $[\widetilde{\go{g}},\widetilde{\go{g}}]$. As $\widetilde{\go{g}}={\cal Z}(\widetilde{\go{g}})\oplus [\widetilde{\go{g}},\widetilde{\go{g}}]$ where ${\cal Z}(\widetilde{\go{g}})$ is the center of  $\widetilde{\go{g}}$, we have 
  $$\widetilde{B}(z_{1}+u,z_{2}+u')= \kappa(z_{1},z_{2})+B(u,u')\hskip5pt \text{ for }z_{1},z_{2}\in {\cal Z}(\widetilde{\go{g}}) \text{ and }u,u'\in [\widetilde{\go{g}},\widetilde{\go{g}}],$$
  where  $B$ is the Killing form of  $[\widetilde{\go{g}},\widetilde{\go{g}}]$ and  $\kappa$ a non degenerate form on  ${\cal Z}(\widetilde{\go{g}})$.
  We fix once and for all such a form $\widetilde{B}$.
  \vskip 5pt
  
  \begin{definition} \label{defb(X,Y)}  For $X$ and $Y$ in   $\widetilde{{\go g}}$, we define the normalized Killing form   by setting :
$$b(X,Y)= -\frac{k+1}{4\,\dim V^+} \widetilde{ B }(X,Y)\ .$$ \index{b@$b(\,.\,,\,.\,)$}
\end{definition}
A first consequence of this definition is the following Lemma.
\vskip 5pt
\begin{lemme}\label{lemmeb}\-

For  $j\in\{0,\ldots,k\}$ one has 
  $b(H_{\lambda_j},H_{\lambda_j})=-2$. Moreover if  $(Y_{j},H_{\lambda _j},X_j) $ is an  $\go{sl}_{2}$-triple  such that $X_{j }\in \widetilde{ {\go g} }
^{\lambda _j}$ and  $Y_{j }\in \widetilde{ {\go g} }
^{-\lambda _j}$, then $b(X_j,Y_{j})=1$.

\end{lemme}

\begin{proof} As the elements  $H_{\lambda_{j}}$ are conjugated (Proposition \ref{propWconjugues}), as the roots  $\lambda_{j}$ are strongly orthogonal, and as $H_{0}=H_{\lambda_{0}}+H_{\lambda_{1}}+\dots+H_{\lambda_{k}}$ ( Theorem  \ref{th-decomp-Eij}) one has:
$$\widetilde{ B }(H_{\lambda _j},H_{\lambda _j})=\frac{1}{k+1}\widetilde{ B }(H_0,H_0)=\frac{1}{k+1}
{\rm tr}_{\widetilde{ {\go g} }}(\ad H_0)^2=8\,\frac{{\rm dim} V^+}{k+1}\,.$$
And then from the definition of $b$, we obtain  $b(H_{\lambda_j},H_{\lambda_j})=-2$.

On the other hand
$$\widetilde{ B }(Y_{j},X_j)=\frac{1}{2}\widetilde{ B }(Y_{j},[H_{\lambda
_j},X_j])=-\frac{1}{2}\widetilde{ B }(H_{\lambda _j},H_{\lambda
_j})=-4\,\frac{\dim V^+}{k+1}\, ,$$
and hence  $b(Y_{j},X_{j})=1$.

\end{proof}

Let  $A$ be a subset of $\{0,1,\dots,k\}$.  Consider the graded algebra $\widetilde{\go{g}}_{A}$ defined in Corollary \ref{cor-gA} which is graded by $H_{A}=\sum_{j\in A}H_{\lambda_{j}}$. We denote by  $b_{A}$ a normalized nondegenerate bilinear form $\widetilde{\go{g}}_{A}$ (defined as  $b$ on $\widetilde{\go{g}}$).

 \begin{lemme} Let  $\widetilde{\go{G}}_{A}$ be the subalgebra of $\widetilde{ {\go g} }_A$ generated by $V^+_A$ and $V^-_A$. If $X$ and  $Y$ belong to   $\widetilde{ {\go g} }_A$, then:
 $$b_A(X,Y)=b(X,Y)\ .$$
\end{lemme}

\begin{proof} We know from Proposition \ref{prop-lie(V+V-)simple} that $\widetilde{\go{G}}_{A}$ is  absolutely  simple. Then the dimension of the space of invariant bilinear forms on  $\widetilde{\go{G}}_{A}$ is equal to $1$ ( \cite{Bou3}, Exercice 18 a) of \S 6). Hence the restrictions of $b$ and  $b_{A}$ to  $\widetilde{\go{G}}_{A}$ are proportional. For  $j\in A$, $X_{j }\in \widetilde{ {\go g} }
^{\lambda _j}$ and  $Y_{j }\in \widetilde{ {\go g} }^{-\lambda _j}$ such that $(Y_{j },H_{\lambda_{j}},X_{j}) $ is an $\go{sl}_{2}$-triple, we have  $b_{A}(X_{j},Y_{j})=b(X_{j},Y_{j})=1$ (  Lemma \ref{lemmeb}).

\end{proof}

     \vskip 20pt
 \subsection{The relative invariant $\Delta_{0}$}\hfill
  \vskip 10pt
  
  Recall (\S $1.7$) that the group we are interested in and which will act on $V^+$ is   
  $$G= {\cal Z}_{\text{Aut}_{0}(\widetilde{\go{g}})}(H_{0})=\{g\in \text{Aut}_{0}(\widetilde{\go{g}})\,,\, g.H_{0}=H_{0}\}$$
  Recall also that  $(G,V^+)$ is a prehomogeneous vector space. Let  $S$ be the complementary set of the union of the open orbits in $V^+$.

  Recall also the definition of a relative invariant:
  \begin{definition}\label{def-inv-rel}  A rational function  $R$ on $V^+$ is a relative invariant \index{relative invariant} under  $G$ if there exists a rational character  $\chi $ of $G$ (that is a morphism $\chi: G \longrightarrow F^*$) such that 
$$R(g.X)=\chi (g)R(X) \hbox{  for all }g\in G \hbox{  and all  }X\in V^+\setminus S .$$
\end{definition}

\begin{rem} \label{rem-extension-invariant} From the density of $G$ in  $\overline{G}=G(\overline{F})={\cal Z}_{\text{Aut}_{_{0}}(\overline{\widetilde{\go{g}}})}(H_{0})={\cal Z}_{\text{Aut}_{e}(\overline{\widetilde{\go{g}}})}(H_{0})$ (\S 1.7), and from the density of  $V^+$ in $\overline{V^+}$, the natural extension of $R$ from $V^+$ to  $\overline{V^+}$, is a relative  invariant of $(\overline{G}, \overline{V^+})$
\end{rem}
\vskip 5pt

For further purposes let us give the following definition (\cite{Bou2} (Chap. VIII, \S1, $n^\circ5$, Prop. 6, p.75)).
\begin{definition}\label{def-theta-h} If $(y,h,x)$ is an $\go{sl}_{2}$-triple in $\widetilde{ {\go g} }$ and $t\in \overline{F}$ we define the following two elements of $\text{\rm Aut}_{e}(\overline{\widetilde{ {\go g} }})$:
\begin{align*}&\theta_{x}(t)=e^{t\ad_{\overline{\widetilde{\go{g}}} }(x)}e^{t^{-1}\ad_{\overline{\widetilde{\go{g}}}}(y)} e^{t\ad_{\overline{\widetilde{\go{g}}} }(x)}\\
&h_{x}(t)=\theta_{x}(t)\theta_{x}(-1)
\end{align*}\index{theta@$\theta_{x}$}\index{hx@$h_{x}$}
(Although $\theta_{x}$ and $h_{x}$ depend in fact on $(y,h,x)$, to simplify the notation we will not indicate this dependance unless absolutely necessary)
\end{definition}

\begin{lemme}\label{lem-tId-dansG}
For $t\in F^*$, one has $t{\rm Id}_{V^+}\in L_{|_{V^+}}$.
More precisely  if $(Y,H_{0},X)$ is an $\go{sl}_{2}$-triple where $X\in V^+$, $X$ generic,  and $Y\in V^-$, the element $h_{X}(\sqrt{t})$ belongs to $\text{Aut}_{0}(\widetilde{\go{g}})$ and
$$h_X(\sqrt{t})_{|_{V^+}}=t{\rm Id}_{V^+},\quad h_X(\sqrt{t})_{|_{V^-}}=t^{-1}{\rm Id}_{V^-},\quad h_X(\sqrt{t})_{|_{\go{g}}}={\rm Id}_{\go{g}}.$$
Hence $h_{X}(\sqrt{t})\in L\subset G$.
\end{lemme}

\begin{proof}  

Denote by $\go{u}$ the Lie subalgebra (isomorphic to $\go{sl}_{2}$) generated by $(Y,H_{0},X)$. For $t\in F$ consider the automorphism $h_{X}(t)$ given in Definition \ref{def-theta-h}
Recall that  $V^+$ (resp. $V^-$, resp. $\go{g}$) is the space of weight vectors of weight   $2$  (resp. -2, resp. 0) of  $\widetilde{\go{g}}$ for the adjoint action of the algebra  $\go{u}$. From \cite{Bou2} (Chap. VIII, \S1, $n^\circ5$, Prop. 6, p.75), $h_X(t)_{|_{V^+}}$ is scalar multiplication by $t^2$,  $h_X(t)_{|_{V^-}}$ is scalar multiplication by $t^{-2}$,  and  $h_X(t)_{|_{\go{g}}}$ is the identity. Over  $\overline{F}$, we can consider  $\sqrt{t}\in \overline{F}^*$. Then the automorphism $h_X(\sqrt{t})$ belongs to  ${\rm Aut}_{e}(\overline{\go{\widetilde{\go{g}}}})$ and stabilizes $\widetilde{\go{g}}$ as  $h_X(\sqrt{t})_{|_{V^+}}=t{\rm Id}_{V^+}$,  $h_X(\sqrt{t})_{|_{V^-}}=t^{-1}{\rm Id}_{V^-}$, $h_X(\sqrt{t})_{|_{\go{g}}}={\rm Id}_{\go{g}}$. The definition of  $\theta_X(\sqrt{t})$ and $h_X(\sqrt{t})$) implies  that the automorphism  $h_X(\sqrt{t})$ belongs to ${\cal Z}_{\text{Aut}_{0}(\widetilde{\go{g}})}(H_{0})=G$ and more precisely to $L=\bigcap_{i=0}^k G_{H_{\lambda_{i}}}$.

 \end{proof}

 \vskip 5pt
\begin{theorem}\label{thexistedelta_0}\-

$(1)$ There exists on  $V^+$ a unique (up to scalar multiplication) relative invariant polynomial $\Delta_{0}$ \index{Delta0@$\Delta_{0}$} which is absolutely irreducible  (i.e. irreducible as a polynomial on $\overline{V^+}$).  We will denote by $\chi_{0}$ \index{khi0@$\chi_{0}$}the character of $\Delta_{0}$.  

$(2)$ Any relative invariant on  $V^+$ is  (up to scalar multiplication) a power of $\Delta_{0}$.

$(3)$ An element  $X\in V^+$ is generic if and only if $\Delta_{0}(X)\neq 0$.
\end{theorem}

\begin{proof} We begin by constructing a non trivial relative invariant  $P$ of  $(G,V^+)$. We choose a base of  $V^+$ and a base of  $V^-$, for example the dual base  if we identify  $V^- $ and  $(V^+)^*$ by using the form $b$. One can then define a determinant  for all linear map from $V^-$ into  $V^+$. Consider the following linear map:
$$(\ad X)^2:V^-\longrightarrow V^+.$$
Set, for all  $X\in V^+$:
$$P(X)=\det(\ad X)^2.$$
We have, for $g\in G$:

\begin{align*}&P(g.X)=\det(\ad(g.X))^2)=\det(g.(\ad X)^2 g^{-1})\\
&=\det_{V^+}(g)\det_{V^-}(g^{-1})\det(\ad X)^2=\det_{V^+}(g)^2P(X)
\end{align*}
 If  $X$ is generic,  it can be put in an  $\go{sl}_{2}$-triple  $(Y,H_{0},X)$ ($Y\in V^-$), and then  $(\ad X)^2$ is an isomorphism between  $V^-$ and $V^+$. Hence $P\neq 0$. By Lemma \ref{lem-tId-dansG}, $t{\rm Id}_{V^+}\in G_{| V^+}$. Therefore the character of  $P$ is non trivial, hence $P$ is non constant.
 
By Remark \ref{rem-extension-invariant}, the natural extension of  $P$ to $\overline{V^+}$ is a relative invariant of  $(\overline{G},\overline{V^+})$. Then (by \cite{Sato-Kimura}, Proposition 12, p. 64), there exists a non trivial relative invariant  $\Delta_{0}$ of  $(\overline{G},\overline{V^+})$ which is an irreducible polynomial  and any other relative invariant is of the form $c. \Delta_{0}^m$ $(m\in \Z)$.

\vskip 5 pt 
We will need the following Lemma.
 
 \begin{lemme}\label{lem-invariantdefinisurF}
There exists  $\alpha\in \overline{F}^*$ such that  $\alpha\Delta_{0}$ takes values in  $F$ on  $V^+$.
\end{lemme}
 {\it Proof of the Lemma}:  
Let  ${\cal G}={\rm Gal}(\overline{F},F)$ be the Galois group of $\overline{F}$ over $F$. Let  $\sigma\in {\cal G}$. Then  $\sigma$ acts on $\overline{G}$ and fixes each point of $G$. It acts also on  $\overline{V^+}$ and fixes $V^+$. Then $\sigma$ acts on  $\Delta_{0}$ by  $\Delta_{0}^\sigma(x)=\sigma(\Delta_{0}(\sigma^{-1}x))$. The action on characters of  $\overline{G}$ is similarly defined. The polynomial   $\Delta_{0}^\sigma$ is still irreducible. Let now  $\overline{G}$ act on $\Delta_{0}^{\sigma}$.

 Let $\chi_{0} $ be the character of $\Delta_{0}$ (it is a character of $\overline{G}$ for the moment). Let  $x\in \overline{V^+}$ and  $g\in \overline{G}$.  One has:

$\Delta_{0}^\sigma(g.x)=\sigma(\Delta_{0}(\sigma^{-1}(g.x)))=\sigma(\Delta_{0}(\sigma^{-1}(g).\sigma^{-1}(x)))=\sigma(\chi_{0}(\sigma^{-1}(g))\sigma(\Delta_{0}(\sigma^{-1}(x)))$

$=\chi_{0}^\sigma(g)\Delta_{0}^\sigma(x).$ Hence   $\Delta_{0}^\sigma$   is an irreducible relative invariant of $(\overline{G},\overline{V^+})$, with character $\chi_{0}^\sigma$.   As this representation  is irreducible (${\bf (H_{3})}$), there exists  $c_{\sigma}\in \overline{F}$ such that 
$$\Delta_{0}^\sigma=c_{\sigma}\Delta_{0}.$$

Let  $x_{0}$ be a generic element of $V^+$, this implies that  $\Delta_{0}(x_{0})\neq 0$. Define $\alpha=\frac{1}{\Delta_{0}(x_{0})}.$ Then $$c_{\sigma}=\frac{\Delta_{0}^\sigma(x_{0})}{\Delta_{0}(x_{0})}= \frac{\sigma(\Delta_{0}(x_{0}))}{\Delta_{0}(x_{0})}=\frac{\alpha}{\sigma(\alpha)}.$$
For  $x\in V^+$ one has:
$$\sigma(\alpha)\Delta_{0}^\sigma(x)=\sigma(\alpha)c_{\sigma}\Delta_{0}(x)=\sigma(\alpha)\frac{\alpha}{\sigma(\alpha)}\Delta_{0}(x)=\alpha\Delta_{0}(x).$$
As $x\in V^+$, this can also be written:
$$\forall \sigma\in {\cal G},\,\,\sigma(\alpha\Delta_{0}(x))=\alpha\Delta_{0}(x)$$

One knows that the fixed points of  ${\cal G}$ in $\overline{F}$ are exactly the points in $F$. Hence for all  $x\in V^+$,  one has $\alpha\Delta_{0}(x)\in V^+$. 

 \hfill  {\it End of the proof of Lemma \ref{lem-invariantdefinisurF}}.

\vskip 20pt
{\bf From now on we will denote by  $\Delta_{0}$ the modified relative invariant of Lemma \ref{lem-invariantdefinisurF} which takes it values in  $F$ on $V^+$}. 
\vskip 20pt
Let  $P$ be a relative invariant of  $(G,V^+)$. Its extension  $\overline{P}$ to $\overline{V^+}$ is a relative invariant of  $(\overline{G}, \overline{V^+})$ (use the density of $G$ in $\overline{G}$ and of $V^+$ in  $\overline{V^+}$). Hence it exists  $a\in \overline{F}$ and  $m\in \Z$ such that  $\overline{P}=a\Delta_{0}^m$. As  $P(x)=a\Delta_{0}(x)^m$ for a generic  point $x$ in  $V^+$ and as  $\Delta_{0}(V^+)\subset F$, one get that $a\in F$. 

Hence assertions  $1) $ and  $2)$ are proved.

 One knows that  $X\in V^+$ is generic if and only if there exists an $\go{sl}_{2}$-triple $(Y,H_{0},X)$ ($Y\in V^-$) (Proposition \ref{prop-generiques-cas-regulier}). In that case $P(X)\neq0$ where $P$ is the relative invariant    defined at the beginning of the proof. Hence $\Delta_{0}(X)\neq0$. Conversely if $\Delta_{0}(X)\neq0$ for $X\in V^+$, then  $P(X)\neq0$, and therefore  $(\ad X)^2:V^-\longrightarrow V^+$ is an isomorphism. Hence $V^+={\rm Im}(\ad X_{|_{\go g}})$, and  $X$ is generic.
 
Assertion $3)$ is now proved.

\end{proof}
\vskip 5pt

    \vskip 20pt
 \subsection{The case $k=0$}\label{subsectionk=0}\hfill
  \vskip 10pt
  
  Recall that the case  $k=0$ corresponds to the graded algebra $$\widetilde{ {\go l}}_0={\widetilde{\go{g}}^{-\lambda_{0}}}\oplus[{\widetilde{\go{g}}^{-\lambda_{0}}},{\widetilde{\go{g}}^{\lambda_{0}}}]\oplus {\widetilde{\go{g}}^{\lambda_{0}}}.$$
  
 This algebra is an absolutely simple algebra of split rank  $1$ and it satisfies hypothesis  ${\bf (H_{1})}$ and ${\bf (H_{2})}$ (Proposition \ref{ell-0-simple}).  As this algebra is graded by  $H_{\lambda_{0}}$   and as there exist $X\in {\widetilde{\go{g}}^{\lambda_{0}}}$ and  $Y\in {\widetilde{\go{g}}^{-\lambda_{0}}}$ such that $(Y,H_{\lambda_{0}}, X)$ is an  $\go{sl}_{2}$-triple (\cite{Seligman}, Corollaire du Lemme 6, p.6, or  \cite{Schoeneberg}, Proposition 3.1.9 p.23), this algebra  $\widetilde{ {\go l}}_0$ satisfies also  ${\bf (H_{3})}$.

  \begin{lemme}\label{lemme-diagrammes-k=0}\hfill
  
 The absolutely simple Lie algebras of split rank  $1$ graded by  $H_{\lambda}$ ($\lambda$ being the unique restricted root) and which satisfy ${\bf (H_{1})},{\bf (H_  {2})}$ and  ${\bf (H_{3})}$ (this last condition is automatically satisfied) have the following  Satake-Tits diagrams $(d\in \N^*)$:
  \vskip 15pt
  
\hskip 70pt \hbox{\unitlength=0.5pt
\begin{picture}(280,30)
\put(-80,0) {$A_{2\delta-1}$}
  \put(10,10){\circle*{10}}
\put(15,10){\line (1,0){30}}
\put(50,10){\circle*{10}}
\put(60,10){\circle*{1}}
\put(65,10){\circle*{1}}
\put(70,10){\circle*{1}}
\put(75,10){\circle*{1}}
\put(80,10){\circle*{1}}
\put(90,10){\circle*{10}}
\put(95,10){\line (1,0){30}}
\put(130,10){\circle{10}}
\put(130,10){\circle{16}}
\put(135,10){\line (1,0){30}}
\put(170,10){\circle*{10}}
\put(180,10){\circle*{1}}
\put(185,10){\circle*{1}}
\put(190,10){\circle*{1}}
\put(195,10){\circle*{1}}
\put(200,10){\circle*{1}}
\put(210,10){\circle*{10}}
\put(215,10){\line (1,0){30}}
\put(250,10){\circle*{10}}
 \put(30,-20){$\delta-1$}  \put(180,-20){$\delta-1$} 
 \put(395,0){$B_{2}=C_{2}$}
 \put(505,10){\circle{10}}
 \put(505,10){\circle{16}}
\put(508,12){\line (1,0){41}}
\put(508,8){\line (1,0){41}}
\put(520,5){$>$}
\put(550,10){\circle*{10}}

\end{picture}
} 
\end{lemme}

   \begin{proof}
   
   The proof is a consequence of a careful reading of the tables of Tits  (\cite{Tits}). It can also be  extracted from the recent work of T. Schoeneberg (\cite{Schoeneberg}). For the convenience of the reader, we give some guidelines in connection with this last paper:

 - The fact that there exist no diagram of  type $G_{2}$, $F_{4}$, $E_{6} \text{ (inner forms)}$, $E_{6} \text{ (outer forms)}$, $E_{7}$, $E_{8}$, $D_{4}$ (with {\it trialitarian} action of the Galois group  $\overline{F}/F$) is a consequence of, respectively: Proposition 5.5.1 p.116,   Proposition 5.5.3 p.118,   Proposition 5.5.4 p.118,   Proposition 5.5.13 p.134,   Proposition 5.5.4 p.122,   Proposition 5.5.7 p.126, and of  Proposition 5.5.8 p.127. 
 
 - The case $A_{n}$ is a consequence of  Proposition 4.5.21 p. 93 (inner forms) and of pages 112-113 (outer forms).
 
 - The case $B_{n}$ is a consequence of Proposition 5.4.4 p. 113.
 
 - The case  $C_{n}$ is a consequence of Proposition 5.4.5 p. 113.
 
 - The case $D_{n}$ ({\it  non trialitarian})  is a consequence of  Proposition 5.4.6 p. 114 and of the fact that the diagrams on p. 115-116 do not occur.
 
 One can note that ${\bf (H_{1})}$ excludes  the diagrams of split rank  $1$ where the the unique white root has a coefficient  $>1$ in the highest root of the underlying Dynkin diagram, and those which have two white roots connected by an arrow.

   \end{proof} 
   
   \begin{cor}\label{cor-classification-k=0}\hfill
   
   The graded algebras  $$\widetilde{ {\go l}}_0={\widetilde{\go{g}}^{-\lambda_{0}}}\oplus[{\widetilde{\go{g}}^{-\lambda_{0}}},{\widetilde{\go{g}}^{\lambda_{0}}}]\oplus {\widetilde{\go{g}}^{\lambda_{0}}}$$
  are either isomorphic to $\go{sl}_{2}(D)$ where  $D$ \index{D@$D$} is a central division algebra over $F$, of degree  $\delta$, or isomorphic to  $\go{o}(q,5)$ where  $q$ is a non degenerate quadratic form on  $F^5$ which is the direct sum of an hyperbolic plane and an anisotropic form of dimension $3$ (in other word a form of index  $1$). 
   
   \end{cor}
   
   \begin{proof}
   
   These are the only Lie algebras  over $F$ whose Satake-Tits diagram is of the type given in the preceding Lemma (\cite{Tits}, \cite{Schoeneberg}).
   
   \end{proof}
   
  \begin{definition}\label{def-1-type}
  By  \ref{prop-memedimension},  all the algebras  $\widetilde{\go{l}}_{i}$ are isomorphic  either  to

  \hskip 40pt \hbox{\unitlength=0.5pt
\begin{picture}(280,30)
\put(-80,0) {$A_{2\delta-1}$}
  \put(10,10){\circle*{10}}
\put(15,10){\line (1,0){30}}
\put(50,10){\circle*{10}}
\put(60,10){\circle*{1}}
\put(65,10){\circle*{1}}
\put(70,10){\circle*{1}}
\put(75,10){\circle*{1}}
\put(80,10){\circle*{1}}
\put(90,10){\circle*{10}}
\put(95,10){\line (1,0){30}}
\put(130,10){\circle{10}}
\put(130,10){\circle{16}}
\put(135,10){\line (1,0){30}}
\put(170,10){\circle*{10}}
\put(180,10){\circle*{1}}
\put(185,10){\circle*{1}}
\put(190,10){\circle*{1}}
\put(195,10){\circle*{1}}
\put(200,10){\circle*{1}}
\put(210,10){\circle*{10}}
\put(215,10){\line (1,0){30}}
\put(250,10){\circle*{10}}
 \put(30,-20){$\delta-1$}  \put(180,-20){$\delta-1$} 
 \end{picture}
} or to 
\hbox{\unitlength=0.5pt
\begin{picture}(280,30)
 \put(20,0){$B_{2}$}
 \put(90,10){\circle{10}}
 \put(90,10){\circle{16}}
\put(93,12){\line (1,0){41}}
\put(93,8){\line (1,0){41}}
\put(105,5){$>$}
\put(135,10){\circle*{10}}

\end{picture}
}
\vskip 2pt 
In the first case we will say that $\widetilde{\go g}$  is of $1$-type $A$ $($or $(A, \delta)$ to be more precise$)$, in the second case  we wil say that  $\widetilde{\go g}$ is of  $1$-type $B$.

   \end{definition}
   
   \vskip 5pt
   
   \begin{theorem}\label{th-k=0}\hfill
   
   $1)$ If $\widetilde{ {\go l}}_0=\go{sl}_{2}(D)$ where  $D$ is a central division algebra  over $F$, of degree  $\delta$, the group $G$   is the group of isomorphisms  of  $ \go{sl}_{2}(D)$ of the form  
   $$\begin{array}{rll}
  \go{sl}_{2}(D)\ni X=\begin{pmatrix} a&x\\
   y&b\\
   \end{pmatrix}&\longmapsto&\begin{pmatrix} u&0\\
   0&v\\
   \end{pmatrix} \begin{pmatrix} a&x\\
   y&b\\
   \end{pmatrix}\begin{pmatrix} u^{-1}&0\\
   0&v^{-1}\\
   \end{pmatrix}
   \end{array}
$$

where $a,b,x,y \in D$, with  $\tr_{D/F}(a+b)=0$ ($\tr_{D/F}$ \index{trace@$\tr_{D/F}$} is the reduced trace), and where  $u,v\in D^*$.

Therefore the action of  $G$ on $V^+\simeq D$ can be identified with the action of  $D^*\times D^*$ on  $D$ given by
$$(u,v).x=uxv^{-1},$$
 the group  $G$ being isomorphic to  $(D^*\times D^*)/H$ where $H=\{(\lambda,\lambda), \,\lambda\in F^*\}$.

 Hence there are two orbits: $\{0\}$ and $D^*$,   and the fundamental relative invariant is the reduced norm $\nu_{D/F}$ \index{nu@$\nu_{D/F}$} of $D$ over $F$. Its degree is $\delta$.

   $2)$ If    $\widetilde{ {\go l}}_0=\go{o}(q,5)$ where $q$ is a non degenerate quadratic form over  $F^5$ which is the sum of an hyperbolic plane  and an anisotropic  form   $Q$ of dimension  $3$, then the group  $G$ can be identified  with the group  $SO(Q)\times F^*$ acting by the natural action on  $F^3$.  The fundamental relative invariant is then  $Q$. There are four orbits, namely  $\{0\}$ and three open orbits  which are the sets  ${\cal O}_{i}=\{x\in F^3, Q(x)\in u_{i}\}$, ($i=1,2,3$), where  $u_{i}$ runs over the three classes modulo $F^{*^2} $  distinct from $-d(Q)$ ($d(Q)$ being the  discriminant of  $Q$).
   \end{theorem}
   
   \begin{proof} 1) Let us make explicit the structure of   $\go{sl}_{2}(D)$. For the material below, see \cite{Schoeneberg}, p. 93.
   
   It is well known that  
   $$ \go{sl}_{2}(D)=\{\begin{pmatrix} a&x\\
   y&b\\
   \end{pmatrix}, \text{ where } a,b,x,y \in D,\text{ with } \tr_{D/F}(a+b)=0\}.$$
  A maximal split abelian subalgebra is given by:
   $$\go{a}=\{\begin{pmatrix} t&0\\
   0&-t\\
   \end{pmatrix},\,  t\in F\}.$$
   It is easy to see that  
   $$\go{m}={\cal Z}_{\go{sl}_{2}(D)}(\go{a})=\{\begin{pmatrix} a&0\\
   0&b\\
   \end{pmatrix}, a,b \in D,  \tr_{D/F}(a+b)=0\}=\go{a}\oplus \begin{pmatrix} [D,D]&0\\
   0&[D,D]\\
   \end{pmatrix}$$
   (recall that  $[D,D]=\ker  (\tr_{D/F})$, by \cite{Bou4}, \S17, $n^\circ3$, Corollaire of  proposition 5, p. A VIII.337).
   
   Therefore the anisotropic kernel of  $\go{sl}_{2}(D)$ is given by 
   $$[\go{m}, \go{m}]=[{\cal Z}_{\go{sl}_{2}(D)}(\go{a}),{\cal Z}_{\go{sl}_{2}(D)}(\go{a})]=\begin{pmatrix} [D,D]&0\\
   0&[D,D]\\
   \end{pmatrix}\simeq [D,D]\oplus [D,D].$$
   
  Its  Satake-Tits diagram is 
   
   \vskip 15pt

   \hskip 200pt \hbox{\unitlength=0.5pt
\begin{picture}(280,30)
\put(-150,0) {$A_{d-1}\times A_{d-1}$}
  \put(10,10){\circle*{10}}
\put(15,10){\line (1,0){30}}
\put(50,10){\circle*{10}}
\put(60,10){\circle*{1}}
\put(65,10){\circle*{1}}
\put(70,10){\circle*{1}}
\put(75,10){\circle*{1}}
\put(80,10){\circle*{1}}
\put(90,10){\circle*{10}}
\put(170,10){\circle*{10}}
\put(180,10){\circle*{1}}
\put(185,10){\circle*{1}}
\put(190,10){\circle*{1}}
\put(195,10){\circle*{1}}
\put(200,10){\circle*{1}}
\put(210,10){\circle*{10}}
\put(215,10){\line (1,0){30}}
\put(250,10){\circle*{10}}
 \put(30,-20){$\delta-1$}  \put(180,-20){$\delta-1$}

\end{picture}
} 
\vskip 30pt

  The grading of  $\widetilde{\go g}=\go{sl}_{2}(D)$ is then defined by  $H_{0}=\begin{pmatrix} 1&0\\
   0&-1\\
   \end{pmatrix}$, and this implies that  $\widetilde{\go g}=\go{sl}_{2}(D)= V^-\oplus {\mathfrak g}\oplus V^+$ where 
   $$V^-= \{\begin{pmatrix} 0&0\\
   y&0\\
   \end{pmatrix}, \,y\in D\}\simeq D,\quad V^+= \{\begin{pmatrix} 0&x\\
   0&0\\
   \end{pmatrix}, \,x\in D\}\simeq D$$
  and 
   $$\go{g}=\{\begin{pmatrix} a&0\\
   0&b\\
   \end{pmatrix}, a,b\in D, \tr_{D/F}(a+b)=0\}.$$
   
   We will now determine the group  $G= {\cal Z}_{\text{Aut}_{0}(\go{sl}_{2}(D))}(H_{0})=\{g\in \text{Aut}_{0}(\go{sl}_{2}(D))\,,\, g.H_{0}=H_{0}\}$ and its action on  $V^+\simeq D$. \\
    Let  $g\in G$. As $g.H_{0}=H_{0}$, one get  $g.V^{-}\subset V^{-},\, g.V^{+}\subset V^{+}, g.\go{g}\subset\go{g}$.
Let  $\overline{g}$ be the natural  extension of  $g$ as an automorphism of  $\go{sl}_{2}(D) \otimes \overline{F}=\go{sl}_{2d}( \overline{F})$. From above, $\overline{g}$ stabilizes $\overline{V^{-}},\overline{V^{+}}$ and  $\overline{\go{g}}$. By \cite{Bou2} (Chap. VIII, \S 13, $n^\circ 1$, (VII), p.189), there exists  $U\in GL(2d,\overline{F})$   such that 
$\overline{g}.x=UxU^{-1}, \forall x\in \go{sl}_{2d}( \overline{F})$. Let us write  $U$ in  the  form 
$$U=\begin{pmatrix}\alpha&\beta\\
\gamma&\delta
\end{pmatrix}, \,\, \alpha,\beta,\gamma,\delta \in M_{d}( \overline{F}).$$

As $\overline{g}$ stabilizes $\overline{V^{+}}$, for all $x\in M_{d}(\overline{F})$, there exists  $x'\in M_{d}(\overline{F})$ such that 
$$\begin{pmatrix}\alpha&\beta\\
\gamma&\delta
\end{pmatrix}\begin{pmatrix} 0&x\\
   0&0\\
   \end{pmatrix}=\begin{pmatrix} 0&x'\\
   0&0\\
   \end{pmatrix}\begin{pmatrix}\alpha&\beta\\
\gamma&\delta
\end{pmatrix}$$
It follows that  $\gamma=0$. Similarly the invariance of  $\overline{V^{-}}$ implies that  $\beta=0$. Hence  $U= \begin{pmatrix}\alpha&0\\
0&\delta
\end{pmatrix}$, with  $\alpha,\delta \in GL_{d}(\overline{F})$. Let us now write down that the conjugation by  $U$ (that is, the action of  $\overline{g}$) stabilizes $\go{sl}_{2}(D)$. For all $a,b\in D$, the element 
$$ \begin{pmatrix}\alpha&0\\
0&\delta
\end{pmatrix} \begin{pmatrix}a&0\\
0&b
\end{pmatrix} \begin{pmatrix}\alpha^{-1}&0\\
0&\delta^{-1}
\end{pmatrix}= \begin{pmatrix}\alpha a \alpha^{-1}&0\\
0&\delta b \delta^{-1}
\end{pmatrix}$$
belongs to  $\go{sl}_{2}(D)$. Therefore the map  $a\longmapsto \alpha a \alpha^{-1}$ (resp. $b\longmapsto\delta b \delta^{-1}$ ) will be an  automorphism of the associative algebra  $D$. By the  Skolem-Noether Theorem (\cite{Blanchard} (Th\'eor\`eme III-4 p.70), \cite{Pierce} (12.6 p.230) there exists $u_0$ (resp. $v_0$) in $D^*$ such that  $\alpha a \alpha^{-1}=u_0 a u_0^{-1}$ for all  $a\in D$ (resp. $ \delta b \delta^{-1}=v_0 b v_0^{-1}$ for all  $b\in D$). Hence  $u_0^{-1}\alpha$ (resp. $v_0^{-1}\delta$) is an element of  $M_{d}(\overline{F})=\overline{D}$ which commutes with any element of  $D$, and hence it belongs to the center of  $M_{d}(\overline{F})=\overline{D}$. Therefore  $u_0^{-1}\alpha= \lambda.1 $ and $v_0^{-1}\delta= \mu.1 $ ($\lambda,\mu \in \overline{F}$), i.e. $\alpha=\lambda u_0$ and $\delta=\mu v_0$.  
  As  $\bar{g}$ stabilizes $V^+$ and $V^-$ and 
$$ \begin{pmatrix}\lambda u_0&0\\
0&\mu v_0
\end{pmatrix} \begin{pmatrix}a&x\\
y&b
\end{pmatrix} \begin{pmatrix}\lambda^{-1}u_0^{-1}&0\\
0&\mu^{-1} v_0^{-1}
\end{pmatrix}= \begin{pmatrix}u_0 a u_0^{-1}&\lambda \mu^{-1} u_0xv_0^{-1}\\
\lambda^{-1} \mu v_0yu_0^{-1}&v_0 b v_0^{-1}
\end{pmatrix},$$ 
we deduce that  $\lambda \mu^{-1} \in F^*$ and  $\overline{g}$ is the conjugation by  $V= \begin{pmatrix}  u&0\\
0&v
\end{pmatrix}$ with  $u=\lambda \mu^{-1}  u_0\in D^*$ and $v=v_0\in D^*$.

Therefore the assertions concerning the action and the orbits are clear. It is also clear that  the conjugation by $\begin{pmatrix}u&0\\
0&v
\end{pmatrix}$ induces the trivial automorphism if and only if $u=v=\lambda\in F^*$. This proves that $G= (D^*\times D^*)/H$. 

As $\nu_{D/F}$ is a polynomial on  $D$ which takes values in  $F$, it is also clear that  $P(\begin{pmatrix} 0&x\\
   0&0\\
   \end{pmatrix})=\nu_{D/F}(x)$ is a relative invariant. As the reduced norm is an irreducible polynomial  (\cite{Saltman}), this relative invariant is the fundamental one.
   
   \vskip 5pt
   
   2) We will now give a realization  \footnote{We thank Marcus Slupinski for having indicated this realization to us.} of the Lie algebra  (there is only one up to isomorphism) whose Satake-Tits diagram is
   
     \vskip 15pt
  
\hskip 70pt \hbox{\unitlength=0.5pt
\begin{picture}(280,30)
 
 \put(285,0){$B_{2}$}
 \put(355,10){\circle{10}}
\put(358,12){\line (1,0){41}}
\put(358,8){\line (1,0){41}}
\put(370,5){$>$}
\put(400,10){\circle*{10}}

\end{picture}
}.
\vskip 5pt 
For this, let  $D$  be a central division algebra of degree  $2$ over  $F$. There is only one such algebra  by \cite{Blanchard} (Corollaire V-2, p. 130), and of course $D= (\frac{\pi,\varepsilon}{F})$ is the unique quaternion division algebra over  $F$ (\cite{Lam}, Th. 2.2. p.152). Here $\pi$ is a uniformizer of  $F$ and    $\varepsilon$ \index{epsilon@$\varepsilon$} is a unit  which is not a square.  Let $a\longmapsto \overline{a}$ ($a\in D$), be the usual conjugation in a quaternion algebra. The derived Lie algebra  $[D,D]$ is then the space of pure quaternions, that is the set of  $a\in D$ such that  $\overline{a}=-a$.

Set 
$$\widetilde{ {\mathfrak g}}=\{X=\begin{pmatrix}a&b\\
c&-\overline{a}
\end{pmatrix}, a\in D, b,c\in [D,D]\}.$$
It is easy to verify that  $\widetilde{\go{g}}$ is a Lie algebra for the bracket $[X,X']=XX'-X'X$, $X,X'\in \widetilde{\go{g}}$. Set also 
$$\go{a}=\{\begin{pmatrix}t&0\\
0&-t
\end{pmatrix}, t\in F\}.$$
It is clear that  $\go{a}$ is a split torus in  $\widetilde{\go{g}}$. If  $\go{a}'$  is a split torus containing  $\go{a}$, then  $\go{a}'$ is contained in the eigenspace  for the eigenvalue  $0$ of  $\ad(\go{a})$, that is  $$\go{a}'\subset \{\begin{pmatrix}a&0\\
0&-\overline{a}
\end{pmatrix}, a\in D\}\simeq D.$$
As  $D=F.1\oplus [D,D]$, and as  $[D,D]$ is anisotropic (see for example \cite{Schoeneberg} Corollaire 4.4.3. p.78), we obtain that  $\go{a}'=\go{a}$, hence  $\go{a}$ is a split maximal torus in  $\widetilde{\go{g}}$. The Lie algebra  $\widetilde{\go{g}}$  is graded by  $H_{0}=\begin{pmatrix}1&0\\
0&-1
\end{pmatrix}$:

$$\widetilde{ {\mathfrak g}}=V^-\oplus {\mathfrak g}\oplus V^+ ,$$
where 
$$V^+= \{\begin{pmatrix}0&b\\
0&0
\end{pmatrix},b\in [D,D]\},$$
$$V^-= \{\begin{pmatrix}0&0\\
c&0
\end{pmatrix},c\in [D,D]\},$$
$$\go{g}= \{\begin{pmatrix}a&0\\
0&-\overline{a}
\end{pmatrix},a\in D\}\simeq D.$$
Note that  $[\go{g}, \go{g}]\simeq [D,D]$ is anisotropic of dimension  3.
Let us now show that  $\widetilde{ {\mathfrak g}}$ is  simple. If  $I$ is an ideal of  $\widetilde{ {\mathfrak g}}$, as $[\go{a},I]\subset I$,  one has:
$$I=(V^-\cap I)\oplus( \go{g}\cap I)\oplus (V^+\cap I).$$

If  $(V^+\cap I)\neq \{0\}$ and if  $b\in (V^+\cap I)\setminus \{0\}$ then the elements of the form $$[\begin{pmatrix}a&0\\
0&-\overline{a}
\end{pmatrix}, \begin{pmatrix}0&b\\
0&0
\end{pmatrix}]= \begin{pmatrix}0&ab+b\overline{a}\\
0&0
\end{pmatrix}= \begin{pmatrix}0&ab-\overline{ab}\\
0&0
\end{pmatrix}= \begin{pmatrix}0& 2 \text{Im}(ab)\\
0&0
\end{pmatrix}$$
run over  $V^+$ if  $a\in D$. An analogous statement is true if  $(V^-\cap I)\neq \{0\}$. If 
$( \go{g}\cap I)\neq \{0\}$, then $(V^-\cap I)\neq \{0\}$ and  $(V^+\cap I)\neq \{0\}$. On the other hand, if  $V^+\subset I$ and  $V^-\subset I$, then it is easy to see that  $\go{g}\subset I$. Finally we have shown that  $\widetilde{ {\mathfrak g}}$ has no non trivial ideal. Hence  $\widetilde{ {\mathfrak g}}$ is a simple Lie algebra of dimension $10$. Therefore  $\overline{\widetilde{\go{g}}}=\widetilde{\go{g}}\otimes_{F}\overline{F}$ is a semi-simple Lie algebra of  dimension $10$ over $\overline{F}$. There is only one such algebra, it is the orthogonal algebra $\go{o}(5,\overline{F})$ whose Dynkin diagram is  $B_{2}$. Therefore the algebra $\widetilde{ {\mathfrak g}}$, with the grading described before, is indeed the algebra whose Satake-Tits diagram is   

     \vskip 15pt
  
\hskip 70pt \hbox{\unitlength=0.5pt
\begin{picture}(280,30)
 
 \put(285,0){$B_{2}$}
 \put(355,10){\circle{10}}
\put(358,12){\line (1,0){41}}
\put(358,8){\line (1,0){41}}
\put(370,5){$>$}
\put(400,10){\circle*{10}}

\end{picture}
}.

(To avoid the preceding dimension argument, it is possible to compute explicitly $\overline{\widetilde{\go{g}}}$ by using the fact that  $\overline{F}$ is a splitting field of  $D$ and show  that 
$$\overline{\widetilde{\go{g}}}=\{\begin{pmatrix}A&B\\
C&-\tau(A)
\end{pmatrix}, A\in M_{2}(\overline{F}), B,C\in [M_{2}(\overline{F}),M_{2}(\overline{F})]=\go{sl}_{2}(\overline{F})\}$$
where the anti-involution  $\tau$ of  $M_{2}(\overline{F})$ is defined by  $\tau(\begin{pmatrix}\alpha&\beta\\
\gamma&\delta
\end{pmatrix})=\begin{pmatrix}\delta&-\beta\\
-\gamma&\alpha
\end{pmatrix}$. This algebra is  $\go{o}(5,\overline{F})$).

Remind that   $[D,D]=\text{Im}D=\{a\in D, \overline{a}=-a\}$. Set
$$W=\{\begin{pmatrix}h&\lambda\\
\mu&-h\\
\end{pmatrix}, \text{ where }h\in \text{Im}D=[D,D],\lambda,\mu\in F\}.$$

One sees easily that  $[\widetilde{\go{g}},W]\subset W$, where $[\,\,,\,\,]$ is the usual bracket of matrices. The corresponding representation is of  dimension 5. An easy  but a little tedious computation shows that the symmetric bilinear form  
$$\Psi(\begin{pmatrix}h&\lambda\\
\mu&-h\\
\end{pmatrix},\begin{pmatrix}h'&\lambda'\\
\mu'&-h'\\
\end{pmatrix})= -\frac{1}{2}(hh'+h'h+\lambda'\mu+\lambda \mu')$$
is invariant under the action of  $\widetilde{\go{g}}$. The form  $\Psi$ is the bilinear form  associated to the quadratic form  "determinant" $q(\begin{pmatrix}h&\lambda\\
\mu&-h\\
\end{pmatrix})= -h^2-\lambda\mu$. This form  is the direct sum of an hyperbolic plane and the anisotropic form $Q(h)=-h^2$ on  $[D,D]$. The space  $F^3$ in the statement of the Theorem  is therefore $[D,D]$ and the algebra  $\widetilde{\go{g}}$  is realized as the algebra  $\go{o}(q,5)$ as in the statement.

We need also to consider  $\overline{W}$. One has:

$$\overline{W}=\{\begin{pmatrix}U&x \text{Id}_{2}\\
y \text{Id}_{2}&-U\\
\end{pmatrix}, \text{ where }U\in [M_{2}(\overline{F}), M_{2}(\overline{F})]=\go{sl}_{2}(\overline{F}),x, y \in \overline{F}\}.$$

Let us denote by  $\overline{q}$, $\overline{Q}$ and $\overline{\Psi}$ the lifts    of $q,Q$ and  $\Psi$  to $\overline{W}$, respectively. 
These are given by  (remark that $UU'+U'U$ is a scalar matrix if $U,U'$ are in ${\go sl}_2(\bar{F})$): 
$$\overline{q}(\begin{pmatrix}U&x \text{Id}_{2}\\
y \text{Id}_{2}&-U\\
\end{pmatrix})= -U^2-xy$$
$$\overline{Q}(\begin{pmatrix}U&0\\
0&-U\\
\end{pmatrix})= -U^2$$
$$\overline{\Psi}(\begin{pmatrix}U&x \text{Id}_{2}\\
y \text{Id}_{2}&-U\\
\end{pmatrix}, \begin{pmatrix}U'&x' \text{Id}_{2}\\
y' \text{Id}_{2}&-U'\\
\end{pmatrix})= -\frac{1}{2}(UU'+U'U+x'y+xy')$$

\vskip 10pt

The Lie algebra  $\overline{\widetilde{\go{g}}}$ acts on  $\overline{W}$ by the adjoint action: $[\overline{\widetilde{\go{g}}},\overline{W}]\subset \overline{W}$. 
More explicitly a calculation shows that for  $\begin{pmatrix}A&B\\
C&-\tau(A)
\end{pmatrix}\in \overline{\widetilde{\go{g}}}$ and $\begin{pmatrix}U&x \text{Id}_{2}\\
y \text{Id}_{2}&-U\\
\end{pmatrix}\in \overline{W}$ one has

$$[\begin{pmatrix}A&B\\
C&-\tau(A)
\end{pmatrix},\begin{pmatrix}U&x \text{Id}_{2}\\
y \text{Id}_{2}&-U\\
\end{pmatrix}]=\begin{pmatrix}[A,U]+yB-xC&-(UB+BU)+x(A+\tau(A))\\
CU+UC-y(A+\tau(A))&[\tau(A),U]+xC-yB\\
\end{pmatrix}$$
and this last matrix  is effectively an element of $\overline{W}$.\\
 
This implies that the representation of  $\overline{\widetilde{\go{g}}}$ in $W$ is faithful, that the form  $\overline{\Psi}$ is invariant under $\overline{\widetilde{\go{g}}}$ and  therefore  $\overline{\widetilde{\go{g}}}$ is realized as   $\go{o}(\overline{W},\overline{q})$.

Let  $\varphi \in \text{Aut}(\overline{\widetilde{\go{g}}})$.   
As $\overline{\widetilde{\go{g}}} \simeq \go{o}(\overline{W},\overline{q})$ is a split simple Lie algebra, we know by \cite{Bou3}, Chap. VIII, \S 13, $n^\circ$2, (VII), p. 199, that 
$$\text{Aut}_{0}(\overline{\widetilde{\go{g}}})=\text{Aut}_{e}(\overline{\widetilde{\go{g}}})=\text{Aut}(\overline{\widetilde{\go{g}}})$$
and that there exists a unique  $M\in SO(\overline{W}, \overline{q})$ such that, for  $X\in \overline{\widetilde{\go{g}}}$ one has 
$$\varphi(X)=MXM^{-1}, $$
the right hand side  being a product of endomorphisms of  $\overline{W}$. In the rest of the proof  we will write $\varphi=M$ , by abuse of notation, and we will consider that  $ \varphi\in SO(\overline{W}, \overline{q})$.

Under the action of  $H_{0}$ the space $\overline{W}$ decomposes into three eigenspaces corresponding to the eigenvalues  $-2, 0, 2$ :

$$\overline{W}_{-2}=\{\begin{pmatrix}0&0\\
y \text{Id}_{2}& 0\\
\end{pmatrix}, y\in \overline{F}\},\,\,\overline{W}_{0}=\{\begin{pmatrix}U&0\\
0 & -U\\
\end{pmatrix}, U\in\go{sl}_{2}(\overline{F})\},\,\,\overline{W}_{2}=\{\begin{pmatrix}0&x \text{Id}_{2}\\
0 & 0\\
\end{pmatrix}, x\in \overline{F}\}.$$

We will need the following Lemma.

\begin{lemme}\label{lemme-caracterisation-Z(H)}\hfill

Let  $\varphi \in \text{\rm Aut}(\overline{\widetilde{\go{g}}})=SO(\overline{W}, \overline{q})$.

{\rm a)} $\varphi$ fixes  $H_{0} \Longleftrightarrow   \begin{cases}
\varphi\text{  stabilizes } \overline{W}_{-2},\overline{W}_{0}, \overline{W}_{2};\cr
 \varphi_{|_{\overline{W}_{0}}}\in SO(\overline{W}_{0}, \overline{Q}) ;\cr
 \text{ there exists }\alpha_{\varphi}\in \overline{F}^* \text{ such that }\varphi_{|_{\overline{W}_{-2}}}=\alpha^{-1}_{\varphi}{\rm Id}_{|_{\overline{W}_{-2}}} \text{ and  } \varphi_{|_{\overline{W}_{2}}}={\alpha_{\varphi}} {\rm Id}_{|_{\overline{W}_{2}}}.
\end{cases}
$
\vskip 10pt
{\rm b)} Let  $g\in  \text{\rm Aut}(\widetilde{\go{g}})$ and let  $\overline{g}$ be its natural extension to an element of   $\text{\rm Aut}(\overline{\widetilde{\go{g}}})$. Then one has:

$g\in G= {\cal Z}_{\text{Aut}_{0}(\widetilde{\go{g}})}(H_{0}) \Longleftrightarrow   \begin{cases}
\overline{g}\text{  stabilizes } {W}_{-2}, {W}_{0},  {W}_{2};\cr
\overline{g}_{|_{ {W}_{0}}}\in SO( {W}_{0},  {Q}) ;\cr
\alpha_{\overline{g}} \in F^*.\cr
\end{cases}$

\end{lemme}
\vskip 10pt

 {\it Proof of the Lemma}:  
 
a) Suppose that $\varphi$ commutes with $H_{0}$, then $H_{0}=\varphi H_{0}\varphi^{-1}$  (products of endomorphisms of $\overline{W}$).   ($i=-2,0,2$). Let  $T\in \overline{W}_{i}$ ($i=-2,0,2$). Then  $H_{0}.\varphi (T)= \varphi H_{0}\varphi^{-1} \varphi (T)=\varphi H_{0}(T)= \varphi(iT)=i\varphi(T)$. Hence $\varphi$ stabilizes the spaces $\overline{W}_{i}$.

As $\overline{W}_{-2}$ and  $\overline{W}_{2}$ are  $1$-dimensional , there exist $\alpha_{\varphi},\beta_{\varphi}\in \overline{F}^*$ such that 

$$\varphi(\begin{pmatrix}0&0\\
y \text{Id}_{2}& 0\\
\end{pmatrix})=\begin{pmatrix}0&0\\
\beta_{\varphi} y \text{Id}_{2}& 0\\
\end{pmatrix}, \,\, \varphi(\begin{pmatrix}0&x \text{Id}_{2}\\
0 & 0\\
\end{pmatrix})=\begin{pmatrix}0&\alpha_{\varphi} x \text{Id}_{2}\\
0 & 0\\
\end{pmatrix},\,\, x,y\in \overline{F}.$$

But as  $ \varphi\in SO(\overline{W}, \overline{q})$,  one has
$$\overline{q}(\varphi(\begin{pmatrix}0&x \text{Id}_{2}\\
y \text{Id}_{2}&0\\
\end{pmatrix}))=\overline{q}(\begin{pmatrix}0&\alpha_{\varphi} x \text{Id}_{2}\\
\beta_{\varphi} y \text{Id}_{2}&0\\
\end{pmatrix})=-\alpha_{\varphi}\beta_{\varphi} xy=\overline{q}(\begin{pmatrix}0&  x \text{Id}_{2}\\
  y \text{Id}_{2}&0\\
\end{pmatrix})=-xy.$$
Hence $\beta_{\varphi}=\alpha^{-1}_{\varphi}.$
As $\overline{q}_{|_{\overline{W}_{0}}}=\overline{Q}$ one get that  $ \varphi_{|_{\overline{W}_{0}}}\in SO(\overline{W}_{0}, \overline{Q})$.

Conversely suppose that  $\varphi$ satisfies the conditions on the  right hand side of  a). Let us look how $\varphi H_{0} \varphi^{-1}$ acts on  $\overline{W}$:

\xymatrix{
    {\begin{pmatrix}U&x \text{Id}_{2}\\
y \text{Id}_{2}&-U\\
\end{pmatrix}}  \ar@{|->}[r]^{\varphi^{-1}}   & {\begin{pmatrix}\varphi^{-1}(U)&\alpha^{-1}_{\varphi} x \text{Id}_{2}\\
\alpha_{\varphi} y \text{Id}_{2}&-\varphi^{-1}(U)\\
\end{pmatrix}}\ar@{|->}[r]^{H_{0} }&{\begin{pmatrix}0&2\alpha^{-1}_{\varphi} x \text{Id}_{2}\\
-2\alpha_{\varphi} y \text{Id}_{2}&0\\
\end{pmatrix}}\ar@{|->}[r]^{\varphi }&{\begin{pmatrix}0&2 x \text{Id}_{2}\\
-2y \text{Id}_{2}&0\\
\end{pmatrix}} 
 .}
 But one has:
 
 \xymatrix{
    {\begin{pmatrix}U&x \text{Id}_{2}\\
y \text{Id}_{2}&-U\\
\end{pmatrix}}  \ar@{|->}[r]^{{H_{0}}}   &  {\begin{pmatrix}0&2 x \text{Id}_{2}\\
-2y \text{Id}_{2}&0\\
\end{pmatrix}} 
 .}
 Therefore  $\varphi H_{0}\varphi^{-1}=H_{0}$, and this proves a).
 
 b) is a consequence of a).
 
 \hfill  {\it End of the proof of Lemma \ref{lemme-caracterisation-Z(H)}}.

   \vskip15pt
 If  $g\in G$, let   $\alpha_{g}$ be the element  $\alpha_{\overline{g}}\in F^*$ obtained in the preceding Lemma.
  \vskip15pt

 \begin{lemme}\label{lemme-action-G}\hfill
 
 The action of   $g\in G$ on  $V^+=\{\begin{pmatrix}0&b\\
 0&0\end{pmatrix},\,\, b\in [D,D]\}$ is as follows
 $$g.\begin{pmatrix}0&b\\
 0&0\end{pmatrix}= \begin{pmatrix}0&\alpha_{g}g_{|_{[D,D]}}(b)\\
 0&0\end{pmatrix}.$$
  \end{lemme}
  
  \vskip 10pt
 {\it Proof of the Lemma}: 
  \vskip 5pt
  
 In order to compute  $g.\begin{pmatrix}0&b\\
 0&0\end{pmatrix}=g\begin{pmatrix}0&b\\
 0&0\end{pmatrix}g^{-1}$, we will compute its action on the element 
 
 $\begin{pmatrix}h&\lambda\\
 \mu&-h
 \end{pmatrix}\in W$, using the preceding  Lemma \ref{lemme-caracterisation-Z(H)}. One has:

\centerline {\xymatrix{
    {\begin{pmatrix}h&\lambda\\
\mu&-h\\
\end{pmatrix}}  \ar@{|->}[r]^{g^{-1}}   & {\begin{pmatrix}g^{-1}(h)&\alpha^{-1}_{g} \lambda\\
\alpha_{g} \mu&-g^{-1}(h)\\
\end{pmatrix}}\ar@{|->}[r]^{\begin{pmatrix}0&b\\
 0&0\end{pmatrix}} &{\begin{pmatrix}\alpha_{g}\mu b&-bg^{-1}(h)-g^{-1}(h)b\\
0&-\alpha_{g}\mu b\\
\end{pmatrix}}\ar@{|->}[d]^{g }\\
{}&{}&{\begin{pmatrix}\alpha_{g}\mu g(b)& -\alpha(bg^{-1}(h)+g^{-1}(h)b)\\
0&-\alpha_{g}\mu g(b)\\
\end{pmatrix}} 
 (*).}
  } 
 Let us show now that the action of the element  $ \begin{pmatrix}0&\alpha_{g}g_{|_{[D,D]}}(b)\\
 0&0\end{pmatrix}$ is the same. One has:
 
 $$[ \begin{pmatrix}0&\alpha_{g}g_{|_{[D,D]}}(b)\\
 0&0\end{pmatrix}, \begin{pmatrix}h&\lambda\\
\mu&-h\\
\end{pmatrix} ]= {\begin{pmatrix}\alpha_{g}\mu g(b)& -\alpha(g(b) h+ hg(b))\\
0&-\alpha_{g}\mu g(b)\\
\end{pmatrix}}(**)$$

As the bilinear form associated to  $Q$ is   $B(b,h)=bh+hb$ ($b,h\in [D,D]$) and as $g_{|_{W_{0}}}\in SO(W_{0},Q)$ we obtain   $bg^{-1}(h)+g^{-1}(h)b=g(b) h+ hg(b)$, hence (*)=(**).

\vskip 5pt

\hfill {\it End of the proof of Lemma \ref{lemme-action-G}}.

\vskip 10pt

From Lemma  \ref{lemme-action-G} we know that the map  $$\begin{array}{rll}
G& \longrightarrow & SO(W_{0},Q)\times F^*\cr
g&\longmapsto &(g_{|_{W_{0}}},\alpha_{g})\cr
\end{array}$$
is a bijection. We have therefore proved the first part of the statement  2) Theorem \ref{th-k=0}.

From  \cite{Lam} (Th\'eor\`eme 2.2. 1)  p. 152 ), there are four classes modulo ${F^*}^2$ in $F^*$, which are the classes of  $1,\epsilon, \pi, \epsilon\pi$ ($\epsilon$ is a non square unit and $\pi$ is a uniformizer).  And from  \cite{Lam} (Corollary 2.5 3), p.153-154) the anisotropic form  $Q$ represents all (non zero) classes  except the class of  $-d(Q)$. If for  $b_{1}, b_{2}\in V^+\simeq [D,D]$, the elements  $Q(b_{1})$ and  $Q(b_{2})$   are in the same class  then there exists  $t\in F^*$ such that $Q(b_{1})=t^2Q(b_{2})=Q(tb_{2})$. Witt's Theorem  implies then that  there exists  $g\in O(Q)$ such that  $b_{1}=tg(b_{2})$. As the dimension is  3, det($-\text{Id}_{V^+}$)$=-1$ and one can suppose that   $g\in SO(Q)$. The Theorem is proved.

 \end{proof}
 
 \begin{rem}\label{0=j} As the algebras  $\widetilde{ {\go l}}_0$ and  $\widetilde{ {\go l}}_j$ 	are isomorphic (Proposition \ref{prop-memedimension}),  Corollary \ref{cor-classification-k=0} and Theorem \ref{th-k=0} will also be true for  $\widetilde{ {\go l}}_j$
. \end{rem}
 
 \begin{lemme} \label{lemmecodim}\hfill
\vskip 5pt
 Let us fix an element   $ 
X^1=X_1+X_2+\cdots +X_k \quad (X_j\in \widetilde{
       {\go g}}^{\lambda _j}\backslash \{0\}) $. Then for  $X\in \widetilde{ {\go g}}^{\lambda _0}$ one has  
       $$V^+=[\go g,X^1+X]+ \widetilde{ {\go g}}^{\lambda _0}\ .$$
       The   codimension   of the  $G$-orbit of  $X+X^1$ in  $V^+$ is given by  
$$\codim [\go g,X+X^1 ] = 
\begin{cases}
\ell \text{ if } X=0\ ;\\
0 \text{ if } X\not = 0 \ .
\end{cases}$$
       
       \end{lemme}
       
       \begin{proof}\hfill
       
      If $X\neq0$, then by  Proposition \ref{X0+...+Xkgenerique},   $X+X^1$ is generic in  $V^+$.  Hence $V^+=[\go g, X+X^1]$ and  $\codim [\go g,X+X^1 ] = 0$.
       
      If  $X=0$, we know by Corollary \ref{cor-decomp-j} that  
      $$V^+=V^+_1\oplus \left( \oplus_{j=1}^k E_{0,j}(1,1)  \right)\oplus
\widetilde{ {\go g}}^{\lambda _0}\ . $$

 $X^1$ is generic in  $V_{1}^+$ by  Proposition \ref{X0+...+Xkgenerique} applied to  $\widetilde{\go{g}}_{1}$. Therefore:
$$V_{1}^+=[\go{g}_{1},X^{1}]\subset [\go{g},X^1].\eqno (*)$$ 
   Let $Y_j\in
\widetilde{\go g}^{-\lambda_j}$ be such that  $\{Y_j,H_{\lambda
_j},X_j\}$ is an   $\go {sl}_2$-triple. Let    $A\in E_{0,j}(1,1)$, then  $B=[Y_j,A]$ is an element of   $E_{0,j}(1,-1)\subset \go{g}$ and one has
$$[B,X^1]=[B, X_{1}+X_{2}+\dots+X_{k}]=[B,X_j]=[[Y_{j},A],X_{j}]= [[Y_j,X_j],A]=[H_{j},A]=A\ .$$
Hence 
$\oplus_{j=1}^kE_{0,j}(1,1)\subset [\go g,X^1]$ and from decomposition $(*)$  above we get
$$V^+=[\go g,X^1]+ \widetilde{ {\go g}}^{\lambda _0}\ .$$    
To obtain the result on the codimension,  it is enough to prove that the preceding sum is direct.
The relations 
\begin{align*}
[\go z_{\go g}(\go a^0),X^1]&\subset \oplus_{j=1}^k
\widetilde {\go g}^{\lambda_j}\subset V^+_1\ ,\\
[E_{i,j}(1,-1), X^1]&\subset
         \begin{cases}
                       \{0\} \text{ if } j=0  ,\\
                      E_{i,j}(1,1) \text{ if } j\not =0\ ,
         \end{cases}
\end{align*}
imply that 
$$[\go g,X^1]\subset V^+_1\oplus\left(\oplus_{j=1}^k E_{0,j}(1,1)\right)\
.$$
Hence  
$$[\go g, X^1]= V^+_1\oplus\left(\oplus_{j=1}^k E_{0,j}(1,1)\right)\
,$$
and finally
$$V^+=[\go g, X^1]\oplus \widetilde {\go g}^{\lambda_0}\ .$$

       \end{proof}

    \vskip 20pt
 \subsection{Properties of  $\Delta_{0}$}\label{sec:proprietes-Delta-0}\hfill
  \vskip 10pt

Let  $\delta_{i}$ \index{deltai@$\delta_{i}$ (Part I)} be the fundamental relative invariant of  the prehomogeous vector space $\widetilde{ {\go l}}_i$ ($0\leq i\leq k$). By Theorem  \ref{thexistedelta_0}$,  \delta_{i}$ is absolutely irreducible.  As all the algebras $\widetilde{ {\go l}}_i$ are isomorphic (Remark \ref{0=j}),  	 the  $\delta_{i}$'s have all the same degree .  
\vskip 5pt

\begin{notation}\label{notation-kappa}

Let us denote by  $\kappa$ \index{kappa@$\kappa$} the common degree of the polynomial  $\delta_{i}$. By Theorem  \ref{th-k=0} there exists two types of graded algebras of  rank 1. 
One has \begin{align*}
 \kappa&=
         \begin{cases}
                      \delta \text{ in case   (1) } \text{corresponding  to the diagram}   \hskip 200pt \hbox{\unitlength=0.5pt
\begin{picture}(280,30)
  \put(-360,10){\circle*{10}}
\put(-355,10){\line (1,0){30}}
\put(-320,10){\circle*{10}}
\put(-310,10){\circle*{1}}
\put(-305,10){\circle*{1}}
\put(-300,10){\circle*{1}}
\put(-295,10){\circle*{1}}
\put(-290,10){\circle*{1}}
\put(-280,10){\circle*{10}}
\put(-275,10){\line (1,0){30}}
\put(-240,10){\circle{10}}
\put(-240,10){\circle{16}}
\put(-235,10){\line (1,0){30}}
\put(-200,10){\circle*{10}}
\put(-190,10){\circle*{1}}
\put(-185,10){\circle*{1}}
\put(-180,10){\circle*{1}}
\put(-175,10){\circle*{1}}
\put(-170,10){\circle*{1}}
\put(-165,10){\circle*{10}}
\put(-160,10){\line (1,0){30}}
\put(-125,10){\circle*{10}}
 \put(-330 ,-20){$\delta-1$}  \put(-190,-20){$\delta-1$} 
  \put(-100,0){$(\delta\in \N^*) $}\end{picture}} \\
  {}\\
                     2 \text{ in case  (2) } \text{corresponding to the diagram}   \hskip 70pt \hbox{\unitlength=0.5pt
\begin{picture}(280,30)
 \put(-70,10){\circle{10}}
 \put(-70,10){\circle{16}}
\put(-67,12){\line (1,0){41}}
\put(-67,8){\line (1,0){41}}
\put(-55,5){$>$}
\put(-25,10){\circle*{10}}
\end{picture}
}   \\
  \end{cases}
\end{align*}

\end{notation}

  \begin{theorem}\label{thpropridelta_0}\- 

  \noindent      {\rm (1)}  $\Delta _0$ is a homogeneous polynomial of degree   $\kappa (k+1)$.

 \noindent      {\rm (2)} For $j=0,\ldots ,k$, let $X_j$ be an element of   $\widetilde{
         {\go g}}^{\lambda
        _j}\backslash\{0\}$ and let  $x_j\in F$. Then 
      $$\Delta _0\Bigl(\sum _{j=0}^k x_jX_j\Bigr)=\prod _{j=0}^k x_j^\kappa\; . \Delta
          _0\Bigl(\sum _{j=0}^k X_j\Bigr)\ .$$

\end{theorem}

\begin{proof}\hfill

(1) Consider the prehomogeneous vector space  $\overline{\widetilde{ {\go l}}_0}= \widetilde{ {\go l}}_0\otimes_{F}\overline{F}=\overline{{\widetilde{\go{g}}^{-\lambda_{0}}}}\oplus[\overline{{\widetilde{\go{g}}^{-\lambda_{0}}}},\overline{{\widetilde{\go{g}}^{\lambda_{0}}}}]\oplus \overline{{\widetilde{\go{g}}^{\lambda_{0}}}}$. By  \cite{M-R-S} (Proposition 2.16) the degree of  $\delta_{0}$ ($=\kappa$) is equal to the number of strongly orthogonal roots (over  $\overline{F}$),    $\beta_{0,1},\dots,\beta_{0,\kappa}$,  appearing in the descent  applied to the graded algebra  $\overline{\widetilde{ {\go l}}_0}$ (see \cite{M-R-S} for details). More generally let us denote by  $\beta_{i,1},\dots,\beta_{i,\kappa}$ the strongly orthogonal roots  appearing in the descent  applied to the graded algebra $\overline{\widetilde{ {\go l}}_i}$. But then the  set of $\kappa(k+1)$ roots  $\beta_{0,1},\dots,\beta_{0,\kappa},\dots,\beta_{k,1},\dots,\beta_{k,\kappa}$ is a maximal set of strongly orthogonal roots in  $\overline{V^+}$ (if not it would exist a root  $\lambda_{k+1}$ over  $F$ which is strongly orthogonal to $\lambda_{0},\dots,\lambda_{k}$). Then, again by Proposition 2.16 of  \cite{M-R-S}, one obtain that the degree of  $\Delta_{0}$ is  $\kappa(k+1)$.
\vskip 5pt
(2) For  $i=0,\dots,k$, let us fix elements  $X_{i}\in \widetilde{ {\go g}}^{\lambda _i}\setminus\{0\}$. Consider the polynomial map on  $\widetilde{ {\go g}}^{\lambda _i}$ given by 
$$\mu_{i}  :X\mapsto
\Delta _0(X+X^i)\text{ where } X^i=X_0+X_2+\cdots +X_{i-1}+X_{i+1}+\dots+X_{k}\ .$$
Let   $L_i$ be the the group similar to  $G$ for the graded algebra  ${\widetilde{ {\go l}}_i}$, that is  $L_{i}= {\cal Z}_{\text{Aut}_{0}({\widetilde{ {\go l}}_i})}(H_{\lambda_{i}})$.  By Corollary \ref{cor-decomp-j}   one gets that   ${\widetilde{ {\go l}}_k}={\cal Z}_{\widetilde{\go g}}(\widetilde{ {\go l}}_0\oplus  \dots\oplus  \widetilde{ {\go l}}_{k-1}$). As all the algebras   $\widetilde{ {\go l}}_{j}$ are conjugated   (Proposition \ref{propWconjugues}) one has also  ${\widetilde{ {\go l}}_i}={\cal Z}_{\widetilde{\go g}}(\widetilde{ {\go l}}_0\oplus \dots \oplus \widetilde{ {\go l}}_{i-1}\oplus \widetilde{ {\go l}}_{i+1}\oplus \dots \oplus \widetilde{ {\go l}}_k$). Therefore the elements of  $L_{i}$, which are products of exponentials (of adjoints)  of nilpotent elements of  ${\widetilde{ {\go l}}_i}\otimes \overline{F}$, fix all the element  $X^i$. 

  Since $L_i\subset \overline{G}=G(\overline{F})$ and since  $\Delta_0$ is,  by construction, the restriction to $V^+$ of a relative invariant polynomial of $(\overline{G}, \overline{V^+})$, we deduce that
 for  $g\in L_{i}$ and  $X\in \widetilde{ {\go g}}^{\lambda _i} $ one has :
$$\mu_{i}(g.X)=\Delta _0(g.X+X^i)=\Delta _0(g.(X+X^i))=\chi_{0}(g)\Delta _0(X+X^i)=\chi_{0}(g)\mu_{i}(X).$$

Hence  $\mu_{i}$ is a relative invariant for the prehomogeneous space  $(L_{i},\widetilde{ {\go g}}^{\lambda _i})$, with character  ${\chi_{0}}_{|_{L_{i}}}$.  This invariant is non zero, as  $\mu_{i}(X_{i})=\Delta_{0}(X_{0}+\dots+X_{k})\neq0$. For $t\in F$, let  $g_{0}(t)$ be an element of  $L_{0}$ such that $g_{0}(t)_{|_{\widetilde{ {\go g}}^{\lambda _0}}}=t\text{Id}_{\widetilde{ {\go g}}^{\lambda _0}}$ (Lemme \ref{lem-tId-dansG}). If  $g\in G$ is such that $g.\widetilde{ {\go l}}_0=\widetilde{ {\go l}}_i$,  then the element  $g_{i}(t)=gg_{0}(t)g^{-1}\in L_{i}$ satisfies
$g_{i}(t)_{|_{\widetilde{ {\go g}}^{\lambda _i}}}=t\text{Id}_{\widetilde{ {\go g}}^{\lambda _i}}$ and  $\chi_{0}(g_{i}(t))=\chi_{0}(g_{0}(t))$. Therefore the polynomials $\mu_{i}$ have the same homogeneous degree, say $p$. Then 

$\Delta_{0}(g_{0}(t)\dots g_{k}(t).(X_{0}+\dots+X_{k}))=\Delta_{0}(t(X_{0}+\dots+X_{k}))=t^{\kappa(k+1)}\Delta_{0}(X_{0}+\dots+X_{k})$

$=\chi_{0}(g_{0}(t))\Delta_{0}(X_{0}+g_{1}(t).X_{1}+g_{2}(t).X_{2}+\dots+g_{k}(t).X_{k})$

$=....$

$=\chi_{0}(g_{0}(t))\chi_{0}(g_{1}(t))\dots\chi_{0}(g_{k}(t))\Delta_{0}(X_{0}+\dots+X_{k})$

$=t^{p(k+1)}\Delta_{0}(X_{0}+\dots+X_{k})$

Hence $\kappa=p$ is the common degree of the $\mu_{i}$'s, and  $\mu_{i}=c_{i}\delta _{i}$, with  $c_{i}\in F^*$ (remind that  $\delta_{i}$ is the fundamental relative invariant of  $(L_{i},\widetilde{ {\go g}}^{\lambda _i})$).

Also for  $(x_{0},x_{1},\dots,x_{k})\in F^{k+1}$:

$\Delta_{0}(x_{0}X_{0}+\dots+x_{k}X_{k})=\prod _{j=0}^k x_j^\kappa\; . \Delta
          _0\ (X_{0}+\dots+X_{k})$

\end{proof}

    \vskip 20pt
 \subsection{The polynomials   $\Delta_{j}$}\label{sec:sectiondeltaj}\hfill
  \vskip 10pt
  
  Let  $j\in \{1,2,\dots,k\}$. By \ref{thexistedelta_0} applied to the graded regular algebra $(\widetilde{g}_{j}, H_{\lambda_{j}}+\dots+H_{\lambda_{k}})$,  there exists an absolutely irreducible polynomial  $P_{j}$ on $V^+_{j}$ which is relatively invariant under the action of  $G_{j}={\cal Z}_{\text{Aut}_{0}(\widetilde{\go{g}}_{j})}(H_{\lambda_{j}}+\dots+H_{\lambda_{k}})\subset  \overline{G}$.   By Corollary \ref{cor-structureG} applied to  $(\widetilde{g}_{j}, H_{\lambda_{j}}+\dots+H_{\lambda_{k}})$, one has $G_j=\textrm{Aut}_e(\go g_j).(\cap_{ s=j}^k (G_j)_{H_{\lambda_s}})$.
  
  Let $\chi _j$ \label{chij} \index{khij@$\chi _j$} be the corresponding character of $G_j$. We will define extensions of these polynomials to   $V^+$, using the following decomposition:
$$V^+=V^+_j\oplus V^\perp_j \hbox{  where }
V^\perp_j=\bigl(\plus _{r<s,r<j} E_{r,s}(1,1)\bigr)\oplus \bigl(\plus  _{r<j}\widetilde{ {\go
g}}^{\lambda _r}\bigr)\ .$$

\begin{definition} \label{defdeltaj} We denote by  $\Delta _j$ \index{Deltaj@$\Delta _j$}  the unique polynomial (up to scalar multiplication)    such that  
$$\Delta _j(X+Y)=P_j(X)\hbox{  for  }X\in V^+_j, Y\in V^\perp_j\ ,$$
where  $P_j$  is an absolutely irreducible polynomial on $V^+_j$,  relatively invariant  under the action of  $G_j$.
\end{definition}
\vskip 5pt

It may be noticed that, as $P_{j}$ is the restriction to $V^+_{j}$ of an irreducible polynomial on $\overline {V_{j}^+}$, which is relatively invariant under the action of $\overline{G_j}$, the polynomial $\Delta_{j}$ is   the restriction to $V^+$ of a polynomial defined on $\overline{V^+}$.\vskip 5pt 
\begin{theorem}\label{thproprideltaj} Let  $j,s\in \{0,1,\ldots ,k\}$ and let $X_s\in  \widetilde{ {\go g}}^{\lambda _s}\setminus \{0\}$.

\noindent {\rm (1)}  $\Delta _j$ is an absolutely irreducible polynomial of degree $\kappa (k+1-j)$.

\noindent {\rm (2)} For  $X\in V^+$ one has
 \begin{align*}
      \Delta _j(g.X)&=\chi _j(g)\Delta _j(X) \hbox{  for  }g\in G_j \ ;\cr
       \Delta _j(g.X)& =\Delta _j(X) \hbox{  for }g\in \hbox{Aut}_e(\go g_j) \ ;\cr
        \Delta _j(g.X)&=\Delta _j(X) \hbox{  for }g= \exp(\ad Z) \hbox{  where }Z\in
               \plus_{r<s}E_{r,s}(1,-1)\ .
\end{align*}
\noindent {\rm (3)}    For $s=0,\ldots ,k$,  let  $x_s\in F$.  Then, for $j=0,\ldots ,k$, 

$$\Delta _j\Bigl(\sum _{s=0}^k x_sX_s\Bigr)=\prod _{s=j}^k x_s^\kappa \; . \Delta
_j\Bigl(\sum _{s=j}^k X_s\Bigr)\ .$$


\noindent {\rm (4)}  The polynomial   $X\in
V^+_j\mapsto \Delta _0(X_0+X_1+\cdots +X_{j-1}+X)$ is non zero and  equal  (up to scalar multiplication) to the restriction of 
$\Delta _j$ to $V^+_j$.
\end{theorem}  

\begin{proof}\hfill

Statements $(1)$ and $(3)$ are just  Theorem  \ref{thpropridelta_0} applied to the graded algebra  $\widetilde{\go{g}}_{j}$.

As  ${\go{g}}_{j}={\cal Z}_{{\go{g}}}(\widetilde{ {\go l}}_0\oplus \widetilde{ {\go l}}_1\oplus \dots\oplus \widetilde{ {\go l}}_{j-1})$, it is easy to see that  $V^\perp_j$ is stable under  $\ad(\go{g}_{j})$. 
  Hence $\overline{G}_{j}.\overline{V}^\perp_j\subset \overline{V}^\perp_j$.  
  
  The first assertion in  $(2)$ is a consequence of the definition of  $\Delta_{j}$.

We know from  \textsection 1.7, that the groups ${\rm Aut}_e(\go g_j)$ and  ${\rm Aut}_0(\go g_j)$ are respectively isomorphic to   ${\rm Aut}_e([\go g_j, \go g_j])$ and  ${\rm Aut}_0([\go g_j,\go g_j])$. Then, from   (\cite{Bou2} Chap VIII \textsection $11$ $n^\circ 2$, Proposition 3 page 163), the group  ${\rm Aut}_e(\go g_j)$ is the derived group of  ${\rm Aut}_0(\go g_j)$. It follows that the character  $\chi_j$ is trivial on  ${\rm Aut}_e(\go g_j)$, this is the second assertion of  (2).\\

As any element of  $${\go n}_j\index{ngothiquej@${\go n}_j$}=\plus _{j\leq r<s} E_{r,s}(1,-1)\subset \go g_j $$ is nilpotent, we obtain that $\Delta _j$ is invariant under the action of the group generated by the elements    $\exp(\ad {\go n}_j)$.\\
 One has the decomposition:
$${\go n}_{0}=\plus _{r<s} E_{r,s}(1,-1)={\go n}_j\oplus K(0)\index{K0@$K(0)$}\oplus K(1)\index{K1@$K(1)$}$$

where 
$$K(0)=\plus_{r<s\leq j-1} E_{r,s} (1,-1) \text{ and  }K(1)=\plus_{  r \leq j-1,\,\, j\leq s} E_{r,s} (1,-1).$$

Note that  $$K(1)=\{X\in \go{n}, [H_{\lambda_{0}}+\ldots+H_{\lambda_{j-1}},X]=X\}$$
and 
$$\go{n}_{j}\oplus K(0)=\{X\in \go{n}, [H_{\lambda_{0}}+\ldots+H_{\lambda_{j-1}},X]=0\}.$$
Therefore, as  $[\go{n}_{j}, K(0)]=0$, we get   $[\go{n}_{j},K(0)\oplus K(1)]\subset K(1)$.

Hence any element of  $\exp(\ad(\go{n}))$ can be written $\exp(\ad Z')\exp(\ad Z)$ with $Z'\in \go{n}_{j}$ and  $Z\in K(0)\oplus K(1)$.  Therefore it is enough to show that  $\Delta_{j}$ is invariant by  $\exp(\ad Z)$.

But $[K(0),V_{j}^+]=[\plus_{r<s\leq j-1} E_{r,s} (1,-1), (\plus_{j\leq r'<s'} E_{r',s'} (1,1))\plus(\plus_{\ell=j}^k\widetilde{\go{g}}^{\lambda_{\ell}})]=\{0\}$. One has also  $[K(1),V^{+}_{j}]\subset V_{j}^{\perp}$ as $V^{+}_{j}$ is the eigenspace of $\ad(H_{\lambda_{j}}+\ldots+H_{\lambda_{k}})$ for the eigenvalue  $2$ and as  $K(1)$ is the eigenspace of  $\ad(H_{\lambda_{j}}+\ldots+H_{\lambda_{k}})$ for the eigenvalue $-1$. Hence for  $X\in V^{+}_{j}$ and for  $Z\in K(0)\plus K(1)$ one has $\exp(\ad Z).X=X+X'$, with $X'\in V_{j}^{\perp}$. This implies that  $\Delta_{j}(\exp (\ad Z).X=\Delta_{j}(X)$.  

It remains to prove $(4)$. Let $Q_{j}$ be the polynomial on  $V^+_{j}$ defined by 
$$Q_{j}(X)= \Delta_{0}(X_{0}+\ldots+X_{j-1}+X).$$

The polynomial $Q_{j}$ is non zero because $Q_{j}(X_{j}+\ldots+X_{k})= \Delta_{0}(X_{0}+\ldots+X_{k})\neq 0$ (as $X_{0}+\ldots+X_{k}$ is  generic by the criterion of Proposition \ref{prop-generiques-cas-regulier}).

 This polynomial  $Q_{j}$ is also relatively invariant under $G_{j}$ (because  $G_{j} \subset \bar{G}$ centralizes the elements  $X_{0},\ldots,X_{j-1}$). For $t\in F$, let $g_{t}$ be the element of  $G_{j}$ whose action on  $ V^+_{j}$ is $t.\text{Id}_{ {V}^+_{j}}$ (Lemma \ref{lem-tId-dansG}). By Theorem  \ref{thpropridelta_0} one has:
 $$Q_{j}(g_{t}X)=\Delta_{0}(X_{0}+\ldots+X_{j-1}+tX)=t^{\kappa(k-j+1)}Q_{j}(X).$$
 
 Hence $Q_{j}$ is a relative invariant of the same degree as $\Delta_{j}$. Therefore  $Q_{j}=\alpha \Delta_{j}$, with  $\alpha\in F^*$.
 
 \end{proof}

\vskip 30pt
  
 
 \section{Classification of regular graded Lie algebras}\label{section-classification}
 \vskip 10pt
 \subsection{General principles for the classification}\label{section2-1-principes}\hfill
 
 \vskip 5pt
 
 Our aim is to classify the regular graded Lie algebras defined in  Definition \ref{def-regulier}. The notations are those of section \ref{section-PH}. We have seen in section   \ref{section1-extension} that the graded algebra
  $$\overline{\widetilde {\go g}}=\overline{V^-}\oplus \overline{\go g}\oplus \overline{V^+}$$
  
  is a regular prehomogeneous space of commutative type, over the algebraically closed field  $\overline{F}$, defined by the data $(\widetilde{\Psi}, \alpha_{0})$. We associate to such an object the Dynkin diagram  $\widetilde{\Psi}$, on which the vertex corresponding to  $\alpha_{0}$ is circled. Such a diagram is called the weighted   Dynkin diagram  of the graded algebra  $\overline{\widetilde {\go g}}$. The classification is the same as over  $\C$. It was given in  \cite{M-R-S}. This is the list:
  
  \vskip 30pt
  
  \hskip 150pt{$A_{2n-1}$ \hskip 20pt  \hbox{\unitlength=0.5pt
\begin{picture}(250,30)
\put(10,10){\circle*{10}}
\put(15,10){\line (1,0){30}}
\put(50,10){\circle*{10}}
\put(60,10){\circle*{1}}
\put(65,10){\circle*{1}}
\put(70,10){\circle*{1}}
\put(45,-20){$n-1$}
\put(75,10){\circle*{1}}
\put(80,10){\circle*{1}}
\put(90,10){\circle*{10}}
\put(95,10){\line (1,0){30}}
\put(130,10){\circle*{10}}
\put(130,10){\circle{16}}
\put(135,10){\line (1,0){30}}
\put(170,10){\circle*{10}}
\put(180,10){\circle*{1}}
\put(185,10){\circle*{1}}
\put(190,10){\circle*{1}}
\put(170,-20){$n-1$}
\put(195,10){\circle*{1}}
\put(200,10){\circle*{1}}
\put(210,10){\circle*{10}}
\put(215,10){\line (1,0){30}}
\put(250,10){\circle*{10}}
\end{picture}
}
  
}
  \vskip 30pt
  
  \hskip 150pt{$B_{n}$ \hskip 20pt  \hbox{\unitlength=0.5pt
\begin{picture}(250,30)(-10,0)
\put(10,10){\circle*{10}}
\put(10,10){\circle{16}}
\put(15,10){\line (1,0){30}}
\put(50,10){\circle*{10}}
\put(55,10){\line (1,0){30}}
\put(90,10){\circle*{10}}
\put(95,10){\line (1,0){30}}
\put(130,10){\circle*{10}}
\put(135,10){\circle*{1}}
\put(140,10){\circle*{1}}
\put(145,10){\circle*{1}}
\put(150,10){\circle*{1}}
\put(155,10){\circle*{1}}
\put(160,10){\circle*{1}}
\put(165,10){\circle*{1}}
\put(170,10){\circle*{10}}
\put(174,12){\line (1,0){41}}
\put(174,8){\line (1,0){41}}
\put(190,5.5){$>$}
\put(220,10){\circle*{10}}
\end{picture}
}}
 \vskip 30pt
 
\hskip 150pt{$C_{n}$ \hskip 20pt  \hbox{\unitlength=0.5pt
\begin{picture}(250,30)
\put(10,10){\circle*{10}}
\put(15,10){\line (1,0){30}}
\put(50,10){\circle*{10}}
\put(55,10){\line (1,0){30}}
\put(90,10){\circle*{10}}
\put(95,10){\circle*{1}}
\put(100,10){\circle*{1}}
\put(105,10){\circle*{1}}
\put(110,10){\circle*{1}}
\put(115,10){\circle*{1}}
\put(120,10){\circle*{1}}
\put(125,10){\circle*{1}}
\put(130,10){\circle*{1}}
\put(135,10){\circle*{1}}
\put(140,10){\circle*{10}}
\put(145,10){\line (1,0){30}}
\put(180,10){\circle*{10}}
\put(184,12){\line (1,0){41}}
\put(184,8){\line(1,0){41}}
\put(200,5.5){$<$}
\put(230,10){\circle*{10}}
\put(230,10){\circle{16}}
\end{picture}
}}

 \vskip 30pt
 \hskip 150pt{$D_{n,1}$ \hskip 20pt \hbox{\unitlength=0.5pt
\begin{picture}(240,40)(0,0)
\put(10,10){\circle*{10}}
\put(10,10){\circle{16}}
\put(15,10){\line (1,0){30}}
\put(50,10){\circle*{10}}
\put(55,10){\line (1,0){30}}
\put(90,10){\circle*{10}}
\put(95,10){\line (1,0){30}}
\put(130,10){\circle*{10}}
\put(140,10){\circle*{1}}
\put(145,10){\circle*{1}}
\put(150,10){\circle*{1}}
\put(155,10){\circle*{1}}
\put(160,10){\circle*{1}}
\put(170,10){\circle*{10}}
\put(175,10){\line (1,0){30}}
\put(210,10){\circle*{10}}
\put(215,14){\line (1,1){20}}
\put(240,36){\circle*{10}}
\put(215,6){\line(1,-1){20}}
\put(240,-16){\circle*{10}}
\end{picture}
}}

 \vskip 30pt
 \hskip 150pt{$D_{2n,2}$ \hskip 20pt \hbox{\unitlength=0.5pt
\begin{picture}(240,40)(0,0)
\put(10,10){\circle*{10}}
\put(15,10){\line (1,0){30}}
\put(50,10){\circle*{10}}
\put(55,10){\line (1,0){30}}
\put(90,10){\circle*{10}}
\put(95,10){\line (1,0){30}}
\put(130,10){\circle*{10}}
\put(140,10){\circle*{1}}
\put(145,10){\circle*{1}}
\put(150,10){\circle*{1}}
\put(155,10){\circle*{1}}
\put(160,10){\circle*{1}}
\put(170,10){\circle*{10}}
\put(175,10){\line (1,0){30}}
\put(210,10){\circle*{10}}
\put(215,14){\line (1,1){20}}
\put(240,36){\circle*{10}}
\put(240,36){\circle{16}}
\put(215,6){\line (1,-1){20}}
\put(240,-16){\circle*{10}}
\end{picture}
}
}\hskip 15pt  ($2n$ vertices)

 \vskip 30pt
 \hskip 150pt{$E_{7}$ \hskip 25pt \hbox{\unitlength=0.5pt
\begin{picture}(240,40)(-10,0)
\put(10,10){\circle*{10}}
\put(15,10){\line (1,0){30}}
\put(50,10){\circle*{10}}
\put(55,10){\line (1,0){30}}
\put(90,10){\circle*{10}}
\put(90,5){\line (0,-1){30}}
\put(90,-30){\circle*{10}}
\put(95,10){\line (1,0){30}}
\put(130,10){\circle*{10}}
\put(135,10){\line (1,0){30}}
\put(170,10){\circle*{10}}
\put(175,10){\line (1,0){30}}
\put(210,10){\circle*{10}}
\put(210,10){\circle{16}}
\end{picture}
}
} \vskip 40pt
 
  Remind that the circled root    $\alpha_{0}$ is the unique root in $\widetilde{\Psi}$ whose restriction to  $\go{a}$ is the root  $\lambda_{0}$.
 
 Therefore the  Satake-Tits diagram of $\widetilde{\go{g}}$ is such that $\alpha_{0}$ is a white root  which is not connected by an arrow to another white root. We associate to the regular graded Lie algebra $(\widetilde{\go{g}},H_{0})$, the  Satake-Tits diagram of $\widetilde{\go{g}}$ where the white root  $\alpha_{0}$ is circled. Such a diagram will be called the weighted  Satake-Tits diagram of $(\widetilde{\go{g}},H_{0})$.  As the list of the allowed  Satake-Tits diagrams over $F$ are well known (\cite {Schoeneberg}, \cite {Tits})  it is easy to give the list of the weighted Satake-Tits Diagrams.
 \vskip 10pt
Conversely suppose we are given  a  weighted Satake-Tits diagram $\bb D$ such that   the underlying  weighted Dynkin diagram is in the list above. Take  $\widetilde{\go{g}}$  any Lie algebra over $F$ whose Satake-Tits diagram is the underlying diagram of $\bb D$. Define $H_{0}$ to be the unique element in $ \go{a}$ satisfying the equations  $\alpha_{0}(H_{0})=2$ ($\alpha_{0}$ being the circled root) and $\beta(H_{0})= 0$ if  $\beta$ is one of the other simple roots. Then $(\widetilde{\go{g}},H_{0})$ is a graded Lie algebra satisfying the conditions in Definition \ref {def-alg-grad-1} and its weighted Satake-Tits diagram is $\bb D$.
  
The list of the   weighted Satake-Tits diagrams is  part of the Table 1 below. At this place we must keep in mind that unlike  the real case, if $F$ is a p-adic local field of characteristic $0$, the Satake-Tits diagram does not completely determine the isomorphism class of $\widetilde{\go{g}}$ (see \cite{Schoeneberg} Chap. 3,4). The classification of the (absolutely simple) graded Lie algebras will then be complete if we are able to give the description, for each weighted Satake-Tits diagram, of the possible  Lie algebras.
    
Beside the Satake-Tits diagram, two other ingredients are needed for the classification of the simple algebras over, namely the so-called {\it anisotropic kernel} which is the anisotropic subalgebra corresponding to the black roots and the action of the Galois group Gal($\overline F;F)$ on the root system, which is partly encoded by the arrows in the diagram (\cite{Schoeneberg} Chap. 3, 4).
     
If we consider the diagrams arising in Table 1 we split them into four subsets:
     \vskip 10pt
     
$a)$ The diagrams with only white roots and no arrows (case (3), (6), (8), (11) (13) in Table 1). According to the discussion above, each of  these cases corresponds to only one isomorphism class of simple algebras. The corresponding Lie algebras are described in Section \ref{sec:section-Table} and in Table 1. 
  \vskip 10pt
$b)$ The diagrams where the black roots are isolated in the sense that they are only connected  to white roots (cases (4), (5), (7), (10), (12), we note also that in these cases there are no arrows in the diagram). This means that the anisotropic kernel is a direct product of algebras corresponding to the diagram ``$\bullet$''. It is known (see for example \cite {Schoeneberg}, section 2.2 and Example 3.2.5) that such an algebra is the derived Lie algebra of the unique quaternion division algebra over $F$. Hence, again, the corresponding Satake-Tits diagrams determines only one isomorphism class of simple algebras. See Section \ref{sec:section-Table} and   Table 1.
  \vskip 10pt
$c)$ The case (1):

\hbox{\unitlength=0.4pt
\begin{picture}(310,00)
\put(10,10){${\linethickness{1mm}\line(1,0){20}}$}
 \put(30,10){\line (1,0){17}}
 \put(52,10){\circle{10}}
 \put(57,10){\line (1,0){17}}
 \put(74,10){${\linethickness{1mm}\line(1,0){20}}$}
 \put(99,10){\circle*{1}}
\put(104,10){\circle*{1}}
\put(109,10){\circle*{1}}
\put(114,10){\circle*{1}}
\put(119,10){\circle*{1}}
\put(122,10){${\linethickness{1mm}\line(1,0){20}}$}
 \put(142,10){\line (1,0){15}}
 \put(164,10){\circle{10}}
 \put(164,10){\circle{16}}
 \put(171,10){\line (1,0){15}}
 \put(186,10){${\linethickness{1mm}\line(1,0){20}}$}
 \put(211,10){\circle*{1}}
\put(216,10){\circle*{1}}
\put(221,10){\circle*{1}}
\put(226,10){\circle*{1}}
\put(231,10){\circle*{1}}
\put(233,10){${\linethickness{1mm}\line(1,0){20}}$}
 \put(253,10){\line (1,0){17}}
 \put(275,10){\circle{10}}
  \put(280,10){\line (1,0){17}}
 \put(297,10){${\linethickness{1mm}\line(1,0){20}}$}
 \end{picture}
}  (2k+1 white vertices),
\text{where} \hbox{\unitlength=0.4pt
\begin{picture}(30,30)
\put(10,10){${\linethickness{1mm}\line(1,0){20}}$}\end{picture}} = \hbox{\unitlength=0.4pt
\begin{picture}(200,00)
\put(10,10){\circle*{10}} \put(15,10){\line(1,0){20}}  \put(40,10){\circle*{10}}\put(50,10){\circle*{1}}\put(55,10){\circle*{1}}\put(60,10){\circle*{1}}\put(65,10){\circle*{1}}\put(70,10){\circle*{1}}\put(75,10){\circle*{10}}\put(80,10){\line(1,0){20}}\put(105,10){\circle*{10}} \put(30,-15){$\delta-1$} \put(120,10){ ($\delta\in \N^*)$.} \end{picture}}

This case corresponds to $\widetilde{\go{g}}=\go{sl}(2(k+1),D)$, where $D$ is a central division algebra over $F$ of dimension $\delta^2$. If $\varphi$  denotes the Euler function, it is known that there are $\varphi(\delta)$ non isomorphic central division algebras over $F$ of dimension $\delta^2$ (\cite{Blanchard}, Corollaire 5.2)  and hence there are $\varphi(\delta)$ non isomorphic graded Lie algebras corresponding to this weighted Satake-Tits diagram. 
  \vskip 10pt
d) The quasi-split (non-split) Satake-Tits diagrams (i.e. there are only white roots in the diagram, but there are arrows). This corresponds to the following two cases.

$\bullet$ case (2):

\vskip10pt
$A_{2n-1}$ \hbox{\unitlength=0.5pt
\begin{picture}(250,30)(-50,0)
\put(10,40){\circle{10}}
\put(15,40){\line (1,0){30}}
\put(50,40){\circle{10}}
\put(60,40){\circle*{1}}
\put(65,40){\circle*{1}}
\put(70,40){\circle*{1}}
\put(75,40){\circle*{1}}
\put(80,40){\circle*{1}}
\put(85,40){\circle*{1}}
\put(90,40){\circle*{1}}
\put(95,40){\circle*{1}}
\put(100,40){\circle*{1}}
\put(110,40){\circle{10}}
\put(115,40){\line (1,-1){21}}
\put(140,15){\circle{10}}
\put(140,15){\circle{16}}
\put(115,-10){\line (1,1){21.5}}
\put(110,-10){\circle{10}}
\put(110,8){\vector(0,1){23}}
\put(110,22){\vector(0,-1){23}}
\put(100,-10){\circle*{1}}
\put(95,-10){\circle*{1}}
\put(90,-10){\circle*{1}}
\put(85,-10){\circle*{1}}
\put(80,-10){\circle*{1}}
\put(75,-10){\circle*{1}}
\put(70,-10){\circle*{1}}
\put(65,-10){\circle*{1}}
\put(60,-10){\circle*{1}}
\put(50,-10){\circle{10}}
\put(50,8){\vector(0,1){23}}
\put(50,22){\vector(0,-1){23}}
\put(15,-10){\line (1,0){30}}
\put(10,-10){\circle{10}}
\put(10,8){\vector(0,1){23}}
\put(10,22){\vector(0,-1){23}}
\end{picture}
}
 
\vskip 5pt 
 According to \cite{Schoeneberg} (Remark 3.2.10. and Remark 3.2.11.) this  diagram corresponds to special unitary Lie algebras on $E^{2n}$, where $E$ is any quadratic extension of the base field $F$. See Section \ref  {sec:section-Table} .

\vskip 10pt

$\bullet$ Case (9):

\vskip10pt
$D_{m}$ \hskip 30pt 
\hbox{\unitlength=0.5pt
\begin{picture}(240,40)(0,-10)
\put(10,10){\circle{10}}
\put(10,10){\circle{16}}
\put(15,10){\line (1,0){30}}
\put(50,10){\circle{10}}
\put(55,10){\line (1,0){30}}
\put(90,10){\circle{10}}
\put(95,10){\line (1,0){30}}
\put(130,10){\circle{10}}
\put(140,10){\circle*{1}}
\put(145,10){\circle*{1}}
\put(150,10){\circle*{1}}
\put(155,10){\circle*{1}}
\put(160,10){\circle*{1}}
\put(170,10){\circle{10}}
\put(175,10){\line (1,0){30}}
\put(210,10){\circle{10}}
\put(215,14){\line (1,1){20}}
\put(240,36){\circle{10}}
\put(215,6){\line (1,-1){20}}
\put(240,-16){\circle{10}}
\put(240,2){\vector(0,1){23}}
\put(240,16){\vector(0,-1){23}}
\end{picture}
}
\vskip 5 pt
According to \cite {Schoeneberg} (section 4.5.1.) these cases correspond to orthogonal Lie algebras for quadratic forms of Witt index $m+1$ in $F^{2m}$ (see Section \ref{sec:section-Table}).

 
 

\vskip 10 pt
 
We will now give a simple ``diagrammatic''   or ``combinatorial`` algorithm which allows to determine the weighted Satake-Tits diagram of $\widetilde{\go{g}}_{1}$ (section \ref{section-descente1}) from the diagram of  $\widetilde{\go{g}}$. By induction this algorithm will give    the 1-type (Definition \ref{def-1-type}). This algorithm allows also to determine easily the rank of the graded algebra.  (Remark  \ref{rem-combinatoire} b) and c) below).

For the definition of the extended Dynkin diagram see \cite{Bou1}, Chap. VI \S4 $n^\circ 3$ p.198.

 \begin{prop}\label{prop-diagg1}\hfill

   Let us make the following  operations on the Satake-Tits diagram of the regular graded lie algebra $\widetilde{\go{g}}$:
 
 $1)$ One extends  the  Satake-Tits diagram by considering the underlying extended Dynkin diagram where the additional root  $-\omega$  is white ($\omega$ being the greatest root of $ \widetilde{ {\cal R}}^+$).
 
$ 2)$ One removes the vertex  $\alpha_{0}$ (the circled root), as well as all the white vertices  which are connected to  $\alpha_{0}$ through a chain of black vertices, and one removes also these  black vertices.
 
$ 3)$ One circles the vertex  $-\omega$ which has been added.
 
The diagram which is obtained after these three operations   is the weighted Satake-Tits diagram of  $\widetilde{\go{g}}_{1}$.

 \end{prop}
 
 \begin{proof}
 
 Remind that (Proposition \ref{prop-gtilde(1)}) 
 $$ \widetilde{ {\cal R}}_1=\{\beta \in \widetilde{ {\cal R}}\mid  \ \beta \sorth
\alpha, \, \forall \alpha \in S_{\lambda_{0}} \}.$$

Hence 
$${\cal R}_1= \widetilde{{\cal R}}_1\cap{\cal R}=\{\beta \in  { {\cal R}}\mid  \ \beta \sorth
\alpha, \, \forall \alpha \in S_{\lambda_{0}} \}.$$

Remind also that we have defined the following sets of positive roots:
$$\widetilde{ {\cal R}}_1^+=\widetilde{ {\cal R}}_1\cap \widetilde{ {\cal R}}^+,\,\, {\cal R}_1^+={\cal R}_{1}\cap  \widetilde{ {\cal R}}^+=\widetilde{ {\cal R}}_1\cap {\cal R}^+,$$
which correspond to the basis  $\widetilde{\Psi}_{1}$ and  $\Psi_{1}$ of  $ \widetilde{ {\cal R}}_1$ and  ${\cal R}_1$ respectively.

Let us first prove a Lemma.
 
 \begin{lemme}\label{lemme-base-R1}\hfill
 
 One has  $\Psi_{1}= {\cal R}_{1}\cap \Psi$.

 \end{lemme}

Proof of the Lemma: Let  $\langle {\cal R}_{1}\cap \Psi \rangle ^+$ be the set of positive linear combinations of elements in  ${\cal R}_{1}\cap \Psi$.
It is enough to show that  $\langle {\cal R}_{1}\cap \Psi \rangle ^+= {\cal R}_{1}^+$.

The inclusion  $\langle {\cal R}_{1}\cap \Psi \rangle ^+\subset {\cal R}_{1}^+$ is obvious. Conversely let  $\beta\in {\cal R}_{1}^+$. Let us write  $\beta$ as a positive linear combination of elements of  $\Psi$:
$$\beta=\sum_{\beta_{i}\in \Psi}x_{i}\beta_{i},\hskip 20pt x_{i}\in \N.$$
As  $\beta$ is orthogonal  to any root which restricts to  $\lambda_{0}$, one has  
$$\langle \beta, \alpha_{0}\rangle=0=\sum_{\beta_{i}\in \Psi}x_{i}\langle \beta_{i}, \alpha_{0}\rangle.$$
As $\langle \beta_{i}, \alpha_{0}\rangle \leq 0$ (scalar product of two roots in the same base), one gets
$$\beta=\sum_{\beta_{i}\in \Psi, \beta_{i}\perp \alpha_{0}}x_{i}\beta_{i}.$$
Let now  $\{\gamma_{1},\gamma_{2},\ldots,\gamma_{m}\}$ be the set of black roots which are connected to  $\alpha_{0}$ through a chain of black roots  (in the   Satake-Tits diagram). Moreover we suppose that this set of roots is ordered  in such a way that $\alpha_{0}+\gamma_{1}+\ldots+\gamma_{p}\in \widetilde{\cal R}^+$ for all $p=1,\ldots,m$. This is always possible. 

We will now show by induction on $j$ that if  $\beta_{i}$ belongs to the support of $\beta$ then  $\beta_{i}\perp \gamma_{j}$ for $j=0,\ldots,m$ (where we have set  $\gamma_{0}= \alpha_{0}$) . What we have done before is the first step of the induction.
Suppose that  $\beta_{i}\perp \gamma_{j}$, for all  $j\leq p$ ($p\in \{ 0,\ldots, m-1\}$). As $\beta\in {\cal R}_{1}^+$, using the induction hypothesis, one gets:
$$\langle \beta, \alpha_{0}+\gamma_{1}+\ldots+\gamma_{p+1}\rangle=0=\sum_{\beta_{i}\in \Psi, \beta_{i}\perp \gamma_{j}\,(j=0,\ldots,m-1)}x_{i}\langle \beta_{i}, \gamma_{p+1}\rangle.$$
Again, as  $\langle \beta, \gamma_{p+1}\rangle \leq 0$, one obtains 
$$\beta=\sum_{\beta_{i}\in \Psi, \beta_{i}\perp \gamma_{j}\,(j=0,\ldots,m)}x_{i} \beta_{i}.$$
\vskip 5 pt
End of the proof of the Lemma.
\vskip 5pt

Hence  $\widetilde{\Psi}_{1}=\Psi_{1}\cup \{\alpha_{1}\}$ where  $\alpha_{1}$ is the unique root of   $\widetilde{\Psi}_{1}$ such that  $\alpha_{1}(H_{1})=2$ (Proposition  \ref{prop.alpha0}). Let $W_{1}$ be the Weyl group of  ${\cal R}_{1}$, and let $w_{1}$ be the unique element in  $W_{1}$ such that  $w_{1}(\Psi_{1})=-\Psi_{1}$. Then  Proposition \ref{propWconjugues} (in its ``absolute`` version over $\overline{F}$, whose proof is exactly the same) implies that $w_{1}(\alpha_{1})=\omega$ where  $\omega$ is the greatest root of $\widetilde{ {\cal R}}_1$, which is also the greatest root of $\widetilde{ {\cal R}}$.

Let us denote  by $\rm{Dyn}(\,.\,)$ the Dynkin diagram of the basis ''$. $'' of a root system. One has:
$${\rm Dyn}(\widetilde{\Psi}_{1})={\rm Dyn}(\Psi_{1}\cup \{\alpha_{1}\})={\rm Dyn}(w_{1}(\Psi_{1}\cup \{\alpha_{1}\})={\rm Dyn}(-\Psi_{1}\cup \{\omega\})={\rm Dyn}(\Psi_{1}\cup \{-\omega\}).$$

The preceding Lemma implies that the Dynkin diagram of  $\widetilde{\go{g}}_{1}$ is the underlying Dynkin diagram  of the Satake-Tits diagram described in the statement. It remains to show that the  ``{\it colors }'' (black or white) are the right one.  

For this,   let  $X$ be the root lattice of  $\widetilde{\cal R}$ (that is the $\Z$-module generated by  $\widetilde{\cal R}$). The  restriction morphism  $\rho$ extends to a surjective morphism:
$$\rho:X\longrightarrow \overline{X},$$
where  $ \overline{X}$ is the root lattice of $\Sigma$ (in  $\go{a}^*$).  Consider also the sublattice 
$$X_{0}=\{\chi \in X\,|\, \rho(\chi)=0\}.$$
The map $\rho$ induces a bijective morphism: $\overline{\rho}: X/X_{0}\longrightarrow \overline{X}$. Set
 $${\cal R}_{0}=\{\alpha\in {\cal R}\,|\, \rho(\alpha)=0\}=\{\alpha\in \widetilde{\cal R}\,|\, \rho(\alpha)=0\}.$$

The order on  ${\cal R}$ induces an order on  ${\cal R}_{0}$  by setting  ${\cal R}_{0}^+={\cal R}^+\cap {\cal R}_{0}$. We choose an additive order   $\leq$ on $X_{0}$ in such a way that  ${\cal R}_{0}^+\subset X_{0}^+$ (for the definition of an additive order on a lattice and for the notation we refer to the paper by Schoeneberg (\cite{Schoeneberg}, p.37), who call it  ``group linear order''). For this it is enough to consider an hyperplane in the vector space generated by $X_{0}$, whose intersection with  $X_{0}$ is reduced to $\{0\}$ and such that ${\cal R}_{0}^+$ is contained in one of the half-spaces defined by this hyperplane. $X_{0}^+$ is then defined  as the intersection of  $X_{0}$ with this half-space.

Similarly we choose an additive order on the lattice  $\overline{X}$ such that  $\Sigma^+\subset \overline{X}^+$ and we set :
$$(X/X_{0})^+=\{\overline{\chi}\in X/X_{0}\,|\,\overline{\rho}(\overline{\chi})\in \overline{X}^+\}$$.

Let  $\Gamma$ be the Galois group of the finite Galois extension  on which  $\widetilde{\go{g}}$ splits (\cite{Schoeneberg}, p. 29). The data  $(X/X_{0})^+$ and $X_{0}^+$ define what is called by Schoeneberg a $\Gamma$-order on $X$ (\cite{Schoeneberg}, Definitions 3.1.37 and 3.1.38 p. 37). This additive order is defined by 
$$X^+= (X/X_{0})^+ \cup X_{0}^+$$
where  $(X/X_{0})^+$ stands here for  the set of elements which are in a strictly positive class. If $\alpha\in \widetilde{\cal R}^+$, then either $\rho(\alpha)=0$, and in this case    $\alpha\in X_{0}^+\subset X^+$, or  $\rho(\alpha)\in \Sigma^+$ and then   $\alpha\in X^+$. This shows that the order chosen at the beginning of this paper  and which was defined by $\widetilde{\cal R}^+$ comes from a  $\Gamma$-order in the sense of  Schoeneberg on the corresponding root lattices.  Similarly, at step $1$ of the descent,  the  order defined by $\widetilde{\cal R}_{1}^+$ comes from a  $\Gamma$-order. Then from the proof of  Lemma 4.3.1 p. 72 of  \cite{Schoeneberg}, we obtain that  $w_{1}\in (W_{1})_{\Gamma}=\{w\in W({\cal R}_{1}), w(X^1_{0})=X^1_{0}\}$ (where $X^1$ is the root lattice of  ${\cal R}_{1}$,  and  $X^1_{0}$ is the subset of  $X^1$ which vanish on $\go{a}_{1}$), and that $w_{1}$ sends a black root on a  black  root  and a white root  on a white root.
 This ends the proof.

 \end{proof}
 \begin{rem}\label{rem-combinatoire}\hfill
 
 a) As expected, if one applies the procedure of  Proposition \ref{prop-diagg1} to the Satake-Tits diagrams of  Lemma \ref {lemme-diagrammes-k=0}, one  obtains the empty diagram.
 
 b) It is worth noting that the rank of  $\widetilde{\go{g}}$ (cf. Definition \ref{def-rang}) is the number of times one must apply the procedure of  Proposition \ref{prop-diagg1} until one obtains the empty diagram.
 
 c) The last diagram,    obtained before the empty diagram, when one iterates this procedure, is necessarily one of the two diagrams of Lemma \ref{lemme-diagrammes-k=0}. It defines therefore the $1$-type de $\widetilde{\go{g}}$.
 
 d) It may happen that the iteration of the procedure of  Proposition \ref{prop-diagg1}  gives a non-connected  Satake-Tits diagram   (see example below). In that case the next iteration  is made only on the connected component containing the circled root.  
 \end{rem}
 
 \begin{example}
 
 \end{example} The following split diagram corresponds to a graded algebra  $\widetilde{\go{g}}$ verifying the hypothesis  (${\bf H}_{1}$), (${\bf H}_{2}$), (${\bf H}_{3}$):

 \centerline{ \hbox{\unitlength=0.5pt
\begin{picture}(250,30)(-10,0)
\put(10,10){\circle{10}}
\put(10,10){\circle{16}}
\put(15,10){\line (1,0){30}}
\put(50,10){\circle{10}}
\put(60,10){\circle*{1}}
\put(65,10){\circle*{1}}
\put(70,10){\circle*{1}}
\put(75,10){\circle*{1}}
\put(80,10){\circle*{1}}
\put(85,10){\circle*{1}}
\put(95,10){\circle{10}}
\put(100,10){\line (1,0){30}}
\put(135,10){\circle{10}}
\put(140,10){\circle*{1}}
\put(145,10){\circle*{1}}
\put(150,10){\circle*{1}}
\put(155,10){\circle*{1}}
\put(160,10){\circle*{1}}
\put(165,10){\circle*{1}}
\put(170,10){\circle*{1}}
\put(175,10){\circle{10}}
\put(179,12){\line (1,0){41}}
\put(179,8){\line (1,0){41}}
\put(195,5.5){$>$}
\put(225,10){\circle{10}}
\end{picture} $B_{n}$
}
}

 \vskip 20pt
 The extended diagram is:
 
 \vskip 10pt
\centerline{ \hbox{\unitlength=0.5pt
\begin{picture}(250,30)(-10,0)
\put(10,10){\circle{10}}
\put(10,10){\circle{16}}
\put(15,10){\line (1,0){30}}
\put(50,10){\circle{10}}
\put(50,5){\line(0,-1){31}}
\put(50,-30){\circle{10}}
\put(60,10){\circle*{1}}
\put(65,10){\circle*{1}}
\put(70,10){\circle*{1}}
\put(75,10){\circle*{1}}
\put(80,10){\circle*{1}}
\put(85,10){\circle*{1}}
\put(95,10){\circle{10}}
\put(100,10){\line (1,0){30}}
\put(135,10){\circle{10}}
\put(140,10){\circle*{1}}
\put(145,10){\circle*{1}}
\put(150,10){\circle*{1}}
\put(155,10){\circle*{1}}
\put(160,10){\circle*{1}}
\put(165,10){\circle*{1}}
\put(170,10){\circle*{1}}
\put(175,10){\circle{10}}
\put(179,12){\line (1,0){41}}
\put(179,8){\line (1,0){41}}
\put(195,5.5){$>$}
\put(225,10){\circle{10}}
\end{picture}
}
}  

\vskip 20pt

The diagram obtained by applying the procedure of Proposition  \ref{prop-diagg1} is : 
\vskip 20pt

 \centerline{ \hbox{\unitlength=0.5pt
\begin{picture}(30,30)(-10,0)
\put(10,10){\circle{10}}
\put(10,10){\circle{16}}
\end{picture} $A_{1}$
}}

\hbox{\unitlength=0.5pt
\begin{picture}(250,30)(-10,0)
\put(50,10){\circle{10}}
\put(60,10){\circle*{1}}
\put(65,10){\circle*{1}}
\put(70,10){\circle*{1}}
\put(75,10){\circle*{1}}
\put(80,10){\circle*{1}}
\put(85,10){\circle*{1}}
\put(95,10){\circle{10}}
\put(100,10){\line (1,0){30}}
\put(135,10){\circle{10}}
\put(140,10){\circle*{1}}
\put(145,10){\circle*{1}}
\put(150,10){\circle*{1}}
\put(155,10){\circle*{1}}
\put(160,10){\circle*{1}}
\put(165,10){\circle*{1}}
\put(170,10){\circle*{1}}
\put(175,10){\circle{10}}
\put(179,12){\line (1,0){41}}
\put(179,8){\line (1,0){41}}
\put(195,5.5){$>$}
\put(225,10){\circle{10}}
\end{picture} $B_{n-2}$
}

\vskip 20pt

When one applies again  the procedure  to the diagram   \hbox{\unitlength=0.5pt
\begin{picture}(30,30)(-10,0)
\put(10,10){\circle{10}}
\put(10,10){\circle{16}}
\end{picture}  
}, one obtains the empty diagram. Hence the rank is $2$. 
 \subsection{Notations and  Table }\label{sec:section-Table}\hfill
 
 \vskip 5pt

\begin{notation}\label{notations-tables} {\bf (Notations for Table 1)}\hfill
 
 \vskip 5pt

We first  define the {\it type} \index{type}of the graded Lie algebra $\tilde{\go g}$ according to possible values which can be taken by $e$ and $\ell$. This notion of type allows to split the classification of graded Lie algebras according to the number of $G$-orbits in $V^+$.

\begin{definition}\label{def-type} \footnote {We caution the reader that our definition of type in the non archimedean case is not related to the definition of type in the archimedean($F=\R)$  case given  in \cite{B-R}. Our definition of type is related to the structure of the  open $G$-orbits (see Theorem \ref{thm-orbouverte} below).} .\hfill 

$\bullet$ $\tilde{\go g}$ is said to be of type I if  $\ell=\delta^2,\;\delta\in\N^*$  and   $e=0$ or $4$. \\
$\bullet$ $\tilde{\go g}$  is said to be of type II if $\ell=1$ and  $e=1$, $2$ or $3$,\\
$\bullet$ $\tilde{\go g}$  is said to be of type III if $\ell=3$. 

 \vskip 5pt

  We say that  $\tilde{\go g}$ is  of    $1$-type $(A,1)$ if $\tilde{\go l}_0$ is isomorphic to $\go sl_2(D)$, where $D$ is a central division algebra over $F$, and of $1$-type $B$, if  $\tilde{\go l}_0$ is isomorphic to ${\go o}(q_{(4,1)})$.

\end{definition}

\vskip 10pt
The following notations are used in Table 1.

\vskip 5pt
- We denote always by $D$ a central division algebra over $F$. Its degree is denoted by  $\delta$ (remind that this means that the dimension is $\delta^2$). If its degree is  $2$, then  $D$ is necessarily the unique quaternion  division algebra over $F$.
\vskip 5pt

- $M(m,D)=M_{m}(D)$ is the algebra    of $m\times m$ matrices with coefficients in $D$.
\vskip 5pt
- $\go{sl}(m,D)$ is the derived Lie algebra of $M(m,D)$. It is also the space of matrices in  $M(m,D)$ whose reduced trace is zero. (Recall that if  $x=(x_{i,j})\in M(m,D)$, $Tr_{red}(x)=\sum \tau(x_{i,i})$ where  $\tau$ is the reduced trace of  $D$. (\cite{Weil}, IX,\S2, Corollaire 2 p.169). Recall also that the reduced trace of the quaternion division algebra is  $\tau(x)=x+\overline{x}$.
\vskip 5pt
- $E=F(y)$ is a quadratic extension of  $F$ and  the canonical conjugation $\sigma$ on $E$ is given by  $E$: $\sigma(a+by)=a-by.$

- $H_{2n}$ is the hermitian form on  $E^{2n}$ defined by  
$H_{2n}(u,v)={^tu}S_{n}\sigma(v) $ where $u,v$ are columns vectors of  $E^{2n}$ and where  $S_n=\left(\begin{array}{cc}0 & I_n\\ I_n & 0\end{array}\right)$.  

\vskip 5pt

- $\go{su}(2n,E,H_{2n})$$=\{X\in \go{sl}(2n,E), XS_{n}+S_{n}{^t(\sigma(X))}=0\}$ (this is the so-called unitary algebra of the form $H_{2n}$)
\vskip 5pt
- ${\rm Herm}_{\sigma}(n,E)$ is defined as follows:
$${\rm Herm}_{\sigma}(n,E)=\{U\in M(n,E), \, {^t\sigma(U)}=U\}$$
\vskip 5pt

- $q_{(p,q)}$, with  $p\geq q$ is a quadratic form of  Witt index $q$ on $F^{p+q}$. $\go{o}(q_{(p,q)})$ is the corresponding orthogonal algebra.

\vskip 5pt

- $\mathrm{Skew}(2n,F)$ is the space of skew-symmetric matrices of type $2n\times 2n$  with coefficients in $F$.

\vskip 5pt

- $\go{sp}(2n,F)$ is the usual symplectic Lie algebra (the matrices in it being of type $2n\times 2n$, with coefficients in  $F$). 
\vskip 5pt

 - ${\rm Sym}(n,F)$ is the space of symmetric matrices of type $n\times n$ with coefficients in  $F$.
\vskip 5pt

- The  quaternion division algebra over $F$ is denoted by  $\H$   and its canonical anti-involution is denoted by $\gamma:x\mapsto \bar{x}$.
\vskip 5pt
  - On  $\H^{2n}$ we denote also by $H_{2n}$ the hermitian form defined by  
$H_{2n}(u,v)={^t\gamma(u)}S_{n}v $ where $u,v$ are columns vectors of  $\H^{2n}$, and where, as above, $S_n=\left(\begin{array}{cc}0 & I_n\\ I_n & 0\end{array}\right)$.

  - $\go{u}(2n, \H, H_{2n})=\{A\in M(2n, \H), \, {^t\ga(A)}S_{n}+S_{n}A=0\}$
\vskip 5pt
  - ${\rm SkewHerm}(n,\mathbb H)=\{A\in M(n, \H), \, {^t\ga(A)}+A=0\}$

\vskip 5pt
  - ${\rm Herm}(n,\mathbb H)=\{A\in M(n, \H), \, {^t\ga(A)}=A\}$

\vskip 5pt

  - $K_{2n}$ is the $\gamma$-skewhermitian form on  $\H^{2n}$ defined by  
$K_{2n}(u,v)={^t\gamma(u)}K_{2n}v $ where $u,v$ are columns vectors of  $\H^{2n}$ and where, by abuse of notation,we also set  $K_{2n}=\left(\begin{array}{cc} 0 & I_n\\ -I_n & 0\end{array}\right)$.
 \vskip 5pt - $\go{u}(2n, \H, K_{2n})=\{A\in M(2n, \H), \, {^t\gamma(A)}K_{2n}+K_{2n}A=0\}$.
\end{notation}

\include{2zetaracinescopie}

\include{2zetatablescor1}

 \newpage
  \section {The  $G$-orbits  in $V^+$}\label{sec:G-orbites}
\subsection{Representations of  $\go {sl}(2,F)$}\label{sl2module}\hfill
 \vskip 10pt

(For the convenience of the reader we recall here some classical facts about $\go {sl}_2$ modules, see \cite{Bou2}, chap.VIII, \textsection 1, Proposition and Corollaire). \medskip

Let  $(Y,H,X)$ be an  $\go {sl}_2$-triple 
 in $ \widetilde{\go g}$ and let  $E$ be a subspace of $ \widetilde{\go g}$ which is invariant under this  $\go {sl}_2$-triple. \\
Let  $M$ be an irreducible submodule of dimension  $m+1$ of  $E$, then $M$ is generated by a primitive element $e_0$ (ie. $X.e_0=0$) of weight  $m$ under the action of  $H$ and the set of elements  $e_p:=\dfrac{(-1)^p}{p!} (Y)^p.e_0$, of weight $m-2p$ under the action of  $H$,  is a basis  of$M$.

We have also the   weight decomposition  $\displaystyle E=\oplus_{p\in\Z} E_p$ into weight spaces of weight  $p$ under the action of  $H$. The following properties are classical:

\begin{enumerate}\item  The non trivial element of the Weyl group of $SL(2,F)$ acts on  $E$ by 
$$w=e^Xe^Ye^X,$$
\item If, as before,  $M$ is an irreducible submodule of $E$ of dimension  $m+1$, then \\
$w:M_{m-2p}\to M_{2p-m}$ is an isomorphism. More precisely, on the base $(e_p)$ defined above, one has 
$$w.e_p=(-1)^{m-p}e_{m-p}=\frac{1}{(m-p)!} (Y)^{m-p}.e_0.$$
\item For  $Z\in E_p$, one has  $w^2 Z=(-1)^p Z$.
\item $X^j: E_p\to E_{p+2j}$ is $\left\{\begin{array}{cc} \textrm{ injective for } & j\leq -p\\
\textrm{ bijective for } & j=-p\\
\textrm{ surjective for  } & j\geq -p\end{array}\right.$

$Y^j: E_p\to E_{p-2j}$ is $\left\{\begin{array}{cc} \textrm{ injective for } & j\leq p\\
\textrm{ bijective for } & j=p\\
\textrm{ surjective for  } & j\geq p\end{array}\right.$
\end{enumerate}
\subsection{First reduction}\hfill
 \vskip 10pt

Remind the definition of  $V_1^+$ (cf. Corollaire  \ref{cor-decomp-j}):
$$V_1^+= \widetilde{\go g}_1\cap V^+=\{ X\in V^+| [H_{\lambda_0}, X]=0\}.$$
The decomposition of  $V^+$ into eigenspaces of  $H_{\lambda_0}$ is then given by
$$V^+=V_1^+\oplus  \widetilde{\go g}^{\lambda_0}\oplus W_1^+, \textrm{ where } W_1^+=\{ X\in V^+| [H_{\lambda_0}, X]=X\}=\{ X\in V^+| [H_{\lambda_1}+\ldots +H_{\lambda_k},X]=X\}.$$

\begin{prop}\label{prop.prem-reduc} \hfill

 Let  $X\in V^+$.\\
\begin{enumerate} \item If  $\Delta_0(X)=0$ then  $X$ is conjugated under   ${\rm Aut}_e(\go g)\subset G$ to an element of  $V_1^+$.
\item if $\Delta_1(X)\neq 0$ then $X$ is conjugated under  ${\rm Aut}_e(\go g)\subset G$  to an element of  $  \widetilde{\go g}^{\lambda_0}\oplus V_1^+$.
\end{enumerate}
\end{prop}

\begin{proof} (we give the proof for the convenience of the reader although it is the same as in the real case, see Prop 2.1 page 38 of [BR]),  

 (1) Let  $X$ be a non generic element of $V^+$. Then $X$ belongs to the semi-simple part of  $ \widetilde{\go g}$ and satisfies $(\ad X)^3=0$. By the Jacobson-Morozov Theorem  (\cite{Bou2} chap. VIII, \textsection 11 Corollaire of  Lemme 6), there exists an   $\go {sl}_2$-triple $(Y,H,X)$. Decomposing these elements  according to the decomposition  $ \widetilde{\go g}=V^-\oplus \go g\oplus V^+$, one sees easily that one can suppose that  $Y\in V^-$ and  $H\in\go g$. \medskip

The eigenvalue of  $\ad H$ are the weights of the representation of this  ${\go sl}_2$-triple in  $ \widetilde{\go g}$, they are therefore integers. Hence $H$ belongs to a maximal abelian split subalgebra of   $ \go g $. By (\cite{Schoeneberg} Theorem 3.1.16 or \cite{Seligman}, I.3), there exists $g\in {\rm Aut}_e(\go g)\subset G$ such that  $g.H\in\go a$. One can choose  $w\in N_{{\rm Aut}_e(\go g)}(\go a)$ such that  $wg.H$ belongs to the Weyl  chamber $\overline{\go a}^+:=\{ Z\in \go a; \lambda(Z)\geq 0 \,{\rm for  } \,\lambda\in \Sigma^+\}$.  Let 
$(Y',H',X')=wg(Y,H,X)$. We show first that  
$$\lambda_0(H')=0.$$
Any element in  $ \widetilde{\go g}^{\lambda_0}\backslash \{0\}$ is primitive with weight  $\lambda_0(H')$ under the action of this  ${\go sl}_2$-triple, hence $\lambda_0(H')\geq 0$.\medskip

Suppose that  $\lambda_0(H')>0$. Let $\lambda\in \widetilde{\Sigma}^+\backslash \Sigma$. By Theorem  \ref{thbasepi} we obtain $\lambda=\lambda_0+\sum_i m_i\lambda_i$ with  $m_i\in\N$ and  $\lambda_i\in \Sigma^+$ and hence $\lambda(H')>0$. Therefore, for all $Z\in \widetilde{\go g}^\lambda\subset V^+$, we get 
$$Z=-\frac{1}{\lambda(H')} \ad X' (\ad Y'( Z)).$$
It follows that  $\ad X': \go g\to V^+$ is surjective, and this is not possible because $X'$ is not generic. Therefore $\lambda_0(H')=0$.\medskip

Let us show that  $X'\in V_1^+$. Let  $Y_0\in \widetilde{\go g}^{-\lambda_0}\backslash \{0\}$. The  ${\go sl}_2$-module generated by $Y_0$ under the action of $(Y',H',X')$ has lowest weight $0$ ($\ad Y'. Y_0=0$ and $\ad H'.Y_0=0$)  and hence it is the trivial module. It follows that  $\ad X'.Y_0=0$, and then $X'$ commutes with  $ \widetilde{\go g}^{\lambda_0}$ and with  $ \widetilde{\go g}^{-\lambda_0}$. Hence $X'$ commutes with  $H_{\lambda_0}$ and therefore $X'\in V_1^+$.    Since $wg\in{\rm Aut}_e(\go g)$, statement (1) is proved.\medskip

(2) Let  $X\in V^+$ such that  $\Delta_1(X)\neq 0$. This element decomposes as follows
$$X=X_0+X_1+X_2,\quad {\rm with }\, X_0\in  \widetilde{\go g}^{\lambda_0}, X_1\in W_1^+, X_2\in V_1^+.$$
From the definition of  $\Delta_1$ (cf. Definition \ref{defdeltaj}), one has 
$$\Delta_1(X)=\Delta_1(X_2).$$
Therefore $X_2$ is generic in  $V_1^+$. By  Proposition \ref{prop-generiques-cas-regulier} applied to  $ \widetilde{\go g}_1$, there exists  $Y_2\in V_1^-$ such that  $(Y_2, H_{\lambda_1}+\ldots +H_{\lambda_k}, X_2)$ is a  ${\go sl}_2$-triple. The weights of this triple  on $V^+$ are $2$ on $V_1^+$, $1$ on $W_1^+$ and  $0$ on  $ \widetilde{\go g}^{\lambda_0}$. Let us note:
$$\go g_{-1}:=\{ X\in\go g| [H_{\lambda_1}+\ldots +H_{\lambda_k},X]=-X\}.$$
The map  $\ad X_2$ is a bijection from  $\go g_{-1}$ onto  $W_1^+$, hence there exists $Z\in \go g_{-1}$ such that  $[X_2,Z]=X_1$. If we write the decomposition of $e^{\ad Z}X$ according to the weight spaces of $V^+$, we obtain 
$$e^{\ad Z}X=X_0+[Z,X_1]+\frac{1}{2} [Z,[Z,X_2]]+X_1+[Z,X_2]+X_2$$
$$=X_0+[Z,X_1]+\frac{1}{2} (\ad Z)^2. X_2+X_2.$$
Hence $e^{\ad Z}X$ belongs to  $X_2\oplus  \widetilde{\go g}^{\lambda_0}\subset V_1^+\oplus   \widetilde{\go g}^{\lambda_0}$ and this gives (2).
\end{proof}
\begin{theorem}\label{th-V+caplambda} Any element of  $V^+$ is  conjugated under ${\rm Aut}_e(\go g)$ (and hence under $G$) to an element of  $  \widetilde{\go g}^{\lambda_0}+\ldots\oplus   \widetilde{\go g}^{\lambda_k}$.
\end{theorem}

\begin{proof} Let us show that any element of  $V^+$ is conjugated under  ${\rm Aut}_e(\go g)$ to an element of  $V_1^+\oplus   \widetilde{\go g}^{\lambda_0}$.\\
Let $X\in V^+$. If  $X$ is not generic, then  $X$ is ${\rm Aut}_e(\go g)$-conjugated  to an element of $V_1^+$ (Proposition \ref{prop.prem-reduc}, (1)). 

Suppose now that  $X$ is generic.  As  the Lie algebra of ${\rm Aut}_e(\go g)$ is equal to $[\go g,\go g]$,  the Lie algebra of $F^*\times {\rm Aut}_e(\go g)$ is $F\times [\go g,\go g]$,  which is the Lie algebra of $\go g$,  (see Remark \ref{rem-simple}, remember also that, as $\go g\oplus V^+$ is a maximal parabolic subalgebra, the center of $\go g $ has dimension one). Therefore the orbit of $X$ under the group $F^*.{\rm Aut}_e(\go g)$ is open. Suppose that ${\rm Aut}_e(\go g).X\cap \{Y\in V^+,\Delta_{1}(Y)\neq 0\}=\emptyset$. Then, as $\{Y\in V^+,\Delta_{1}(Y)=0\}$ is a cone, we would have $F^*.{\rm Aut}_e(\go g).X\subset \{Y\in V^+,\Delta_{1}(Y)=0\}$. This is impossible, as a Zariski open set is never a subset of a closed one. Hence ${\rm Aut}_e(\go g).X\cap \{Y\in V^+,\Delta_{1}(Y)\neq 0\}\neq\emptyset$. From Proposition \ref{prop.prem-reduc} (2), we obtain that $X$ is conjugated under  ${\rm Aut}_e(\go g)$ to an element of  $  \widetilde{\go g}^{\lambda_0}\oplus V_1^+$.

The same argument applied to  $ \widetilde{\go g}_j$ ($j=1,\ldots, k$)  shows that if $X\in V_j^+$ then the   ${\rm Aut}_e(\go g_j)$-orbit of  $X_j$ meets  $   \widetilde{\go g}^{\lambda_j} \oplus V_{j+1}^+$.  Any element of  ${\rm Aut}_e(\go g_j)$ stabilizes  $  \widetilde{\go g}^{\lambda_0}+\ldots\oplus   \widetilde{\go g}^{\lambda_{j-1}}$ and, by Corollary \ref{cor-decomp-j}, one has  $V_k^+= \widetilde{\go g}^{\lambda_k}$. The result is then obtained by induction. 

\end{proof}
\subsection{An involution which permutes the roots in  $E_{i,j}(\pm 1, \pm 1)$}\hfill
 \vskip 10pt

Recall that  $ \widetilde{G}={\rm Aut}_0( \widetilde{\go g}).$ 

For  $i=0,\ldots, k$, we fix  $X_i\in \widetilde{\go g}^{\lambda_i}$. There exist then  $Y_i\in  \widetilde{\go g}^{-\lambda_i}$ such that $(Y_i,H_{\lambda_i}, X_i)$ is an $\go sl_2$-triple. The action of the non trivial  Weyl group element of this triple is given by 
$$w_i:=e^{\ad X_i}e^{\ad Y_i}e^{\ad X_i}=e^{\ad Y_i}e^{\ad X_i}e^{\ad Y_i}\in N_{ \widetilde{G}}(\go a).$$
\begin{lemme}\label{actionwi} Let  $j\neq i$ and $p=\pm1$. One has
\begin{enumerate}\item If  $H\in \go a$ then  $w_i.H=H-\lambda_i(H)H_{\lambda_i};$
\item If  $X\in E_{i,j}(1,p)$ then  $w_i.X= \ad Y_i.X\in  E_{i,j}(-1,p)$;
\item  If  $X\in E_{i,j}(-1,p)$ then  $w_i.X= \ad X_i.X\in  E_{i,j}(1,p)$;
\item If  $X\in E_{i,j}(\pm1,p)$ then  $w_i^2.X=-X$.
\end{enumerate}
\end{lemme}
\begin{proof} The first statement is obvious as     $w_i$ acts as the reflection on  $\go a$ associated to $\lambda_i$.\\
As each  $E_{i,j}(\pm1,p)$ is included in a weight space for the action of the  $\go sl_2$-triple $(Y_i,H_{\lambda_i}, X_i)$, the other statements are immediate consequences of the properties of  $\go sl_2$-modules given in  section \ref{sl2module}.
\end{proof}
For  $i\neq j$, we set 
$$w_{i,j}:=w_iw_j=w_jw_i.$$\index{wij@$w_{i,j}$}
The preceding Lemma implies that  $w_{i,j}$  satisfies the following properties.
\begin{cor}\label{actionwij} For $i\neq j$ and $p,q\in\{\pm1\}$, one has
\begin{enumerate}\item $w_{i,j}$ is an isomorphism from  $E_{i,j}(p,q)$ onto $E_{i,j}(-p,-q)$
\item The restriction of  $w_{i,j}^2$ to  $E_{i,j}(p,q)$ is the identity.
\end{enumerate}
\end{cor}
\begin{rem}\label{rem-actionwij}
The involution  $w_{i,j}$ permutes the roots $\lambda$ such that $ \widetilde{\go g}^{\lambda}\subset E_{i,j}(\pm 1,\pm 1)$. More precisely one has 
$$ \widetilde{\go g}^{\lambda}\subset E_{i,j}(  1, -1)\Longrightarrow w_{i,j}(\lambda)=\lambda-\lambda_i+\lambda_j,$$
$$ \widetilde{\go g}^{\lambda}\subset E_{i,j}(  1,  1)\Longrightarrow w_{i,j}(\lambda)=\lambda-\lambda_i-\lambda_j,$$
$$ \widetilde{\go g}^{\lambda}\subset E_{i,j}(  -1, -1)\Longrightarrow w_{i,j}(\lambda)=\lambda+\lambda_i+\lambda_j,$$
\end{rem}


\subsection{Construction of automorphisms interchanging  $\lambda_i$ and $\lambda_j$}\hfill
 \vskip 10pt

Let  $i$ and  $j$ be two distinct elements of  $\{0,\ldots, k\}$. By Proposition \ref{propWconjugues},  the roots  $\lambda_i$ and $\lambda_j$ are conjugated by the  Weyl  group $W$. The aim of this section is to construct explicitly an element of $G$ which exchanges  $\lambda_i$ and  $\lambda_j$.

\begin{lemme}\label{lem-mu+wijmu} Let   $\lambda$ be a root such that  $ \widetilde{\go g}^\lambda\subset E_{i,j}(1,-1)$.  Then  $\lambda+w_{i,j}(\lambda)$ is not a root.
\end{lemme}

\begin{proof} If $\lambda=\dfrac{\lambda_i-\lambda_j}{2}$ then  $w_{i,j}(\lambda)=-\lambda$,   this implies the statement. \\
Let  now   $\lambda\neq\dfrac{\lambda_i-\lambda_j}{2}$. Suppose that  $\mu=\lambda+w_{i,j}(\lambda)$  is a root.

 As $ \widetilde{\go g}^\lambda\subset E_{i,j}(1,-1)$, one has  
$w_{i,j}(\lambda)=\lambda-\lambda_i+\lambda_j$  (Remark \ref{rem-actionwij}) and    $\mu=2\lambda-\lambda_i+\lambda_j$. It follows that for  $s\in\{0,\ldots, k\},$ the root   $\mu$ is orthogonal to   $\lambda_s$, and hence strongly orthogonal to $ \lambda_s$ (Corollaire \ref{cor-orth=fortementorth}). Let us write
$$\lambda=\frac{\mu}{2}+\frac{\lambda_i-\lambda_j}{2},  \quad\textrm{and }\quad w_{i,j}(\lambda)=\lambda-\lambda_i+\lambda_j=\frac{\mu}{2}-\frac{\lambda_i-\lambda_j}{2}.$$\\
If  $\alpha$ and  $\beta$ are two roots, we set, as usually
$n(\alpha,\beta)=\alpha(H_\beta)=2\frac{\langle \alpha,\beta\rangle}{\langle \beta,\beta\rangle}.$
As  $\lambda- w_{i,j}(\lambda)=\lambda_i-\lambda_j$ is not a root, and as  $\lambda \neq w_{i,j}(\lambda)$,   we obtain that $n(\lambda,w_{i,j}(\lambda))\leq 0$. Remind that we have supposed that  $\mu=\lambda+w_{i,j}(\lambda)$  is a root. As $w_{i,j}(\lambda)-\lambda$ is not a root,  the  $\lambda$-chain through $w_{i,j}(\lambda)$ cannot be symmetric with respect to  $w_{i,j}(\lambda)$. This implies that $n(\lambda,w_{i,j}(\lambda))<0$. As $\lambda$ and  $w_{i,j}(\lambda)$ have the same length (and hence $n(\lambda,w_{i,j}(\lambda))=n(w_{i,j}(\lambda),\lambda)=-1)$, we get
$$ -1=n(w_{i,j}(\lambda),\lambda)=2-n(\lambda_i,\lambda)+n(\lambda_j,\lambda).$$
Consider the root  $\delta=\lambda+\lambda_{j}=w_j(\lambda)=\dfrac{\mu}{2}+\dfrac{\lambda_i+\lambda_j}{2}$. Then   $n(\delta,\lambda_i)=n(\delta,\lambda_j)=1$ and the preceding relation gives 
$$3=n(\lambda_i,\delta)+n(\lambda_j,\delta).$$
It follows that either  $n(\lambda_i,\delta)$ or  $n(\lambda_j,\delta)$ is  $\geq 2$. Suppose for example that  $n(\lambda_i,\delta)\geq 2$. Consider the  $\delta$-chain through $-\lambda_i$. As $(-\lambda_i-\delta)(H_0)=-4$,     $-\lambda_i-\delta$ is not a root.  It follows that  $-\lambda_i+2\delta=\mu+\lambda_j$ is a root. This is  impossible as  $\mu$ is strongly orthogonal to  $\lambda_j$. \medskip

Hence  $\mu=\lambda+w_{i,j}(\lambda)$ is not a root. 

\end{proof}

\begin{lemme}\label{lem-sl2ij} There exist  $X\in E_{i,j}(1,-1)$ and  $Y\in E_{i,j}(-1,1)$ such that $(Y,H_{\lambda_i}-H_{\lambda_j}, X)$ is an  $\go sl_2$-triple.
\end{lemme}

\begin{proof}  By definition, for a root $\lambda$, the element $H_{\lambda}$ is the unique element such that $\beta(H_{\lambda})=n(\beta,\lambda)$, for all $\beta\in \widetilde \Sigma$. If  $\lambda=(\lambda_i-\lambda_j)/2$ is a root it is then easy to see that $H_\lambda=H_{\lambda_i}-H_{\lambda_j}$.  Any  non zero element    $X\in   \widetilde{\go g}^\lambda\subset E_{i,j}(1,-1)$ can be completed in an  ${\go sl}_2$-triple $(Y,H_{\lambda_i}-H_{\lambda_j}, X)$ where  $Y\in  \widetilde{\go g}^{-\lambda}\subset E_{i,j}( -1,1)$. \\
 If $(\lambda_i-\lambda_j)/2$ is not a root, let us fix a root  $\lambda$   such that $ \widetilde{\go g}^\lambda\subset E_{i,j}(1,-1)$ and set  $\lambda'=-w_{i,j}(\lambda)$.    One has  $ \widetilde{\go g}^{\lambda'}\subset E_{i,j}(1,-1)$.   Also  $\lambda' \neq \lambda$  ($\lambda'=-w_{i,j}(\lambda)=-\lambda+\lambda_{i}-\lambda_{j}=\lambda$ would imply that $\lambda=\frac{\lambda_{i}-\lambda_{j}}{2}$) . Let  $X_0\in  \widetilde{\go g}^\lambda\setminus \{0\}$. Choose $Y_0\in  \widetilde{\go g}^{-\lambda}\subset E_{i,j}(-1,1)$ such that  $(Y_0, H_\lambda, X_0)$ is an  ${\go sl}_2$-triple. Then  $(w_{i,j}(X_0), H_{\lambda'},w_{i,j}(Y_0))$ is an  ${\go sl}_2$-triple as  $w_{i,j}(H_\lambda)=-H_{\lambda'}$ and as 
$w_{i,j}(X_0)\in \widetilde{\go g}^{-\lambda'}$ and  $w_{i,j}(Y_0)\in \widetilde{\go g}^{\lambda'}$ . Define  
$$X=X_0+w_{i,j}(Y_0)\quad\textrm{and}\quad  Y=Y_0+w_{i,j}(X_0)=w_{i,j}(X).$$ \\
Then one has 
$$[H_{\lambda_i}-H_{\lambda_j}, X]=2X \quad\textrm{and}\quad [H_{\lambda_i}-H_{\lambda_j}, Y]=2Y.$$\\
It remains to prove that $[Y,X]=H_{\lambda_i}-H_{\lambda_j}$. 
Let  $Z=[Y,X]$. By the preceding lemma,   $\lambda-\lambda'$ is not a root, hence $[Y_0, w_{i,j}(Y_0)]=0$ and  $[w_{i,j}(X_0), X_0]=0$ and this implies
$$Z=[Y_0,X_0]-w_{i,j}([Y_0,X_0])=H_\lambda-w_{i,j}(H_{\lambda})\in\go a.$$

This shows that $Z\neq 0$ ($Z=0\Leftrightarrow H_{\lambda}=w_{i,j}(H_\lambda)\Leftrightarrow \lambda=w_{i,j}(\lambda)=\lambda-\lambda_{i}+\lambda_{j}\Leftrightarrow \lambda_{i}=\lambda_{j}$).

 Lemma  \ref{actionwi} implies that  $w_{i,j}(Z)=Z-\lambda_i(Z)H_{\lambda_i}-\lambda_j(Z)H_{\lambda_j}$. As  $w_{i,j}Z=-Z$, we obtain
 $$Z=\frac{\lambda_i(Z)H_{\lambda_i}+\lambda_j(Z)H_{\lambda_j}}{2}.$$
Therefore $Z\in  \go a^0=\oplus_{i=0}^kF\; H_{\lambda_{i}}$. 

Let $H\in \go a^0$. Then $H=\sum _{i=0}^k\frac{\lambda_{i}(H)}{2}H_{\lambda_{i}}$ and an easy calculation shows that $[H,Y]= \frac{\lambda_{j}(H)-\lambda_{i}(H)}{2}Y$. Therefore 
$$ \widetilde{B}(H,Z)=\frac{\lambda_j(H)-\lambda_i(H)}{2}  \widetilde{B}(Y,X).$$
On the other hand, the roots  $\lambda_i$ and $\lambda_j$ are $W$-conjugate (Proposition \ref{propWconjugues}), hence $ \widetilde{B}(H_{\lambda_i},H_{\lambda_i})= \widetilde{B}(H_{\lambda_j},H_{\lambda_j})$ for all $i,j$. Define  $C_1:=  \widetilde{B}(H_{\lambda_i},H_{\lambda_i})\in F^*$. Then
$$ \widetilde{B}(H,H_{\lambda_i}-H_{\lambda_j})=C_1\frac{\lambda_i(H)-\lambda_j(H)}{2},\quad {\rm for}\, H\in\go a^0.$$
As $ \widetilde{B}$ is nondegenerate on  $\go a^0$, if we set  $C_2:=- \widetilde{B}(X,Y)\in F^*$, we obtain
$$Z=\frac{C_2}{C_1}(H_{\lambda_i}-H_{\lambda_j}).$$
If we replace  $Y$ by $\dfrac{C_1}{C_2}Y$,  then $(Y,H_{\lambda_i}-H_{\lambda_j}, X)$ is an $\go sl_2$-triple.

\end{proof}

\vskip 5pt

Let  $(Y,H_{\lambda_i}-H_{\lambda_j}, X)$ be the $\go sl_2$-triple obtained in the preceding Lemma. The action of the non trivial element of the Weyl group  of this $\go sl_2$-triple  is given by 
$$\gamma_{i,j}\index{gammaij@$\gamma_{i,j}$}=e^{\ad X}e^{\ad Y}e^{\ad X}=e^{\ad Y}e^{\ad X}e^{\ad Y} \in {\rm Aut}_e(\go g).$$

\begin{prop}\label{prop-gammaij} For  $i\neq j\in\{0,\ldots ,k\}$, the elements $\gamma_{i,j}$ belong to $ N_{{\rm Aut}_e(\go g)}(\go a^0)$ and 
\begin{enumerate}\item $$\gamma_{i,j}(H_{\lambda_s})=\left\{\begin{array}{ll} H_{\lambda_i} & {\rm for }\, s=j\\H_{\lambda_j} & {\rm for }\, s=i\\H_{\lambda_s} & {\rm for }\, s\notin\{i,j\}\end{array}\right.$$
\item The  action of  $\gamma_{i,j}$ is trivial on each root space $ \widetilde{\go g}^{\lambda_s}$ for  $s\notin\{i,j\}$ and it is a bijective involution from   $ \widetilde{\go g}^{\lambda_i}$ onto   $ \widetilde{\go g}^{\lambda_j}$.
\end{enumerate}
\end{prop}

\begin{proof}  
As  $X\in E_{i,j}(1,-1)$, for $H\in\go a^0$ one has 
$$\gamma_{i,j}(H)=H-\frac{\lambda_i(H)-\lambda_j(H)}{2} (H_{\lambda_i}-H_{\lambda_j}),$$
and this gives the relations $(1)$.

For  $\{s,l\}\cap\{i,j\}=\emptyset$ and $p,q\in\{\pm1\}$, one has $[E_{s,l}(p,q),E_{i,j}(p,q)]=\{0\}$, therefore  the action  $\gamma_{i,j}$ is trivial on the spaces $E_{s,l}(p,q)$, and in particuliar of  $ \widetilde{\go g}^{\lambda_s}$ for  $s\notin\{i,j\}$. The relations   (1) imply  that $\gamma_{i,j}$ is an isomorphism from $ \widetilde{\go g}^{\lambda_i}$ onto $ \widetilde{\go g}^{\lambda_j}$. It is an involution because the action of  $\gamma_{i,j}^2$ on the even weight spaces of  $H_{\lambda_i}-H_{\lambda_j}$ is trivial (section \ref{sl2module}).\\
\end{proof}
It is worth noting that the action of $\gamma_{i,j}^2$ on  $ \widetilde{\go g}$ is not trivial . Indeed, if  $X$ is in an odd weight space for  $H_{\lambda_i}-H_{\lambda_j}$ then  $\gamma_{i,j}^2(X)=-X$. Therefore, on order to obtain an involution, we will modify $\gamma_{i,j}$. This is the purpose of the next proposition.

\begin{prop}\label{prop-gamijtilde} For  $i\neq j\in\{0,\ldots, k\}$, the element   $\widetilde{\gamma_{i,j}}=\gamma_{i,j}\circ w_i^2$ \index{gammaijtilde@$\widetilde{\gamma_{i,j}}$}  belongs to $N_{ \widetilde{G}}(\go a^0)$ and it verifies the following relations: 
$$\widetilde{\gamma_{i,j}}(H_{\lambda_s})=\left\{\begin{array}{ll} H_{\lambda_i} & {\rm for }\, s=j\\H_{\lambda_j} & {\rm for }\, s=i\\H_{\lambda_s} & {\rm for }\, s\notin\{i,j\}\end{array}\right.$$
and  $$\widetilde{\gamma_{i,j}}^2={\rm Id}_{ \widetilde{\go g}}.$$\end{prop}

\begin{proof}   By section \ref{sl2module}, the action of  $w_i^2$ on the spaces $E_{i,s}(\pm1,\pm1)$ (for $s\neq i$) is the scalar multiplication by  $-1$, and is trivial on  the sum of the other spaces. Moreover, the action of $\gamma_{i,j}$ on the spaces  $E_{s,l}(\pm1,\pm1)$ with  $\{s,l\}\cap\{i,j\}=\emptyset$ is trivial, and is  an isomorphism from $E_{i,s}(\pm1,\pm1)$ onto  $E_{j,s}(\pm1,\pm1)$. Therefore one obtains 
$$\widetilde{\gamma_{i,j}}^2(X)=\left\{\begin{array}{ll} -{\gamma_{i,j}}^2(X) & {\rm for }\, X\in\oplus_{s\notin\{i,j\}}E_{i,s}(\pm1,\pm1)\oplus E_{j,s}(\pm1,\pm1)\\
{\gamma_{i,j}}^2(X)&  {\rm for  }\, X\in\oplus_{\{s,l\}\cap\{i,j\}=\emptyset}E_{s,l}(\pm1,\pm1).\end{array}\right.$$
As the subspace $\oplus_{s\notin\{i,j\}}E_{i,s}(\pm1,\pm1)\oplus E_{j,s}(\pm1,\pm1)$ of the odd weightspaces for  $H_{\lambda_i}-H_{\lambda_j}$, the element  $\gamma_{i,j}^2$ acts by  $-1$ on it, and this   ends the proof.
\end{proof}

\subsection{Quadratic forms}\label{sec:section-qXiXj}\hfill
 \vskip 10pt
 Remind that the quadratic form  $b$ is a normalization of the Killing form (Definition \ref{defb(X,Y)}).
 For $X\in V^+$, let  $Q_X$  \index{QX@$Q_X$}be the quadratic form on $V^-$ defined by 
 $$Q_X(Y)=b(e^{\ad X}Y,Y),\quad Y\in V^-.$$

 If  $g\in G$, the quadratic forms $Q_X$ and $Q_{g.X}$ are equivalent.\\
 Therefore we will study the quadratic form $Q_X$ for  $X\in \oplus_{j=0}^k \widetilde{\go g}^{\lambda_j}.$ 
 
  The grading of  $ \widetilde{\go g}$ is orthogonal for $ \widetilde{B}$ and hence also for $b$. One obtains
 $$Q_X(Y)=\frac{1}{2} b((\ad X)^2Y,Y)=-\frac{1}{2} b([X,Y],[X,Y])$$

Let 
$$X=\sum_{j=0}^kX_j,\quad X_j\in \widetilde{\go g}^{\lambda_j}.$$

Let  $E_{s,l}(-1,-1)\subset V^-$. The action of  $\ad X_i \ad X_j$ on  $E_{s,l}(-1,-1)$ is non zero if and only if $(s,l)=(i,j)$ where  $i\neq j$ or  $s=l=i=j$. The quadratic forms  $q_{X_i,X_j}$ \index{qxixj@$q_{X_i,X_j}$}on  $E_{i,j}(-1,-1)$ (\resp   $ \widetilde{\go g}^{\lambda_j}$) for $i\neq j$ (\resp $i=j$) are defined by
$$q_{X_i,X_j}(Y)=-\frac{1}{2} b([X_i,Y],[X_j,Y]),\quad {\rm for  }\, Y\in E_{i,j}(-1,-1) ({\rm \resp } \widetilde{\go g}^{\lambda_j}).$$

The decomposition of  $V^-$ implies  that the quadratic form $Q_X$ is  equal to 
$$(\oplus_{j=0}^k q_{X_j,X_j})\oplus (\oplus_{i<j}^k 2\, q_{X_i,X_j}).$$

 \begin{theorem}\label{th-qnondeg} Let
$X=\sum_{j=0}^kX_j$
 where $ X_j\in \widetilde{\go g}^{\lambda_j}.$ Let $i,j\in\{0,\ldots, k\}$. 
 \begin{enumerate} \item  If $X_{i}\neq0$  and  $X_{j}\neq0$, then $q_{X_i,X_j}$ is non degenerate.
 \item Let $m$ be the number of indices $i$ such that  $X_i\neq 0$. Then
 $${\rm rank }\, Q_X =m\ell+\frac{m(m-1)}{2}d,$$
 where  $\ell={\rm dim}\,  \widetilde{\go g}^{\lambda_i}$ and  $d={\rm dim} \,E_{i,j}(-1,1)$ for $i\neq j$.
 \end{enumerate}
 \end{theorem}
 \begin{proof}  The bilinear form associated to  $q_{X_i,X_j}$ is given by 
 $$L_{X_i,X_j}(u,v)=-\frac{1}{2} b([X_i,u], [X_j,v]),\quad u,v\in E_{i,j}(-1,-1) ({\rm \resp } \widetilde{\go g}^{\lambda_j})\, {\rm for  }\, i\neq j  ({\rm \resp }i=j).$$
 
If  $i=j$,   $(\ad X_i)^2$ is an isomorphism from  $ \widetilde{\go g}^{-\lambda_i}$ onto  $ \widetilde{\go g}^{\lambda_i}$. As the form  $b$ (proportional to  $ \widetilde{B}$) is non degenerate on  $ \widetilde{\go g}^{-\lambda_i}\times  \widetilde{\go g}^{\lambda_i}$, the form $L_{X_i,X_i}$ is non degenerate on  $ \widetilde{\go g}^{\lambda_i}$.\medskip

If $i\neq j$, let us consider two roots  $\lambda$ and  $\mu$ in the decomposition of $E_{i,j}(-1,-1)$. As   $\ad X_j( \widetilde{\go g}^{\mu})\subset  \widetilde{\go g}^{\mu+\lambda_j}$ and   $\ad X_i( \widetilde{\go g}^{\lambda})\subset  \widetilde{\go g}^{\lambda+\lambda_i}$, the restriction of $L_{X_i, X_j}$ to  $ \widetilde{\go g}^\lambda\times \widetilde{\go g}^\mu$ is non zero if and only if  $\lambda+\lambda_i=-(\mu+\lambda_j)$, that is if and only if  $\mu=-w_{i,j}(\lambda)$.\\
Let  $u\in  \widetilde{\go g}^\lambda$ such that, for all $v\in \widetilde{\go g}^{-w_{i,j}(\lambda)}$, one has $L_{X_i,X_j}(u,v)=0$. 
By section  \ref{sl2module},  $\ad X_i$ is an isomorphism from  $ \widetilde{\go g}^{\lambda}\subset E_{i,j}(-1, -1)$ onto $  \widetilde{\go g}^{\lambda+\lambda_i}\subset E_{i,j}(1,-1)$ and   $\ad X_j$ is an isomorphism from  $ \widetilde{\go g}^{-w_{i,j}(\lambda)}\subset E_{i,j}(-1, -1)$ onto  $ \widetilde{\go g}^{-(\lambda+\lambda_i)}\subset E_{i,j}(-1,1)$.  As the restriction of  $b$ to $ \widetilde{\go g}^{\lambda+\lambda_i}\times \widetilde{\go g}^{-(\lambda+\lambda_i)}$ is non degenerate, we get $\ad X_i (u)=0$ and hence  $u=0$.  This proves the first statement

The second statement is an immediate consequence of the formula $Q_{X}= (\oplus_{j=0}^k q_{X_j,X_j})\oplus (\oplus_{i<j}^k 2\, q_{X_i,X_j})$ seen before.

  \end{proof}

 \begin{prop}\label{prop-equivalenceqXiXj}  There exist $\go{sl}_2$-triples $(Y_s,H_{\lambda_s}, X_s)$, $s\in\{0,\dots, k\}$,  such that,  for $i\neq j$, the quadratic forms  $q_{X_i,X_j}$ are all $G$-equivalent (this means that there exists $g\in G$ such that $q_{X_{0},X_{1}}=q_{X_{i},X_{j}}\circ g)$.\\
 Moreover, if $\ell =1$, these forms satisfy the following conditions  (remember that  $e={\rm dim}\; \tilde{\go g}^{(\lambda_i+\lambda_j)/2}$ $ \leq$ $ d={\rm dim} \,E_{i,j}(-1,1)$ for $i\neq j$):
  \begin{enumerate}
 \item If   $e\neq 0$, then the restriction of  $q_{X_i,X_j}$ to  $\tilde{\go g}^{- (\lambda_i+\lambda_j)/2}$  is anisotropic of rank $e$ and represents $1$ (i.e. there exists   $u\in \tilde{\go g}^{- (\lambda_i+\lambda_j)/2}$ such that  $q_{X_i,X_j}(u)=1$).
 \item If  $d-e\neq 0$, and if $W_{i,j}(-1,-1)$ denotes the direct sum of the spaces  $\tilde{\go g}^{-\mu}\subset E_{i,j}(-1,-1)$ where  $\mu\in\tilde{\Sigma}$ and $\mu\neq (\lambda_i+\lambda_j)/2$, then the restriction of  $q_{X_i,X_j}$ to $W_{i,j}(-1,-1)$ is hyperbolic of rank  $d-e$ (and therefore $d-e$ is even).\end{enumerate}
 \end{prop}

 \begin{proof} Let $X_0\in \widetilde{\go g}^{\lambda_0}$. We fix a $\go sl_2$-triple $(Y_0,H_{\lambda_0}, X_0)$. Let  $j\in\{1,\ldots, k\}$. We choose the maps  $\gamma_{0,j}$ such that Proposition \ref{prop-gammaij} is satisfied and we set  $Y_j=\gamma_{0,j} Y_0$ and $X_j=\gamma_{0,j}X_0$. Then $(Y_j,H_{\lambda_j}, X_j)$ is an  $\go sl_2$- triple  and for  $i\neq j$, we have $\gamma_{0,i}(X_j)=X_j$. Then, for  $Y\in E_{i,j}(-1,-1)$, we obtain 
 $$q_{X_i,X_j}(Y)=-\frac{1}{2} b([\gamma_{0,i}(X_0), Y], [X_j,Y])$$
 $$=-\frac{1}{2} b([X_0, \gamma_{0,i}^{-1}(Y)], [X_j,\gamma_{0,i}^{-1}(Y)])=q_{X_0,X_j}(\gamma_{0,i}^{-1}(Y)).$$
Therefore  $q_{X_i,X_j}$ is equivalent to $q_{X_0,X_j}$for all $j\neq 0$ and  $j\neq i$. \medskip

Let $j\geq 2$. One has  $X_j=\gamma_{0,1}(X_j)=\gamma_{0,1}\gamma_{0,j}(X_0)$. As the restriction of  $\gamma_{0,1}^2$ to  $ \widetilde{\go g}^{\lambda_0}$ is the identity, one has 
$X_0=\gamma_{0,1}^2 (X_0)=\gamma_{0,1}(X_1)$ and hence $X_j=\gamma_{0,1}\gamma_{0,j}\gamma_{0,1}(X_1)$. As $g=\gamma_{0,1}\gamma_{0,j}\gamma_{0,1}$ fixes $X_0$, one obtains, for $Y\in E_{0,j}(-1,-1)$,
$$q_{X_0, X_j}(Y)=-\frac{1}{2} b([X_0, g^{-1}Y], [X_1,g^{-1}Y])= q_{X_0, X_1}(g^{-1}Y).$$
Therefore $q_{X_0,X_j}$ is equivalent to  $q_{X_0,X_1}$.\medskip

 Suppose now that $\ell=1$. As  the form $q_{X_i,X_j}$ for $i\neq j$ is equivalent to $q_{X_0,X_1}$, it is sufficient to prove the assertions for $i=0$ and $j=1$.

 If $e\neq 0$ then $(\lambda_0+\lambda_1)/2$ is a root,  and its coroot is  $H_{\lambda_0}+H_{\lambda_1}$. Let  $Y$ be   a non zero element in $\tilde{\go g}^{-(\lambda_0+\lambda_1)/2}$. Let  $(Y,H_{\lambda_0}+H_{\lambda_1}, X)$ be an $\go sl_2$-triple with $X\in \tilde{\go g}^{(\lambda_0+\lambda_1)/2}$ and 
denote by   $w=e^{\ad X}e^{\ad Y} e^{\ad X}$ the non trivial Weyl group element associated to this  $\go sl_2$-triple .  As $X_0$ is if weight $2$ for the action of this  $\go sl_2$-triple, one has  $wX_0= \dfrac{({\rm ad}\;Y)^2}{2}X_0$ and  $wX_0$ is a non zero element  of  $\tilde{\go g}^{-\lambda_1}$. As $\ell=1$, there exists $a\in F^*$ such that $wX_0=aY_1$. Therefore, using the normalization of  $b$ (Lemme \ref{lemmeb}), we get
$$q_{X_0,X_1}(Y)=\dfrac{1}{2}b({\rm ad}(Y))^2X_0,X_1)=b(w X_0,X_1)=a\; b(Y_1,X_1)=a\neq 0.$$
Hence the  restriction of  $q_{X_0,X_1}$ to  $\tilde{\go g}^{-(\lambda_0+\lambda_1)/2}$ is anisotropic.

 Let us now  prove that this restriction represents $1$. 
As $\dfrac{\la_0-\la_1}{2}$ is a root and its coroot is $H_{\la_0}-H_{\la_1}$,  we can fix an  $\go sl_2$-triple $(Y', H_{\lambda_0}-H_{\lambda_1},X')$ with $Y'\in \tilde{\go g}^{-(\la_0-\la_1)/2}$, $X'\in \tilde{\go g}^{(\la_0-\la_1)/2}$   such that  
$\gamma_{0,1}=e^{\ad X'}e^{\ad Y'}e^{\ad X'}$. As $X_0$ is of weight  $2$ for the action of this  $\go sl_2$-triple , one has  $\gamma_{0,1}(X_0)= \dfrac{1}{2} (\ad Y')^2(X_0)$.
From the normalization of  $b$ (Lemme \ref{lemmeb}), we get 
  $$1=b(X_1,Y_1)=b(\gamma_{0,1}(X_0), Y_1)=-\frac{1}{2}b([Y',X_0],[Y',Y_1]).$$
Set $Z=[Y_1,Y']\in \tilde{\go g}^{-(\la_0+\la_1)/2}$. 
 Using the Jacobi identity, one has 
$[X_0,Z]=\ad(Y_1)([X_0,Y'])$ and  $[X_1,Z]=-[Y_1,[Y',X_1]]-[Y',[X_1,Y_1]]=-Y'$. Therefore
$$q_{X_0,X_1}(Z)=-\frac{1}{2} b([X_0,Z],[X_1,Z])=-\frac{1}{2}b(\ad(Y_1)([X_0,Y']), -Y')$$
$$=-\frac{1}{2}b([X_0,Y'], [Y_1Y'])=1. $$

We prove now the last assertion. We have $d> e$. For  $s=0$ or $1$, we note $w_s=e^{\ad X_s}e^{\ad Y_s} e^{\ad X_s}$ and  $w_{0,1}=w_0w_1=w_1w_0$. Let  $\mu\in\tilde{\Sigma}$ such that  $\mu\neq(\lambda_0+\lambda_1)/2$ and such that  $\tilde{\go g}^{\mu}\subset E_{0,1}(1,1)$. From Remark  \ref{rem-actionwij}, we have $w_{0,1} \mu=\mu-\lambda_0-\lambda_1$. Hence $\mu'=-w_{0,1} \mu$ is a root, distinct from $\mu$ (because  $\mu'=\mu$ would imply  $\mu=(\lambda_0+\lambda_1)/2$ and this is not the case) such  that ${\go g}^{-\mu'}\subset E_{0,1}(-1,-1)$. Let us fix an $\go sl_2$-triple $(X_{-\mu}, H_\mu, X_\mu)$ where $X_{\pm\mu}\in\tilde{\go g}^{\pm\mu}$. Applying  $w_{0,1}$ we obtain the $\go sl_2$-triple   $(w_{0,1}X_{\mu}, H_{\mu'}, w_{0,1}X_{-\mu})$ where $w_{0,1}X_\mu\in\tilde{\go g}^{-\mu'}$. Using Lemma \ref{actionwi}, we  obtain
$$q_{X_0,X_1}(X_{-\mu})= -\dfrac{1}{2} b([X_0,X_{-\mu}], [X_1,X_{-\mu}]) =-\dfrac{1}{2} b(w_0X_{-\mu}, w_1X_{-\mu})=-\dfrac{1}{2} b(X_{-\mu},w_{0,1}X_{-\mu})=0$$
and 
\begin{align*}q_{X_0,X_1}(w_{0,1}X_{\mu})&=-\dfrac{1}{2} b([X_{0},w_0w_{1}X_{\mu}], [X_{1}, w_{0}w_1X_{\mu}])=-\dfrac{1}{2} b(w_1X_{\mu}, w_0X_{\mu})\\
&=-\dfrac{1}{2} b(X_{\mu},w_{0,1}X_{\mu})=0.\end{align*}
This implies that the restriction of  $q_{X_0,X_1}$ to the vector space generated by $X_{-\mu}$ and $w_{0,1}X_\mu$ is a hyperbolic plane. This ends the proof.

\end{proof}


\begin{rem}\label{rem-lienmuller} Rather than the forms $q_{X_i, X_j}$ which were already used in the real case by N.~Bopp and H. Rubenthaler (\cite{B-R}), I . Muller  (\cite{Mu98},\cite{Mu}) introduced the quadratic form 
$f_{\lambda_i,\lambda_j}$ on $E_{i,j}(-1,1)$ defined as follows: if   $(Y_i,H_{\lambda_i}, X_i)$ is an $\go sl_2$-triple then
$$f_{\lambda_i,\lambda_j}(u)=-\frac{1}{2}b([u, X_i], [u, Y_j]).$$
From the preceding proof we see that
 $$f_{\lambda_i,\lambda_j}(u)=q_{X_i, X_j}([Y_j,u]),\quad u\in E_{i,j}(-1,1).$$
\end{rem}
 \subsection{Reduction to the diagonal. Rank of an element}
 \hfill

\begin{prop}\label{prop G-diag} For any element  $Z$ in  $V^+$, there exists a unique  $m\in\N$, and for  $j=0,\ldots, m-1$, there exist non zero elements   $Z_j\in\tilde{\go g}^{\lambda_j}$   (non unique) such that $Z$ is  $\text{Aut}_e(\go g)$-conjugated  (and hence $G$-conjugated) to  $Z_0+Z_1+\ldots+Z_{m-1}$, or, equivalently, $Z$ is  $\text{Aut}_e(\go g)$-conjugated to a generic element of $V^+_{k-m+1}$. 
\end{prop}

 \begin{proof} By Theorem  \ref{th-V+caplambda}, any non zero element of  $V^+$ is   $\text{Aut}_e(\go g)$-conjugated to an element of the form
$$Y=Y_0+Y_1+\ldots+Y_k,\quad{\rm with }\, Y_j\in \tilde{\go g}^{\lambda_j}.$$
Let  $m$ be the number of indices  $j$ such that $Y_j\neq 0$. For  $i\neq j$, the element   $\gamma_{i,j}\in N_{{\rm Aut}_e(\go g)}(\go a^0)$ obtained in Proposition \ref{prop-gammaij} exchanges  $ \tilde{\go g}^{\lambda_i}$ and  $ \tilde{\go g}^{\lambda_j}$ and fixes   $ \tilde{\go g}^{\lambda_s}$ for  $s\notin\{i,j\}$. Therefore $Y$ is $\text{Aut}_e(\go g)$-conjugated either to an element of the form
$$Z_0+Z_1+\ldots+Z_{m-1},\quad{\rm where }\, Z_j\in \tilde{\go g}^{\lambda_j}\backslash\{0\},$$
or, equivalently, to an element of the form $Y_{k-m+1}+\ldots +Y_k$ where $Y_j\in  \tilde{\go g}^{\lambda_j}\backslash\{0\}$, that is to a generic element of  $V_{k-m+1}^+$.\medskip

Let us now show that for  $m\neq m'$, the elements $Z=Z_0+Z_1+\ldots+Z_m$ and $Z'=Z'_0+Z'_1+\ldots+Z'_{m'}$  where the  $Z_j$'s  and  $Z'_j$'s  are non zero in  $\tilde{\go g}^{\lambda_j}$,  are not  $G$-conjugate. If they were, the quadratic forms $Q_Z$ and  $Q_{Z'}$ would   have the same rank. But, according to Theorem  \ref{th-qnondeg}, one has

$${\rm rank}\, Q_Z-{\rm rank}\, Q_{Z'} =(m-m')\Big(\ell+\frac{m+m'-1}{2}d\Big).$$
Hence ${\rm rank}\, Q_Z\neq {\rm rank}\, Q_{Z'}$ if $m\neq m'$.  
 \end{proof}

\begin{definition}\label{def-rang-element} For  $Z\in V^+$, the rank of $Z$ \index{rank of an element} is  defined to be the integer   $m$ appearing in the preceding Lemma.
\end{definition}

Remember from Notation \ref{notdle} that $\ell$ is the common dimension of the  spaces  $\tilde{\go g}^{\lambda_i}$ and that $e$ is the common dimension of the  spaces  $\tilde{\go g}^{(\lambda_i+\lambda_j)/2}$. 
The  structure of the $G$-orbits in  $V^+$ depends on the integers $(\ell, e)$.    

In the next (sub)sections, we will completely describe the  $G$-orbits  in  $V^+$, not only the open one, see Theorem \ref{thm-delta2}, Theorem \ref{thm-orbites-e04}, Theorem \ref{thm-orbites-e1},  Theorem \ref{thm-orbites-e2}  and Theorem \ref{th-d=3} below.

As the open orbits will be of particular interest for our purpose, let us summarize here our results concerning the number of these orbits (this will be  a consequence of the results obtained in the next three  sections):

\begin{theorem}\label{thm-orbouverte}\- 

\noindent {\rm (1)} If  $\ell=\delta^2,\; \delta\in\N^*$ and  $e=0$ or $4$ (i.e. if $\tilde{\go g}$ is of type I), the group  $G$ has a unique open orbit in  $V^+$,\\
{\rm (2)}  if  $\ell=1$ and   $e\in\{1,2,3\}$ (i.e. if $\tilde{\go g}$ is of type II), the number of open  $G$-orbits in  $V^+$ depends on $e$ and on the parity of  $k$:

{\rm (a)} if  $e=2$ then $G$ has a unique open orbit in $V^+$ if $k$ is even and  $2$ open orbits if $k$ is odd,

{\rm (b)} if  $e=1$ then  $G$ has a unique open orbit in  $V^+$ if $k=0$, it has $4$ open orbits if $k=1$, it has  $2$ open orbits if $k\geq 2$ is even , and  $5$ open orbits if $k\geq 2$ is odd .\\
{\rm (c)}  if  $e=3$, then $G$ has $4$ open orbits.

{\rm (3)} If  $\ell=3$ (i.e. if $\tilde{\go g}$ is of type III, in that case $e=d= 4$), the group $G$ has $3$ open orbits in  $V^+$ if  $k=0$ and  $4$ open orbits if  $k\geq 1$.
 
\vskip20pt
 \end{theorem}

 { \bf We know from Remark \ref{rem-simple} that we can always assume that $\tilde{\go g}$ is simple. This will be the case in the sequel of this section}.

\subsection{ $G$-orbits in the case where $(\ell, d, e)=(\delta^2,  2\delta^2,0)$ (Case ($1$) in Table 1)}\hfill
\vskip 10pt

\begin{theorem}\label{thm-delta2}  Suppose that either  $k=0$ and $\ell=\de^2$ or  $(\ell, d,e)=(\delta^2,  2\delta^2, 0)$, then       $\chi_0(G)=F^*$    and the group
 $G$ has exactly $k+1= \text{ rank}(\tilde{\go{g}})$ non zero orbits in $V^+$.  These orbits are characterized by the rank of their elements, and a set of representatives is given by  the elements $X_0+\ldots +X_j$ ($j=0,\ldots,k$) where the  $X_j$'s are non zero elements of  $\tilde{\go g}^{\lambda_j}$.  
  Any two generic elements of $\oplus_{j=0}^k\tilde{\go g}^{\lambda_j}$ are conjugated by the subgroup  $L=Z_G(\go a^0)$.

\end{theorem}
\begin{proof}    From the classification (cf. Table 1 - (1)), we can suppose that   $\tilde{\go g}= {\go sl}(2(k+1), D)$ 
 where  $D$ is a central division algebra of degree $\delta$ over$F$, graded by the element $H_0= \left(\begin{array}{cc}I_{k+1} & 0\\ 0 & -I_{k+1}\end{array} \right)$.  Then  $V^+$ is isomorphic to the matrix space $M(k+1,D)$ through the map  $$B\mapsto X(B)=\left(\begin{array}{cc}0& B\\ 0 & 0\end{array} \right).$$
 
 The maximal split abelian subalgebra  $\go a$ is the the set $H(\phi_0,\ldots, \phi_{2k+1})=diag(\phi_0,\ldots, \phi_{2k+1})$ where the $\phi_{j}$'s belong to $F$ and the maximal set of strongly orthogonal roots associated to this grading is given by $\lambda_j(H(\phi_0,\ldots, \phi_{2k+2}))=\phi_{k+1-j}-\phi_{k+2+j}$ for  $j\in\{0,\ldots k+1\}$. \medskip


 Let us denote by $\nu$ the reduced norm of the simple central algebra $M(k+1,D)$. Remember that if $E$ is a splitting field for $M(k+1,D)$, then $M(k+1,D)\otimes E\simeq M((k+1)\delta,E)$, and if $\varphi:M(k+1,D)\longrightarrow M((k+1)\delta,E)$ is the canonical embedding, then $\nu(x)=\det(\varphi(x))$, for $x\in M(k+1,D)$ (see for example, Proposition IX.6 p.168 in \cite{Weil}). Also if  $x=(x_{i,j})$ is a triangular matrix in $M(k+1,D)$ and if $\nu_{0}$ denotes the reduced norm of $D$, then $\nu(x)=\prod_{i=1}^{i=k+1}\nu_{0}(x_{i,i})$ (\cite {Weil}, Corollary IX.2 p.169).
\medskip
 
 Let us now describe the group  $G={\mathcal Z}_{{\rm Aut}_0({\go sl}(2(k+1), D))}(H_0)$. 
 
 Consider  first an element $g\in {\rm Aut}_0({\go sl}(2(k+1), D))$ and denote by  $g_E$ the natural extension of $g$ to  ${\go sl}(2(k+1), D)\otimes E={\go sl}(2(k+1)\delta, E)$. We know from  \cite{Bou2} (Chap. VIII, \S 13, $n^\circ 1$, (VII), p.189), that there exists $U\in GL(2(k+1)\delta, E)$ such that  $g_E.x=UxU^{-1}$ for all $x\in {\go sl}(2(k+1)\delta, E)$. Let us write $U$ in the form
 $$U=\left(\begin{array}{cc} u_1 & u_3\\ u_4 & u_2\end{array}\right),\quad u_j\in M((k+1)\delta, E).$$
 Let now $g\in G$. As $g.H_0=H_0$, we have  $g.V^-\subset V^-$,  $g.V^+\subset V^+$and  $g.\go g\subset\go g$. Then  $g_E$ stabilizes $V^+\otimes E$, $V^-\otimes E$ and  $\go g\otimes E$, and a simple computation shows that $u_3=u_4=0$. On the other hand, the algebra   $\go g$ is the set of matrices $\left(\begin{array}{cc} X & 0\\ 0 & Y\end{array}\right)\in{\go sl}(2(k+1), D)$ with $X,Y\in M( k+1, D)$. As  $g_E$ also stabilizes  $\go g$, the maps  $X\mapsto u_1 Xu_1^{-1}$ and $Y\mapsto u_2 Yu_2^{-1}$ are automorphisms of $M( k+1, D)$, for the ordinary associative product.

 By the   Skolem-Noether Theorem, any automorphism of  $M( k+1 , D)$ is inner, hence it exists  $v_1$ and  $v_2$ in  $GL(k+1,D)$ such that  $u_1Xu_1^{-1}=v_1X v_1^{-1}$ and $u_2Xu_2^{-1}=v_2X v_2^{-1}$ for all $X\in M(k+1,D)$. Therefore $v_1^{-1}u_1$ and  $v_2^{-1}u_2$  belong to the center of  $M((k+1)\delta, E)$. It follows that there exist  $\lambda_1$ and $\lambda_2$ in  $E^*$ such that $u_1=\lambda_1 v_1$  and  $u_2=\lambda_2 v_2$. Hence the automorphism  $g$ is given by the conjugation by  $U=\left(\begin{array}{cc} \lambda_1 v_1 & 0\\ 0 & \lambda_2 v_2\end{array}\right)$  and its action on $V^+$ is given by $g.X(Z)=X(\lambda_1\lambda_2^{-1} v_1 Z v_2^{-1})$ for   $Z\in M(k+1,D)$.   As $g$ stabilizes $V^+$  this implies that $\lambda_1\lambda_2^{-1}\in F^*$ and $g$ is given by the conjugation by $diag(g_1,g_2)=\left(\begin{array}{cc} g_1 & 0\\ 0 & g_2\end{array}\right)$ where $g_1=\lambda_1\lambda_2^{-1}v_1\in GL(k+1,D)$ and $g_2=v_2\in GL(k+1,D)$. \\

The polynomial  $\Delta_0$  defined by  $\Delta_0(X(Z))=  \nu(Z)$  is then relatively invariant under $G$  and it  character is $\chi_0(g)=\Delta_0 (g_1g_2^{-1})$ for $g=diag(g_1,g_2)$ . Therefore $\chi_0(G)=F^*$. \medskip

The group $L=\cap_{j=0}^k G_{H_{\lambda_j}}$ corresponds to the action of the elements $ diag(g_1,g_2)$ where $g_{1}$ and $g_{2}$ are diagonal matrices with coefficients in $D^*$.\medskip

From Proposition  \ref{prop G-diag}, any element in  $V^+$ is conjugated to an non zero element $Z\in\oplus_{j=0}^k\tilde{\go g}^{\lambda_j}$. Such an element corresponds to a  matrix of $M(k+1, D)$ whose coefficients are zero except those on the 2nd diagonal. This shows that the group  $L$ acts transitively on $\oplus_{j=0}^k(\tilde{\go g}^{\lambda_j}\setminus \{0\})$.  \end{proof}

 \subsection{$G$-orbits  in the case  $\ell=1$}\hfill
\vskip 10pt
{\bf In this section, we will always   assume that $\ell=1$.}

{\bf As for $k=0$   the  $G$-orbits in  $V^+$ were already described in Theorem \ref{th-k=0}, we   also suppose in this section that $k\geq 1$.} \medskip

If $\ell=1$, the $G$-orbits in  $V^+$ were studied by I. Muller in a more general context (the so-called quasi-commutative prehomogeneous vector spaces), see \cite{Mu98}. Our results,   are more precise than hers in the sense that we obtain  explicit representatives for the orbits, and also we give detailed proofs.

 \begin{definition}\label{def-qe} According to Proposition \ref{prop-equivalenceqXiXj}, we fix $\go{sl}_2$-triples $(Y_s,H_{\lambda_s}, X_s)$, $s\in\{0,\dots, k\}$,  with $Y_s\in\tilde{\go g}^{-\la_s}$ and $X_s\in \tilde{\go g}^{\la_s}$ such that,  for $i\neq j$, the quadratic forms  $q_{X_i,X_j}$ are all $G$-equivalent (this means that there exists $g\in G$ such that $q_{X_{0},X_{1}}=q_{X_{i},X_{j}}\circ g)$.  We set 
$$q_e:=q_{X_0,X_1}=q_{an,e}+q_{hyp,d-e},$$\index{qaae@$q_e$}\index{qane@$q_{an,e}$}\index{qape@$q_{hyp,d-e}$}
where  $q_{hyp,d-e}$ is an hyperbolic quadratic form of rank $d-e$,  $q_{an,0}=0$ and for $e\neq 0$,  $q_{an,e}$ is an anisotropic quadratic form of rank $e$ which represents $1$. 
\end{definition}

We set  $Im(q_e)^*= Im(q_e)\cap F^*$. Let $a\in F^*$. As $q_e$ represents $1$, if  $q_e$ is equivalent to  $a q_e$, (which will be denoted  $aq_e\sim q_e$), then $a\in Im(q_e)^*$.   \medskip

Remember (Corollary \ref{cor-structureG}) that
 
 $$G={\mathcal Z}_{{\rm Aut}_0(\tilde{\go g})}(H_0)={\rm Aut}_e(\go g).L\quad{\rm where }\quad L=Z_G(\go a^0).$$

\begin{lemme}\label{chiG}(\cite{Mu98}, Proposition 3.2 page 175). 
\begin{enumerate}\item $F^{*2}\subset \chi_0(G)=\chi_0(L)$.
\item If $k+1$ is odd then $\chi_0(G)=F^*$. 
\item If $k+1$ is even  then  $\chi_0(G)\subset\{a\in F^*; aq_e\sim q_e\}.$
\end{enumerate}
\end{lemme}
 
\begin{proof}\hfill

\noindent { (1)} We know from Theorem \ref{thproprideltaj} (2), that the character $\chi_0$ is trivial on  ${\rm Aut}_e(\go g)$ and hence  $\chi_0(G)=\chi_0(L)$. \medskip

Let  $g=h_{X_0}(t)$ (Definition \ref{def-theta-h}). From  \cite{Bou2} (Chap VIII \textsection $1$ $n^\circ 1$, Proposition 6 p. 75), we have $g(X_0+\ldots+X_k)=t^2X_0+X_1+\ldots +X_k$. As $\ell =1$, we obtain by Theorem \ref{thpropridelta_0} that   $\Delta_0(t^2X_0+X_1+\ldots +X_k)=t^{2}\Delta_0(X_0+X_1+\ldots +X_k)$. Therefore  $\chi_0(g)=t^{2}\in F^{*2}$, and this proves the first assertion.\medskip

\noindent { (2)} Consider the generic element  $X=X_0+\ldots +X_k\in V^+$ and  $g=h_X(\sqrt{t})$. Theorem  \ref{thpropridelta_0}   implies  $\chi_0(g)=t^{k+1}$ and this proves  the second assertion.\medskip

\noindent { (3)} Let $l\in L$. Then  $l$ stabilizes each of the spaces  $\tilde{\go g}^{\lambda_j}$ and  $E_{i,j}(\pm 1,\pm 1)$. As $\ell=1$, there exist scalars $x_j(l)\in F^*$ such that  $l.X_j=x_j(l) X_j$. Theorem \ref{thpropridelta_0} implies then $\chi_0(l)=\prod_{j=0}^k x_j(l)$.

On the other hand one gets easily that  $q_{l.X_j,l.X_{k-j}}=x_j(l)x_{k-j}(l)q_{X_j,X_{k-j}}$. As  $q_{l.X_j,l.X_{k-j}}$ and  $q_{X_j,X_{k-j}}$ are $G$ - equivalent  to $q_e$, we get 
$x_j(l)x_{k-j}(l) q_e\sim q_e$. As  $k+1$ is even, the scalar  $\chi_0(l)=\prod_{j=0}^{(k-1)/2} x_j(l)x_{k-j}(l)$ is such that  $\chi_0(l) q_e\sim q_e$.  This gives the third  assertion.\end{proof}

\begin{lemme} (compare with \cite{Mu98}, lemme 2.2.2 page 168).
Let  $A\in E_{i,j}(-1,1)$. Then there exists  $B\in E_{i,j}(1,-1)$ such that  $(B,H_{\lambda_j}-H_{\lambda_i}, A)$ is an ${\go sl}_2$-triple   if and only if the map ${\rm ad}(A)^2:\tilde{\go g}^{\lambda_i}\to \tilde{\go g}^{\lambda_j}$ is injective (remember from Lemma \ref {lem-sl2ij} that such an ${\go sl}_2$-triple always exists).

In that case, the element $\theta_A=\theta_{A}(1)= e^{\ad A} e^{\ad B} e^{\ad A}$ of  ${\rm Aut}_e(\go g)\subset G$ satisfies  $\theta_A(X_s)=X_s$ for $s\neq i,j$, $\theta_A(X_i)= a X_j$ and  $\theta_A(X_j)=a^{-1} X_i$ where $a=q_{X_i,X_j}([Y_j,A])$.

\end{lemme}
\begin{proof} It is a well known property of  $\go sl_2$-triples, that if such a triple exists the map  ${\rm ad}(A)^2:\tilde{\go g}^{\lambda_i}\to \tilde{\go g}^{\lambda_j}$ is injective. \medskip

Conversely, suppose that the map  ${\rm ad}(A)^2:\tilde{\go g}^{\lambda_i}\to \tilde{\go g}^{\lambda_j}$ is injective, and hence  bijective.  As  $A$ is nilpotent in $\go{g}$ (more precisely  $(\ad A)^3=0$), the  Jacobson-Morosov Theorem gives the existence of an  $\go sl_2$-triple $(B,u,A)$ in $\go g$. If one decomposes the element $B$ and $u$ according to $\go g={\mathcal Z}_{\go g}(\go a^0) \oplus(\oplus_{r\neq s} E_{r,s}(1,-1))$, we easily see that one can suppose that  $B\in E_{i,j}(1,-1)$ and  $u\in {\mathcal Z}_{\go g}(\go a^0)$. 

We will show that  $u=H_{\lambda_j}-H_{\lambda_i}$. 

As $u$ commutes with the elements  $H_{\lambda_s}$, the endomorphism   $\ad  u$  stabilizes all eigenspaces of these elements. Therefore, as $\ell=1$,  there exists  $\alpha\in F$ such that  $[u,X_i]=\alpha X_i$. \medskip

 As $[B,X_i]=0=(\ad\; A)^3 X_i $ and  as by hypothesis ${\rm ad}(A)^2 X_i\neq 0$, the  $\go sl_2$ -  module generated by  $X_i$  under the action of  $(B,u,A)$  is irreducible of dimension $3$ and a base of this module is  $(X_i, [A,X_i], {\rm ad}(A)^2 X_i)$. This implies that  $\alpha=-2$.

Let   $\theta_A=e^{\ad A} e^{\ad B} e^{\ad A}$ be the non trivial element of the Weyl group of  the  $\go sl_2$-triple $\{B,u,A\}$. Then (cf. \textsection \ref{sl2module}): $$\theta_A(X_i)=\dfrac{1}{2} {\rm ad}(A)^2 X_i\in \tilde{\go g}^{\lambda_j}.$$ 
Hence, there exists $a\in F^*$ such that  $\theta_{A}(X_i)=a X_j$. As $0\neq B(X_i,Y_i)=B(\theta_A(X_i),\theta_A(Y_i))$ , we have  $\theta_A(Y_i)=a^{-1} Y_j$. And as  $(\theta_A(Y_i),\theta_A(H_{\lambda_i}),\theta_A(X_i))$ is again an  $\go sl_2$-triple, we obtain that  $\theta_A(H_{\lambda_i})=H_{\lambda_j}$.\medskip

On the other hand, a simple computation shows that 
$ \theta_A(H_{\lambda_s})=H_{\lambda_s} $ if $s\neq i,j$ and 
$  \theta_A(H_{\lambda_i})=H_{\lambda_i}+u $.
Hence   $u=H_{\lambda_j}-H_{\lambda_i}$, which gives the first assertion of the Lemma.
\medskip

By Remark \ref{rem-lienmuller}  and the normalization of $b$ (Lemma \ref{lemmeb}), we get $$q_{X_i,X_j}([Y_j,A])=-\dfrac{1}{2}b([A,X_i], [A,Y_j])=b(\dfrac{1}{2}(\ad A)^2 X_i,Y_j)=b(\theta_A(X_i), Y_j)=a\ b(X_j,Y_j)=a.$$
This ends the proof.\end{proof}


\begin{cor}\label{gia}(Recall that $\ell=1$. Compare with \cite{Mu98}   Corollaire 4.2.2 and  Remarques 4.1.6)\hfill

  Let $a\in F^*$ such that $aq_e\sim q_e$ (hence $a\in Im(q_e)^*$). Let  $i\neq j\in\{0,\ldots k\}$.
\begin{enumerate} \item There exists $g_{i,j}^a\in L\cap {\rm Aut}_e(\go g)$ such that  $g_{i,j}^a(X_i)=aX_i$, $g_{i,j}^a(X_j)=a^{-1} X_j$ and  $g_{i,j}^a(X_s)=X_s$ for $s\neq i,j$. 

\item  If either
\begin{enumerate}\item $k+1$ $= \text{rank}(\tilde{\go{g}})$ is odd,\\
or 
\item $k+1$ $= \text{rank}(\tilde{\go{g}})$ is  even  and  if there exists  a regular graded Lie algebra   $(\tilde{\go r}, \tilde{H}_0)$ satisfying the hypothesis  $({\mathbf H}_1), ({\mathbf H}_2)$ and  $({\mathbf H}_3)$   such that the algebra  $\tilde{\go r}_1$ obtained at the first step of the descent   (cf. Theorem \ref{th-descente}) is equal to  $\tilde{\go g}$ (in other words, the algebra $\tilde{\go g}$ is the first step in the descent from a bigger graded Lie algebra),
\end{enumerate}
\medskip
then there exists  $g_i^a\in L$ such  that $g_i^a(X_i)=a X_i$ and  $g_i^a(X_s)=X_s$ for  $s\neq i$. In particular we have  $\{a\in F^*; aq_e\sim q_e\}\subset \chi_0(G)$.
\end{enumerate}
\end{cor} 
\begin{rem} From the classification (cf. Table 1) and from Proposition \ref{prop-diagg1}, the condition on the descent in case {\it 2 (b)} occurs if $e=d\in\{1,2\}$ (Table 1,   (2) and  (6))   or $(d,e)=(2,0)$ (Table 1,  (1)) or   $e=0, 4$, and    $k+1\geq 4$ (Table 1,   (11) and  (12)).
\end{rem}
\begin{proof}\hfill

 (1) By hypothesis, the quadratic form $q_{X_i,X_j}$ is equivalent to $q_e$ and $a\in Im(q_e)^*$, thus there exists $Z\in E_{i,j}(-1,-1)$ such that  $q_{X_i, X_j}(Z)=a$. As ${\rm ad}(Y_j): E_{i,j}(-1,1)\to E_{i,j}(-1,-1)$ is an isomorphism, there exists $A\in E_{i,j}(-1,1)$ such that  $[Y_j,A]=Z$. Then, as we have already seen at the end of the preceding proof, one has
$$a=q_{X_i,X_j}([Y_j,A])= \dfrac{1}{2} b({\rm ad}(A)^2 X_i, Y_j).$$
As $a\neq 0$,  the map  ${\rm ad}(A)^2:\tilde{\go g}^{\lambda_i}\to \tilde{\go g}^{\lambda_j}$ is non zero, and hence injective because $\ell=1$ and the preceding Lemma says that there exists an  $\go sl_2$-triple $(B,H_{\lambda_j}-H_{\lambda_i}, A)$ with  $B\in E_{i,j}(1,-1)$. 

Moreover the element  $\theta_A\in {\rm Aut}_e(\go g)\subset G$ fixes each  $X_s$ and each $H_{\lambda_s}$ for  $s\neq i,j$ and  we have $\theta_A(H_{\lambda_i})=H_{\lambda_j}$,  $\theta_A(X_i)=aX_j$ and  $\theta_A(X_j)=a^{-1}X_i$.  From the proof of Proposition \ref{prop-equivalenceqXiXj}, let us set  $f_{i,j}=\gamma_{0,i}\gamma_{0,j}\gamma_{0,i}$. Then the automorphisms $f_{i,j}\in {\rm Aut}_e(\go g)\subset G$ satisfy  the properties of  Proposition \ref{prop-gammaij}    and we have  $f_{i,j}(X_i)=X_j$ and  $f_{i,j}(X_s)=X_s$ for $s\neq i,j$. Then, from the preceding Lemma, the  automorphism $g_{i,j}^a=f_{i,j}\circ \theta_A\in  L\cap {\rm Aut}_e(\go g) $ and  has the other required properties.\medskip

  (2) As the involutions  $f_{0,i}$ exchange  $X_0$ and $X_i$ and fix $X_s$ if  $s\neq i$, we can suppose that  $i=0$.\medskip

Let  $k+1$ be odd. The element $g=  g_{0,1}^a\ldots g_{0,k}^a\in L$ satisfies $g(X_0)=a^k X_0$ and $g(X_s)=a^{-1}X_s$ if  $s> 0$. From Lemma \ref{lem-tId-dansG}, the element  $h_{X_0+\ldots+X_k}(\sqrt{a})\in L$ acts by multiplication by $a$ on  $V^+$ and  $h_{X_0}(a^{-k/2})$ fixes  $X_s$ if  $s\neq 0$, and   also $h_{X_0}(a^{-k/2})(X_0)=a^{-k}X_0$. It follows that the element $g_0^a=h_{X_0+\ldots+X_k}(\sqrt{a})\circ h_{X_0}(a^{-k/2})\circ g$  has the required properties.\medskip

Let now  $k+1$ be even and suppose that there exists a regular graded Lie algebra  $(\tilde{\go r}, \tilde{H}_0)$,  such that the algebra  $\tilde{\go r}_1$ obtained by performing one step in the descent  (cf. Theorem \ref{th-descente}.) is equal to  $\tilde{\go g}$. As we are only concerned by the action of an element of $G$ on $V^+$, we can always suppose that $\tilde{\go r}$ is  simple  (cf. Remark  \ref{rem-simple}).

Let    $(\tilde{\lambda}_0,\ldots \tilde{\lambda}_{k+1})$ the maximal set of stronly orthogonal roots associated to $\tilde{\go r}$.   Then   $(\tilde{\lambda}_1,\ldots \tilde{\lambda}_{k+1})=(\lambda_0,\ldots \lambda_k)$. Let us fix the elements    $\tilde{X}_i\in\tilde{\go r}^{\tilde{\lambda}_i}$ satisfying the conditions of Proposition  \ref{prop-equivalenceqXiXj} in such a way that $(\tilde{X}_1,\ldots \tilde{X}_{k+1})=(X_0,\ldots, X_k)$.\medskip

Let  $R$ be the analogue of the group  $G$ for the algebra $\tilde{\go r}$. Set  $\go r={\mathcal Z}_{\tilde {\go r}}(\tilde H_{0})$   and  $\tilde{\go a}^0=\oplus_{i=0}^{k+1} F H_{\tilde{\lambda}_i}$. Therefore $R={\rm Aut}_e(\go r) {\mathcal Z}_R(\tilde{\go a}^0)$. The first assertion of the Corollary  gives the existence of an element   $r\in {\mathcal Z}_R(\tilde{\go a}^0)\cap{\rm Aut}_e(\go r)$ such that  $r.\tilde{X_0}=a \tilde{X_0}$, $r.\tilde{X_1}=a^{-1} \tilde{X_1}$  and  $r.\tilde{X_s}= \tilde{X_s}$ for  $s\neq 0$. 

As the automorphism $r$ also  fixes   $H_{\lambda_0}$, it stabilizes $\tilde{\go r}_1=\tilde{\go g}$. Let   $r_1\in{\rm Aut}(\tilde{\go g})$ be the  restriction  of  $r$ to  $\tilde{\go g}$. As $r$ centralizes  $\tilde{\go a}^0$, it is clear that   $r_1$ centralizes $\go a^0$. Moreover, one  has   $r_1.X_0=a^{-1}X_0$ and  $r_1.X_s=X_s$ if  $s\neq 0$,  from our choice of the elements  $\tilde{X_j}$.  Then,  if we set   $g=r_1\circ h_{X_0}(a)$, we have    $g.X_0=aX_0, g.X_s=X_s$ for  $s\neq 0$  and also $\chi_0(g)=a$. \\

It remains to prove that  $r_1\in{\rm Aut}_0({\tilde{\go r}_1})$. But as  $r_1\in{\rm Aut}(\tilde{\go r}_1)$, it suffices to prove that  $r_1$ belongs to ${\rm Aut}_e(\tilde{\go r}_1\otimes \bar{F})$.  As we have supposed that  $k+1$ is even, the rank $k+2$ of  $\tilde{\go r}$   is $\geq   3$.  Using the classification  (Proposition \ref{prop-diagg1} and Table 1), one sees easily that the weighted Dynkin diagram    of  $\tilde{\go r}$ is of type   $A_{2n-1}$ (corresponding to the cases $(1)$ (with $\delta=1)$ and $(2)$ in Table 1), $C_n $ (corresponding to the case $(6)$ in Table 1) or $D_{2n,2}$   (corresponding to the cases   $(11)$ and  $(12)$ in Table 1). Here in all cases   $n=k+2$. \\

- If   $\tilde{\go{r}}$ is of type $A_{2n-1}$, then    $\tilde{\go{r}}\simeq \go{sl}(2(k+2),F)$. From the description of the roots $\lambda_{j}$ given in the proof of Theorem  \ref{thm-delta2}  and also from  \cite{Bou2} (Chap. VIII, \S 13, $n^\circ 1$, (VII), p.189), it is easy to see that the group ${\mathcal Z}_{{\rm Aut}_e(\tilde{\go r}\otimes \bar{F})}(\tilde{\go a}^0)$  is the group of conjugations by invertible diagonal matrices Therefore the restriction $r_{1}$ of $r$ to $\tilde{\go{g}}$ belongs effectively to ${\rm Aut}_e(\tilde{\go r}_1\otimes \bar{F})$.\\

- If   $\tilde{\go{r}}$ is of type   $C_n$, then, by   \cite{Bou2}(Chap. VIII \textsection 13    $n^\circ 3$  (VII) page 205.), one has  ${\rm Aut}_e(\tilde{\go g}\otimes \bar{F})={\rm Aut}(\tilde{\go g}\otimes \bar{F})$, 	and this implies the result. \\

- If  $\Psi=D_{2n,2}$ then the algebra $\tilde{\go r}\otimes \bar{F}$ is isomorphic to the orthogonal algebra ${\go o}(q_{(2n,2n)},\bar{F})$. We realize it as in  (\cite{Bou2} Chap VIII \textsection 13 $n^\circ 4$ page 207). One fixes a  Witt bases in which the matrix of  $q_{(2n,2n)}$ is the square matrix  $s_{4n}$  of size  $4n$ whose coefficients  are all zero except those of the second diagonal which are equal to $1$. The algebra  $\tilde{\go r}\otimes \bar{F} $ is then the set of matrices 
$Z=\left(\begin{array}{cc} A&B \\  C&-s_{2n}\;^{t}A s_{2n}\end{array}\right)$ with  $B=-s_{2n}\;^{t}B s_{2n}$ and  $C=-s_{2n}\;^{t}Cs_{2n}$. This algebra is then graded by the element $H_0=\left(\begin{array}{cc}  I_{2n}&0 \\  0&- I_{2n} \end{array}\right) $ where $I_{2n}$ is the identity matrix of size $2n$ and one can choose  $\lambda_0$ in such a way that      $H_{\lambda_0}=\left(\begin{array}{c|c}\begin{array}{cc}  0_{2n-2}& 0 \\0&   I_{2} \end{array}&0 \\\hline 0 &   \begin{array}{cc} -I_2 &0 \\ 0 & 0_{2n-2}  \end{array}\end{array}\right). $

 Recall that a {\it similarity} of a quadratic form $Q$ is a linear isomorphism $g$ of the underlying space $E$ such that $Q(gX)=\lambda(g)Q(X)$, where the scalar $\lambda(g)$ is called the ratio of $g$. If $\dim E= 2l$, then a similarity $g$ is said to be {\it direct} if $\det(g)=\lambda(g)^l$.

 But we know from    \cite{Bou2} (Chap VIII \textsection 13 $n^\circ 4$ page 211), that the group ${\rm Aut}_e(\tilde{\go r}\otimes \bar{F})$ is the group of automorphisms of the form  $\varphi_s: Z\mapsto sZs^{-1}$ where $s$ is a direct similarity of   $q_{(2n,2n)}$. It is easy to see that a similarity $s$ commutes with  $H_0$ if and only if  $s=\left(\begin{array}{cc}  g & 0\\ 0 & \mu s_{2n} \;^{t}g^{-1} s_{2n} \end{array}\right) $ with  $\mu\in \bar{F}^*$ and    $g\in GL(2n,\bar{F})$. Moreover, if  $s$ commutes with    $H_{\lambda_0}$, then    $g$ is of the form  $g=\left(\begin{array}{cc}  g_1 & 0\\ 0 &g_2 \end{array}\right) $ with $g_1\in GL(2n-2,\bar{F})$ and  $g_2\in GL(2,\bar{F})$. It is then clear that  $\varphi_s$ restricts to a direct similarity  of $q_{(2n-2,2n-2)}$ and hence the restriction belongs to  ${\rm Aut}_e(\tilde{\go r}_1\otimes   \bar{F})$.  
 
 \end{proof}

 The following result on anisotropic quadratic forms of rank 2 will be used later.
 
 \begin{lemme}\label{lem-q2} Let   $Q$ be  an anisotropic quadratic form of rank $2$. Let $Im(Q)^*$ be the set of non zero scalars which are represented by  $Q$. Then\\
$(1)$     $Im(Q)^*$ is the union of exactly  $2$ classes  $a$ and  $b$ in $F^*/F^{*2}$.  \\
$(2)$  Let $Q'$ be another anisotropic quadratic form of rank 2. Then $Q'\sim Q$ if and only if  $Im(Q')^*=Im(Q)^*$. In particular, one has   $\mu Q\sim Q$ if and only if  $\mu=1$ or  $\mu=ab$ modulo $F^{*2}$.
\end{lemme}
\begin{proof}  Let  $\pi$ be a uniformizer of  $F$ and let   $\varepsilon$ be a unit of  $F^*$  which is not a square. According to \cite{Lam} (Theorem VI. 2.2, p.152), the set $\{1, \varepsilon,\pi, \varepsilon\pi\}$ is a set of representatives of $F^*/F^{*2}$.   If $-1\notin F^{*2}$ we suppose  moreover that  $\ep=-1$ . \\

$(1)$ As  $Q$ is anisotropic of rank $2$, there exist $v,w\in F^*/F^{*2}$ with $v\neq -1$ such that  $Q\sim w(x^2+vy^2)$. Therefore it is enough to prove the assertion for  $Q_v=x^2+vy^2$.

From  \cite{Lam} (Chapter I, Corollary 3.5 page 11), we know that $\mu\in Im(Q_v)^*$ if and only if the quadratic form $x^2+vy^2-\mu z^2$ is isotropic. On the other hand the quadratic form  $Q_{an}=x^2-\varepsilon y^2-\pi z^2+\varepsilon\pi t^2$ is the unique anisotropic form of rank $4$, up to equivalence (see \cite{Lam}   Chapter VI, Theorem 2.2 (3) page 152).

If  $-1\in F^{*2}$ then  $v\in\{\varepsilon,\pi, \varepsilon\pi\}$. Suppose that for example that $v=\varepsilon$. Then $Q_{v}$ cannot represent $\pi$  or $\varepsilon\pi$, because in that case the forms $x^2+vy^2-\pi z^2$ or $x^2+vy^2-\varepsilon\pi z^2$ would be isotropic,   and then $Q_{an}=x^2-\varepsilon y^2-\pi z^2+\varepsilon\pi t^2$ would be isotropic too.  Finally $Q_v$  represents  exactly the classes of  $1$ and $v$. The same argument works for the other possible values of $v$.

  If $\varepsilon=-1\notin F^{*2}$    then  $-1$ is sum of two squares   (\cite{Lam}  Chapter VI, Corollary 2.6 page 154) and   $v\in\{1,\pi,-\pi\}$ (as the form $Q_{v}$ is anisotropic, $v$ cannot be equal to $\varepsilon=-1$). The forms     $x^2+y^2$ and $ -(x^2+y^2)$ have the same discriminant and represent both the element $-1$. They are therefore equivalent   (\cite{Lam} Chapter I, Proposition 5.1 page 15). Then, as $Q_{an}=x^2+y^2-\pi z^2-\pi t^2\sim -(x^2+y^2)-\pi z^2-\pi t^2\sim x^2+y^2+\pi z^2+\pi t^2$, the same argument as above shows that the form  $Q_1$ represents exactly the classes of $1$ and  $-1$. By the same way, if  $v=\pm \pi$  one shows that  $Q_v$ represents exactly the classes of $1$ and $v$. \\

$(2)$   From above we know that an anisotropic quadratic form   $Q=ax^2+by^2$ with  $a,b\in F^*/F^{*2}$  represents  $a$ and  $b$ if $a\neq b$ and  it represents $\pm a$ if $a=b$ and if $-1\notin F^{*2}$ (because then $-1$  is a sum of two squares, the case where $-1\in F^{*2}$ has not to be considered because the form $Q$ would be isotropic). From  \cite{Lam} (Chapter I, Proposition 5.1 page 15),  we know that $Q'=cx^2+dy^2\sim Q$  if and only if  $ab=cd$ modulo $F^{*2}$ and  $\{a,b\}\cap \{c,d\}\neq \emptyset$.  This implies the second assertion.  

 \end{proof}

 \begin{lemme}\label{lem-ensembleSe} 
Define  $S_e=\{a\in F^*, aq_{e}\sim q_{e}\}$\index{Se@$S_e$}. Then   
$$S_e=\{a\in F^*, aq_{an,e}\sim q_{an,e}\}=\left\{\begin{array}{ccc} F^* & \text{for} & e=0\;\text{or}\; 4\\
F^{*2} & \text{for} & e=1\;\text{or}\; 3\\
N_{E/F}(E^*)=\text{Im}(q_{an,e})^*& \text{for} & e=2,\end{array}\right.$$
where $E$ is a quadratic extension of $F$, uniquely determined  by the class of conjugation of 
$q_{an,2}$ and $N_{E/F}$ is the norm associated to  $E$.\end{lemme}
 \begin{proof} By Definition \ref{def-qe}, we have $q_e=q_{an,e}+q_{hyp, d-e}$ where $q_{an,0}=0$ and for $e\neq 0$, $q_{an,e}$ is an anisotropic quadratic form of rank $e$ which represents $1$ and $q_{hyp, d-e}$ is a hyperbolic quadratic form of rank $d-e$. \\
As  $\mu q_{hyp, d-e}\sim q_{hyp,d-e}$ for all $\mu\in F^*$ (\cite{Lam}  Chapter I, Theorem 3.2 page 9),  Witt's decomposition Theorem (\cite{Lam}  Chapter I, Theorem 4.2 page 12), implies that  $aq_e\sim q_e$ if and only if  $aq_{an,e}\sim q_{an,e}$. Thus $S_e=\{a\in F^*, aq_{an,e}\sim q_{an,e}\}$.
\vskip 5pt
It is clear that if $e=0$    then  $S_e=F^*$, and if $e=1$, then   $S_e=F^{*2}$.\\
By (\cite{Lam} Chapter VI, Theorem 2.2 page 152), all the anisotropic quadratic forms of rank $4$ are equivalent. Thus for $e=4$, we have $S_e=F^*$.\\
Suppose that $e=3$. From (\cite{Lam} Chapter VI, Corollary 2.5 page 153-154), we know that if $a\;q_{an,e}\sim q_{an,e}$ then $-disc(a\;q_{an,e})=-a\;disc(q_{an,e})=-disc(q_{an,e})$  (here $disc$ means the discriminant defined modulo $F^{*2}$). Hence $a\in F^{*2}$. Conversely if $a\in F^{*2}$ we have of course that $a\;q_{an,e}\sim q_{an,e}$. Therefore we have $S_e=F^{*2}$.

Let $e=2$. 
From Lemma \ref{lem-q2}, the form   $q_{an,2}$ represents exactly two classes of squares. As $1$ is represented, there exists another class of squares, say $-\alpha.(F^*)^2$ which is represented (here $-\alpha\notin (F^*)^2$ of course). \\
- Suppose first that $\alpha\notin (F^*)^2$.  Then the quadratic  extension $E=F[\sqrt{\alpha}]$ of $F$   is such that $N_{E/F}(E^*)=\text{Im}(q_{an,2})^*$, where   $N_{E/F}$ is the norm associated to  $E$.  \\
- If $\alpha \in (F^*)^2$. Then we are in the case where $-1 \notin (F^*)^2$ and $q_{an,2}$ represents the two classes $1$ and $-1$. In that case $q_{an,2}\simeq x^2+y^2$  because two anisotropic forms of rank $2$  are equivalent if and only if they  represent the same two classes of squares (Lemma \ref{lem-q2}) and because $-1$ is the sum of two squares (\cite{Lam} Chapter VI, Corollary 2.6. p.154). In that case if we set $E=F[\sqrt{-1}]$, we have $N_{E/F}(E^*)=\text{Im}(q_{an,2})^*$.
Hence we have shown that  there exists a quadratic extension $E=F[\xi]$ of  $F$ (unique up to isomorphism) such that  $N_{E/F}(E^*)=\text{Im}(q_{an,2})^*$. 
\vskip 5pt

Suppose for both cases in the discussion above that  $E' =F[\sqrt{\beta}]$ is another quadratic extension such that $N_{E'/F}(E'^*)=\text{Im}(q_{an,2})^*$  ($\beta=-1$ in the second case). Then $N_{E'/F}(x+\sqrt{\beta}y)=x^2-\beta y^2$ is a quadratic form which represents the two same classes of squares as  $N_{E'/F}(x+\sqrt{\alpha}y)=x^2-\alpha y^2$. Therefore, by Lemma   \ref{lem-q2} these two forms are equivalent and hence have the same discriminant. That is $\alpha\in \beta {F^*}^2$.This implies that $E=F[\sqrt{\alpha}]=F[\sqrt{\beta}]=E'$.

\end{proof}
 
\begin{prop}\label{prop-k=1} Recall that $\ell=1$. Suppose   $k=1$ (i.e. $\text{\rm rank}(\tilde{\go g})=2)$. Let us denote by $[F^*:\chi_0(G)]$ the index of  $\chi_0(G)$ in  $F^*$ (equal to  $1,2$ or $4$ according to Lemma \ref{chiG} (1)).  Then
\begin{enumerate}
\item  The group $G$ has $1+[F^*:\chi_0(G)]$ non zero orbits in $V^+$: the non open orbit of    $X_0$ and the open orbits of  $X_0+vX_1$ where  $v\in F^*/\chi_0(G)$.

\item Two generic elements in  $\  \tilde{\go g}^{\lambda_0}\oplus \tilde{\go g}^{\lambda_1}$ are  $G$-conjugated if and only if they are  $L$-conjugated.
\item  We have $\chi_0(G)=S_e$.
\end{enumerate}
\end{prop}

\begin{proof}    
By Proposition \ref{prop G-diag}, any non zero element in  $V^+$ is $G$-conjugated to an element  $Z=x_0X_0+x_1X_1$ in  $V^+$ with  $(x_0,x_1)\neq (0,0)$. 
\medskip

If rank($Z$)$=1$  (i.e.   $x_0x_1=0$), as  we can use the element $\gamma_{0,1}$ of Proposition \ref{prop-gammaij} which exchanges  $\tilde{\go g}^{\lambda_0}$ and  $\tilde{\go g}^{\lambda_1}$, we can suppose  $x_1=0$ and  $x_0\neq 0$. The element  $h_{X_0+X_1}(\sqrt{x_0}^{-1})$ associated to $X_0+X_1$   belongs to  $L$ (see Lemma \ref{lem-tId-dansG}) and acts by  $x_0^{-1}Id_{V^+}$ on  $V^+$. It follows that  $Z$ is  $L$-conjugated  to  $X_0$.\medskip

Suppose now that rank($Z$)$=2$  (i.e.   $x_0x_1\neq 0$). Then, as above, we obtain that  $Z$ is  $L$-conjugated  to  $Z_v=X_0+vX_1$ with  $v=x_0^{-1} x_1\neq 0$. \medskip

By  Theorem \ref{thpropridelta_0}, one has  $\Delta_0(Z_v)=v\Delta_0(Z_1)$. It follows easily that if  $v\notin w\chi_0(G)$, then the elements  $Z_v$ and  $Z_w$ are not  $G$-conjugated. \medskip

Suppose that   $v=w\mu $ with  $\mu\in \chi_0(G)=\chi_0(L)$ (cf. Lemma \ref{chiG}). Let   $g\in L$ such that $\chi_0(g)=\mu$. As $g\in L$, it  stabilizes the spaces  $\tilde{\go g}^{\lambda_j}$ for  $j=0$ or  $1$. As  ${\rm dim} \tilde{\go g}^{\lambda_j}=\ell=1$,   there exist  $\alpha$ and $\beta$ in  $F^*$  such  that $g.X_0=\alpha X_0$ and  $g.X_1=\beta X_1$. Therefore   $\Delta_0(g.(X_0+X_1))=\Delta_0(\alpha X_0+\beta X_1)=\alpha\beta\Delta_0(X_0+X_1)$, and hence  $\alpha\beta=\chi_0(g)=\mu$.  \medskip

Then  $g.(X_0+wX_1)=\alpha X_0+\beta wX_1=\alpha^{-1}(\alpha^2 X_0+vX_1)$.  The element  $h_{X_0}(\alpha)$ associated to  $X_0$ belongs to $L$ and  satisfies $ h_{X_0}(\alpha)X_0=\alpha^2 X_0$ and  $h_{X_0}(\alpha)X_1=X_1$. The element  $h_{X_0+X_1}(\sqrt{\alpha}^{-1})$  also belongs to  $L$ and acts by multiplication by  $\alpha^{-1}$ on  $V^+$. As  $g.(X_0+wX_1)= h_{X_0+X_1}(\sqrt{\alpha}^{-1})\circ h_{X_0}(\alpha)(X_0+vX_1)$, the elements  $X_0+wX_1$ and $X_0+vX_1$ are  $L$-conjugated. This ends the proof of $(1)$ and $(2)$.
\medskip

\noindent It remains to prove the  last assertion.  

Lemma  \ref{chiG} and  Lemma \ref{lem-ensembleSe} imply that $F^{*2}\subset \chi_0(G)\subset\{a\in F^*; aq_e\sim q_e\}=S_e$.\\
Therefore,  if  $e=1$ or $e=3$,  we have $S_e=F^{*2}=\chi_0(G)$. 
 
Let us suppose now that  $e$ even  and hence $e\in\{0,2,4\}$. The proofs will depend on the values of  $e$ and $d$. \\
If   $(d,e)=(2,0)$ or  $(2,2)$ (cases (1) and (2) in Table 1) then the algebra $\tilde{\go g}$ satisfies the condition   {\it 2 (b)} of corollary \ref{gia} and hence $\chi_0(G)=S_e$.

It remains to study the cases where  $e\in\{0,2,4\}$ and  $d\geq 4$, corresponding to the cases  (8), (9) and (10) of  Table 1. From Remark  \ref{rem-simple}, one can suppose that     $\tilde{\go g}={\go o}(q_{(m+r,m-r)})$ with  $2r=e$ and  $m\geq 4$. We will describe precisely the  group  $G$ is this case.

The quadratic form  $q_{(m+r,m-r)}$ is the sum of  $m-r$ hyperbolic planes  and of an anisotropic quadratic form  $q_{an}^{2r}$ of rank $2r$, where  $q_{an}^{0}=0$. There  exists a basis   $(e_1,\ldots , e_{m},  e_{-m}, \ldots e_{-1})$ in which the matrix of $q_{(m+r,m-r)}$  is given by $$S_{2m,2r}=\left(\begin{array}{ccc} 0 & 0 & s_{m-r}\\ 0 & J_{an,2r} & 0\\ s_{m-r}  &0 & 0\end{array}\right),$$
where  $s_n$ stands for the square matrix of size  $n$ whose coefficients are all zero  except those on the second diagonal  which are equal to  $1$ and where   $J_{an,2r}$ is the  square matrix of size $2r $    of the form  $q_{an}^{2r}$ which is supposed to be diagonal for $r\neq 0$ and equal to the empty block for $r=0$. \medskip
Hence
$$\tilde{\go g}=\{X\in M_{2m}(F),\, ^tXS_{2m,2r}+S_{2m,2r}X=0\}$$

  Then the maximal split abelian subalgebra $\go{a}$  of $\tilde{\go g}$ is the set of diagonal elements  $\sum_{i=1}^{m-r}\varphi_{i} (E_{i,i}-E_{2m-i+1,2m-i+1})$ with $\varphi_i\in F$ (as usual   $E_{i,j}$ is  the matrix  whose coefficients are all zero except   the coefficient of index $(i,j) $ which is equal to $1$) and the algebra   $\tilde{\go g}$ is graded by the element $H_0=2E_{1,1}-2E_{2m,2m}$ . It is then easy to see that   $V^+$ is isomorphic  $F^{2m-2}$ by the map

   $$y\in F^{2m-2}\mapsto X(y)=\left(\begin{array}{ccc} 0 & y& 0\\ 
0&0&-S_{2(m-1),2r}^{-1}\;^{t}y\\
0&0&0\end{array}\right)\in V^+.$$

\noindent Set  $X_0=X(1,0,\ldots,0)\in \tilde{\go g}^{\lambda_0}$ and  $X_1=X(0,0,\ldots,1)\in \tilde{\go g}^{\lambda_1}$.
\medskip

We denote by  ${\mathcal Sim}_0(q_{(m+r,m-r)})$ the group of direct similarities of  $q_{(m+r,m-r)}$, that is the group of elements  $A\in GL(2m,F)$ such there exists $\mu\in F^*$ satisfying  $^{t}AS_{m,r} A=\mu S_{m,r}$ and such that  ${\rm det}(A)=\mu^m$. We will denote by  $\mu(A)=\mu$ the ratio of  $A$.\medskip

If $e=r=0$, the algebra $\tilde{\go g}$ is split and from   ([Bou] Chap VIII \textsection 13 $n^\circ 4$ page 211), we know that the group   ${\rm Aut}_0(\tilde{\go g})$ is the group of automorphisms of the form  $Z\mapsto AZA^{-1}$ where $A\in{\mathcal Sim}_0(q_{(m,m)})$. It is easy to see that an  element  $A\in {\mathcal Sim}_0(q_{(m,m)})$ commutes with  $H_0$ if and only if there  exists  $b\in F^*$ and  $g_1\in {\mathcal Sim}_0(q_{(m-1,m-1)})$ of ratio  $ \mu(g_1)$ such that $A= \left(\begin{array}{ccc} \mu(g_1) b^{-1} & 0 & 0  \\
0 & g_1 & 0 \\ 0 &0 &b\end{array}\right)=b\left(\begin{array}{ccc} \mu(g_1) b^{-2} & 0 & 0  \\
0 & b^{-1}g_1 & 0 \\ 0 &0 &1\end{array}\right)$. As $b^{-1}g_1$ is a direct similarity of ratio  $\mu(g_1) b^{-2}$, we obtain that the group  $G$ is isomorphic to the group of direct similarities  ${\mathcal Sim}_0(q_{(m-1,m-1)})$ of the quadratic form $q_{(m-1,m-1)}$. Its action   on $V^+$ is given by  $g.X(y)=X(\mu(g)\; yg^{-1})$ for $g\in  {\mathcal Sim}_0(q_{(m-1,m-1)})$ with ratio $\mu(g)$.\medskip

\noindent Suppose now that   $r\neq 0$. Let  $g\in G$ and denote by $\bar{g}$ its natural extension to  $\tilde{\go g}\otimes \bar{F}$. As  $g$ commutes with  $H_0$, the same is true for  $\bar{g}$ and therefore , $\bar{g}$ stabilizes  $\bar{V}^+$ and  $\bar{V}^-$ (ie. $\bar{g}\bar{V}^+\subset \bar{V}^+$ and $\bar{g}\bar{V}^-\subset \bar{V}^-$). From the split case above, the action of $\bar{g}$ on  $\tilde{\go g}\otimes \bar{F}$ is given by the conjugation of an element  $\bar{A}=\left(\begin{array}{ccc} \mu & 0 & 0  \\
0 & \bar{g}_1 & 0 \\ 0 &0 &1 \end{array}\right)$with  $\mu\in \bar{F}^*$,  $^{t}\bar{g}_1S_{2(m-1),2r}\bar{g}_1=\mu S_{2(m-1),2r}$ and   ${\rm det}(\bar{g}_1)=\mu^{m-1}$. \medskip

\noindent Such an element acts on  $V^+$   by 
$$\bar{A}X(y)\bar{A}^{-1}=X(\mu  \;y \bar{g}_1^{-1})=\left(\begin{array}{ccc} 0 & \mu  \;y \bar{g}_1^{-1}& 0\\
0&0&-\bar{g}_1 S_{m-1,r}^{-1}\;^{t}y\\
0&0&0\end{array}\right),\quad  y\in F^{2m-2}.$$

\noindent As  $g\in G$ stabilizes $V^+$ and  $V^-$, we see that the coefficients of   $\bar{g}_1$ are in $F$ and   that  $\mu\in F$. Moreover the restriction of  $\bar{g}_1$ to  $F^{2m-2}$, which we denote by  $g_1$, is a direct similarity of  $q_{(m+r-1,m-r-1)}$ with ratio  $\mu$, i.e.  $g_1\in {\mathcal Sim}_0(q_{(m-1+r,m-1-r)})$. \medskip

Hence, in each case, the group  $G$ is isomorphic to the group  ${\mathcal Sim}_0(q_{(m+r-1,m-1-r)})$ of direct similarities of  $q_{(m+r-1,m-1-r)}$. Its action on $V^+$ is given by  $g.X(y)=\mu(g) X(yg^{-1})$ where $\mu(g)$ is the ratio of  $g$. Then the polynomial  $\Delta_0(X(y))=q_{(m+r-1,m-1-r)}(y)$ is relatively invariant under the action of  $G$ and its character is given by $\chi_0(g)=\mu(g)$.\medskip

All the anisotropic quadratic forms of rank  $4$ are equivalent (\cite{Lam}   Chapter VI, Theorem 2.2 page 152) and all the hyperbolic quadratic forms are equivalent   (\cite{Lam}   Chapter I, Theorem 3.2 page 9). Therefore, if $e=r=0$ or $e=4=2r$,   for any  $\mu\in F^*$, there exists $g\in {\mathcal Sim}_0(q_{(m+r-1,m-1-r)})$ whose ratio is  $\mu$. Hence $\chi_0(G)=F^*$.\medskip

If  $e=2=2r$, then by Lemma  \ref{lem-q2},  the anisotropic form  $q_{an}^{2}$ represents exactly $2$ classes   $a$ and $b$  in  $F^*/F^{*2}$ and    $\mu q_{an}^{2}\sim q_{an}^{2}$ if and only if    $\mu =1$ or $\mu=ab$ modulo $F^{*2}$. It follows that  if    $g\in {\mathcal Sim}_0(q_{(m+r-1,m-1-r)})$ then  $\mu(g)\in \{1,ab\}$ modulo $F^{*2}$. Hence the subgroup $\chi_0(G)$ is of index $2$ in $F^*$.  From Lemma \ref{chiG} and Lemma \ref{lem-ensembleSe}, we have $\chi_0(G)\subset S_e$ and $S_e$  is of index $2$ in $F^*$, hence $\chi_0(G)=S_e$.\medskip

This ends the proof of Proposition \ref{prop-k=1}.

\end{proof}

\begin{theorem}\label{thm-orbites-e04} Recall that $\ell =1$. In the case where  $e=0$ or $4$ we have    $\chi_0(G)=F^*$    and the group  
 $G$ has exactly  $k+1= \text{ rank}(\tilde{\go{g}})$  non zero orbits in $V^+$. These orbits are characterized by their rank and a representative of the unique open orbit is $X_0+\ldots +X_k$.  
 
 Moreover two generic elements in $\oplus_{j=0}^k\tilde{\go g}^{\lambda_j}$ are conjugated under the subgroup $L=Z_G(\go a^0)$.

\end{theorem}
\begin{proof} From Proposition \ref{prop G-diag}, any non zero element in  $V^+$ is  $G$-conjugated to an element  $Z=x_0X_0+\ldots +x_{m-1}X_{m-1}$ of  $V^+$ such that $\prod_{s=0}^{m-1} x_s\neq 0$. It suffices to prove that  $Z$ is  $L$-conjugated to  $X_0+\ldots X_{m-1}$.

If  $k+1=2$, the result is a consequence of    proposition \ref{prop-k=1}. 
   
Suppose now that  $k+1\geq 3$.  

 If  $d\leq 4$ (as  the conditions are then $\ell=1$, $e=0 \text{ or } 4$,  and $k+1\geq 3$, this corresponds to the   cases   1 (with $\delta=1$), (11) and   (12 ) in  the Table 1)  then the algebra  $\tilde{\go g}$ satisfies one of the two hypothesis in Corollary \ref{gia} (2). Hence,  for  $i\in\{0,,\ldots m-1\}$,  there exists $g_i^{x_i}\in L$  such that  $g_i^{x_i}X_i=x_iX_i$ and $g_i^{x_i}X_s=X_s$ for  $s\neq i$. It follows that if $g=\prod_{i=0}^{m-1} g_i^{x_i}$, then $g.(X_0+\ldots X_{m-1})=Z$, and this proves the required result.\medskip
 
In the case where  $e\in\{0,4\}$,  $k+1\geq 3$ and $d>4$ (and $\ell=1$ of course), the classification shows that this corresponds to the only case (13) in   Table 1, and that $k+1=3$. Again  Corollary \ref{gia} gives the result. 

\end{proof}

The next Theorem gives the $G$-orbits in $V^+$ in the case where  $e=1$ or $e=3$. If $e=d=1$ (case (6) in Table 1) this classification coincides with the classification of similarity classes of quadratic forms over  $F$: a quadratic form  $Q$ is {\it similar } to a quadratic form $Q'$ if and only if there exists  $a\in F^*$ such that   $Q$ is equivalent to $aQ'$.
\begin{theorem}\label{thm-orbites-e1}  Recall that $\ell=1$. Suppose that  $e=1$ or $e=3$. 
\begin{enumerate}\item If   $k+1=\text{ rank}(\tilde{\go{g}})=2$ (this corresponds to the cases (3), (4)  and  (6) with $n=2$ in Table 1), then  $\chi_0(G)=F^{*2}$. The group  $G$ has  $5$ non zero orbits, for which  $4$ are open.   A set of representatives of the open orbits is given by the elements $X_0+vX_1$ for $v\in F^*/F^{*2}$. \\Two generic elements of $ \tilde{\go g}^{\lambda_0}\oplus \tilde{\go g}^{\lambda_1}$ are  $L$-conjugated if and only if   they are  $G$-conjugated.
\item  If   $k+1=\text{ rank}(\tilde{\go{g}})\geq 3$ then   $e=d=1$. This corresponds to the case (6) in Table 1, namely the symplectic algebra. Remember also that  $F^*/F^{*2}=\{1,\varepsilon,\pi,\varepsilon\pi\}$ where  $\pi$ is a uniformizer  of $F$ and where  $\ep$ is a unit which is not a square.
\begin{enumerate}
\item If  $k+1=\text{ rank}(\tilde{\go{g}})\geq 3$ is odd  then $\chi_0(G)=F^{*}$ and the group  $G$ has $2$ open orbits with representatives given by 
$$X_0+\sum_{j=1}^{k/2} X_{2j-1}-X_{2j}$$ and  
$$ X_0-\varepsilon X_1-\pi X_2+\sum_{j=2}^{k/2} X_{2j-1}-X_{2j} .$$

\item If  $k+1=\text{ rank}(\tilde{\go{g}})\geq 3$ is even, then  $\chi_0(G)=F^{*2}$ and the group   $G$ has  $5$ open orbits in $V^+$, with representatives  given by
$$\sum_{j=0}^{(k-1)/2} X_{2j}-X_{2j+1},$$ 
$$X_0+vX_1+\sum_{j=1}^{(k-1)/2} X_{2j}-X_{2j+1},\;{\rm with }\;  v\in \{-\varepsilon,-\pi, \varepsilon\pi\}.$$

 and  $$X_0-\varepsilon X_1-\pi  X_2+\varepsilon\pi X_3+\sum_{j=2}^{(k-1)/2} X_{2j}-X_{2j+1}\;.$$

\item For $m\in\{0,\ldots, k\}$, two generic elements of  $V_m^+$ are $G$-conjugated  if and only if they are $G_m$-conjugated. 

If  $k+1=\text{ rank}(\tilde{\go{g}})=2p+1$ with $p\geq 1$ then the group  $G$ has  $7p$ non zero orbits  and if  $k+1=\text{ rank}(\tilde{\go{g}})=2p+2$ with  $p\geq 1$ then  $G$ has  $7p+5$ non zero orbits. The representatives of these orbits  are the representatives  of the $G_m$-orbits  of the generic elements of  $V_m^+$  where $m\in \{0,\ldots, k\}$.
\item Let  $X=x_0X_0+\dots +x_kX_k$ and  $X'=x'_0X_0+\dots +x'_kX_k$ be two generic elements in  $V^+$. Then  $X$and  $X'$ are  $L$ -conjugated  if and only if there exists $\mu\in F^*$ such that  $\mu x_ix'_i\in F^{*2}$ for all  $i\in\{0,\ldots , k\}$.

\end{enumerate}
\end{enumerate}
\end{theorem}

\begin{proof} If  $k+1=2$   the result is a consequence of  Proposition \ref{prop-k=1}.

If  $k+1\geq 3$ then the classification in Table 1 implies that   $d=e=1$. As we have noticed in Remark \ref{rem-simple}, one can suppose that   $\tilde{\go g}$ is the symplectic algebra $${\go sp}(2n,F)=\left\{\left(\begin{array}{cc} A & B\\ C & -^{t}A\end{array}\right);\; A, B, C\in M(k+1,F), ^{t}B=B, ^{t}C=C\right\}, \text{ with }k+1=n,$$
which is graded by the element  $H_0=\left(\begin{array}{cc} I_{k+1} & 0\\ 0 & -I_{k+1} \end{array}\right)$. It follows that,  $V^+$ (respectively $V^-$) can be identified with the space $Sym(k+1,F)$ of symmetric matrices of size  $k+1$ on $F$ through the map  $$B\in Sym(k+1,F)\mapsto X(B)=\left(\begin{array}{cc} 0 & B\\ 0 &0\end{array}\right)$$ (respectively through the map $$B \in Sym(k+1,F)\mapsto Y(B)=\left(\begin{array}{cc} 0 & 0\\ B &0\end{array}\right))$$ .   \medskip

The algebra  $\go a$ is then the set of  matrices $H (t_k,\ldots ,t_0) =\left(\begin{array}{cc} diag(t_k,\ldots ,t_0) & 0\\ 0& -diag(t_k,\ldots ,t_0) \end{array}\right)$ where  $diag(t_k,\ldots ,t_0) $ stands for the diagonal matrix whose diagonal elements are respectively  $ t_k,\ldots ,t_0$. The set of strongly orthogonal roots  associated to these choices is given by  $\lambda_j(H(t_k,\ldots ,t_0) )=2t_j$ for  $j=0,\ldots , k+1$ and the space   $\tilde{\go g}^{\lambda_j}$ is the space of matrices  $X(x_j {\mathbf E}_{k+1-j,k+1-j})$ with $x_j\in F$, where  ${\mathbf E}_{i,j}$ is the square matrix of size $k+1$ whose coefficients are zero except the coefficient of index $(i,j)$ which is equal to $1$.\medskip

 By  (\cite{Bou2}  Chap VIII, \textsection  13, $n^0 3$), we know that the group  ${\rm Aut}_0(\tilde{\go g})$ is the group  of conjugations by the similarities of the symplectic form defining  $\tilde{\go g}$. This implies that  $G$ is the group of elements  $[{\mathbf g},\mu]={\rm Ad}\left(\begin{array}{cc} {\mathbf g} & 0\\ 0 &\mu\;^{t}{\mathbf g}^{-1}\end{array}\right)$ where  ${\mathbf g}\in GL(k+1,F)$ and $\mu\in F^*$. Let us denote by $g=[{\mathbf g},\mu]$ such an element of  $G$. Its action on  $V^+$ is given by  $[{\mathbf g},\mu]B=\mu^{-1}{\mathbf g}B\;^{t}{\mathbf g}$ for  $B\in Sym(k+1,F)$.  We normalize the relatively invariant polynomial  $\Delta_0$ by setting $\Delta_0(X(B))={\rm det}(B)$ for $B\in Sym(k+1,F)$ and from above we see that  $$\chi_0([{\mathbf g},\mu])=\mu^{-(k+1)}{\rm det}({\mathbf g})^2.$$
 In particular,  one has  $\chi_0(G)=F^*$  if $k+1$ is odd and  $\chi_0(G)=F^{*2}$ if $k+1$ is even.\medskip

\noindent Hence the orbits of $G$ in  $V^+$ are the classes of similar quadratic forms. 

In order to give a set of representatives of these  $G$-orbits, we will first normalize the elements $X_j$. Remember that the  $ X_j,\; j=0,\ldots, k$ satisfy the conditions of Proposition \ref{prop-equivalenceqXiXj}. This means that  for $i\neq j$, the quadratic form  $q_{X_i,X_j}$ represents  $1$.  This quadratic form is defined on $E_{i,j}(-1,-1)$ by $q_{X_i,X_j}(Y)= -\dfrac{1}{2}b([X_i,Y],[X_j,Y])$.  As ${\rm dim} V^+=\dfrac{(k+1)(k+2)}{2}$, the normalization of the Killing form given in  Definition \ref{defb(X,Y)}  is 
$$b(X,Y)=-\dfrac{k+1}{2(k+1)(k+2)} \tilde{B}(X,Y)={\rm Tr}(XY),\quad X,Y\in\tilde{\go g}.$$
Then, if we set  $X_j=X(v_j{\mathbf E}_{k+1-j,k+1-j})$ with  $v_j\in F^*$, a simple computation shows that for  $Y=Y(y ({\mathbf E}_{k+1-i,k+1-j}+{\mathbf E}_{k+1-j,k+1-j}))\in E_{i,j}(-1,-1)$,  we have
$$q_{X_i,X_j}(Y)= v_iv_j y^2.$$
Hence  $q_{X_i,X_j}$ represents  $1$ if and only if  $v_iv_j\in F^{*2}$. Therefore for all $j\in\{0,\ldots k\}$, there exists $a_j\in F^*$ such that  $v_j=a_j^2 v_0$. Any element   $X=\sum_{j=0}^k x_j X_j$ is then conjugated to   $X(\sum_{j=0}^k x_j  {\mathbf E}_{k+1-j,k+1-j})$ by the element $g=[\mathbf g, v_0^{-1}]$ where  $\mathbf g$ is the diagonal matrix $diag(a_k,\ldots ,a_0)$.\medskip

For  $X=X(B)\in V^+$ with  $B\in Sym(k+1,F)$, let us denote by $f_X$  the quadratic form on $F^{k+1}$ defined by  $B$ (ie. $ f_X(z)=\;^{t}zB z$).  From above we obtain that for $X=\sum_{j=0}^k x_j X_j$, the quadratic form  $f_X$ is similar to the form  $T\in F^{k+1}\mapsto \sum_{j=0}^k x_j T_j^2$.\medskip

We describe now the similarity classes of quadratic forms.  From  Witt's Theorems (\cite{Lam}  Theorem I.4.1 and  Theorem I.4.2, p.12), any quadratic form  $Q$ of rank  $r$ is the orthogonal sum of an unique (up to equivalence) anisotropic form $Q_{an}$ of rank  $r_{an}$,   and a hyperbolic form $Q_{(m,m)}$ which is the sum of $m$ hyperbolic planes   with  $2m+r_{an}=r$ ($m$ is the so-called {\it Witt index} of $Q$). Moreover,  two quadratic forms are similar  if and only if  they have the same  Witt index and if their anisotropic parts  are  similar. \medskip
 
 We recall the following classical results  (\cite{Lam}   Chapter VI, Theorem 2.2 page 152, and Corollary 2.5 p. 153-154): \\
 $(1)$  Every quadratic form of rank $\geq 5$ is isotropic.\\
 $(2)$  Up to equivalence, there exists a unique anisotropic form of rank $4$, given by  $x^2-\varepsilon y^2-\pi z^2+\varepsilon\pi  t^2$.\\
  $(3)$  If  $Q$ is an anisotropic form of rank  $3$, then $Q$ represents every class modulo $F^{*2}$ except $-disc(Q)$ where $disc(Q)$ is the discriminant of  $Q$.\vskip 0,5cm
 
 If  $Q'$ is anisotropic of rank  $3$ with the same discriminant as $Q$ then $Q+disc(Q)t^2$ and $Q'+disc(Q) t^2$ are anisotropic of rank  $4$, hence they are equivalent. Witt's cancellation Theorem  (\cite{Lam}   Chapter I, Theorem 4.2, p.12) implies then that $Q\sim Q'$. Therefore there exist  $4$ equivalence classes of anisotropic quadratic forms  of rank $3$ characterized by the discriminant. Hence all the anisotropic quadratic forms of rank 3 are similar. Such a form is given by  $x^2-\varepsilon y^2-\pi z^2$.\vskip 0,5cm

We describe first the similarity classes of anisotropic quadratic forms of rank $2$. We know from Lemma  \ref{lem-q2}, that an anisotropic quadratic form $Q$ of rank  $2$ represents exactly  two classes of squares  $a$ and  $b$  in $F^*/F^{*2}$ which characterize the equivalence class of $Q$. Moreover  $\mu Q$ is equivalent to $Q$ if and only if  $\mu=1$ or $ab$ modulo $F^{*2}$.  \\
Let  $a\neq b$  be two elements of  $F^*/F^{*2}$. If  $ab\neq -1$ the form $ax^2+by^2$ is anisotropic  and represents $a$ and $b$ and if  $a=-b$ with  $-1\notin F^{*2}$ then  $-1$ the sum of  two squares (\cite{Lam}   Chapter VI, Corollary 2.6 page 154)  and then $ax^2+ay^2$ is anisotropic and represents $\pm a$. As there are four classes modulo $F^{*2}$, there are $\binom{4}{2}=6$   equivalence classes of anisotropic quadratic forms of rank $2$. \\
Let $a\neq b\in F^*/F^{*2}$ defining the equivalence class of an anisotropic form $Q$ of rank  $2$. Then, for  $w\neq 1, ab$ modulo $F^{*2}$, one has  $F^*/F^{*2}=\{1, ab, w, w ab\}$.  As $ab\,Q$ is equivalent to $ Q$, and as  $wQ$ is not  equivalent to $Q$,   a form which is similar to $Q$ is equivalent either to $Q$ or to  $ab\,Q$. Hence there are  $3$ similarity classes  of anisotropic quadratic forms of rank  $2$.
A set  of representatives of these classes are given by  $x^2+vy^2$ with  $v\in\{-\varepsilon, -\pi, \varepsilon\pi\}$. \medskip

Let  $Q$ be a quadratic form of rank  $k+1\geq 3$  and Witt index  $m$:  $Q=Q_{an}+Q_{(m,m)}$. Hence  ${\rm rank} ( Q_{an}) \leq 4$ and  $  k+1={\rm rank}(Q_{an})+2m$. By the classical results we recalled above, we get:\\
  
  - If $k+1$ is odd, then   ${\rm rank}(Q_{an})=1$ or  $3$, and  hence there are two similarity classes of quadratic forms, 
  
  -  and if  $k+1$ is even, then ${\rm rank}(Q_{an})=0, 2$ or  $4$ 	and hence there are  $5$ similarity classes of quadratic forms. \medskip

\noindent   The statements  $2(a)$ and  $2(b)$ are consequences of the  description of the anisotropic quadratic forms of rank $\leq 4$ given above.\medskip
  
\noindent Let us prove statement  {\it 2.(c)}.  
We denote by  $\iota$ the natural injection from  $ Sym(k+1-m,F)$ into  $ Sym(k+1,F)$ given by  $M\mapsto \iota(M)=\left(\begin{array}{c|c}  \begin{array}{ccc}  & & \\
& M &\\
 & & \end{array}&   \begin{array}{c} \\ 0\\ \\ \end{array}\\
\hline  
0 &0_m\end{array}\right)\in Sym(k+1,F)$. Therefore the space  $V_m^+$ is identified to the space  $ Sym(k+1-m,F)$ by the map  $M\mapsto X(\iota(M))$. An element $X(\iota(M))\in V_m^+$ is generic in  $V_m^+$ if and only if  ${\rm det}(M)\neq 0$. The group  $G_m$ is the group of elements  $g_1=[\mathbf g_1,\mu]$ with $\mathbf g_1\in GL(k+1-m,F)$ and  $\mu\in F^*$ acting on  $ Sym(k+1-m,F)$ by $[\mathbf g_1,\mu].M=\mu^{-1} {\mathbf g}_1 M\;^{t} {\mathbf g}_1$.\medskip

Let  $Z=X(\iota(M))$ and  $Z'=X(\iota(M'))$ be two generic elements in  $V_m^+$. If $Z$ and $Z'$ are $G$-conjugated  then there exist $\mathbf g\in GL(k+1,F)$ and  $\mu\in F^*$ such that  ${\mathbf g}\iota(M)\;^{t} {\mathbf g}=\mu \; \iota(M')$. Let  ${\mathbf g}_1\in GL(k+1-m,F)$ be the submatrix of ${\mathbf g}$ of the  $k+1-m$ first rows and the $k+1-m$ first columns of  ${\mathbf g}$. Then one sees easily that  ${\mathbf g}_1 M\;^{t} {\mathbf g}_1=\mu\; M'$. As $Z$ and  $Z'$ are generic in  $V_m^+$, the matrices $M$ and $M'$ are invertible  and hence  $ {\mathbf g}_1\in GL(k+1-m,F)$. This shows that  $Z$ and  $Z'$ are $G_m$-conjugated.\\
Conversely, if  $Z$ and  $Z'$ are  $G_m$-conjugated then there exist ${\mathbf g}_1 \in GL(k+1-m,F)$ and $\mu\in F^*$ such that ${\mathbf g}_1 M\;^{t}{\mathbf g}_1=\mu \; M'$. We set ${\mathbf g}=\left(\begin{array}{cc} {\mathbf g}_1 & 0\\ 0 & I_m\end{array}\right)$. The element  $[{\mathbf g},\mu]$ belongs to  $G$ and satisfies $Z'=[{\mathbf g},\mu].Z$. Hence $Z$ and  $Z'$ are $G$-conjugated. 

This proves that two generic elements in  $V_m^+$ are $G$-conjugated if and only if they are $G_m$-conjugated. \\
Let $S$ be the number of  non zero $G$-orbits in $V^+$ and let  $S_r$ be the number of $G$-orbits of rank $r$ in  $V^+$. 
By Proposition \ref{prop G-diag}, we have   $S=\sum_{r=1}^{k+1} S_r$. From above, $S_r$ is exactly the number of  open $G_{k+1-r}$- orbits in $V_{k+1-r}$.  Using the statements $1$, $2(a)$ and  $2(b)$ we see that $S_1=1$, $S_2=4$ and $S_{2p+1}=2$ for  $p\geq 1$,  $S_{2p}=5$  for  $p\geq 2$ . For $k=2$, we therefore have  $S=S_1+S_2+S_3=7$. If  $k=2p\geq 4$ is even, we get  $S=S_1+S_2+\sum_{j=1}^{p} S_{2j+1}+\sum_{j=2}^p S_{2j}= 5+2p+5(p-1)=7p$ and for  $k=2p+1\geq 3$ odd, we get $S=S_1+S_2+\sum_{j=1}^{p} S_{2j+1}+\sum_{j=2}^{p+1} S_{2j}=5+2p+5p=7p+5$.  This ends the proof of    {\it 2.(c)}.
\medskip

 Here the subalgebra  $\go a^0$ is equal to  $\go a$.  Therefore the group $L$ is the group of elements  $[{\mathbf g},\mu]$ where $\mu\in F^*$ and where ${\mathbf g}$ is a diagonal matrix of $GL((k+1),F)$. If  ${\mathbf g}$ is the diagonal matrix in  $GL(k+1,F)$ whose diagonal elements  are $(a_k,\ldots a_0)$, then we have    $[{\mathbf g},\mu].\sum_{j=0}^k x_j X_j=\mu^{-1}(a_0^2 x_0X_0+\ldots +a_k^2 x_k X_k)$. This proves  $2(d)$ and ends the proof of the Theorem.
 
 \end{proof}

\begin{theorem}\label{thm-orbites-e2}   Recall that $\ell=1$. Suppose that  $e=2$.
\begin{enumerate}\item If  $k+1=\text{ rank}(\tilde{\go{g}})=2$,  then  $[F^*: \chi_0(G)]=2$. This case is case $(9)$ and case (2) with $n=2$  in Table 1. The group  $G$ has  $3$ non zero orbits, for which  $2$ are open. A set of representatives of the open orbits is given by  $X_0+vX_1$ where $v\in F^*/\chi_0(G)$. Two generic elements of $ \tilde{\go g}^{\lambda_0}\oplus \tilde{\go g}^{\lambda_1}$ are conjugated under the group $L=Z_G(\go a^0)$ if and only if they are $G$-conjugated.

\item If $k+1=\text{ rank}(\tilde{\go{g}})\geq 3$ then  $e=d=2$. This case corresponds to  case $(2)$ in Table 1, namely  $\tilde{\go g}={\go u}(2n, E, H_n)$ where $E$ is a quadratic extension of $F$.
   
    \begin{enumerate}\item 
 If $k+1=\text{ rank}(\tilde{\go{g}})$ is odd then   $\chi_0(G)=F^*$ and the group   $G$  has a unique open orbit in  $V^+$.
 
\item If  $k+1=\text{ rank}(\tilde{\go{g}})$  is even then  $ \chi_0(G) =N_{E/F}(E^*)$    and the group  $G$ has two open orbits given by  ${\mathcal O}_1=\{X\in V^+; \Delta_0(X)=1\;{\rm mod}\; N_{E/F}(E^*)\}$ and   ${\mathcal O}_\eta=\{X\in V^+; \Delta_0(X)=\eta\;{\rm mod}\; N_{E/F}(E^*)\}$, where  $\eta\in F^*$ is such that $\{1,\eta\}$ is a system of representatives of $F^*/N_{E/F}(E^*)$.
\item  For $m\in\{0,\ldots, k\}$, two generic elements in  $V_m^+$ are  $G$-conjugated if and only if  they are  $G_m$-conjugated. If $k+1=2p$ is even then the group  $G$ has $3p$ non zero orbits and  if $k+1=2p+1$ is odd, the group  $G$ has  $3p+1$ non zero orbits. The   representatives  of these orbits are  the representatives of the orbits under  $G_m$ of the generic elements of  $V_m^+$  where $m\in \{0,\ldots, k\}$.
\item Let $X=x_0X_0+\dots +x_kX_k$ and  $X'=x'_0X_0+\dots +x'_kX_k$ be two generic elements of  $V^+$. Then $X$ and  $X'$ are $L$-conjugated if and only if there exists $\mu\in F^*$ such  $\mu x_ix'_i\in N_{E/F}(E^*)$ for all $i\in\{0,\ldots , k\}$.

 \end{enumerate} 

 \end{enumerate} 
\end{theorem}
\begin{proof}  The case  $k+1=2$ is a consequence of  Proposition \ref{prop-k=1}. 

If  $k+1\geq 3$, then from Table 1,  we have   $d=e=2$. Using Remark  \ref{rem-simple}, we can suppose that  $\tilde{\go g}$ is  ${\go u}(2n,E,H_n)$  where $E=F[\sqrt{\xi}]$ is a quadratic extension of  $F$. We denote by $x\mapsto \bar{x}$ the natural conjugation of $E$ and by    $N_{E/F}$ the norm on  $E$ defined by associated to this extension  $N_{E/F}(x)=x\bar{x}$. We also set $n=k+1$.

Define $S_n=\left(\begin{array}{cc}0 & I_n\\ I_n & 0\end{array}\right)$. Consider the Lie algebra
$$\tilde{\go g}=\{ Z\in {\go sl}(2n, E); ZS_n+ S_n\; ^{t}\bar{Z}=0\},$$ 
which is graded by the element $H_0=\left(\begin{array}{cc} I_n& 0\\ 0 & -I_n\end{array}\right)$. This implies that  
$$\tilde{\go g}=\left\{ Z=\left(\begin{array}{cc} A & B\\ C & -\;^{t}\bar{A}\end{array}\right); A,B,C\in M(n ,E)\;\; Tr(A -\;^{t}\bar{A})=0,\;^{t}\bar{B}=-B \;{\rm and }\;  ^{t}\bar{C}=-C\right\}.$$
The subspace  $V^+$ is then isomorphic to the space  ${\rm Herm}(n ,E)$ of hermitian matrices in  $M(n,E)$    (ie. matrices  $B$ such that  $^{t}\bar{B}=B$) through the map   $$B\in {\rm Herm}(n,E)\mapsto X(B)=\left(\begin{array}{cc}0 & \sqrt{\xi}B\\0 & 0\end{array}\right)\in V^+.$$\medskip

The space  $\go a$ is then the space of matrices   
$$H(t_0,\ldots,t_{n-1})=\left(\begin{array}{cc} diag(t_{n-1},\ldots,t_0)& 0\\ 0 & diag(-t_{n-1},\ldots,-t_0)\end{array}\right),$$ with  $(t_0,\ldots,t_{n-1})\in (F^*)^n$ and the roots  $\lambda_j$ are given by $\lambda_j(H(t_0,\ldots,t_{n-1}))=2t_j$. 

 We fix a basis of  $\tilde{\go g}^{\lambda_j}$ by setting $X_j=X(B_j)$ where $B_j\in {\rm Herm}(n,E)$ is a diagonal matrix whose coefficient are zero  except the coefficient of index  $(n-j,n-j)$ which is equal to $1$.  As in the proof of Theorem \ref{thm-orbites-e1}, we see easily that $X_0,\ldots ,X_k$ satisfy conditions of Proposition \ref{prop-equivalenceqXiXj}.\medskip

We will now describe the group  $G$.

As $\tilde{\go{g}}\otimes_{F}E=\go{sl}(n,E)$, the group  ${\rm Aut}_0(\tilde{\go g})$ is the subgroup of  ${\rm Aut}_0({\go {sl}}(2n,E))$ which stabilizes $\tilde{\go g}$, and hence (using \cite{Bou2}  Chap VIII, \textsection  13, $n^0 1$, VII, p.189) it is the group of the automorphisms  ${\rm Ad}(g)$ (conjugation by $g$) where  $g\in GL(2n,E)$ is such that there exists $\mu\in E^*$ satisfying  $^{t}\bar{g} S_ng=\mu S_n$. The group  $G$ is the subgroup of  ${\rm Aut}_0(\tilde{\go g})$  of elements which commute with $H_0$. Therefore an element of  $G$ corresponds to the action of ${\rm Ad}(g)$ where  $g=\left(\begin{array}{cc} {\mathbf g}_1 & 0 \\ 0 & {\mathbf g}_2\end{array}\right)\in GL(2n,E)$ is such that there exists  $\mu\in E^*$ satisfying  ${\mathbf g}_2=\mu\;^{t}\bar{{\mathbf g}_1}^{-1}$ and  ${\mathbf g}_1=\mu ^{t}\bar{{\mathbf g}_2}^{-1}$.  This  implies that  $\bar{\mu}=\mu$, hence $\mu\in F^*$, and ${\mathbf g}_2=\mu\;^{t}\bar{{\mathbf g}_1}^{-1}$. Finally:
$$G=\left\{ {\rm Ad}(g);\; g=\left(\begin{array}{cc} {\mathbf g} & 0 \\ 0 & \mu\;^{t}\bar{{\mathbf g} }^{-1}\end{array}\right), \mu\in F^*\; {\mathbf g}\in GL(n,E)\right\}.$$
We denote by $[{\mathbf g},\mu]$ such an element of  $G$. The action of  $[{\mathbf g},\mu]$ on  $V^+$ corresponds to the action  
$[{\mathbf g} ,\mu]. B=\mu^{-1} {\mathbf g} B\;^{t}\bar{\mathbf g}$ on  ${\rm Herm}(n,E)$.\medskip

The polynomial  $B\in {\rm Herm}(n,E)\mapsto {\rm det}(B)\in F^*$ is relatively invariant under the action of  $G$ and we normalize the polynomial  $\Delta_0$ on $V^+$ by setting $\Delta_0(X(B))={\rm det}(B),\; B\in {\rm Herm}(n,E)$. This implies that 
$$\Delta_0(x_0X_0+\ldots x_{n-1} X_{n-1})=x_0\ldots x_{n-1}.$$

\noindent Therefore  $$\chi_0([{\mathbf g}, \mu])=\mu^{-n} N_{E/F}({\rm det}({\mathbf g}))\qquad\qquad (*)$$
and hence if $n$ is even, we have  $\chi_0(G)= N_{E/F}(E^*)$ and if  $n$ is odd , we have $\chi_0(G)=F^*$. This is a part of statements $(2)(a)$ and $(2)(b)$\medskip

\noindent We describe now the $G$-orbits in  $V^+$. 
We will prove the results by induction on  $n= \text{ rank}(\tilde{\go{g}})$. \medskip

In what follows, we identify $V^+$ with $Herm(n,E)$ and we recall that the action of  $G$ is given by  $[{\mathbf g},\mu].B=\mu^{-1} {\mathbf g}B\;^{t}\bar{{\mathbf g}}$.   By Proposition \ref{prop  G-diag}, any generic element of $V^+$ is $G$-conjugated to  an element of the form $x_0X_0+\ldots+x_{n-1}X_{n-1}$. Therefore it suffices to study the $G$-orbits in the space of diagonal matrices (with coefficients in $F)$ under this action.
We set $I_\eta=\left(\begin{array}{cc} I_{n-1} & 0\\ 0 &\eta \end{array}\right).$

 If $n=2$ (ie. $k=1$), one has ${\rm det}(I_\eta)=\eta$  and ${\rm det}(I_2)=1$. As $\chi_0(G)=N_{E/F}(E^*)$, we see  from relation  $(*)$, that the elements  $I_\eta$ and $I_2$ 	are not conjugated.  \medskip

Let $X= \left(\begin{array}{cc}x_1& 0\\ 0 &x_0\end{array}\right)$ with  $x_0x_1\neq 0$.   As  $F^*/N_{E/F}(E^*)=\{1,\eta\}$, we obtain that if  $x_0x_1\notin N_{E/F}(E^*)$ then  $x_0=\eta a\bar{a}x_1$ for an element  $a\in E^*$ and hence $X= x_1\left(\begin{array}{cc}1& 0\\ 0 &a\end{array}\right)I_\eta\left(\begin{array}{cc}1& 0\\ 0 &\bar{a}\end{array}\right)=[{\mathbf g}, x_1^{-1}].I_\eta$ where ${\mathbf g}=\left(\begin{array}{cc}1& 0\\ 0 &a\end{array}\right)$, and therefore  $X$  is $G$-conjugated to  $I_\eta$.\medskip

If  $x_0x_1\in N_{E/F}(E^*)$, then:

 -  either  $x_0$ and $x_1$ belong to  $N_{E/F}(E^*)$, then  $x_0=a\bar{a}$ and $x_1=b\bar{b}$ and hence  $X$ is conjugated to  $I_2$ by the matrix  ${\mathbf g}=\left(\begin{array}{cc}b& 0\\ 0 &a\end{array}\right)$.

 - or $x_0$ and  $x_1$ belong to  $\eta N_{E/F}(E^*)$, then $x_0=\eta a\bar{a}$ and  $x_1=\eta b\bar{b}$ and then  $X=[{\mathbf g},\eta^{-1}].I_2$. \\
This ends the proof for  $n=2$.\medskip

We will need the following result for the induction: 
$$\textrm{ there exists } \; {\mathbf g}\in GL(2,E)\;\textrm{ such that }\; \eta I_2={\mathbf g}\;^{t}\bar{{\mathbf g}}.\qquad\qquad (**)$$
(Remember that from above $\eta I_2$ is either conjugated to $I_{2}$ or to $I_{\eta}$, this proves that it is actually conjugated to $I_{2}$).

The quadratic form on $F^4$ defined by $q(x,y,z,t)=x^2-\xi y^2+ z^2-\xi t^2$ is either isotropic or anisotropic of rank $4$,  hence it represents any $a\in F^*$. Thus  there exist $a,b,c,d\in F$ such that  $\eta=a^2-\xi b^2+c^2-\xi d^2$. We set $\alpha=a+\sqrt{\xi} b$, $\beta=c+\sqrt{\xi} d$ and  
${\mathbf g}=\left(\begin{array}{cc}\alpha & \beta\\ -\bar{\beta} &\bar{\alpha}\end{array}\right)$. One has  ${\rm det}({\mathbf g})=|\alpha|^2+|\beta|^2=a^2-\xi b^2+c^2-\xi d^2=\eta$. Finally  ${\mathbf g}\in GL(2,E)$ and  ${\mathbf g}\;^{t}\bar{\mathbf g}=(|\alpha|^2|+|\beta|^2)I_2=\eta I_2$.\medskip

This ends the proof of  $(**)$.\medskip

\noindent We suppose now that the statements  { 2.(a)}  and  { 2.(b)}  in the Theorem  are true if   $\text{ rank}({\tilde{\go{g}}})=p\leq n$ and we will prove that they remain true for $n+1$.

Let  $X=\left(\begin{array}{ccc} x_n & 0 & 0 \\ 0 & \ddots & 0\\ 0 & 0 & x_0\end{array}\right)$ with $x_j\in F^*$. We denote by  $n(X)$ the cardinality of set  $\{j\in\{0,\ldots, ,n\}, x_j\in \eta N_{E/F}(E^*)\}$.\medskip

\noindent Suppose first that  $n(X)$ is even.

As there exist automorphisms  $\gamma_{i,j}$ interverting  $X_i$ and  $X_j$ (see Proposition \ref{prop-gammaij}), we can suppose that  $x_j\in \eta N_{E/F}(E^*)$ if and only if  $j=0,\ldots, n(X)-1$ and then  $x_j=\eta a_j\bar{a}_j$ with $a_j\in E^*$, and if $j\geq n(X)$, we have  $x_j\in N_{E/F}(E^*)$ and hence  $x_j= a_j\bar{a}_j$ with $a_j\in E^*$. We denote by ${\bf g}_1$ a matrix in $GL(2,E)$ satisfying  $(**)$.

Let   ${\mathbf g}_\eta$ the block diagonal matrix      whose  $n-n(X)$ first blocks are just scalars equal to 1 and whose $\frac{1}{2}n(X)$ last blocks are equal to  ${\bf g}_1$. Let  $diag(a_n,\ldots , a_0)$ be the diagonal matrix whose diagonal coefficients are $a_n,\ldots, a_0$. Then  ${\mathbf g}=diag(a_n,\ldots , a_0) {\mathbf g}_\eta$ satisfies  ${\mathbf g}\;^{t}\bar{\mathbf g}=X$, in other words   $X$ is  $G$-conjugated  to  $I_{n+1}$.\medskip

\noindent Suppose now that  $n(X)$ is odd. \medskip

If  $n$ is even then   $n+1$ is odd, and hence  $n(\eta X)$ is even. From above, there exists ${\mathbf g}\in GL(n+1,E)$ such that   $\eta X={\mathbf g}\;^{t}\bar{\mathbf g}$. This means that    $X$ is $G$-conjugated to  $I_{n+1}$ by the element  $[{\mathbf g},\eta]$.\medskip

 We have proved that if  $n$ is even, any generic element is  $G$-conjugated to $I_{n+1}$.\medskip

If  $n$ is odd,  as ${\rm det}(I_{n+1})=1$,  ${\rm det}( I_\eta)=\eta$ and  $\chi_0(G)\subset N_{E/F}(E^*)$, the relation   $(*)$ implies that  $I_{n+1}$ and $I_\eta$ are not  $G$-conjugated.\medskip

 As   $n(X)$ is   odd,  using again the automorphisms $\gamma_{i,j}$, we can suppose that $x_0=\eta a\bar{a}$ with  $a\in E^*$. We can write $X=\left(\begin{array}{cc} X_1 & 0\\ 0 & x_0\end{array}\right)$ where  $X_1\in M(n,E)$ is the diagonal matrix  $diag(x_n,\ldots, x_1)$.
Then  $n(X_1)$ is even, and hence there exists ${\mathbf g}_1\in GL(n,E)$ tel que $X_1=  {\mathbf g}_1\;^{t}\bar{{\mathbf g}_1}$. If we set     ${\mathbf g}=\left(\begin{array}{cc} {\mathbf g}_1& 0\\ 0 & a\end{array}\right)\in GL(n,E)$, we get   $X={\mathbf g} I_\eta\;^{t}\bar{\mathbf g}$ and hence  $X$ is $G$-conjugated  to  $I_\eta$. \medskip

This ends the prove of statements  { (2) (a)}  and   {(2)  (b)} of the Theorem.\medskip

\noindent Let us now prove the statement 2 (c). Let $\iota$ be the natural injection of ${\rm Herm}(n-m,E)$ into ${\rm Herm}(n,E)$ which associates to  $M\in {\rm Herm}(n-m, E)$ the matrix  $\iota(M)=\left(\begin{array}{c|c}  \begin{array}{ccc}  & & \\
& M &\\
 & & \end{array}&   \begin{array}{c} \\ 0\\ \\ \end{array}\\
\hline  
0 &0_m\end{array}\right)\in M(n,E)$. This way we identify  the  $V_m^+$ to  $ {\rm Herm}(n-m,E)$ by the map  $M\mapsto X(\iota(M))$. An  element $X(\iota(M))\in V_m^+$  is generic in  $V_m^+$ if and only if  ${\rm det}(M)\neq 0$. The group  $G_m$ is the group of elements  $g_1=[\mathbf g_1,\mu]$ where  $\mathbf g_1\in GL(n-m,E)$ and  $\mu\in F^*$,  acting on $ {\rm Herm}(n-m,E)$ by  $[\mathbf g_1,\mu].M=\mu^{-1} {\mathbf g}_1 M\;^{t}\bar{\mathbf g}_1$.\medskip

Let $Z=X(\iota(M))$ and  $Z'=X(\iota(M'))$ be two generic elements in  $V_m^+$. If $Z$ 	and $Z'$ are  $G$-conjugated  then there exist $\mathbf g\in GL(n,E)$ and  $\mu\in F^*$ such that  ${\mathbf g}\iota(M)\;^{t}\bar{\mathbf g}=\mu\, \iota(M')$. Let ${\mathbf g}_1\in GL(n-m,E)$ be the submatrix  of ${\mathbf g}$ given by the first  $m-n$ rows and the first  $m-n$ first columns of  ${\mathbf g}$. An easy computations shows that  ${\mathbf g}_1 M\;^{t}\bar{\mathbf g}_1=\mu M'$. As $Z$ and $Z'$ are generic in  $V_m^+$, the matrices $M$ and  $M'$ are invertible  and therefore  $ {\mathbf g}_1\in Gl(n-m,E)$. The preceding relation shows then that $Z$ and  $Z'$ are  $G_m$-conjugated.\\
Conversely, if $Z$ and $Z'$ are $G_m$-conjugated then there exist ${\mathbf g}_1 \in GL(n-m,E)$ and $\mu\in F^*$ such that  ${\mathbf g}_1 M\;^{t}\bar{\mathbf g}_1=\mu M'$. We set  ${\mathbf g}=\left(\begin{array}{cc} {\mathbf g}_1 & 0\\ 0 & I_m\end{array}\right)$. The element $[{\mathbf g},\mu]$ belongs to  $G$ and we have  $Z'=[{\mathbf g},\mu].Z$. Hence   $Z$ and  $Z'$ are $G$-conjugated. The assertion concerning the number of orbits can be easily proved inthe same way as the corresponding assertion in Theorem \ref {thm-orbites-e1}. The statement   (2) (c)   is now proved.
\medskip

\noindent The group $L$ is the group of elements $l=[{\mathbf g},\mu]$ where  ${\mathbf g}=diag(a_{n-1},\ldots , a_0)$ is a diagonal matrix  and where $\mu\in F^*$. For such  an element  $l$, we have  $l.\sum_{j} x_jX_j=\sum_j (\mu^{-1} x_j a_j\bar{a}_j) X_j$. The last statement is then easy.

 This ends the proof. 
 
 \end{proof} 
 
  \begin{cor}\label{cor-Lorbits} Recall that $\ell=1$ and $S_e$ is defined in Lemma \ref{lem-ensembleSe}. Let $X=x_0X_0+\ldots +x_kX_k$ and $X'=x'_0X_0+\ldots +x'_kX_k$ be two generic elements of $\oplus_{j=0}^k\tilde{\go g}^{\la_j}$. Then $X$ and $X'$ are $L$-conjugate if and only if there exists $\eta\in F^*$ such that $\eta x_jx'_j\in S_e$ for $j=0,\ldots, k$.
 \end{cor} 
\begin{proof} \hfill

- Suppose first that $e=0 $ or $e=4$. Then $S_{e}=F^{*}$ and the result is an immediate consequence of Theorem \ref{thm-orbites-e04}.

- Let  $e\in\{1,2,3\}$.

Suppose $k+1=2$ and take $X=x_0X_0+x_{1}X_{1}$ and  $X'=x_0'X_0+x_{1}'X_{1}$ two generic elements. As the scalar multiplications belong to $L$, these elements are $L$-conjugate to $Y= X_{0}+\frac{x_{1}}{x_{0}}X_{1}$ and $Y'= X_{0}+\frac{x'_{1}}{x'_{0}}X_{1}$ respectively. From  Proposition \ref{prop-k=1} we know that $Y$ and $Y'$ are $L$-conjugate if and only if there exists $u\in S_e$ such that $\frac{x_{1}}{x_{0}}=u\frac{x'_{1}}{x'_{0}}$. As $F^{*2}\subset S_e$, we deduce that   $\frac{x_{1}x'_{1}}{x_{0}x'_{0}}\in S_e$. This implies the result in this case.

If $k+1\geq 3$, the result is exactly the assertion of Theorem \ref{thm-orbites-e1} (2)(d) for $e=1$ or $3$, and  the assertion of Theorem \ref{thm-orbites-e2}  (2)(d) for $e=2$.

 \end{proof}
 
   \subsection{ $G$-orbits  in the case  $\ell=3$}\label{subsection(l=3)}\hfill

In this subsection we suppose  $\ell=3$. \medskip

By Remark \ref{rem-simple}, we can assume that  $ \widetilde{\go g}$ is simple.
 If    $\ell=3$ (and rank($ \widetilde{\go g}$)$=k+1 > 1$) then it corresponds to case $(7)$ in Table 1 and its Satake-Tits diagram  is of type  $C_{2(k+1)}$, ($k\in \N$) and is given by

\hskip 150pt \hbox{\unitlength=0.5pt\begin{picture}(280,30)

  \put(55,10){\circle*{10}}
\put(60,10){\line (1,0){30}}
\put(95,10){\circle{10}}
\put(100,10){\circle*{1}}
\put(105,10){\circle*{1}}
\put(110,10){\circle*{1}}
\put(115,10){\circle*{1}}
\put(120,10){\circle*{1}}
 \put(130,10){\circle*{10}} 
\put(135,10){\line (1,0){30}}
\put(170,10){\circle{10}}
\put(175,10){\line (1,0){30}}
 \put(210,10){\circle*{10}}
\put(213,12){\line (1,0){41}}
\put(213,8){\line (1,0){41}}
\put(225,5){$<$}
\put(255,10){\circle{10}}

\end{picture}}

Note that case $(5)$ in Table 1 is a particular case of case $(7)$.

By  (\cite{Schoeneberg} Proposition 5.4.5.), $ \widetilde{\go g}$ splits  over any quadratic extension  $E$ of $F$, and (up to isomorphism) $ \widetilde{\go g}\otimes_F E\simeq {\go sp}(4(k+1),E)$.  We will use the following classical realization of $ \widetilde{\go g}$ (also used in  \cite{Mu}):\medskip

Let $F^{*2}$ the set of squares in  $F^*$. Let  $\varepsilon$ be a unit of  $F$ which is not a square and let  $\pi$ be a uniformizer of  $F$.   Recall that  the set 
$$ \{ {1},  {\varepsilon},  {\pi},  {\varepsilon\pi}\}$$
is a set of representatives of $F^*/F^{*2}$.
 We set  $E=F[\sqrt{\varepsilon}]$ and we note  $\mapsto \bar{x}$ the conjugation in $E$. For  $n\in\N$, we denote by  $I_n$ the identity matrix of size $n$. Consider the symplectic form  $\Psi$ on  $E^{4(k+1)}$ defined by 
$$\Psi(X,Y)={^{t}X} K_{2(k+1)} Y,\quad\textrm{where}\quad K_{2(k+1)}=\left(\begin{array}{cc} 0 & I_{2(k+1)}\\ -I_{2(k+1)} & 0\end{array}\right).$$

Then
$$ \widetilde{\go g}\otimes_F E={\go sp}(4(k+1),E)=\left\{\left(\begin{array}{cc} A & B\\ C & -^{t}A\end{array}\right);\; A, B, C\in M(2(k+1),E), {^{t}B}=B, {^{t}C}=C\right\}.$$
We set:
$$J_\pi=\begin{pmatrix} 0&\pi\\ 1&0\end{pmatrix}\in M(2,E),\,J=\begin{pmatrix} J_\pi &0&0\\ 0 & \ddots&0 \\ 0&0 &J_\pi \end{pmatrix} \in M(2(k+1),E),\, T= \begin{pmatrix} J & 0\\ 0 & ^{t}J\end{pmatrix} \in M(4(k+1),E).$$

\noindent The subalgebra  $ \widetilde{\go g}$ is then the subalgebra of  $ \widetilde{\go g}\otimes_F E$ whose elements are the matrices $X$ satisfying  $T\overline{X}=XT$. \medskip

Let us make precise the different objects which have been introduced  in earlier section in relation with the Lie algebra $ \widetilde{\go g}$.\\

\noindent The algebra $ \widetilde{\go g}$  is graded by the element $H_0= \begin{pmatrix} I_{2(k+1)} & 0\\ 0 & -I_{2(k+1)}\end{pmatrix} $.  More precisely we have
$$V^-=\left\{ \begin{pmatrix} 0 & 0\\ C &0\end{pmatrix} ;\;  C\in M(2(k+1),E),  ^{t}C=C, \;\;^{t}J\overline{C}=CJ\right\},$$
$$  V^+=\left\{\begin{pmatrix}  0 & B\\ 0 &0\end{pmatrix} ;\;   B \in M(2(k+1),E), ^{t}B=B, J\overline{B}= B\; ^{t}J \right\},$$
and 
$$\quad {\go g}=\left\{ \begin{pmatrix} A & 0\\ 0 & -{^{t}A}\end{pmatrix} ;\; A \in M(2(k+1),E), J\overline{A}=A J\right\}.$$

The subspace  $\go a$ defined by
$$\go a=\left\{H(t_0,\ldots, t_k)= \begin{pmatrix} \mathbf{H} & 0\\ 0 & -\mathbf{H}\end{pmatrix} ; \textrm{ with } \mathbf{H}= \begin{pmatrix}  t_kI_2 & & & \\  & t_{k-1} I_2  & & \\ & & \ddots & \\  & & & t_0 I_2\end{pmatrix} \in M(2(k+1),F)\right\}$$
is a maximal split abelian subalgebra of  $ \widetilde{\go g}$. If the linear forms  $\eta_j$ on  $\go a$ are defined by 
$$\eta_j(H(t_0,\ldots, t_k))= t_j,$$
then the root system of the pair  $( \widetilde{\go g},\go a)$ is given by 
$$ \widetilde{\Sigma}=\{ \pm \eta_i\pm \eta_j, \textrm{ for } \;\; 0\leq i<j\leq k\}\cup\{ 2\eta_j;\textrm{ for } \;\; 0\leq  j\leq k\}$$
and the set of strongly orthogonal roots given in  Theorem \ref{th-descente} is the set  
$$\{\lambda_0,\ldots,\lambda_k\}\quad\textrm{ where  }\quad \lambda_j= 2\eta_j.$$
We set also
$${\mathbb S}^+=\left\{X\in M(2,E); \; ^{t}X=X, J_\pi\overline{X}=X\; {^{t}J}_\pi\right\}=\left\{ \begin{pmatrix}  \pi \bar{x} &\mu\\ \mu& x\end{pmatrix} ; x\in E, \mu\in F\right\},$$
$$ {\mathbb S}^-=\left\{Y\in M(2,E); \; ^{t}Y=Y, \; ^{t}J_\pi\overline{Y}=YJ_\pi\right\}=\left\{ \begin{pmatrix} y &\mu\\ \mu&\pi \bar{y}\end{pmatrix} ; y\in E, \mu\in F\right\},$$ 
and 
$${\mathbb L}=\left\{ A\in M(2,E), J_\pi\overline{A}=AJ_\pi\right\}=\left\{ \begin{pmatrix} x &\pi y\\ \bar{y}&  \bar{x}\end{pmatrix} ; x,y\in E\right\}.$$
For  $M\in M(2,E)$, the matrix $E_{i,i}(M)\in M(2(k+1),E)$ is block diagonal,   the $(k+1-i)$-th  block being equal to  $M$  and all other $2\times 2$ blocks being equal to $0$. In other words:
$$E_{i,i}(M)= \begin{pmatrix}  X_k &  &\\ & \ddots & \\ & & X_0\end{pmatrix} \;{\rm where }\; X_j=0 \textrm{ if }\; i\neq j\;{\rm and } \; X_i=M.$$
All the algebras $ \widetilde{l}_j$ are isomorphic and given by 
 $$ \widetilde{l}_j=\left\{ \begin{pmatrix}  E_{j,j}(A) &E_{j,j}(X)\\ E_{j,j}(Y)&E_{j,j}( -\;^{t} A)\end{pmatrix}, A\in {\mathbb L}, X\in {\mathbb S}^+, Y\in   {\mathbb S}^-  \right\}.$$\\
The algebra  $ \widetilde{\go g}_j$ is the centralizer of  $ \widetilde{\go l}_0\oplus\ldots\oplus  \widetilde{\go l}_{j-1}$. It is given by 
$$ \widetilde{\go g}_j=\left\{ \begin{pmatrix}  {\mathbf A}_j & 0 & {\mathbf X}_j  & 0\\ 0 & 0 & 0 & 0\\ {\mathbf Y}_j  & 0 &- ^{t} {\mathbf A}_j  & 0\\ 0 & 0 & 0 & 0\end{pmatrix}  \in  \widetilde{\go g}\right\},$$
where the  matrices $ {\mathbf A}_j,  {\mathbf X}_j$ and  $ {\mathbf Y}_j$ are square matrices of size $2(k+1-j)$.\medskip 


Let us now describe the group  $G={\mathcal Z}_{{\rm Aut}_0( \widetilde{\go g)}}(H_0)$. \\
By (\cite{Bou2} Chap VIII, \textsection  13, $n^0 3$), we know that the group  ${\rm Aut}_0( \widetilde{\go g}\otimes_F E)={\rm Aut}( \widetilde{\go g}\otimes_F E)$ is the group of automorphisms of the form   $\varphi_s: X\mapsto sXs^{-1}$ where  $s$ is a symplectic similarity of the form  $\Psi$. This means that  $s\in GL(4(k+1),E)$ and there exists a scalar $\mu(s)\in E^*$ such that  $^{t}sK_{2(k+1)}s=\mu(s) K_{2(k+1)}$.  The scalar $\mu(s)$ is called the ratio  of the similarity $s$. It is easily seen that any similarity  $s$   with ratio $\mu$ such that  $\varphi_s(H_0)=H_0$ can be written $s=s({\mathbf g},\mu):=\left(\begin{array}{cc} {\mathbf g}& 0\\ 0 &\mu ^{t}{\mathbf g}^{-1}\end{array}\right)$ with   ${\mathbf g}\in GL(2(k+1), E)$   and $\mu\in E^*$.  As $s(\mu I_{2(k+1)},\mu^2)=\mu I_{4(k+1)}$, the elements  $s(\mu I_{2(k+1)},\mu^2)$ with  $\mu\in E^*$ act trivially on  $ \widetilde{\go g}\otimes_F E$. More precisely one can easily show that $\varphi_{s(g,\lambda)}=\text{Id}$ if and only if there exists  $\mu\in E^*$  such that $g=\mu  I_{2(k+1)}$ and $\lambda=\mu^2$.

Let  $H_E$ be the subgroup of  $GL(2(k+1), E)\times E^* $ whose elements are of the form  $(\mu I_{2(k+1)}, \mu^2)$ with  $\mu\in E^* $. Then, from above, the map 
$({\mathbf g}, \lambda)\mapsto \varphi_{s({\mathbf g},\lambda)}$ induces an isomorphism from  $GL(2(k+1),E) \times E^* /H_E$ onto the centralizer of  $H_0$ in ${\rm Aut}_0( \widetilde{\go g}\otimes_F E)$, which we will denote by  $G_E$. If  $({\mathbf g},\mu)\in GL(2(k+1),E) \times E^*$, we will denote by $[{\mathbf g},\mu]$ its class in  $G_E$. The action of   $G_E$ on  $ \widetilde{\go g}\otimes_F E$ stabilizes  $V^+\otimes_FE$ and    $V^-\otimes_FE$. More precisely the action on $V^+\otimes_FE$ is as follows:$$[{\mathbf g},\mu]\left(\begin{array}{cc} 0 &\mathbf X\\ 0 &0\end{array}\right)=\left(\begin{array}{cc} 0 &\mu^{-1} {\mathbf g}{\mathbf X}\; ^{t}{\mathbf g}\\ 0 &0\end{array}\right).$$

The group we are interested in, namely  $G={\mathcal Z}_{{\rm Aut}_0(\widetilde{\go g})}(H_0)$, is the subgroup of $G_E$ of all elements which normalize $ \widetilde{\go g}$. As  $V^+$ and $V^-$ generate  $ \widetilde{\go g}$,    $G$ is also the subgroup of  $G_E$ of all elements normalizing  $V^+$ and  $V^-$.\medskip

A result  in the spirit of the  following is indicated without proof in \cite{Mu} (p. 112).

\begin{prop}\label{prop-G} Let  $G^0(2(k+1))$ be the subgroup of elements  ${\mathbf g}\in GL(2(k+1), E)$ such that $J\bar{{\mathbf g}}=  {\mathbf g}J$. Then
$$G= \{g=[{\mathbf g},\mu], {\mathbf g}\in G^0(2(k+1))\cup\sqrt{\varepsilon} G^0(2(k+1)), \mu\in F^*\}$$\end{prop}
	
\begin{proof}

 We identify  $V^+$ with the space  $Sym_J(2(k+1))$ of symmetric matrices   $\mathbf X\in M(2(k+1),E)$ such that  $J\overline{\mathbf X}={\mathbf X}\;^{t}J$  through the map  ${\mathbf X}\mapsto  X=\left(\begin{array}{cc} 0 & {\mathbf X}\\ 0 & 0\end{array}\right)$.  
 
 As the extension $E=F(\sqrt{\varepsilon})$ is unramified (\cite {Lam}, Chap. VI, Remark 2.7 p. 154), the uniformizer $\pi$ of $F$ is still a uniformizer in $E$. We will now define a unit $e$ of $E$ which is not a square (and hence, as before for $F$, the  set   $\{1, e, \pi, e\pi\}$ will be a set of representatives of the classes  in  $E^*/E^{*2}$).  
\\
If $-1\in F^{*2}$ then one can easily see that   $\sqrt{\varepsilon}$ is not a square in  $E^*$. In this case we set  $e=\sqrt{\varepsilon}$ (of course $e$ is still a unit in $E$). If $-1\notin F^{*2}$, we can  suppose    that $\varepsilon=-1$. Then by    (\cite{Lam} Corollary V.2.6, p.154) $-1$  is a sum of two squares in $F^*$. That is  $-1=x_0^2+y_0^2$ ($x_0,y_0 \in F^*$) and in this case we set $e=x_0+y_0\sqrt{\varepsilon}\notin E^{*2}$ (again one verifies easily that $e$ is not a square in $E^*$).\medskip

Let   $g\in G_E$. From the definition of the group  $G_E$, one sees that there exist ${\mathbf g}\in GL(2(k+1), E)$ and $\mu\in\{1,e, \pi, e\pi\}$ such that  $g=[{\mathbf g},\mu]$.  We will now fix such a pair  $({\mathbf g},\mu)$. If  $g\in G$  then  $[{\mathbf g},\mu]V^+\subset V^+$, and this implies that for all  $\mathbf X\in Sym_J(2(k+1))$, one has
$$  J\overline{\mu^{-1} {\mathbf g}{\mathbf X}\; ^{t}{\mathbf g}}=\mu^{-1}  {\mathbf g}{\mathbf X}\; ^{t}{\mathbf g}\;^{t}J.$$
As  $J^2=\pi I_{4(k+1)}$, we obtain that for all ${\mathbf X}\in Sym_J(2(k+1))$, one  has
$$({\mathbf g}^{-1} J\overline{{\mathbf g}} J^{-1}) {\mathbf X} \; ^{t}({\mathbf g}^{-1} J\overline{{\mathbf g}} J^{-1}) =\mu^{-1} \overline{\mu}\; {\mathbf X}.\qquad\qquad \qquad\qquad (*)$$

The proposition will then be a consequence of the following Lemma:
\begin{lemme}\label{lemma-G} Let  $M\in M(2(k+1), E)$ and  $\nu\in E^*$ such that  $ M  {\mathbf X}\; ^{t}M= \nu{\mathbf X}$ for all  $\mathbf X\in Sym_J(2(k+1))$,  then there exists  $a\in E^*$ such that  $\nu=a^2$ and $M=a I_{2(k+1)}.$
\end{lemme}

 {\it Proof of Proposition \ref{prop-G} (Lemma \ref{lemma-G} being assumed)}.

Remember that we have fixed   a pair $({\mathbf g},\mu)$ in $GL(2(k+1), E)\times\{1,e,\pi, e\pi\}$ such that $g=[{\mathbf g},\mu]$.

$\bullet$ If  $\mu=1$ or  $\mu=\pi$ then $  \mu^{-1} \overline{\mu}=1$. The preceding Lemma and the relation $(*)$ imply ${\mathbf g}^{-1} J\overline{{\mathbf g}} J^{-1}=\pm I_{2(k+1)}$  and hence  $ J\overline{{\mathbf g}}=\pm {\mathbf g}J$. If    $J\overline{{\mathbf g}}= {\mathbf g}J$ then ${\mathbf g}\in   G^0(2(k+1))$. If     $J\overline{{\mathbf g}}= -{\mathbf g}J$  then $\sqrt{\varepsilon}{\mathbf g}$ satisfies  $J \overline{\sqrt{\varepsilon}{\mathbf g}}=\sqrt{\varepsilon} {\mathbf g}J$ and hence  ${\mathbf g}\in \sqrt{\varepsilon} G^0(2(k+1))$. Conversely, it is easy to see that if ${\mathbf g}$ verifies  $ J\overline{{\mathbf g}}=\pm {\mathbf g}J$ and if  $\mu$ belongs to $F^*$ then   $[{\mathbf g},\mu]$ stabilizes  $V^+$ and $V^-$.  \medskip

$\bullet \bullet$ We show now that if    $\mu=e$ or $\mu= e\pi$ then the element  $[{\mathbf g},\mu]$   of  $G_E$ does not belong to $G$.
 Suppose that   $\mu=e$ or  $\mu= e\pi$ and $[{\mathbf g},\mu]\in G$. \\
If $-1\in F^{*2}$ then  $-1=\alpha_0^2$ with  $\alpha_0\in F^{*}$ and we have set $e=\sqrt{\varepsilon}\notin E^{*2}$. Then  $ \mu^{-1} \overline{\mu}=-1=\alpha_0^2$. The preceding lemma implies   $J\overline{{\mathbf g}}=\epsilon \alpha_0 {\mathbf g}J$ with  $\epsilon=\pm1$. Taking the conjugate of this equality, one obtains $J{\mathbf g}=\epsilon \alpha_0 \overline{\mathbf g}J$. And therefore 
 $$\overline{\mathbf{g}}=\epsilon \alpha_0J^{-1}\mathbf{g}J=\alpha_0^2 J^{-2}\overline{\mathbf{g}}J^2=-\overline{\mathbf{g}},$$
and this is impossible as $\mathbf{g}\neq 0$.

\noindent If $-1\notin F^{*2}$, we have set $\varepsilon=-1 = x_0^2+y_0^2$, with  $x_0,y_0$ in  $F^*$, and $e=x_0+y_0\sqrt{\varepsilon}$.   Then  $ \mu^{-1} \overline{\mu}=\dfrac{\bar{e}^2}{e\bar{e}}=-\bar{e}^2=(\bar{e}\sqrt{\varepsilon})^2$. Hence the preceding Lemma gives $J\overline{{\mathbf g}}=\epsilon \bar{e}\sqrt{\varepsilon} {\mathbf g}J$ with $\epsilon^2=1$, and this implies that   $J{\mathbf g}=-\epsilon e\sqrt{\varepsilon}\overline{\mathbf g}J$ and therefore $$\overline{\mathbf{g}}=\epsilon \bar{e}\sqrt{\varepsilon}J^{-1}\mathbf{g}J=-\varepsilon\bar{e} e J^{-2}\overline{\mathbf{g}}J^2=-\overline{\mathbf{g}}.$$
Again this is impossible as $\mathbf{g}\neq 0$.

This proves the proposition.
\end{proof}

\noindent{\it Proof of Lemma} \ref{lemma-G}:

 For  $k=0$, we have  $Sym_J(2,E)=\mathbb S^+=\left\{\left(\begin{array}{cc} \pi \bar{x} &\lambda\\ \lambda& x\end{array}\right); x\in E,\lambda\in F\right\}$. Let us set  $M=\left(\begin{array}{cc} a & b\\ c& d\end{array}\right)\in GL(2,E)$. The relation which is satisfied by  $M$ and $\nu$ implies that for all $x\in E$ and  $\lambda\in F$, one has
$$\left\{\begin{array}{ccc} a^2\pi \overline{x}+b^2x+ 2 ab \lambda&=&\nu \pi\overline{x}\\
c^2 \pi\overline{x}+d^2 x+2dc\lambda&=& \nu x\\
ac\pi\overline{x}+bdx +(ad+bc)\lambda&=&\nu\lambda\end{array}\right.$$
The first two equations   imply   $b=c=0$ and $a^2=d^2=\nu$ and the third  equation  implies then $a=d$. Therefore it exists $a\in E^*$ such that 
$$M= a I_2,\quad{\rm and }\quad \nu=a^2\in E^{*2}.$$

On the other hand, a similar computation shows that if  $M\in M(2,E)$ satisfies $M{\mathbf X} \;^{t}M=0$ for all  ${\mathbf X}\in{\mathbb S}^+$ then $M=0$.\medskip

By induction we suppose that the expected result is true up to the rank  $k-1$. Let  $(M,\nu)\in M(2(k+1),E)\times E^*$ satisfying the hypothesis of  Lemma  \ref{lemma-G}. Let us write
$$M=\left(\begin{array}{cc} M' &  \begin{array}{c} M_k\\ \vdots \end{array}\\ \begin{array}{ccc} L_k &\ldots & L_{1}\end{array} & M_0\end{array}\right),$$
with  $M'\in M(2k,E)$ and  $M_j, L_j\in M_2(E)$. The equality  $M{\mathbf X}\; ^{t}M=\nu \mathbf X$ for all matrices  ${\mathbf X}\in Sym_J(2(k+1), E)$ of the form 
$${\mathbf X}=\left(\begin{array}{cc} {\mathbf X}' &  \begin{array}{c} 0\\ \vdots\\0 \end{array}\\ \begin{array}{ccc} 0 &\ldots & 0\end{array} & {\mathbf X}_0\end{array}\right),$$
(with  ${\mathbf X}'\in Sym_J(2k,E)$ and ${\mathbf X}_0\in Sym_J(2,E)$) implies then that for all ${\mathbf X}'\in Sym_J(2k,E)$ and all  ${\mathbf X_0}\in Sym_J(2,E)$, one has

  $$\left\{\begin{array}{cc } M' {\mathbf X}'\; ^{t}M'=\nu{\mathbf X}' & \\
  M_0{\mathbf X_0}\; ^{t}M_0=\nu {\mathbf X}_0&\\
 M_j {\mathbf X_0}\; ^{t}M_j=0& {\rm for }\quad j=1,\ldots k.\end{array}\right.$$
  By induction , there exists  $a\in E^*$ such that  $\nu=a^2$, $M_k=\ldots =M_1=0$ and  $M'=\epsilon' a I_{2k}, M_0=\epsilon_0 a I_2$ with $\epsilon',\epsilon_0\in\{\pm 1\}$. \\
  Taking  ${\mathbf X_0}=0$ and  $\mathbf X'$ diagonal by blocks, each block being equal to  ${\mathbf Y}\in Sym_J(2,E)$, one obtains 
  $$   L_j{\mathbf Y}\;^{t}L_j=0 \quad{\rm for } \quad  j=1,\ldots k,\quad \textrm{ for all } \; {\mathbf Y}\in Sym_J(2,E)
$$ and hence  $L_k=\ldots =L_1=0$ from the case  $k=0$.\medskip
 
 \noindent  The equality  $M{\mathbf X}\; ^{t}M=a^2 \mathbf X$ for all matrices  ${\mathbf X}\in Sym_J(2(k+1), E)$ of the form
$${\mathbf X}=\left(\begin{array}{cc}0&  \begin{array}{c} {\mathbf Y}\\0\\ \vdots\end{array}\\ \begin{array}{ccc} ^{t}{\mathbf Y} & 0&\ldots \end{array} & 0\end{array}\right),$$
where   ${\mathbf Y}\in M(2,E)$ and $J_\pi \overline{{\mathbf Y}}={\mathbf Y} \;{^{t}J_\pi}$ implies then that  $\epsilon'\epsilon_0=1$.

{\it This ends   the proof of Lemma  \ref{lemma-G}}.  
\cqfd
\medskip

  \begin{definition}\label{def-G0} The subgroup  $G^0$ of  $G$ is defined to be the subgroup of elements  $[\mathbf g,1]$ of  $G$, with  $\mathbf g\in G^0(2(k+1))$. Hence we have: 
$${\rm Aut}_e(\go g)\subset G^0\subset G.$$
\end{definition}
\begin{rem}\label{rem-inclusion-groupes} Recall that for  $j\in\{0,\ldots, k\}$, the group $G_j$ (with $G_0=G$) is the analogue of the group $G$ for the Lie algebra $\tilde{\go g}_j$ and that $G_j\subset \overline{G}$. From the preceding Definition, one has an injection  $G_j^0\hookrightarrow G^0$ given by  $[\mathbf g_j, 1]  \mapsto [\mathbf g, 1]$ 
 where $\mathbf g_j\in G^0(2(k+1-j))$ and where 
  $\mathbf g=\left(\begin{array}{cc} {\mathbf g_j}& 0\\ 0 &I_j\end{array}\right)$. In what follows, we will identify  $G_j^0$ with a subgroup of  $G^0$ under this injection. In particular, an element  $g_j\in G_j^0$ will have a trivial action on   $\oplus_{s=0}^{j-1} \tilde{\go g}^{\lambda_j}$.\\
  In the same manner, we will denote by  $L_i^0$ the corresponding subgroup  for the Lie algebra $\tilde{\go l}_i$.
  \end{rem}

\noindent We will now normalize the relative invariants  $\Delta_j$ on $V_j^+$ for  $j=0,\ldots, k$. The determinant is an irreducible polynomial on the space of symmetric matrices  which satisfies: 
$$\textrm{det}( \mu^{-1} {\mathbf g} {\mathbf X}\; ^{t}{\mathbf g})=\mu^{-2(k+1)}\textrm{det}({\mathbf g})^2 \textrm{det}( {\mathbf X}),\quad {\mathbf X}\in Sym_J(2(k+1), E), [{\mathbf g},\mu]\in G.$$
Therefore one can normalize the fundamental relative invariant  $\Delta_0$  by setting 
$$\Delta_0\left(\begin{array}{cc}0&  {\mathbf X}\\ 0& 0\end{array}\right)=(-1)^{k+1}\textrm{det}( {\mathbf X}).$$
The character $\chi_0$ is given by  
$$\chi_0(g) =\mu^{-2(k+1)} \textrm{det}({\mathbf g})^2$$ for  $g=[{\mathbf g},\mu]\in G$. This implies that
$$\chi_0(G)\subset F^{*2}.$$

\noindent Similarly, the fundamental relative invariant  $\delta_0$   of  $ \widetilde{\go g}^{\lambda_0}$ (cf. Definition \ref{defdeltaj}) is given by 
$$\delta_0(X)=-{\rm det}( {\mathbf X}); \quad {\rm for }\; X=\left(\begin{array}{cc} 0 &E_{0,0}({\mathbf X})\\0 & 0\end{array}\right),\;  {\mathbf X}\in{\mathbb S}^+.
$$\medskip

\noindent Let us now fix   elements     $\gamma_{i,j}\in{\rm Aut}_e(\go g)\subset G^0$  satisfying the properties of Proposition  \ref{prop-gammaij}. Then
$$\gamma_{i,j} \left(\begin{array}{cc} 0 &E_{i,i}({\mathbf X})\\0 & 0\end{array}\right) =\left(\begin{array}{cc} 0 &E_{j,j}({\mathbf X})\\0 & 0\end{array}\right), \;{\rm for }\;  {\mathbf X}\in{\mathbb S}^+.$$\\
 We normalize   the relative invariant polynomials  $\delta_j$ of   $\widetilde{\go g}^{\lambda_j}$ by setting
$$\delta_j(X)=\delta_0(\gamma_{j,0}(X)),\quad X\in \widetilde{\go g}^{\lambda_j}.$$
From Theorem \ref{th-k=0}, the  $L_j$-orbits de $\widetilde{\go g}^{\lambda_j}\backslash\{0\}$ are given by 
$$\{X\in   \widetilde{\go g}^{\lambda_j}; \delta_j(X)\equiv v\;{\rm mod} F^{*2}\},\quad v\in (F^*/F^{*2})\backslash\{-disc(\delta_j)\}\}.$$
Then, we normalize the polynomials  $ \Delta_i$ (defined by their restriction to $V_i^+$) (cf. Definition \ref{defdeltaj}) by setting
$$\Delta_{i}(X_i+\ldots +X_k)=\prod_{i=j}^k \delta_j(X_j),\quad X_i\in \widetilde{\go g}^{\lambda_i}.$$
\begin{lemme}\label{lem-G0-orbites}  Let  $\mathbf X$ and  $\mathbf X'$ be two elements of  $\mathbb S^+$ such that ${\rm det}(\mathbf X)\equiv {\rm det}(\mathbf X')\;{\rm mod}\; F^{*2}$. Then there exists $\mathbf g\in G^0(2)$ such that ${\mathbf X}= {\mathbf g}{\mathbf X}'\;^{t}{\mathbf g}$.
\end{lemme}

\begin{proof} We consider here the case  $k=0$, that is the algebra $\widetilde{\go l}_0=\widetilde{\go g}^{-\lambda_0}\oplus \go l_0\oplus \widetilde{\go g}^{\lambda_0}$ with $\go l_0=[\widetilde{\go g}^{-\lambda_0},\widetilde{\go g}^{\lambda_0}]\simeq \mathbb L$,  $\widetilde{\go g}^{\lambda_0}\simeq \mathbb S^+\simeq F^3$ and $\widetilde{\go g}^{-\lambda_0}\simeq \mathbb S^-$. Let  $Q$ be the quadratic form on  $\mathbb S^+$ defined by   $Q(\mathbf Y)=-{\rm det}( \mathbf Y)=-\pi a^2+\varepsilon\pi b^2+c^2$ for $\mathbf Y=\left(\begin{array}{cc} \pi(a+\sqrt{\varepsilon}b) &c\\ c & a-\sqrt{\varepsilon}b\end{array}\right)\in\mathbb S^+$ (where $a,b,c\in F$). Remember also that this form is anisotropic. \medskip

\noindent     Let $U$ be the connected algebraic subgroup of  $G^0(2)\subset GL(2,E)$ whose Lie algebra is $\go{u}=[\mathbb L, \mathbb L]\subset \go{sl}(2,E)$.  We will first give a surjection from $U$ onto $SO(Q)$. By \cite{Tauvel-Yu} (Corollary 24.4.5 and Remark 24.2.6) the group $U$ is the intersection of the algebraic subgroups of $GL(2,E)$ whose Lie algebra contains $\go{u}.$ As the Lie algebra of $SL(2,E)$ is $\go{sl}(2,E)$ (see for example \cite{Borel}, Chap. I, 3.9 (d)), we get that $U\subset SL(2,E)$ and hence the elements in $U$ have determinant $1$.

 We denote by  $\Psi_{\mathbf g}$ the action of an element    ${\mathbf g}$ of  $ U$   on $\mathbb S^+$, in other words  $\Psi_{\mathbf g}(\mathbf Y)={\mathbf g}{\mathbf Y}\; ^{t}{\mathbf g}$. As $U\subset SL(2,E)$ one has $\Psi_{\mathbf g}\in O(Q)$. From Lemma \ref{lemma-G} one has  $\Psi_{\mathbf g}=Id_{{\bb S}^+}$ if and only if ${\bf g}=\pm I_{2}$. Therefore the map 
 $$\begin{array}{rll}
\Psi:  U_{1}=U/\{\pm I_{2}\}&\longrightarrow&O(Q)\\
 {\bf g}&\longmapsto& \Psi_{\mathbf g}
 \end{array}$$
 
 (defined up to a abuse of notation) is injective.

Let   $\overline{F}$ be an algebraic closure of  $F$.  It is well known that  $SO(Q,\overline{F})$ is connected. As $U$ is connected, the same is true for $U_{1}$. Therefore $\Psi(U_{1}({\overline F}))\subset SO(Q, \overline{F})$, and $\Psi$ is an isomorphism from $U_{1}$ on its image.

From \cite{Tauvel-Yu} (Theorem 24.4.1), the differential of $\Psi$ is injective. As the Lie algebra of $U_{1}$ (which is equal to $\go{u}$) has the same dimension ($3$) as $\go{o}(Q)$, the map $\Psi$ is also a submersion. Hence $\Psi(U_{1}({\overline F}))$ is open in $SO(Q,\overline{F})$. By \cite{Borel} Chap. I, Corollary 1.4, the group $\Psi(U_{1}({\overline F}))$ is also closed in $SO(Q,\overline{F})$. Therefore $\Psi(U_{1}({\overline F}))= SO(Q,\overline{F})$. 

\vskip 30pt

 Let  $\varphi\in SO(Q)$ and  $ \widetilde{\mathbf g}\in U(\overline{F})$ be such that $\Psi_{ \widetilde{\mathbf g}}=\varphi$. The element $g=\left(\begin{array}{cc}  \widetilde{\mathbf g} & 0\\ 0 & \;^{t} \widetilde{\mathbf g}^{-1}\end{array}\right)$ normalizes   $\widetilde{\go g}^{\lambda_0}$, hence by duality $g$  normalizes   $\widetilde{\go g}^{-\lambda_0}$  and therefore it normalizes   $\go l_0$ and $[\go l_0,\go l_0]$. This implies that $ \widetilde{\mathbf g}\in U$. Finally the map $\mathbf g\mapsto \Psi_\mathbf g$ is surjective from  $U$ onto  $SO(Q)$.\medskip

Let us now prove the Lemma. From the assumption, there exists  $x\in F^*$ such that  $Q(\mathbf X)=Q(x \mathbf X')$. From Witt's Theorem $SO(Q)$ acts transitively on the  set $\{{\mathbf Y}\in\mathbb S^+; Q({\mathbf Y})=t\}$, hence there exists  $\mathbf g\in U\subset G^0(2)$ such that $\mathbf X=x {\mathbf g}\mathbf X'\; ^{t}{\mathbf g}$. 

In order to prove the Lemma, it is now enough to prove that for all ${\mathbf Y}\in \mathbb S^+\setminus\{0\}$, the elements   $\varepsilon\mathbf Y$, $\pi \mathbf Y$ and  $\mathbf Y$ are conjugated under $G^0(2)$, as $F^*/F^{*2}=\{1,\varepsilon,\pi,\varepsilon\pi\}$.

Let ${\mathbf Y}\in \mathbb S^+\setminus\{0\}$. As  $J_\pi=\pi J_\pi^{-1}$, one has  $\pi {\mathbf Y}=J_\pi \overline{\mathbf Y}\; ^{t}J_\pi$. And as  ${\rm det}({\mathbf Y})={\rm det}(\overline{\mathbf Y})$, the preceding discussion implies that there exists ${\mathbf g}\in  U\subset G^0(2)$ such that  ${\mathbf g}{\mathbf Y}\;^{t}{\mathbf g}=\overline{\mathbf Y}$.  As  $J_\pi\in G^0(2)$, it follows that  $\pi{\mathbf Y}$ and ${\mathbf Y}$ are conjugated by the element $J_\pi{\mathbf g}$ of $G^0(2)$.\\

  Let us  prove that $ \varepsilon\mathbf Y$ is $G^0(2)$-conjugate to $\mathbf Y$. As $-disc(Q)= {\varepsilon}$,  a system of representatives of the $L_0$ - orbits in $\tilde{\go g}^{\lambda_0}\simeq \mathbb S^+$  is given by   
$$X_{1,0}=\left(\begin{array}{cc} \pi &0\\ 0 & 1\end{array}\right),\quad X_{ 0,1}=\left(\begin{array}{cc} 0 &1\\ 1 & 0\end{array}\right),\;\;{\rm and }\;\; X_{\alpha,0}=\left(\begin{array}{cc} \pi\overline{\alpha} &0\\ 0 & \alpha\end{array}\right),$$
where $\alpha\overline{\alpha}= {\varepsilon}$ (if  $-1=\alpha_0^2\in F^{*2}$ then $\alpha=\alpha_0\sqrt{ {\varepsilon}}$, and if  $-1\notin F^{*2}$ then $ {\varepsilon}=-1=x_0^2+y_0^2$ with $x_0,y_0\in F^*$  and  $\alpha= x_0+y_0\sqrt{ {\varepsilon}}$) (see Theorem \ref{th-k=0} 2)).\\
 By proposition \ref{prop-G}, $L_0$ is the group of elements $[\mathbf g,\mu]$ with $\mathbf g\in G^0(2)\cup\sqrt{ {\varepsilon}} G^0(2)$ and $\mu\in F^*$. Hence, there exist $y\in F^*$ and $\mathbf g\in G^0(2)$ such that $y {\mathbf g} \mathbf Y\;^{t}{\mathbf  g}$ equals to $X_{1,0}, X_{0,1}$ or $X_{\alpha,0}$. Thus it is enough to prove the result for this set of representatives.\\
  
Let  $\mathbf g_ {\varepsilon}=\left(\begin{array}{cc} \sqrt{ {\varepsilon}}&0\\ 0 & -\sqrt{ {\varepsilon}}\end{array}\right)\in G^0(2)$,  then  $\mathbf{g}_{\ \varepsilon}X_{1,0}\;^{t}\mathbf{g}_{ {\varepsilon}}= {\varepsilon}X_{1,0}$ and  $\mathbf{g}_{ {\varepsilon}}X_{\alpha,0}\;^{t}\mathbf{g}_{ {\varepsilon}}= {\varepsilon}X_{\alpha,0}$. If $\mathbf g_\alpha=\left(\begin{array}{cc} \alpha&0\\ 0 & \overline{\alpha}\end{array}\right)\in G^0(2)$,  then $\mathbf{g}_\alpha X_{ 0,1}\;^{t}\mathbf{g}_\alpha= {\varepsilon}X_{0,1}$. \\
The Lemma is proved.

\end{proof}

  \begin{cor}\label{cor-conjLj} Let  $k\in\N$. Let    $X=X_0+\ldots +X_k$ be a generic element of  $\oplus_{i=0}^k\tilde{\go g}^{\lambda_i}$. Let  $j\in\{0,\ldots, ,k\}$ and $X'_j\in \tilde{\go g}^{\lambda_j}$ such that   $\delta_j(X_j)=\delta_j(X'_j)$ mod $F^{*2}$. Then $X_0+\ldots X_{j-1}+X'_j+X_{j+1}+\ldots X_k$ is $L^0_j$-conjugated to    $X$.
 \end{cor}
  \begin{proof} Let  $X_j= \left(\begin{array}{cc} 0 &E_{j,j}({\mathbf X})\\0 & 0\end{array}\right)$ and $X'_j= \left(\begin{array}{cc} 0 &E_{j,j}({\mathbf X'})\\0 & 0\end{array}\right)$, where   ${\mathbf X}$ and  $\mathbf X'$ belong to  $\mathbb S^+$. Then  $\delta_j(X_j)=\delta_j(X'_j)$ mod $F^{*2}$ if and only if   ${\rm det}({\mathbf X})={\rm det}({\mathbf X}')$ mod $F^{*2}$. By Lemma  \ref{lem-G0-orbites}, the elements   ${\mathbf X}$ and  ${\mathbf X'}$ are  $G^0(2)$-conjugated. The result is then a consequence of the definition of  $L_j^0$ (see Remark \ref{rem-inclusion-groupes}).
  
  \end{proof}
\vskip 3pt
\begin{lemme}\label{lem-k=1} Suppose that  $k=1$. Let  $X=X_0+X_1$ and $Y=Y_0+Y_1$  be two elements of  $V^+$ such that   $X_j,Y_j\in\widetilde{\go g}^{\lambda_j}\backslash\{0\}$ and  $\delta_0(X_0)\equiv \delta_1(X_1)\;{\rm mod}\; F^{*2}$ and  $\delta_0(Y_0)\equiv \delta_1(Y_1)\;{\rm mod}\; F^{*2}$. Then $X$ and $Y$  are in the same     $G^0$-orbit.

\end{lemme}
\begin{proof} The normalization of  $\delta_1$, the hypothesis,   and Corollary \ref{cor-conjLj}  imply  that the elements  $X_1$ and  $\gamma_{0,1}(X_0)$(respectively  $Y_1$ and  $\gamma_{0,1}(Y_0)$) are in the same $L_1^0$-orbit. As  $L_1^0$ acts trivially on  $\widetilde{\go g}^{\lambda_0}$, one can suppose that $X_1=\gamma_{0,1}(X_0)$  and $Y_1=\gamma_{0,1}(Y_0)$.\medskip 

If   $\delta_0(X_0)\equiv \delta_0(Y_0)\;{\rm mod}\; F^{*2}$ then Corollary  \ref{cor-conjLj}   implies that  $X_j$ and $Y_j$ are conjugated by an element   $l_j$ in $L_j^0$ for  $j=0,1$ and then $X$ and $Y$ are conjugated by   $l_0l_1\in G^0$.\medskip

\noindent  We assume now that   $\delta_0(X_0)\not\equiv \delta_0(Y_0)\;{\rm mod}\; F^{*2}$. We can then write 
$$X=\left(\begin{array}{cc} 0 & \begin{array}{cc} {\mathbf X} & 0\\ 0 & {\mathbf X} \end{array}\\ 0 &0\end{array}\right)\quad {\rm and } \quad Y=\left(\begin{array}{cc} 0 & \begin{array}{cc} {\mathbf Y} & 0\\ 0 & {\mathbf Y} \end{array}\\ 0 &0\end{array}\right),$$
with  ${\mathbf X},{\mathbf Y}\in\mathbb S^+\backslash\{0\}$, and $\det({\mathbf X}) \not\equiv  \det({\mathbf Y}) \;{\rm mod}\; F^{*2}$.\\
 Consider the polynomial  $Q(a,b)$ defined on  $F^2$ by 
 
$$Q(a,b)= {\rm det}(a{\mathbf X}-{\mathbf Y})-b^2{\rm det}({\mathbf X}).$$

Then   $Q(a,b)= {\rm det}({\mathbf X}) (a^2-b^2)+2a\; C_0({\mathbf X},{\mathbf Y})+{\rm det}({\mathbf Y})$ where  $C_0$ is a polynomial function in the variables $({\mathbf X},{\mathbf Y})$. We set 
$C_1({\mathbf X},{\mathbf Y})=({\rm det}({\mathbf X}))^{-1}C_0({\mathbf X},{\mathbf Y})$, and hence
$$Q(a,b)= {\rm det}({\mathbf X}) (a^2-b^2)+2a\;{\rm det}({\mathbf X}) C_1({\mathbf X},{\mathbf Y})+{\rm det}({\mathbf Y}).$$
We obtain then  
$$Q(a,b)=  {\rm det}({\mathbf X})\Big[ \big(a+C_1({\mathbf X},{\mathbf Y})\big)^2-b^2\Big]+ {\rm det}({\mathbf Y})- {\rm det}({\mathbf X})C_1({\mathbf X},{\mathbf Y})^2.$$
As the quadratic form  $(A,B)\mapsto  A^{2}-B^{2}$ is isotropic, it represents all elements of  $F$ (\cite{Lam} Theorem I.3.4 p.10) and hence it exists  $(A_0,B_0)$ such that $$ {\rm det}({\mathbf X})(A_0^2-B_0^2)+ {\rm det}({\mathbf Y})- {\rm det}({\mathbf X})C_1({\mathbf X},{\mathbf Y})^2=0.$$
 It follows that the pair  $(a_0,b_0)=(A_0-C_1({\mathbf X},{\mathbf Y}),B_0)$ satisfies $Q(a_0,b_0)=0$ or equivalently,  ${\rm det} (a_0 {\mathbf X}-{\mathbf Y})=b_0^2{\rm det}({\mathbf X})$.  \\
  As $\mathbf X$ and $\mathbf Y$ belong to $\mathbb S^+$, we have ${\rm det} (a_0 {\mathbf X}-{\mathbf Y})=0$ if and only if $ a_0 {\mathbf X}-{\mathbf Y} =0$. As $\det({\mathbf X}) \not\equiv  \det({\mathbf Y}) \;{\rm mod}\; F^{*2}$,   we deduce that  $a_0\neq 0$ and $b_0\neq 0$. Therefore  $a_0 {\mathbf X}-{\mathbf Y}$ and   $-a_0{\mathbf X}$ are two non zero elements of  ${\mathbb S}^+$ such that  ${\rm det}(a_0{\mathbf X}-{\mathbf Y})\equiv{\rm det}(-a_0{\mathbf X})\;{\rm mod}\;F^{*2}$.  By Lemma \ref{lem-G0-orbites}, it exists  $ {\mathbf g}_1 \in  G^0(2)$ such that  $a_0{\mathbf X}-{\mathbf Y}=-   a_0 {\mathbf g}_1{\mathbf X}\;^{t}{\mathbf g}_1$, which is equivalent to  
$${\mathbf Y}=a_0{\mathbf X}+a_0 {\mathbf g}_1 {\mathbf X}(^{t}{\mathbf g}_1).$$
Let us set  $${\mathbf g}=\left(\begin{array}{cc} I_2 & {\mathbf g}_1\\ - {\mathbf X}(^{t}{\mathbf g}_1) {\mathbf X}^{-1} & I_2\end{array}\right).$$
 As $^{t}\mathbf X=\mathbf X$,   an easy computation shows that:
$$ {\mathbf g}\left(\begin{array}{cc} a_0{\mathbf X} & 0\\ 0 &a_0{\mathbf X}\end{array}\right)(^{t}{\mathbf g})=\left(\begin{array}{cc} {\mathbf Y} & 0\\ 0 & {\mathbf Y}'\end{array}\right),\;{\rm with  }\;\;{\mathbf Y}'=a_0{\mathbf X}+a_0{\mathbf X} (^{t}{\mathbf g}_1){\mathbf X}^{-1} {\mathbf g}_1{\mathbf X}\in \bb S^+.\eqno (*)$$

As ${\mathbf Y}=a_0{\mathbf X}+a_0 {\mathbf g}_1 {\mathbf X}(^{t}{\mathbf g}_1)\neq0$, one has  $ {\mathbf X}\neq - {\mathbf g}_1 {\mathbf X}(^{t}{\mathbf g}_1)$ and this is equivalent to   ${\mathbf X}(^{t}{\mathbf g}_1) {\mathbf X}^{-1} {\mathbf g}_1 \neq -I_2$. This implies that    ${\mathbf Y}'\neq 0$. Then,  by computing the determinant in $(*)$ above, we get  ${\rm det}({\mathbf g})\neq 0$ and  ${\rm det}({\mathbf Y})={\rm det}({\mathbf Y}')$ modulo $F^{*2}$. 

As   $J_\pi\overline{{\mathbf g}_1}={\mathbf g}_1J_\pi$ and   $J_\pi\overline{\mathbf X}={\mathbf X}\; ^{t}J_\pi$, a simple computation shows that $J\overline{{\mathbf g}}={\mathbf g}J$ and hence  ${\mathbf g}\in G^0(4)$. 

The relation ${\rm det}({\mathbf Y})={\rm det}({\mathbf Y}')$ modulo $F^{*2}$ and Lemma  \ref{lem-G0-orbites} imply  that there exists  ${\mathbf g}_0\in G^0(2)$ such that $$ {\mathbf g}_0  {\mathbf Y}'\; ^{t}{\mathbf g}_0=  {\mathbf Y} .$$
We set
$${\mathbf l}_0:=\left(\begin{array}{cc}  I_2 & 0\\ 0 & {\mathbf g}_0\end{array} \right)\in G^0(4).$$
Then    $g=[{\mathbf g}, 1]$ and  $l_0=[{\mathbf l}_0, 1]$ belong to $G^0$ and satisfy
 $ l_0g( a_0X)=Y$. Using Lemma \ref{lem-G0-orbites} again, there exists $\mathbf g_2\in G^0(2)$ such that $a_0\mathbf X= {\mathbf g}_2 {\mathbf X}\;^{t}{\mathbf g}_2$. Taking ${\mathbf g}'_2=\left(\begin{array}{cc}\mathbf g_2 & 0\\ 0 & \mathbf g_2\end{array}\right)\in G^0(4)$ and $g_2=[\mathbf g'_2,1]\in G^0$, we deduce that $l_0g g_2 X=Y$  and $l_0g g_2\in G^0$. The Lemma is proved.

\end{proof}


We are now able to describe the $G$-orbits in $V^+$ in the case $\ell=3$.

Let us set:
$$e_1=1,\,e_2=\pi,\,e_3= \varepsilon\pi $$

and remember that $\{1,\pi, \varepsilon\pi\}=F^*/F^{*2} \backslash\{ -disc(\delta_0)\}$.\\

For  $l\in\{1,2,3\}$ and  $X=X_m+\ldots +X_k\in V^+$ with $X_j\in \widetilde{\go g}^{\lambda_j}\backslash\{ 0\}$, we define 
$$n_l(X)= \#\{  j\in\{m,\ldots, k\} \text{  such that }\delta_j(X_j)=e_l\text{ mod } F^{*2}\}$$. 
\begin{theorem}\label{th-d=3} \hfill

Recall that $\ell=3$.

$1)$   Let $X=X_0+\ldots +X_k$ and $X'=X'_0+\ldots +X'_k$ be two elements of $V^+$ such that   $X_i \in  \widetilde{\go g}^{\lambda_i}\backslash\{ 0\}$ $($resp.  $X_j' \in  \widetilde{\go g}^{\lambda_j}\backslash\{ 0\}$$)$.  Then   the following assertions are equivalent:\\
$(a)$ $X$ and $X'$ are in the same $G$-orbit,\\
$(b)$  $n_l(X)\equiv n_l(X')$ mod $2$ for  $l=1, 2$ and $3$,\\
$(c)$ $X$ and $X'$ are in the same $G^0$-orbit. \medskip

$2)$ Suppose that the rank of $\widetilde{ \go{g}}$ is $k+1$. Then the number of $G$-orbits in  $V^+$ is  $4(k+1)$ with   $3$ open orbits if   the rank is $1$ $($i.e. $k=0$$)$ and  $4$ open orbits if  the rank is $\geq 2$ $($i.e. $k\geq 1$$)$. 
\medskip
 
$3)$ For  $v\in \{1,\pi, \varepsilon\pi\}$, let us fix a representative $X_0(v)\in\widetilde{\go g}^{\lambda_0}$ of the orbit  $\{Y\in  \widetilde{\go g}^{\lambda_0}; \delta_0(Y)\equiv v\;{\rm mod}\; F^{*2}\}$ and, if  $k\geq 1$,  we set  $X_j(v)=\gamma_{0,j}(X_0(v))$, for $j=1,\ldots k$. A set of representatives of the non zero orbits is then: 
$$X_0(1),\quad X_0(\pi),\quad X_0( \varepsilon\pi),$$
and,  if  $k\geq 1$, for  $m\in\{ 0,\ldots, k\}$,
$$\begin{array}{l}X_m(1)+\ldots+X_{k-1}(1) +X_k(1),\\
 X_m(1)+\ldots +X_{k-1}(1)+X_k(\pi),\\
 X_m(1)+\ldots+X_{k-1}(1) +X_k(\varepsilon\pi),\\
 X_m(1)+\ldots+X_{k-2}(1)+X_{k-1}(\pi)+X_k(\varepsilon\pi).\end{array}.$$
$($where we assume that if  $k=1$ then $X_{k-2}(1)=0$$)$. \\
For $k\geq 1$, the $4$ open orbits  are those of the preceding elements where $m=0$. For $k=0$, the $3$  open orbits are those of the elements $X_0(1), X_0(\pi)$ and $ X_0(  \varepsilon\pi)$.\end{theorem}
\begin{proof}\hfill 

 $ 1)$ 
  Clearly, $(c)$ implies $(a)$.
 
If  $X$ and  $X'$ are in the same  $G$-orbit then  $\Delta_0(X)=\Delta_0(X')$ mod $F^{*2}$ (because $\chi_0(G)\subset F^{*2}$), and hence 
$$\prod_{j=0}^k \delta_j(X_j)\equiv \prod_{j=0}^k \delta_j(X'_j)\;{\rm mod }\; F^{*2}.$$

This implies that 
$$\pi^{n_{2}(X)}(\varepsilon\pi)^{n_{3}(X)}=\pi^{n_{2}(X')}(\varepsilon\pi)^{n_{3}(X')} \;{\rm mod }\; F^{*2}.$$
Which is the same as:

$$\pi^{n_{2}(X)+n_{3}(X)}(\varepsilon)^{n_{3}(X)}=\pi^{n_{2}(X')+n_{3}(X')}(\varepsilon)^{n_{3}(X')} \;{\rm mod }\; F^{*2}.$$

And as $F^*/{F^*}^2\simeq (\Z/2\Z)^2$ this last equality implies that

$$n_{2}(X)+n_{3}(X)=n_{2}(X')+n_{3}(X')\,\text{ mod } 2 \, \text{ and } \,n_{3}(X)=n_{3}(X') \,\text{ mod } 2, $$

and therefore $n_{2}(X)=n_{2}(X')\,\text{ mod } 2$.

As  $n_1(X)+n_2(X)+n_3(X)= n_1(X')+n_2(X')+n_3(X')=k+1$, we obtain also   $n_1(X)\equiv n_1(X')$ mod $2$.  Finally
$$n_{\ell}(X)=n_{\ell}(X') \, \text { mod } 2, \text {for } \ell =1,2,3.$$
 Thus $(a)$ implies $(b)$.
 \medskip

Suppose that    $n_l(X)=n_l(X')$ mod $2$ for $l=1, 2$ and $3$. We will show by induction on $k$ that $X$ and  $X'$ are  $G^0$-conjugated. By Corollary \ref{cor-conjLj} the result is true for  $k=0$. \medskip

 Suppose now that   $k\geq 1$.\\
$\bullet$ If there exists an  $l\in\{1,2,3\}$ such that  $n_l(X)\equiv n_l(X')\equiv1$ mod $2$. Then, applying eventually the elements   $\gamma_{i,j}\in{\rm Aut}_e(\go g)\subset G^0$, we can suppose that  $\delta_0(X_0)=\delta_0(X'_0)=e_l$. From the case $k=0$ there exists  $g_0\in L_0^0$ such that  $g_0X'_0=X_0$. As   $L_0^0$ stabilizes the space $\oplus_{j=1}^k \widetilde{\go g}^{\lambda_j}$, we get  $g_0X'=X_0+X'_1+\ldots X'_k$. \\
By induction on the elements  $X_1+\ldots +X_k$ and  $X_1'\ldots +X'_k$ of $V_1^+$ there exists    $g_1\in G_1^0$ such that  $g_1(X_1'+\ldots +X'_k)=X_1+\ldots +X_k$. As $g_1$ stabilizes $ \widetilde{\go g}^{\lambda_0}$, we obtain  $g_1g_0 X'=X$ and hence  $X$ and $X'$ are  $G^0$-conjugated.\medskip

$\bullet$ Suppose that  $n_l(X)\equiv n_l(X')\equiv 0$ mod $2$ for all $l\in\{1,2,3\}$.  This implies that $k\geq 1$ is odd. The case  $k=1$ is a consequence of  Lemma \ref{lem-k=1}. Hence we assume $k\geq 3$.\\
 As  $n_1(X)+n_2(X)+n_3(X)= k+1 $, we cannot have $n_{\ell}(X)= 0$ (or  $n_{\ell}(X')= 0$) for all $\ell= 1,2,3$. Therefore there exist   $r\neq s$ and $r'\neq s'$ such that    $\delta_r(X_r)\equiv \delta_s(X_s)\;{\rm mod}\; F^{*2}$ and  $\delta_{r'}(X'_{r'})=\delta_{s'}(X'_{s'})\;{\rm mod}\; F^{*2}$ . Using the elements   $\gamma_{i,j}\in G^0$ if necessary, we can suppose that $r=r'=k-1$ and $s=s'=k$. 
 
 Then  $X_{k-1}+X_k$ and $X'_{k-1}+X'_k$ are two elements of  $V_{k-1}^+$ such that 
  $$\delta_{k-1}(X_{k-1})= \delta(X_k) )\;{\rm mod}\; F^{*2} \text { and }\delta_{k-1}(X'_{k-1})=\delta(X'_k)\;{\rm mod}\; F^{*2}.$$
  
  Then by  Lemma \ref{lem-k=1} there exists   $g_{k-1}\in G_{k-1}^0$  such that  $g_{k-1}(X'_{k-1}+X'_k)=(X_{k-1}+X_k)$ and hence  $g_{k-1}(X')= X'_0+\ldots +X'_{k-2}+X_{k-1}+X_k$.
  
The elements $ \widetilde{X}'=\gamma_{0,k-1}\gamma_{1,k}(X'_0+ \ldots +X'_{k-2})$ and   $ \widetilde{X}=\gamma_{0,k-1}\gamma_{1,k}(X_0+ \ldots +X_{k-2})$ of $V_{2}^+$ satisfy the condition  $n_l( \widetilde{X})= n_l( \widetilde{X}')\equiv 0\;{\rm mod}\; 2$ for all $l\in\{1,2,3\}$. By induction (applied to $V_2^+$) there exists   $g'_2\in G_2^0$ such that  $g'_2  \widetilde{X}'= \widetilde{X}$.  \medskip
   
  \noindent  The element   $\gamma_{0,k-1}\gamma_{1,k}(X_{k-1}+X_k)\in \widetilde{\go g}^{\lambda_0}+ \widetilde{\go g}^{\lambda_1}$ is fixed by $g'_2$. We obtain:  
 
  $$g'_2\gamma_{0,k-1}\gamma_{1,k}g_{k-1}(X')=g'_2\gamma_{0,k-1}\gamma_{1,k}( X'_0+\ldots +X'_{k-2}+X_{k-1}+X_k)$$
$$ =g'_2( \widetilde{X}')+\gamma_{0,k-1}\gamma_{1,k}(X_{k-1}+X_k)=  \widetilde{X}+\gamma_{0,k-1}\gamma_{1,k}(X_{k-1}+X_k)=\gamma_{0,k-1}\gamma_{1,k}(X),$$
 and this proves the first assertion.\medskip
 
 \noindent{ 2)}  Let $Z\in V^+\setminus \{0\}$ . We know from Theorem \ref{th-V+caplambda} that the element  $Z$ is $G$-conjugated to an element of the form  $Z_0+\ldots +Z_k$ with $Z_j\in \widetilde{\go g}^{\lambda_j}$. Let  $m$ be the number of indices  $j$ such that  $Z_j=0$. Using the elements  $\gamma_{i,j}\in G$ of Proposition \ref{prop-gammaij},  we see that  $Z$  is  $G$-conjugated  to an element of the form $X=X_m+\ldots +X_k$ with $X_j\in \widetilde{\go g}^{\lambda_j}\backslash\{0\}$ for $j=m,\ldots, k$.\\
 
  $Z$ belongs to an open orbit if and only if $m=0$ (if not we would have $\Delta_{0}(Z)=0$).
 
 If $k=0$ we have already seen (in Theorem \ref{th-k=0} 2)) that the number of open orbits is $3$.
 If $k\geq 1$,  the number of open orbits is equal, according to the first assertion, to the number of classes modulo $2$ of triples $(n_{1}(X),n_{2}(X),n_{3}(X)$) such that $n_{1}(X)+n_{2}(X)+n_{3}(X)=k+1$. This number of classes is $4$. \\

$3)$ Suppose $m\neq m'$. Then, according to Theorem \ref{th-qnondeg}, two elements  $X=X_m+\ldots +X_k$ with  $X_j\in \widetilde{\go g}^{\lambda_j}\backslash\{0\}$ and  $Y=Y_{m'}+\ldots +Y_k$ with  $Y_i\in \widetilde{\go g}^{\lambda_i}\backslash\{0\}$ are not in the same $G$-orbit because  ${\rm rang}\; Q_X\neq {\rm rang}\; Q_Y$.

 Finally, to conclude the proof, it will be enough to show that  two generic elements  $Y=Y_m+\ldots Y_k$ and $Y'=Y'_m+\ldots Y'_k$ of  $V_m^+$  are  $G$-conjugated if and only if they are  $G_m$-conjugated. Let  $g=[\mathbf g,\mu]\in G$ such that $gY=Y'$. Denote by  $\mathbf g_1 \in M(2(k-m+1),E)$ the submatrix of  $\mathbf g$ of the coefficients in the first  $k-m+1$ rows and columns. Set $${\mathbf g}'=\left(\begin{array}{cc} {\mathbf g_1} & 0\\ 0 & I_m\end{array}\right)$$\\
 A simple block by block computation shows that $[{\mathbf g}',\mu]\cdot Y=Y'$. This implies that $\Delta_{m}(Y)=\mu^{-2}\det({\mathbf g}_{1})^2\Delta_{m}(Y')$. As $Y$ and $Y'$ are generic in $V_{m}^+$, it follows that $\det({\mathbf g}_{1})\neq0$. Hence
  the element   $[{\mathbf g}_1,\mu]$ belongs to $G_m$ (see Proposition \ref{prop-G}).  \\
   Conversely, if $Y$ and $Y'$ generic in $V_m^+$ are $G_m$ - conjugate then, by the first assertion,  they are $G_m^0$-conjugate. Since $G_m^0\subset G^0$, this achieves the proof of the Theorem.\\
 
 \end{proof}

 \newpage
  \section{The symmetric spaces $G/H$}\label{section-G/H}

\subsection{The involutions}\label{section-involutions}\hfill

Let  $I^+$ \index{I+-@$I^{\pm}$}be a generic element of $V^+$. By Proposition \ref{prop-generiques-cas-regulier}, there exists $I^-\in V^-$ such that  $(I^-,H_0,I^+)$ is an  $\go sl_2$-triple. The action on $ \widetilde{\go g}$ of the non trivial element of the Weyl group of this $\go sl_2$-triple     is given by  the element  $w\in \widetilde{G}$ defined by  
$$w=e^{{\rm ad} \; I^+}e^{{\rm ad} \; I^-}e^{{\rm ad} \; I^+}=e^{{\rm ad} \; I^-}e^{{\rm ad} \; I^+}e^{{\rm ad} \; I^-}.$$
We denote by  $\sigma$ \index{sigma@$\sigma$} the corresponding isomorphism of $ \widetilde{\go g}$: $$\sigma(X)=w.X,\quad X\in  \widetilde{\go g}.$$
We denote also by  $\sigma$ the automorphism of $ \widetilde{G}$ induced by  $\sigma$:
$$\sigma(g)=w g w^{-1},\quad \textrm{ for } g\in \widetilde{G}.$$

If  $X\in \widetilde{\go g}$ is nilpotent, then  $\sigma(e^{{\rm ad}\; X})=e^{{\rm ad}\; \sigma(X)}.$

\begin{theorem}\label{th-invol}\hfill

The automorphism  $\sigma$ is an involution of  $ \widetilde{\go g}$ which satisfies the following properties:
\begin{enumerate}\item Define $\go h \index{hgothique@$\go h$}=\{ X\in\go g;  \sigma(X)=X\}$. Then $\go h={\go z}_{\go g}(I^+)= {\go z}_{\go g}(I^-)$.
\item Define $\go q=\index{qgothique@$\go q$}\{ X\in\go g;  \sigma(X)=-X\}$.Then ${\rm ad}\; I^+$ is an isomorphism from  $\go q$ onto $V^+$ and   ${\rm ad}\; I^-$ is an isomorphism from $\go q$ onto  $V^-$.
\item $\sigma(I^+)=I^-$ and  $\sigma(V^+)=V^-$. Moreover one has  $\sigma(X)=\dfrac{1}{2} \big({\rm ad }\; I^-\big)^2 X,\;\textrm{ for  } X\in V^+$ and   $\sigma(X)=\dfrac{1}{2} \big({\rm ad }\; I^+\big)^2 X,\;\textrm{ for } X\in V^-$.

\item $\sigma(H_0)=-H_0$ and  $\sigma(\go g)=\go g$. Moreover, for  $X\in\go g$, one has  $\sigma(X)=X+({\rm ad}\; I^-\; {\rm ad}\; I^+) X$.
\end{enumerate}

\end{theorem}
\begin{proof} For the convenience of the reader we give the proof although it is the same as for the real case (See \cite{B-R}). It is just elementary representation theory of the $\go sl_2$-triple $(I^-,H_0,I^+)$.

  The irreducible components of  $ \widetilde{\go g}$ under the action of this $\go sl_2$-triple  are of dimension  $1$ or  $3$  (because the weights of the primitive elements are $0$ or $2$). The action of $w^2$ is trivial on each of these components as they have odd dimension (see section \ref{sl2module}).  Hence  $\sigma$ is an involution of  $ \widetilde{\go g}$.\medskip

The subalgebra  $\go g$ is the sum of the  $0$-weight spaces  of theses irreducible components. If the dimension of the component is  $1$ (respectively  $3$) then the action of  $w$ is trivial (respectively multiplication by  $-1$). Therefore $\go h$ is the sum of the irreducible components of dimension $1$, and this proves the assertion (1), and  $\go q$ is the sum of the $0$-weight spaces of the irreducible components of dimension  $3$, and this proves the assertion  (2).\medskip

The space $V^+$ is the sum  of the subspaces of primitive elements of the irreducible components of dimension $3$. Hence the action of $w$ on $V^+$ is given by $\frac{1}{2}({\rm ad}\; I^-)^2$ (see section \ref{sl2module}). This implies that  $\sigma(I^+)=I^-$ and  $\sigma(V^+)=V^-$. \\
If  $X\in V^-$ then  $Y=({\rm ad }\; I^+)^2 X$ belongs to  $V^+$ and  $\sigma(Y)=({\rm ad } \;\sigma(I^+))^2 \sigma(X)=({\rm ad } \; I^-)^2 \sigma(X)$. From the preceding discussion, we obtain  $\sigma(Y)=\frac{1}{2}({\rm ad }\;  I^-)^2(Y)$. As $({\rm ad } \; I^-)^2$ is injective on  $V^+$, we get 
$$\sigma(X)= \frac{1}{2} ({\rm ad } \; I^+)^2 X.$$
The assertion  (3) is now proved.\medskip

As $H_0=[I^-, I^+]$, we have  $\sigma(H_0)=-H_0$ and therefore  $\sigma(\go g)=\go g$. If $X\in\go h$, we have ${\rm ad}\; I^-\;{\rm ad}\; I^+\; X=0$ and this means that $ \sigma(X)=X=X+ {\rm ad}\; I^-\;{\rm ad}\; I^+\; X$. \\
If  $X\in\go q\setminus\{0\}$, then  ${\rm ad}\; I^+ X$ is a non zero element of $V^+$.  Hence it is a primitive element of an irreducible component of dimension $3$. Therefore  ${\rm ad}\; I^+ {\rm ad}\; I^-\;{\rm ad}\; I^+X= -2\; {\rm ad}\; I^+ X$. As  ${\rm ad}\; I^+$ is injective on  $\go q$, we obtain ${\rm ad}\; I^-\;{\rm ad}\; I^+X=-2 X$. And hence 
$$w.X=-X=X+{\rm ad}\; I^-\;{\rm ad}\; I^+X.$$
This proves  (4).
\end{proof}

\begin{definition}\label{def-conditionC} An $\go{sl}_2$-triple $( I^-,H_0, I^+)$   is called a diagonal $\go{sl}_2$-triple   if  $I^+=X_0+\ldots X_k$ is a generic element of  $V^+$ such that  $X_j\in  \widetilde {\go g}^{\lambda_j}\setminus \{0\}$ and if  $I^-=Y_0+\ldots +Y_k$ where  $Y_j\in  \widetilde {\go g}^{-\lambda_j}\setminus \{0\}$  and   $(Y_j, H_{\lambda_j} X_j)$ is an ${\go sl}_2$-triple for all $j\in\{0,\ldots, k\}$. \end{definition}

\begin{rem} Any open $G$-orbit in $V^+$ contains an element  $I^+$ which can be put in a  diagonal $\go{sl}_2$-triple $( I^-,H_0, I^+)$ (this is a consequence of Theorem \ref{th-V+caplambda}).
\end{rem}

For the rest of this section, we fix  a  diagonal $\go{sl}_2$-triple $( I^-,H_0, I^+)$ and we will denote by  $\sigma$ the corresponding involution of  $ \widetilde{\go g}$.

Recall also that  $\go a$ is a maximal split abelian subalgebra of $\go g$ containing $H_0$ and that  $\go a^0$ is the subspace of  $\go a$ defined by 
$$  \go a^0=   \oplus_{j=0}^k FH_{\lambda_j}.$$

\begin{definition} \label{def-SECq}A maximal split abelian subalgebra of $\go q$ is called a Cartan subspace  of  $\go q$.
\end{definition}

\begin{lemme}\label{lem-SECq}\hfill

The maximal split abelian subalgebra  $\go a$ of $\go g$ satisfies $\go a=\go a\cap \go h\oplus {\go a}^0$. Moreover $\go a\cap \go h=\{H\in \go a; \lambda_j(H)=0\;{\rm for }\; j=0,\ldots , k\}$.\\
 The subalgebra   $\go a^0$ is a Cartan subspace of  $\go q$. 
\end{lemme}

\begin{proof} From Theorem  \ref{th-invol} (4), for $H\in \go a$, we get  $\sigma(H)= H +({\rm ad} I^-{\rm ad} I^+) H=H-\sum_{j=0}^k \lambda_j(H)H_{\lambda_j}$. This proves that   $\go a$ is  $\sigma$-stable and also the given decomposition of $\go a$.\medskip

Of course  $\go a^0$ is a split abelian subalgebra of  $\go q$.  It remains to show that  $\go a^0$ is maximal among such subalgebras. Let $X$ be an element  of  $ \go q$ such that   $\go a^0+F X$ is split abelian in  $\go q$. From the root space decomposition of  $\go g$ relatively to  $\Sigma$, we get  
$$X=U+\sum_{\lambda\in\Sigma} X_\lambda,\quad \textrm{ where  }\; U\in {\go z}_{\go g}(\go a)\; {\rm and }\; X_\lambda\in {\go g}^\lambda.$$
As  $X$ centralizes $\go a^0$,   if  $X_\lambda\neq 0$ for $\lambda\in\Sigma$,  we obtain that  $\lambda_{|_{\go {a}^0}}=0$. Corollary \ref{cor-orth=fortementorth} implies now that  $\lambda$ is strongly orthogonal to all roots  $\lambda_j$ and hence  ${\rm ad} I^+ X_\lambda=0$. Therefore $X_\lambda \in \go h$. As  $\sigma(U)$ belongs to  ${\go z}_{\go g}(\go a)$ and  $\sigma(X)=-X$, we  have  $X=U$. This implies  (maximality of $\go{a}$)   that $X\in \go a$ and hence $X\in \go a\cap\go q=\go a^0$. This proves that $\go a^0$ is a Cartan subspace of 
$\go q$.

\end{proof}

 \begin{lemme}\label{lemme-involution-modif} Let  $\underline{\sigma}$ \index{sigmasousligne@$\underline{\sigma}$} be the involution of  $ \widetilde{\go g}$ defined by  $\underline{\sigma}(X)=\sigma(X)$ for $X\in\go g$ and by  $\underline{\sigma}(X)=-\sigma(X)$ for $X\in V^-\oplus V^+$. Let  $ \widetilde{\go q}= \{X\in  \widetilde{\go g}, \underline{\sigma}(X)=-X\}$. Then  $\go a^0$ is a Cartan subspace of  $ \widetilde{\go q}$.
\end{lemme}
\begin{proof} As $\sigma=\underline{\sigma}$ on $\go g$, the space $\go a^0$ is a split abelian subspace of  $ \widetilde{\go q}$. It remains to prove the maximality. Let $X\in \widetilde{\go q}$ such that  $\go a^0+FX$ is abelian split. Then  $X$ commutes with $\go a^0$ and hence with $H_0$. Therefore  $X\in \go g$.  Then  $X\in \go a^0$ by Lemma \ref{lem-SECq}.  
\end{proof}

\begin{rem}\label{rem-decomp-racines} From  \cite{HW}  (Proposition 5.9), the set of roots $\Sigma( \widetilde{\go g}, \go a^0)$   of  $ \widetilde{\go g}$  with respect to $\go a^0$ is a root system which will be denoted by  $ \widetilde{\Sigma}^0$. The decomposition  $ \widetilde{\go g}$ given in Theorem  \ref{th-decomp-Eij}:
 $$ \widetilde{\go g}={\go z}_{ \widetilde{\go g}}(\go a^0)\oplus\Big( \oplus_{0\leq i<j\leq k} E_{i,j}(\pm1,\pm1)\Big)\oplus\Big(\oplus_{j=0}^k \widetilde{\go g}^{\lambda_j}\Big),$$
 is in fact the root space decomposition associated to the root system $ \widetilde{\Sigma}^0$. Setting
 $$\eta_j(H_{\lambda_i})=\delta_{i,j}=\left\{\begin{array}{ll} 1 & {\rm if}\; i=j\\ 0 &{\rm if }\; i\neq j\end{array}\right.,$$
we obtain 
 $$ \widetilde{\Sigma}^0=\{ \pm\eta_i\pm\eta_j (0\leq i<j\leq k), \pm 2 \eta_j (0\leq j\leq k)\}.$$
 And this shows that  $ \widetilde{\Sigma}^0$ is a root system of type $C_{k+1}$. In fact, as seen in the next Proposition, this root system   is the root system of a subalgebra $\tilde{\go{g}}_{C_{k+1}}\subset \tilde{\go{g}}$, which is isomorphic to $\go{sp}(2(k+1),F)$ and which contains $ \go a^0$ as a maximal split abelian subalgebra. \end{rem}
 
 \begin{prop}\label{prop-inclusionSP}\hfill\\
 For $0\leq i\leq k-1$, consider a family of  $\go{sl}_{2}$-triples $(B_{i},H_{\lambda_{i}}-H_{\lambda_{i+1}},A_{i})$ where $B_{i}\in E_{i,i+1}(-1,1)$ and $A\in E_{i,i+1}(1,-1)$ (such triples exist by Lemma \ref{lem-sl2ij}). Consider also an  $\go{sl}_{2}$-triple of the form $(B_{k},H_{\lambda_{k}},A_{k})$, where $B_{k}\in \tilde{\go g}^{-\lambda_{k}}$ and $A_{k}\in \tilde{\go g}^{\lambda_{k}}$. Then these $k+1$ $\go{sl}_{2}$-triples generate a subalgebra $\tilde{\go{g}}_{C_{k+1}}\subset \tilde{\go{g}}$  which is isomorphic to $\go{sp}(2(k+1),F)$ and which contains $ \go a^0$ as a maximal split abelian subalgebra.
  \end{prop}
  
  \begin{proof} The linear forms $\eta_{0}-\eta_{1},\eta_{1}-\eta_{2},\ldots,\eta_{k-1}-\eta_{k}, 2\eta_{k}$ form a basis of the root system $\widetilde{\Sigma}^0$ which is of type $C_{k+1}$ as seen in the preceding Remark. As the $\eta_{i}$'s form the dual basis of the $H_{\lambda_{i}}$'s, it is well known that the set of elements $\{H_{\lambda_{0}}-H_{\lambda_{1}}, H_{\lambda_{1}}-H_{\lambda_{2}}, \ldots, H_{\lambda_{k-1}}-H_{\lambda_{k}}, H_{\lambda_{k}}\}$ is a basis of the dual root system in $\go{a}^0$, which is of course of type $B_{k+1}$. Define $H_{i}=H_{\lambda_{i}}-H_{\lambda_{i+1}}$ with $0\leq i\leq k-1$, and $H_{k}=H_{\lambda_{k}}$. As usual we define also $n(\alpha,\beta)=\alpha(H_{\beta})$, for $\alpha,  \beta \in  \widetilde{\Sigma}^0$ 	and where $H_{\beta}\in \go{a}^0$ is the coroot of $\beta$. It is also convenient to set $\alpha_{i}=\lambda_{i}-\lambda_{i+1}$ for $i=0,\ldots,k-1$, and $\alpha_{k}=\lambda_{k}$. 
  
  Then the generators satisfy the following relations,  for $i,j\in \{0,\ldots,k\}$:
  
  \hskip 10pt $(1)$ $[H_{i},H_{j}]=0$,
  
  \hskip 10pt $(2)$ $[B_{i},A_{j}]=\delta_{i,j}H_{i}$, 
  
  \hskip 10pt $(3)$ $[H_{i}, A_{j}]= n(\alpha_{j},\alpha_{i})A_{j}$,
  
  \hskip 10pt $(3')$ $[H_{i}, B_{j}]= - n(\alpha_{j},\alpha_{i})B_{j}$,
  
   \hskip 10pt $(4)$ $(\ad B_{j})^{-n(\alpha_{i},\alpha_{j})+1}B_{i}=0 \text{ if }i\neq j$,
   
    \hskip 10pt $(5)$ $(\ad A_{j})^{-n(\alpha_{i},\alpha_{j})+1}A_{i}=0 \text{ if }i\neq j$.
    
    The relations $(1),(2),(3),(3')$ are obvious. Let us show relation $(4)$. The $\go{sl}_{2}$-triple $(B_{j},H_{j},A_{j})$ defines a structure of finite dimensional $\go{sl}_{2}$-module on $\tilde{\go g}$. We have $[A_{j},B_{i}]=0$ and $[H_{j}, B_{i}]= -n(\alpha_{i},\alpha_{j})B_{i}$ by relations $(2)$ and $(3')$. This means that $B_{i}$ is a primitive vector of weight $-n(\alpha_{i},\alpha_{j})$. Therefore  $B_{i}$ generates an $\go{sl}_{2}$-module of dimension $-n(\alpha_{i},\alpha_{j})+1$. And this implies $(4)$. The same argument proves $(5)$.
    
    The above relations are the well known Serre relations for $\go{sp}(2(k+1),F)$. Hence the algebra generated by these elements is isomorphic to $\go{sp}(2(k+1),F)$.
    
     \end{proof}
    
     \begin{rem}\label{rem-sp=admissible}
    The preceding construction of the   subalgebra $\tilde{\go{g}}_{C_{k+1}}$ uses  the same argument as the construction of the so-called ``admissible'' subalgebras (see \cite{Rubenthaler2}, Th\'eor\`eme 3.1 p.273).
 \end{rem}

\subsection{The minimal $\sigma$-split parabolic subgroup $P$ of  $G$.}\label{sec:sigma-parabolic}
\vskip 10pt
\begin{definition}\label{def-sigma-parabolique} (\cite{HW} \textsection 4.) A parabolic subgroup  $R$  of $G$ (resp. a parabolic subalgebra $\go r$ of $\go g$) is called a $\sigma$-split parabolic \index{sigmasplitparabolic@$\sigma$-split parabolic} subgroup of $G$ (resp. a $\sigma$-split parabolic subalgebra of $\go g$) if  $\sigma(R)$ (resp. $\sigma(\go r)$) is the opposite parabolic subgroup  (resp. parabolic subalgebra) of  $R$ (resp. of $\go r$).
\end{definition}

Let   $( I^-,H_0, I^+)$ be  a   diagonal $\go sl_2$-triple (see Definition \ref{def-conditionC}). As above we denote by   $\sigma$ the involution of  $ \widetilde{\go g}$ associated to this triple. Hence we have the decomposition   $\go g=\go h\oplus \go q$ where $\go h={\go z}_{\go g}(I^+)={\go z}_{\go g}(I^-)$.  We also denote by  $\sigma$ the involution of  $ \widetilde{G}={\rm Aut}_0( \widetilde{\go g})$  given by the conjugation by the element 
$$w=e^{{\rm ad} \; I^+}e^{{\rm ad} \; I^-}e^{{\rm ad} \; I^+}=e^{{\rm ad} \; I^-}e^{{\rm ad} \; I^+}e^{{\rm ad} \; I^-}.$$

As  $\sigma(H_0)=-H_0$, the group  $G$ is invariant under the action of  $\sigma$. Let  $G^{\sigma}\subset G$  be the fixed point group under  $\sigma$. The Lie algebra of  $G^\sigma$ is equal to $\go h$. 

Define  now:

\centerline{$H=Z_G(I^+)$}
\vskip 5pt
 
  Then $H=Z_G(I^-)$ and the Lie algebra of  $H$ is $\go h$. Therefore $H$ is an open subgroup of  $G^\sigma$ and $G/H$, which is identified with the (open) orbit of $I^+$, is a symmetric space.

 \medskip

Consider the subalgebra $\go p$ of  $\go g$ defined by 
$$\go p\index{pgothique@$\go p$}={\go z}_{\go g}(\go a^0)\oplus\Big(\oplus_{0\leq i<j\leq k} E_{i,j}(1,-1)\Big).$$

\begin{prop}\label{prop-siP}\hfill

 The subalgebra $\go p$ is a minimal  $\sigma$-split parabolic subalgebra of   $\go g$. Its Langlands decomposition is given by  $\go p=\go{l}+\go{n}=\go m_1\oplus \go a_{\go p}\oplus \go n$ where 
$$\left\{ \begin{array}{l} \go n=\oplus_{0\leq i<j\leq k}E_{i,j}(1,-1)\\
\go{l}=\go m_1\oplus \go a_{\go p}= {\go z}_{\go g}(\go a^0)\\
\go a_{\go p}=\{H\in \go a; (\lambda\in \Sigma, \lambda(\go a^0)=0)\Longrightarrow \lambda(H)=0\}\\
{\go m}_1\;\textrm{is the orthogonal of  }\; {\go a}_{\go p}\;{\rm in}\; {\go z}_{\go g}(\go a^0)={\go m}_{1}\oplus \go a_{\go p}\end{array}\right..$$
\end{prop}
\begin{proof} The Theorem \ref{th-decomp-Eij} and the  Proposition  \ref{proplambda>0}   imply that  $\go p$ contains all  root spaces corresponding to the negative roots in $\Sigma$. Hence  $\go p$ contains a minimal parabolic subalgebra of $\go g$. Therefore  $\go p$ is a parabolic subalgebra of  $\go g$.\medskip

Let  $\Gamma$ be the set of roots $\lambda\in\Sigma$ such that  $\go g^\lambda\subset \go p$. From the definition of  $\go p$, one has 
$$\Gamma=\Sigma^-\cup\{\alpha\in \Sigma^+; \lambda(\go a^0)=0\} \; {\rm and }\; \go p={\go z}_{\go g}(\go a)\oplus\big(\oplus_{\lambda\in \Gamma} \go g^\lambda\big).$$
It follows that $\Gamma\cap -\Gamma=\{\lambda\in \Sigma; \lambda(\go a^0)=0\}$. And then  $\go n= \oplus_{\lambda\in \Gamma\setminus (\Gamma\cap -\Gamma)} \go g^\lambda= \oplus_{0\leq i<j\leq k}E_{i,j}(1,-1)$ is the nilradical of $\go p$.  If $H\in \go a^0$, then $\sigma(H)=-H$. Therefore if $X\in E_{i,j}(1,-1)$ then  $\sigma(X)\in E_{i,j}(-1,1)$. Then  $\go l=\sigma(\go p)\cap \go p={\go z}_{\go g}(\go a^0)$  is  a $\sigma$-stable Levi component of  $\go p$ and   $\go p$ is a  $\sigma$-split parabolic subalgebra. 
  Let $\go a_p$ be the maximal split abelian subalgebra of the center of $\go l$. Then $\go a_p=\cap_{\lambda\in \Gamma\cap -\Gamma} \ker(\lambda)$. Hence the Langlands decomposition is given by $\go p=\go m_1\oplus\go a_p\oplus \go n$ where  $\go m_1$ is the orthogonal of  $\go a_p$ in  $\go l$ for the Killing form.

As $\go a^0$ is a Cartan subspace of  $\go q$ and  as $\sigma(\go p)\cap \go p={\go z}_{\go g}(\go a^0)$, Proposition  4.7  (iv) of  \cite{HW} (see also Proposition 1.13 of  \cite{HH}) implies that $\go p$ is a minimal $\sigma$-split parabolic subalgebra of $\go{g}$.

 \end{proof}

Let  $N \index{N@$N$}=\exp^{{\rm ad}\;\go n}\subset G$ and   $L=Z_G(\go a^0)$. Then  $P\index{P@$P$}=LN$ is a parabolic subgroup of $G$. From the above discussion  $P$ is in fact a  minimal $\sigma$-split parabolic subgroup of  $G$ with  $\sigma$-stable Levi component $L=P\cap \sigma(P)$  and with nilradical  $N$. 
\medskip

 We denote by  $A$ and  $A^0$ the split tori of $G$ whose Lie algebras are respectively $\go a$ and  $\go a^0$. Then  $A$ is a maximal split torus of  $G$ and the set of weights of  $A$ in  $\go g$ is a root system  $\Phi(G,A)$ isomorphic to  $\Sigma$.  More precisely, any root in $ \Sigma$ is the differential of a unique root in  $\Phi(G,A)$. In the sequel of the paper we will always identify these two root systems. For  $\lambda\in\Sigma$ and   $a\in A$, we will denote by $ a^\lambda$ the eigenvalue of the action of  $a$ on  $\go g^\lambda$. 

 \vskip 20pt 
 
 \subsection{ The prehomogeneous vector space  $(P,V^+)$}\hfill
 
 \vskip 5pt
 For the convenience of the reader, and although it  will be a consequence of the proof of Theorem \ref{th-delta_{j}invariants}, let us give a simple proof of the prehomogeneity of $(P,V^+)$.
 
 \begin{prop}\label{(P,V)PH}\hfill
 
 The representation $(P,V^+)$ is prehomogeneous.
  \end{prop}
  
  \begin{proof} Remember that prehomegeneity is an infinitesimal condition. Therefore it is enough to prove that $(\overline{\go{p}}, \overline{V}^+)$ is prehomogeneous. But the parabolic subalgebra $\overline{\go{p}}$ of $\overline{\go{g}}$ will contain a Borel subalgebra $\overline{\go{b}}$.  The space $(\overline{\go{b}}, \overline{V}^+)$ is   prehomogeneous by \cite{M-R-S}, Prop. 3.8 p. 112 (the proof is the same over $\overline{F}$ as over $\C$).
  
  \end{proof}

We will now  show that the polynomial  $\Delta_j$ (see  Definition \ref{defdeltaj}) are the fundamental relative invariants of this prehomogeneous  vector space .\medskip


  Recall  $\ell$ is the common dimension of the  $ \widetilde{\go g}^{\lambda_j}$'s and that $\ell$ is either a square  or equal to  $3$ (cf. Theorem \ref{th-k=0}). 
  
  \medskip

   \begin{theorem}\label{th-delta_{j}invariants}\hfill
   
   The polynomials  $\Delta_j$ are irreducible, and relatively invariant under the action of  $P$: there exists a rational character ${\chi}_j$ of  $P$ such that 
  $$\Delta_j(p.X)= {\chi}_j(p)\Delta_j(X),\quad X\in V^+, p\in P$$
  
  More precisely: 
\begin{itemize}  \item $ {\chi}_j(n)=1,\quad n\in N,$
\item $ {\chi}_j(a)=a^{\kappa(\lambda_j+\ldots +\lambda_k)}, \quad a\in A,$
\item $ {\chi}_j(m)=1,\quad m\in L\cap H.$ \hskip 6pt $(H=Z_G(I^+))$
\end{itemize}
  \end{theorem}
   
 \begin{proof}  The fact that the $\Delta_{j}$'s are  irreducible, and their  invariance   under $N$, have already been obtained in  Theorem  \ref{thproprideltaj} and its proof.
 
 We have   $A\subset G_j$ and  $a\in A$ acts by  $a^{\lambda_j}$ on $ \widetilde{\go g}^{\lambda_j}$. Therefore by   Theorem \ref{thproprideltaj}  (3)   we get  $$\Delta_j(a(X_0+\ldots +X_k))=\Delta_j(\sum_{s=0}^k a^{\lambda_s} X_s)=\prod_{s=j}^k a^{\kappa\lambda_s} \Delta_j(X_0+\ldots +X_k),$$  
  for $X_s\in\tilde{\go g}^{\lambda_s}\setminus\{0\}$. (This gives the value of $\chi_j$ on $A$).
 
Let us show now that  $\Delta_j$ is relatively invariant under  $L$ (this is  not given by Theorem   \ref{thproprideltaj}  because  $L$   is    in general not  a subgroup of $G_j$).
  
 Let  $Z=X+Y\in V^+$ where  $X\in V_j^+$ and  $Y\in V_j^\perp$ ($V_j^\perp$ was defined at the beginning of section \ref{sec:sectiondeltaj}). By Theorem   \ref{thproprideltaj}   (4), there exists a constant $c$ such that 
 $ \Delta_j(X+Y)=\Delta_j(X)= c \Delta_0(X_0+\ldots +X_{j-1}+X)$. 
 An element of  $L$ normalizes  $V_j^+$, $V_j^\perp\otimes \overline{F}$ and each root space $ \widetilde{\go g}^{\lambda_j}$. As $\Delta_0$ is relatively invariant under $\overline{G}$ and  $L\subset\overline{G}$, we get for $m\in L$:
 $$\Delta_j(m(X+Y))=c\Delta_0(X_0+\ldots +X_{j-1}+mX)=c\chi_0(m) \Delta_0(m^{-1}X_0+\ldots +m^{-1}X_{j-1}+ X).$$ 
Again  by Theorem  \ref{thproprideltaj}   (4), there exists a constant   $ c_j(m)$ such that  $ \Delta_0(m^{-1}X_0+\ldots +m^{-1}X_{j-1}+ X)=c_j(m) \Delta_j(X)=c_j(m) \Delta_j( X+Y)$. Therefore   $\Delta_j(m(X+Y))=c\,c_j(m)\chi_0(m)  \Delta_j(X+Y)$, and hence  $\Delta_j$ is relatively invariant under $L$.

This proves that the $\Delta_{j}$'s are relatively invariant under the parabolic subgroup $P$.

Let $m\in L\cap H$. Then $\Delta_{j}(mI^+)=\Delta_{j}(I^+)=\chi_{j}(m)\Delta_{j}(I^+)$.
Hence $\chi_{j}(m)=1$.

 \end{proof}

 We define the dense open subset of  $V^+$  as follows: 
 $${\mathcal O}^+\index{O+@${\mathcal O}^+$}:=\{ X\in V^+; \Delta_0(X)\Delta_1(X)\ldots \Delta_k(X)\neq 0\}.$$
 We will now prove that  ${\mathcal O}^+$ is the union of the open $P$-orbits of $V^+$. 
 
  \begin{lemme}\label{lem-NO+} Any element of  ${\mathcal O}^+$ is  conjugated under $N$ to an element of  $\oplus_{j=0}^k ( \widetilde{\go g}^{\lambda_j}\setminus \{0\})$.
\end{lemme}
\begin{proof}
  Let  $X\in {\mathcal O}^+$. As  $\Delta_1(X)\neq 0$, we know from the proof of Proposition \ref{prop.prem-reduc}, that there exists  $Z\in \go g$ such that   $[H_{\lambda_1}+\ldots H_{\lambda_k}, Z]=-Z$ (and hence  $[H_{\lambda_0}, Z]=Z$,) and  ${\rm e}^{{\rm ad}\; Z} X\in V_1^+\oplus  \widetilde{\go g}^{\lambda_0}$. 
  
Therefore    $Z\in \oplus_{j=1}^k E_{0,j}(1,-1)\subset \go n$. Let $X^1\in V_1^+$ and $X^0\in  \widetilde{\go g}^{\lambda_0}$ such that  ${\rm e}^{{\rm ad}\; Z} X=  X^0+X^1$. Then, for $j\geq 1$, 
  $\Delta_j(X)=\Delta_j(X^0+X^1 )=\Delta_j(X^1)$ as  $X^0\in V_1^\perp$. 
  
 As  $X\in {\mathcal O}^+$, we have $\Delta_j(X^1)\neq 0$ for  $j\geq 1$. Then, by induction,  we  obtain that $X$ is $N$-conjugated to an element  of $\oplus_{j=0}^k ( \widetilde{\go g}^{\lambda_j})$. And as $X$ is generic for the $G$  action, we see that in fact  $X$ in $N$-conjugated to an element of  $\oplus_{j=0}^k ( \widetilde{\go g}^{\lambda_j}\setminus \{0\})$.
 
 \end{proof}
  
  \noindent\begin{rem}\label {rem-Vj-V(j+1)} Applying this Lemma to   $ \widetilde{\go g}_j$,  we see that if  $X\in V_j^+$ such  $\Delta_s(X)\neq 0$ for  $s\geq j$,  then  $X$ is $N$-conjugated to an element of the form $Y_j+X^{j+1} $ with $X^{j+1}\in V_{j+1}^+$ and  $Y_j\in  \widetilde{\go g}^{\lambda_j}$. 
  \end{rem}\medskip
  
   \noindent\begin{rem}\label {rem-Delta-normalisation} {\bf (Normalization)}
  
 Suppose $\ell =3$. Recall that  $L_j$ is the analogue  of the group  $G$ for the graded Lie algebra  $ \widetilde{\go l}_j$, that is   $L_j={\mathcal Z}_{{\rm Aut}_0(  \widetilde{\go l}_j)}(H_{\lambda_j})$. In the following Theorem we denote by $\delta_{j}$ a choice of a relative invariant of the prehomogeneous space $(L_{j},  \widetilde{\go g}^{\lambda_j})$. And then we normalize the $\Delta_{j}$'s in such a way that
  $$\Delta_{j}(X_{j}+X_{j+1}+\ldots+X_{k})=\delta_{j}(X_{j})\Delta_{j+1}(X_{j+1}+\ldots+X_{k})=\delta_{j}(X_{j})\ldots\delta_{k}(X_{k}) $$
  for  $X_{j}\in   \widetilde{\go g}^{\lambda_j}\setminus \{0\}$ and $ \text{ for }j=0,\ldots,k$. Conversely one could  also choose arbitrarily the $\Delta_{j}$'s, and this choice defines uniquely the $\delta_{j}$'s, according to the above formulas.  
  
  \end{rem}

 \begin{theorem}\label{th-Porbites} 
 \hfill
  
 In all cases, the dense open set ${\mathcal O}^+$ is the union of the open $P$-orbits in $V^+$.

 \begin{enumerate}\item  If $\tilde{\go g}$ is of Type $I$ (that is, if   $\ell$ is a square and $e=0$ or $4$),  then ${\mathcal O}^+$ is the unique open  $P$-orbit in  $V^+$.
 \item   Let $\tilde{\go g}$ be of  of Type  $II$ (that is  $\ell=1$ and $e\in\{1,2,3\}$)   and  let $S_e=\{a\in F^*; aq_e\sim q_e\}$ (see Lemma \ref{lem-ensembleSe}). Then the  subgroup   $P$ has  $|F^*/S_e|^{k}$ open orbits in   $V^+$ given for  $k\geq 1$ by
 \begin{align*} {\mathcal O}_u^+\index{O-u@${\mathcal O}_u^+$}&=\{ X\in V^+;\dfrac{ \Delta_j(X)}{\Delta_k(X)^{k+1-j}} u_j\ldots u_{k-1}\in S_e \;{\rm for }\; j=0,\ldots k-1\}\\
  {}&= P.(u_{0}X_{0}+\ldots+u_{k-1}X_{k-1}+ X_{k})
  \end{align*}
  
where   $u=(u_0,\ldots, u_{k-1})\in (F^*/S_e)^{k}$. (i.e. $P$ has $4^k $ open orbits in $V^+$ if $e=1$ or $3$, and $2^k$ open orbits if $e=2$).
 
\item  If $\tilde{\go g}$ is of Type $III$ (that is if  $\ell=3$), then the subgroup  $P$ has  $3^{k+1}$ open orbits in $V^+$ given by

  $$ {\mathcal O}_u^+=\{ X\in V^+; \Delta_j(X) u_j\ldots u_k\in F^{*2} \;{\rm for  }\; j=0,\ldots k\},$$
 where $u=(u_0,\ldots, u_k)\in \prod_{i=0}^k\Big(\delta_i( \widetilde{\go g}^{\lambda_i}\setminus\{0\})/F^{*2}\Big).$
  \end{enumerate}

  \end{theorem}
  \begin{proof} As the  $\Delta_j$'s are relatively invariant under  $P$, the  union of the open  $P$-orbits is a subset of  ${\mathcal O}^+$.

  \noindent{  $(1)$} Suppose first that  $\ell$ is a square   and $e=0$ or $4$ (in other words $\tilde{\go{g}}$ is of Type I). Let $X\in {\mathcal O}^+$.  By Lemma  \ref{lem-NO+}, $X$ is  $N$-conjugated  to an element  $\sum_{j=0}^k Z_j$ with $Z_j\in  \widetilde{\go g}^{\lambda_j}\setminus  \{0\}$.  This element is of course generic for $G$.\medskip 
   
  From Theorem  \ref{thm-delta2} and  \ref{thm-orbites-e04}, two generic elements of the"diagonal" $\oplus_{j=0}^k\tilde{\go g}^{\lambda_j}$ are  $L$-conjugated. Hence all the  elements of  $\mathcal O^+$are $P$-conjugated.
    \medskip

      \noindent{  $(2)$}  We suppose now that  $\ell=1$ and  $e=1$, $2$ or $3$. 
   
    Let  $k\geq 1$.  For $j\in\{0,\ldots, ,k\}$ we fix a non zero element  $X_j$ of $ \tilde{\go g}^{\lambda_j}$ such that for  $I^+=X_0+\ldots +X_k$ one has  $\Delta_j(I^+)=1$ for all  $j$.

   Let  $Z\in {\mathcal O}^+$. As before, $Z$ is $N$-conjugated  to an element   $X\in \oplus_{j=0}^k({\widetilde{\go g}^{\lambda_j}\setminus  \{0\}})$. As $\ell=1$, we can write $X=\sum_{j=0}^k x_jX_j$ with $x_j\neq 0$.

 By Theorem \ref{thproprideltaj}, the polynomials  $\Delta_j$ are $N$-invariant. If we set  $u_s=\dfrac{x_s}{x_k}$ modulo $S_e$ for  $s=0,\ldots , k-1$, we get 
$$\frac{\Delta_j(Z)}{\Delta_k(Z)^{k+1-j}}=\frac{\Delta_j(X)}{\Delta_k(X)^{k+1-j}}= \prod_{s=j}^{k-1} u_j\;{\rm modulo }\; S_e,$$
and this implies  $Z\in \mathcal O_u^+$.

 Conversely, let $u=(u_0,\ldots, u_{k-1})\in (F^*/S_e)^k$ and let   $Z$ and  $Z'$ be two elements of  $ \mathcal O_u^+$. These elements are  respectively $N$-conjugated to diagonal elements    $X=\sum_{j=0}^k x_jX_j$  and $X'=\sum_{j=0}^k x'_jX_j$.  From the definition of  $ \mathcal O_u^+$, we have  $\dfrac{x_j}{x_k}\ldots \dfrac{x_{k-1}}{x_k}=\dfrac{x'_j}{x'_k}\ldots \dfrac{x'_{k-1}}{x'_k}$ modulo $S_e$ for all $j\in\{0,\ldots k-1\}$. Therefore   $\dfrac{x_j}{x_k} =\dfrac{x'_j}{x'_k} $ modulo $S_e$ for all  $j\in\{0,\ldots k-1\}$. This implies that $\frac{1}{x_{k}x'_{k}}x_{i}x'_{i}\in S_e$  for all $i\in \{0,1,\ldots,k\}$ (because $F^{*2}\subset S_e$ by Lemma \ref{chiG}). By Corollary \ref{cor-Lorbits}, the elements $ X$ and $X'$  are  $L$-conjugated. 
 It follows that two elements in $ \mathcal O_u^+ $ are $P$-conjugated.\medskip

 If $u$ and  $v$ are two elements in  $(F^*/S_e)^{k-1}$ such that  $ \mathcal O_u^+\cap  \mathcal O_v^+\neq \emptyset$, then  $u_j\ldots u_{k-1}=v_j\ldots v_{k-1}$ modulo $S_e$ for  $j=0,\ldots k-1$. Therefore  $u_j=v_j$ modulo $S_e$ for all $j$ and hence  $ \mathcal O_u^+= \mathcal O_v^+$.
  The statement $(2)$ is now proved.
\vskip0,5cm

  \noindent{$ (3)$} Consider finally the case  $\ell=3$.           Remember  (see Theorem \ref{th-k=0}, 2)) that in this case,     $\delta_j$ is a quadratic form which represents three classes modulo $F^{*2}$ (all classes  in  $(F^*/{F^*)^2}$ distinct from  $-disc(\delta_j)$). Let  $X\in V^+$ be a generic element of    $(P,V^+)$. We will show that $X$  belongs to   $  \mathcal O_u^+$ for some  $u=(u_0,\ldots, u_k)\in \prod_{i=0}^k\big(\delta_i( \widetilde{\go g}^{\lambda_i}\setminus\{0\})/F^{*2}\big) $.\medskip

   As before, the element  $X$ is  $N$-conjugated to an element  $Z=\sum_{j=0}^k Z_j$ where $Z_j\in \widetilde{\go g}^{\lambda_j}\setminus \{0\}$.

 From the normalization made in Remark \ref{rem-Delta-normalisation},  we have $\Delta_j(Z)=\prod_{s=j}^k \delta_s(Z_s)$ where  $\delta_s$ is a  fundamental relative invariant of  $(L_s, \widetilde{\go g}^{\lambda_{s}})$. 
  
  If we define  $u_s=\delta_s(Z_s)\;{\rm modulo}\; F^{*2}$, we get  $ \Delta_j(X)u_j\ldots u_k=\Delta_j(Z)u_j\ldots u_k=\prod_{s=j}^k \delta_s(Z_s) u_s\in F^{*2}$, and therefore    $X$  belong to $ \mathcal O_u^+$  with  $u=(u_0,\ldots, u_k)$.\medskip

Conversely,  let $X,X'\in  \mathcal O_u^+$  These elements are $N$-conjugated to (respectively) two ``diagonal'' elements $Z=Z_0+\ldots +Z_k$ and  $Z'=Z'_0+\ldots +Z'_k$ (Lemma \ref{lem-NO+}). From the definitions we have  $\delta_j(Z_j)=\delta_j(Z'_j)$ modulo $F^{*2}$ for $j=0,\ldots,k$. By  Corollary \ref{cor-conjLj}, there exists  $l_j\in L^0_j$ such that  $l_jZ_j=Z'_j$ (Recall that $L^0_i$ is the subgroup of $L_i$ defined in Definition \ref{def-G0}). As $L_j^0$ centralizes   $\oplus_{s\neq j}  \widetilde{\go g}^{\lambda_s}$, we get  $l_0\ldots l_k. Z=Z'$ . Moreover   $l_0\ldots l_k\in P$. Hence two elements in $ \mathcal O_u^+$ are $P$-conjugated. \medskip

  If $u$ and  $v$ are two elements of $\big(\Delta_k( \widetilde{\go g}^{\lambda_k}\setminus\{0\})/F^{*2}\big)^{k+1}$ such that  $ \mathcal O_u^+\cap  \mathcal O_v^+\neq \emptyset$ then $u_j\ldots u_k=v_j\ldots v_k\;{\rm modulo}\; F^{*2}$ for  $j=0,\ldots k$ . And hence $v_j=u_j$ modulo $F^{*2}$ for all $j$, 	and therefore $ \mathcal O_u^+=  \mathcal O_v^+$.  
  
  Assertion $(3)$ is proved.

  The fact that ${\cal O}^+$   is the union of the open $P$-orbits is now clear.

  \end{proof}
   \subsection{ The involution $\gamma$}\index{gamma@$\gamma$}
 \hfill

From  Remark \ref{rem-decomp-racines}  we know that the root system of   $( \widetilde{\go g},\go a^0)$ is always of type $C_{k+1}$ and consists of the linear forms  $\pm \eta_j\pm \eta_i$ for  $i\neq j$ and $\pm 2\eta_j$, $1\leq i,j\leq k$ where  

$$\eta_j(H_{\lambda_i})=\delta_{ij}.$$

 We know also (\cite{Bou1}) that then, there exists an element $w$ of the Weyl group of $C_{k+1}$ such that 
$$(*)\qquad w.\eta_i=-\eta_{k-i}\;{\rm for  }\; i=0,\ldots,k.$$
As this Weyl group is isomorphic to  $N_{ \widetilde{G}}(\go a^0)/Z_{ \widetilde{G}}(\go a^0)$, there exists an element  $\gamma\in N_{ \widetilde{G}}(\go a^0)$ such that  $w={\rm Ad}(\gamma)_{|_{\go{a}^0}}$. The property $(*)$ implies that  $\gamma$ normalizes  $\go g$, exchanges  $V^+$ and $V^-$ 	and normalizes also  $P$ (ie. $\gamma P\gamma^{-1}=P$). \medskip

In the Theorem below, we will give explicitly such an element  $\gamma$,  which moreover, will be an involution of  $ \widetilde{\go g}$.
\vskip 10pt

We choose   a diagonal   $\go sl_2$-triple $(I^-,H_0, I^+)$   that is such that  $I^+=X_0+\ldots +X_k$ ($X_{j}\in \widetilde{\go{g}}^{\lambda_{j}}\setminus\{0\}$), $I^-=Y_0+\ldots +Y_k$ ($Y_{j}\in \widetilde{\go{g}}^{-\lambda_{j}}\setminus\{0\}$), where each $(Y_{j}, H_{\lambda_{j}}, X_{j})$ is an  $\go sl_2$-triple. 

 Let  $(Y,H_{\lambda_i}-H_{\lambda_j}, X)$ be an  $\go sl_2$-triple such that  $X\in E_{i,j}(1,-1)$ and  $Y\in E_{i,j}(-1,1)$ (see Lemma \ref{lem-sl2ij}). Remember  the elements  $$\gamma_{i,j}=e^{\ad X}e^{\ad Y}e^{\ad X}=e^{\ad Y}e^{\ad X}e^{\ad Y}.$$ which have been introduced in Proposition \ref{prop-gammaij}. We suppose moreover that the sequence $(X_{i})_i$ is such that $\gamma_{i,k-i}(X_{i})=X_{k-i}$, for $0\leq i\leq n$, where $n$ is  the integer defined by  $k=2n+2$ if $k$ is even and  $k=2n+1$ if $k$ is odd. It is always possible to choose such a sequence.

Once we have chosen such an $\go sl_2$-triple, we normalize the polynomials $\Delta_{j}$ by the condition:
 $$\Delta_j(I^+)=1,\quad {\rm for  }\; j=0,\ldots, k.$$

\begin{theorem}\label{th-involution-gamma} \hfill

Suppose that $(I^-,H_0, I^+)$ is a diagonal  $\go sl_2$-triple satisfying the preceding conditions. 

There exists an element  $\gamma\in N_{ \widetilde{G}}(\go a^0)$such that 
\begin{enumerate}\item $  \gamma.H_{\lambda_j}=-H_{\lambda_{k-j}}\;{\rm for  }\; j=0,\ldots, k$;
\item $\gamma.X_j=Y_{k-j}\;{\rm for}\; j=0,\ldots, k$;
\item $\gamma^2={\rm Id}_{ \widetilde{\go g}}.$
\end{enumerate}
Such an element normalizes $\go g$, exchanges  $V^+$ and $V^-$ and normalizes $G$, $P$, $L$, $A_0$ and  $N$.

\end{theorem}
 \begin{proof} We will first show the existence  of an involution $ \widetilde{\gamma}$ of  $ \widetilde{G}$ such that  $ \widetilde{\gamma}(H_{\lambda_j})=H_{\lambda_{k-j}}$ and  $ \widetilde{\gamma}(X_j)=X_{k-j}$ for $j=0,\ldots, k$. For $k=0$, then the trivial involution satisfies this property. \medskip

We suppose that $k>0$. Let  $w_i$ be the non trivial element of the  Weyl group   associated to the   ${\go sl}_2$-triple $(Y_i ,  H_{\lambda_i}, X_i)$.
Recall  (Proposition \ref{prop-gammaij} and Proposition \ref{prop-gamijtilde})   that    the elements  $\widetilde{\gamma}_{i,j}=\gamma_{i,j}\circ w_i^2\in N_{\widetilde{G}}(\go a^0)$ satisfy the following properties:
$$\widetilde{\gamma}_{i,j}^2={\rm Id}_{\widetilde{\go g}}\quad {\rm and }\quad \widetilde{\gamma}_{i,j}(H_{\lambda_s})=\left\{\begin{array}{ccc} H_{\lambda_i} & {\rm for } & s=j\\
H_{\lambda_j} & {\rm for } & s=i\\
H_{\lambda_s} & {\rm for } & s\notin\{i,j\}\end{array}\right.$$

Consider again the integer   $n$   defined by  $k=2n+2$ if $k$ is even and  $k=2n+1$ if $k$ is odd and set:
$$\widetilde{\gamma}:=\widetilde{\gamma_{0,k}}\circ \widetilde{\gamma_{1,k-1}}\circ\ldots \circ \widetilde{\gamma_{n, k-n}}.$$

Then $$ \widetilde{\gamma}(H_{\lambda_j})=H_{\lambda_{k-j}}\;{\rm for }\; j=0,\ldots, k$$

As the pairs of roots $(\lambda_{i},\lambda_{k-i})$  are mutually stronly orthogonal, the involutions $\widetilde{\gamma_{i,k-i}}$ (see  Proposition \ref{prop-gamijtilde}) commute, and hence  $\widetilde{\gamma}$ is an involution of $\widetilde{\go g}$.\medskip

From our choice of the sequence $X_{j}$, and as the   action of $w_i^2$ on  $\oplus_{s=0}^k\widetilde{\go g}^{\lambda_s}$ is trivial, we obtain that   $\widetilde{\gamma} (X_j)= X_{k-j}$.\medskip

The involution  $\widetilde{\gamma}$ centralizes $I^+$ and $H_{0}$, and hence it centralizes  $I^-$. Therefore  $\widetilde{\gamma}$ commutes with   $w= e^{{\rm ad}\; I^+}e^{{\rm ad}\; I^-}e^{{\rm ad}\; I^+}$ which is the element of  $\widetilde{G}$ which defines the involution $\sigma$ associated to the  ${\go sl}_2$-triple $(I^-, H_0, I^+)$. Define 
$$\gamma=\widetilde{\gamma}w=w\widetilde{\gamma}.$$
The automorphism  $\gamma$ commutes with  $w$, and hence  $\gamma$ is an involution of  $\widetilde{\go g}$. Moreover, using Theorem  \ref{th-invol}, we get 
$$\begin{array}{lll} \gamma(H_{\lambda_j})=\sigma(H_{\lambda_{k-j}})=-H_{\lambda_{k-j}}& {\rm for }& j=0,\ldots, k\\
\gamma(X_j)=\sigma(X_{k-j})=Y_{k-j} & {\rm for }& j=0,\ldots, k\end{array}$$
This implies that  $\gamma(H_0)=-H_0$ and hence  $\gamma$ stabilizes  $\go g$, normalizes  $G$ and exchanges $V^+$ and  $V^-$.\medskip

As  $ \gamma(H_{\lambda_j})=-H_{\lambda_{k-j}}$ for  $ j=0,\ldots, k$, the element $\gamma$ stabilizes $\go a^0$ and exchanges  $E_{i,j}(1,-1)$ and  $E_{k-i, k-j}(-1,1)$. Therefore $\gamma$ stabilizes $\go n$ and  $\go{l}={\go z}_{\go g}(\go a^0)$, and hence it stabilizes $\go p$. It follows that $\gamma$ normalizes  $A^0$,  $L$, $N$ and $P$.

 \end{proof}
\subsection{ The  open $P$-orbits in  $V^-$ and the polynomials  $\nabla_j$ \index{nablaj@$\nabla_j$}}\label{sec:PorbitesV-}\hfill

\vskip 10pt

In this section we fix an $\go{sl}_{2}$-triple $(I^-,H_{0},I^+)$ satisfying the same  conditions as for  Theorem \ref{th-involution-gamma} where the involution $\gamma$ is defined. We set $H=Z_{G}(I^+)$

\begin{definition}\label{def-nabla} For  $j=0,\ldots, k$, we denote by  $\nabla_j$ the polynomial on  $V^-$ defined by 
$$\nabla_j(Y)=\Delta_j(\gamma(Y)),\quad {\rm for }\; Y\in V^-.$$
\end{definition}

 \begin{theorem}\label{th-nabla}\hfill
 
  The polynomials  $\nabla_j$ are irreducible of degree $\kappa(k+1-j)$.
 \begin{enumerate}\item $\nabla_0$ is the  relative invariant of  $V^-$ under the action of $G$.
 \item For  $j=0,\ldots ,k$, the polynomial  $\nabla_j$ is a relatively invariant polynomial on $V^-$ under the action of the parabolic subgroup  $P$. More precisely we have 
 $$\nabla_j(p.Y)=\chi_j^-(p) \nabla_j(Y),\quad{\rm for }\; p\in P,$$
 where  $\chi_j^-$ \index{khij-@$\chi_j^-$} is a character of  $P$ with the following properties:
 
 $\bullet \chi_j^-(n) =1, \quad{\rm for }\; n\in N, $
 
 $\bullet \chi^-_j(a)=a^{-\kappa(\lambda_0+\ldots +\lambda_{k-j})},  \quad{\rm for }\; a\in A,$
 
 $\bullet \chi_j^-(l)=1,{\rm for }\; l\in L\cap H.$
 
 \end{enumerate}
 \end{theorem}
 \begin{proof}  As $\gamma$ is linear, $\nabla_{j}$ is effectively an irreducible polynomial of the same degree as $\Delta_{j}$, that is  $\kappa(k+1-j)$ ($\nabla_{j}$ is  non zero  since  $\gamma(I^-)=I^+$). As  $\gamma$ normalizes  $G$ and $P$, the fact that the  $\nabla_j$'s are relatively invariant  is  direct consequence of the same property for the  $\Delta_j$'s (Theorem \ref{thproprideltaj}).\medskip

  For $p\in P$ we have:
 $$ \nabla_j(p.I^-)=\chi_j^-(p)\nabla_{j}(I^-) =\Delta_j(\gamma p\gamma^{-1}\gamma I^+)=\chi_j(\gamma p\gamma^{-1})\nabla_{j}(I^-),$$
 and therefore
 $$\chi_j^-(p) =\chi_j(\gamma p\gamma^{-1}).$$
As  $\gamma$ normalizes $N$, $L$ and $A$ and commutes with  $\sigma$, the assertion concerning the values of $\chi_{j}^-$ on $N$, $A$, and $L\cap H$ is a consequence of the same properties for the  $\chi_j$'s (cf. Theorem  \ref{th-delta_{j}invariants}).

 \end{proof}
 \vskip 5pt
 Let $\mathcal O^-$ be the dense  open subset of  $V^-$ defined by 
 $${\mathcal O}^-\index{Ocal-@${\mathcal O}^-$}=\{ Y\in V^-; \nabla_0(Y)\nabla_1(Y)\ldots \nabla_k(Y)\neq 0\}.$$
 
  Suppose $\ell =3$. Recall that  $L_j$ is the analogue  of the group  $G$ for the graded Lie algebra  $ \widetilde{\go l}_j$, that is   $L_j={\mathcal Z}_{{\rm Aut}_0(  \widetilde{\go l}_j)}(H_{\lambda_j})$. In the following Theorem we denote by $\delta_{j}^-$ the  relative invariant of the prehomogeneous space $(L_{j},  \widetilde{\go g}^{-\lambda_j})$ defined by the identity
  
    $$\nabla_{j}(Y_{j}+Y_{j+1}+\ldots+Y_{k})=\delta_{j}^-(Y_{j})\nabla_{j+1}(Y_{j+1}+\ldots+Y_{k})=\delta_{j}^-(Y_{j})\ldots\delta_{k}^-(Y_{k}) $$
  for  $Y_{j}\in   \widetilde{\go g}^{\lambda_j}\setminus \{0\}$ and $ \text{ for }j=0,\ldots,k$. 
   
Using the involution  $\gamma$, the following description of the open $P$-orbits in $V^-$ is an easy consequence of Theorem \ref{th-Porbites}.

 \begin{theorem}\label{th-PorbitesV-} \hfill
 
 The dense open subset  ${\mathcal O}^-$ is the union of the open   $P$-orbits  in $V^-$.
  \begin{enumerate}\item  If $\tilde{\go g}$ is of Type $I$ (that is if $\ell$ is a square and $e\in \{0,4\}$),  then ${\mathcal O}^-$ is the unique open  $P$-orbit in  $V^-$.
 \item  Let $\tilde{\go g}$ be  of Type $II$ (that is if  $\ell=1$ and $e\in\{1,2,3\}$) and let  $S_e=\{a\in F^*; aq_e\sim q_e\}$ as in Lemma \ref{lem-ensembleSe}. Then the  subgroup   $P$ has  $|F^*/S_e|^{k}$ open orbits in   $V^-$ given for  $k\geq 1$ by
  $${\mathcal O}_u^-=\{ Y\in V^-;\dfrac{ \nabla_j(Y)}{\nabla_k(Y)^{k+1-j}} u_j\ldots u_{k-1}\in S_e, \;{\rm for }\; j=0,\ldots k-1\},$$
where  $u=(u_0,\ldots, u_{k-1})\in (F^*/S_e)^{k}$. (i.e. $P$ has $4^k $ open orbits in $V^-$ if $e=1$ or $3$, and $2^k$ open orbits if $e=2$)

\item If $\tilde{\go g}$ is of Type $III$ (that is if  $\ell=3$),  then   $P$ has  $3^{k+1}$open orbits in $V^-$ given by

  $${\mathcal O}^-_u=\{ Y\in V^-; \nabla_j(Y) u_j\ldots u_k\in F^{*2} \;{\rm for  }\; j=0,\ldots k\},$$
 where $u=(u_0,\ldots, u_k)\in \prod_{i=0}^k\big(\delta_i^-( \widetilde{\go g}^{-\lambda_i}\setminus\{0\})/F^{*2}\big).$

  \end{enumerate}

 \end{theorem}
  \vskip 15pt 
  The following Lemma gives the relationship between the characters of the $\nabla_{j}$'s and those of the $\Delta_{j}$'s.
  \begin{lemme}\label {lem-propchij} For $g\in G$ and $p\in P$, we have:
  
  $$\chi_0^-(g)=\frac{1}{\chi_0(g)}$$

and  
$$\chi_j^-(p)=\frac{\chi_{k-j+1}(p)}{\chi_0(p)},\hskip 15pt  j\in\{1,\ldots k\}.$$
\end{lemme}
\begin{proof} As the  Killing form induces a $G$-invariant duality between $V^+$ and  $V^-$, we have: 
$${\rm det}_{V^+} {\rm Ad}(g)={\rm det}_{V^-} {\rm Ad}(g^{-1}),\quad g\in G.$$
On the other hand, for $X\in V^+$, let us consider the determinant $P(X)={\rm det}_{(V^-,V^+)}({\rm ad}\; X)^2$ of the map $({\rm ad}\; X)^2: V^-\to V^+$ (for any choice of basis). This polynomial   $P$ is relatively invariant under the action of $G$ because $P(g.X)= {\rm det}_{(V^-,V^+)}(g ({\rm ad}\; X)^2 g^{-1})=({\rm det}_{V^+} {\rm Ad}(g))^2P(X)$. Hence it is, up to a multiplicative constant,  a power of $\Delta_0$. Therefore there exists $\al\in \N$ such that

$$({\rm det}_{V^+} {\rm Ad}(g))^2=\chi_0(g)^{\alpha}.$$
If we take  $g\in G$ such that $g_{|_{V^+}}= t \; Id_{V^+}\in G_{|_{V^+}},\;  t\in F^*$ (cf. Lemma \ref{lem-tId-dansG}), then Theorem \ref{thpropridelta_0} implies that $\alpha=\dfrac{2 {\rm dim} V^+}{\kappa (k+1)}$.

\medskip

The same argument for the dual space $(G,V^-)$ implies that  $({\rm det}_{V^-} {\rm Ad}(g))^2=\chi_0^-(g)^{\alpha}.$ Therefore we get  $\chi_0^-(g)^{\alpha}=\chi_0(g)^{-{\alpha}}$ for all  $g\in G$. As the group $X^*(G)$ of rational characters of $G$ is a lattice (see for example \cite{Re}, p.121), it has no torsion. Therefore $\chi_0^-(g)=\dfrac{1}{\chi_0(g)}$, and the first assertion is proved.
\medskip

Let us show the second assertion. As all the characters we consider are trivial on $N$, it is enough to prove the relation for  $m\in L$.  We consider the subgroup   $L'=L_0\ldots L_k$ of $\bar{L}=L(\overline{F})$ (keep in mind that the groups $L_{j}$ are in general not included in $G$). The polynomials  $\Delta_j$ are the restrictions to  $V^+$ of polynomials defined on  $\overline{V}^+$, which are  relatively invariant under the action of  $\bar{G}$. Therefore  the polynomials  $\nabla_j$ also are  restrictions to  $V^-$ of polynomials of  $\overline{V}^-$ which are relatively invariant under $\bar{G}$.
\medskip

 We  first show that for $m\in L$, there exists $l\in L'$ such that  $l^{-1}m\in \overline L\cap \overline{H}$. \hskip 50pt (*)\medskip

 (Here $\overline L= L(\overline{F})=Z_{\overline G}({\go{a}}^0)$ and $\overline H=H(\overline{F})= Z_{\overline{G}}(I^+)$).

 If  $\ell$ is a square, as $m.X_{j}\in \tilde{\go{g}}^{\lambda_{j}}$, it follows from Theorem  \ref{th-k=0} that  $m.X_j$ is  $L_j$-conjugated to  $X_j$ and hence there exists  $l_j\in L_j$ such that  $m.X_j=l_j.X_j$. Then the element  $l=l_0\ldots l_k\in L'$ is such that  $l^{-1}m.I^+=I^+$ and therefore $l^{-1}m\in  \overline{L}\cap \overline{H}$. 
 
 If $\ell=3$, we will use the  description of  $G$ given in Proposition  \ref{prop-G}. Let  $E=F[\sqrt{\ep}]$ be  a   quadratic extension of  $F$ where $\ep$ is a unit which is not a square, and let $\pi$ be a uniformizer of  $F$.  The group  $G$ is the group of automorphisms of $\tilde{\go g}$ given by conjugation by matrices of the form $\left(\begin{array}{cc} \mathbf g & 0\\ 0 & \mu\; ^{t}\mathbf g^{-1}\end{array}\right)$ where $\mathbf g\in G^0(2(k+1))\cup \sqrt{\ep} G^0(2(k+1))$ and  $\mu\in F^*$. We denote by $[\mathbf g,\mu]$ such an element of  $G$. 
  
The space $\go a^0$ is the space of matrices  $$\left(\begin{array}{cc} {\mathbf H}(t_k,\ldots, t_0) & 0\\ 0 & -{\mathbf H}(t_k,\ldots, t_0)  \end{array}\right),\; \textrm{where  }
\;{\mathbf H}(t_k,\ldots, t_0) =\left(\begin{array}{ccc} t_k I_2 & 0 & 0\\ 0 & \ddots & 0\\ 0 & 0 & t_0 I_2\end{array}\right),$$
and  $(t_k,\ldots , t_0)\in F^{k+1}$.  Therefore the centralizer of  $\go a^0$ in $G$, that is  $L$,   is the subgroup of elements  $m=[\mathbf g, \mu]$ where  ${\mathbf g}=diag(\mathbf g_k,\ldots,\mathbf g_0)$  is a  $2\times 2$ block diagonal matrix  whose diagonal elements  $\mathbf g_j$ belong to  $G^0(2)\cup \sqrt{\ep} G^0(2)$. Let $  l_j\in GL(2(k+1),E)$ be the $2\times 2$ block diagonal matrix whose all diagonal blocks are the identity in  $GL(2,E)$ except the $(k-j+1)$-th bock which is equal to  $\mathbf g_j$ and the $(2(k+1)-j)$-th block  which is equal to  $\mu\;^{t}\mathbf g_j^{-1}$. Then from the definition  of $L_j$, we have ${\rm Ad}l_j\in L_j$ and $m={\rm Ad}(l_k\ldots l_0)$ belongs to  $L'$.
\medskip

Hence we have proved that in all cases, for all  $m\in L$, there exists  $l\in L'$ such that $l^{-1}m\in \overline{L}\cap \overline{H}$. As the characters  $\chi_j$ and  $\chi_j^-$ are trivial on  $\overline{L}\cap \overline{H}$, we have     $\chi_j(m)=\chi_j(l)$ and  $\chi_j^-(m)=\chi_j^-(l)$ for all  $j\in\{0,\ldots, k\}$. It suffices therefore to prove the result for $l\in L'=L_0L_1\ldots L_k\subset\overline{G}$.\medskip

Let  $l\in L'$ and  $j\in\{1,\ldots, k\}$. Consider the decomposition of  $V^+$ into weight spaces under the action of  $ H_{\lambda_j}+\ldots H_{\lambda_k} $:
  $$V^+=V_j^+\oplus U_j^+\oplus W_j^+,$$
  where  $V_j^+$, $U_j^+$ and  $W_j^+$ are the spaces of weight  $2$, $1$ and $0$ under  $ H_{\lambda_j}+\ldots H_{\lambda_k} $, respectively. More precisely:
  $$\begin{array}{ccl} V_j^+ &= &\oplus_{s=j}^k \tilde{\go g}^{\lambda_s}\oplus(\oplus_{j\leq r<s} E_{r,s}(1,1))\\
  {}&{}\\
  U_j^+ & = & \oplus_{r<j\leq s} E_{r,s}(1,1)\\
  {}&{}\\
W_j^+ &= &\oplus_{s=0}^{j-1} \tilde{\go g}^{\lambda_s}\oplus(\oplus_{r<s\leq j-1} E_{r,s}(1,1)).\end{array}$$
\medskip

Similarly we denote by  $V^-=V_j^-\oplus U_j^-\oplus W_j^-$ the decomposition  of $V^-$ into weight spaces of weight  $-2,-1$ and  $0$ under  $ H_{\lambda_j}+\ldots H_{\lambda_k} $.
 As the eigenspace, in $V^+$, for the eigenvalue $r$ of $ H_{\lambda_j}+\ldots H_{\lambda_k} $, is the same as the eigenspace, for the eigenvalue $2-r$ of $ H_{\lambda_1}+\ldots H_{\lambda_{j-1}} $, it is easy to see that 
 $W_j^-=\gamma (V_{k+1-j}^+)$. Therefore  $(\gamma (G_{k+1-j}), W_j^-)$ is a regular irreducible prehomogeneous vector space  whose fundamental relative invariant is the restriction of  $\nabla_{k-j+1}$ to $W_j^-$.

Let us write  $l=l_1l_2$ where  $l_1\in L_j\ldots L_k\subset \overline{G_j}$ and  $l_2\in L_0\ldots L_{j-1}=\gamma (L_{k+1-j}\ldots L_k)\subset \gamma (\overline{G_{k+1-j}})$.

As $l_1$ acts trivially on  $\overline{W_j}^+$, and as $l_{2}$  acts trivially on $\overline{V_{j}^+}$, we have 
$$\chi_0(l_1)=\Delta_0(X_0+\ldots+X_{j-1}+l_1(X_j+\ldots +X_k))$$
$$= \Delta_j(l_1(X_j+\ldots +X_k))= \Delta_j(l(X_j+\ldots +X_k))=\chi_j(l).$$

Define $l'_2:=\gamma l_2\gamma^{-1}\in L_{k+1-j}\ldots L_k\subset \overline{G_{k+1-j}}$.

If we consider the decomposition $V^+=V_{k+1-j}^+\oplus U_{k+1-j}^+\oplus W_{k+1-j}^+$ and if we apply the same argument as before to  $\gamma l\gamma^{-1}$ we get
$\chi_0(l'_2)=\chi_{k+1-j}(\gamma l\gamma^{-1})$, and this  is equivalent to 
$$\chi^-_0(l_2)=\chi_{k+1-j}^-(l).$$
Applying the first assertion of the Lemma, we obtain  $\chi_0(l_2)=\chi_{k+1-j}^-(l)^{-1}$, and hence finally
$$\chi_0(l)=\chi_0(l_1)\chi_0(l_2)=\chi_j(l) \chi_{k+1-j}^-(l)^{-1},$$
and this proves the second assertion.

\end{proof}

Let  $\Omega^+$ and  $\Omega^-$ \index{Omega+-@$\Omega^{\pm}$}be  the set of generic elements in  $V^+$ and  $V^-$, respectively. In other words:
$$\Omega^+=\{x\in V^+, \Delta_0(X)\neq 0\}, \quad \Omega^-=\{x\in V^-, \nabla_0(X)\neq 0\}.$$

\begin{definition} \label{iota}Let $\i  :  \Omega^+\longrightarrow \Omega^-$\index{iota@$\iota$} be the map which sends $X\in \Omega^+$ to the unique element  $Y\in \Omega^-$ such that  $(Y, H_0, X)$ is an ${\go sl}_2$-triple.\end{definition}

If  $(Y,H_0, X)$ is an  ${\go sl}_2$-triple, then for each $g\in G$, $(g.Y,H_0, g.X)$ is again an  ${\go sl}_2$-triple. Therefore the map $\iota$ is    $G$-equivariant.

\begin{prop}\label{prop-delta-iota} For  $X\in \Omega^+$, we have
$$\nabla_0(\i(X))=\frac{1}{\Delta_0(X)},\quad{\rm and  }\quad\nabla_{j}( \i(X))=\frac{\Delta_{k+1-j}(X)}{\Delta_0(X)},\quad j=1,\ldots ,k.$$

\end{prop}

\begin{proof} Fix a diagonal   ${\go sl}_2$-triple $(I^-, H_0, I^+)$ which satisfies the condition of Theorem  \ref{th-involution-gamma}.  Then  $I^+=X_0+\ldots +X_k$ and $I^-=Y_0+\ldots +Y_k$ where  $(Y_i, H_{\lambda_i},X_i)$ are  $\go sl_2$-triples. 

As  $\mathcal O^+$ is open dense in  $\Omega^+$ and as the function we consider here are continuous on  $\Omega^+$, it suffices to prove the result for $ X\in \mathcal O^+$. The proof depends on the Type of  $\tilde{\go g}$. \medskip

\noindent If  $\tilde{\go g}$ is of Type  $I$ (ie. $\ell$ is a square and  $e=0$ or  $4$) then any  $X\in \mathcal O^+$ is  $P$-conjugated  to  $I^+$ (see   Theorem \ref{th-Porbites}). Therefore it suffices to prove that for all $p\in P$, one has 
$$\nabla_j(p.I^-)=\frac{\Delta_{k+1-j}(p.I^+)}{\Delta_0(p.I^+)}.$$

From the normalization of the polynomials  $\nabla_j$ and  $\Delta_j$, the Lemma \ref{lem-propchij} implies 
$$\nabla_j(p.I^-)=\chi_j^-(p) =\dfrac{\chi_{k+1-j}(p)}{\chi_0(p)}=\frac{\Delta_{k+1-j}(p.I^+)}{\Delta_0(p.I^+)},$$
and this proves the statement in this case. \medskip

\noindent Suppose now that  $\tilde{\go g}$ is of Type $II$, that is that $\ell=1$ and  $e=1$ or  $2$. By Lemma \ref{lem-NO+},  $X$ is  $N$-conjugated to an element  $Z=\sum_{j=0}^k z_j X_j$ with  $z_0,\ldots, z_k \in F^*$. By the hypothesis on the  $X_j$'s and from the definition of the involution  $\gamma$, we have 
$\i(Z) =z_0^{-1} Y_0+\ldots +z_k^{-1} Y_k$ and  $\gamma (\i(Z))=z_k^{-1} X_0+\ldots +z_0^{-1} X_k$. By Theorem  \ref{th-nabla} and \ref{th-delta_{j}invariants}, the polynomials  $\nabla_j$ and $\Delta_{j}$ are $N$-invariant and hence we have 
$$\nabla_0(\i(X))=\nabla_0(\i(Z))=\Delta_0(\gamma.\i(Z))=\prod_{s=0}^k z_s^{-1}=\frac{1}{\Delta_0(Z)}=\frac{1}{\Delta_0(X)},$$
and  
$$\nabla_j(\i(X))=\nabla_j(\i(Z))=\Delta_j(\gamma.\i(Z))= \prod_{s=0}^{k-j} z_s^{-1}=\frac{\Delta_{k+1-j}(Z)}{\Delta_0(Z)}=\frac{\Delta_{k+1-j}(X)}{\Delta_0(X)},$$
and this again proves the statement.  \medskip

\noindent Suppose now that  $\tilde{\go g}$ is of Type  $III$, this means that  $\ell=3$.  We use the notations and the material developed   in  \textsection \ref{subsection(l=3)}. In particuliar  we realize the algebra  $\tilde{\go g}$ as a subalgebra of  ${\go sp}(4(k+1), E)$ where $E=F[\sqrt{\ep}]$ and where  $\ep\in F^*\setminus F^{*2}$ is a unit. We will first describe precisely the involution  $\gamma$ and the polynomials  $\nabla_j$.

We define   $I^+=\left(\begin{array}{cc} 0 & {\mathbf I}^+\\ 0 & 0\end{array}\right)$
where  ${\mathbf I}^+\in M(2(k+1), E)$ is the  $2\times2$ block diagonal matrix whose diagonal blocks are all equal to 
$J_1=\left(\begin{array}{cc} 0 & 1\\ 1 & 0\end{array}\right).$ Then  $I^-=\left(\begin{array}{cc} 0 &0\\ - {\mathbf I}^+& 0\end{array}\right).$
We normalize the polynomials  $\Delta_j$ in such a way that    $\Delta_j(I^+)=1$ for all  $j\in\{0,\ldots, k\}$.
  
 We set $\Gamma=\left(\begin{array}{ccc} 0 & & 1\\ & \adots & \\ 1 & &0\end{array}\right)$. We show now that $ \gamma=\left(\begin{array}{cc} 0 & \Gamma\\ -\Gamma & 0\end{array}\right)$ satisfies the properties of  \ref{th-involution-gamma}. As  $\gamma$ centralizes the matrix $K_{2(k+1)}=\left(\begin{array}{cc} 0 &I_{2(k+1)}\\ -I_{2(k+1)} & 0\end{array}\right)$, the element 
 $\gamma$ (or to be more precise, the conjugation by $\gamma$) belongs to  ${\rm Aut}_0(\tilde{\go g}\otimes_F E)$. In order to verify that  $\gamma \in\tilde{G}$, it suffices to verify that  $\gamma$ normalizes $\tilde{\go g}$.   Recall that  $\tilde{\go g}$ is the set of matrices $Z\in \tilde{\go g}\otimes_F E$ such that   $T\overline{Z}=Z T$ where  $ T=\left(\begin{array}{cc} J &  0\\ 0 & ^{t}J\end{array}\right)$ with  $J=\left(\begin{array}{ccc} J_\pi &0&0\\ 0& \ddots& 0\\ 0&0& J_\pi\end{array}\right)$
 and  $J_\pi= \left(\begin{array}{cc}0 & \pi\\ 1& 0\end{array}\right)$.
As
$\gamma^{-1}T\gamma=\left(\begin{array}{cc}  \Gamma \;^{t}J \Gamma & 0\\0 & \Gamma J \Gamma \end{array}\right)$
and  $ \Gamma J \Gamma=\;^{t}J$, we obtain  $\gamma^{-1}T\gamma=T$. It follows that $\gamma$ normalizes  $\tilde{\go g}$ and hence  $\gamma\in \tilde{G}$.

For  $Z= \left(\begin{array}{cc} {\mathbf A} &   {\mathbf X}\\  {\mathbf Y} & -^{t} {\mathbf A}\end{array}\right)\in\tilde{\go g}
$, we have
$$\gamma.Z=\left(\begin{array}{cc}- \Gamma\; ^{t} {\mathbf A}  \Gamma& - \Gamma  {\mathbf Y} \Gamma\\- \Gamma  {\mathbf X} \Gamma &  \Gamma  {\mathbf A} \Gamma \end{array}\right).$$

If  $ \mathbf Z\in M(2(k+1,E)$ can be written as $ \mathbf  Z= \left(\begin{array}{ccc} {\mathbf Z}_{0,0} & \ldots & {\mathbf Z}_{0,k} \\ \vdots & & \vdots \\ {\mathbf Z}_{k,0}  & \ldots & {\mathbf Z}_{k,k} \end{array}\right)$ with ${\mathbf Z}_{r,s}\in M(2,E)$, then 
$$\Gamma {\mathbf Z}\Gamma=\left(\begin{array}{ccc} J_1 {\mathbf Z}_{k,k}J_1 & \ldots &J_1 {\mathbf Z}_{k,0}J_1 \\ \vdots & & \vdots \\J_1 {\mathbf Z}_{0,k}  J_1& \ldots & J_1{\mathbf Z}_{k,0}J_1 \end{array}\right).$$ 
 
It is now easy to verify that  $\gamma$ normalizes $\go a^0$ and satisfies the properties of Theorem \ref{th-involution-gamma} . 

From the normalization made in section \ref{subsection(l=3)}, we have  
 $$\Delta_j\left(\begin{array}{cc} 0 & \mathbf X\\ 0 & 0\end{array}\right)=(-1)^{k-j+1} {\rm det}(\tilde{\mathbf X}_j)$$
 where $\tilde{\mathbf X}_j$ is the square matrix of size $2(k+1-j)$ defined by the $2(k+1-j)$ first rows and columns of  $\mathbf X$.  Explicitly, if 
 $\mathbf X=({\mathbf X}_{r,s})_{r,s=0,\ldots, k}$ where  ${\mathbf X}_{r,s}\in M(2,E)$, we have  
 $$\Delta_j(  X)=(-1)^{k-j+1} \left|\begin{array}{ccc} {\mathbf X}_{0,0} & \ldots & {\mathbf X}_{0,k-j} \\ \vdots & & \vdots \\ {\mathbf X}_{k-j,0}  & \ldots & {\mathbf X}_{k-j,k-j} \end{array}\right|$$

From the definition of $\nabla_j$, we get 
 $$\nabla_j\left(\begin{array}{cc} 0 &0\\  \mathbf Y & 0\end{array}\right)=\Delta_j\left(\begin{array}{cc} 0 & -\Gamma \mathbf Y \Gamma\\ 0& 0\end{array}\right).$$
Therefore if 
 $ \mathbf Y= \left(\begin{array}{ccc} {\mathbf Y}_{0,0} & \ldots & {\mathbf Y}_{0,k} \\ \vdots & & \vdots \\ {\mathbf Y}_{k,0}  & \ldots & {\mathbf Y}_{k,k} \end{array}\right),$
we have 
$$\nabla_j(Y)=(-1)^{k+1-j} \left|\begin{array}{ccc}-J_1 {\mathbf Y}_{k,k}J_1 & \ldots &-J_1 {\mathbf Y}_{k,j}J_1 \\ \vdots & & \vdots \\-J_1 {\mathbf Y}_{j,k}  J_1& \ldots & -J_1{\mathbf Y}_{j,j}J_1 \end{array}\right|= \left|\begin{array}{ccc}J_1 {\mathbf Y}_{k,k}J_1 & \ldots &J_1 {\mathbf Y}_{k,j}J_1 \\ \vdots & & \vdots \\J_1 {\mathbf Y}_{j,k}  J_1& \ldots & J_1{\mathbf Y}_{j,j}J_1 \end{array}\right|.$$
As ${\rm det}(J_1)^2=1$, we obtain:  
$$\nabla_j(Y)= \left|\begin{array}{ccc}   {\mathbf Y}_{k,k}  & \ldots &  {\mathbf Y}_{k,j}  \\ \vdots & & \vdots \\  {\mathbf Y}_{j,k}  & \ldots &  {\mathbf Y}_{j,j}  \end{array}\right|=\left|\begin{array}{ccc}   {\mathbf Y}_{j,j}  & \ldots &  {\mathbf Y}_{j,k}  \\ \vdots & & \vdots \\  {\mathbf Y}_{k,j}  & \ldots &  {\mathbf Y}_{k,k}  \end{array}\right|.$$

Let  $X=\left(\begin{array}{cc} 0 & \mathbf X\\ 0 & 0\end{array}\right)\in \Omega^+$. A simple computation shows that  $\i(X)=\left(\begin{array}{cc} 0 &0\\  -\mathbf X^{-1} & 0\end{array}\right)$. 
If  $X\in \mathcal O^+$, then by Lemma \ref{lem-NO+} , there exists $n\in N$ such that $n.X=\left(\begin{array}{cc} 0 & \mathbf Z\\ 0 & 0\end{array}\right)$, where 
  $\mathbf Z=\left(\begin{array}{ccc} {\mathbf Z}_k &0 & 0\\
0 & \ddots &0\\ 0 & 0 &  {\mathbf Z}_0\end{array}\right)$ and  ${\mathbf Z}_j\in M(2,E)$. We set $Z=n.X$. From above we get:
$$\nabla_j(\i(Z))=(-1)^{k-j+1}\prod_{s=0}^{k-j} \frac{1}{{\rm det}({\mathbf Z}_s)}=(-1)^{k-j+1}\frac{\prod_{s=k-j+1}^{k}{\rm det}({\mathbf Z}_s)}{\prod_{s=0}^{k}{\rm det}({\mathbf Z}_s)}=\frac{\Delta_{k+1-j}(Z)}{\Delta_0(Z)}.$$ 
As the $\Delta_{j}$'s and the $\nabla_j$'s are invariant under $N$, we have $$\nabla_j(\i(X))=\frac{\Delta_{k+1-j}(X)}{\Delta_0(X)}$$ for all  $X\in \mathcal O^+$and hence for all $X\in \Omega^+$.\end{proof}

\begin{definition} Let  $s=(s_0,\ldots, s_k)\in \mathbb C^{k+1}$. We denote by $|\nabla|^s$ \index{nablas@$|\nabla|^s$}and  $|\Delta^s|$ \index{Deltas@$|\Delta^s|$} the functions  respectively defined on  $\mathcal O^+$ and $\mathcal O^-$ by 
$$|\Delta|^s(X)=|\Delta_0(X)|^{s_0}\ldots |\Delta_k(X)|^{s_k},\quad{\rm for } \; X\in \mathcal O^+,$$
$$|\nabla|^s(Y)=|\nabla_0(Y)|^{s_0}\ldots |\nabla_k(Y)|^{s_k},\quad{\rm for }\;  Y\in \mathcal O^-.$$
\end{definition}
\begin{definition} We denote by  $t$  the involution on $\mathbb C^{k+1}$ defined  by
$$t(s)\index{ts@$t(s)$}=(-s_0-s_1-\ldots - s_k, s_k, s_{k-1},\ldots, s_1),$$
for  $s=(s_0,\ldots, s_k)\in \mathbb C^{k+1}$.
\end{definition} 
\begin{cor}\label{cor-nabla-psi} Let $X\in \Omega^+$. For $s\in  \mathbb C^{k+1}$, we have $$|\nabla|^s(\i(X))=|\Delta|^{t(s)}(X).$$
In particular, the polynomials $|\nabla|^s$ and  $|\Delta|^{s'}$ have the same  $A^0$-character if and only if  $s'=t(s)$.
\end{cor}
\begin{proof} The first statement is a straightforward consequence of Proposition  \ref{prop-delta-iota}. The second assertion  follows then from Theorem \ref {th-delta_{j}invariants} and \ref{th-nabla} \end{proof}
\newpage
{\,}
 \vskip 200pt
 
  \part{}{\centering{\Huge {\bf Local Zeta Functions}}}
 
\newpage
\section{Preliminaries }\label{sec:Preliminaries}
 \subsection{Summary of some results and notations from Part I}\label{section-notation}
\hfill

For the convenience of the reader, we summarize briefly here the basic definitions and results from Part I which are needed for the study of the local zeta functions.
\me 
Let $F$ be a p-adic field of characteristic $0$, i.e. a finite extension of $\Q_{p}$.

Let  $\Or_F$ be the ring of integers of  $F$ and denote by $\Or_F^*$ the subgroup of units of  $F^*$. 
 We will always suppose that the residue class field $\bf k=\Or_F/\Or_F^*$ has characteristic $\neq 2$ (non dyadic case). Let $q$ denote the cardinal of $\bf k$.

We fix a set of representatives $\{1,\ep,\pi,\ep\pi\}$ of $\Cr=F^*/F^{*2}$ , where $\ep$ is a unit which is not a square in  $F^*$ and $\pi$ is a uniformizer of  $F$.\\

Let  $\overline{F}$ be an algebraic closure of $F$. In the sequel, if $U$ is a $F$-vector space, we will set 
 $$\overline U=U\otimes_{F}\overline{F}.$$
 
The regular graded Lie algebra which are our basic objects were defined as follows in  Part I.

 \begin{definition}\label{def-alg-grad}
 
 A reductive  Lie algebra $\widetilde{ {\mathfrak g}}$ over $F$   satisfying the following three  hypothesis will be called {\it a regular graded Lie algebra}:
 
 \begin{itemize}
 \item[$(\bf H_1)$]\label{H1}{\it  There exists an  element $H_0\in \widetilde{ {\mathfrak g}}$ such that 
$\ad H_0$ defines a   ${\mathbb  Z}$-grading  of the form }
$$\widetilde{ {\mathfrak g}}=V^-\oplus {\mathfrak g}\oplus V^+\quad (V^+\not = \{0\})\ ,$$
{\it where } $[H_0,X]=\begin{cases}
 0&\hbox{  for }X\in {\mathfrak g}\ ;\\
         2X&\hbox{  for }X\in V^+\ ;\\
        -2X&\hbox{  for }X\in V^-\ .
\end{cases}$

 ({\it Therefore, in fact,  } $H_{0}\in \go{g}$)
\item[$(\bf H_2)$]\label{H2}{\it  The $($bracket$)$ representation  of ${\mathfrak g}$ on  $\overline {V^+}$ is irreducible. $($In other words, the representation $({\mathfrak g},V^+)$ is absolutely irreducible$)$}
\item[$(\bf H_3)$]\label{H3} {\it There exist $I^+\in V^+$ and $I^-\in V^-$ such that $\{I^-,H_{0},I^+\}$ is an $\go{sl}_{2}$-triple. }
\end{itemize} 
\end{definition}

\vskip 5pt

 
We  fix a maximal split abelian Lie subalgebra  $\go{a}$ of $\go{g}$ containing $H_{0}$. Then $\go{a}$ is also maximal split abelian in $\widetilde{\go g}$. Let   $\widetilde \Sigma$  and $\Sigma$  be the roots system of $(\widetilde{\go g},\go{a})$ and  $(\go{g},\go{a})$ respectively.\me

Let   $H_\la$ be the coroot of  $\la\in \widetilde \Sigma$. For  $\mu\in \go a^*$, we denote by  $\widetilde{\go g}^{\mu}$ the subspace of weight  $\mu$ of
 $\widetilde{\go g}$.\me

From   Theorem \ref{th-descente}, Proposition \ref{prop-systememax} and   Theorem \ref{th-decomp-Eij}, we know that there exists a unique  $k\in \N$   and a family of $2$ by $2$ strongly orthogonal roots  $\la_0,\ldots, \la_k$ in  $\widetilde \Sigma^+ \setminus \Sigma$, unique modulo the action of the Weyl group  of $\Sigma$, such that:
 \begin{itemize}
 \item[$(\bf S1)$] $H_{0}=H_{\lambda_{0}}+H_{\lambda_{1}}+\dots+H_{\lambda_{k}},$
 \item[$(\bf S2)$] If we set 
 $\go a^0=\oplus_{j=0}^k FH_{\la_j}$  and if for $p,q\in\Z$, we define $ E_{i,j}(p,q)$ to be the space of $ X\in \widetilde{ {\go g}}$ such that 
 $[H_{\lambda
_s},X]=
\begin{cases}
pX&\hbox{  if }s=i\ ;\cr
qX&\hbox{ if }s=j\ ;\cr
0&\hbox{ if }s\notin \{i,j\}\ .
\end{cases}$. \\
Then we have the following decompositions
\begin{align*}
(a) \hskip 25pt {\go g}&={\mathcal Z}_{\go g}({\go a}^0)\oplus\bigl(\oplus _{i\not =
j}E_{i,j}(1,-1)\bigr)\ ;\cr 
(b) \hskip 15pt V^+&=\bigl(\oplus_{j=0}^k \widetilde{ {\go g}}^{\lambda _j}\bigr)
\oplus\bigl(\oplus _{i<j}E_{i,j}(1,1)\bigr)\ ;\cr
(c) \hskip 15pt V^-&=\bigl(\oplus_{j=0}^k \widetilde{ {\go g}}^{-\lambda _j}\bigr) \oplus\bigl(\oplus
_{i<j}E_{i,j}(-1,-1)\bigr)\ .
\end{align*}
\end{itemize}

In the rest of the paper, we fix such a family  $(\la_0,\ldots, \la_k)$ of roots. The integer  $k+1$ is called the {\it rank } of the graded Lie algebra $\widetilde{\go g}$. From   Proposition \ref{prop-memedimension}, we know that for $j=0,\ldots, k$,  the spaces  $\widetilde{\go g}^{\la_j}$ have the same dimension and also that for $i\neq j$, the spaces  $E_{i,j}(\pm1,\pm 1)$ have  the same dimension.\me

Remember also from \ref{notdle} the definition of the following constants:
\medskip
           \begin{align*}
           \ell&=\dim \widetilde{ {\go g}}^{\lambda _j} \hbox{  for }j=0,\ldots ,k\;\cr
           d&=\dim E_{i,j}(\pm 1,\pm 1) \hbox{  for }i\not = j\in \{0,\ldots ,k\}\;\cr
           e&=\dim \tilde{\go g}^{(\lambda_i+\lambda_j)/2}  \hbox{  for }i\not = j\in \{0,\ldots ,k\}\ (e \text{ may be equal to } 0).\
                \end{align*}

From the classification of the simple graded Lie algebras  (see section \ref{sec:section-Table}), the integer  $\ell$ is either the square of an integer, or equal to  $3$ and  $e\in\{0,1,2,3,4\}$. Moreover $d-e$ is even.
 
\me

For  a reductive Lie algebra $\go k$   let  $\text{Aut}(\go k)$  be the group of automorphisms of  $\go k$ and     $\text{Aut}_e(\go k)$ the subgroup of elementary automorphisms, that is the automorphisms which are finite products of  $e^{\ad x}$, where  $\ad(x)$ is nilpotent on  $\go k$. Define  
$$\text{Aut}_0(\go k)=\text{Aut}(\go k)\cap \text{Aut}_e(\widetilde{\go k}\otimes\overline{F}).$$

Let   $G= {\mathcal Z}_{\text{Aut}_{0}(\widetilde{\go{g}})}(H_{0})=\{g\in \text{Aut}_{0}(\widetilde{\go{g}}),\, g.H_{0}=H_{0}\}$ be the centralizer of $H_{0}$ in   $\text{Aut}_{0}(\widetilde{\go{g}})$ (see section \ref{sub-section-generic-strongly-orth}).\me

The group  $G$ is the group of  $F$-point  of an algebraic group defined over  $F$. Moreover, the Lie algebra of the group  $G$  is $\go g\cap[\widetilde{\go g}, \widetilde{\go g}]=[\go g,\go g]+[V^-,V^+]$.\me

We recall also that  $(G,V^+)$ is an absolutely irreducible prehomogeneous vector space, of commutative type, and regular. By Theorem \ref{thexistedelta_0},  there exists on  $V^+$ a unique (up to scalar multiplication) relative invariant polynomial $\Delta_{0}$ which is absolutely irreducible  (i.e. irreducible as a polynomial on $\overline{V^+}$).  We denote by $\chi_0$ the character of $G$ associated to $\De_0$.\me

Similarly, $(G,V^-)$ is a prehomogeneous vector space of the same type, whose absolutely irreducible polynomial relative invariant is denoted by  $\nabla_0$.  Its associated character is denoted by $\chi_0^-$.  \vskip 5pt

We set 
\beq\label{def-m} \index{m@$m$}m=\frac{{\rm dim}\; V^+}{{\rm deg}(\De_0)}.\eeq

 Let $dX$ (resp. $dY)$ be arbitrary Haar measures on $V^+ $ (resp. $V^-$). We know from  the  proof of Lemma \ref{lem-propchij} that 
 $\det g_{|_{V^+}}=\chi_{0}(g)^m$ and $\det g_{|_{V^-}}=\chi_{0}(g)^{-m}$.
This implies that the measures 
 \begin{equation}\label{eq-d^*X} d^*X=\frac{dX}{|\Delta_{0}(X)|^m} \,\, \text{ and } \displaystyle d^*Y=\frac{dY}{|\nabla_{0}(Y)|^m}\index{d*X@$d^*X$}\index{d*Y@$d^*Y$}\end{equation}
 are $G$-invariant measures on $V^+$ and $ V^-$ respectively.
 
\vskip 5pt

{\bf In this Part II, we will always suppose that  $\ell=1$,  except in  section 2 in which the results of  sections \ref{sec:section-parabolique},    \ref{sec:section-Fourier-Normalization},   \ref{sec:section-MeanFunctions} and   \ref{sec:section-Weil-formula} are valid without any condition on $\ell$. For $\ell=1$, we have $m=1+\dfrac{kd}{2}$.}
\me

According to section \ref{sec:NKF},  we fix    a non degenerate  extension $\widetilde{B}$ of the Killing form of  $[\widetilde{\go{g}},\widetilde{\go{g}}]$ to $\widetilde{\go{g}}$. We define the {\it normalized Killing form} by setting:
\beq\label{def-b}b(X,Y)= -\frac{k+1}{4\,\dim V^+} \widetilde{ B }(X,Y)\quad X\in, Y\in \tilde{\go g}.\eeq

Let   $X_s\in\tilde{\go g}^{\la_s}$ for   $s=0,\ldots k$. For $i\neq j$ let us consider the quadratic form $q_{X_i,X_j}$ on  $E_{i,j}(-1,-1)$ defined by  \beq\label{def-qXiXj}q_{X_i,X_j}(Y)=-\frac{1}{2} b([X_i,Y], [X_j,Y]).\eeq
\begin{definition}\label{def-qe-I+-I-}\hfill

 According to Definition \ref{def-qe}, we fix $\go{sl}_2$-triples $(Y_s,H_{\lambda_s}, X_s)$, $s\in\{0,\dots, k\}$,  with $Y_s\in\tilde{\go g}^{-\la_s}$ and $X_s\in \tilde{\go g}^{\la_s}$ such that,  for $k\geq 1$ and $i\neq j$, the quadratic forms  $q_{X_i,X_j}$ are all $G$-equivalent. We  set 
$$q_e:=q_{X_0,X_1}=q_{an,e}+q_{hyp,d-e},$$
where  $q_{hyp,d-e}$ is an hyperbolic quadratic form of rank $d-e$,  $q_{an,0}=0$ and for $e\neq 0$,  $q_{an,e}$ is an anisotropic quadratic form of rank $e$ which represents $1$. \\
We define 
$$ I^+=X_0+\ldots +X_k,\quad\text{and}\quad I^-=Y_0+\ldots Y_k,$$
And we  normalize the polynomials  $\De_0$ and  $\nabla_0$ by setting  $\De_0(I^+)=\nabla_0(I^-)=1$. 
\end{definition}

\begin{definition}\label{def-ensembleS} \hfill.  

According to Lemma \ref{lem-ensembleSe},  we set   $S_e=\{a\in F^*, aq_{e}\sim q_{e}\}=\{a\in F^*, aq_{an,e}\sim q_{an,e}\}$ for $k\geq 1$. By convention, for $k=0$, we set $e=0$ and $S_e=F^*$. Then   
$$S_e=\left\{\begin{array}{ccc} F^* & \text{for} & e=0\;\text{or}\; 4\\
F^{*2} & \text{for} & e=1\;\text{or}\; 3\\
N_{E/F}(E^*)=\text{Im}(q_{an,e})^*& \text{for} & e=2.\end{array}\right.$$
We denote by   $\Sr_e$ a set of representatives of  $F^*/S_e$ in $F^*/F^{*2}$.
\end{definition}
As two hyperbolic forms of the same rank are always equivalent, we also have $S_e=\{a\in F^*, aq_{an, e}\simeq q_{an,e}\}$.

The precise description of the  $G$-orbits  in  $V^+$ has been given in  section \ref{sec:G-orbites}. The number  of orbits depends on $e$ and, in some cases, on $k$ (see Theorem \ref{thm-orbouverte}). This is summarized as follows:\me

\begin{enumerate}\item If    $e=0$ or $4$, the group  $G$ has a unique open orbit in  $V^+$,
 
\item if  $e\in\{1,2,3\}$, the number of open  $G$-orbits in  $V^+$ depends on $e$ and on the parity of  $k$:

\begin{enumerate}\item if  $e=2$ then $G$ has a unique open orbit in $V^+$ if $k$ is even and  $2$ open orbits if $k$ is odd,

\item if  $e=1$ then  $G$ has a unique open orbit in  $V^+$ if $k=0$, it has $4$ open orbits if $k=1$, it has  $2$ open orbits if $k\geq 2$ is even , and  $5$ open orbits if $k\geq 2$ is odd .

\item  if  $e=3$, then $G$ has $4$ open orbits.
\end{enumerate}
\end{enumerate}

  Let us denote by $\Om^\pm$ the union of the open  $G$-orbits   $V^\pm$. Then one has:  
$$\Om^+=\{X\in V^+; \De_0(X)\neq 0\}\quad \text{and}\quad \Om^-=\{Y\in V^-; \nabla_0(Y)\neq 0\}.$$

If  $X\in \Om^+$, there exists a unique element  $\iota(X)\in \Omega^-$ such that  $ \{\iota(X),H_0,X\}$ is an  ${\go sl}_2$-triple.

From  the remark   after Definition \ref{iota}, we know that the map  $\iota:\Om^+\to\Om^-$ is a  $G$-equivariant  isomorphism from $\Om^+$ on  $\Om^-$.

\begin{definition}\label{def-hx} Let $(y,h,x)$ be a $\go sl_2$-triple with $y\in V^-, h\in\go a$ and $x\in V^+$. According to Definition \ref{def-theta-h}, for $t\in \overline{F}^*$, we  define the following elements of $\text{\rm Aut}_{e}(\overline{\widetilde{ {\go g} }})$:
$$\theta_{x}(t)=e^{t\ad_{\overline{\widetilde{\go{g}}} }(x)}e^{t^{-1}\ad_{\overline{\widetilde{\go{g}}}}(y)} e^{t\ad_{\overline{\widetilde{\go{g}}} }(x)}\quad{\rm and } \quad h_{x}(t)=\theta_{x}(t)\theta_{x}(-1).$$

By  Lemma \ref{lem-tId-dansG}, if $t\in F$, the element
 $h_{I^+}(\sqrt{t})$ (defined by the $\go{sl}_{2}$-triple $(I^-,H_{0},I^+)$) belongs to $ G$  (and even to the Levi subgroup $L=Z_G(\go a^0)$). It stabilizes $\go g$ and acts by $t.Id_{V^+}$ on $V^+$.   We set 
 $$m_t=h_{I^+}(\sqrt{t}).$$
\end{definition}

 \begin{rem}\label{rem-hx} Let  $(y,h,x)$ be a $\go sl_2$-triple with $y\in V^-, h\in\go a$ and $x\in V^+$. Then, by  (\cite{Bou2} Chap. VIII, \textsection 1, Proposition 6), $h_x(t)$ acts by $t^n$ on the subspace of $\tilde{\go g}$ of weight $n$ under the action of $h$.
\end{rem}

\subsection  {The subgroups $P$ and  $\tilde{P}$ and their open orbits in  $V^+$}\label{subsection-P}
\hfill

Set $\go a^0=\oplus_{j=0}^k FH_{\la_j}$. The parabolic subgroup  $ P$ of $G$ was defined as follows  (section \ref{sec:sigma-parabolic}):

$$P=LN,\;\text{where}\; L=Z_G(\go a^0)\;\text{and} \;\; N=\exp \text{ad}\; \go n,\; \go n=\oplus_{0\leq i<j\leq k}E_{i,j}(1,-1).$$

The representations   $(P,V^+)$ and  $(P,V^-)$ are prehomogeneous  (Proposition \ref{(P,V)PH}).  The corresponding sets of fundamental relative invariants are respectively denoted by  $\De_0,\ldots, \De_k$ and  $\nabla_0,\ldots, \nabla_k$. These polynomials are normalized by the conditions $\De_j(I^+)=\nabla_j(I^-)=1$ for  $j=0,\ldots, k$. Let  $\chi_j$ and  $\chi_j^-$ be the characters of  $P$ associated respectively to  $\De_j$ and  $\nabla_j$ pour $j=0,\ldots, k$.\me

From   Theorem \ref{th-Porbites} and Theorem \ref{th-PorbitesV-}, we know that the union   $\cO^\pm$  of the open    $P$-orbits in  $V^\pm$ is given by $$\cO^+=\{X\in V^+; \De_j(X)\neq 0, j=0,\ldots,k\}\quad\text{and}\quad  \cO^-=\{Y\in V^+; \nabla_j(Y)\neq 0, j=0,\ldots,k\}.$$
 
 \begin{definition}\label{def-xj-PLtilde}\hfill
 
  \begin{enumerate}\item  As $\ell=1$ and as  $L$ centralizes $\go a^0=\oplus_{j=0}^k FH_{\la_j}$, each element  $l\in L$ acts by scalar multiplication on  $X_j$ for $j=0,\ldots ,k$. Define the character  $x_j$ \index{xj@$x_{j}$}of  $L$ by
 $$l.X_j=x_j(l) X_j,\quad l\in L.$$
\item The subgroups $\tilde{L}$ and $\tilde{P}$ \index{Ltilde@$\tilde{L}$} \index{Ptilde@$\tilde{P}$}are defined by
$$\tilde{L}=\{l\in L; x_j(l)\in S_e,\; j=0,\ldots, k\},\quad \tilde{P}=\tilde{L}N.$$
\end{enumerate}
\end{definition}

\begin{rem}\label{rem-image-muj}\hfill

Define the map $\cT$\index{T@$\cT$} by 
$$\begin{array}{rll}
\cT: L&\longrightarrow&(F^*)^{k+1}\\
l&\longmapsto &({x_{0}(l),\ldots,x_{k}(l))}
\end{array}$$
 
 From Corollary \ref{cor-Lorbits} we can see that $$\text{\rm Im} \cT=\cup_{x\in \Sr_{e}}xS_{e}^{k+1}.$$
In particular we have $\cT(\tilde L)=S_{e}^{k+1}$.

 \end{rem}

\begin{definition}\label{def-Ptilde-orbites} For $a=(a_0,\ldots,a_k)\in \Sr_e^{k+1}$, define the sets  
$$\cO^+(a)\index{O+@$\cO^+(a)$}=\{X\in V^+; \De_j(X)a_j\ldots a_k\in S_e,\;\text{for }\; j=0,\ldots,k\},$$
and 
$$\cO^-(a)\index{O-@$\cO^-(a)$}=\{Y\in V^-; \nabla_j(Y)a_0\ldots a_{k-j}\in S_e,\;\text{for }\; j=0,\ldots,k\}.$$
\end{definition}
\begin{lemme}\label{lemme-Ptilde}\hfill

   Remember that     $m_x\in L$, $x\in F^*$ is the element of  $L$   whose action on $V^+$ is the multiplication by $x$ (cf. Definition \ref{def-hx}). 

{\rm 1)} One has  $$L=\cup_{x\in\Sr_e} m_x \tilde{L} $$
{\rm 2)} The representations  $(\tilde{P}, V^\pm)$ are prehomogeneous  vector spaces which  have the same sets of fundamental relative invariants as  $(P,V^\pm)$,

{\rm 3)} The open orbits of  $\tilde{P}$  in $V^+$ (resp. $V^-$) are the sets 
$$\cO^+(a)=\tilde{P}.I^+(a),\;\text{where}\; I^+(a)=a_0X_0+\ldots +a_k X_k,$$
  $$\text{(respectively }\;\cO^-(a)=\tilde{P}.I^-(a),\;\text{where}\; I^-(a)=a_0^{-1} Y_0+\ldots +a_k^{-1} Y_k),$$

for  $a=(a_0,\ldots,a_k)\in \Sr_e^{k+1}$.

Moreover, any open  $P$-orbit  in  $V^\pm$ is given by   $\cup_{x\in\Sr_e} m_x\cO^\pm(a)$ for  $a=(a_0,\ldots, a_{k-1},1)\in\Sr_e^{k+1}$.

\end{lemme}
\begin{proof} \hfill

If $e=0$ or  $4$, we have    $S_e=F^*$ and hence $\Sr_e=\{1\}$ and  $P=\tilde{P}$. Then  Theorem \ref{th-Porbites} and Theorem \ref{th-PorbitesV-}   imply the result.\me

Suppose now that  $e\in\{1,2,3\}$. This implies $k\geq 1$.  \me

Using  Corollary \ref{cor-Lorbits}, we know that two diagonal elements $a_0X_0+\ldots +a_kX_  k$ and  $b_0X_0+\ldots +b_kX_k$ are  $L$-conjugated if and only if there exists  $x\in\Sr_e$ such that  $x a_j b_j\in S_e$ for all  $j=0,\ldots, k$. 

If  $l\in L$, then  $l.I^+=x_0(l)X_0+\ldots +x_k(l) X_k$. Hence there exists  $x\in\Sr_e$ such that  $xx_j(l)\in S_e$ for all  $j=0,\ldots, k$. Therefore  $m_xl\in \tilde{L}$,  this proves assertion 1).\me
 
As $P/\tilde{P}$ is a finite group, assertion 2) is true.\me

Let us now show assertion 3). Let  $a=(a_0,\ldots, a_k)\in\Sr_e^{k+1}$. From Theorem \ref{th-delta_{j}invariants}, we have  for $l\in\tilde{L}$ and  $n\in N$:
$$\De_j(ln.(a_0X_0+\ldots +a_kX_k))=\prod_{s=j}^k x_s(l)a_s.$$
As $x_s(l)\in S_e$ for  $s=0,\ldots, k$, we obtain from the definition of   $\cO^+(a_0,\ldots,a_k)$ that  $\tilde{P}.(a_0X_0+\ldots +a_kX_k)\subset\cO^+(a_0,\ldots,a_k)$.\me

Conversely, let  $Z\in \cO^+(a_0,\ldots,a_k)$. From Lemma \ref{lem-NO+}, there exists  $n\in N$ et $z_0,\ldots z_k$ in $F^*$ such that  
$n.Z=z_0X_0+\dots +z_kX_k$. From the definition of  $ \cO^+(a_0,\ldots,a_k)$ (and as $F^{*2}\subset S_{e}$), there exist  $\mu_j\in S_e$ such that   $z_j=a_j\mu_j$ for   $j=0,\ldots, k$. 

  From Remark \ref{rem-image-muj} there exists  $l\in \tilde{L}$ such that  $x_j(l)=\mu_j^{-1}$ for  $j=0,\ldots, k$. Then  we have $ln.Z=a_0X_0+\ldots +a_kX_k$ and hence  $$ \cO^+(a_0,\ldots,a_k)=\tilde{P}.I^+(a).$$

The proof concerning the open  $\tilde{P}$-orbits  in $V^-$ is similar. And the last  assertion is a consequence of the  description of the $P$-orbits in  $V^\pm$ given in     Theorem \ref{th-Porbites}  and Theorem \ref{th-PorbitesV-}. \end{proof}

\section{Mean functions and Weil formula}\label{sec:moyennes-Weil}
\hfill

 Except for the last section \textsection \ref{section-calcul-gamma} which is only valid for $\ell=1$, the results of this paragraph are valid for any regular graded algebra   (see Definition \ref{def-alg-grad}).   Hence $\ell$ is either the square of an integer, or $\ell=3$ (Notation \ref{notdle}).

\subsection{A class of  parabolic subgroups}\label{sec:section-parabolique}
\hfill
\vskip 5pt

Let us fix $p\in\{0,\ldots, k-1\}$ and define the following elements:
 
  $$I_{p,1}^+\index{Ip1+@$I_{p,1}^+$}=\sum_{i=0}^pX_{i},\hskip 30pt H_{0}^{p,1}=\sum_{i=0}^pH_{\lambda_{i}},\quad\text{and}\quad I_{p,1}^-\index{Ip1-@$I_{p,1}^-$}=\sum_{i=0}^p Y_i,$$ and
  $$I_{p,2}^+\index{Ip2+@$I_{p,2}^+$}=\sum_{i=p+1}^kX_i, \hskip 30pt H_{0}^{p,2}=\sum_{i=p+1}^kH_{\la_{i}},\quad\text{and}\quad I_{p,2}^-\index{Ip2-@$I_{p,2}^-$}=\sum_{i=p+1}^kY_{i}.$$ 
Hence we have
$$I^+=I_{p,1}^++I_{p,2}^+,\quad H_0=H_{0}^{p,1}+H_{0}^{p,2},\quad\text{and }\quad I^-=I_{p,1}^-+I_{p,2}^-,$$
and  $\{I_{p,j}^-, H_0^{p,j}, I_{p,j}^+\}$ is an  $\go{sl}_2$-triple, for  $j=1,2$.

 For a fixed  $p\in \{0,1,\ldots,k\}$ and  $i,j\in \Z$ we define the following spaces:
  $$K_{p}(i,j)\index{Kpij@$K_{p}(i,j)$}=\{X\in \go{g},\, [H_{0,p}^1,X]=iX,\,[H_{0,p}^2,X]=jX\}.$$
 The same way as in   Theorem \ref{th-decomp-Eij} one shows that  $K_{p}(i,j)\neq \{0\}\Longrightarrow |i|+|j|\leq 2$. 
 
 More precisely one obtains the following decompositions:
 $$\begin{array}{lll}
 V^- &=&K_{p}(-2,0)\oplus K_{p}(-1,-1)\oplus K_{p}(0,-2),\cr
 \go{g}&=&K_{p}(1,-1)\oplus K_{p}(0,0)\oplus K_{p}(-1,1),\cr
 V^+ &=&K_{p}(2,0)\oplus K_{p}(1,1)\oplus K_{p}(0,2).
 
 \end{array}$$
 
 Hence there are four new regular  prehomogeneous vector spaces (corresponding to graded algebras) at hand. Namely $(K_{p}(0,0), K_{p}(2,0))$, $(K_{p}(0,0), K_{p}(0,2))$, $(K_{p}(0,0), K_{p}(-2,0))$, $(K_{p}(0,0), K_{p}(0,-2))$. For further use we need to make explicit the fundamental relative invariants corresponding to these spaces.

 The graded Lie algebra $\tilde{\go{l}}_{j}={\tilde{\go g}}^{-\lambda _j}\oplus  [{\tilde{\go g}}^{-\lambda _j},{\tilde{\go g}}^{\lambda _j}]\oplus {\tilde{\go g}}^{\lambda _j}$ satisfies also the conditions of Definition \ref{def-alg-grad}. 
 Recall that $\kappa$  is    the common degree of the corresponding relative invariants (see  Notation \ref{notation-kappa}). More precisely $\kappa=\de$ if $\ell=\de^2$ and $\kappa=2$ if $\ell=3$.

  \begin{lemme} \label{lem-K(0,2)-abs-irr} The representations  $(K_{p}(0,0), K_{p}(2,0))$ and 
  $(K_{p}(0,0), K_{p}(0,2))$ are absolutely irreducible.
  \end{lemme}
  
  \begin{proof} Recall that we denote by  $\overline F$ the algebraic closure of $F$and for any $F$-vector space $U$, we define 
  $$\overline{U}=U\otimes _{F} \overline{F}$$
The space $\overline{K_{p}(0,2)}$ contains the highest weight vector  of  $\overline{V^+}$ (for the action of $\overline{\go{g}}$).  Let  $v_{1}$ be such a non zero vector. Of course $v_{1}$ is also a highest weight vector   for the action of $\overline{K_{p}(0,0)}$.  Suppose that      $\overline{K_{p}(0,2)}$ is reducible under the action of $\overline{K_{p}(0,0)}$. Then it would exist a  second highest weight vector $v_{2}$ for this action.    
   Let us examine how these vectors move under the action of ${\overline{n}}^{^-}$ ($=$ the sum of the negative root spaces in  $\overline{\go{g}}$). One has
   $${\overline{n}}^{^-}=\overline{K_{p}(1,-1)}{\oplus\overline{K_{p}(0,0)}}^{^-}.$$
   The action of  $\overline{K_{p}(1,-1)}$ on  $\overline{K_{p}(0,2)}$ is as follows:
   $$v_{1},v_{2}\in  \overline{K_{p}(0,2)}\hskip 5pt  \overset{\ad \overline{K_{p}(1,-1)}}{\longrightarrow} \hskip 5pt  \overline{K_{p}(1,1)}\hskip 5pt  \overset{\ad \overline{K_{p}(1,-1)}}{\longrightarrow} \hskip 5pt  \overline{K_{p}(2,0)}  \hskip 5pt  \overset{\ad \overline{K_{p}(1,-1)}}{\longrightarrow} \hskip 5pt \{0\}.$$
   On the other hand the action of  ${\overline{K_{p}(0,0)}}^{^-}$ stabilizes $ \overline{K_{p}(0,2)}$
   $$v_{1},v_{2}\in  \overline{K_{p}(0,2)}\hskip 5pt  \overset{\ad \overline{K_{p}(0,0)}^{^-}}{\longrightarrow} \hskip 5pt  \overline{K_{p}(0,2)}.$$
  Let ${\cal U}(\overline{K_{p}(0,0)})$ be the enveloping algebra of $\overline{K_{p}(0,0)}$. As  $\overline{K_{p}(0,2)}$ is supposed to be reducible under  $\overline{K_{p}(0,0)}$ we obtain   $$\overline{V_{1}}={\cal U}(\overline{K_{p}(0,0)})v_{1}\subsetneq \overline{K_{p}(0,2)}.$$
   As  $\overline{V^+}$ is irreducible under  $\overline{\go{g}}$ we should have:
   $$\begin{array}{rll}
   \overline{V^+}&=&{\cal U}(\overline{K_{p}(1,-1)}){\cal U}(\overline{K_{p}(0,0)})v_{1}\cr
   {}&=&{\cal U}(\overline{K_{p}(1,-1)})\overline{V_{1}}\subset \overline{V_{1}}\oplus \overline{K_{p}(1,1)}\oplus \overline{K_{p}(2,0)} \subset \overline{K_{p}(0,2)}\oplus \overline{K_{p}(1,1)}\oplus \overline{K_{p}(2,0)}=\overline{V^+}.
   \end{array}$$
   But this is not possible as  $\overline{V_{1}}\subsetneq \overline{K_{p}(0,2)}$.

As  $\overline{K_{p}(2,0)}$ contains a lowest weight vector of the representation $(\overline{\go{g}}, \overline{V^+})$, the proof for $\overline{K_{p}(2,0)}$ is analogue. 
\end{proof}

It is worth noticing that     from the preceding Lemma \ref{lem-K(0,2)-abs-irr} and from  Proposition \ref{ell-0-simple}   the  Lie algebras  $K_{p}(-2,0)\oplus [K_{p}(2,0),K_{p}(-2,0)]\oplus K_{p}(2,0)$ and  $K_{p}(0,-2)\oplus [K_{p}(0,2),K_{p}(0,-2)]\oplus K_{p}(0,2)$ are absolutely simple.

\vskip 10pt

Let us fix  $p\in\{0,1,\ldots,k\}$. Consider now the maximal   parabolic  subalgebra of  $\go{g}$ given by
$$\begin{array}{rll}
{\go{p}}_{p}\index{pgothiquep@${\go{p}}_{p}$}&=& {\cal Z}_{\go{g}}(H_{0,p}^1)\oplus K_{p}(1,-1)=  {\cal Z}_{\go{g}}(H_{0,p}^2)\oplus K_{p}(1,-1) ={\cal Z}_{\go{g}}(F.H_{0,p}^1\oplus F.H_{0,p}^2)\oplus K_{p}(1,-1)\cr
{}&=& K_{p}(0,0)\oplus K_{p}(1,-1)
\end{array} $$
 (As $H_{0}=H_{0,p}^1+H_{0,p}^2$ and as  $\go{g}={\cal Z}_{\widetilde{\go{g}}}(H_{0})$, we get  $ {\cal Z}_{\go{g}}(H_{0,p}^1)= {\cal Z}_{\go{g}}(H_{0,p}^2) = {\cal Z}_{\go{g}}(F.H_{0,p}^1\oplus F.H_{0,p}^2)$.)
 \vskip 7pt

 The maximal   parabolic subgroup  $P_{p}$ corresponding to  $\go{p}_{p}$ is now defined as follows. Define first:
  $$N_{p}\index{Np@$N_{p}$}=\exp\; \ad K_{p}(1,-1)$$
  $$ L_{p}\index{Lp@$L_{p}$}={\cal Z}_{G}(F.H_{0,p}^1\oplus F.H_{0,p}^2)={\cal Z}_ {\text{Aut}_{0}(\widetilde{\go{g}})}(F.H_{0,p}^1\oplus F.H_{0,p}^2)={\cal Z}_{G}(F.H_{0,p}^1)={\cal Z}_{G}(F.H_{0,p}^2) .$$
  Then 
  $$P_{p}\index{Pp@$P_{p}$}=L_{p}N_{p }$$
  
  The Lie algebra of $L_{p}$ is  $K_{p}(0,0)$, the Lie algebra of  $N_{p}$ is $K_{p}(1,-1)$  and hence the Lie algebra of $P_{p}$ is $K_{p}(0,0)\oplus K_{p}(1,-1)$.

   \begin{definition}\label{def-invariants}\hfill 
 
Let $p=0,\ldots,k-1$. From the definitions  in  section \ref{sec:sectiondeltaj} and section \ref{sec:PorbitesV-} the fundamental relative invariant of $(K_{p}(0,0), K_{p}(0,2))$ is $\Delta_{p+1}$ whose degree is $\kappa(k+1-(p+1))= \kappa(k-p)$. Similarly the fundamental relative invariant of $(K_{p}(0,0), K_{p}(-2,0)$ is $\nabla_{k-p}$ whose degree is $\kappa(k+1-(k-p))= \kappa(p+1)$.

We denote by $\widetilde \Delta_{k-p}$ the fundamental relative invariant of $(K_{p}(0,0), K_{p}(2,0))$ whose degree is $\kappa(k+1-(k-p))=\kappa(p+1)$.

We denote by $\widetilde \nabla_{p+1}$ the fundamental relative invariant of $(K_{p}(0,0), K_{p}(0,-2))$ whose degree is $\kappa(k+1-(p+1))=\kappa(k-p)$.

It is easy to see from the definitions that the polynomials $\widetilde \Delta_{p}$ ($p=0,\ldots,k-1$) are the fundamental relative invariants of the prehomogeneous vector space $(P^-, V^+)$ where $P^-$ is the parabolic subgroup opposite to $P$. Also the polynomials $\widetilde \nabla_{p}$  ($p=0,\ldots,k-1$) are the fundamental relative invariants of the prehomogeneous vector space $(P^-, V^-)$.

 \end{definition}

  \vskip 0,5cm
  
  \subsection{Fourier transform and normalization of measures}\label{sec:section-Fourier-Normalization}\hfill

  We will, in the sequel of this paragraph, normalize the Haar measures such that there is no constant appearing in some integral formulas occurring later, in particular in Theorem \ref{Formules-int-simples} and Theorem \ref{th-Weil}. We use  here  the same method as in (\cite {B-R} chap. 4) and in (\cite{Mu} \S 3.2). For the convenience of the reader we give some details.
  
  If $U$  is a subspace of $\tilde{\go g}$, let $\cS(U)$ be the space of locally constant functions on $U$ with compact support.

  Remember that we use the normalized Fourier Killing form  $b(\cdot,\cdot)$ 	defined by 
  $$b(X,Y)= -\frac{k+1}{4\,\dim V^+} \widetilde{ B }(X,Y)\quad X\in, Y\in \tilde{\go g}.$$
 Using $b$ the space $V^-$ is identified with the dual space ${V^+}^*$, and the natural action of $g\in G$ on $V^-$ corresponds to the dual action of $g$ on ${V^+}^*$.

 We fix once and for all a non trivial  additive character $\Sigma$ of  $F$. Consider  Haar measures $dX$ and $dY$ on $V^+$ and $V^-$  respectively.

  \begin{definition} \label{def-Fourier}\index{Fourier transform}The Fourier transform $f\in \cS(V^+)$  is the function $\cF(f)\in \cS(V^-)$ defined by
 $$\cF(f)(Y)\index{Fcal@$\cF$}=\int_{V^+} f(X) \psi(b(X,Y)) dX,\quad Y\in V^-,$$
 and the inverse Fourier transform of  $g\in\cS(V^-)$ is the function $\overline{\cF}(g)\in \cS(V^+)$ defined by
 $$\overline{\cF}(g)(X)\index{Fcalsurligne@$\overline{\cF}(g)$}=\int_{V^-} g(Y) \overline{\psi(b(X,Y))} dY,\quad X\in V^+.$$
   
 Using the same character $\psi$  and the same form $b$, we define similarly:
 
 - Fourier transforms $\cF_{2,0}$ \index{Fcalij@$\cF_{i,j}$}and $\overline\cF_{2,0}$ between $\cS(K_{p}(2,0))$ and $\cS(K_{p}(-2,0))$, 
 
 - Fourier transforms $\cF_{0,2}$ and $\overline\cF_{0,2}$ between $\cS(K_{p}(0,2))$ and $\cS(K_{p}(0,-2))$,
 
 - Fourier transforms $\cF_{1,1}$ and $\overline\cF_{1,1}$ between $\cS(K_{p}(1,1))$ and $\cS(K_{p}(-1,-1))$,
 
 - Fourier transforms $\cF_{-1,1}$ and $\overline\cF_{-1,1}$ between $\cS(K_{p}(-1,1))$ and $\cS(K_{p}(1,-1))$.

 \end{definition}
 
 It is well known that there exists a pair of Haar measures $(dX,dY)$ on $V^+$ and $V^-$ respectively wich are so-called {\it dual } for $\cF$. This means that $\overline{\cF}\circ \cF=Id_{\cS(V^+)}$. Moreover if such a pair $(dX,dY)$ of dual measures is given, any other pair of dual  measures  is of the form $(\alpha dX,\frac{1}{\alpha}dY)$, with $\alpha\in \R^{+ *}$. 
 
 Let us indicate one way to construct such dual measures. Define the Fourier transform on $F$ by
 $$\cF_{1}h(y)= \int_{F}h(x)\psi(xy)dx,\hskip 7pt h\in \cS(F),\, \cF_{1}h\in \cS(F),$$
 and fix once and  for all a self-dual measure (for $\cF_{1}$) $dx$ on $F$.
 
 Let us  choose a basis $(e_{1},\ldots,e_{n})$ of $V^+$ which contains a base of each root space contained in $V^+$. Consider the dual base $(e_{1}^*,\ldots,e_{n}^*)$ of $V^-$ ({\it i.e} $b(e_{i},e_{j}^*)=\delta_{ij}$). Denote by  $X=\sum_{i=1}^{n}x_{i}e_{i}$ (resp. $Y=\sum_{j=1}^{n}y_{j}e_{j}^*$) a general element in $V^+$ (resp. in $V^-$) and define the measures
 $$d_{1}X=dx_{1}\ldots dx_{n}\,\,\, d_{1}Y=dy_{1}\ldots dy_{n},$$
 where $dx_{i},dy_{j}$ are copies of $dx$.
 
 Then the measures $d_{1}X, d_{1}Y$ are dual for $\cF$. Recall that $\i$ is the bijective map from $\Omega^+$ to $\Omega^-$, defined for $X\in \Omega^+$ by the fact that $( \i(X),H_{0},X)$ is an $\go{sl}_{2}$-triple. This map is clearly $G$-equivariant: $\i(g.X)=g.\i(X)$ for $X\in \Omega^+$ and $g\in G$. Moreover this map can also be defined as a map from  $\overline{\Omega^+}$ onto $\overline{\Omega^-}$, which is $\overline{G}$-equivariant, where  $\overline{G}$ is the group which is defined the same way as $G$, but for the extended grading of $\overline{\tilde{\go g}}$. 
 
 \vskip 5 pt
  Let us denote by $K_{p}(\pm 2,0)'$  and  $K_{p}(0,\pm 2)'$ the sets of generic elements  in $K_{p}(\pm 2,0)$ and  $K_{p}(0,\pm 2)$, respectively (as $(L_{p}, K_{p}(\pm 2,0))$ and $(L_{p}, K_{p}(  0,\pm2))$ are prehomogeneous, the notion of generic element is clear).
 \vskip 5pt

 It is worth noticing that similarly one can consider  maps still denoted by  $\iota $ from $K_{p}(2,0)'$ onto $K_{p}(-2,0)'$ and from $K_{p}(0,2)'$ onto $K_{p}(0,-2)'$ defined by the conditions that $(\iota(u), H_{0}^{p,1}, u)$ and  $(\iota(v), H_{0}^{p,2}, v)$ are $\go{sl}_{2}$-triples ($u\in K_{p}(2,0)'$, $v\in K_{p}(0,2)'$).

 We then  normalize the various relative invariants by the conditions:
 $$\nabla_{0}(\i(X))=\frac{1}{\Delta_{0}(X)} \,\, \text{ for all }X\in \Omega^+,$$
 $$\nabla_{k-p}(\i(u))=\frac{1}{\widetilde\Delta_{k-p}(u)} \,\, \text{ for all }u\in K_{p}(2,0)',$$
 $$\widetilde \nabla_{p+1}(\i(v))=\frac{1}{\Delta_{p+1}(v)} \,\, \text{ for all }v\in K_{p}(0,2)'.$$
  It is easy to see, due to the $\overline{G}$-equivariance and the $\cZ_{G}(FH_{0}^{p,1}+FH_{0}^{p,2})$-equivariance,  that these  conditions are equivalent to the same equalities holding only for fixed elements $X, u, v$.

We will also impose the conditions:
$$\Delta_{0}(u+v)=\widetilde \Delta_{k-p}(u)\Delta_{p+1}(v)  \,\, \text{ for all } u\in K_{p}(2,0), \,v\in K_{P}(0,2),$$
 $$\nabla_{0}(u'+v')=  \nabla_{k-p}(u')\widetilde \nabla_{p+1}(v')  \,\, \text{ for all } u'\in K_{p}(-2,0), \,v'\in K_{P}(0,-2).$$

 
 Let $(Y,H_{0},X)$ be an $\go{sl}_{2}$-triple with $X\in V^+$, $Y\in V^-$. Recall from Definition \ref{def-hx} that we have set for $t\in F^*$
 $$\theta_{X}(t)= \exp(t\ad X)\exp(\frac{1}{t}\ad Y) \exp(t\ad X)= \exp(\frac{1}{t}\ad Y) \exp(t\ad X)\exp(\frac{1}{t}\ad Y),$$
 and $\theta_{X}=\theta_{X}(-1)$.
 Then  
  $$\theta_{X}(X)=(-1)^2Y=Y,\, \theta_{X}(Y)=(-1)^2X,\, \theta_{X}(H_{0})=-H_{0}.$$
  
  Moreover $\theta_{X}^2$ acts by $(-1)^r$ on the weight $r$ space for any finite dimensional module.
  Hence $\theta_{X}$ is an involution of $\overline{\tilde{\go g}}$. Is is also easy to see that for $g\in G$ and $X\in \Omega^+$ we have $\theta_{gX}=g\theta_{X}g^{-1}.$
  
  The restriction ${\theta_{X}}_{|_{V^+}}$ is a linear map from $V^+$ to $V^-$. From the preceding identity one obtains
  $$\det({\theta_{gX}}_{|_{V^+}})=(\det g_{|_{V^-}})\det({\theta_{X}}_{|_{V^+}})(\det {g^{-1}}_{|_{V^+}})= (\det g_{|_{V^-}})^2\det({\theta_{X}}_{|_{V^+}}).$$
Hence $\det({\theta_{X}}_{|_{V^+}})$ is a relative invariant whose character is $(\det g_{|_{V^-}})^2$. As we know (\cite{B-R}, p.95) that $\det g_{|_{V^-}}=\chi_{0}(g)^{-m}$, where $m= \frac{\dim V^+}{\text{deg}(\Delta_{0})}= \frac{\dim V^+}{\kappa(k+1)}$, we get that there exists $c\in F^*$ such that 
\beq\label{eq-detthetaX}\det({\theta_{X}}_{|_{V^+}})=c\Delta_{0}(X)^{-2m}. \eeq

Fix $X_{0}\in \Omega^+$, and remark that $\theta_{X_{0}}g\theta_{X_{0}}^{-1}\in G$. Then
$$\nabla_{0}(\theta_{X_{0}}(gX))=\nabla_{0}(\theta_{X_{0}}g\theta_{X_{0}}^{-1}\theta_{X_{0}}(X))=\chi_{0}^{-1}(\theta_{X_{0}}g\theta_{X_{0}}^{-1})\nabla_{0}(\theta_{X_{0}}(X)).$$ 
Therefore the map $X\longmapsto \nabla_{0}(\theta_{X_{0}}(X))$ is a polynomial relative invariant with the same degree as $\nabla_{0}$ (hence the same degree as $\De_0$). Thus $\nabla_{0}(\theta_{X_{0}}(X))=\alpha \Delta_{0}(X)$ with $\alpha\in F^*$. Replacing $X$ by $X_{0}$, and using the normalization  above we obtain
$$\alpha \Delta_{0}(X_0)=\nabla_{0} (\theta_{X_{0}}(X_{0}))=\nabla_{0}(\i(X_0))=\frac{1}{\Delta_{0}(X_0)}.$$
Hence $\displaystyle \alpha=\frac{1}{\Delta_{0}(X_0)^2}$ and 
\beq\label{eq-nablatheta}\nabla_{0}(\theta_{X_{0}}(X))=\frac{\Delta_{0}(X)}{\Delta_{0}(X_{0})^2}\,\,\, (X\in \Omega^+).\eeq

\begin{prop}\label{prop-mesures-1}\hfill

There exists a unique pair of Haar measures $(dX,dY)$ on $V^+\times V^-$ such that 

$1)$ $\overline \cF\circ \cF=Id_{\cS(V^+)}$ (in other words the measures are dual for $\cF$) 

$2)$ $\forall f\in L^1(V^-) $ and $\forall X_{0}\in \Omega^+$
\beq\label{eq-int-V+/V-}\int_{V^+}f(\theta_{X_{0}}(X))dX=|\Delta_{0}(X_{0})|^{2m}\int_{V^-}f(Y) dY.\eeq

In fact if the   identity (\ref{eq-int-V+/V-})   is true for one fixed $X_{0}$, then it holds for all.

Moreover if $d^*X=\frac{dX}{|\Delta_{0}(X)|^m}$ and $d^*Y=\frac{dY}{|\nabla_{0}(Y)|^m}$ denote the corresponding $G$-invariant measures on $V^+$ and $V^-$ respectively, we have for $\Psi\in L^1(\Omega^-, d^*Y)$ and $\Phi \in L^1(\Omega^+,d^*X)$:
$$\int_{\Omega^+}\Psi(\theta_{X_{0}}(X))d^*X=\int_{\Omega^-}\Psi(Y)d^*Y\,\,\,\text{and} \,\,\, \int_{\Omega^-}\Phi(\theta_{X_{0}}(Y))d^*Y=\int_{\Omega^+}\Phi(X)d^*X.$$

\end{prop}

\begin{proof}\hfill

Let us start with the pair of dual Haar measures $(d_{1}X,d_{1}Y)$ described before. Suppose that $(dX,dY)$ is a couple of Haar measures satisfying the condition $1)$. Then there exists $\lambda\in{ \R^*}^+$ such that $dX=\lambda d_{1}X$ and $dY=\frac{1}{\lambda}d_{1}Y$. We make  the change of variable $Y=\theta_{X_{0}}(X)$ in the left hand side of $2)$. From the definition of $d_{1}X$ and $d_{1}Y$ we get $d_{1}Y= |\det({\theta_{X_{0}}}_{|_{V^+}})|d_{1}X$. Hence $$ \begin{array}{rl}
\displaystyle \int_{V^+}f(\theta_{X_{0}}(X))dX &\displaystyle=\int_{V^+}f(\theta_{X_{0}}(X))\lambda d_{1}X=\int_{V^-}f(Y)\frac{1}{|\det({\theta_{X_{0}}}_{|_{V^+}})|}\lambda d_{1}Y\cr
{}&\displaystyle=\int_{V^-}f(Y)\frac{1}{|\det({\theta_{X_{0}}}_{|_{V^+}})|}\lambda^2 dY = \frac{\lambda^2}{|c|}|\Delta_{0}(X_{0})|^{2m}\int_{V^-}f(Y)dY ),\cr
\end{array}$$
where the last equality comes from  (\ref{eq-detthetaX}).

Therefore if we take $\lambda=\sqrt{|c|}= |\det({\theta_{X_{0}}}_{|_{V^+}})|^{\frac{1}{2}}|\Delta_{0}(X_{0})|^{2m}$ we have proved the first part of the statement. And the fact that $c$ does not depend on $X_{0}$ implies that  if (\ref{eq-int-V+/V-}) holds for one $X_{0}$ then it holds for all.

Using (\ref{eq-nablatheta}) and 2) we obtain:

$$ \begin{array}{rl}
\displaystyle\int_{\Omega^+}\Psi(\theta_{X_{0}}(X))d^*X&\displaystyle= \int_{\Omega^+}\Psi(\theta_{X_{0}}(X))\frac{dX}{|\Delta_{0}(X)|^m}\cr
{} &=\displaystyle |\Delta_{0}(X_{0})|^{-2m}\int_{V^+}\Psi(\theta_{X_{0}}(X))\frac{dX}{|\nabla_{0}(\theta_{X_{0}}(X))|^m}\cr
{}&\displaystyle=\int_{V^-}\Psi(Y)d^*Y.

\end{array}$$

Setting $\Phi=\Psi\circ\theta_{X_{0}}$ in this identity, we obtain the last assertion.
\end{proof}




\begin{prop}\label{prop-mesures-2}\hfill

Fix $X_{0}\in \Omega^+$ and let $Y_{0}=\i(X_{0})$. Define $H=\cZ_{G}(X_{0})=\cZ_{G}(Y_{0})$.
One defines two $G$-invariant measures $d_{1}\dot g$ and  $d_{2}\dot g$ on $G/H$ by setting for $\Phi\in \cS(G.X_{0})$ and $\Psi\in\cS(G.Y_{0}) $
$$\int_{G.X_{0}}\Phi(X)d^*X=\int_{G/H}\Phi(g.X_{0})d_{1}\dot g\,\,\, \text{and}\,\,\,\int_{G.Y_{0}}\Psi(Y)d^*Y=\int_{G/H}\Psi(g.Y_{0})d_{2}\dot g.$$

Then $d_{1}\dot g= d_{2}\dot g$.
\end{prop}

\begin{proof}\hfill

From the last assertion of Proposition \ref{prop-mesures-1} we obtain 

$\int_{G/H}\Phi(g.X_{0})d_{1}\dot g=\int_{\Omega^+}\Phi(X)d^*X=\int_{\Omega^-}\Phi(\theta_{X_{0}}(Y))d^*Y =\int_{G/H}\Phi(\theta_{X_{0}}(g.Y_{0}))d_{2}\dot g $

$=\int_{G/H}\Phi(\theta_{X_{0}}g\theta_{X_{0}}\theta_{X_{0}}(Y_{0}))d_{2}\dot g=\int_{G/H}\Phi(g.X_{0})d_{2}\dot g$

(The last equality is obtained by taking into account that the conjugation  $g\longmapsto \theta_{X_{0}}g\theta_{X_{0}}$ is an involution of $G$ which stabilizes $H$, so we make  the change of variable $g=\theta_{X_{0}}g\theta_{X_{0}}$).
\end{proof}

\begin{definition} We will denote by $d\dot g$ \index{ddotg@$d\dot g$}the $G$-invariant measure appearing in Proposition \ref{prop-mesures-2}.

\end{definition}

\begin{prop}\label{prop-mesures-3}\hfill

Let $\sigma$ be an involution of $\tilde{\go{g}}$ such that $\sigma(H_{0})=-H_{0}$. We also suppose  that there exists $X_{0}\in V^+$ such that $(\sigma(X_{0}),H_{0}, X_{0})$ is an $\go{sl}_{2}$-triple. Then for $\Psi\in L^1(\Omega^-,d^*Y)$ we have
$$\int_{\Omega^+}\Psi(\sigma(X))d^*X= \int_{\Omega^-}\Psi(Y)d^*Y.$$
\end{prop}

\begin{proof}\hfill

One should note that if $\sigma$ is such an involution, then $\sigma(\go{g})=\go{g}$ and $\sigma(V^+)=V^-$. From Proposition \ref {prop-mesures-1} we obtain
$$\int_{\Omega^+}\Psi(\sigma(X))d^*X= \int_{\Omega^+}\Psi\circ \sigma\theta_{X_{0}}(\theta_{X_{0}}(X))d^*X=\int_{\Omega^-}\Psi(\sigma\theta_{X_{0}}(Y))d^*Y.$$

It is easy to see that if we set $Y_{0}=\sigma(X_{0}$) we have 
$$\theta_{X_{0}}\circ \sigma= \sigma\circ \exp(-\ad(Y_{0}))\exp(-\ad(X_{0}))\exp(-\ad(Y_{0}))=\sigma\circ\theta_{X_{0}}.$$

Therefore $\sigma\circ\theta_{X_{0}}$ is an involution preserving $V^-$. Hence if we make the change of variable 
$Y=\sigma\theta_{X_{0}}(Y)$ we obtain the assertion. \end{proof}

\begin{prop}\label{prop-mesures-4}\hfill

Remember that for $X\in \Omega^+$, the element $\i(X)\in \Omega^-$ is defined by the fact that $(\i(X),H_{0},X)$ is an $\go{sl}_{2}$-triple.
For $\Psi\in L^1(\Omega^-,d^*Y)$ we have
$$\int_{\Omega^+}\Psi(\i(X))d^*X= \int_{V^-}\Psi(Y)d^*Y.$$

\end{prop}

\begin{proof}\hfill

Let $\Omega^+=\cup_{i=1}^{r_0} \Omega_{i}^+$ be the decomposition of $\Omega^+$ into open $G$-orbits. Let $\Omega_{i}^-=\iota(\Omega_{i}^+)$. Take $X_{i}\in \Omega_{i}^+$ and define $Y_{i}=\iota(X_{i})\in \Omega_{i}^-$. Define also $\cH_{i}=\cZ_{G}(X_{i})=\cZ_{G}(Y_{i})$. It is enough to prove the assertion for $\Psi\in \cS(\Omega_{i}^-)$. Using Proposition \ref{prop-mesures-2} we obtain:

$\int_{\Omega_{i}^+}\Psi(\i(X))d^*X=\int_{G/\cH_{i}}\Psi(\i(gX_{i}))d\dot g=\int_{G/\cH_{i}}\Psi(g\i(X_{i}))d\dot g=\int_{G/\cH_{i}}\Psi(gY_{i})d\dot g=\int_{\Omega_{i}^-}\Psi(Y)d^*Y.$
 \end{proof}

We now consider the subset  $(e_{{i}_{1}},\ldots,e_{i_{r}}) $ of $(e_{1},\ldots,e_{n})$ which is a basis of $K_{p}(1,1)$ (see the discussion after definition \ref{def-Fourier}). Then $(e_{{i}_{1}}^*,\ldots,e_{i_{r}}^*) $ is the dual basis of $K_{p}(-1,-1)$ with respect to the form $b$.
If $d_{1}w$ and $d_{1}w'$ are the Haar measures  on $K_{p}(1,1)$ and $K_{p}(-1,-1)$ respectively, defined through the preceding basis and the choice of a self dual measure on $F$, then these measures are dual for $\cF_{1,1}$.

\begin{prop}\label{prop-mesures-5} \hfill

There exists a unique pair of Haar measures $(dw,dw')$ \index{dw@$(dw,dw')$} on $K_{p}(1,1)\times K_{p}(-1,-1)$ such that 

$1)$ $\overline{ \cF}_{1,1}\circ \cF_{1,1}=Id_{\cS(K_{p}(1,1)}$ (in other words the measures are dual for $\cF_{1,1}$) 

$2)$ $\forall g\in L^1(K_{p}(1,1)) $ and $\forall u\in K_{p}(2,0)'$ and  $\forall v\in K_{p}(0,2)'$
$$\int_{K_{p}(1,1)}g(\theta_{u+v} (w))dw=|\Delta_{p+1}(v)|^{\frac{\dim(K_{p}(1,1))}{\deg(\Delta_{p+1})}}|\widetilde{\Delta}_{k-p}(u)|^{\frac{\dim(K_{p}(1,1))}{\deg(\widetilde\Delta_{k-p})}} \int_{K_{p}(-1,-1)}g(w')dw'.$$

As ${{\theta_{u+v}}^2_{|_{K_{p}(1,1)}}}=Id_{K_{p}(-1,-1)}$ it is easy to see that condition $2)$ is equivalent to the following condition $2')$.

$2')$ $\forall f\in L^1(K_{p}(-1,-1)) $ and $\forall u\in K_{p}(2,0)'$ and  $\forall v\in K_{p}(0,2)'$
$$\int_{K_{p}(1,1)}f(w)dw=|\Delta_{p+1}(v)|^{\frac{\dim(K_{p}(1,1))}{\deg(\Delta_{p+1})}}|\widetilde{\Delta}_{k-p}(u)|^{\frac{\dim(K_{p}(1,1))}{\deg(\widetilde\Delta_{k-p})}} \int_{K_{p}(-1,-1)}f(\theta_{u+v}(w'))dw'.$$

\end{prop}

\begin{proof}\hfill

For $u,v$ as in $2)$, the map 
$\theta_{u+v} $ is a linear isomorphism from $K_{p}(1,1)$ onto $K_{p}(-1,-1)$ and for $g\in \cZ_{G}(FH_{0}^{p,1}+FH_{0}^{p,2})$ we have:
 $$  \det({\theta_{g(u+v)}}_{|_{K_{p}(1,1)}}) =  \det(g_{|_{K_{p}(-1,-1)}})^2 \det({\theta_{u+v}}_{|_{K_{p}(1,1)}}).$$ 

Therefore $\det({\theta_{u+v}}_{|_{K_{p}(1,1)}})$ is a relative invariant of the prehomogeneous vector space $(K_{p}(0,0),K_{p}(2,0)\oplus K_{p}(0,2)$.  Hence there exist $c\in F^*$ and $\alpha,\beta \in \Z$ such that 

\beq\label{eq-thetau+v}\det({\theta_{u+v}}_{|_{K_{p}(1,1)}})=c\widetilde  \Delta_{k-p}(u)^\alpha\Delta_{p+1}(v)^\beta.\eeq
For $t_{1},t_{2}\in F$, let us consider the element $g_{0}=h_{u}(t_{1})h_{v}(t_{2})\in \cZ_{G}(FH_{0}^{p,1}+FH_{0}^{p,2})$ (see Definition \ref{def-hx}). This element acts by $ t_{2}^{-1}t_{1}^{-1}$ on $K_{p}(-1,-1)$, by $t_{1}^2$ on $K_{p}(2,0)$ and by $t_{2}^2$ on $K_{p}(0,2)$ (see Remark \ref{rem-hx}). Hence
$$\det({g_{0}}_{|_{K_{p}(-1,-1)}})^2= t_{2}^{-2\dim K_{p}(-1,-1)}t_{1}^{-2\dim K_{p}(-1,-1)}$$
   But from (\ref{eq-thetau+v}) above we also have:
   $$\det({\theta_{g_{0}(u+v)}}_{|_{K_{p}(1,1)}})=ct_{1}^{2\alpha \deg \widetilde \Delta_{k-p} }t_{2}^{2\beta \deg \Delta_{p+1}}\det({\theta_{ (u+v)}}_{|_{K_{p}(1,1)}}).$$
   Therefore $\alpha=-\frac{\dim K_{p}(1,1)}{\deg \widetilde\Delta_{k-p}} \,\,\, \text{ and } \beta=-\frac{\dim K_{p}(1,1)}{\deg \Delta_{p+1}}$. Hence
\beq\label{thetau+vprecis}\det({\theta_{u+v}}_{|_{K_{p}(1,1)}})=c\widetilde  \Delta_{k-p}(u)^{-\frac{\dim K_{p}(1,1)}{\deg \widetilde\Delta_{k-p}}}\Delta_{p+1}(v)^ {-\frac{\dim K_{p}(1,1)}{\deg \Delta_{p+1}}}.\eeq   
Let us start with the pair of dual Haar measures $(d_{1}w,d_{1}w')$ described before. Suppose that $(dw,dw')$ is a couple of Haar measures satisfying the condition $1)$. Then there exists $\lambda\in{ \R^*}^+$ such that $dw=\lambda d_{1}w$ and $dw'=\frac{1}{\lambda}d_{1}w'$. We make  the change of variable $w'=\theta_{u+v}(w)$ in the left hand side of $2)$. From the definition of $d_{1}w$ and $d_{1}w'$ we get $d_{1}w'= |\det({\theta_{u+v}}_{|_{K_{p}(1,1)}}) |d_{1}w$. Hence

$$ \begin{array}{rl}
\displaystyle \int_{K_{p}(1,1)}g(\theta_{u+v}(w))dw &\displaystyle=\int_{K_{p}(1,1)}g(\theta_{u+v}(w))\lambda d_{1}w=\int_{K_{p}(-1,-1)}g(w')\frac{1}{  |\det({\theta_{u+v}}_{|_{K_{p}(1,1)}} |}\lambda d_{1}w'\cr
{}&\displaystyle=\int_{K_{p}(-1,-1)}g(w')\frac{1}{  |\det({\theta_{u+v}}_{|_{K_{p}(1,1)}}) |}\lambda^2 d w' \cr
{}&\displaystyle= \frac{\lambda^2}{|c|}|\widetilde  \Delta_{k-p}(u)|^{\frac{\dim K_{p}(1,1)}{\deg \widetilde\Delta_{k-p}}}|\Delta_{p+1}(v)|^ {\frac{\dim K_{p}(1,1)}{\deg \Delta_{p+1}}}\int_{K_p(-1,-1)}g(w')dw',
\end{array}$$
where the last equality comes from (\ref{thetau+vprecis}).
Therefore if we take $\lambda=\sqrt{|c|}$ the result is proved.

 \end{proof}

As above we consider  Haar measures $d_{1}A$ and $d_{1}A'$      on $K_{p}(-1,1)$ and $K_{p}(1,-1)$ respectively, which are dual for $\cF_{-1,1}$ (see Definition \ref{def-Fourier}).These measures are defined through the choice of a basis of $ K_{p}(-1,1) \subset V^+$, the dual basis of $ K_{p}(1,-1) \subset V^-$ (via the form $b$) and the choice of a self dual measure on $F$.

\begin{prop}\label{prop-mesures-6} \hfill

There exists a unique pair of Haar measures $(dA,dA')$ \index{dA@$(dA,dA')$} on $K_{p}(-1,1)\times K_{p}(1,-1)$ such that 

$1)$ $\overline{ \cF}_{-1,1}\circ \cF_{-1,1}=Id_{\cS(K_{p}(-1,1)}$ (in other words the measures are dual for $\cF_{-1,1}$) 

$2)$ $\forall g\in L^1(K_{p}(1,-1)) $ and $\forall u\in K_{p}(2,0)'$ and  $\forall v\in K_{p}(0,2)'$
$$\int_{K_{p}(-1,1)}g(\theta_{u+v} (A))dA=|\Delta_{p+1}(v)|^{\frac{\dim(K_{p}(1,-1))}{\deg(\Delta_{p+1})}}|\widetilde{\Delta}_{k-p}(u)|^{-\frac{\dim(K_{p}(1,-1))}{\deg(\widetilde\Delta_{k-p})}} \int_{K_{p}(1,-1)}g(A')dA'.$$
As ${{\theta_{u+v}}^2_{|_{K_{p}(-1,1)}}}=Id_{K_{p}(-1,1)}$ it is easy to see that condition is equivalent to the following condition $2')$.

$2')$ $\forall f\in L^1(K_{p}(-1,1)) $ and $\forall u\in K_{p}(2,0)'$ and  $\forall v\in K_{p}(0,2)'$
$$\int_{K_{p}(-1,1)}f(A)dA=|\Delta_{p+1}(v)|^{\frac{\dim(K_{p}(1,-1))}{\deg(\Delta_{p+1})}}|\widetilde{\Delta}_{k-p}(u)|^{-\frac{\dim(K_{p}(1,-1))}{\deg(\widetilde\Delta_{k-p})}} \int_{K_{p}(1,-1)}f(\theta_{u+v}(A'))dA'.$$

\end{prop}

\begin{proof}\hfill

The proof is almost the same as for the preceding Proposition. The only change comes from the fact that the element $g_{0}=h_{u}(t_{1})h_{v}(t_{2})$ acts by $t_{1}t_{2}^{-1}$ on $K_{p}(1,-1)$. 
\end{proof}

\begin{rem} \label{rem-mesures}

It is clear that Proposition \ref{prop-mesures-1} applies   to the graded algebras $K_{p}(-2,0)\oplus K_{p}(0,0)\oplus K_{p}(2,0)$ and $K_{p}(0,-2)\oplus K_{p}(0,0)\oplus K_{p}(0,2)$. 

Therefore  there exists a unique pair of measures $(du,du')$ \index{du@$(du,du')$}on $K_{p}(2,0)\times K_{p}(-2,0)$ such that 

1) These measures are dual for the Fourier transform $\cF_{2,0}$ between $\cS(K_{p}(2,0))$ and $\cS(K_{p}(-2,0))$,

2)  $\forall f\in L^1(K_{p}(-2,0) $ and $\forall u_{0}\in  K_{p}(2,0)'$
$$\int_{K_{p}(2,0)}f(\theta_{u_{0}}(u))du=|\widetilde\Delta_{k-p}(u_{0})|^{\frac{2\dim K_{p}(2,0)}{\deg \widetilde \Delta_{k-p}}}\int_{K_{p}(-2,0)}f(u')du'.$$

Similarly there exists a unique pair of measures $(dv,dv')$ \index{dv@$(dv,dv')$} on $K_{p}(0,2)\times K_{p}(0,-2)$ such that 

1) These measures are dual for the Fourier transform $\cF_{0,2}$ between $\cS(K_{p}(0,2))$ and $\cS(K_{p}(0,-2))$  

2)  $\forall f\in L^1(K_{p}(0,-2) $ and $\forall v_{0}\in  K_{p}(0,2)'$
$$\int_{K_{p}(0,2)}f(\theta_{v_{0}}(v))dv=|\Delta_{p+1}(v_{0})|^{ \frac{2\dim K_{p}(0,2)}{\deg   \Delta_{p+1}}}\int_{K_{p}(0,-2)}f(v')dv'.$$

  \end{rem}
  
  \begin{definition}\label{def-normalisation} {\bf Normalization of the measures}\hfill

  From now on we will always make the following choice of measures:
 
  a) On $V^+\times V^-$ we take the measures $(dX,dY)$ which were defined in Proposition \ref{prop-mesures-1}.
  
  b) On $K_{p}(2,0)\times K_{p}(-2,0)$ and on $K_{p}(0,2)\times K_{p}(0,-2)$ respectively, we take the pairs measures $(du,du')$ and $(dv,dv')$ which were defined in Remark \ref{rem-mesures}.
  
  c) On $K_{p}(1,1)\times K_{p}(-1,-1)$ we take the pair of measures $(dw,dw')$ which were defined in Proposition \ref{prop-mesures-5}.
  
  d) On $K_{p}(-1,1)\times K_{p}(1,-1)$ we take the pair of measures $(dA,dA')$ which were defined in Proposition \ref{prop-mesures-6}.

  \end{definition}
  
  \begin{prop}\label{prop-decomposition-dX}\hfill
  
    a) $dX=dudvdw $

   b)  $dY=du'dv'dw'$
  \end{prop}
  
  \begin{proof}\hfill
  
    Define $dX\coloneqq dudvdw$ and $dY\coloneqq du'dv'dw'$. It is now enough to prove that the pair $(dX,dY)$ satisfies the two conditions of Proposition \ref{prop-mesures-1}.
    
    1) It is clear from the definitions that the measures $dX,dY$ are dual for $\cF$.  
    
    2) Let $f\in \cS(V^+)$, $u_{0}\in K_{p}(2,0)'$, $v_{0}\in K_{p}(0,2)'$. Then we have, using Remark \ref{rem-mesures} and Proposition \ref{prop-mesures-5}:
    \vskip 5pt
    $\begin{array}{rl}
   {}& \displaystyle \int_{V^+}f(\theta_{u_{0}+v_{0}}(u+v+w))dudvdw=\int_{V^+}f(\theta_{u_{0}}(u)+\theta_{v_{0}}(v)+\theta_{u_{0}+v_{0}}(w))dudvdw\cr
    {}&\displaystyle =|\widetilde \Delta_{k-p}(u_{0})|^{\frac{2\dim K_{p}(2,0)+\dim K_{p}(1,1)}{\deg \widetilde \Delta_{k-p}}}|\Delta_{p+1}(v_{0})|^{\frac{2\dim K_{p}(0,2)+\dim K_{p}(1,1)}{\deg   \Delta_{p+1}}}\int_{V^-} f(u'+v'+w')du'dv'dw'.
    \end{array}$
   \vskip 5pt
   
    From Proposition \ref{propdimV}   we know that $\dim V^+=(k+1)(\ell +\frac{kd}{2})$, $\dim K_{p}(2,0)=(p+1)(\ell +\frac{pd}{2})$, $\dim K_{p}(0,2)=(k-p)(\ell +\frac{(k-p-1)d}{2})$. Moreover as $K_{p}(1,1)=\oplus_{0\leq i\leq p; p+1\leq j\leq k} E_{i,j}(1,1)$, we obtain that $\dim K_{p}(1,1)=(p+1)(k-p)d$. We also have $\deg \widetilde \Delta_{k-p}=\kappa(p+1) $ and $\deg \Delta_{p+1}=\kappa(k-p)$.
   A simple calculation shows then that 
   $${\frac{2\dim K_{p}(2,0)+\dim K_{p}(1,1)}{\deg \widetilde \Delta_{k-p}}}={\frac{2\dim K_{p}(0,2)+\dim K_{p}(1,1)}{\deg   \Delta_{p+1}}}=\frac{2}{\kappa}(\ell +\frac{kd}{2})=2m.$$
   Finally we get
   $$ \int_{V^+}f(\theta_{u_{0}+v_{0}}(u+v+w))dudvdw=|\Delta_{0}(u_{0}+v_{0})|^{2m} \int_{V^-}f( u'+v'+w')du'dv'dw'$$
   
   \end{proof}


\subsection{Mean functions}\label{sec:section-MeanFunctions}
\hfill

 We will first define the mean functions $T^+_{f}(u,v)$ ($u\in K_{p}(2,0), v\in K_{p}(0,2)$) and $T^-_{f}(u',v')$ ($u'\in K_{p}(-2,0), v\in K_{p}(0,-2)$) which were introduced by I. Muller (\cite {Mu} Definition 4.3.1 p. 83) and used in the real case by N. Bopp and H. Rubenthaler (\cite{B-R} p.104).   \vskip 5pt
  Recall that $K_{p}(\pm 2,0)'$  and  $K_{p}(0,\pm 2)'$ are the sets of generic elements  in $K_{p}(\pm 2,0)$ and  $K_{p}(0,\pm 2)$, respectively.
 \vskip 5pt
  In the rest of the section we adopt the following convention concerning the variables:
  
   $u\in K_{p}(2,0)$, $v\in K_{p}(0,2)$, $u'\in K_{p}(-2,0)$, $v'\in K_{p}(0,-2)$, $A\in K_{p}(-1,1)$, $A'\in K_{p}(1,-1)$, $w\in K_{p}(1,1)$, $w'\in K_{p}(-1,-1)$, $X\in V^+$, $Y\in V^-$.

 \begin{lemme} \label{lem-formule-int}\hfill
 
 There exist positive constants $\al,\be,\ga$ and $\de$ which depend on the choice of Lebesgue measures on the different spaces such that the following assertions hold.
 
 1)  If  $v\in K_{p}(0,2)'$, then  for all $f\in L^1(K_{p}(1,1))$ one has :
 $$\int_{K_{p}(1,1)}f(w)dw=\alpha|\Delta_{p+1}(v)|^{\frac{\dim K_{p}(1,-1)}{ \deg \Delta_{p+1}}}\int_{K_{p}(1,-1)}f([A',v])dA'.$$
 
 2)  If $v\in K_{p}(0,2)'$, then  for all $g\in L^1(K_{p}(-1,1))$ one has :
 
 $$\int_{K_{p}(-1,1)}g(A)dA=\beta|\Delta_{p+1}(v)|^{\frac{\dim K_{p}(1,-1)}{ \deg \Delta_{p+1}}}\int_{K_{p}(-1,-1)}g([w',v])dw' .$$
 
   3) If $u'\in K_{p}(-2,0)'$, then  for all $f\in L^1(K_{p}(-1,-1))$ one has :
 $$\int_{K_{p}(-1,-1)}f(w')dw'=\gamma|\nabla_{k-p}(u')|^{\frac{\dim K_{p}(1,-1)}{\deg \nabla_{k-p}}}\int_{K_{p}(1,-1)}f([A',u'])dA'.$$ 
 
  4)   If $u'\in K_{p}(-2,0)'$, then  for all $g\in L^1(K_{p}(-1,1))$ one has :
 $$\int_{K_{p}(-1,1)}g(A)dA=\delta|\nabla_{k-p}(u')|^{\frac{\dim K_{p}(1,-1)}{\deg \nabla_{k-p}}}\int_{K_{p}(1,1)}g([w,u'])dw.  $$
 
 (Remember that $\dim K_{p}(1,-1)=\dim K_{p}(-1,1)=\dim K_{p}(1,1)=\dim K_{p}(-1,-1)$).

 \end{lemme}

\begin{proof}: 

We only prove  assertion 1).  The proofs of 2), 3) and 4) are similar.

If $v \in K_{p}(0,2)'$ the map :
$$\ad(v):K_{p}(1,-1)\longrightarrow K_{p}(1,1)$$
is an isomorphism  (classical result for  $\go{sl}_{2}$-modules). Therefore there exists a non zero positive constant  $\alpha(v)$ such that 
$$\int_{K_{p}(1,1)}f(w)dw=\alpha(v)\int_{K_{p}(1,-1)}f([A',v])dA'  $$
On the other hand the subspace  $K_{p}(1,-1)$ is stable under   $L_p={\cal Z}_{G}(F.H_{0,p}^1\oplus F.H_{0,p}^2)$. This implies that there exists  a character $\chi_{_{1}}$ of  $L_p$ such that for  $F\in L^1(K_{p}(1,-1))$ and  $g\in L_p$ one has:
$$\int_{K_{p}(1,-1)}F(gA')dA'=\chi_{1}(g)\int_{K_{p}(1,-1)}F(A')dA'.$$
Moreover, the set $K_{p}(0,2)'$   is also  $L_p$-invariant.
Hence:
$$ \begin{array}{rll}
 \displaystyle\int_{K_{p}(1,1)}f(w)dw&=& \displaystyle\alpha(gv)\int_{K_{p}(1,-1)}f([A',gv])dA'\cr
{}&=& \displaystyle\alpha(gv)\int_{K_{p}(1,-1)}f(g[g^{-1}A',v])dA'\cr
{}&=& \displaystyle\alpha(gv)\chi_{_{1}}^{-1}(g)\int_{K_{p}(1,-1)}f(g[A',v])dA'\cr
{}&=& \displaystyle\alpha(gv)\chi_{_{1}}^{-1} (g)\frac{1}{\alpha(v)}\int_{K_{p}(1,1)}f(g.w)dw.\cr
\end{array}$$

But there is also a character  $\chi_{_{2}}$ of  $L_p$ such that 
$$\int_{K_{p}(1,1)}f(g.w)dw= \chi_{_{2}}(g)\int_{K_{p}(1,1)}f(w)dw.$$
Hence
$$\alpha(gv)=\chi_{_{1}} (g) \chi_{_{2}}(g)^{-1}\alpha(v). $$

This means that $\alpha(v)$ is a positive relative invariant on $K_{p}(0,2)$ under the action of  $L_p$, hence there exist $n\in \Z$ and $\alpha\in{ \R^*}^+$ such that
$$\alpha(v)=\alpha|\Delta_{p+1}(v)|^{n}.$$

 To compute $n$, let us take $g=  h_{I_{p,2}^+}(t)$ as defined in  Definition \ref{def-hx}. Then $gv=t^{2}v$.

Thus, we have
 $$\begin{array}{rll}
 \displaystyle\int_{K_{p}(1,1)}f(w)dw&= &\displaystyle\alpha(v) \int_{K_{p}(1,-1)}f([A', v])dA'\cr
 {}&= &\displaystyle\alpha(t^{2}v)\int_{K_{p}(1,-1)}f([A',t^{2}v])dA'\cr{}&=&\displaystyle\alpha(t^{2}v)\int_{K_{p}(1,-1)}f([t^{2}A',v])dA'\cr
{}&=&\displaystyle\alpha(t^{2}v)t^{-2\dim(K_{p}(1,-1))}\int_{K_{p}(1,-1)}f([A',v])dA'\cr
 \end{array}$$
 Hence $\alpha(t^{2}v)=t^{2\dim(K_{p}(1,-1))}\alpha(v)$. 
We obtain then: 
 $$\alpha|\Delta_{p+1}(t^{2}v)|^{n}=\alpha t^{2n \deg \Delta_{p+1}}|\Delta_{p+1}(v)|^{n} = \alpha t^{2\dim K_{p}(1,-1)}|\Delta_{p+1}(v)|^{n}.$$
Therefore $n=\frac{\dim K_{p}(1,-1)}{ \deg \Delta_{p+1}}$ and  $\alpha(v)=\alpha| \Delta_{p+1}(v)|^{\frac{\dim K_{p}(1,-1)}{ \deg \Delta_{p+1}}}$

This ends the proof of assertion 1).

\end{proof}
\begin{prop} \label{Prop-constantes=1}\hfill

If the measures are normalized as in Definition \ref{def-normalisation}, then the four constants $\alpha,\beta,\delta,\gamma$ occuring in Lemma \ref{lem-formule-int} are equal to $1$.

\end{prop}
\begin{proof}\hfill

Let $u\in K_{P}(2,0)'$, $v\in K_{p}(0,2)'$, and define $u'=\iota(u)$, $v'=\iota(v)$. If  $A\in K_{p}(1,-1)$, then $\ad u A=\ad v' A=0$ and hence $\ad v\,\ad u' A= \ad u' \,\ad v A=\ad u' \ad(u+v) A=\ad [u',u+v]A +\ad(u+v)\ad u'A=A+ \ad(u+v)\ad u' A=A+ \ad(u+v)\ad (u'+v') A=\theta_{u+v}(A)$ (for the last equality see Theorem \ref{th-invol}). Hence we have 
\beq\label{eq-theta-1}\forall  A\in K_{P}(1,-1),\,\,\ad v\,\ad u' A= \ad u' \,\ad v A=\theta_{u+v}(A).\eeq

It is equivalent to say  that the following diagram is commutative:

$$\xymatrix{
&K_{p}(-1,1)& \\
K_{p}(-1,-1)\ar[ur]^*+{\ad v}&&K_{p}(1,1)\ar[ul]_*+{\ad u'}\\
&\ar[ul]^*+{\ad u'}K_{p}(1,-1)\ar[uu]|*+{\theta_{u+v}}\ar[ur]_*+{\ad v}&\\
}$$

Take now $w'\in K_{p}(-1,-1)$. Then (recall that  $\te_{u+v}(v)=v'$)

$\ad \theta_{u+v}(v) \ad \, v \, w'=\ad [\theta_{u+v}(v),v] \,w'+\ad\,v\, \ad\theta_{u+v}(v) w'=-w'+\ad\,v\,\ad\theta_{u+v}(v) w'=-w'.$

Hence \begin{equation}\label{eq-theta-2}
  \ad v\circ\theta_{u+v}\circ\ad v\, w'=\theta_{u+v}\circ\theta_{u+v}\circ\ad v\circ\theta_{u+v}\circ \ad v\, w'=\theta_{u+v}\circ\ad \theta_{u+v}(v)\circ\ad v\, w'  =-\theta_{u+v}(w'). 
\end{equation}
Similarly one can also prove that for all $w\in K_{p}(1,1)$:
\begin{equation}\label{eq-theta-3}
\ad u'\circ\theta_{u+v}\circ\ad u' w=-\theta_{u+v}(w)
\end{equation}

-- { \bf Step 1}. We will first prove that $\alpha\delta=\alpha\beta=\beta\gamma=\gamma\delta=1$.

Let $u\in K_{p}(2,0)'$ and $v\in K_{p}(0,2)'$.   

Using successively relation $4)$ and $1)$ in Lemma \ref{lem-formule-int}, the relation $\nabla_{k-p}(u')=\widetilde\Delta_{k-p}(u)^{-1}$ and equation (\ref{eq-theta-1}) we obtain:
$$\begin{array}{rl}
\displaystyle \int_{K_{p}(-1,1)}g(A)dA&\displaystyle =\delta|\nabla_{k-p}(u')|^{\frac{\dim K_{p}(1,-1)}{\deg \nabla_{k-p}}}\int_{K_{p}(1,1)}g(-\ad u' w)dw \cr
{}&\displaystyle =\alpha\delta|\nabla_{k-p}(u')|^{\frac{\dim K_{p}(1,-1)}{\deg \nabla_{k-p}}} |\Delta_{p+1}(v)|^{\frac{\dim K_{p}(1,-1)}{ \deg \Delta_{p+1}}}\int_{K_{p}(1,-1)}g(\ad u' \ad v A')dA'\cr
{}&\displaystyle =\alpha\delta|\widetilde\Delta_{k-p}(u)|^{-\frac{\dim K_{p}(1,-1)}{\deg \nabla_{k-p}}} |\Delta_{p+1}(v)|^{\frac{\dim K_{p}(1,-1)}{ \deg \Delta_{p+1}}}\int_{K_{p}(1,-1)}g(\ad u' \ad v A')dA'\cr
{}&\displaystyle =\alpha\delta|\widetilde\Delta_{k-p}(u)|^{-\frac{\dim K_{p}(1,-1)}{\deg \nabla_{k-p}}} |\Delta_{p+1}(v)|^{\frac{\dim K_{p}(1,-1)}{ \deg \Delta_{p+1}}}\int_{K_{p}(1,-1)}g(\theta_{u+v}(A'))dA'

 \end{array}$$
 Then Proposition \ref{prop-mesures-6} implies that $\alpha\delta=1$.
 
 Using successively relation $1)$ in Lemma \ref{lem-formule-int}, Proposition \ref{prop-mesures-6}, equation  $2)$ in Lemma \ref{lem-formule-int}, and finally equation \eqref{eq-theta-2} above we obtain:
 
 $$\begin{array}{rl}
\displaystyle \int_{K_{p}(1,1)}f(w)dw&\displaystyle =\alpha|\Delta_{p+1}(v)|^{\frac{\dim K_{p}(1,-1)}{\Delta_{p+1}}}\int_{K_{p}(1,-1)}f([A',v])dA' \cr
{}&\displaystyle =\alpha |\Delta_{p+1}(v)|^{\frac{\dim K_{p}(1,1)}{\Delta_{p+1}}} |\Delta_{p+1}(v)|^{-\frac{\dim K_{p}(1,1)}{\Delta_{p+1}}}|\widetilde \Delta_{k-p}(u)|^{\frac{\dim K_{p}(1,1)}{\deg \widetilde \Delta_{k-p}}}\int_{K_{p}(-1,1)}f([\theta_{u+v}(A),v])dA\cr
{}&\displaystyle =\alpha\beta|\widetilde \Delta_{k-p}(u)|^{\frac{\dim K_{p}(1,1)}{\deg \widetilde \Delta_{k-p}}} |\Delta_{p+1}(v)|^{\frac{\dim K_{p}(1,-1)}{ \deg \Delta_{p+1}}}\int_{K_{p}(-1,-1)}f([\theta_{u+v}([w',v]),v])dw'\cr
{}&\displaystyle =\alpha\beta|\widetilde \Delta_{k-p}(u)|^{\frac{\dim K_{p}(1,1)}{\deg \widetilde \Delta_{k-p}}} |\Delta_{p+1}(v)|^{\frac{\dim K_{p}(1,-1)}{ \deg \Delta_{p+1}}}\int_{K_{p}(-1,-1)}f(-\theta_{u+v}(w'))dw'\cr

 \end{array}$$
  Then Proposition \ref{prop-mesures-5} implies that $\alpha\beta=1.$
 
 The proofs that $\beta\gamma=1$ and $\gamma\delta=1$ are similar.
 
{\bf - Step 2} We will now prove that  $\alpha \gamma =1$. From Step 1, this will imply  that $\alpha=\beta=\gamma=\delta=1$.

Let $f\in \cS(K_{p}(1,1)$. Then by Lemma \ref{lem-formule-int}, 3), for $w\in K_{p}(1,1)$

\beq\label{eq-f-gamma}\begin{array}{rl}  
\displaystyle f(w)&\displaystyle= (\overline \cF_{1,1}\circ \cF_{1,1}  f)(w)= \int_{K_{p}(-1,-1)}\cF_{1,1}  f(w')\overline{\psi(b(w',w))} dw'\cr
 \displaystyle{}&\displaystyle= \gamma|\nabla_{k-p}(u')|^{\frac{\dim K_{p}(1,-1)}{\deg \nabla_{k-p}}}\int_{K_{p}(1,-1)}\cF_{1,1}  f([A',u'])\overline{\psi(b([A',u'],w))} dA'. 
  \end{array}\eeq
  On the other hand, using Lemma \ref{lem-formule-int}, 1) we have 
  $$\begin{array}{rl}  
 \displaystyle\cF_{1,1}(w')&\displaystyle=   \int_{K_{p}(1,1)}   f(w)\psi(b(w,w')) dw\cr
 \displaystyle{}&\displaystyle=\alpha|\Delta_{p+1}(v)|^{\frac{\dim K_{p}(1,-1)}{\deg \Delta_{p+1}}}\int_{K_{p}(1,-1)} f([A',v])\psi(b(A',[v,w'])) dA'.
  \end{array}$$
  If we define $\varphi(A')=f([A',v])$ the preceding equation becomes:
  
  $$   
 \displaystyle\cF_{1,1}(w') \displaystyle=\alpha|\Delta_{p+1}(v)|^{\frac{\dim K_{p}(1,-1)}{\deg \Delta_{p+1}}}\overline \cF_{1,-1}\varphi(-[v,w'])
  \,\,\,\, \text {  (see Definition \ref{def-Fourier}).}$$
  Then, from equation \eqref {eq-theta-1} we get
  $$\displaystyle\cF_{1,1}([A',u']) =\alpha|\Delta_{p+1}(v)|^{\frac{\dim K_{p}(1,-1)}{\deg \Delta_{p+1}}}\overline \cF_{1,-1}\varphi(\ad v\ad u' A')=\alpha|\Delta_{p+1}(v)|^{\frac{\dim K_{p}(1,-1)}{\deg \Delta_{p+1}}}\overline \cF_{1,-1}\varphi(\theta_{u+v}(A')).$$
  Therefore equation (\ref{eq-f-gamma}) above can be written as
  $$f(w)= \gamma \alpha|\nabla_{k-p}(u')|^{\frac{\dim K_{p}(1,-1)}{\deg \nabla_{k-p}}}|\Delta_{p+1}(v)|^{\frac{\dim K_{p}(1,-1)}{\deg \Delta_{p+1}}}\int_{K_{p}(1,-1)}\overline \cF_{1,-1}\varphi(\theta_{u+v}(A'))\overline{\psi(b(A',\ad u'w))} dA'.$$
  Using Proposion \ref{prop-mesures-6}, 2') (and the relation $\nabla_{k-p}(u') = \widetilde \Delta_{k-p}(u)^{-1}$) we obtain:
   $$\begin{array}{rl}  
    \displaystyle f(w)&\displaystyle=\gamma\alpha \int_{K_{p}(-1,1)}\overline \cF_{1,-1}\varphi(A)\overline{\psi(b(\theta_{u+v}A,\ad u' w))} dA\cr
    \displaystyle{}&\displaystyle=\gamma\alpha \int_{K_{p}(-1,1)}\overline \cF_{1,-1}\varphi(A)\overline{\psi(b( A,\theta_{u+v}\circ\ad u' w))} dA\cr
    \displaystyle{}&\displaystyle=\gamma\alpha \int_{K_{p}(-1,1)}\overline \cF_{1,-1}\varphi(A) \psi(b( A,-\theta_{u+v}\circ\ad u' w)) dA\cr
    \displaystyle{}&\displaystyle=\gamma\alpha \varphi(-\theta_{u+v}\circ\ad u' w)= \gamma\alpha f([-\theta_{u+v}\circ\ad u' w,v])\cr
     \displaystyle{}&\displaystyle=\gamma\alpha f(\ad v\circ\theta_{u+v}\circ\ad u' w).
     \end{array}$$
     
     Define $A'\in K_{p}(1,-1)$ such that $w=\ad v A'$ (always possible). Then (using \eqref{eq-theta-1})
     $$f(w)= \gamma\alpha f(\ad v\circ\theta_{u+v}\circ\ad u'\circ \ad v A')= \gamma\alpha f(\ad v\circ \theta_{u+v}^2 A')   = \gamma\alpha f(\ad v A') =\gamma\alpha f(w).$$
     Hence $\gamma\alpha=1$.   
     
\end{proof}

{\begin{theorem} \label{Formules-int-simples} \hfill

1) For  $f\in L^1(V^+)$ we have
$$\int_{V^+}f(X)dX=\int_{u\in K_{p}(2,0)}\int_{A\in K_{p}(1,-1)}\int_{v\in K_{p}(0,2)}f(e^{\ad A}(u+v))|\Delta_{p+1}(v)|^{\frac{(p+1)d}{\kappa}}du\,dA\,dv,$$

2) For  $g\in L^1(V^-)$ we have 
$$\int_{V^-}g(Y)dY=\int_{u'\in K_{p}(-2,0)}\int_{A\in K_{p}(1,-1)}\int_{v'\in K_{p}(0,-2)}g(e^{\ad A}(u'+v'))|\nabla_{p+1}(u')|^{\frac{(k-p)d}{\kappa}}du'\,dA\,dv'.$$

\end{theorem}

\begin{proof}

We only prove the first formula. The proof of the second is similar. 
The right hand side is equal to  
$$\int_{u\in K_{p}(2,0)}\int_{A\in K_{p}(1,-1)}\int_{v\in K_{p}(0,2)}f(u+\frac{1}{2}(\ad A)^2v+(\ad A )v+v)|\Delta_{p+1}(v)|^{\frac{(p+1)d}{\kappa}}du\,dA\,dv.$$
We make first the change of variable $u\longrightarrow (u+\frac{1}{2}(\ad A)^2v)$ (which is  possible as  $\frac{1}{2}(\ad A)^2v\in K_{p}(2,0)$) and we get 
$$\int_{u\in K_{p}(2,0)}\int_{A\in K_{p}(1,-1)}\int_{v\in K_{p}(0,2)}f(u +[A,v]+v)|\Delta_{p+1}(v)|^{\frac{(p+1)d}{\kappa}}du\,dA\,dv,$$
and then, using Lemma \ref{lem-formule-int}, Proposition \ref{Prop-constantes=1} and Proposition \ref{prop-decomposition-dX},  we obtain
$$\int_{K_{p}(2,0)}\int_{K_{p}(1,1)}\int_{K_{p}(0,2)}f(u+w+v)du\,dw\,dv=\int_{V^+}f(X)dX.$$ 

\end{proof}

  \begin{definition} \label{def-moyennes}  (Introduced   in (\cite{Mu} Definition 4.3.1) for a special kind of maximal parabolic, see also (\cite{B-R} Definition 4.20).)

  Let  $p\in \{0,\ldots,k\}$   
 1) For  $f\in L^1(V^+)$ and  $(u,v)\in K_{p}(2,0)\times K_{p}(0,2)$, we define
 $$T_{f}^{p,+}(u,v)\index{Tp+@$T_{f}^{p,+}(u,v)$}=|\Delta_{p+1}(v)|^{\frac{(p+1)d}{2\kappa}}\int_{K_{p}(1,-1)}f(e^{\ad A}(u+v))dA.$$
 
  2) For $g\in L^1(V^-)$ and  $(u',v')\in K_{p}(-2,0)\times K_{p}(0,-2)$, we define
 $$T_{g}^{p,-}(u',v') \index{Tp-@$T_{g}^{p,-}(u',v')$}=|\nabla_{k-p}(u')|^{\frac{(k-p)d}{2\kappa}}\int_{K_{p}(1,-1)}f(e^{\ad A}(u'+v'))dA.$$
 \end{definition}
 \vskip 5pt
Then Theorem \ref{Formules-int-simples} can be re-formulated in the following manner:
 
 \vskip 10pt
\begin{theorem}  \label{th-formule-int-T_{f}} \hfill

 1) Let  $f\in L^1(V^+)$, then   the function $(u,v)\longmapsto T_{f}^{p,+}(u,v)|\Delta_{p+1}(v)|^{\frac{(p+1)d}{2\kappa}}$ belongs to $L^1(K_{p}(2,0)\times K_{p}(0,2))$ and 
 
 $$\int_{V^+}f(X)dX= \int_{u\in K_{p}(2,0)}\int_{v\in K_{p}(0,2)}T_{f}^{p,+}(u,v)|\Delta_{p+1}(v)|^{\frac{(p+1)d}{2\kappa}}du\,dv.$$
 
 2) Let   $g\in L^1(V^-)$, then then   the function $(u',v')\longmapsto T_{g}^{p,-}(u',v')|\nabla_{k-p}(u')|^{\frac{(k-p)d}{2\kappa}}$ belongs to $L^1(K_{p}(-2,0)\times K_{p}(0,-2))$ and 
 
 $$\int_{V^-}g(Y)dY= \int_{u'\in K_{p}(-2,0)}\int_{v'\in K_{p}(0,-2)}T_{g}^{p,-}(u',v')|\nabla_{k-p}(u')|^{\frac{(k-p)d}{2\kappa}}du'\,dv'.$$
 \end{theorem}
\vskip 10 pt
 We investigate now the smoothness of $T_{f}^{p,+}$ and $T_{g}^{p,-}.$
\vskip 10pt

\begin{theorem} \label{th-proprietes-T_{f}}\hfill

\begin{enumerate}\item\begin{enumerate}\item  For $f\in {\cal S}(V^+)$, $u\in K_{p}(2,0)$, $v\in K_{p}(0,2)'$, $A\in K_{p}(1,-1)$, the function
$$A\longmapsto f(e^{\ad A}(u+v))$$
belongs to   ${\cal S}(K_{p}(1,-1))$ and  $T_{f}^{p,+}(u,v)$  is everywhere defined on  $K_{p}(2,0)\times K_{p}(0,2)'$. Moreover $T_{f}^{p,+}(u,v)$ is locally constant on $K_{p}(2,0)\times K_{p}(0,2)'$. 

\item If $f\in {\cal S}(V^+)$ with compact support in  $ {\cal O}^+_{p+1}=\{X\in V^+,\, \Delta_{p+1}(X)\neq 0\}$, then  $T_{f}^{p,+}$ is everywhere defined on $K_{p}(2,0)\times K_{p}(0,2)$ and    is locally constant with compact support in  $K_{p}(2,0)\times K_{p}(0,2)'$.\end{enumerate}
\item \begin{enumerate}\item For $g\in {\cal S}(V^-)$, $u'\in K_{p}(-2,0)'$, $v'\in K_{p}(0,-2)$, $A\in K_{p}(1,-1)$, the function
$$A\longmapsto g(e^{\ad A}(u'+v'))$$
belongs to   ${\cal S}(K_{p}(1,-1))$ and  $T_{g}^{p,-}(u',v')$  is everywhere defined on  $K_{p}(-2,0)'\times K_{p}(0,-2)$. Moreover $T_{g}^{p,-}(u',v')$ is locally constant on $K_{p}(-2,0)'\times K_{p}(0,-2)$.

\item If $g\in {\cal S}(V^-)$ with compact support in  $ {\cal O}^-_{k-p}=\{Y\in V^-,\, \nabla_{k-p}\neq 0\}(Y)$, then  $T_{g}^{p,-}$ is everywhere defined on $K_{p}(-2,0)\times K_{p}(0,-2)$ and     is locally constant with compact support in  $K_{p}(-2,0)'\times K_{p}(0,-2)$.
 \end{enumerate}
\end{enumerate}
\end{theorem}

\begin{proof} \hfill

 We only prove the first assertion. The proof of the second one will be similar.

 We fix $f\in \cal S(V^+)$. By ( \cite{Re}, II.1.3), we can suppose that there exist $f_{1}\in \cal S(K_{p}(2,0))$, $f_{2}\in \cal S(K_{p}(1,1))$ and $f_{3}\in \cal S(K_{p}(0,2))$ such that 
$$f(u+x+v)=f_{1}(u)f_{2}(x)f_{3}(v),\quad u\in K_{p}(2,0), x\in K_{p}(1,1)\;\text{ and }\; v\in K_{p}(0,2).$$
Let  $u\in K_{p}(2,0)$ and  $v\in K_{p}(0,2)'$. We consider the map
$$\alpha: \left\{\begin {array}{rll}
 K_p(1,-1) & \longrightarrow  & \C\cr
A&\longmapsto&f(e^{\ad A}(u+v))=f_{1}(u+\frac{1}{2}(\ad A)^2v)f_{2}([A,v])f_{3}(v).\cr
\end{array}\right.$$  
 Denote by $K_{2}$ the (compact) support of $f_{2}$. As the map $A\longmapsto [A,v]=\alpha_{v}(A)$ is a linear isomorphism  from $K_{p}(1,-1)$ onto $K_{p}(1,1)$ we see that there exists a compact subset $K_{2}'\subset K_{p}(1,-1)$ such that $f_{2}([A,v])=0$ if $A\notin K_{2}'$. \\
This shows that $\alpha$ has compact support. \vskip 5pt

In the rest of the proof we will denote by $||\,\,||$ a fixed norm on the various involved vector spaces.

We will now show that $\alpha$ is locally constant.  This means that for $A_0\in K_p(1,-1)$, there exists $\ep>0$ such that 
$$h\in K_P(1,-1),\;\Vert h\Vert<\ep\Longrightarrow \al(A_0+h)=\al(A_0).$$

 We consider first the function $A\longmapsto f_{1}(u+\frac{1}{2}(\ad A)^2v)$. As the map $A\mapsto (\ad A)^2v$ is continuous, for any $\ep>0$ there exists $\eta>0$ such that
$$||h||\leq \eta \Rightarrow \Vert  \ad (A_0+h)^2v-\ad(A_0)^2v\Vert <\epsilon.$$
As $f_{1}$ is locally constant, this implies that the function $A\longmapsto  f_{1}(u+\frac{1}{2}(\ad A)^2v)$ is locally constant. A similar argument shows that the function $A\longmapsto  f_{2}([A,v])$ is locally constant. Hence we have proved that $\alpha \in \cal{S}(K_{p}(1,-1))$.

But now, the fact that $T_{f}^{p,+}(u,v)$  is everywhere defined on  $K_{p}(2,0)\times K_{p}(0,2)'$ is clear as the integral which defines $T_{f}^{p,+}(u,v)$ is the integral of $\alpha$ over a compact set. 

The fact that $T_{f}^{p,+}$ is locally constant on $K_{p}(2,0)\times K_{p}(0,2)'$ will be a consequence on the proof of 1)b). See below. \me
 
Let us now prove 1) b). Here we suppose that $f\in \cal{S}(\cal{O}_{p+1})$.  Denote by $K_{i}$ the support of $f_{i}$ ($i=1,2,3$).

 If $\Delta_{p+1}(u+v)=\Delta_{p+1}(v)=0$ (in other words $v\notin K_{p}(0,2)'$), then $f(e^{\ad A}(u+v))=f(u+\frac{1}{2}(\ad A)^2v+(\ad A )v+v)=0$ because Supp$(f)\subset \cal{O}_{p+1}$. Then   $K_{3}\subset  K_{p}(0,2)'$ and  $T_{f}^{p,+}$ is now defined everywhere with support in $ K_{p}(2,0)\times K_{p}(0,2)'$.\me
 
 Let us now prove that $T_{f}^{p,+}$ has compact support. 
 

 Remember from above that the map $\alpha_{v}$ defined by $\alpha_{v}(A)=[A,v]$ is an isomorphism form $K_{p}(1,-1)$ onto $K_{p}(1,1)$. Consider now the map
$$\begin {array}{rll}
 K_{3}\times K_{2} & \stackrel{\varphi}\longrightarrow  & K_{p}(1,-1)\cr
 (v,x)&\longmapsto&\varphi(v,x)= \alpha_{v}^{-1}(x).\cr
\end{array}$$
As it is continuous,  its image  $K=\varphi(K_{3}\times K_{2})$ is compact. Moreover the support of the map
 $$A\longmapsto f_{2}([A,v])$$
 is contained in  $\alpha_{v}^{-1}(K_{2})$, so that the support of all the maps $A\longmapsto f_{2}([A,v])$ when $v\in K_{3}$ is contained in $K=\varphi(K_{3}\times K_{2})\subset K_{p}(1,-1)$. In other words, the support of all the maps $$A\longmapsto f(e^{\ad A}(u+v))$$
 when $v\in K_{3}$, is contained in $K$.

 Since the map $\phi:(A,v)\mapsto \frac{1}{2}(\ad A)^2v$ is continuous from $K_p(1,-1)\times K_p(0,2)$ to $K_p(2,0)$,  the set $K_0=\phi(K\times K_3)$ is a compact subset of $K_p(2,0)$. Thus  there exists $C\in {\R^*}^+$ such that for $v\in K_{3}$ and $A\in K$, we have 
$$||\frac{1}{2}(\ad A)^2v||\leq C.$$
Then  $||u+\frac{1}{2}(\ad A)^2v||\geq |\,||u||-||\frac{1}{2}(\ad A)^2v||\,|\geq \,||u||-||\frac{1}{2}(\ad A)^2v||\geq ||u|| -C$. As $f_{1}$ has compact support, there exists $C'>0$ such that for $v\in K_{3}$ and $A\in K$, 
 $$||u||>C'\Longrightarrow f_{1}(u+\frac{1}{2}(\ad A)^2v)=0.$$

 Hence 
 $$T_{f}^{p,+}(u,v)=|\Delta_{p+1}(v)|^{\frac{(p+1)d}{2\kappa}}\int_{K_{p}(1,-1)}f(e^{\ad A}(u+v))dA$$
 has compact support for $f\in \cal S(\cal{O}_{p+1})$.\me
 
 It  remains to  prove that $T_{f}^{p,+}$ is locally constant. 

As the integral in the definition of  $T_{f}^{p,+}$ is equal to zero for $v\notin K_{3}$ and as $|\Delta_{p+1}(v)|$ is locally constant on $K_{p}(0,2)'$, it is enough to prouve that 
$$\Psi(u,v)=\int_{K_{p}(1,-1)}f_{1}(u+\frac{1}{2}(\ad A)^2v)f_{2}([A,v])f_{3}(v)dA$$
is locally constant.

This means that,   for $u\in K_p(2,0)$ and $v\in K_p(0,2)$,  there exists $\epsilon>0$ such that
$$(h,k)\in K_{p}(2,0)\times K_{p}(0,2), \,\,||h||\leq \epsilon, ||k||\leq\epsilon\Longrightarrow \Psi(u+h,v+k)=\Psi(u,v).$$

 As $f_i$ is locally constant with compact support,  there exists $\epsilon_{i}>0$ such that $f_{i}$ is constant on each ball of radius $\epsilon_{i}$ ($i=1,2$). Also  there exists $\delta_{2}>0$, such that for $||k||\leq\delta_{2}$,  and $A\in K$, we have  $||[A,k]||\leq c||A||\,||k||\leq \epsilon_{2}$ (for a   $c\in {\R^*}$) , and hence 
\beq\label{Dem-eq1}f_{2}([A,v+k])=f_{2}([A,v]+[A,k])=f_{2}([A,v]).\eeq

Similarly, there exists $\delta_{1}>0$ and $c>0$ such that 
$$||h+\frac{1}{2}(\ad A)^2k||\leq||h|| +c||A||^2||k||\leq \epsilon_{1}\text{ if } ||h||\leq \delta_{1} \text{, }||k||\leq \delta_{1} \text{ and } A\in K.$$
Hence 
\beq\label{Dem-eq2}f_{1}(u+h+\frac{1}{2}(\ad A)^2(v+k))=f_{1}(u+\frac{1}{2}(\ad A)^2(v)+h+\frac{1}{2}(\ad A)^2(k))=f_{1}(u+\frac{1}{2}(\ad A)^2(v)).\eeq
Finally (\ref{Dem-eq1}) and (\ref{Dem-eq2}) imply that $ \Psi(u+h,v+k)=\Psi(u,v)$ if $||h||,\,||k||\leq Min(\delta_{1},\delta_{2},\epsilon_{3})$.
This ends the proof of 1)b). 

Lets us return to the end of the proof of 1)a). Taking now $K_{3}$ to be an open compact subset of $K_{p}(0,2)'$ which contains a fixed element $v$ (instead of the support of $f_{3}$), and taking into account that $|\Delta_{p+1}(.)|^{\frac{(p+1)d}{2\kappa}}$ is locally constant on $K_{3}$, the same proof as before shows that $   T_{f}^{p,+} $ is locally constant on $K_{p}(0,2)\times K_{p}(0,2)'$. 
\end{proof} 
\me

\centerline{\bf In the rest of the paper all the involved measures   are normalized as in Definition \ref{def-normalisation}.}
\me

\subsection{Weil formula and the computation of $T_{{\cal F}f}^{p,-}$}\label{sec:section-Weil-formula}
\hfill

\begin{notation}\label{notation-Fgeneral} 
From now on, for sake of simplicity, and as the context will be clear, we will always denote by the same letter $\cF$, the Fourier transforms between various subspaces of $V^+$, $V^-$ and $\go g$.
\end{notation} 

 In this section we will   make a connection, roughly speaking,  between  $T_{{\cal F}f}^{p,-}$ and  ${\cal F}T_{f}^{p,+}$ (for the precise statement see Theorem \ref{th-T_{Ff}-T_{f}} below).

Define  $E=K_{p}(1,-1)$. Then  $E^*$ is identified to  $K_{p}(-1,1)$ through the normalized Killing  form $b$ (see (\ref{def-b})). More precisely  an element $B\in K_{p}(-1,1)$ is identified with  the form $y ^*\in E^*$ defined by
$$\langle y^*,A\rangle=b(B,A)\, \text{ for }A\in K_{p}(1,-1). $$
As usual the Fourier transform of a function $f\in \cal{S}(K_{p}(1,-1))$  by 
$${\cal F}f(B)=\overset{\wedge}f(y^*)=\int_{K_{p}(1,-1)}f(A)\psi(\langle y^*,A\rangle)dA=\int_{K_{p}(1,-1)}f(A)\psi(b(B,A ))dA.$$

(Remember that we have chosen, once and for all,   a non trivial unitary additive character $\psi$ of  $F$).

 Suppose now that a non degenerate quadratic form $Q$ is given on $E$. Denote by $$\beta(A,A')=Q(A+A')-Q(A)-Q(A') \hskip 20pt\text{ for  }A,A' \in E$$
 the corresponding bilinear form. This bilinear form induces a linear isomorphism $\alpha_{\beta}$ between $E$ and $E^*$ given by
 $$b( \alpha_{\beta}(A),A')=\beta(A,A') \hskip 10pt \text { for } A,A'\in E.$$
 
 Then there exists a constant $C>0$ such that for any $f^*\in \cal{S}(E^*)$
 
 \beq\label{ConstanteC}\int_{E^*}f^*(B)dB=C\int_{E}f^*(\alpha_{\beta}(A))dA.\eeq

 \begin{theorem} \label{th-Weil} {\rm (A. Weil) } 
 
 There exists  a constant $\gamma_{\psi}(Q)$ \index{gammapsi@$\gamma_{\psi}(Q)$} of module 1 such that for $f\in \cal{S}(E)$
 $$\int_{E}\cal{F}f(\alpha_{\beta}(A))\psi(Q(A))dA=\frac{1}{\sqrt{C}}\gamma_{\psi}(Q)\int_{E}f(A)\psi(-Q(A))dA.$$ \end{theorem}
 
 \begin{proof} \hfill

 For another   Haar measure $d\mu(A)$ on $E$, we consider  for $f\in \cal{S}(E)$, a ``new'' Fourier transform by
  
 $$\cal{F}'f(A')=\int_{E}f(A)\psi(b(\alpha_{\beta}(A'),A)d\mu(A) \hskip 10pt(A,A'\in E).$$
 
 First of all we determine $\lambda>0$ such the measure $d\mu(A)=\lambda dA$ is self dual for the Fourier transform $\cal{F}'$. From the definitions we see that $\cal{F}' f(A')=\lambda\cal{F}f(\alpha_{\beta}(A'))$.
 
 Then, using (\ref{ConstanteC}) above,  the measure $\lambda dA$ is self dual for $\cal{F}'$ if and only if
 $$\int_{E}\cal{F}'f(A')\lambda dA'= \lambda^2\int_{E}\cal{F}f({\alpha_{\beta}(A'))}dA'= \frac{\lambda^2}{C}\int_{E^*}\cal{F}f(B)dB= \frac{\lambda^2}{C}f(0)=f(0).$$
 Hence the self dual measure for $\cal{F}'$ is $\sqrt{C}dA$. Once this self dual measure is known, the formula in the Theorem is juste a rewriting of formula (1-4) p. 500 in \cite{R-S} (where self-dual measures are needed). And this is a particular case of a more general result of A. Weil concerning the Fourier transform of quadratic characters of abelian groups (\cite{W}). 
 
 \end{proof}

 We will now make a particular choice of $Q$. Consider   two generic elements $u'\in K_{p}(-2,0)'$ and  $v\in K_{p}(0,2)'$.  
We define the quadratic form $Q_{u',v}$ on  $K_{p}(1,-1)$ (as in section \ref{sec:section-qXiXj})   by
$$Q_{u',v}(A)=\frac{1}{2}b((\ad A)^2u',v).$$

We compute now the bilinear form  $\beta_{u',v}$ corresponding to this form (the computation is the same as in \cite{B-R}, p.106, we give it here for the convenience of the reader):
$$\begin{array}{rll}
\beta_{u',v}(A,A')&=&Q_{u',v}(A+A')-Q_{u',v}(A)-Q_{u',v}(A')\cr
&=&\frac{1}{2}b((\ad A+\ad A')^2u',v)-\frac{1}{2}b((\ad A)^2u',v)-\frac{1}{2}b((\ad A')^2u',v)\cr
&=&\frac{1}{2}b([A,[A',u']],v)+\frac{1}{2}b([A',[A,u']],v)\cr
&=&-\frac{1}{2}b([A',u'],[A,v])+\frac{1}{2}b(A',[[A,u'],v])\cr
&=&\frac{1}{2}b(A',[[A,v],u'])+\frac{1}{2}b(A',[[A,u'],v])\cr
&=&b(A',[[A,v],u'])\;  \text{ (since}\;  \ad u'\;\ad v=\ad v\; \ad u') \cr
&=&b(A',[[A,u'],v])
\end{array}$$
Hence if the map  $\alpha_{\beta_{u',v}}:E\longrightarrow E^*$ is defined by 
$$\langle \alpha_{\beta_{u',v}}(A), A'\rangle=\beta_{u',v}(A,A'),$$ we obtain (modulo the identification $E*=K_{p}(-1,1))$:
$$ \alpha_{\beta_{u',v}}(A)=\ad v\ad u' (A)= \ad u' \ad v(A).$$

We have seen previously that there exists a constant $c_{ \alpha_{\beta_{u',v}}}>0$ such that for all $f^*\in {\cal S}(E^*)$  
$$\int_{E^*}f^*(B)dB=c_{ \alpha_{\beta_{u',v}}}\int_{E}f^*( \alpha_{\beta_{u',v}}(x))dx.$$
 Using Lemma \ref{lem-formule-int} we have: 
 $$\displaystyle \begin{array}{rll
}\displaystyle \int_{K_{p}(1,-1)}f^* (\alpha_{\beta_{u',v}}(A))dA&=&\displaystyle\int_{K_{p}(1,-1)}f^*(\ad u' \ad v(A))dA\cr
&=&\displaystyle|\Delta_{p+1}(v)|^{-\frac{(p+1)d}{\kappa}}\int_{K_{p}(1,1)}f^*(\ad u'(w))dw\cr
 
&=&\displaystyle|\Delta_{p+1}(v)|^{-\frac{(p+1)d}{\kappa}}|\nabla_{k-p}(u')|^{-\frac{(k-p)d}{\kappa}}\int_{K_{p}(-1,1)}f^*(B)dB\cr
 
\end{array}$$
Hence 
$$c_{ \alpha_{\beta_{u',v}}}=\displaystyle|\Delta_{p+1}(v)|^{-\frac{(p+1)d}{\kappa}}|\nabla_{k-p}(u')|^{-\frac{(k-p)d}{\kappa}}$$
 
 Then if we set  $ \gamma(u',v)\index{gammauv@$\gamma(u',v)$} =\gamma_{\psi}(Q_{u',v})$, Theorem \ref {th-Weil}, can be specialized to:
 
 \vskip 10pt
\begin{theorem} \label{th-Weil-(u',v)}\rm {(See Corollary 4.27 p. 107 of \cite{B-R}, for the case  $p=0$ and $F=\R$)}

Let $u'\in K_{p}(-2,0)'$ and $v\in K_{p}(0,2)'$. Then for    $f\in {\cal S}(K_{p}(1,-1))$ we have:
$$\begin{array}{lll}\displaystyle \int_{K_{p}(1,-1)}{\cal F}f(\ad u' \ad v (A))\psi(Q_{u',v}(A))dA&&\cr
=\displaystyle\gamma(u',v)|\Delta_{p+1}(v)|^{-\frac{(p+1)d}{2\kappa}}|\nabla_{k-p}(u')|^{-\frac{(k-p)d}{2\kappa}}\int_{K_{p}(1,-1)}f(A) \psi(-Q_{u',v}(A))
dA&&
\end{array}$$

\end{theorem}

 \vskip 10pt
\begin{theorem} \label{th-T_{Ff}-T_{f}}\rm{(I. Muller (\cite{Mu}), for the case $\ell=1$, and $p=k-1$,  see also Theorem 4.28 p. 107 of \cite{B-R} for the case  $p=0$ and $F=\R$). }

If $f\in {\cal S}(V^+)$ and if  $(u',v)\in K_{p}(-2,0)'\times K_{p}(0,-2)$ then
$$T_{{\cal F}f}^{p,-}(u',v')={\cal F}_{v}(\gamma(u',v)({\cal F}_{u}T_{f}^{p,+})(u',v))(v').$$
\end{theorem}
  \begin{proof}\hfill
  
  We know from Theorem \ref{th-proprietes-T_{f}},  1)a) and 1)b) that $T_{f}^{p,+}(u,v)$ is everywhere defined for $u\in K_{p}(2,0)$ and $v\in K_{p}(0,2)'$ and that $T_{{\cal F}f}^{p,-}(u',v')$ is everywhere defined for $u'\in K_{p}(-2,0)'$ and $v'\in K_{p}(0,-2)$. The same Theorem tells us also that  $T_{f}^{p,+}(u,v)$ is almost everywhere defined on $K_{p}(2,0)\times K_{p}(0,2)$ and that the function $u\longmapsto T_{f}^{p,+}(u,v)$ is integrable over $K_{p}(2,0)$ for almost all $v\in K_{p}(0,2)$.
  
  We note also that here, although the base field $F$ is $p$-adic and that we work with a general   maximal parabolic $P_{p}$, the main steps of the proof are the same as in  (\cite{B-R} Theorem 4.28, p.107) where $F=\R$ and where $p=0$.
  
  In the rest of the proof we will omit the spaces where the integrations are performed, and we make the following convention: the integration in the variables $A,A'$ will be on $K_{p}(1,-1)$, in $u$ on $K_{p}(2,0)$, in $v$ on $K_{p}(0,2)$.
  
    Let $f\in \cal{S}(V^+)$. As we have already seen we can suppose that 
  $$f(u+w+v)=f_{1}(u)f_{2}(w)f_{3}(v) \hskip 10pt \text { where }\hskip 5pt f_{1}\in \cal{S}(K_{p}(2,0)),  f_{2}\in \cal{S}(K_{p}(1,1)),  f_{3}\in \cal{S}(K_{p}(0,2)).$$
  {\it We suppose now that $v$ is fixed in $K_{p}(0,2)'$}. We have then
  
  $$\begin{array}{rll}
  \displaystyle \int_{u} |T_{f}^{p,+}(u,v)|du&\leq&\displaystyle  |\Delta_{p+1}(v)|^{\frac{(p+1)d}{2\kappa}}\int_{u}\int_{A}|f(e^{\ad A}(u+v))|dAdu\cr
 
  &=&\displaystyle   |\Delta_{p+1}(v)|^{\frac{(p+1)d}{2\kappa}}|f_{3}(v)|\int_{u}\int_{A}|f_{1}(u+\frac{1}{2}(\ad A)^2v)|\,|f_{2}([A,v])|du dA.
  \end{array}$$
   We make the change of variable $u+\frac{1}{2}(\ad A)^2v\rightarrow u$ in the $u$-integral and  we note  that the $A$-integral is over a compact set (due to the fact that $A\longmapsto [A,v]$ is an isomorphism from $K_{p}(1,-1)$ onto $K_{p}(1,1)$). Hence, for any $v\in K_{p}(0,2)'$
   \beq\label{Dem-Fubini}\int_{u}| T_{f}^{p,+}(u,v)|du\leq \int_{u} T_{|f|}^{p,+}(u,v)du<+\infty.\eeq
Therefore  
   $$u\longmapsto T_{f}^{p,+}(u,v)\in L^1(K_{p}(2,0))\hskip 10pt \text { for } v\in K_{p}(0,2)'.$$
   
   Let us now compute the $u$-Fourier transform of $T_{f}^{p,+}$ given by 
   $$\cal{F}_{u}T_{f}^{p,+}(u',v)=|\Delta_{p+1}(v)|^{\frac{(p+1)d}{2\kappa}}\int_{u,A}f(u+\frac{1}{2}(\ad A)^2v+[A,v]+v)\psi(b(u,u'))dAdu.$$
   Due to (\ref{Dem-Fubini}) the order of the integrations does not matter. Making the change of variable $u+\frac{1}{2}(\ad A)^2v\rightarrow u$  we obtain
   $$\cal{F}_{u}T_{f}^{p,+}(u',v)=|\Delta_{p+1}(v)|^{\frac{(p+1)d}{2\kappa}}\int_{u,A}f(u+[A,v]+v)\psi(b(u-\frac{1}{2}(\ad A)^2v,u'))dAdu.$$
   We introduce now the function
   $$J (u',A,v)=\int_{u}f(u+[A,v]+v)\psi(b(u,u'))du=(\cal{F}_{u}f)(u',[A,v],v).$$
   
      Then from the definition of the form $Q_{u',v}$
   \beq\label{Dem-FuQ}\cal{F}_{u}T_{f}^{p,+}(u',v)=|\Delta_{p+1}(v)|^{\frac{(p+1)d}{2\kappa}}\int_{A}J (u',A,v)\psi(-Q_{u',v}(A))dA.\eeq 
   
   {\it We suppose now that $u'\in K_{p}(-2,0)'$.}
   
   As $v$ is generic in $K_{p}(0,2)$ the function $(u,A)\longmapsto f(u+[A,v]+v)$ belongs to $\cal{S}(K_{p}(2,0)\times K_{p}(1,-1))$ and the function $A\longmapsto J (u',A,v)$ belongs to $\cal{S}(K_{p}(1,-1))$. As $u'$ is generic too, the quadratic form $Q_{u',v}$ is non degenerate and we can apply the Weil formula as in  Theorem \ref{th-Weil-(u',v)} to the integral in the right hand side of $(4)$. We obtain
$$\int_{A}J(u',A,v)\psi(-Q_{u',v}(A))dA $$
\beq\label{Dem-FuQ2}=\gamma(u',v)^{-1}|\Delta_{p+1}(v)|^{\frac{(p+1)d}{2\kappa}}|\nabla_{k-p}(u')|^{\frac{(k-p)d}{2\kappa}}\int_{A'}(\cal{F}_{A}J )(u', \ad u' \ad v (A'), v)\psi(Q_{u',v}(A')) dA' ,  \eeq
   where
   $$(\cal{F}_{A}J)(u', \ad u' \ad v\, A', v)= \int_{A}(\cal{F}_{u}f)(u',[A,v],v)\psi(b((A,\ad u' \ad v\, A'))dA.$$
   If we make the change of variable $w=[A,v]$, and if we use Lemma \ref{lem-formule-int}, together with the fact that $b((A,\ad u' \ad v\, A')=b([A,v],[u',A'])$ the preceding formula becomes
  \beq\label{Dem-FAJ}(\cal{F}_{A}J )(u', \ad u' \ad v\, A', v)= |\Delta_{p+1}(v)|^{-\frac{(p+1)d}{\kappa}}(\cal{F}_{w}\cal{F}_{u}f)(u',[u',A'],v).\eeq   
   From (\ref{Dem-FuQ}), (\ref{Dem-FuQ2}) and (\ref{Dem-FAJ}), we obtain
\beq\label{Dem-gammaFu}\gamma(u',v)(\cal{F}_{u}T_{f}^{p,+})(u',v)=|\nabla_{k-p}(u')|^{\frac{(k-p)d}{2\kappa}}\int_{A'}(\cal{F}_{w}\cal{F}_{u}f)(u',[u',A'],v)\psi(Q_{u',v}(A'))dA'.\eeq
   \vskip 5pt 
   {\it From now on we only  suppose that $u'$ is generic in $K_{p}(-2,0)$ ($v$ may be singular).}
   \vskip 5pt

   The Fourier transform $\cal{F}_{w}\cal{F}_{u}f\in \cal{S}(K_{p}(-2,0)\times K_{p}(-1,-1)\times K_{p}(0,2))$ and as $u'$ is generic we obtain, as before, that 
   $$\int_{v,A'}|(\cal{F}_{w}\cal{F}_{u}f)(u',[u',A'],v)|dA'dv<+\infty.$$
   
   It follows than from (\ref{Dem-gammaFu}) that the function $v\longmapsto \gamma(u',v)(\cal{F}_{u}T_{f}^{p,+})(u',v)$ belongs to $L^1(K_{p}(0,2)$ for $u'$ generic in $K_{p}(-2,0)$ and this allows, by Fubini's Theorem, to change the order of the integrations in $A'$ and 
   $v$ in the following computation of its Fourier transform.
   
   We have
   \begin{align*}&\cal{F}_{v}(\gamma(u',v)(\cal{F}_{u}T_{f}^{p,+})(u',v))(v')\cr
   &=|\nabla_{k-p}(u')|^{\frac{(k-p)d}{2\kappa}}\int_{v}\int_{A'}(\cal{F}_{w}\cal{F}_{u}f)(u',[u',A'],v)\psi(b(\frac{1}{2}(\ad A')^2u',v)+b(v',v))dA'dv\cr
    &=|\nabla_{k-p}(u')|^{\frac{(k-p)d}{2\kappa}}\int_{A'}\int_{v}(\cal{F}_{w}\cal{F}_{u}f)(u',[u',A'],v)\psi(b(\frac{1}{2}(\ad A')^2u',v)+b(v',v))dvdA'\cr
    &=|\nabla_{k-p}(u')|^{\frac{(k-p)d}{2\kappa}}\int_{A'} (\cal{F}_{v}\cal{F}_{w}\cal{F}_{u}f)(u',[u',A'],v'+\frac{1}{2}(\ad A')^2u') dA'.\cr
   \end{align*}
   But $u'+[u',A']+v'+\frac{1}{2}(\ad A')^2u'= e^{\ad(- A')(u'+v')}$ and $\cal{F}_{v}\cal{F}_{w}\cal{F}_{u}f=\cal{F}f$ and we recognize the definition of $T_{\cal{ F}f}^{p,-}$ (Definition \ref{def-moyennes}, 2)). Hence
   $$T_{{\cal F}f}^{p,-}(u',v')={\cal F}_{v}(\gamma(u',v)({\cal F}_{u}T_{f}^{p,+})(u',v))(v')$$
   for $u'$ generic in $K_{p}(-2,0)$.
   
  \end{proof}
  
  \begin{rem} \label{rem-analogie-GJ}\hfill
  
  Godement and Jacquet have used a mean function $\varphi_{\Phi}$ (\cite {G-J} p.37) which looks like our mean function $T^{p,+}_{f}$ (or $T^{p,-}_{g}$) but which  is not equal to our's in the case of $GL_{n}(F)\times GL_{n}(F)$ acting on $M_{n}(F)$ which corresponds to their situation. Also they proved that $\cF(\varphi_\Phi)=\varphi_{\cF(\Phi)}$ (Loc. cit Lemma 3.4.0 p. 38), which again  looks similar to the preceeding Theorem \ref{th-T_{Ff}-T_{f}}.
  \end{rem}

\subsection{Computation  of the Weil constant $\ga(u',v)$}\label{section-calcul-gamma}
\hfill
\vskip 5pt

  (Remember that in this section, we suppose that $\ell=1$).\me

The aim of this paragraph is to compute the Weil constant  $\ga(u',v)=\ga_\psi(Q_{u',v})$ introduced in Theorem  \ref{th-Weil}, in the case where  $u'\in K_{k-1}(-2,0)'$ and  $v\in K_{k-1}(0,2)'= \tilde{\go g}^{\la_k}\setminus \{0\}$. We recall first some general results of A.~Weil \cite{W} and of  S. Rallis et G. Schiffmann \cite{R-S} concerning $\ga_\psi(Q)$  . \vskip 5pt

First of all, if   $Q$ and  $Q'$ are two equivalent non degenerate quadratic forms ,we have
$$\ga_\psi(Q)=\ga_\psi(Q').$$
 
Remember that  there are four classes modulo the squares   in $F$, namely  $\Cr=F^*/F^{*2}=\{1,\ep,\pi,\ep\pi\}$ where  $\ep$ is a unit which is not a square, and  where  $\pi$ is a uniformizer for $F$.\me

For $a,b\in F^*$, let us denote by  $(a,b)_\Cr$ the Hilbert symbol defined by 
$$(a,b)_\Cr =\left\{\begin{array}{ccc} 1 & {\rm if} \;\; b\;\textrm{belongs to the norm group of }\;  F[\sqrt{a}]\\ -1 & {\rm if \,\,not }\end{array}\right..$$
Denote by  $\chi_a$ the quadratic character of  $F^*$ (associated to the quadratic extension  $F[\sqrt{a}]$) defined by
$$\chi_a(b)=(a,b)_\Cr.$$

If  $\go q_1$ is the quadratic form on  $F$ defined by  $\go q_1(x)=x^2$ we set: 

\beq\label{def-alpha}\al(a)=\ga_\psi(a\go q_1),\quad \varphi(a)=\frac{\al(a)}{\al(1)}\quad a\in F^*.\eeq

As $a^2\go q_{1}$ is equivalent to $\go q_{1}$ one can consider $\alpha$ and $\varphi$ as functions on $\Cr$.

A fundamental result of  A. Weil (\cite{R-S} (1-6)) tells us   that 
$$\varphi(ab)=\varphi(a)\varphi(b)(a,b)_\Cr,\quad a,b\in F^*.$$
As  $\ga_\psi(Q)\ga_\psi(-Q)=1$ for any non degenerate quadratic form  $Q$, one has   $\al(1)^{-1}=\al(-1)$ and hence  $$\al(ab)=\al(a)\al(b)(a,b)_\Cr\al(-1).$$
One has also the following result:
\begin{lemme}\label{lem-gamma(xQ)}(\cite{R-S}  Proposition I-7) Let $Q$ be a non degenerate quadratic form  on a vector space  $E$ of dimension $n$, and discriminant  $disc(Q)$. 
\begin{enumerate}
\item If  $n=2r$ is even then  
$$\ga_\psi(xQ)=\ga_\psi(Q) (x, (-1)^r disc(Q))_\Cr,\quad x\in F^*$$
\item If $n=2r+1$ is odd then
$$\ga_\psi(xQ)=\ga_\psi(Q) (x,(-1)^rdisc(Q))_\Cr \al(x)\al(-1).
$$
\end{enumerate}
\end{lemme}
\begin{proof} 

The first assertion is exactly the assertion of  Proposition I-7 of  \cite{R-S}.  As our formulation of  assertion 2)  is slightly different from  that in \cite{R-S}, we give some details. \me

From   (\cite{R-S} Proposition 1-3), we know that if  $Q(x_1,\ldots, x_n)=\sum_{s=1}^n a_s x_s^2$ then  $\ga_\psi(Q)= \prod_{s=1}^n \al(a_s)$. Using the fact that  $\al(a)\al(b)=\al(ab)(a,b)_\Cr\al(1)$, one obtains  by induction that 
$$\ga_\psi(Q)= \al(1)^{n-1} \al(a_1\ldots a_n)\prod_{i<j}(a_i,a_j)_\Cr.$$
Set $D=disc(Q)=a_{1}\ldots a_{n} \text{ mod}F^{*2}$. As  $(xa_i,xa_j)_\Cr=(x,x)_\Cr(x,a_i)_\Cr(x,a_j)_\Cr (a_i,a_j)_\Cr$,  $(x,x)_\Cr=(-1,x)_\Cr$,  and  $\prod_{i<j}(x,a_{i})_{\Cr}(x,a_{j})_{\Cr}=(x^{n-1},D)_{\Cr}$,  one obtains  
$$\ga_\psi(xQ)=\ga_\psi(Q) (x,(-1)^{\frac{n(n-1)}{2}})_\Cr(x^{n-1},D)_\Cr \frac{\al(x^nD)}{\al(D)}.$$
If  $n$ is even then  $(x^{n-1},D)_\Cr =(x,D)_\Cr$ and  $\al(x^nD)=\al(D)$.  This gives the first assertion . And if $n$ is odd, we have  $(x^{n-1},D)_\Cr =1$ and  $\al(x^nD)=\al(xD)=\al(x)\al(D) (x,D)_\Cr\al(-1)$. This gives the second assertion.
\end{proof}

In what follows we assume that  $k\geq 1$.  \vskip 5pt

\begin{notation}\label{notation-Delta-souligne}  We denote by  $\tilde{\go l}_k$ the Lie algebra generated by $\tilde{\go g}^{\pm\la_k}$, that is 
$$\tilde{\go l}_k=\tilde{\go g}^{-\la_k}\oplus FH_{\la_k}\oplus \tilde{\go g}^{\la_k}.$$
Let $\underline{\tilde{\go g}}$ be the centralizer of $\tilde{\go l}_k$ in $\tilde{\go g}$ . We will denote by underlined letters  the elements and algebras associated to  $\underline{\tilde{\go g}}$.  The algebra $\underline{\tilde{\go g}}$ is  $3$-graded by the element  $\underline{H}_0=\sum_{j=0}^{k-1} H_{\la_j}$, and is regular of rank  $k$.
 We denote by  
$\underline{\tilde{\go g}}=\underline{V}^-\oplus \underline{\go g}\oplus \underline{V}^+$  the corresponding decomposition. It can be noted that in the notations  of section  \ref{sec:section-parabolique} we have  $\underline{V}^\pm= K_{k-1}(\pm 2,0)$
 and  $\underline{\go g}=K_{k-1}(0,0)$ and  $\tilde{\go g}^{\pm\la_k}=K_{k-1}(0,\pm2)$. We have also $\underline{\go g}= Z_{\underline{\tilde{\go g}}}(\go a^0)\oplus \left(\oplus_{i\neq j<k}E_{i,j}(1,-1)\right).$\me
 
     If    $\underline{\go a}^0=\oplus_{j<k} F H_{\la_j}$, then  $\underline{G}$ stands for the group  $\cZ_{{\rm Aut}_0(\underline{\tilde{\go g}})}(\underline{H_0})$ whose Lie algebra is  $\underline{\go g}$ and  $\underline{P}=\underline{L}\underline{N}$ is the parabolic subgroup of $\underline{G}$ corresponding to the Lie algebra $$\underline{\go p}=\cZ_{\underline{\go g}}(\underline{\go a}^0)\oplus\oplus_{r<s<k} E_{r,s}(1,-1).$$  
Let  $\underline{\De}_j$ be the polynomial relative invariants   of $(\underline{P},\underline{V}^+) $ normalized such $\underline{\De}_j(X_0+\ldots+X_{k-1})=1$, etc... (in the  notations of Definition \ref{def-invariants},  we have for example $\underline{\De}_0= \tilde \Delta_{1})$.
\end{notation}
\me

 \begin{rem}\label{rem-e-souligne} From the classification  in Table I  and the convention in Definition \ref{def-ensembleS}, we see  that the integers $\underline{\ell}$ and $\underline{e}$ associated to $\underline{\tilde{\go g}}$ satisfy    $\underline{\ell}=\ell$,  and  $\underline{e}=e$  if $k\geq 2$,   $\underline{e}=0$ if $k=1$. 
\end{rem}
\begin{lemme}\label{lem-decomposition-Delta} 

(i) Let  $u,v$ and  $A$ be respectively in  $\underline{V}^+, \tilde{\go g}^{\la_k}$ and $K_{k-1}(1,-1)$, then 
$$\De_j(e^{\ad A}(u+v))=\left\{\begin{array}{ccc}  \underline{\De}_{j}(u)\De_k(v) & {\rm if} &  j=0,\ldots, k-1,\\
\De_k(v)&  {\rm if } &  j=k.\end{array}\right.$$
(ii)  If $v'=\nu Y_k$, we set  $\tilde{\nabla}_k(v')=\nu$. Then , for  $u',v'$ and $A$ respectively in   $\underline{V}^-,\tilde{\go g}^{-\la_k}$ and $K_{k-1}(1,-1)$, we have  
$$\nabla_j(e^{\ad A}(u'+v'))=\left\{\begin{array}{ccc} \underline{\nabla}_0(u')\widetilde{\nabla}_k(v') & {\rm if} & j=0,\\
\underline{\nabla}_{j-1}(u')& {\rm if} & j=1,\ldots, k.\end{array}\right.$$

\end{lemme}
\begin{proof} See (\cite{B-R} Lemma 5.18). 

 \end{proof}

\begin{cor}\label{cor-reduction-orbite} Let   $k\geq 1$  and $a=(a_0,\ldots, a_k)\in\Sr_e^{k+1}$.   Recall that $\cO^\pm(a)$ is the $\tilde{P}$-orbit in $V^\pm$ defined in Definition \ref{def-Ptilde-orbites}.  

\begin{enumerate}\item  Let $u\in \underline{V}^+$ and  $v\in \tilde{\go g}^{\la_k}$ then  
$$u+v\in \cO^+(a_0,\ldots a_{k})\Longleftrightarrow u\in\underline{\cO}^+(a_0,\ldots a_{k-1}),\;{\rm and }\;  \De_k(v)a_k\in S_e.$$
 where $$\underline{\cO}^+(a_0,\ldots a_{k-1})=\{X\in\underline{V}^+; \underline{\De}_j(X) a_j\ldots a_{k-1}\in S_e\;\text{for}\; j=0,\ldots, k-1\}.$$

\item Let  $u'\in  \underline{V}^-$ and  $v'\in \tilde{\go g}^{-\la_k}$ then 
$$u'+v'\in \cO^-(a_0,\ldots a_{k})\Longleftrightarrow  u'\in \underline{\cO}^-(a_0,\ldots a_{k-1})\;\;{\rm and }\; \tilde{\nabla}_k(v')a_k\in S_e.$$
 where $$\underline{\cO}^-(a_0,\ldots a_{k-1})=\{X\in\underline{V}^+; \underline{\nabla}_j(X) a_j\ldots a_{k-1}\in S_e\;\text{for}\; j=0,\ldots, k-1\}.$$
\end{enumerate}
 Notice that for $k=1$, the set $\underline{\cO}^\pm(a_0)$ is not  a $\tilde{\underline{P}}$ orbit (see Remark \ref{rem-e-souligne}).

\end{cor} 
\begin{proof}   Let  $u\in \underline{V}^+$ and  $v\in \tilde{\go g}^{\la_k}$. By Lemma \ref{lem-decomposition-Delta}, we have
$\De_j(u+v)=\underline{\De}_j(u)\De_k(v)$ for  $0\leq j\leq k-1$ and  $\De_k(u+v)=\De_k(v)$.   Then the definition of the open sets  $\cO^+(a_0,\ldots,a_k)$ implies immediately the first assertion.\me

Let  $u'\in  \underline{V}^-$ and  $v'\in \tilde{\go g}^{-\la_k}$. By Lemma \ref{lem-decomposition-Delta} we have  $\nabla_j(u'+v')=
\underline{\nabla}_{j-1}(u')$ if  $j=1,\ldots, k$ and  $\nabla_0(u'+v')= \underline{\nabla}_0(u')\tilde{\nabla}_k(v')$, where $\tilde{\nabla}_k(y_kY_k)=y_k$. Therefore 
$u'+v'\in  \cO^-(a_0,\ldots a_{k})$ if and only if $ \underline{\nabla}_{j-1}(u')a_0\ldots a_{k-j}\in S_e$
for  $j=1,\ldots, k$ et $\underline{\nabla}_0(u')\tilde{\nabla}_k(v')a_0\ldots a_k\in S_e$. These  conditions are equivalent to 
$ \underline{\nabla}_{j}(u')a_0\ldots a_{k-1-j}\in S_e$ for  $j=0,\ldots, k-1$ and  $\tilde{\nabla}_k(v')a_k\in S_e$. This gives the second assertion.

\end{proof}

\begin{prop}\label{prop-calcul-gamma} Let $k\geq 1$. Fix $\underline{a}=(a_0,\ldots ,a_{k-1})\in\Sr_e^k$ and
 $c_k\in \Sr_e$. 
  Let  $u'\in  \underline{\cO}^-(\underline{a})$ and let  $v\in \tilde{\go g}^{\la_k}\setminus\{0\}$such that  $\De_k(v)c_k\in S_e$. Then  $\ga(u',v)$ depends only on  $\underline{a}$ and on  $c_k$. 

Therefore let us set: 
 $$\ga_k(\underline{a}, c_k)=\ga(u',v),\quad \text{for}\;\; u'\in\underline{\cO}^-(\underline{a}),\; v\in  K_p(0,2),\; \De_k(v)c_k\in S_e.$$
\item  Remember  that $\al(a)=\ga_\psi(a\go q_1)$ for  $a\in F^*$, where $\go q_1(x)=x^2,$ $x\in F$.  
Remember also  that  if  $e=2$,   we have  $S_e=N_{E/F}(E^*)$ where   $E$ is the quadratic extension of  $F$ such that  $ N_{E/F}(E^*)=\text{Im}(q_{an,2})^*$ (see Definition \ref{def-ensembleS}). Denote by $\varpi_E$ the quadratic character associated to $E$, that is   $\varpi_E(a)=1$ if $a\in N_{E/F}(E^*)$ and  $\varpi_E(a)=-1$  if $a\notin N_{E/F}(E^*)$ (hence if $E=F[\sqrt{\xi}]$, one has  $\varpi_E=\chi_\xi$).
Then 

$\ga_k(\underline{a},c_k)=$
\begin{align*}
&- \ga_\psi(q_e)^k \text { \rm if }  e=0\;{\rm or  }\; e=4\;  (\text{case where}\; \Sr_e=\{1\})\\
&- \ga_\psi(q_e)^k \varpi_E(c_k)^k\prod_{j=0}^{k-1}  \varpi_E(a_j)  \text { \rm if }  e=2 \;\; \;  (\text{case where}\; \Sr_e=F^*/N_{E/F}(E^*))\\
 &-   \prod_{j=0}^{k-1}\ga_\psi(q_e) ((-1)^{\frac{d-1}{2}}disc(q_e), a_jc_k )_\Cr \al(a_jc_k )\al(-1)\text{ \rm if  }  e=1\;{\rm or  }\; e=3, (\text{case where}\; \Sr_e=\Cr=F^*/F^{*2}).\\
   \end{align*}
($disc(q_e)$ is the  discriminant of $q_e)$.
 
\end{prop}
\begin{proof} Let $u'\in \underline{\cO}^-(\underline{a})$  and  $v\in \tilde{\go g}^{\la_k}=FX_k$ such that  $\De_k(v)c_k\in S_e$. This implies that  $v=c_k\nu X_k$ with  $\nu\in S_e$. Set $Y_{\underline{a}}=a_0Y_0+\ldots a_{k-1} Y_{k-1}$. We will prove simultaneously the two assertions.\me

 Let us first show that there exists $p\in \tilde{P}$ such that  $pu'=Y_{\underline{a}}$ and  $pX_k=X_k$. 

   From Lemma \ref{lem-NO+}    applied to the algebra  $\underline{\tilde{\go g}}$, there exists  $\underline{n}\in\underline{N}$  and  scalars  $y_0,\ldots y_{k-1}$ such that   $\underline{n}u'=y_0Y_0+\ldots+y_{k-1}Y_{k-1}$.  From the definition  of  $\underline{N}$, one has  $\underline{N}\subset N$ and  $\underline{n}.v=v$. 
 
 As  $y_0Y_0+\ldots+y_{k-1}Y_{k-1}\in \underline{\cO}^-(\underline{a})$, there exists $\mu_j\in S_e$ such that $y_j=\mu_j a_j$ for $j=0,\ldots, k-1$. 

From Remark \ref{rem-image-muj}, the elements  $\mu_0^{-1}X_0+\ldots +\mu_{k-1}^{-1} X_{k-1}+X_k$ and $X_0+\ldots +X_k$ are  $\tilde{L}$-conjugated. Hence, there exists  $l\in \tilde{L}$ such that  $x_j(l)=\mu_j^{-1}$ for $j=0,\ldots, k-1$ and  $x_k(l)=1$. The element $p=l\underline{n}$ satisfies then the required property.\me

Then  $pu'=Y_{\underline{a}}$ and  $pv=v=\nu c_k X_k$ with $p\in \tilde{P}$.

Hence the quadratic form  $Q_{u',v}$ is equivalent to the form $c_k \nu  Q_{Y_{\underline{a}}, X_k}$, and therefore
$$\ga(u',v)= \ga_\psi(c_k \nu   Q_{Y_{\underline{a}}, X_k}).$$

The space $K_{k-1}(1,-1)$ (on which the forms live) is the direct sum of the spaces $E_{j,k}(1,-1)$ for $j=0,\ldots, k-1$.  These spaces are orthogonal for  $Q_{a_0Y_0+\ldots+a_{k-1}Y_{k-1}, X_k}$. Moreover, if  $A\in E_{j,k}(1,-1)$, we have  $b((\ad A)^2 Y_i, X_k)=0$ for  $j\neq i$. Therefore
$$Q_{u',v}\sim (\oplus _{j=0}^{k-1} c_k\nu  a_jQ_{j,k}),\quad {\rm where }\; Q_{j,k}(A)=\frac{1}{2} b((\ad A)^2Y_j,X_k),\; {\rm for }\; A\in E_{j,k}(1,-1).$$
Using    Remark \ref{rem-lienmuller} we see that for  $Y\in E_{j,k}(-1,-1)$, one has  $Q_{j,k}([X_j,Y])=q_{X_j,X_k}(Y)$ where $q_{X_i,X_j}$ is defined  as in (\refeq{def-qXiXj}). As $\ad X_j$ is an isomorphism from 
$E_{j,k}(-1,-1)$ onto  $E_{j,k}(1,-1)$, the quadratic forms  $Q_{j,k}$ and  $q_{X_j,X_k}$ are equivalent. And as all the forms  $q_{X_i,X_j}$ are equivalent, we obtain that   $Q_{j,k}\sim q_{X_0,X_1}=q_e$.
 
From (\cite{R-S} Proposition I-3)
 we obtain that
$$\ga(u',v)=\prod_{j=0}^{k-1}\ga_\psi(c_k a_j\nu  q_e).$$
 

If $e$ is odd then  $\nu \in S_e=F^{*2}$  and hence  $\ga_\psi(c_k a_j\nu q_e)=\ga_\psi(c_ka_j q_e)$. Moreover is $e$ is odd then $d$ is odd too (see  table 1 in \cite{B-R}). As the rank of $q_e$ is  $d$, the result is a consequence of  Lemma  \ref{lem-gamma(xQ)}  (2).
\me

If $e=0$ or $4$, the form  $q_e$ is the sum of an anisotropic quadratic form of rank $e$  and a hyperbolic form of rank $d-e$. As all the anisotropic quadratic forms of  rank $4$ are equivalent, and as the same is true for the hyperbolic forms of rank  $d-e$, we obtain that  $c_k a_j\nu q_e$ is equivalent to $q_e$, and therefore $\ga_\psi(c_k a_j\nu q_e)=\ga_\psi(q_e)$.  This gives the result in that case.\me

Suppose now that $e=2$. Then $q_{2}$ is equivalent to $q_{an,2}+q_{hyp,2}$ where $q_{hyp,2}$ is a hyperbolic form of rank $d-2$ and $q_{an,2}$  is an anisotropic form of rank  $2$ such that  $\text{Im}(q_{an,2})^*=N_{E/F}(E^*)$. It follows that  $q_{an,2}$ is equivalent to   $x^2-\xi y^2$. Therefore $(-1)^{d/2} disc(q_2)= -disc(q_{an,2})=\xi$.  As  $\chi_{\xi}$ coincides with the quadratic character  $\varpi_{E}$  associated to $E$ and as  $\nu \in S_e=N_{E/F}(E^*)$,  Lemma  \ref{lem-gamma(xQ)}  (1)    implies that 
$$\ga_\psi(c_k a_j\nu   q_2)=\ga_\psi(q_2) (c_ka_j \nu, \xi)_\Cr= \ga_\psi(q_2) \varpi_E(c_k a_j).$$
Hence
$$\ga(u',v)=\prod_{j=0}^{k-1} \ga_\psi(q_2) \varpi_E(c_k a_j),$$
and this ends the proof.

 \end{proof}

 \section{Functional equation of the zeta functions associated to   $(\tilde{P},V^+)$}\label{sec:zetaPH}

\subsection{The (A2') Condition in the Theorem  $k_p$ of  F. Sato.}
\hfill
\vskip 5pt

{\bf In this paragraph, we suppose that $F$ is algebraically closed}. \me

A maximal split abelian subalgebra  $\go a $ of $ {\go g}$ is then a Cartan subalgebra of  $\tilde{\go g}$ and of $\go{g}$. We  denote by  $\tilde{\Si}$ and $\Sigma$ the corresponding root systems. Let  $\Sigma^+\subset\tilde{\Si}^+$ the  positive subsystems  defined   in Theorem \ref{thbasepi}. Then  $\mu\in \tilde{\Si}^+\setminus \Sigma^+$  if and only if    $\tilde{\go g}^\mu\subset V^+$ and   $\mu\in  {\Si}^+$ if and only if $\tilde{\go g}^\mu= \go{g}^\mu\subset E_{i,j}(1,-1)\subset \go g$ for  $i>j$ (Proposition \ref{proplambda>0}). 

\begin{prop}\label{recurrence} Let  $k\geq 1$.  Let  $\mu\in \tilde{\Si}^+\setminus \Sigma^+$. Let  $H_\mu$ be the coroot of $\mu$ and let    $\{X_{-\mu}, H_\mu, X_\mu\}$ be an $\go sl_2$-triple. Let us denote by  $\tilde{\go l}[\mu]$ the subalgebra (isomorphic to $\go{sl}_{2}(F))$ generated by   $FX_{\pm\mu}$ and set 
$$\tilde{\go g}[\mu]=Z_{\tilde{\go g}}(\tilde{\go l}[\mu]).$$ 

If $\tilde{\go g}[\mu]\cap V^+\neq \{0\}$ then  
\begin{enumerate}\item the Lie algebra $\tilde{\go g}[\mu]$ is reductive and $3$-graded  by the element  $H_0-H_\mu$. 
The corresponding grading will be denoted by
$$\tilde{\go g}[\mu]=V[\mu]^-\oplus \go{g}[\mu]\oplus V[\mu]^+.$$
\item  The algebra  $\go a[\mu]= \{H\in \go a, \mu(H)=0\}$ is a Cartan subalgebra of  $\tilde{\go g}[\mu]$ and the corresponding root system is given by  
$$\tilde{\Si}[\mu]=\{\al\in \tilde{\Si}; (\al,\mu)=0\}.$$
We will set: $\tilde{\Si}[\mu]^+=\tilde{\Si}[\mu]\cap  \tilde{\Si}^+$. \me

   \item Moreover, the Lie algebra $\tilde{\go g}[\mu]$ is regular  of rank   $k$ if $\mu$ is a long root, and of  rank of  $k-1$ if $\mu$ is a short root.
\end{enumerate}
 \end{prop}
\begin{proof} Although the root $\mu $ is here arbitrary, the proof of the 2 first points is similar to  Proposition \ref{prop-gtilde(1)}).\me

If $\mu$ is a long root of $\tilde{\Si}$ then by  Proposition \ref{prop-systememax}, there exists a element $w$ of the Weyl group such that $w\mu=\la_0$. The third assertion  follows from   Corollary  \ref{cor-decomp-j}.\me

Suppose that $\mu$ is a short root in $\tilde{\Si}$. From  Table  1 in  section \ref{sec:section-Table}, the root system $\tilde{\Si}$ is of type $B_n$ or $C_n$ with $n\geq 2$.\me

If $\tilde{\Si}$ is of type $B_n$ then $k=1$ and $\mu=\dfrac{\la_0+\la_1}{2}$. As $\la_0$ and $\la_1$ are long and orthogonal, we have  $H_\mu=H_0$ and then  $\tilde{\go g}[\mu]\cap V^+=\{0\}$.\me

As we suppose that  $\tilde{\go g}[\mu]\cap V^+\neq \{0\}$ the root system $\tilde{\Si}$ is of type $C_n$ with $n=k+1\geq 3$ and we have 
$$\tilde{\Si}^+=\{\frac{\la_i+\la_j}{2}; 0\leq i\leq j\leq k\}.$$

For our purpose, we can suppose that $\mu=\dfrac{\la_k+\la_{k-1}}{2}$. Hence $H_{\mu}=H_{\la_k}+H_{\la_{k-1}}$ and $H_0-H_\mu=H_{\la_0}+\ldots +H_{\la_{k-2}}$. Then $\tilde{\Si}[\mu]$ is equal to $\{\pm\la_0\}$ if $k=2$ and $\tilde{\Si}[\mu]$ is of type $C_{k-1}$ if $k\geq 3$. This gives the last assertion. 

\end{proof}

Define 
$$N_{0}\index{N0@$N_{0}$}=\exp \ad\,\go{n}_{0},\text{ \hskip 20pt where } \go{n}_{0}\index{ngothique0@$\go{n}_{0}$}=\oplus_{\mu\in{\Si}^+}\tilde{\go g}^{-\mu}.$$
\begin{prop}\label{Prop231} (compare with \cite{Mu} Lemme 2.2.1 and   Proposition 2.3.1).\hfill

 Let $x\in V^+$. Then there exists a family of two by two strongly orthogonal roots  $\{\mu_1,\ldots, \mu_r\}$ such that 
$$N_{0}.x\cap (\oplus_{j=1}^r\tilde{\go g}^{\mu_j})\neq \emptyset.$$
\end{prop}

\begin{proof} The proof goes by induction on the    rank of $\tilde{\go g}$.  The result is clear for  $k=0$ since $V^+=\tilde{\go g}^{\la_0}$.

 Suppose now that the result is true  for any irreducible  graded regular algebra of rank  strictly less than $k+1$.\me

If $\mu\in \tilde{\Si}^+\setminus {\Si}^+$, then  $\mu(H_0)=2$, and hence for all  $\mu,\mu'$ in  $\tilde{\Si}^+\setminus {\Si}^+$, $\mu+\mu'$ is never  a root. Therefore  $n(\mu,\mu')\geq 0$. From the classification of the 3-graded  Lie algebras we consider, we know that $G_{2}$ does never occur and hence $n(\mu,\mu')\in\{0,1,2\}$ for  $\mu,\mu'\in \tilde{\Si}^+\setminus {\Si}^+$.\me

We will also use the fact that for $\mu,\mu'$ in  $\tilde{\Si}^+\setminus {\Si}^+$, we have $\mu\perp \mu'  \Longleftrightarrow \mu\sorth \mu'$ (the proof is similar to the proof of Corollary \ref{cor-orth=fortementorth}).\me

Let  $x\in V^+$.  Then  $\displaystyle x=\sum_{\mu\in {\tilde{\Si}^+\setminus {\Si}^+ }}  X_\mu$ with  $X_\mu\in\tilde{\go g}^\mu\subset V^+$.   We denote  $$s(x)=\{\mu\in  \tilde{\Si}^+\setminus {\Si}^+ , X_\mu\neq 0\}.$$

\no{\bf 1st case: }  We suppose that there exists a long root  $\mu_0$ among the roots of maximal  height in $s(x)$.  \me

For  $j=0,1,2$, we set  $s(x)_j=\{\mu\in s(x); n(\mu,\mu_0)=j\}$.  As  $\mu_0$ is a long root,    $n(\mu,\mu_0)=2$ if and only if  $\mu=\mu_0$. Hence 
 \beq\label{s(x)} s(x)=\{\mu_0\}\cup s(x)_0\cup s(x)_1.\eeq

Let  $x=X_{\mu_0}+\sum_{\mu\in s(x)_0} X_\mu+\sum_{\mu\in s(x)_1} X_\mu$.\me

 If $s(x)_1=\emptyset$ then $s(x-X_{\mu_0})\subset \mu_0^\perp$. From Proposition \ref{recurrence}, one has  $x-X_{\mu_0}\in V[\mu_0]^+$ and we obtain the result by induction from $\tilde{\go g}[\mu_0]$, which is of rank $k$ (as $\go n[\mu_0]\subset \go n$ from our definition of $\tilde{\Si}[\mu_0]^+$).

 We suppose that $s(x)_1\neq \emptyset$. 

For  $\mu\in s(x)_1$, we have  $\mu-\mu_0\in - {\Si}^+$  (it is a root because $n(\mu,\mu_0)=1$ and it is negative because  $\mu_0$ is of maximal height).  Let us fix  $Z_{\mu-\mu_0}\in\tilde{\go g}^{\mu-\mu_0}\subset {\go n}_{0}$ such that  $[Z_{\mu-\mu_0}, X_{\mu_0}]=-X_\mu$.\me

Define  $A=\sum_{\mu\in s(x)_1} Z_{\mu-\mu_0}\in\go n_{0}$. \me

Let   $\mu\in s(x)_1$ and  $\mu'\in s(x)_0$. Suppose that  $\mu-\mu_0+\mu'\in \tilde{\Si}$. As   $n(\mu-\mu_0+\mu', \mu_0)=-1$, then the linear form   $\mu-\mu_0+\mu'+\mu_0=\mu+\mu'$ would be a root of 
$\tilde{\Si}^+$ and this is not  possible. Therefore  $\mu-\mu_0+\mu'\notin \tilde{\Si}$ and hence  $[Z_{\mu-\mu_0},X_{\mu'}]=0$. Then 
 $\sum_{\mu'\in s(x)_0} [A,X_{\mu'}]=0$, and this implies that  $$e^{\text{ad} A} (\sum_{\mu'\in s(x)_0} X_{\mu'})=\sum_{\mu'\in s(x)_0} X_{\mu'}.$$

Take  $\mu$ and $\mu'$ in  $ s(x)_1$. Then   $n(\mu-\mu_0+\mu', {\mu}_0)=0$. And therefore the element  $y= [A,\sum_{\mu'\in s(x)_1}X_{\mu'}]=\sum_{\mu\in s(x)_1}\sum_{\mu'\in s(x)_1}[Z_{\mu-\mu_0},X_{\mu'}]$ is such that  $s(y)\subset \mu_0^\perp$. Considering the preceding case one obtains  $\text{ad}(A)^2  \sum_{\mu'\in s(x)_1}X_{\mu'}=[A,y]=0$. Then 
$$e^{\text{ad}A} (\sum_{\mu'\in s(x)_1}X_{\mu'})= \sum_{\mu'\in s(x)_1}X_{\mu'}+y,\;\text{where}\; s(y)\subset \mu_0^\perp.$$

As   $[A,X_{\mu_0}]=\sum_{\mu\in s(x)_1} [Z_{\mu-\mu_0},X_{\mu_0}]=-\sum_{\mu\in s(x)_1}X_\mu$, we have
$$e^{\text{ad}\; A}X_{\mu_0}=X_{\mu_0}-\sum_{\mu\in s(x)_1}X_\mu-\frac{1}{2}[A,\sum_{\mu\in s(x)_1}X_\mu]=X_{\mu_0}-\sum_{\mu\in s(x)_1}X_\mu-\frac{y}{2}.$$

Finally we get
\begin{align*}
 e^{\text{ad}\; A}  x&= e^{\text{ad}\; A} X_{\mu_{0}}+e^{\text{ad}\; A}(\sum_{\mu'\in s(x)_0}X_{\mu'})+e^{\text{ad}\; A}(\sum_{\mu'\in s(x)_1}X_{\mu'}) \\
 &=(X_{\mu_0}-\sum_{\mu\in s(x)_1}X_\mu-\frac{y}{2})+ (\sum_{\mu'\in s(x)_0}X_{\mu'})+( \sum_{\mu'\in s(x)_1}X_{\mu'}+y)\\
 &=X_{\mu_0}+\sum_{\mu'\in s(x)_0} X_{\mu'}+\frac{y}{2},\quad\text{where}\; s(y)\subset \mu_0^\perp.
 \end{align*}

If we set  $y_0=\sum_{\mu'\in s(x)_0} X_{\mu'}+\frac{y}{2}$, then  $s(y_0)\subset \mu_0^\perp$ and hence  $y_0\in V[\mu_0]^+$. Again the result is obtained by induction from  $\tilde{\go g}[\mu_0]$. \me

 \no{\bf 3rd case:} We suppose that all roots of maximal height in  $s(x)$ are short. From Table 1 in section \ref{sec:section-Table}, we know that the only cases were there are roots of different length correspond to cases where $\tilde{\Si}$ is of type  $B_n$  or of type $C_n$ for  $n\geq 2$.\me
 
$\bullet $ If $\tilde{\Si}$ is of  type $C_n$,  we have  $\tilde{\Si}^+\setminus \Sigma^+ =\{ \la_j; 0\leq j\leq k\}\cup \{\dfrac{\la_j+ \la_i}{2}; 0\leq i<j\leq k\}$ where 
 $n=k+1$. We also have  $\la_k\notin s(x)$ (because the long root  $\la_k$ is the greatest root in  $\tilde{\Si}^+$). \me
 
  If  $s(x)\subset \la_k^\perp$ then $x\in V[\la_k]^+$. Again  the result  is obtained by induction from $\tilde{\go g}[\la_k]$.\me

 We suppose that $s(x)\not\subset \la_k^\perp$.  This implies that there exists  $i<k$ such that  $ \dfrac{\la_k+\la_i}{2}\in s(x)$.  
 
 If $k=1$, then  $x\in \tilde{\go g}^{\mu_0}$ and the result follows.
 
 From know on, we suppose that $k\geq 2$.
 Define  
 
 \centerline{$i_0=\max \{i; \dfrac{\la_k+\la_i}{2}\in s(x)\}$ and  $\mu_0=\dfrac{\la_k+\la_{i_0}}{2}$.\me}
 
Hence
 $$s(x)\subset\{\mu_0\}\cup  \{\dfrac{\la_i+\la_k}{2}, i< i_0\}\cup\{\dfrac{\la_i+\la_{j}}{2}; 0\leq i\leq j<k\}.$$
And therefore one can write 
 $$x=X_{\mu_0}+x_1+x_2,\;\text{with} \; s(x_1)\subset  \{\dfrac{\la_i+\la_k}{2}, i< i_0\}\;\text{and }\; s(x_2)\subset \{\dfrac{\la_i+\la_{j}}{2}; 0\leq i\leq j<k\}.$$

 If  $\mu=\dfrac{\la_i+\la_{k}}{2}$,  with  $i<i_0$, belongs to   $s(x)$, one has  $\mu-\mu_0\in \tilde{\Si}$ (as $n(\mu,\mu_0)=1$).

 Then  we fix the elements  $Z_{\mu-\mu_0}\in\tilde{\go g}^{\mu-\mu_0}\subset E_{i, {i_0}}(1,-1)\subset \go n_{0}$   such that  $[Z_{\mu-\mu_0},X_{\mu_0}]=-X_\mu$.\me
 
   Let us define  $A=\sum_{\mu=\frac{(\la_i+\la_{k})}{2},0\leq i<i_0} Z_{\mu-\mu_0}\in \go n_{0}$ and  $n_0=e^{\text{ad}A}$. One has then  $[A,X_{\mu_0}]=-x_1$ and  $[A,x_1]\in \oplus_{i<i_0, j<i_0}[E_{i, {i_0}}(1,-1), E_{j,k}(1,1)] =\{0\}$.    Therefore  $(\text{ad} A)^2 x_1=0$.

  Let  $i<i_0$. For $ r,s\leq k$, one has $[E_{i,i_0}(1,-1), E_{r,i_0}(1,1)]\subset E_{i,r}(1,1)$ and     $[E_{i,i_0}(1,-1), E_{r,s}(1,1)]= \{0\}$ for $r\neq i_0$ and $s\neq i_0$. 
 As $A\in \oplus_{ i<i_0} E_{i,i_0}(1,-1)$, we deduce that $e^{\ad A}$ normalizes $\oplus_{r\leq s<k}E_{r,s}(1,1)$
And therefore 
 $$n_0.x=X_{\mu_0}+y\;\text{ with }\; s(y)\subset \{\dfrac{\la_i+\la_{j}}{2}; 0\leq i\leq j<k\}.$$

Set $y=y_0+y_1$ with  $s(y_0)\subset \mu_0^\perp=\{\dfrac{\la_i+\la_j}{2};    \{i,j\}\cap\{i_0,k\}=\emptyset\}$ and  $s(y_1)\subset \{\dfrac{\la_i+\la_{j}}{2};  1\leq i\leq j<k; i_0\in\{i,j\}\}$. 

Under the action of  $H_{\mu_0}$ the elements in $E_{i,k}(1,-1)$ are of weight $0$ if $i=i_{0}$ and of weight $-1$ if $i\neq i_{0}$.
This implies that  $\text{ad}(X_{\mu_0})$ is surjective  from  $\oplus_{i=0}^{k-1} E_{i,k}(1,-1)$ onto   $\oplus_{i=0}^{k-1} E_{i,i_0}(1,1)$. Hence there exists  $B\in \oplus_{i=0}^{k-1} E_{i,k}(1,-1)$ such that $[X_{\mu_0}, B]= y_1$. 

As  $[E_{i,k}(1,-1), E_{r,s}(1,1)]=\{0\}$  for  $i<k$ and  $r\leq s< k$, one has  $[B,y_1]=[B,y_0]=0$. 

Therefore the element  $n_1=e^{\text{ad} B}\in N_{0}$ is such that  $n_1n_0x= n_1(X_{\mu_0}+y_0+y_1)= X_{\mu_0}+y_0$ with  $s(y_0)\subset \mu_0^\perp$. We obtain the result by applying the induction to the element  $n_1n_0x-X_{\mu_0}\in\tilde{\go g}[\mu_0]$. \me

$\bullet $ If  $\tilde{\Si}$ is of type  $B_n$ for  $n\geq 2$, then  $k=1$ and  $V^+=\tilde{\go g}^{\la_0}\oplus E_{0,1}(1,1)\oplus \tilde{\go g}^{\la_1}$.

 From the tables and notations in  (\cite{Bou1} Chapitre VI) , if we take  $\tilde{\Si}^+=\{\ep_i\pm \ep_j; 1\leq i<j\leq n\}\cup\{\ep_i; 1\leq i\leq n\}$, then  $\la_0=\ep_1-\ep_2=\al_1$, and  $\la_1=\ep_1+\ep_2$ is the greatest root. Therefore   $\tilde{\go g}^\mu\subset V^+$ if and only if $\mu=\ep_1$ or $\mu=\ep_1\pm\ep_j$ for $j\geq 2$. \me

 Moreover,  $\mu_0=\dfrac{\la_0+\la_1}{2}=\varepsilon_{1}$ is the unique short root in  $\tilde{\Si}^+$ such that  $\tilde{\go g}^{\mu_0}\subset V^+$. 
  As $\la_0$ and $ \la_1$ are orthogonal, one has $H_{\mu_0}=H_{\la_0}+H_{\la_1}=H_0$.\me

From the assumption on  $s(x)$, the root    $\mu_0$ is of maximal height in   $s(x)$. Let $\mu\in s(x)$, $\mu\neq\mu_0$. As  $n(\mu,\mu_0)=2$ , one has $\mu-\mu_0\in \tilde{\Si}^-$. Let us fix  $Z_{\mu-\mu_0}\in \tilde{\go g}^{\mu-\mu_0}\subset  \go n_{0}$ such that  $[Z_{\mu-\mu_0}, X_{\mu_0}]=-X_\mu$ and set  $A=\sum_{\mu\in s(x),\mu\neq\mu_0} Z_{\mu-\mu_0}$. \me

As $\tilde{\Si}$ is of  type $B_n$, we have that if  $\mu\neq \mu_0$ is such  that  $\mu(H_0)=2$, the root  $\mu-\mu_0$ is short and orthogonal to  $\mu_0$ (in the preceding  notations, one has  $\mu-\mu_0=\pm\ep_i$ for $i\geq 2$). It follows that if  $\mu$ and $\mu'$ are two elements of  $s(x)$, distinct from $\mu_0$, then  $\mu-\mu_0$ is a short negative root and hence  $\mu+\mu'-\mu_0\notin\tilde{\Si}$ ( because $\mu-\mu_0=-\ep_i$ with  $i\geq 2$ and  $\mu'=\ep_1-\ep_j$ with  $j\geq 2$). \me

Therefore
  $[A,\sum_{\mu\in s(x), \mu\neq \mu_0}X_\mu]=0$.  Then if  $n=e^{\text{ad}\; A}$, we have  $n.x=X_{\mu_0}$.\me

This ends the proof.

\end{proof}
 As $F$ is algebraically closed, the group   $P$ has a unique open orbit in  $V^+$ given by 
$$\cO^+=\{ X\in V^+; \prod_{j=0}^k \De_j(X)\neq 0\}.$$

We define  $\cS=V^+-\cO^+$. If $\chi_{j}$ ($j=0,\ldots,k$) is  the character of $P$ corresponding to $\Delta_{j}$, we denote by ${\X}_{V^+}(P)$ the group of characters of $P$ generated by the ${\chi_{j}}'s.$
\begin{theorem}\label{ConditionA2} {\rm (Condition (A2') \index{condition@Condition (A2')} of F. Sato)(see \cite{Sa} \textsection 2.3)}
\begin{enumerate}\item $P$ has a finite number of orbits in $\cS$ (and hence in $V^+$).

\item Let $x\in \cS$ and let  $P_{x}^0$ the identity component of the centralizer of $x$ in  $P$. Then there exists a non trivial character  $\chi$ in $\X_{V^+}(P)$ and  $p\in P_{x}^0$ such that  $\chi(p)\neq 1$. \end{enumerate}
\end{theorem}
\begin{proof}
Of course $N_{0}\subset P$ and   $P=LN_{0}$. From Proposition \ref{Prop231}, for any element $x$ in  $\cS$, there exists a family $\cR\subset \tilde{\Si}^+_2$ of strongly orthogonal roots such that  $N_{0}.X\cap (\oplus_{\mu\in\cR} \tilde{\go g}^\mu)\neq \emptyset$.

As $\tilde{\Si}^+\setminus \Sigma^+$ is finite, there are also only a finite number  of families of strongly orthogonal roots in  $\tilde{\Si}^+\setminus \Sigma^+$. Therefore, in order to prove assertion 1), it will be enough to show that  if $\cR=\{\mu_1,\ldots\mu_r\}\subset \tilde{\Si}^+\setminus \Sigma^+ $ is such a family of strongly orthogonal roots then $\oplus_{j=1}^r \tilde{\go g}^{\mu_j}$ is included in a finite number of $L$-orbits .  

Fix  such a family $\cR=\{\mu_1,\ldots\mu_r\}$. If  $\mu\in\cR$, we denote by  $H_\mu$ the coroot  $\mu$ and by  $\tilde{\go l}[\mu]$ the algebra generated by  $ \tilde{\go g}^{\pm\mu}$, that is  $\tilde{\go l}[\mu]=\tilde{\go g}^{-\mu}\oplus FH_\mu\oplus \tilde{\go g}^{\mu}$. Let  $L[\mu]$ the group associated to this graded algebra. As $F$ is algebraically closed, the group  $L[\mu]$ is the centralizer in  $\text{Aut}_e(\tilde{\go l}[\mu])$ of  $H_\mu$, and hence  $L[\mu]\subset \text{Aut}_e(\tilde{\go g})$. \me

Let us show that  $L[\mu]\subset L$.    Let  $l\in L[\mu]$. Then  $l.H_\mu=H_\mu$. 
If $Z\in \go{a}$ is  orthogonal to $\mu$ then $[\tilde{\go l}[\mu],Z]=\{0\}$ and therefore $l.Z=Z$. If  $\mu$ is one of the roots $\la_j$,  we get immediately that  $l\in L$. If not, there exist $i<j$ such that  $\tilde{\go g}^\mu\subset E_{i,j}(1,1)$. Then  $\mu(H_\mu-2H_{\la_i})=\mu(H_\mu-2H_{\la_j})=0$, and then   $l.H_{\la_i}=H_{\la_i}$ and $l.H_{\la_j}=H_{\la_j}$, and $\mu(H_{\la_s})=0$ for $s\neq i,j$. Therefore $l$ centralizes $\go a^0$, and hence   $L[\mu]\subset L$. \me

From   Theorem \ref{th-k=0}, each $L[\mu_j]$ has a finite number of orbits in  $\tilde{\go g}^{\mu_j}$. As the roots  $\mu_j$   are strongly orthogonal, the subgroups  $L[\mu_j]$ commute 2 by 2. Therefore the group $L[\mu_1]\ldots L[\mu_r]\subset L$ has a finite number of orbits in  $\oplus_{\mu\in\cR} \tilde{\go g}^\mu$. This proves  assertion 1).\me

\no Assertion  2) will be proved by induction on $k$. The case $k=0$ is obvious. Suppose that  $k\geq 1$.

We suppose that assertion 2)  is true  for all regular graded algebras (in the sense of  Definition \ref{def-alg-grad}) of rank $k$
 and we will prove it if  $\tilde{\go g}$ is of rank  $k+1$.\me

Remember the notations  \ref{notation-Delta-souligne}, that is 
$$\underline{\tilde{\go g}}=\tilde{\go g}[\la_k],\;\;\text{et}\;\; \underline{V}^\pm=V[\la_k]^\pm.$$
( where  $\tilde{\go g}[\la_k]=Z_{\tilde{\go g}}(\tilde{\go l}_k)$ and  $\tilde{\go l}_k=\tilde{\go g}^{-\la_k}\oplus [\tilde{\go g}^{-\la_k},\tilde{\go g}^{\la_k}]\oplus \tilde{\go g}^{\la_k}$).\\
The Lie algebra $\underline{\tilde{\go g}}$, which is of rank $k$,  is graded by the element  $\underline{H}_0=\sum_{j=0}^{k-1} H_{\la_j}$. We will denote by underlined letters  all the elements or spaces attached to  $\underline{\tilde{\go g}}$.

As we have supposed that  $F$ is algebraically closed, we have  $\text{Aut}_0(\tilde{\go g})=\text{Aut}_e(\tilde{\go g})$ and hence   $L$ is the centralizer  of $\go a^0=\oplus_{j=0}^k F H_{\la_j}$ in $\text{Aut}_e(\tilde{\go g})$. Let   $\underline{l}\in\text{Aut}_e(\underline{\tilde{\go g}})$. As  $[\tilde{\go l}_k,\underline{\tilde{\go g}}]=\{0\}$, we have  $\underline{l}.H_{\la_k}=H_{\la_k}$ and  $\underline{l}$  acts trivially on  $\tilde{\go g}^{\la_k}$. 

Therefore any element of  $\underline{L}$ defines an element in  $L$ which acts trivially on   $\tilde{\go g}^{\la_k}$. 
 
Remember that  $\chi_j$ is the character of $P$ attached to  $\De_j$ ($j=0,\ldots,k$) and that ${\X}_{V^+}(P)$ is the group of characters of $P$ generated by the ${\chi_j}'s$.  
 
Let  $l\in \underline{L}\subset L$ and $j=0,\ldots,k-1$. By Lemma \ref{lem-decomposition-Delta},  we have  $\underline{\chi}_j(l)=\chi_j(l)$ and $\chi_k(l)=1$.

As  $\underline{\go n}=\oplus_{i<j<k}E_{i,j}(1,-1)\subset \go n$, we have  $\underline{P}\subset P$  and any element of  $\underline{P}$ acts trivially on $\tilde{\go g}^{\la_k}$. As any character in $\X_{V^+}(P)$  is trivial on  $N$, the mapping  $\underline{\chi}_j\longrightarrow \chi_j $  ($j=0,\ldots,k-1$) can be seen as an inclusion $\X_{\underline{V}^+}(\underline{P})\hookrightarrow\X_{V^+}(P)$.\me

Let  $x\in \cS$. By Proposition  \ref{Prop231}, there exists $n_0\in N_{0}$ and a family  $\cR=\{\mu_1,\ldots \mu_r\}$ of stronly orthogonal roots in  $\tilde{\Si}^+\setminus \Sigma^+$ such that 
the element  $X=n_0.x$ belongs to  $\tilde{\go g}^{\mu_1}\oplus\ldots \oplus\tilde{\go g}^{\mu_r}$. We suppose that  $r$ is minimal, that is   $X=X_{\mu_1}+\ldots +X_{\mu_r}$ with  $X_{\mu_j}\in \tilde{\go g}^{\mu_j}-\{0\}$. As $n_0^{-1}P_X^0n_0=P_x^0$, it is enough to prove assertion 2) for  $X$.\me

 For  $j=1,\ldots,r$, let   $H_{\mu_j}$ be the coroot of $\mu_j$ and fix  $X_{-\mu_j}\in \tilde{\go g}^{-\mu_j}$ such that  $\{X_{-\mu_j}, H_{\mu_j}, X_{\mu_j}\}$ is an  $\go sl_2$-triple. Define  $$h=H_{\mu_1}+\ldots H_{\mu_r}.$$

As the roots  $\mu_j$ are strongly orthogonal, if we set   $Y=X_{-\mu_1}+\dots +X_{-\mu_r}\in V^-$ and $X=X_{\mu_1}+\dots +X_{\mu_r}\in V^+$, then $\{Y,h,X\}$ is an  $\go sl_2$-triple.  If  $Z\in V^+$, then  $[X,Z]=0$ and  $\ad(Y)^3Z=0$. Therefore the element  $Z$ generates a submodule of dimension at most  $3$ under the action of this triple. It follows  that  the weights of  $\text{ad}(h)$ in  $V^+$ belong to  $\{0,1,2\}$. \me

Remember  from Definition \ref{def-hx}, that if   $\{Z^-,u,Z\}$ is an $\go sl_2$-triple where  $Z^-\in V^-$, $u\in \go a$, and  $Z\in V^+$, we denote by  $h_Z(t)$  the element in $L$ which acts by   $t^m$ on the space of weight $m$ under  $\text{ad}\; u$.\me

\no{\bf 1st case :} If  $\De_k(X)\neq 0$ then  $\la_k\in \cR$. Hence   $X=Z_0+Z_2$ with  $Z_0 \in\underline{V}^+$ and $Z_2\in \tilde{\go g}^{\la_k}$.
As  $X\in \cal S$ and $\De_k(X)\neq 0$, there exists  $j<k$ such that  $\De_j(X)=0$. By Lemma \ref{lem-decomposition-Delta}, we have   $\De_k(X)=\De_k(Z_2)\neq 0$ and  $\De_j(Z_0+Z_2)=\underline{\De}_j(Z_0)\De_k(Z_2)=0$, and hence  $Z_0\in\underline{\cS}$. 
By induction, there exists    $\chi \in\X_{\underline{V}^+}(\underline{P})\subset \X_{ {V}^+}( {P})$ and  $p\in \underline{P}_{Z_0}^0\subset P$ such that  $\chi(p)\neq 1$. As $p\in \underline{P}$, one has  $p.Z_2=Z_2$, and hence  $p\in P_{X}^0$. This gives the result in that case.  \me

\no {\bf 2nd  case}. Suppose  $\De_k(X)=0$ and  $X\in \underline{V}^+$ (this is equivalent to the condition   $\la_k(h)=0$). 
Then for all $t\in F^*$, the element  $l_t=h_{X_{\la_k}}(t)\in L$ satisfies   $l_t.X=X$ and $\chi_k(l_t)=t^2$.  Hence $l_t\in P_X^0$ and   for  $t\neq\pm 1$, we have  $\chi_k(l_t)\neq 1$.\me

\no {\bf 3rd case}. Suppose $\De_k(X)=0$ and  $X\notin \underline{V}^+$ (this is equivalent to $\la_k\notin\cR$ and  $\la_k(h)=1$ or  $2$).\me

If  $\la_k(h)=1$ then there exists a unique root  $\mu$ in  $\cR$ such that  $\lambda_{k}(H_{\mu})=n(\la_k,\mu)=1$. One can suppose that  $\mu=\mu_1$ and then, for $j\geq 2$,  the root  $\mu_j$   is orthogonal to $\lambda_{k}$.
As $\la_k$ is a long root, we have  $n(\mu_1,\la_k)=1$.  It follows that $[H_{\la_k},X_{\mu_1}]=X_{\mu_1}$ and  $[H_{\la_k},X_{\mu_j}]=[H_{\mu_1},X_{\mu_j}]=0$ for $j\geq 2$. For $t\in F^*$ we define    $l_t=h_{X_{\la_k}}(t)^2 h_{X_{-\mu_1}}(t)$. Then $l_t.X=X$   (because  $h_{X_{-\mu_1}}(t).(X_{\mu_1}+\ldots +X_{\mu_r})=t^{-2}X_{\mu_1}+X_{\mu_2}+\ldots +X_{\mu_r}$,  $h_{X_{\la_k}}(t)(X_{\mu_1})=t X_{\mu_1} $ and $h_{X_{\la_k}}(t)(X_{\mu_j})=X_{\mu_j} $ for  $j\geq 2$). We have also  $l_t.X_{\la_k}=t^3 X_{\la_k}$  (because  $[H_{-\mu_1}, X_{\la_k}]=-X_{\la_k}$ implies  $h_{X_{-\mu_1}}(t)(X_{\la_k})=t^{-1} X_{\la_k}$ and  $h_{X_{\la_k}}(t)^2(X_{\la_k})=t^4X_{\la_k}$).  Hence  $\chi_k(l_t)=t^3$.  As  $l_t\in P_X^0$ (product of exponentials), assertion 2) is proved. \me

 If  $\la_k(h)=2$, there are two possible cases. Up to a permutation of the $\mu_j$'s, we can suppose that:
 
- either $n(\la_k,\mu_1)=n(\la_k,\mu_2)=1$ and  $n(\la_k,\mu_j)=0$ for  $j\geq 3$,

- or    $n(\la_k,\mu_1)=2$ and   $n(\la_k,\mu_j)=0$ for  $j\geq 2$.\me
 
If  $n(\la_k,\mu_1)=n(\la_k,\mu_2)=1$ and  $n(\la_k,\mu_j)=0$ for  $j\geq 3$, then, as $\la_k$ is long, one has   $n(\mu_1,\la_k)=\mu_{1}(H_{\lambda_{k}})=n(\mu_2,\la_k)=\mu_{2}(H_{\lambda_{k}})=1$. Therefore the element  $u=2H_{\la_k}-H_{\mu_1}-H_{\mu_2}$ commutes with   $X$ and for  $Z\in \tilde{\go g}^{\la_k}$, one has  $[u,Z]=2Z$.  Then, as before, if  $t\in  F^*$, the element $l_t=h_{X_{\la_k}}(t)^2 h_{X_{-\mu_1}}(t)h_{X_{-\mu_2}}(t)\in L$ fixes $X$ and   $\chi_k(l_t)=t^2$. Again assertion 2) is proved. \me

If $n(\la_k,\mu_1)=2$ and   $n(\la_k,\mu_j)=0$ for $j\geq 2$, then we have $n(\mu_1,\la_k)=1$  (from the tables in \cite{Bou1} we know that if $n(\la_k,\mu_1)=2$ then $n( \mu_1,\la_k)=1$ or $2$; $n( \mu_1,\la_k)=2$ would imply $\mu_{1}=\lambda_{k}$, but   $\la_k\notin \cR$). As $\mu_1(H_0)=2$, there exists a unique  $j<k$ such that  $\mu_1(H_{\la_j})=1$. As  $\la_k$ and  $\la_j$ have the same length, we  have $n(\la_j,\mu_1)=2$ ($n(\la_k,\mu_1)=2$ and $n( \mu_1,\la_k)=1$ imply that $||\lambda_{k}||=\sqrt{2}||\mu_{1}||=||\lambda_{j}||$, and then $n( \mu_1,\la_j)=1$ implies  $n(\la_j , \mu_1)=2$). 
As $\la_j(h)\in\{0,1,2\}$, we have  $n(\la_j, \mu_s)=0$ for  $s\geq 2$. Hence  $[H_{\la_k}-H_{\la_j},X]=0$. Then, for  $t\in F^*$, the element  $l_t=h_{X_{\la_k}}(t) h_{X_{-\la_j}}(t) =h_{X_{\la_k}+X_{-\lambda_{j}}}(t)$ is such that  $l_t.X=X$ and $\chi_k(l_t)=t^2$. 
 
 This ends the proof of the Theorem.
 
 \end{proof}

\subsection{Zeta functions associated to  $(\tilde{P},V^+)$: existence of a functional equation.}
\hfill
\vskip 5pt

We denote by  $\widehat{F^*}$ and  $\widehat\Or_F^*$ the groups of characters of $F^*$ and $\Or_F^*$ respectively.  A character  $\om$ in  $\widehat{\Or_F^*}$ will be identified with a character in $\widehat{F^*}$ by setting   $\om(\pi)=1$. Then, for any $\om\in \widehat{F^*}$, there exists a unique   $\de\in\widehat{\Or_F^*}$  and  $s\in \C$ such that  $\om(x)=\de(x)|x|^s$ for $x\in F^*$. 
The complex number  $s$ is only defined modulo  $\displaystyle\left(\frac{2\pi i}{\log q}\right) \Z$, but its real part  $\text{Re}(s)$ is uniquely defined by $\omega$.    We therefore set    $\text{Re}(\om)=\text{Re}(s)$. \me

\begin{definition}\label{def-sharp}
Let $\om=(\om_0,\ldots, \om_k)\in \widehat{F^*}^{k+1}$ and  $s=(s_0,\ldots ,s_k)\in\C^{k+1}$.  Remember that we have defined $m=\frac{\dim V^+}{\deg \Delta_{0}}$. As we have said we consider here the case where $\ell=1$, and then we have $m=1+\frac{kd}{2}$. We set
$$\om^\sharp\index{omegabecarre@$\om^\sharp$}=((\om_0\ldots \om_k)^{-1}, \om_k,\ldots ,\om_1),\quad{\rm and }\quad s^\sharp \index{sbecarre@$s^\sharp$}=(s^\sharp_{0},s^\sharp_{1},\ldots,s^\sharp_{k})=(-(s_0+\ldots +s_k),s_k,\ldots, s_1)$$
and  $s^\sharp-m=(s_0^\sharp-m, s_1^\sharp,\ldots, s_k^\sharp)=(-(s_0+\ldots +s_k)-m,s_k,\ldots, s_1)$.
\end{definition}
\vskip 5pt
For $\om=(\om_j)_{j}\in(\widehat{F^*})^{k+1}$ and  $s=(s_0,\ldots, s_k)\in \C^{k+1}$, we  define
$$(\om,s)(\De(X))=\prod_{j=0}^k \om_j(\De_j(X))\;|\De_j(X)|^{s_j},\quad X\in V^+$$
and 
$$(\om,s)(\nabla(Y))=\prod_{j=0}^k \om_j(\nabla_j(Y))\;|\nabla_j(Y)|^{s_j},\quad Y\in V^-.$$

\begin{definition}\label{def-zetaP} Let $a\in\Sr_e^{k+1}$    (see Definition \ref{def-ensembleS}) and  $(\om,s)\in \widehat{F^*}^{k+1}\times \C^{k+1}$. The zeta functions associated to $(\tilde{P},V^+)$  and to $(\tilde{P},V^-)$ are the functions:
$$\cK^+_a(f,\om,s)\index{Kcal+@$\cK^+_a$}=\int_{\cO^+(a)} f(X)(\om,s)(\De(X)) dX,\quad f\in \cS(V^+),$$
 and  $$\cK^-_a(g,\om,s)\index{Kcal-@$\cK^-_a$}=\int_{\cO^-(a)} g(Y)(\om,s)(\nabla(Y)) dY,\quad g\in \cS(V^-).$$
It is clear that the integrals defining these functions are absolutely convergent for $\text{Re}(s_j)+\text{Re}(\om_j)>0$, $j=0,\ldots, k$.
\end{definition}

\begin{theorem}\label{thm-Sato-kp} \hfill

Let  $a\in\Sr_e^{k+1}$ and  $(\om,s)\in \widehat{F^*}^{k+1}\times \C^{k+1}$.

1)   The zeta functions $\cK^+_a(f,\om,s)$ for   $f\in \cS(V^+)$ and $\cK^-_a(g,\om,s)$ for  $g\in \cS(V^-)$ are rational functions in $q^{s_j}$ and  $q^{-s_j}$, and hence they admit a meromorphic continuation on $\C^{k+1}$.\medskip

 Moreover, if  $\om_j=\tau_j\otimes |\cdot|^{\nu_j}$ with $\tau_j\in\widehat{\Or_F^*}$ and $\nu_j\in\C$,   there exist polynomials $R^+(\om, s)$ and  $R^-(\om,s)$ in the variables $q^{-s_{j}}$, product of polynomials of type $(1-q^{-N-\sum_{j=0}^k N_j(s_j+\nu_j)})$ with $N,N_j\in \N$,  such that, for all $f\in \cS(V^+)$ and all  $g\in\cS(V^-)$, the functions 
$R^+(\om, s)\cK^+_a(f,\om,s)$ and  $R^-(\om, s)\cK^-_a(g,\om,s)$ are polynomials in  $q^{s_j}$ et $q^{-s_j}$.

2)   For  $a\in\Sr_e^{k+1}$ and $g\in \cS(V^-)$, the   zeta functions satisfy the following functional equation:
$$ \cK^+_a(\cF g,\om^\sharp,s^\sharp-m)=\sum_{c\in \Sr_e^{k+1}} C^k_{(a,c)}(\om,s) \cK^-_c(g,\om,s)$$
where  $C^k_{(a,c)}(\om,s)$ are rational functions in $q^{s_j}$ and  $q^{-s_j}$.

There is a similar result for  $f\in \cS(V^+)$ and $a\in\Sr_e^{k+1}$:
$$\cK^-_a(\cF(f),\om^\sharp,s^\sharp-m)=\sum_{c\in \Sr_e^{k+1}} D^k_{(a,c)}(\om,s)  \cK^+_c(f,\om,s).$$
(again $D^k_{(a,c)}(\om,s)\in \C(q^{\pm s_{0}},\ldots,q^{\pm s_{k}})$).
\end{theorem}

\begin{proof}\hfill

1) This result was first proved, for integrals of the same kind as  $\cK^+_a(f,\om,s)$  and $\cK^-_a(g,\om,s)$,   and for one variable, by J-I Igusa (see Lemma 2 in \cite{Ig2}) and sketched for several variables by F. Sato (see Lemma 2.1 in \cite{Sa}), both proofs using  a result of J. Denef (\cite{D}). 
But these authors did not point out the fact that the analogue of the integers  $N$ and $N_{j}$ are positive. For the convenience of the reader, we sketch   briefly the proof for $\cK^+_a(f,\om,s)$. This is essentially Igusa's proof.

First of all we identify $V^+$ with $F^r$ ($r=\dim_{F} V^+$). We can assume that $f$ is the characteristic function of an open compact subset $C$ of $F^r$. We choose an integer $m$ satisfying $\tau_{j}^m=1$ for $j=0,\ldots,k$. For $x=\pi^\alpha u$ ($\alpha\in \Z, u\in \cO^*_{F}$) we set $\tau_{j}(x)=\tau_{j}(u)$. We also define the map $\Delta: F^r \longrightarrow F^{k+1}$ by
$\Delta(X)=(\Delta_{0}(X),\ldots,\Delta_{k}(X))$ ($X\in V^+=F^r$).

Let $\cA$  be a complete   set of representatives of  $F^*/F^{*m}$, where $F^{*m}$ is the set of  elements $a^m$ for $a\in F^*$. We denote by $c=(c_{0},\ldots,c_{k})$ the elements of $\cA^{k+1}$.Then  we obtain
$$\cK^+_a(f,\om,s)= \sum_{c=(c_{i})\in \cA^{k+1}}\tau_{0}(c_{0})\ldots\tau_{k}(c_{k})\int_{D_{c}} |\Delta_{0}(X)|^{s_{0}+\nu_{0}} \ldots|\Delta_{k}(X)|^{s_{k}+\nu_{k}}dX $$
where $D_{c}=\cO^+(a)\cap C \cap \Delta^{-1}(c F^{*m})$.

It is now enough to prove that assertion 1) is true for the integral $\int_{D_{c}} |\Delta_{0}(X)|^{s_{0}+\nu_{0}} \ldots|\Delta_{k}(X)|^{s_{k}+\nu_{k}}dX$. Let $\theta: F\longrightarrow \Q_{p}^{\alpha}$ be a $\Q_{p}$-linear isomorphism. We still denote by $\theta$ its extension as a $\Q_{p}$-linear isomorphism  from $V^+=F^r$ to $\Q_{p}^{r\alpha}$. Let $dY$ be the measure on $\Q_{p}^{r\alpha}$ which is  the image of the measure $dX$ under $\theta$ and let $E_{c}=\theta(D_{c})$. Let also $N$ be the norm map from $F$ to $\Q_{p}$ and denote by $|\,|_{p}$ the absolute value on $\Q_{p}$. Then

$$\int_{D_{c}} |\Delta_{0}(X)|^{s_{0}+\nu_{0}} \ldots|\Delta_{k}(X)|^{s_{k}+\nu_{k}}dX = \int_{E_{c}}|N(\Delta_{0}(\theta^{-1}(Y))|_{p}^{s_{0}+\nu_{0}} \ldots|N(\Delta_{k}(\theta^{-1}(Y))|_{p}^{s_{k}+\nu_{k}}dY.$$

Of course the functions $P_{j}(X)=\Delta_{j}(\theta^{-1}(Y))$ are polynomials.Therefore we must consider integrals of the form
$$\int_{E_{c}}|P_{0}(Y)|^{s_{0}+\nu_{0}}\ldots|P_{k}(Y)|^{s_{k}+\nu_{k}}dY$$
where the $P_{j}$'s are polynomials.

But a result  of Denef  (\cite{D}, Theorem 3.2)\footnote {This result   is still true for several polynomials with the same proof (private communication of J. Denef)} tells us  that, under the condition that $E_{c}$ is a so-called  {\it boolean combination of sets of type $I$, $II$ and $III$} (\cite{D}, p. 2-3) and is contained in a compact set, this kind of integrals are always  rational functions in the variables $p^{\pm s_{j}}$. The second condition on the integration domain    is of course true for our set $E$. The proof that $E$ is a boolean combination  is exactly the same as the proof of Igusa in the one variable case (see \cite{Ig1}, p.1018-1019).

It is easily seen that the integrals under consideration  are  absolutely convergent Laurent series in $q^{\pm s_{j}}$ (for $\text{Re}(s_j)+\nu_{j}>0$. This implies that they are rational functions in $q^{\pm s_{j}}$.

Finally the fact that there exist polynomials polynomials $R^+(\om, s)$ and  $R^-(\om,s)$ in the variables $q^{-s_{j}}$ having the asserted properties is a consequence of the proof of Denef's result (see \cite{D}, p.5-6).

2) As the prehomogeneous vector space $(G, V^+)$ is regular, the relative invariant $\Delta_{0}$ is non degenerate in the sense of Definition 1.2 in \cite{Sa}. As $\Delta_{0}$ is also a relative invariant for $(P,V^+)$, the prehomogeneous vector space $(\tilde P, V^+)$ satisfies the regularity condition condition $(A.1)$ of F. Sato (\cite{Sa}).

By Theorem  \ref{ConditionA2}, the prehomogeneous vector space $(\tilde P, V^+)$ satisfies also condition $(A.2')$ in \cite{Sa}, and therefore the second assertion is a consequence of   Theorem $k_{\go p}$ in  {\it loc. cit} (\S 2.4).

\end{proof}

\subsection{Explicit functional equation.}\label{par-explicite}
\hfill
\vskip 5pt

 We specify first the functional equation  in Theorem \ref{thm-Sato-kp} in the case where  $k=0$ (rank $1$). }\\
In that case,  $V^+$ and $V^-$ are isomorphic to $F$, and $F^*$ is the unique open orbit of the group   $G=P=\tilde{P}$ in  $V^+$, and  the set  $\Sr_e$ is reduced to $\{1\}$.
 \vskip 5pt

We will identify the spaces   $V^\pm$ with  $F$.  Then the zeta functions  $\cK_a^\pm(f,\de,s)$ appearing in Definition  \ref{def-zetaP} are equal.  Since $a=1$,    we will omit the index $a$ and  the exponents  $\pm$   in  the notations below, that is we set for $f\in \cS(F)$: 
$$\cK(f,\de, s)=\int_F f(t)|t|^s \de(t)dt.$$

 If $dt$ is a (additive) Haar measure  on $F$, $d^*t=\dfrac{dt}{|t|}$ is a Haar measure on $F^*$.

 Tate's zeta function (see \cite{Ta}) is then defined, for  $f\in\cS(F)$, $s\in \C$ and  $\de\in\widehat{F^*}$, by
 $$Z(f,\de,s)= \int_{F^*}f(t) \de(t)|t|^s d^*t.$$
This integral is absolutely convergent  for  $s\in\C$ such that $\text{Re}(s)+\text{Re}(\de)-1>0$ and $Z(f,\de,s)$ extends to a meromorphic function  on $\C$ which satisfies the following functional equation:
$$Z(f,\de ,s)\index{zetafdeltas@$Z(f,\de ,s)$}=\rho(\de,s) \index{rhodeltas@$\rho(\de,s)$}Z(\cF(f),\de^{-1}, 1-s),$$
where  $\rho(\de,s)$ is the so-called  $\rho$ factor of  Tate (see \cite{Ta}).

 
 Then Tate's functional equation can be re-written as follows: 
$$ \cK(f,\de, s)=\rho(\de,s+1)\cK(\cF(f),\de^{-1}, -s-1)\;\text{and }\; \cK(\cF(f), \de,s)=\de(-1) \rho(\de, s+1) \cK(f,\de^{-1}, -s-1).$$
 Hence, this gives  the explicit functional equation in Theorem \ref{thm-Sato-kp} for $k=0$.\vskip 10pt

 To compute the factors $D^k_{(a,c)}(\om,s)$ appearing in Theorem \ref{thm-Sato-kp} in the case  $k\geq 1$, we have to introduce new functions.\vskip 5pt

 Let  $S$ be one of the groups $F^*, F^{*2}$ or $N_{E/F}(E^*)$ for a quadratic extension $E$ of $F$. We set $\Sr=F^*/S$ and let  $\widehat{\Sr}$ be the group of characters of  $\Sr$. Any element  $\chi\in \widehat{\Sr}$ extends uniquely to a character of  $F^*$ which  is trivial on $S$. We still denote   by  $\chi$ this extension.
 .\me

 \begin{definition}\label{def-rhotilde} Let  $\de\in\widehat{F^*}$ and $s\in\C$. For  $x\in\Sr$, we define
$$\tilde{\rho}(\de,s;x)\index{rhotilde@$\tilde{\rho}(\de,s;x)$}=\frac{1}{|\Sr|}\sum_{\chi\in\widehat{\Sr}} \chi(x)\rho(\de\chi,s).$$
As a function of $x$, $\tilde{\rho}(\de,s;x)$ is the Fourier transform of the function  $\chi\in\widehat{\Sr} \mapsto \rho(\de\chi,s).$
Therefore $\rho(\de\chi,s)=\sum_{x\in\Sr}\chi(x) \tilde{\rho}(\de,s;x),$ for  $ \chi\in\widehat{\Sr}.$
\end{definition}

\begin{lemme}\label{lem-EFP-k=0}(compare with \cite{Mu} Corollaire 3.6.3).  \hfill

 Let  $\de\in\widehat{F^*}$ and  $a\in \Sr$. Define 
$$\underline{\cK}_{(a)}(f,\de,s)=\int_{\{t\in F^*; ta\in S\}}f(t)|t|^s \de(t) dt.$$ 

 The functions $\underline{\cK}_{(a)}(f,\de,s)$ satisfy the following functional equations:
$$ \underline{\cK}_{(a)}(f,\de,s)
=\sum_{c\in\Sr}\tilde{\rho}(\de,s+1; ac)\underline{\cK}_{(c)}(\cF(f),\de^{-1}, -s-1),\quad f\in \cS(F),$$
and
$$\underline{\cK}_{(a)}(\cF(g),\de,s)=\sum_{c\in\Sr}\de(-1)\tilde{\rho}(\de,s+1; -ac)\underline{\cK}_{(c)}(g,\de^{-1},-s-1)  \quad g\in  \cS(F).$$

\end{lemme}

\begin{proof}: 

From the definition, one has  $\cK(f,\de\chi, s)= \sum_{a\in\Sr}\chi(a)\underline{\cK}_{(a)}(f,\de,s)$  for $\chi\in\widehat{\Sr}$.
By Fourier inversion on the group $\Sr$, we obtain:
$$ \underline{\cK}_{(a)}(f,\de,s)=\frac{1}{|\Sr|}\sum_{\chi\in\widehat{\Sr}} \chi(a)\cK(f,\de\chi, s),\quad a\in\Sr.$$
From Tate's functional equation, we obtain
$$ \underline{\cK}_{(a)}(f,\de,s)=\frac{1}{|\Sr|}\sum_{\chi\in\widehat{\Sr}} \chi(a)\rho(\de\chi,s+1)\cK(\cF(f),\de^{-1}\chi, -s-1))$$
$$=\frac{1}{|\Sr|}\sum_{c\in\Sr_e}\sum_{\chi\in\widehat{\Sr}} \chi(ac)\rho(\de\chi,s+1) \underline{\cK}_{(c)}(\cF(f),\de^{-1}, -s-1))$$
$$=\sum_{c\in\Sr }\tilde{\rho}(\de,s+1;ac) \underline{\cK}_{(c)}(\cF(f),\de^{-1}, -s-1)).$$
This is the first assertion.

 For the second assertion we have similarly:  
$$ \underline{\cK}_{(a)}(\cF(g),\de,s)=\frac{1}{|\Sr |}\sum_{\chi\in\widehat{\Sr }} \chi(a)K(\cF(g),\de\chi, s)$$
$$=\frac{1}{|\Sr |}\sum_{\chi\in\widehat{\Sr }} \chi(a)\rho(\de\chi,s+1)(\de^{-1}\chi)(-1)\cK(g,\de^{-1}\chi, -s-1)$$
$$=\frac{1}{|\Sr |}\sum_{c\in\Sr }\sum_{\chi\in\widehat{\Sr }} \chi(-ac)\rho(\de\chi,s+1)\de^{-1}(-1) \underline{\cK}_{(c)}(g,\de^{-1}, -s-1)$$
$$=\sum_{c\in\Sr }\de(-1)\tilde{\rho}(\de, s+1;-ac) \underline{\cK}_{(c)}(g,\de^{-1}, -s-1).$$\end{proof}

If $k\geq 1$ and if  $x=(x_0,\ldots, x_k)$ is a  $k+1$-tuple of elements (taken in any set), we will set  $\underline{x}=(x_0,\ldots, x_{k-1})$.

\begin{theorem}\label{EF-(P,V^+)}\hfill

 Let $k\geq 1$. Let $\om\in \widehat{F^*}^{k+1}$,  $s\in\C^{k+1}$ and  $f\in \cS(V^+)$. For $a\in\Sr_e^{k+1}$, we have 
$$\cK^-_a(\cF f,\om^\sharp,s^\sharp-m)=\sum_{c\in \Sr_e^{k+1}} D^k_{(a,c)}\index{Dkac@$D^k_{(a,c)}$}(\om,s) \cK^+_c(f,\om,s),$$
where 
 \begin{align*}&D^k_{(a,c)}(\om,s)\\
&=\ga_k(\underline{a},c_k))(\om_0\ldots \om_k)(-1)\tilde{\rho}((\om_0\ldots \om_k)^{-1},-(s_0+\ldots+s_k+\frac{kd}{2});-a_kc_k)
D^{k-1}_{(\underline{a},\underline{c})}(\underline{\om},\underline{s})\end{align*}
with $D^0_{(a_0,c_0)}(\om_0,s_0)= \om_0(-1)\tilde{\rho}(\om_0^{-1}, -s_0;-a_0 c_0).$
Hence 
\begin{align*}&D^k_{(a,c)}(\om,s)\\
&=\prod_{j=1}^k\ga_j((a_0,\ldots, a_{j-1}), c_j) \prod_{j=0}^k(\om_0\ldots \om_j)(-1)\tilde{\rho}((\om_0\ldots \om_j)^{-1},-(s_0+\ldots+s_j+\frac{jd}{2});- a_jc_j),\\
\end{align*}
where the constant    $\ga_j((a_0,\ldots, a_{j-1}), c_j)$  is the one given in Proposition  \ref{prop-calcul-gamma} (depending on $e$).

 Here, the function $\tilde{\rho}$ is associated to the group $\Sr_e$.

\end{theorem}
\begin{proof} \hfill

  From Theorem  \ref{thm-Sato-kp}, the constants  $D^k_{(a,c)}(\om,s)$ do not depend on the function  $f\in\cS(V^+)$. Therefore to compute these constants we can take  $f\in \cS(\cO^+)$ and suppose that   $\text{Re}((s^\sharp-m)_j)+\text{Re}(\om_j^\sharp)>0$.  Then the functions  $\cK^-_a(\cF f,\om^\sharp,s^\sharp-m)$ and  $\cK^+_c(f,\om,s)$ are defined by the converging integrals given in Definition  \ref{def-zetaP}.\vskip 5pt
  
Notations will be as in  \ref{notation-Delta-souligne}. 
From the integration formula in    Theorem  \ref{th-formule-int-T_{f}}, we obtain $$\cK^-_a(\cF f,\om^\sharp,s^\sharp-m)$$
$$=\int_{u'\in \underline{V}^-} \int_{v'\in \tilde{\go g}^{-\la_k}}T^{k-1,-}_{\cF f}(u',v'){\mathbf 1}_{\cO^-(a)}(u'+v')\om^\sharp(\nabla(u'+v'))\; |\nabla(u'+v')|^{s^\sharp-m}\;|\nabla_1(u')|^{\frac{d}{2}} du' dv'$$
and then from Corollary \ref{cor-reduction-orbite}, we get 
$$\cK^-_a(\cF f,\om^\sharp,s^\sharp-m)$$
$$=\int_{u'\in \underline{\cO}^-(a_0,\ldots, a_{k-1})} \int_{v'\in  \tilde{\go g}^{-\la_k};\tilde{\nabla}_k(v') a_k\in S_e}T^{k-1,-}_{\cF f}(u',v')\om^\sharp(\nabla(u'+v'))\; |\nabla(u'+v')|^{s^\sharp-m}\;|\nabla_1(u')|^{\frac{d}{2}} du' dv'.$$

From Lemma  \ref{lem-decomposition-Delta} we have
$$\nabla_j((u'+v'))=\left\{\begin{array}{ccc} \underline{\nabla}_0(u')\tilde{\nabla}_k(v') & {\rm for} & j=0\\
\underline{\nabla}_{j-1}(u')& {\rm for} & j=1,\ldots, k\end{array}\right.$$
Define  $(\underline{\om},\underline{s})=((\om_0,\ldots, \om_{k-1}), (s_0,\ldots, s_{k-1}))$ and  $[\om]=\prod_{j=0}^k \om_j$. We also note that    $m=1+\frac{kd}{2}=\underline{m}+\frac{d}{2}$. Then:$$|\nabla(u'+v')|^{s^\sharp-m}\;|\nabla_1(u')|^{\frac{d}{2}}=|\underline{\nabla}_0(u')\tilde{\nabla}_k(v')|^{s^\sharp_0-m}\prod_{j=1}^k |\underline{\nabla}_{j-1}(u')|^{s^\sharp_j}\;|\nabla_1(u')|^{\frac{d}{2}}$$
$$=|\underline{\nabla}_0(u')\tilde{\nabla}_k(v')|^{-(s_0+\ldots+s_k)-(1+\frac{kd}{2})} |\underline{\nabla}_{0}(u')|^{s_{k}}|\underline{\nabla}_1(u')|^{s_{k-1}}\ldots |\underline{\nabla}_{k-1}(u')|^{s_1}\;|\underline{\nabla}_0(u')|^{\frac{d}{2}}$$
$$=|\underline{\nabla}(u')|^{\underline{s}^\sharp-\underline{m}} |\tilde{\nabla}_k(v')|^{-(s_0+\ldots+s_k+\frac{kd}{2})-1},$$ 
and similarly we have 
$$\om^\sharp(\nabla(u'+v'))=\underline{\om}^\sharp (\underline{\nabla}(u'))\big([\om](\tilde{\nabla}_k(v'))\big)^{-1}.$$
Therefore
\begin{align*}&\cK^-_a(\cF f,\om^\sharp,s^\sharp-m)\\
&=\int_{u'\in \underline{\cO}^-(a_0,\ldots, a_{k-1})} \int_{v'\in  \tilde{\go g}^{-\la_k};\tilde{\nabla}_k(v') a_k\in S_e}T^{k-1,-}_{\cF f}(u',v')\underline{\om}^\sharp (\underline{\nabla}(u'))\big([\om](\tilde{\nabla}_k(v'))\big)^{-1}\\
 &{\hskip 200pt} \times |\underline{\nabla}(u')|^{\underline{s}^\sharp-\underline{m}} |\tilde{\nabla}_k(v')|^{-(s_0+\ldots+s_k+\frac{kd}{2})-1}du'dv'\end{align*}
For $u'\in\underline{V}^-$, define the function
 $$G_{u'}(v)= \ga_k(u',v) (\cF_u T_f^{k-1,+})(u',v),\quad v\in \tilde{\go g}^{\la_k}.$$ 
Theorem   \ref{th-T_{Ff}-T_{f}}, says that  
$$\cF_v(G_{u'})(v')=T^{k-1,-}_{\cF f}(u',v').$$

    The maps  $v\in \tilde{\go g}^{\la_k}\mapsto \De_k(v)$ and  $v'\in\tilde{\go g}^{-\la_k}\mapsto \tilde{\nabla}_k(v')$ are isomorphisms  from $\tilde{\go g}^{\la_k}$ and  $\tilde{\go g}^{-\la_k}$ respectively  onto  $F$. Therefore the functional equation obtained in Lemma  \ref{lem-EFP-k=0} for $S=S_e$ implies  (note that as $f\in \cS(\cO^+)$, Theorem \ref{th-proprietes-T_{f}} implies that  $ T^{k-1,+}_f\in\cS(\underline{V}^+\times \tilde{\go g}^{\la_k})$)
   \begin{align*}& \int_{v'\in\tilde{\go g}^{-\la_k};\tilde{\nabla}_k(v') a_k\in S_e}T^{k-1,-}_{\cF f}(u',v')\big([\om](\tilde{\nabla}_k(v'))\big)^{-1} |\tilde{\nabla}_k(v')|^{-(s_0+\ldots+s_k+\frac{kd}{2})-1}dv'\\
&=\sum_{c_k\in\Sr_e} [\om](-1)\tilde{\rho}([\om]^{-1},-(s_0+\ldots+s_k+\frac{kd}{2});-a_kc_k)\\
& \hskip50pt \times \int_{v\in \tilde{\go g}^{\la_k};\De_k(v)c_k\in S_e} \ga_k(u',v) (\cF_u T_f^{k-1,+})(u',v)[\om](\De_k(v))  |v|^{ s_0+\ldots+s_k+\frac{kd}{2}} dv.
\end{align*}
We know from Proposition  \ref{prop-calcul-gamma}, that for  $u'\in \underline{\cO}^-(\underline{a})$ and  $v\in\tilde{\go g}^{\la_k}$ with $\De_k(v)c_k\in S_e$, one has  $\ga(u',v)=\ga_k(\underline{a},c_k)$. Hence
\begin{align*}&\cK^-_a(\cF f,\om^\sharp,s^\sharp-m)\\
&=\sum_{c_k\in\Sr_e} [\om](-1)\tilde{\rho}([\om]^{-1},-(s_0+\ldots+s_k+\frac{kd}{2});-a_kc_k)\ga_k(\underline{a},c_k)\\
 &\hskip 20pt\times \int_{v\in \tilde{\go g}^{\la_k};\De_k(v)c_k\in S_e} [\om](\Delta_{k}(v))  |v|^{ s_0+\ldots+s_k+\frac{kd}{2}} \int_{u'\in \underline{\cO}^-(\underline{a})}  (\cF_u T_f^{k-1,+})(u',v)\underline{\om}^\sharp (\underline{\nabla}(u')) |\underline{\nabla}(u')|^{\underline{s}^\sharp-\underline{m}}du'dv.
 \end{align*}
 If $k=1$, then ${\underline V}^-$ is equal to $\tilde{\go g}^{-\la_0}$ which is isomorphic to $F$ (see Notation \ref{notation-Delta-souligne}). Hence $\underline{a}=(a_0)$ and $u'\in  \underline{\cO}^-(\underline{a})$ if and only if $u'=y_0 Y_0$ with $y_0a_0\in S_e$. If $k\geq 2$, then $\underline{\cO}^-(\underline{a})$ is a $\underline{\tilde{P}} $ open orbit in $V^-$. Applying Lemma \ref{lem-EFP-k=0} (with $S=S_e$) if $k=1$, and  the abstract functional equation (Theorem \ref{thm-Sato-kp}) for $\underline{\tilde{\go g}}$ if $k\geq 2$  to the function  $u\mapsto T_f^{k-1,+}(u,v)$, we  obtain that:

$$ \int_{u'\in \underline{\cO}^-(\underline{a})}  (\cF_u T_f^{k-1,+})(u',v)\underline{\om}^\sharp (\underline{\nabla}(u')) |\underline{\nabla}(u')|^{\underline{s}^\sharp-\underline{m}}du'$$
$$=\sum_{\underline{c}=(c_0,\ldots c_{k-1})\in\Sr_e^k}D^{k-1}_{(\underline{a},\underline{c})}(\underline{\om},\underline{s})\int_{u\in \underline{\cO}^+(\underline{c})}T_f^{k-1,+}(u,v) \underline{\om}(\underline{\De})(u)|\underline{\De}(u)|^{\underline{s}} du.$$

 with $D^0_{(a_0,c_0)}(\om_0,s_0)= \om_0(-1)\tilde{\rho}(\om_0^{-1}, -s_0;-a_0 c_0)$.\vskip 5pt

The preceding equation becomes now: 
\begin{align*}&\cK^-_a(\cF f,\om^\sharp,s^\sharp-m)\\
 &=\sum_{c\in\Sr_e^{k+1}}[\om](-1)\tilde{\rho}([\om]^{-1},-(s_0+\ldots+s_k+\frac{kd}{2}); -a_kc_k)\ga_k(\underline{a},c_k)
D^{k-1}_{(\underline{a},\underline{c})}(\underline{\om},\underline{s})\\
&\hskip 20pt \times \int_{v\in \tilde{\go g}^{\la_k};\De_k(v)c_k\in S_e} \int_{u\in \underline{\cO}^+(\underline{c})}T_f^{k-1,+}(u,v) \underline{\om}(\underline{\De})(u)|\underline{\De}(u)|^{\underline{s}}  [\om](\De_k(v))  |v|^{ s_0+\ldots+s_k+\frac{kd}{2}} dv du.
\end{align*}
From Lemma   \ref{lem-decomposition-Delta}  we have:
$$\De_j(u+v)=\left\{\begin{array}{ccc} \underline{\De}_j(u)\De_k(v) & {\rm for} & j=0,\ldots, k-1\\
\De_k(v)& {\rm for} & j=k\end{array}\right.,$$
which implies  
$$\om(\De)(u+v)=\prod_{j=0}^k \om_j(\De_j(u+v))=\prod_{j=0}^{k-1} \om_j(\underline{\De}_j(u)) [\om](\De_k(v))
=\underline{\om}(\underline{\De})(u) [\om](\De_k(v))$$
and 
$$|\De(u+v)|^s=\prod_{j=0}^{k }| {\De}_j(u+v)|^{s_j}=\prod_{j=0}^{k-1}|\underline{\De}_j(u)|^{s_j} |\De_k(v)|^{s_0+\ldots+s_k}.$$
Moreover from Corollary  \ref{cor-reduction-orbite} we have 
$$u\in \underline{\cO}^+(\underline{c}),\; {\rm and} \;\De_k(v) c_k\in S_e\Longleftrightarrow u+v\in\cO^+(c).$$
Then, if we define 
$$D^k_{(a,c)}(\om,s)=\ga_k(\underline{a},c_k)[\om](-1)\tilde{\rho}([\om]^{-1},-(s_0+\ldots+s_k+\frac{kd}{2});- a_kc_k)
D^{k-1}_{(\underline{a},\underline{c})}(\underline{\om},\underline{s})$$
(which is the announced relation between $D^k_{(a,c)}(\om,s)$ and $D^{k-1}_{(\underline{a},\underline{c})}(\underline{\om},\underline{s})$),

we obtain finally:
\begin{align*}&\cK^-_a(\cF f,\om^\sharp,s^\sharp-m)\\
&=\sum_{c\in\Sr_e^{k+1}}D^k_{(a,c)}(\om,s) \int_{u\in \underline{V}^+}\int_{v\in \tilde{\go g}^{\la_k}} {\mathbf 1}_{\cO^+(c)}(u+v) T_f^{k-1,+}(u,v)\om(\De)(u+v)|\De(u+v)|^s|\De_k(v)|^{\frac{kd}{2}} du dv\\
&=\sum_{c\in \Sr_e^{k+1}}D^k_{(a,c)}(\om,s) \cK^+_c(f,\om,s).
\end{align*}

The explicit computation of $D^k_{(a,c)}(\om,s)$ follows easily by induction.

 \end{proof}

In the case where $e$ is even, the explicit value of  $\ga_k(\underline{a},c_k)$ allows to obtain a simpler expression for $D^k_{(a,c)}(\om,s)$. %

\begin{cor}\label{cor-calcul-constanteEF-e-pair} \hfill

Let $\om\in\widehat{F^*}^{k+1}$ and  $s\in\C^{k+1}$.

{\rm 1.} If $e=0$ or $4$, then $\Sr_e=\{1\}$ and we omit  the dependance on the variables $a,c$. For $f\in\cS(V^+)$, we have 
$$\cK^-(\cF f,\om^\sharp,s^\sharp-m)=D^k(\om,s) \cK^+(f, \om, s),$$
where  $$D^k(\om,s)=\ga_\psi(q)^{\frac{k(k+1)}{2}}\prod_{j=0}^{k}(\om_0\ldots \om_j)(-1)\rho\big((\om_0\ldots \om_j)^{-1},-(s_0+\ldots +s_j+\frac{jd}{2})\big).$$

{\rm 2.} If  $e=2$, remember that $S_e=N_{E/F}(E^*)$ where $E$ is a quadratic extension of  $F$ with  $\varpi_E$ as associated quadratic character . Let  $a=(a_0,\ldots,a_k)$ and   $c=(c_0,\ldots, c_k)$ be elements of  $ \Sr_e^{k+1}$. Then  
\begin{align*}&D^k_{(a,c)}(\om,s)\\
&=\ga_\psi(q_e)^{\frac{k(k+1)}{2}}\prod_{j=1}^k\varpi_E(c_j)^j\varpi_E(a_0\ldots a_{j-1}) \prod_{j=0}^k(\om_0\ldots \om_{j})(-1)\tilde{\rho}((\om_0\ldots \om_{j})^{-1},-(s_0+\ldots+s_{j}+\frac{jd}{2}); -a_{j}c_{j}). 
\end{align*}

\end{cor}
\begin{proof} \hfill

If   $e=0$ or  $4$ then $\Sr_e=\{1\}$. In that case, for all  $x\in \Sr_e$, $z\in\C$ and  $\de\in\widehat{F^*}$, one has  $\tilde{\rho}(\de,z; x)=\rho(\de,z)$. The corollary is the an easy consequence of Proposition  \ref{prop-calcul-gamma} and of    Theorem \ref{EF-(P,V^+)}.

\end{proof}
 

\section{Zeta functions associated to the $H$-distinguished minimal principal series}\label{sec:zeta}
\subsection{Characters of the group  $L$ which are trivial on  $L\cap \cH_p$}
\hfill

Let   $\Om_p^\pm $ \index{Omega+@$\Om_p^\pm $},  $p\in \{1,\ldots,r_0\}\index{r0@$\,r_0$}$ be the open  $G$-orbits  in $V^\pm$    (Theorem \ref{thm-orbouverte}). The integer $r_0$ depends on  $k$ and on  $e$,  and  $r_0\leq 5$ 

 For $1\leq p \leq r_0$, we fix an element  $I_p^+$ \index{Ip+@$I_p^+$}in  $\Om_p^+$ which is in the  "diagonal" $\oplus_{j=0}^k\tilde{\go g}^{\la_j}$. We suppose moreover that  $I_1^+=I^+$. We set   $I_p^-=\iota(I_p^+)\in \oplus_{j=1}^k\tilde{\go g}^{-\la_j}$ and hence  $\{I_p^-, H_0, I_p^+\}$ is an  $\go sl_2$-triple. 

The non trivial element of the Weyl group associated to this $\go sl_2$-triple acts on  $\tilde{\go g}$ by the element $w_p\in Aut_0(\tilde{\go g})$ given by    $w_p=e^{\ad I_p^+}e^{\ad I_p^-}e^{\ad I_p^+}$. 
Let  $\si_p$  \index{sigmap@$\si_p$} be the automorphism of $\text{Aut}_0(\tilde{\go g})$   defined by  $\si_p(g)=w_p gw_p^{-1}$ (for $g\in Aut_0(\tilde{\go g})$).\me

   We know that  $\si_p$ is an involution of  $G$ (Theorem \ref{th-invol}). Moreover  the group  $\cH_p\index{Hcalp@$\cH_p$}=Z_G(I_p^+)=Z_G(I_p^-)$ is an open  subgroup  of the group of fixed points of  $\si_p$. Therefore  $\Om_p^\pm$ is isomorphic to the symmetric space $G/\cH_p$.\me

We recall also that we denote by  $P$ the parabolic subgroup defined by  $P=LN$ where  $L=Z_G(\go a^0)$,   $N=\exp\ad \go n$, and  $\go n=\oplus_{0\leq i<j\leq k} E_{i,j}(1,-1)$ (cf.  \ref{subsection-P}).

 The key point is that the group   $P$  is a minimal  $\si_p$-split parabolic subgroup for $p=1,\ldots,r_0$, see section \ref{sec:sigma-parabolic} (this means that  $\si_p(P)$ is the parabolic subgroup opposite to  $P$, and $P$  is minimal for this property). The subgroup  $L$ is then the unique Levi subgroup  of  $P$ which is stable under $\sigma_{p}$. \me

 As $\ell=1$, the group $L$ acts by a character $x_j(\cdot)$ on each $\tilde{\go g}^{\la_j}$. Thus, for $p\in\{1,\ldots ,r_0\}$, as $I_p^+$ belongs to the "diagonal", the group $L\cap \cH_p$ is the group of elements  $l\in L$ such that $x_j(l)=1$ for $j=0,\ldots, k$, hence $L\cap \cH_p$ does not  depend on $p$. We set 
\beq\label{def-LH} L_H \index{LH@$L_H$}=\{l\in L; x_j(l)=1,\;\text{for}\; j=0,\ldots, k\}=L\cap \cH_p, \; p\in\{1,\ldots , r_0\}.\eeq
Let  $A_L$ the maximal split torus of the center of  $L$.
 We denote by  $\X(L)$ (resp. $\X(A_L)$) the group of rational characters of $L$ (resp. de $A_L$) which are defined over  $F$.   The  real vector space  $a_L$ is defined by  $a_L=\text{Hom}_\Z(\X(L), \R)$.  The restriction map from $L$ to $A_L$ induces an injection from  $\X(L)$  on $\X(A_L)$, and these $2$ lattices have the same rank, which is equal to the dimension of $\go a_L$  (see \cite{Re} V.2.6). Then, the  dual space  $a^*_L$  satisfies $a_L^*\simeq\X(L)\otimes_\Z\R\simeq \X(A_L)\otimes_\Z\R$ and each root  $\al\in\Si$   induces a linear form of $a^*_L$, which we also denote by $\al$.\me

The canonical map  $H_L:L\to a_L$ \index{HL@$H_L$} is then defined by
$$e^{\langle H_L(l),\chi \rangle}=|\chi(l)|,\quad l\in L, \chi\in\X(L).$$

The group of non ramified characters of  $L$ is then defined by: 
$$X(L)=\text{Hom}(L/\text{Ker} H_L, \C^*),$$
The map $\nu\longmapsto \chi_\nu$ from $(a_L^*)_\C$ to $X(L)$ defined by  $\chi_\nu:l\mapsto e^{\langle \nu, H_L(l)\rangle}$ (or equivalently the map 

 $\chi\otimes s\in\X(L)\otimes_\Z \C \longmapsto(l\mapsto   |\chi(l)|^s)$) 
is surjective and its kernel is a lattice in  $(a_L^*)_\C$.  This defines a structure of algebraic variety on  $X(L)$. More precisely  $X(L)$ is a complex torus whose Lie algebra is $(a_L^*)_\C$.

 Let $p\in\{1,\ldots, r_0\}$. The involution $\si_p$ on $L$ induces an involution on  $a_L$ and $a_L^*$, which we denote by the same letter. From the proof of  Lemma \ref{lem-SECq}, the Lie algebra of  $A_L$  decomposes as   $\go a_L=\go a_L\cap \go  h_p\oplus \go a^0$ where  $\go h_p$ is the Lie algebra of  $\cH_p$.  Since the groups $L\cap \cH_p,\; p=1,\ldots , r_0$ are equal (see (\ref{def-LH})), this decomposition does not depend on  $\si_p$. Hence, the induced involution  on $a_L$ does not depend on $p$. We set  $a_L=a_L^\si\oplus a^0$ and $a_L^*=(a_L^\si)^*\oplus a^{0*}$ the decompositions of $a_L$ and  $a_L^*$ in invariant and anti-invariant spaces under the action of  $\si_p$ for all $p$. By definition of  $\go a^0$, a basis of  $(a^{0,*})_\C$ is given by  $\la_0,\ldots, \la_k$.\me

 Let $X(L)_{\si}$ be the image of  $(a^{0*})_\C$ by the map $\nu\mapsto \chi_\nu$. Then, the torus $X(L)_\si$ is the group of non ramified characters given by  $$l\mapsto \prod_{j=0}^k |x_j(l)|^{\mu_j}\;\text{for}\; \mu=\mu_1\la_1+\ldots+\mu_k\la_k\in (a^{0*})_\C.$$

Let  $\de_P$ \index{deltaP@$\de_P$} be the modular function of  $P$, which is given by  $\de_P(ln)=e^{2\langle\rho_P, H_{L}(l)\rangle}$. Here  $\rho_P$ \index{rhoP@$\rho_P$}is   half the sum of the roots of  $A_L$ in  $N$ (counted with multiplicity).  \begin{lemme}\label{lem-rhoP} We have
$$2\rho_P=\frac{d}{2}\sum_{j=0}^k  (k-2j)\la_j.$$ 
This implies that for  $l\in L$,  $$\de_P(l)^{1/2}=\prod_{j=0}^l|x_j(l)|^{\rho_j},\quad{\rm where }\;\;  \rho_j=\frac{d}{4}(k-2j).$$

\end{lemme}
\begin{proof} The Lie algebra    of $P$ decomposes as $\go p= \go l\oplus\go n$ with  $\go l=Z_{\go g}(\go a^0)$ and  $\go n=\oplus_{i<j}E_{i,j}(1,-1)$.
Let  $S_{i,j}$ be the set of roots  $\mu$ such that  $\mu\neq (\la_i-\la_j)/2$ and  $\tilde{\go g}^\mu\subset E_{i,j}(1,-1)$. Let $m_\mu={\rm dim}\; \tilde{\go g}^{\mu}$ be the multiplicity of  $\mu$. We know from    Remark \ref{rem-actionwij}, that if  $\mu\in S_{i,j}$  then   $\mu'=\la_i-\la_j-\mu$ is a root in  $S_{i,j}$ which is not equal to  $\mu$ but has the same multiplicity,  and  of course $\mu+\mu'= \la_i-\la_j$. As $\dim(\tilde{\go g}^{(\la_i-\la_j)/2})=e$,  and $d={\rm dim}\; E_{i,j}(1,-1)=e+\sum_{\mu\in S_{i,j}} m_\mu$, we obtain  
$$\sum_{\mu\in S_{i,j}}m_\mu\mu=\frac{1}{2} \sum_{\mu\in S_{i,j}}m_\mu(\la_i-\la_j)=\frac{d-e}{2}(\la_i-\la_j).$$
Therefore
$$2\rho_P=e\sum_{i<j}\frac{(\la_i-\la_j)}{2}+\frac{d-e}{2}\sum_{i<j} (\la_i-\la_j )=\frac{d}{2}\sum_{i<j} (\la_i-\la_j ).$$
A simple computation gives then the first assertion.
The second assertion is then an immediate consequence of the first one.

\end{proof}

From Remark \ref{rem-image-muj}  we know now that the map  $\cT: l\mapsto (x_0(l),\ldots, x_k(l))$ induces an isomorphism of groups from  $L/L\cap \cH_p=L/L_H$ onto  $\cup_{x\in\Sr_e} x S_e^{k+1}$.\me
 
Hence, any character of  $F^{*k+1}$ induces, by restriction to $\text{\rm Im}\cT$, a character of $L$ trivial on $L\cap \cH_p$ for any $p\in\{1,\ldots,N\}$. If $\de=(\de_{0},\ldots,\de_{k})$ \index{delta@$\de$, $\de_{i}$ (Part II)} is a character of $(F^*)^{k+1}$  the corresponding character of $L$, which is invariant on   $L/L\cap \cH_p=L/L_H$, is still denoted by $\delta$ and is given by 
$$\delta(l)=\prod_{i=0}^k \delta_{i}(x_j(l)).$$

This is made more precise in the following Lemma.

\begin{lemme}\label{lem-caractere-L} 
\hfill

{\bf 1.} Any character of  $\text{\rm Im}\cT=\cup_{x\Sr_e} x S_e^{k+1}$ is the restriction to $\text{\rm Im}\cT$ of a (non unique) character of $(F^*)^{k+1}.$

{\bf 2.} Two characters  $\de=(\de_{0},\ldots,\de_{k})$ and  $\de'=(\de'_{0},\ldots,\de'_{k})$ of $(F^*)^{k+1}$ coincide on  $\text{\rm Im}\cT=\cup_{x\in \Sr_e} x S_e^{k+1}$, if and only if for all $j=0\ldots k$, there exists a character $\chi^{j}$ of  $ {F^*}$ with values in  $\{\pm1\}$, and trivial on  $S_e$ such that   
$$ \de_{j}(x)=\chi^{j}(x) \de_j'(x),\quad{\rm et }\quad \prod_{j=0}^k \chi^{j}(x)=1,\quad  x\in F^*.$$\end{lemme}

\begin{proof} Note first that as  $F^{*2}\subset S_e$, the group $(F^*)^{k+1}/\text{Im}\cT$ is a finite abelian group whose non trivial elements are of order 2.

Let   $\tilde{\de}$ be a character of  $\text{Im}\cT$ and take  $x\in (F^*)^{k+1}\setminus\text{Im}\cT$. Then, as $F^{*2}\subset  \text{Im}\cT$, we have $x^2\in \text{Im}\cT$. Let  $\varpi_x$  be a square root of  $\tilde{\de}(x^2$),  and let  $\cM_1$ be the subgroup of $(F^*)^{k+1}$ generated by $\text{Im}\cT$ and  $x$. One can now  define a character of    $\cM_1$ by setting  $\tilde{\de}_1(ax)=\tilde{\de}(a)\varpi_x$ for $a\in\text{Im}\cT$. As $(F^*)^{k+1}/\text{Im}\cT$ is finite, we obtain the first assertion by induction.

As  $\Sr_e=F^*/S_e$ and  $F^{*2}\subset S_e$, the second assertion is clear.

\end{proof}

\subsection{The minimal principal $H$-distinguished series}
\hfill

Let $p\in\{1,\ldots r_0\}$.    If  $(\pi, V_\pi)$ is a  smooth representation  of  $G$  we denote by  $(\pi^*, V_\pi^*)$ \index{Vpi*@$V_\pi^*$} \index{pi*@$\pi^*$}its dual representation. Such a representation is  said to be  $\cH_p$-distinguished \index{Hpdistinguished@$\cH_p$-distinguished}if the space  $(V_\pi^*)^{\cH_p}$  of  $\cH_p$-invariant linear forms on $V_\pi$ is non-trivial.  This is also equivalent   to the condition \\$\text{Hom}_{\cH_p}(V_\pi, C^\infty(G/\cH_p))\neq\{0\}$ (where $ C^\infty(G/\cH_p)$ is the space of locally constant functions on $G/\cH_p$).\me

If $\chi$ is a character of  $P$  we define as usual the space $I_{\chi }$ \index{Ichi@$I_{\chi }$} to be the space of functions  $v:G\to \C$ which are right invariant under a compact open subgroup of $G $  and which satisfy the condition:
$$v(nlg)=\de_P(l)^{1/2} \chi(l) v(g), \quad n\in N, l\in L, g\in G.$$
The induced representation $I_{P}^{G}(\chi)$ \index{IPG@$I_{P}^{G}(\chi)$} is the right regular representation of $G$ on $I_{\chi}$.

One knows from   (\cite{B-D} Th\'eor\`eme 2.8), that if there exists    an open $P$- orbit $P\ga\cH_p$ in $G/\cH_{p}$  (or equivalently an open orbit $P\gamma I^+_{p}$ in $\Omega_{p}^+$) with $\gamma\in G$, such that $\chi$ is trivial on $L\cap \ga \cH_p\ga^{-1}$, then (under some additional conditions, see below), the representation $I_{P}^{G}(\chi)$ is $\cH_p$-distinguished.

But as any open $P$-orbit in $\Omega_{p}$ meets the diagonal we can suppose that $\gamma. I^+_{p}\in\oplus_{j=0}^k\tilde{\go g}^{\la_j}$ and hence, using (\ref{def-LH}),  the character $\chi$ is trivial on $L\cap \ga\cH_p\ga^{-1}=L_H=L\cap \cH_p$. Therefore we make the following definition.

\begin{definition} The minimal principal $H$-distinguished \index{minimal principal ...}series is the set  of representations $I_P^G(\chi)$ where   $\chi$ is a character  of $L$ trivial on $L_H$ such that  $I_P^G(\chi)$ is  $\cH_p$-distinguished  for any  $p\in\{1,\ldots,r_0\}$.
\end{definition}

Fix a unitary character    $\de=(\de_0,\ldots, \de_k)$ of  $(F^*)^{k+1}$  and fix also  $\mu=\mu_0\la_0+\ldots +\mu_k\la_k\in(  a^{0\; *})_{\C}$. This defines a character  $\de_\mu$ of  $L$ which is trivial on  $L\cap \cH_p$ by 
$$ \de_\mu(l)=\prod_{j=0}^{k} \de_j(x_j(l))|x_j(l)|^{\mu_j},\quad l\in L.$$
We set  $\pi_{\de,\mu}\index{pideltamu@$\pi_{\de,\mu}$}=I_P^G(\de_\mu)\index{IPGdeltamu@$I_P^G(\de_\mu)$}$ and denote by  $I_{\de,\mu}=I_{\de_\mu}$\index{Ideltasubmu@$I_{\de_\mu}$}\index{Ideltamu@$I_{\de,\mu}$} its space.

Let  $K$ be a maximal compact subgroup of $G$ such that $G=PK$ (see \cite{Re}  Th\'eor\`eme V.5.1). Let  $I(\de)$ be the space of functions  $v:K\longmapsto \C$ which are right invariant under an open compact subgroup of $K$   and satisfy the following condition: 
$$v(nlk)=\de(l)v(k) \text{ for } nl\in P\cap K \text{ and }  k\in K.$$
Then the restriction of   functions in  $I_{\de,\mu}$ to  $K$ gives an isomorphism of  $K$-modules from  $I_{\de,\mu}$ onto  $I(\de)$ and we denote by  $\overline{\pi_{\de,\mu}}$ the representation of  $G$ on  $I(\de)$ obtained from  $\pi_{\de,\mu}$ through this  isomorphism (this is the so called compact picture of the induced representation).\me

 We identify the spaces  $(a^{0*})_\C$ and  $\C^{k+1}$ by the isomorphism  $\mu=\sum_{j=0}^k \mu_j\la_j\mapsto (\mu_0,\ldots, \mu_k)$. 

From   (\cite{B-D} Th\' eor\`eme 2.8), we know  that there exists a rational function on   $\cR_p$ on  $X(L)_{\si}\simeq (a^{0*})_\C$, which is a product of functions  of the form  $(1-c q^{\sum_{j=0}^k a_j \mu_j})$, where  $c\in\C^*$ and  $a_j\in\Z$, such that   if  $\cR_p(\mu)\neq 0$, the representation  $(\pi_{\de,\mu},I_{\de,\mu})$ is  $\cH_p$-distinguished. Moreover, in that case,  $\dim (I_{\de,\mu}^*)^{\cH_p}$ is equal to the number of  open $P$-orbits  in $G/\cH_{p}$ (or equivalently in $\Omega_{p}$)  and a basis of  $ (I_{\de,\mu}^*)^{\cH_p}$ is explicitly described  in \cite{B-D} (see also below for our case). Differences between our notations and the results of \cite {B-D} come from the fact that there the authors consider non normalized left induction.

\begin{definition}\label{def-ouvertU}  Set  $\cR\index{Rcal@$\cR$ (Part II)}=\prod_{p=1}^{r_0} \cR_p$  where  $\cR_p$ are the rational functions mentioned above.     Let $\bU$ \index{Ubold@$\bU$} be the open dense subset of  $ (a^{0*})_\C$ defined by $\bU= \{\mu\in (a^{0*})_\C, \cR(\mu)\neq 0\}$.  Hence for all $\mu\in\bU$ and $p=1,\ldots,k$, the representation  $(\pi_{\de,\mu},I_{\de,\mu})$ is  $\cH_p$-distinguished.
\end{definition}

 We will now describe, in our case, the basis   of$(I_{\de,\mu}^*)^{\cH_p}$ given in   (\cite{B-D} Th\'eor\`eme 2.8 and  Th\'eor\`eme 2.16, \cite{La} \textsection 4.2).\me

Let  $p\in\{1,\ldots ,r_0\}$. The open  $(P,\cH_p)$-orbits in  $G$ can be view as the  open $P$-orbits  in $V^+$ contained in  $\Om_p^+\simeq G/\cH_p$. Remember that the open  $P$-orbits in $V^+$ are the sets  $P.I^+(a)$ where  $I^+(a)=a_0X_0+\ldots +a_{k-1} X_{k-1}+X_k$ for  $a=(a_0,\ldots,a_{k-1},1)$ in  $\Sr_e^k\times\{1\}$ (see   Lemma \ref{lemme-Ptilde}). For $\underline{a}=(a_0,\ldots, a_{k-1})\in\Sr_e^{k}$, we set  $(\underline{a},1)=(a_0,\ldots,a_{k-1},1)$.
 \me 

\begin{definition}\label{def-lienP-G-orbites} For  $p\in\{1,\ldots, r_0\}$, we denote by  $\cS_p$ the set of  $\underline{a}\in\Sr_e^k$ such that  $I^+(\underline{a},1)\in \Om_{p}^+$. 
Let us fix  $\ga_{\underline{a}}\in G$ such that  $\ga_{\underline{a}}I^+_{p}=I^+(\underline{a},1)$.  Hence  $\cup_{\underline{a}\in\cS_p} P\ga_{\underline{a}} \cH_p$ is the union of the open  $(P,\cH_p)$-orbits in $G$. 
\end{definition}

\begin{definition}\label{def-noyauPoisson}
Let $\mu\in (a^{0*})_\C$ , $p\in\{1,\ldots, r_0\}$  and  $\underline{a}\in \cS_p$. The Poisson kernel \index{Poisson kernel} $\mathbb P_{\de,\mu}^{\underline{a}}$ \index{Pbb@$\mathbb P_{\de,\mu}^{\underline{a}}$} is the function on $G$ defined by:
$$\P_{\de,\mu}^{\underline{a}}(g)=\left\{\begin{array}{ccc} \de_P(l)^{1/2} (\de_\mu)^{-1}(l) & {\rm if} & g=nl\ga_{\underline{a}}h\in NL\ga_{\underline{a}}\cH_{p},\\ 0 & {\rm if } &g\notin P\ga_{\underline{a}}\cH_{p}.\end{array}\right.$$
\end{definition}
 \me
 Let $C(G,P,\de^*,\mu)$  \index{CG@$C(G,P,\de^*,\mu)$} be the space of continuous functions $w:G\longmapsto \C$  such that 
\beq\label{eq-CGP}w(nlg)=\de_P^{1/2}(l) \de_\mu(l)^{-1} w(g) \hskip 7pt n\in N, l\in L, g\in G.\eeq

We will say that  $\text{Re}(\de_P^{-1/2}\de_\mu)$ is strictly  $P$-dominant \index{strictlyP@strictly  $P$-dominant} if $\langle\text{Re}(\mu)-\rho_P,\al\rangle>0$ for all roots $\al$  of  $A$ in   $N$. From  (\cite{B-D} Th\'eor\`eme 2.16) and  (\cite{La}  Th\'eor\`eme 4 (ii)), we know that  if    $\mu\in(a^{0*})_\C$  is  such that   $\text{Re}(\de_P^{-1/2}\de_\mu)$ is strictly  $P$-dominant and if  $\underline{a}\in\Sr_e^k$,  the function  $\P_{\de,\mu}^{\underline{a}}$    is a right $\cH_p$-invariant element of $C(G,P,\de^*,\mu)$. 

Moreover by (\cite{B-D} (2.29)), the map 
\beq\label{eq-dualite}(w,v)\in C(G,P,\de^*,\mu)\times I_{\de,\mu}\mapsto \langle w,v\rangle=\int_K w(k)v(k) dk\eeq
defines a $G$-invariant duality between  $ C(G,P,\de^*,\mu)$ and  $I_{\de,\mu}$.  Hence  $C(G,P,\de^*,\mu)$ can be viewed as a subspace of $I_{\de,\mu}^*$ and $\P_{\de,\mu}^{\underline{a}}$ is identified to a $\cH_{p}$-invariant linear form on $I_{\de,\mu}$. Let us make this more precise in the following definition.

\begin{definition}\label{def-C+-formes}
Let  $\bC^+$\index{Cbold@$\bC^+$} be the set of  $\mu\in (a^{0*})_\C$   such that  $\text{Re}(\de_P^{-1/2}\de_\mu)$ is strictly  $P$-dominant.

Let   $\mu\in\bU\cap \bC^+$, $p\in\{1,\ldots, N\}$ and   $\underline{a}\in\cS_p$. We denote  by $\xi_{\de,\mu}^{\underline{a}}\in (I_{\de,\mu}^*)^{\cH_p}$ \index{xideltamu@$\xi_{\de,\mu}^{\underline{a}}$} the $\cH_p$-invariant linear form which corresponds to  $\P_{\de,\mu}^{\underline{a}}$ via the preceding duality between  $ C(G,P,\de^*,\mu)$ and $I_{\de,\mu}$.

\end{definition}
The isomorphism  $I_{\de,\mu}\simeq I(\de)$   (given by restriction to  $K$) induces an isomorphism   $\xi\mapsto \overline{\xi}$ from  $I_{\de,\mu}^*$ onto  $I(\de)^*$. From  Th\'eor\`eme  2.8 in \cite{B-D}, the map  $\mu\in \bU \cap \bC^+ \mapsto \overline{\xi_{\de,\mu}^{\underline{a}}}$ has a meromorphic continuation to $(a^{0 *})_\C$ such that $\cR(\mu) \overline{\xi_{\de,\mu}^{\underline{a}}}$ is a polynomial function in the variables  $q^{\pm\mu_j}$.
Moreover, for  $p\in\{1,\ldots,N\}$, $\underline{a}\in \cS_{p}$ and  $\mu\in\bU$, the set of $\xi_{\de,\mu}^{\underline{a}}$    is a basis of  $(I_{\de,\mu}^*)^{\cH_p}$ (loc.cit) .\me

In the rest of the paper we fix  Haar measures $dg$ and  $dk$ on $G$ and  $K$ respectively and a right Haar measure  $d_rp$ on  $P$ such that 
$$\int_G f(g) dg=\int_P\int_K f(pk) \de_P(p)^{-1} d_rp\; dk,\quad f\in L^1(G).$$

The character $\delta_{\mu}$ is extended to $P$ by setting  $\de_\mu(ln)=\de_\mu(l)$ for  $l\in L$ and  $n\in N$.

 Let $C_c^\infty(G)$ be the space of locally constant   functions on $G$ with compact support and define $L_{\de,\mu}: C_c^\infty(G)\longmapsto I_{\de,\mu}$  by 
$$L_{\de,\mu}(\varphi)(g)=\int_P \delta_P(p)^{-1/2} \de_\mu(p)^{-1} \varphi(pg)d_rp,\quad g\in G.$$

It is well known that the map  $L_{\de,\mu}$ is surjective $G$-equivariant (here  $G$ acts on $C_c^\infty(G)$ by the right regular representation).
\begin{lemme}\label{lem-crochet-dualite} For  $\psi\in C(G,P,\de^*,\mu)$ and  $\varphi\in C_c^\infty(G)$, one has 
$$\int_G \psi(g) \varphi(g) dg=\int_K  \psi(k) L_{\de,\mu}(\varphi)(k)dk$$
\end{lemme}
\begin{proof} 
From the above given decomposition of $dg$  and from  and the definition of $C(G,P,\de^*,\mu)$, we have  
$$\int_G \psi(g) \varphi(g) dg= \int_P\int_{  K} \psi(pk) \varphi(pk) \de_P(p)^{-1} d_rp \; dk$$
$$= \int_P\int_{ K}\de_P(p)^{-1/2}\psi(k) \de_\mu(p)^{-1}\varphi(pk) d_rp \;dk$$
$$=\int_{  K}\psi(k) \int_P\de_P(p)^{-1/2} \de_\mu(p)^{-1}\varphi(pk) d_rp dk=\int_{  K}\psi(k)L_{\de,\mu}(\varphi)(k) dk.$$\end{proof}
\begin{cor}\label{cor-pi-star(g)xi} Let  $\mu\in\bU\cap \bC^+$and  $v\in I_{\de,\mu}$. Then, for  $x\in G$,  one has  
$$\langle \pi_{\de,\mu}^*(x)\xi_{\de,\mu}^{\underline{a}},v\rangle= \int_{  K}\P_{\de,\mu}^{\underline{a}}(k)  v(kx^{-1}) dk=\int_{  K} \P_{\de,\mu}^{\underline{a}}(kx) v(k) dk.$$
\end{cor}

\begin{proof} One has $\langle \pi_{\de,\mu}^*(x)\xi_{\de,\mu}^{\underline{a}},v\rangle =\langle \xi_{\de,\mu}^{\underline{a}},\pi_{\de,\mu}(x^{-1})v\rangle $. As  $\mu\in\bU\cap \bC^+$, the linear form  $\xi_{\de,\mu}^{\underline{a}}$ is given by the Poisson kernel. Hence  
 $$\langle\pi_{\de,\mu}^*(x) \xi_{\de,\mu}^{\underline{a}},v\rangle=\int_{  K}\P_{\de,\mu}^{\underline{a}}(k)\big(\pi_{\de,\mu}(x^{-1})v\big)(k) dk=\int_{  K}\P_{\de,\mu}^{\underline{a}}(k)  v(kx^{-1}) dk.$$
This gives the first equality.\me

Take  $\varphi\in C_c^\infty(G)$ such that  $L_{\de,\mu}(\varphi)=v$. As  $\P_{\de,\mu}^{\underline{a}}\in C(G,P,\de^*,\mu)$, Lemma  \ref{lem-crochet-dualite} implies 
$$\int_{  K}\P_{\de,\mu}^{\underline{a}}(k)  v(kx^{-1}) dk=
\int_G\P_{\de,\mu}^{\underline{a}}(g)\varphi(gx^{-1})   dg=\int_G\P_{\de,\mu}^{\underline{a}}(gx)\varphi(g)  dg.$$
For $x$ fixed, the map  $g\mapsto \P_{\de,\mu}^{\underline{a}}(gx)$ belongs also to  $ C(G,P,\de^*,\mu)$.
Again Lemma  \ref{lem-crochet-dualite} implies that   $\int_G\P_{\de,\mu}^{\underline{a}}(gx)\varphi(g)  dg=\int_K\P_{\de,\mu}^{\underline{a}}(kx) v(k)dk.$ 

\end{proof}
\subsection{Zeta functions associated to the minimal principal $H$-distinguished series}
\hfill

Let   $(dX, dY)$ be the pair of Haar measures on $V^+$ and  $V^-$   which were determined in Proposition  \ref{prop-mesures-1}. Recall that $d^*X$ and  $d^*Y$ are the $G$-invariant measures on  $V^+$ and  $V^-$   defined by 
$$d^*X=\dfrac{dX}{|\De_0(X)|^m},\quad dX\;\textrm{and  }\; \; d^*Y=\dfrac{dY}{|\nabla_0(X)|^m}\quad\textrm{ where }\;\; m= 1+\frac{kd}{2}.$$ 
(see (\ref{eq-d^*X})).\me


  By  Proposition \ref{prop-mesures-2}, this determines uniquely a $G$-invariant measure  $d_p\dot{g}$ on  $G/\cH_p$, for all $p\in\{1,\ldots r_0\}$.

Then the fixed Haar measure  $dg$ on  $G$ determines a Haar measure  $d_ph$ on  $\cH_p$  such that

$$\int_G f(g) dg=\int_{G/\cH_p}\int_{\cH_p} f(gh) d_ph \; d_p\dot{g},\quad f\in L^1(G).$$

\begin{definition}\label{def-pi(X)-fonction-zeta} Let  $p=1,\ldots,r_0$ and let  $\xi$ be a $\cH_p$-invariant linear form  on  $I_{\de,\mu}$ (i.e. an element of $(I_{\de,\mu}^*)^{\cH_{p}}$). 
For  $X\in G.I_p^+$ 	and  $Y\in G.I_p^-$, we define  
$$\pi_{\de,\mu}^*(X)\xi=\pi_{\de,\mu}^*(g)\xi\;\;{\rm if }\;\; X=g.I_p^+, g\in G,\quad{\rm } \quad \pi_{\de,\mu}^*(Y)\xi=\pi^*(g)\xi\;\;{\rm if }\;\; Y=g.I_p^-, g\in G.$$
(Remember that $\cH_{p}$ is the stabilizer of $I_p^+$ and of $I_p^-$).

Let  $z\in \C$ and  $w\in I_{\de,\mu}$. The zeta fonction associated to $(\pi_{\delta,\mu},z,\xi,w)$ are formally defined, for  $\Phi\in \cS(V^+)$, by 
 $$Z_p^+(\Phi, z,\xi,w)\index{zetap+@$Z_p^+(\Phi, z,\xi,w)$}=\int_{\Om_p^+} \Phi(X) |\De_0(X)|^z \langle\pi_{\de,\mu}^*(X)\xi,w\rangle d^*X,$$
and for  $\Psi\in  \cS(V^-)$, by 
$$Z_p^-(\Psi, z,\xi,w)\index{zetap-@$Z_p^-(\Psi, z,\xi,w)$}=\int_{\Om_p^-} \Psi(Y)  |\nabla_0(Y)|^z \langle\pi_{\de,\mu}^*(Y)\xi,w\rangle d^*Y.$$

 \end{definition}

\begin{rem}\label{rem-k=0-Om+} 1) If  $k=0$, then $G$ and $P$ are equal and  isomorphic to  $F^*$. Thus we have $\pi_{\de,\mu}=\de_\mu$ and we can take the generalized coefficient $\langle \pi^*(g)\xi,w\rangle$ to be equal to $(\de_\mu)^{-1}(g)$, $g\in G$. The spaces $V^\pm$ are isomorphic to $F$ and an element  $a\in F^*$ acts by the multiplication by $a$ on $V^+=F$, and by the multiplication by $a^{-1}$ on $V^-$. Thus, we have 
$$Z^+(\Phi, z,\xi,w)=Z(\Phi, \de^{-1},z-\mu),\quad{\rm and}\quad Z^-(\Psi,z,\xi,w)=Z(\Psi, \de,z+\mu),$$
where $Z(g,\de,s)$ is   Tate's zeta function (see section \ref{par-explicite}).

\me

2) Let  $\mu\in\bU$,  $\xi\in (I_{\de,\mu}^*)^{\cH_p}$ and  $w\in I_{\de,\mu}$. Then for  $z\in \C$ and  $\Phi\in \cS(\Om^+)$ (that is  $\Phi\in\cS(V^+)$ with compact support in $\Om^+$), the zeta function  $Z_p^+(\Phi, z,\xi,w)$ can be seen as a generalized coefficient of $\pi_{\de,\mu}$. More precisely if  $\varphi\in C_c^\infty(G)$ is such that $\Phi(g.I_p^+)=\int_{\cH_p}\varphi(gh) d_ph$ and if  $\varphi_z\in C_c^\infty(G)$ is the function defined by  $\varphi_z(g)=\varphi(g) |\De_0(g.I_p^+)|^z$, it is easily seen that  
$$Z_p^+(\Phi, z, \xi,w)=\langle\pi_{\de,\mu}^*(\varphi_z)\xi,w\rangle.$$

A similar result holds for the zeta functions $Z_p^-(\Psi, z,\xi,w)$, for $\Psi\in \cS(\Om^-)$.\end{rem}\me

We will now study the  zeta functions associated to the linear forms  $\xi_{\de,\mu}^{\underline{a}}$ by giving a connection between them and the zeta function  $\cK_c^\pm(f,\om,s)$ ($c\in\Sr_e^{k+1}$) introduced in Definition  \ref{def-zetaP}. \me

\begin{definition}\label{def-omdelta-smu}   For  $\mu=\sum_{j=0}^k \mu_j\la_j\in (a^{0*})_\C$, we define  $s(\mu)=(s_0,\ldots, s_k)\in(\C^*)^{k+1}$ by the relation 
 $$\rho_P-\mu=s_0\la_0+(s_0+s_1)\la_1+\ldots+(s_0+\ldots +s_k)\la_k.$$

If     $\de=(\de_0,\ldots,\de_k)$  is   a unitary character of  $(F^*)^{k+1}$, we define $\om(\de)\in\widehat{(F^*)^{k+1}}$ by  
$$\om(\de)=(\om_0,\ldots ,\om_k)=(\de_0^{-1},\de_0\de_1^{-1},\ldots, \de_{k-1}\de_k^{-1}).$$
Then $$  \om_0\ldots \om_j=\de_j^{-1},\; (\text{ for }j=0,\ldots,k)$$
and, using  Lemma \ref{lem-rhoP}:$$  s_0=\frac{kd}{4}-\mu_0,\;\;  s_j= \mu_{j-1}-\mu_j-\frac{d}{2},\quad (\text{ for }j=1,\ldots, k).$$
Also if  $s=(s_0,\ldots, s_k)\in\C^{k+1}$ and  $z\in\C$, we set  $s-z=(s_0-z, s_1,\ldots, s_k)$.
\end{definition}
\me
\begin{lemme}\label{lem-muj-sj>0} Let  $\mu\in (a_{0}^{*})_\C$. Then  $\de_P^{-1/2}\de_\mu$ is strictly  $P$-dominant (or equivalently $\mu\in {\bC}^+$) if and only if 
$s(\mu)=(s_0,\ldots, s_k)$ satisfies the condition  $\text{Re}(s_j)>0$ for  $j=1,\ldots, k$.
\end{lemme}
\begin{proof} By definition  $\mu\in\bC^+$ if and only if    $\langle \text{Re}(\mu)-\rho_P,\al\rangle>0$ for any root  $\al$ of $A$ in  $N$. From the definition of  $N$, for  such a root, there exists a pair $i<j$ such that   $\tilde{\go g}^\al\subset E_{i,j}(1,-1)$. Let  $\rho_P=\sum_{s=0}^k\rho_s\la_s$ and  $\mu=\sum_{s=0}^k \mu_s \la_s$.

Then the condition $\mu \in \bf C^+$ is equivalent to  $\text{Re}(\mu_{i}-\mu_{j})-(\rho_{i}-\rho_{j})>0$ for any pair  $i<j$.


From Lemma  \ref{lem-rhoP}, we get  $\rho_{i-1}-\rho_i=\frac{d}{2}$, and therefore  $\mu\in \bC^+$ if and only if    $\text{Re}(\mu_{i-1}-\mu_i)>\frac{d}{2}$ for $i=1,\ldots,k$.

As we have seen in the preceding definition that    $s_i=\mu_{i-1}-\mu_i-\frac{d}{2}$ for $i=1,\ldots,k$, the Lemma is proved.
\me 

\end{proof}
\begin{definition}\label{def-Poisson-Omega} Let  $p\in\{1,\ldots r_0\}$, $\underline{a}\in\cS_p$, $\om=(\om_0,\ldots ,\om_k)\in (\widehat{F^*})^{k+1}$ and  $s=(s_0,\dots, s_k)\in \C^{k+1}$. Define the functions  $\bP_{\om,s}^{\underline{a},\pm}$  \index{Pboldpm@$\bP_{\om,s}^{\underline{a},\pm}$}on  $V^\pm$ by 
$$\bP_{\om,s}^{\underline{a},+}(X)=\left\{\begin{array}{ccl} (\om,s)(\De(X))& {\rm if}& X\in P.I^+(\underline{a},1),\\
0 & {\rm if }& X\in V^+\setminus   P.I^+(\underline{a},1).\end{array}\right.$$
and 
$$\bP_{\om,s}^{\underline{a},-}(Y)=\left\{\begin{array}{ccl} (\om,s)(\nabla(Y))& {\rm if }& Y\in P.I^-(\underline{a},1),\\
0 & {\rm if}& Y\in V^-\setminus   P.I^-(\underline{a},1).\end{array}\right.$$

\end{definition}
\me
\begin{rem}\label{rem-bP-V+}  The union of the open $P$-orbits in  $V^+$ is the set $\{X\in V^+,   \De_j(X)\neq 0 \text {  for } j=0,\ldots ,k\}$. On the other hand we have $\Om_p^+\subset \Om^+=\{X\in V^+; \De_0(X)\neq 0\}$. Hence, if  $\text{Re}(s_j)>0$ for  $j=1,\ldots k$, the function   $\bP_{\om,s}^{\underline{a},+}$ is continuous on  $\Om_p^+$. If moreover $\text{Re}(s_0)>0$,   $\bP_{\om,s}^{\underline{a},+}$ is continuous on $V^+$. 

Similarly the union of the open  $P$-orbits in  $V^-$ is the set  $\{Y\in V^-, \nabla_j(Y)\neq 0 \text{ for } j=0,\ldots ,k\}$ and we have  $\Om^-_p\subset \Om^-=\{Y\in V^-; \nabla_0(Y)\neq 0\}$. Hence if $\text{Re}(s_j)>0$ for  $j=1,\ldots k$, the function $\bP_{\om,s}^{\underline{a},-}$ is continuous on  $\Om_p^-$ and if moreover  $\text{Re}(s_0)>0$, then  $\bP_{\om,s}^{\underline{a},-}$ is continuous on $V^-$.
\end{rem}
\begin{lemme}\label{lem-PV+} Let  $p\in\{1,\ldots, r_0\}$ and  $\underline{a}\in \cS_p$.
For $g\in G$, we have  $$\P^{\underline{a}}_{\de,\mu}(g)=c_{\de,\mu}(\underline{a})\bP^{\underline{a},+}_{\om(\de), s(\mu)}(g.I_{p}^+)=c_{\de,\mu}(\underline{a}) \bP^{\underline{a},-}_{\om(\de)^\sharp, s(\mu)^\sharp}(g.I_{p}^-),$$
where  $\om(\de)^\sharp$ and  $s(\mu)^\sharp$ are given in definition \ref{def-sharp}  and  $c_{\de,\mu}(\underline{a})=\prod_{j=0}^{k-1} \de_j(a_j)|a_j|^{-(\rho_j-\mu_j)}.$
\end{lemme}

\begin{proof} Let us prove the first equality. Remember that  $I(\underline{a},1)^+=\ga_{\underline{a}}.I_p^+$ (cf. Definition \ref{def-lienP-G-orbites}) . If  $g\notin P\ga_{\underline{a}}\cH_{p}$, then  $g.I_{p}^+\notin P.I^+(\underline{a},1)$, and hence the two members of the   first equality   are equal to zero. These two members are also left invariant by $N$ and right invariant by  $\cH_{p}$. Therefore it is enough to show the results for $g=l\ga_{\underline{a}}$ where  $l\in L$. Define then  $X=g.I_{p}^+=l.I^+(\underline{a},1)$. From Definition \ref{def-omdelta-smu} we have
$$\om(\de)(\De(X))=\prod_{j=0}^k \om_j(\De_j(X))=\de_0^{-1}(\De_0(X))(\de_1^{-1}\de_0)(\De_1(X))\ldots (\de_k^{-1}\de_{k-1})(\De_k(X))$$
$$=\de_0^{-1}(\frac{\De_0(X)}{\De_1(X)})\; \de_1^{-1}(\frac{\De_1(X)}{\De_2(X)})\ldots \de_{k-1}^{-1}(\frac{\De_{k-1}(X)}{\De_k(X)})\de_k^{-1}(\De_k(X)).$$
As $X=l.I^+(\underline{a},1)=a_0 x_0(l)X_0+a_1x_1(l)X_1+\ldots +a_{k-1} x_{k-1}(l) X_{k-1}+x_k(l) X_k$, we obtain:
 $$\om(\de)(\De(X))=\de_k(x_k(l))^{-1}\prod_{j=0}^{k-1}\de_j(x_j(l)a_j)^{-1}=\de(l)^{-1} \prod_{j=0}^{k-1}\de_j(a_j)^{-1}. $$
Also, as  $s(\mu)=(s_0,\ldots,s_k)$ is such that  $s_0+\ldots s_j=\rho_j-\mu_j$, we obtain  
$$|\De(X)|^{s(\mu)}=|\De_0(X)|^{s_0}|\De_1(X)|^{s_1}|\De_2(X)|^{s_2}\ldots |\De_k(X)|^{s_k}$$
$$=
\left(\frac{|\De_0(X)|}{|\De_1(X)|}\right)^{s_0}\left(\frac{|\De_1(X)|}{|\De_2(X)|}\right)^{s_0+s_1}\ldots \left(\frac{|\De_{k-1}(X)|}{|\De_k(X)|}\right)^{s_0+\ldots +s_{k-1}}|\De_k(X)|^{s_0+\ldots+s_k}$$
$$=|x_{k}(l)|^{\rho_{k}-\mu_{k}}\prod_{j=0}^{k-1}(a_{j}|x_{j}(l)|)^{\rho_{j}-\mu_{j}}.$$
Using Lemma  \ref{lem-rhoP}, we get  
$$|\De(X)|^{s(\mu)}=\big(\prod_{j=0}^{k-1}|a_j|^{\rho_j-\mu_j}\big)\de_P(l)^{1/2}\prod_{j=0}^k |x_j(l)|^{-\mu_j}.$$
Therefore
\begin{align*}\bP^{\underline{a},+}_{\om(\de), s(\mu)}(g.I_{p}^+)&=\om(\de)(\De(X))|\De(X)|^{s(\mu)}\\
&=\de(l)^{-1} \prod_{j=0}^{k-1}\de_j(a_j)^{-1}\big(\prod_{j=0}^{k-1}|a_j|^{\rho_j-\mu_j}\big)\de_P(l)^{1/2}\prod_{j=0}^k |x_j(l)|^{-\mu_j}\\
&=(c_{\de,\mu}(\underline{a}))^{-1}\P^{\underline{a}}_{\de,\mu}(g).
\end{align*}

The first equality is proved.\me

If $l\in L$ then  $\{l.I^-(\underline{a},1),H_0,l.I^+(\underline{a},1)\}$ is an $\go sl_2$-triple.  
From  corollaire \ref{cor-nabla-psi} we know that  $|\nabla(l.I^-(\underline{a},1))|^{s(\mu)^\sharp}=|\Delta(l.I^+(\underline{a},1))|^{s(\mu)}$, similarly it is easy to see that   $\om(\de)^\sharp(\nabla(l.I^-(\underline{a},1)))=\om(\de)(\De(l.I^+(\underline{a},1)))$. Hence  $\bP^{\underline{a},+}_{\om(\de), s(\mu)}(g.I_{p}^+)=  \bP^{\underline{a},-}_{\om(\de)^\sharp, s(\mu)^\sharp}(g.I_{p}^-)$. This gives the second  equality.

\end{proof}
\begin{theorem}\label{thm-Zeta-mero} Let $p\in\{1,\ldots, r_0\}$,   $\underline{a}\in\cS_p$ and $z\in\C$. Let also $\delta=(\delta_{0},\delta_{1},\ldots,\delta_{k})$ be a unitary character of $(F^*)^{k+1}$.

\begin{enumerate}
 \item \begin{enumerate}\item  Let  $\mu=\mu_0\la_0+\ldots \mu_k\la_k\in \bU\cap \bC^+$. Suppose     $\text{Re}(z-\mu_0)>m-\dfrac{kd}{4}=1+\dfrac{kd}{4}$.  
 
 Then,  for all  $\Phi\in\cS(V^+)$,  $\xi\in (I_{\de,\mu}^*)^{\cH_p}$ and  $w\in I_{\de,\mu}$,   the integral defining the zeta function   $Z_p^+(\Phi,z,\xi, w)$  (cf. Definition \ref{def-pi(X)-fonction-zeta})  is absolutely convergent  and the following relation holds:
$$ Z_p^+(\Phi,z,\xi_{\de,\mu}^{\underline{a}}, w) =c_{\de,\mu}(\underline{a})\index{cdeltamu@$c_{\de,\mu}(\underline{a})$}\sum_{x\in\Sr_e}\int_K \cK_{(x\underline{a},x)}^+(\cL(k)\Phi ,\om(\de),s(\mu)+z-m) w(k) dk,$$
where  $\cL$ \index{Lcal@$\cL$}is the left regular representation.

\item The function  $Z_p^+(\Phi,z,\xi_{\de,\mu}^{\underline{a}}, w)$  is a rational function in the variables $q^{\pm\mu_j}$ and  $q^{\pm z}$ and hence admits a meromorphic continuation.  Moreover there exists a polynomial  $\cR^+(\de,\mu,z)$ in the variables  $q^{\mu_j}$, $q^{-\mu_j}$ and  $q^{- z}$, which is a product   of polynomials of type $(1-d_jq^{-Nz-\sum_{j=0}^k m_j\mu_j})$, where  $N\in \N$, $m_j\in\Z$ and $d_j\in\C$ are independent of $z$ and of the $\mu_j$'s, such that for all  $\Phi\in \cS(V^+)$, $\xi\in (I_{\de,\mu}^*)^{\cH_p}$ and  $w\in I_{\de,\mu}$,  the function $ \cR^+(\de,\mu,z)Z_p^+(\Phi,z,\xi, w) $ is a  polynomial in $q^{\pm\mu_j}$ and  $q^{\pm z}$.

\end{enumerate}
\item \begin{enumerate}\item  Let  $\mu=\mu_0\la_0+\ldots \mu_k\la_k\in \bU\cap \bC^+$. Suppose that  $\text{Re}(z+\mu_k)>m-\dfrac{kd}{4}=1+\dfrac{kd}{4}$.  Then for all $\Psi\in\cS(V^-)$, $\xi\in (I_{\de,\mu}^*)^{\cH_p}$ and  $w\in I_{\de,\mu}$,     the integral defining the zeta function  $Z_p^-(\Psi,z,\xi, w)$ is absolutely convergent and the following relation holds:
$$Z_p^-(\Psi,z,\xi_{\de,\mu}^{\underline{a}}, w) =c_{\de,\mu}(\underline{a})\sum_{x\in\Sr_e}\int_K  \cK_{(x\underline{a},x)}^-(\cL(k)\Psi,\om(\de)^\sharp, s(\mu)^\sharp+z-m) w(k) dk,$$

\item The function  $Z_p^-(\Psi,z,\xi_{\de,\mu}^{\underline{a}}, w)$  is a rational function in the variables $q^{\pm\mu_j}$ and  $q^{\pm z}$ and hence admits a meromorphic continuation.  Moreover there exists a polynomial  $\cR^-(\de,\mu,z)$ in the variables   $q^{\mu_j}$, $q^{-\mu_j}$ and  $q^{- z}$, which is a product    of polynomials of type $(1-d_jq^{-Nz-\sum_{j=0}^k m_j\mu_j})$ where  $N\in \N$, $m_j\in\Z$ and $d_j\in\C$ are independent of $z$ and of the $\mu_j$'s,  such that for all  $\Psi\in \cS(V^-)$, $\xi\in (I_{\de,\mu}^*)^{\cH_p}$ and  $w\in I_{\de,\mu}$,  the function $ \cR^-(\de,\mu,z)Z_p^-(\Psi,z,\xi, w) $ is polynomial in $q^{\pm\mu_j}$ and  $q^{\pm z}$.\end{enumerate}
\end{enumerate}
\end{theorem}

\begin{proof} For $k=0$, the zeta functions considered here coincide with   Tate's zeta functions (see Remark \ref{rem-k=0-Om+}, 1)), for which the results are well-known.\me

From now on , we suppose that $k\geq 1$.\me

We have already mentioned that the set of   $\xi_{\de,\mu}^{\underline{a}}$ for $\underline{a}\in\Sr_e^k$ is a basis of  $(I_{\de,\mu}^*)^{\cH_p}$.  Therefore it is enough to prove the results for  $\xi=\xi_{\de,\mu}^{\underline{a}}$. Let us prove assertion (1).\me

Let  $\mu=\mu_0\la_0+\ldots +\mu_k\la_k\in \bU\cap \bC^+$. To simplify the notations we set $s=s(\mu)=(s_0,\ldots, s_k)$ and  $\om=\om(\de)$.  From Definition \ref{def-omdelta-smu} and Lemma \ref{lem-muj-sj>0}, one has  

a) $\text{Re}(s_j)=\text{Re}(\mu_{j-1}-\mu_j)-\dfrac{d}{2}>0$ if  $j=1,\ldots k$ and $s_0=\dfrac{kd}{4}-\mu_0$. 

b) Also  from our condition on $z$, we have  $\text{Re}(z+s_0)-m=\text{Re}(z-\mu_0)+\dfrac{kd}{4}-m>0$. 

Taking into account that $\delta$ is here a unitary character, we know from Definition \ref{def-zetaP} that  these  conditions imply that the functions $\cK_c^+(f,\om(\de), z+s(\mu)-m)$, $c\in\Sr_e^{k+1}$, are defined by an absolutely convergent integral  for all  $f\in\cS(V^+)$.  

Using Lemma  \ref{lemme-Ptilde},  we obtain for  $f\in \cS(V^+)$:

$$\sum_{x\in\Sr_e}\cK^+_{(x\underline{a},x)}(f,\om(\de), z+s(\mu)-m)=\sum_{x\in\Sr_e}\int_{\cO(x\underline{a},x)^+}f(X) \om(\de)(\De(X)) |\De(X)|^{s(\mu)+z-m} dX$$
$$
=\int_{P.I(\underline{a},1)^+} f(X) |\De_0(X)|^{z-m} \bP_{\om(\de),s(\mu)}^{\underline{a}, +}(X) dX=\int_{V^+} f(X) |\De_0(X)|^{z-m} \bP_{\om(\de),s(\mu)}^{\underline{a}, +}(X) dX.$$

Remark  \ref{rem-bP-V+}  and the conditions on  $\mu$,  $z$ and  $\Phi$, imply that the  the function 
$$(k,X)\mapsto (\cL(k)\Phi)(X) |\De_0(X)|^{z-m} \bP_{\om(\de),s(\mu)}^{\underline{a}, +}(X)$$ 
is continuous with compact support on  $  K\times V^+$. 

Then, from Lemma  \ref{lem-PV+} and the Theorem of Fubini, we get 
$$c_{\de,\mu}(\underline{a})\sum_{x\in\Sr_e}\int_K \cK_{(x\underline{a},x)}^+(\cL(k)\Phi ,\om(\de),s(\mu)+z-m) w(k) dk$$
$$=c_{\de,\mu}(\underline{a})\int_K \left(\int_{\Om_p^+} (\cL(k)\Phi)(X) |\De_0(X)|^z \bP_{\om(\de),s(\mu)}^{\underline{a}, +}(X) d^*X\right) w(k) dk$$
$$=\int_K\left(\int_{G/\cH_p} \Phi(k^{-1}g.I_p^+) |\De_0(g.I_p^+)|^z \P_{\de,\mu}^{\underline{a}, +}(g) d_p\dot{g} \right) w(k)dk$$
$$=\int_K\left(\int_{G/\cH_p} \Phi(g.I_p^+) |\De_0(g.I_p^+)|^z \P_{\de,\mu}^{\underline{a}, +}(kg) d_p\dot{g} \right) w(k)dk$$
(because  $|\chi_0|$ is trivial on  $K$)
$$=\int_{G/\cH_p} \Phi(g.I_p^+) |\De_0(g.I_p^+)|^z\left(\int_K\P_{\de,\mu}^{\underline{a}, +}(kg)  w(k)dk\right)d_p\dot{g} .$$

Then from Corollary  \ref{cor-pi-star(g)xi}, we obtain:
$$c_{\de,\mu}(\underline{a})\sum_{x\in\Sr_e}\int_K \cK_{(x\underline{a},x)}^+(\cL(k)\Phi ,\om(\de),z+s(\mu)-m) w(k) dk$$
$$=\int_{G/\cH_p} \Phi(g.I_p^+) |\De_0(g.I_p^+)|^z\langle \pi_{\de,\mu}^*(g)\xi_{\de,\mu}^{\underline{a}},w\rangle d_p\dot{g}$$
$$=Z_p^+(\Phi, z, \xi_{\de,\mu}^{\underline{a}}, w),$$
which is assertion  {(1)(a)}.\me

From  Theorem  \ref{thm-Sato-kp} there exists a polynomial  $R^+(\om, s)\in  \C[q^{-s_{0}},q^{-s_{1}},\ldots,q^{-s_{k}}]$ in the variables $q^{-s_{j}}$, product of polynomials of type $(1-b_j q^{-\sum_{j=0}^k N_js_j)}$ with $N_j\in \N$ and $b_j\in\C$, such that 
$R^+(\om, s)\cK_c^+(f,\om,s)\in \C[q^{\pm s_{0}},q^{\pm s_{1}},\ldots,q^{\pm s_{k}}]$ for all $f\in\cS(V^+)$ and  $c\in\Sr_e^{k+1}$.  But here we must consider integrals of the form $\cK_{(x\underline{a},x)}^+(\cL(k)\Phi ,\om(\de),z+s(\mu)-m) $, where by Definition \ref{def-omdelta-smu}, 
$$(z+s(\mu)-m)=(\frac{kd}{4}-\mu_{0}+z-m, \mu_{0}-\mu_{1}-\frac{d}{2},\ldots,\mu_{k-1}-\mu_{k}-\frac{d}{2}).$$
This implies immediately assertion (1)(b).\me

Remember that    $s(\mu)^\sharp=(-(s_0+\ldots +s_k), s_k,\ldots, s_1)$ with $s_0+\ldots+s_k=-\dfrac{kd}{4}-\mu_k$. Therefore   $\text{Re}(z-(s_0+\ldots+s_k))-m=\text{Re}(z+\mu_k)+\dfrac{kd}{4}-m>0$. 

A similar computation to the one in the proof of 1) b) above  shows thats for $g\in\cS(V^-)$ one has 
$$c_{\de,\mu}(\underline{a})\sum_{x\in\Sr_e} \cK_{(x\underline{a},x)}^-(g,\om(\de)^\sharp,z+s(\mu)^\sharp-m)=\int_{V^-} g(Y) |\nabla_0(Y)|^{z-m} \bP_{\om(\de)^\sharp,s(\mu)^\sharp}^{\underline{a}, -}(Y) dY.$$

Then the same arguments prove assertion 2)a) and 2)b).

 \end{proof}

\begin{theorem}\label{thm-principal} {\bf (Main Theorem, version 1)}\hfill

Let $k\geq 1$ (for $k=0$, the zeta functions we consider coincide with Tate's zeta functions by Remark \ref{rem-k=0-Om+}). Let  $\de$ be a unitary character of  $(F^*)^{k+1}$ and  $\mu\in\bU$. If  $\underline{a}\in \Sr_e^{k}$, we denote by  $p_{\underline{a}}$ the unique integer in  $\{1,\ldots ,r_0\}$ such that  $P.I(\underline{a},1)^+\subset \Om_{p_{\underline{a}}}^+$. 
Let also  $\ga_j$  ($j\in\N^*$)  be the function defined in Proposition \ref{prop-calcul-gamma}. 

Then for  $\Phi\in \cS(V^+)$ and $w\in I_{\de,\mu}$, the zeta functions satisfy the following functional equation:

$$Z_{p_{\underline{a}}}^-(\cF \Phi, \frac{m+1}{2}-z,\xi_{\de,\mu}^{\underline{a}},w)=\sum_{\underline{c}\in\Sr_e^k}\cB_{\underline{a},\underline{c}}(\de, \mu)(z)
Z_{p_{\underline{c}}}^+(\Phi, z+\frac{m-1}{2},\xi_{\de,\mu}^{\underline{c}},w),$$
where the functions  $\cB_{\underline{a},\underline{c}}(\de, \mu)(z)$ \index{Bcala@$\cB_{\underline{a},\underline{c}}(\de, \mu)(z)$} are given by 
$$\cB_{\underline{a},\underline{c}}(\de, \mu)(z)=c_{\de,\mu}(\underline{a}) c_{\de,\mu}(\underline{c})^{-1}\sum_{y\in\Sr_{e}} \Big(\ga_k((a_0,\ldots, a_{k-1}), y)\de_k(-1)\tilde{\rho}(\de_k,\mu_k-z+1;- y)$$
$$\times \prod_{j=1}^{k-1}\ga_j((a_0,\ldots, a_{j-1}), yc_j)\prod_{j=0}^{k-1}\de_j(-1)\tilde{\rho}(\de_j,\mu_j-z+1;- a_jc_jy)\Big).$$ 
Moreover these  functions  $\cB_{\underline{a},\underline{c}}(\de, \mu)(z)$ do not depend on the character of $(F^*)^{k+1}$ which defines  $\de_\mu$ (see Lemma \ref{lem-caractere-L}).
\end{theorem}

\begin{proof}  
Let us first show that for  $k\in K$, one has  $\cF(\cL(k)\Phi)=\cL(k)\cF(\Phi)$. 

By definition we have 
$$\cF(\Phi)(Y)=\int_{V^+} \Phi(X) \tau\circ b(X,Y) dX,\quad Y\in V^-.$$
As the measure  $d^*X=\frac{dX}{|\De_0(X)|^m}$ is  $G$-invariant and as $|\De_0(k.X)|=|\De_0(X)|$ for $k\in K$ and $X\in V^+$, the measure  $dX$ is  $K$-invariant. 

As the bilinear form  $b(\cdot,\cdot)$ is  $G$-invariant, we see that  the Fourier transform on $V^+$ commutes with the action of $K$:

$$\cL(k)\cF(\Phi)=\cF\big(\cL(k)\Phi),\text {  for }   k\in K.$$  

From  Theorem \ref{thm-Zeta-mero} (2) (a), we obtain 
$$Z_{p_{\underline{a}}}^-(\cF\Phi,\frac{m+1}{2}-z,\xi_{\de,\mu}^{\underline{a}},w)=c_{\de, \mu}(\underline{a})\sum_{x\in\Sr_e}
\int_K \cK_{(x\underline{a},x)}^-(\cF\big(\cL(k)\Phi\big),\om(\de)^\sharp,s(\mu)^\sharp-z+\frac{m+1}{2}-m) w(k) dk.$$
As $s(\mu)^\sharp-z+\frac{m+1}{2}=(s(\mu)+z-\frac{m+1}{2})^\sharp$ (see Definition \ref{def-sharp}),     Theorem  \ref{EF-(P,V^+)}  implies that 
$$Z_{p_{\underline{a}}}^-(\cF\Phi,\frac{m+1}{2}-z,\xi_{\de,\mu}^{\underline{a}},w) =c_{\de, \mu}(\underline{a})\sum_{x\in\Sr_e}\sum_{c\in\Sr_e^{k+1}}D^k_{((x\underline{a},x),c)}(\om(\de),s(\mu)+z-\frac{m+1}{2})$$
 \beq\label{eq-Zp-Kc+} \hskip 80pt  \times \int_K \cK_c^+(\cL(k)\Phi,\om(\de),s(\mu)+z-\frac{m+1}{2}) w(k) dk, \eeq

where the functions  $D^k$ satisfy the following induction relation:
\begin{align*}&D^k_{(a,c)}(\om',s')\\
& =\ga_k(\underline{a},c_k))(\om'_0\ldots \om'_k)(-1)\tilde{\rho}((\om'_0\ldots \om'_k)^{-1},-(s'_0+\ldots+s'_k+\frac{kd}{2});-a_kc_k)
D^{k-1}_{(\underline{a},\underline{c})}(\underline{\om'},\underline{s'}).\end{align*}

Proposition \ref{prop-calcul-gamma} implies that    $\ga_k(x\underline{a},c_k)=\ga_k(\underline{a}, xc_k)$ for $x\in \Sr_e$. Then by an easy induction from Theorem \ref{EF-(P,V^+)} we see that 
$D^k_{((x\underline{a},x), c)}(\om',s')=D^k_{((\underline{a},1), xc)}(\om',s')$. 

It follows that if we define  $A(\om',s')(\underline{a},c)=\sum_{x\in\Sr_e}D^k_{((x\underline{a},x), c)}(\om',s')$  then  $A(\om',s')(\underline{a},yc)=A(\om',s')(\underline{a},c)$ for all $y\in\Sr_e$. Then,  for $f\in\cS(V^+)$, 
$$\sum_{x\in\Sr_e}\sum_{c\in\Sr_e^{k+1}}D^k_{((x\underline{a},x), c)}(\om',s') \cK_c^+(f,\om',s')=\sum_{c\in\Sr_e^{k+1}}A_k(\om', s')(\underline{a},c) \cK_c^+(f,\om',s')$$
$$ =\sum_{\underline{c}\in\Sr_e^{k}}\sum_{x\in\Sr_e }A_k(\om', s')(\underline{a},(x\underline{c},x)) \cK_{(x\underline{c},x)}^+(f,\om',s') $$
$$=\sum_{\underline{c}\in\Sr_e^{k}}A_k(\om', s')(\underline{a},(\underline{c},1))\sum_{x\in\Sr_e}\cK_{(x\underline{c},x)}^+(f,\om',s').$$

Then by  Theorem \ref{thm-Zeta-mero} and the expression of  $Z_{p_{\underline{a}}}^-(\cF\Phi,\frac{m+1}{2}-z,\xi_{\de,\mu}^{\underline{a}},w)$ given in (\ref{eq-Zp-Kc+}), we obtain that 
$$Z_{p_{\underline{a}}}^-(\cF\Phi,\frac{m+1}{2}-z,\xi_{\de,\mu}^{\underline{a}},w)$$
$$=c_{\de, \mu}(\underline{a})\sum_{\underline{c}\in\Sr_e^k}A_k(\om(\de),s(\mu)+z-\frac{m+1}{2})(\underline{a},(\underline{c},1))c_{\de,\mu}(\underline{c})^{-1}Z_{p_{\underline{c}}}^+(\Phi,z+\frac{m-1}{2},\xi_{\de,\mu}^{\underline{c}},w).$$

In order to prove the functional equation it remains to show that the functions 
  $$\cB_{\underline{a},\underline{c}}(\de,\mu)(z) =c_{\de, \mu}(\underline{a})c_{\de,\mu}(\underline{c})^{-1}A_k(\om(\de),s(\mu)+z-\frac{m+1}{2})(\underline{a},(\underline{c},1))$$

have the   form requested in the statement. \me

From Theorem  \ref{EF-(P,V^+)},   we obtain
$$D^k_{((x\underline{a},x),(\underline{c},1))}(\om(\de),s(\mu)+z-\frac{m+1}{2})=D^k_{((\underline{a},1),(x\underline{c},x))}(\om(\de),s(\mu)+z-\frac{m+1}{2})$$
$$=\ga_k(\underline{a},x)(\om_0\ldots\om_k)(-1)\tilde{\rho}((\om_0\ldots\om_k)^{-1},-(s_0+\ldots+s_k+z-\frac{m+1}{2}+\frac{jd}{2});- x)$$
$$\times \prod_{j=1}^{k-1}\ga_j((a_0,\ldots, a_{j-1}), xc_j) \prod_{j=0}^k(\om_0\ldots \om_j)(-1)\tilde{\rho}((\om_0\ldots \om_j)^{-1},-(s_0+\ldots+s_j+z-\frac{m+1}{2}+\frac{jd}{2});- xa_jc_j).$$

From Definition \ref{def-omdelta-smu}  and the definition of  $s(\mu)$ we have $\om_0\ldots \om_j=\de_j^{-1}$ (if  $\om(\de)=(\om_0,\ldots,\om_k)$), and $s_0+\ldots +s_j=\rho_j-\mu_j$. Hence
$$-(s_0+\ldots+s_j+\frac{jd}{2}+z-\frac{m+1}{2})=-(\frac{(k-2j)d}{4}-\mu_j+\frac{jd}{2}+z-\frac{kd}{4}-1)$$
$$=\mu_j-z+1.$$
Hence 
$$D^k_{((x\underline{a},x),(\underline{c},1))}(\om(\de),s(\mu)+z-\frac{m+1}{2})$$
$$=\ga_k(\underline{a},x)\de_k(-1)\tilde{\rho}(\de_k,\mu_k-z+1;-x) \prod_{j=1}^{k-1}\ga_j((a_0,\ldots, a_{j-1}), xc_j) \prod_{j=0}^k\de_j(-1)\tilde{\rho}(\de_j,\mu_j-z+1;- xa_jc_j).$$
This implies the result concerning the functional equation.

It remains to show that  the functions $\cB_{\underline{a},\underline{c}}(\de, \mu)(z)$ do not depend on the character of $(F^*)^{k+1}$ which defines  $\de_\mu$.

To see this we take   $\Phi\in \cS(\Omega^+_{p_{\underline a}})$. Then from the functional equation we have
$$\cB_{\underline{a},\underline{c}}(\de, \mu)(z)=\frac{Z_{p_{\underline{c}}}^-(\cF \Phi, \frac{m+1}{2}-z,\xi_{\de,\mu}^{\underline{c}},w)}{Z_{p_{\underline{a}}}^+(\Phi, z+\frac{m-1}{2},\xi_{\de,\mu}^{\underline{a}},w)}.$$
By definition (\ref{def-pi(X)-fonction-zeta}) the zeta functions do only depend on the induction parameters.

 \end{proof}
 
 \begin{rem}\label{rem-independance}
 
 Of course one can also check directly that the functions  $\cB_{\underline{a},\underline{c}}(\de, \mu)(z)$ do not depend on the choice of the character $\de$ of   $(F^*)^{k+1}$ defining the character  $\de_\mu$ of  $L$. 
 
If  $\de'$ is another character  $(F^*)^{k+1}$ which defines the same character  $\de_\mu$  then, by Lemma \ref{lem-caractere-L}, for  $j\in\{0,\ldots ,k\}$,   there exists a character $\chi^j$ of  $ F^*$ with values in  $\{\pm1\}$ and trivial on $S_e$,  such that 
$ \de_{j}(x)=\chi^j(x) \de_j'(x)$  and  $ \prod_{j=0}^k \chi^j(x)=1$ for $  x\in F^*$. By definition (cf. Lemma   \ref{lem-PV+}), we have  
$c_{\de,\mu}(\underline{a})=\prod_{j=0}^{k-1} \de_j(a_j) |a_j|^{-(\rho_j-\mu_j)}$. Hence 
\beq\label{eq-delta/delta'}c_{\de',\mu}(\underline{a}) c_{\de',\mu}(\underline{c})^{-1}=\prod_{j=0}^{k-1} \chi^j(c_j) \chi^j(a_j) c_{\de,\mu}(\underline{a}) c_{\de,\mu}(\underline{c})^{-1}.\eeq

From the definition of  $\tilde{\rho}$ (cf. Definition \ref{def-rhotilde}), for a character  $\tilde{\de}$  of  $F^*$, $s\in\C$ and  $b\in\Sr_e$, we have  
$$\tilde{\rho}(\tilde{\de},s;b)=\frac{1}{|\Sr_e|}\sum_{\chi\in\widehat{\Sr_e}} \chi(b)\rho(\tilde{\de}\chi,s).$$
Then, for  $y\in\Sr_e$: $$\tilde{\rho}(\de_k\chi^k,\mu_k-z+1;- y)=\chi^k(-y) \tilde{\rho}(\de_k,\mu_k-z+1;- y)  $$ $$ \text{ and  } \tilde{\rho}(\de_j\chi^j,\mu_j-z+1;- a_jc_jy)=\chi^j(-ya_jc_j)\tilde{\rho}(\de_j,\mu_j-z+1;- a_jc_jy).$$
As  $\prod_{j=0}^k \chi^j(-y)=1$, relation   (\ref{eq-delta/delta'})  implies that   $\cB_{\underline{a},\underline{c}}(\de', \mu)(z)=\cB_{\underline{a},\underline{c}}(\de, \mu)(z)$.

 \end{rem}
 \me 
In the case where  $e=0$ or  $4$  we have   $\Sr_e=\{1\}$ and the groups  $G$ and  $P$ have a unique open orbit in  $V^+$. We therefore omit  the indices $p\in\{1\}$ and  $a\in\Sr_e^k$. Moreover, $L/L\cap \cH$ is isomorphic to $(F^*)^{k+1}$, 
therefore there exists a unique character  $\de$ of $(F^*)^{k+1}$ such that $\de(l)=\de(x_0(l),\ldots ,\de_k(l))$ for  $l\in L$. Then the functional equation is scalar and has the following form.
\begin{cor}\label{EF-globale-04} Let  $e=0$ or  $4$. Then for  $\Phi\in\cS(V^+)$, $\xi\in (I_{\de,\mu}^*)^{\cH}$ and  $w\in I_{\de,\mu}$, we have  
$$Z^-(\cF\Phi,\frac{(m+1)}{2}-z, \xi,w)=d(\de,\mu,z)\;\;  Z^+(\Phi,z+\dfrac{(m-1)}{2},\xi,w))$$

where  $d(\de,\mu,z) =\ga_\tau(q)^{\frac{k(k+1)}{2}}\prod_{j=0}^{k}\de_j(-1)\rho(\de_j,\mu_j+1-z)$.
\end{cor}

\begin{proof} As  $\Sr_e=\{1\}$,  we have  $\tilde{\rho}(\chi,s;1)=\rho(\chi,s)$ for all  $\chi\in\widehat{F^*}$ and  $s\in\C$. Also        $\ga_j(1,\ldots,1)=\ga_\tau(q)^j$ from Proposition  \ref{prop-calcul-gamma}. The Corollary is then an immediate consequence of  Theorem \ref{thm-principal}.

\end{proof}
\me 
\subsection{Another version of the functional equation.}
\hfill

Let  $\de$ be a unitary character of  $(F^*)^{k+1}$ and  $\mu\in\bU\cap \bC^+$.

 It is known from Theorem \ref{thm-Zeta-mero} that if  $\text{Re}(\mu_{0})<-1-\frac{kd}{4}$ then $Z_p^+(\Phi,0,\xi_{\de,\mu}^{\underline{a}},w)$ is defined by an absolutely convergent integral  for any $\Phi\in\cS(V^+)$ and  that if  $\text{Re}(\mu_{k})>1+\frac{kd}{4}$,  then $Z_p^-(\Psi,0,\xi_{\de,\mu}^{\underline{a}},w)$ is also defined by an absolutely  convergent  integral for any $\Psi\in\cS(V^-)$. 
 
 We will first describe the linear form $w\longmapsto Z_p^\pm(\Phi,0,\xi_{\de,\mu}^{\underline{a}},w) $ in terms of the duality (\ref{eq-dualite}).

\me

\begin{definition}\label{def-tildeZ} Let   $\de$ be a unitary character of $(F^*)^{k+1}$ and let  $\mu\in\bU\cap \bC^+$. Let also  $p\in\{1,\ldots, r_0\}$ and  $\underline{a}\in \cS_p$. If  $ \Phi\in \cS(V^+)$ and  $\Psi\in \cS(V^-)$, we define the two following functions of  $C(G,P,\de^*,\mu)$ (cf. (\ref{eq-CGP}))\,.
\me

-  for $\text{Re}(\mu_0)<-1-\frac{kd}{4}$ :
$$\tilde{Z}_p^+(\Phi, \pi_{\de,\mu}, \xi_{\de,\mu}^{\underline{a}})(g) \index{Ztildep+-@$\tilde{Z}_p^\pm(\Phi, \pi_{\de,\mu}, \xi_{\de,\mu}^{\underline{a}})(g)$}=c_{\de,\mu}(\underline{a})\int_{V^+} \Phi(g^{-1}X)\bP_{\om(\de),s(\mu)}^{\underline{a},+}(X)d^*X=\int_{G/\cH_p} \Phi( yI_p^+) \P_{\de,\mu}^{\underline{a}}(gy) dy \quad (g\in G) $$

 -  for $\text{Re}(\mu_{k})>1+\frac{kd}{4}$ :  
 $$\tilde{Z}_p^-(\Psi, \pi_{\de,\mu}, \xi_{\de,\mu}^{\underline{a}})(g)=c_{\de,\mu}(\underline{a})\int_{V^-} \Psi(g^{-1}X)\bP_{\om(\de)^\sharp,s(\mu)^\sharp}^{\underline{a},-}(X)d^*X=\int_{G/\cH_p} \psi(yI_p^-) \P_{\de,\mu}^{\underline{a}}(gy) dy \quad (g\in G).$$
  (The functions  $\P_{\de,\mu}^{\underline{a}}$, $\bP_{\om(\de),s(\mu)}^{\underline{a},+}$ et $\bP_{\om(\de)^\sharp,s(\mu)^\sharp}^{\underline{a},-}$ were given in  Definition \ref{def-noyauPoisson} and  \ref{def-Poisson-Omega} and the second equalities are a consequence of  Lemma \ref{lem-PV+}).

Note also that these functions are well defined because the Poisson kernels $\bP_{\om(\de),s(\mu)}^{\underline{a},+}$ and $ \bP_{\om(\de)^\sharp,s(\mu)^\sharp}^{\underline{a},-} $  are continuous for $\mu\in\bU\cap \bC^+$   satisfying the additional conditions $\text{Re}(\mu_0)<-1-\frac{kd}{4}$ and  $\text{Re}(\mu_{k})>1+\frac{kd}{4}$ respectively (see Remark \ref{rem-bP-V+}).
\end{definition}

Let  $C(K,P\cap K,\de^*)$   be the space of continuous functions  $v:K\longrightarrow \C$   such that  $v( pk)=\de(p)^{-1} v(k)$ for  $k\in K$ and  $p\in P\cap K$. The restriction to  $K$ is an isomorphism from  $C(G,P,\de^*,\mu)$ onto  $C(K,P\cap K,\de^*)$.

\begin{prop} Let   $\de$ be a unitary character of $(F^*)^{k+1}$ and let  $\mu\in\bU\cap \bC^+$. Let also  $p\in\{1,\ldots, r_0\}$ and  $\underline{a}\in \cS_p$

1) Let  $\varphi\in C_c^\infty(G)$, the linear form  $\pi_{\de,\mu}^*(\varphi)\xi_{\de,\mu}^{\underline{a}}$ is given (via the duality  (\ref{eq-dualite})) by the following function  in  $C(G,P,\de^*,\mu)$:
$$\pi_{\de,\mu}^*(\varphi)\xi_{\de,\mu}^{\underline{a}}(g)=\int_G\varphi(y) \P_{\de,\mu}^{\underline{a}}(gy) dy,\quad g\in G.$$

2) If   $\Phi^\pm\in C_c^\infty(\Om_p^\pm)$ we fix    $\varphi^\pm\in C_c^\infty(G)$ such that  $\Phi^\pm(g.I_p^\pm)=\int_{\cH_p} \varphi^\pm(gh) d_ph$ for  $g\in G$. Then 
$$\tilde{Z}_p^{\pm}(\Phi^\pm,\pi_{\de,\mu},\xi_{\de,\mu}^{\underline{a}})=\pi_{\de,\mu}^*(\varphi^\pm)\xi_{\de,\mu}^{\underline{a}}.$$
\end{prop}
\begin{proof} \hfill 

1) Using Corollary \ref{cor-pi-star(g)xi}, we obtain for   $v\in I_{\delta,\mu}$:
$$\langle \pi_{\de,\mu}^*(\varphi)\xi_{\de,\mu}^{\underline{a}},v\rangle=\int_{G}\varphi(y)\langle \pi_{\de,\mu}^*(y )\xi_{\de,\mu}^{\underline{a}},v\rangle dy =\int_{K}(\int_{G}\varphi(y) \P_{\de,\mu}^{\underline{a}}(ky)dy)v(k)dk.$$
Hence the linear form $ \pi_{\de,\mu}^*(\varphi)\xi_{\de,\mu}^{\underline{a}}$ is given via the duality (\ref{eq-dualite})) by the the function $g \mapsto \int_G\varphi(y) \P_{\de,\mu}^{\underline{a}}(gy) dy.$

2)  We only prove the assertion for $\tilde{Z}_p^+(\Phi, \pi_{\de,\mu}, \xi_{\de,\mu}^{\underline{a}})$, the proof for  $\tilde{Z}_p^-(\Psi, \pi_{\de,\mu}, \xi_{\de,\mu}^{\underline{a}})$ is similar.\me

From the definition of $\tilde{Z}_p^{+}(\Phi^+,\pi_{\de,\mu},\xi_{\de,\mu}^{\underline{a}})$ and from 1) we get
\begin{align*}
&\tilde{Z}_p^{+}(\Phi^+,\pi_{\de,\mu},\xi_{\de,\mu}^{\underline{a}})(g)= \int_{G/\cH_p} \Phi^+( yI_p^+) \P_{\de,\mu}^{\underline{a}}(gy) dy= \int_{G/\cH_p}( \int_{\cH_{p}}\varphi^+(yh)d_{p}h)\P_{\de,\mu}^{\underline{a}}(gy)dy\\
&=\int_{G}\varphi^+(y)\P_{\de,\mu}^{\underline{a}}(gy)dy=\pi_{\de,\mu}^*(\varphi^+)\xi_{\de,\mu}^{\underline{a}}(g).\cr
\end{align*}
 \end{proof}

\begin{prop}\label{prop-Ztilde} Let   $\de$ be a unitary character of $(F^*)^{k+1}$ and let  $\mu\in\bU\cap \bC^+$. Let also  $p\in\{1,\ldots, r_0\}$ and  $\underline{a}\in \cS_p$

1)  Let  $\Phi^\pm\in\cS(V^\pm)$. Then   the map 
$$\mu\mapsto \tilde{Z}_p^+(\Phi^+, \pi_{\de,\mu}, \xi_{\de,\mu}^{\underline{a}})_{|_{K}}\in C(K,P\cap K,\de^*)$$ extends meromorphically from the set  $\{\mu \in \bU\cap \bC^+; \text{Re}(\mu_0)<-1-\frac{kd}{4}\}$ to $\bU$, 

and the map  $$\mu\mapsto \tilde{Z}_p^-(\Phi^-, \pi_{\de,\mu}, \xi_{\de,\mu}^{\underline{a}})_{|_{K}}\in C(K,P\cap K,\de^*)$$ extends meromorphically from the set  $\{\mu \in \bU\cap \bC^+;  \text{Re}(\mu_k)>1+\frac{kd}{4}\}$ to $\bU$.\me

2) Using the duality bracket  (\ref{eq-dualite}), the linear form  on $I_{\de,\mu}$ represented by the functions $\tilde{Z}_p^\pm(\Phi, \pi_{\de,\mu}, \xi_{\de,\mu}^{\underline{a}})$    are  given for $w\in I_{\de,\mu}$,  by 
$$\langle \tilde{Z}_p^+(\Phi,\pi_{\de,\mu}, \xi_{\de,\mu}^{\underline{a}}),w\rangle= Z_p^+(\Phi, 0,  \xi_{\de,\mu}^{\underline{a}},w),\quad
\text{and}\quad
\langle \tilde{Z}_p^-(\Psi,\pi_{\de,\mu}, \xi_{\de,\mu}^{\underline{a}}),w\rangle= Z_p^-(\Psi, 0,  \xi_{\de,\mu}^{\underline{a}},w).$$

\end{prop}
\begin{proof} We only prove the two assertions for $\tilde{Z}_p^+(\Phi, \pi_{\de,\mu}, \xi_{\de,\mu}^{\underline{a}})$, the proof for  $\tilde{Z}_p^-(\Psi, \pi_{\de,\mu}, \xi_{\de,\mu}^{\underline{a}})$ is similar.\me

1) The proof  of  Theorem \ref{thm-Zeta-mero} shows that 
$$\tilde{Z}_p^+(\Phi, \pi_{\de,\mu}, \xi_{\de,\mu}^{\underline{a}})(g)=c_{\de,\mu}(\underline{a})\int_{V^+}\Phi(g^{-1}X) \bP_{\om(\de),s(\mu)}^{\underline{a},+}(X) d^*X$$
$$=c_{\de,\mu}(\underline{a})\sum_{x\in\Sr_e}\cK^+_{(x\underline{a},x)}(\cL(g)\Phi,\om(\de), s(\mu)-m).$$

Considering the right hand side of the preceding equality, Theorem  \ref{thm-Sato-kp} implies that, for a fixed $g\in G$, the map $$\mu\,\,\,\,\,\, \longmapsto \tilde{Z}_p^+(\Phi, \pi_{\de,\mu}, \xi_{\de,\mu}^{\underline{a}})(g)$$
is a rational function in the variables  $q^{\pm\mu_j}$ and hence extends to $\bU$ as a meromorphic function.     As the map  $g\mapsto \cL(g)\Phi$ is locally constant (because the representation $\cL$ on $\cS(V^+)$ is smooth), the function $\tilde{Z}_p^+(\Phi, \pi_{\de,\mu}, \xi_{\de,\mu}^{\underline{a}})$ is continuous  for any value of the parameter $\mu$ for which it is defined.  

Moreover as $\tilde{Z}_p^+(\Phi, \pi_{\de,\mu}, \xi_{\de,\mu}^{\underline{a}})(g)=  c_{\de,\mu}(\underline{a})\int_{V^+}\Phi( X) \bP_{\om(\de),s(\mu)}^{\underline{a},+}(gX) d^*X$  for $ \mu \in \bU\cap \bC^+$ such that $ \text{Re}(\mu_0)<-1-\frac{kd}{4}$,  Lemma \ref{lem-PV+} implies that  this function belongs to $C(G,P,\de^*,\mu)$. This remains true by analytic continuation.\me

  2)  Again from the proof of Theorem \ref{thm-Zeta-mero} we get:
  \begin{align*}Z_p^+(\Phi, z, \xi_{\de,\mu}^{\underline{a}}, w)&=
  c_{\de,\mu}(\underline{a})\sum_{x\in\Sr_e}\int_K \cK_{(x\underline{a},x)}^+(\cL(k)\Phi ,\om(\de), s(\mu)-m) w(k) dk\\
  {}&= \int_K (c_{\de,\mu}(\underline{a})\sum_{x\in\Sr_e}  \cK_{(x\underline{a},x)}^+(\cL(k)\Phi ,\om(\de), s(\mu)-m)) w(k) dk\\
  {}&=\int_{K}\tilde{Z}_p^+(\Phi, \pi_{\de,\mu}, \xi_{\de,\mu}^{\underline{a}})(k)w(k)dk\\
  {}&=\langle\tilde{Z}_p^+(\Phi, \pi_{\de,\mu}, \xi_{\de,\mu}^{\underline{a}}),w\rangle 
  \end{align*}

\end{proof}
Consider the representation $|\chi_0|^{-m}\otimes \pi_{\de,\mu}$ of  $G$ on  $I_{\de,\mu}$ given by  $|\chi_0|^{-m}\otimes \pi_{\de,\mu}(g)w=|\chi_0(g)|^{-m} \pi_{\de,\mu}(g)w$. If $p\in\{1,\ldots,r_0\}$, the  $\cH_p$-invariant linear forms  on  $I_{\de,\mu}$ are the same  for $|\chi_0|^{-m}\otimes \pi_{\de,\mu}$ and for $\pi_{\de,\mu}$ and for  $\underline{a}\in\cS_p$, we have  
$$(|\chi_0|^{-m}\otimes \pi_{\de,\mu})^*(g)\xi_{\de,\mu}^{\underline{a}}=|\chi_0(g)|^m \pi^*_{\de,\mu}(g)\xi_{\de,\mu}^{\underline{a}}.$$

 From the definition, if  $\Phi\in\cS(V^+)$ and   $w\in I_{\de,\mu}$, we have 
$$Z_p^+(\Phi,m, \xi_{\de,\mu}^{\underline{a}},w)=\int_{\Om_p^+} \Phi(X)  |\De_0(X)|^m\langle\pi_{\de,\mu}^*(X)\xi_{\de,\mu}^{\underline{a}},w\rangle d^*X$$
$$=\int_{\Om_p^+} \Phi(X)  \langle(|\chi_0|^{-m}\otimes \pi_{\de,\mu})^*(X) \xi_{\de,\mu}^{\underline{a}},w\rangle d^*X.$$
 As in Definition \ref{def-tildeZ}, for $\mu\in\bU\cap \bC^+$ such that $\text{Re}(\mu_0)<\frac{kd}{4}$ and $\Phi\in\cS(V^+)$, we define 
\begin{equation}\label{eq-Ztilde-rho}\tilde{Z}_p^+(\Phi,|\chi_0|^{-m}\otimes \pi_{\de,\mu},  \xi_{\de,\mu}^{\underline{a}})(g)=c_{\de,\mu}(\underline{a})\int_{V^+} \Phi(g^{-1}X)|\Delta_{0}(g^{-1}X)|^m\bP_{\om(\de),s(\mu)}^{\underline{a},+}(X)d^*X.\end{equation}

  Then the same arguments as in the proof of Proposition \ref{prop-Ztilde} show that 
\beq\label{eq-dualite-Ztilde-rho} \langle\tilde{Z}_p^+(\Phi,|\chi_0|^{-m}\otimes \pi_{\de,\mu},  \xi_{\de,\mu}^{\underline{a}}),w\rangle=Z^+_{p}(\Phi,m, \xi_{\de,\mu}^{\underline{a}},w) .\eeq

\begin{definition} Let  $\de$ be a unitary character of $(F^*)^{k+1}$ and  $\mu\in\bU$. Let $\xi=(\xi_1,\ldots, \xi_{r_0})\in \prod_{p=1}^{r_0}  (I_{\de,\mu}^*)^{\cH_p}$. For $\Phi\in \cS(V^+)$,  $\Psi\in \cS(V^-)$,   we define the following linear forms on  $I_{\de,\mu}$ :
$$\tilde{Z}^+(\Phi,|\chi_0|^{-m}\otimes \pi_{\de,\mu}, \xi)=\sum_{p=1}^{r_0} \tilde{Z}_p^+(\Phi,|\chi_0|^{-m}\otimes \pi_{\de,\mu}, \xi_p)$$ 
and also
$$ \tilde{Z}^-(\Psi,\pi_{\de,\mu}, \xi)=\sum_{p=1}^{r_0} \tilde{Z}_p^-(\Psi,\pi_{\de,\mu}, \xi_p).$$
\end{definition}

\begin{theorem}\label{thm-mainbis}   {\bf (Main Theorem, version 2)}\hfill

Let   $\de$ be a unitary character of  $(F^*)^{k+1}$ and  $\mu\in\bU$. There exists an operator  $\displaystyle A^{\de,\mu}\in{\rm End}\Big( \prod_{p=1}^{r_0} (I_{\de,\mu}^*)^{\cH_p}\Big)$ such that for all $\Phi\in \cS(V^+)$ and all  $\xi\in \prod_{p=1}^{r_0}  (I_{\de,\mu}^*)^{\cH_p}$ the following functional equation is satisfied:
$$\tilde{Z}^-(\cF(\Phi),\pi_{\de,\mu},\xi)=\tilde{Z}^+(\Phi, |\chi_0|^{-m}\otimes \pi_{\de,\mu},A^{\de,\mu}\xi).$$
The operator $A^{\de,\mu}$ is represented by a square matrix of size  $r_0$ whose coefficients  $A_{p,q}^{\de,\mu}$ belong to ${\rm Hom} \Big((I_{\de,\mu}^*)^{\cH_q},(I_{\de,\mu}^*)^{\cH_p}\Big)$. In the bases   $(\xi_{\de,\mu}^{\underline{a}})_{\underline{a}\in \cS_p}$  and  $(\xi_{\de,\mu}^{\underline{c}})_{\underline{c}\in \cS_q}$ of  $(I_{\de,\mu}^*)^{\cH_p}$ and  $(I_{\de,\mu}^*)^{\cH_q}$ respectively, the matrix of  of  $A_{p,q}^{\de,\mu}$ is given by: 
$$\Big(A_{p,q}^{\de,\mu}\Big)_{\underline{a},\underline{c}}= \cB_{\underline{a},\underline{c}}(\de,\mu)(\frac{m+1}{2}),$$
where    $ \cB_{\underline{a},\underline{c}}(\de,\mu)(z)$ has been defined in  Theorem \ref{thm-principal}.
\end{theorem}
\begin{proof}\hfill

 Let $p\in\{1,\ldots, r_0\}$ and  $\underline{a}\in\cS_p$.  Let $\xi\in \prod_{s=1}^{r_0}  (I_{\de,\mu}^*)^{\cH_s}$ be the vector whose  $p$-th component is equal to $\xi_{\de,\mu}^{\underline{a}}$ and all other components are equal to $0$.

Using    Proposition \ref{prop-Ztilde},  Theorem \ref{thm-principal} and (\ref{eq-Ztilde-rho}), we obtain for  $w\in I_{\de,\mu}$:
\begin{align*}& \langle \tilde{Z}^-(\cF(\Phi),\pi_{\de,\mu}, \xi),w\rangle=\langle \tilde{Z}_p^-(\cF(\Phi),\pi_{\de,\mu}, \xi_{\de,\mu}^{\underline{a}}), w\rangle=Z_p^-(\cF(\Phi), 0, \xi_{\de,\mu}^{\underline{a}},w)\\
&= \sum_{\underline{c}\in \Sr_e^{k}} \cB_{\underline{a},\underline{c}}(\de,\mu)(\frac{m+1}{2})Z_{p_{\underline{c}}}^+(\Phi,m, \xi_{\de,\mu}^{\underline{c}},w)\\
&=\sum_{q=1}^{r_0}\sum_{\underline{c}\in \cS_q}\cB_{\underline{a},\underline{c}}(\de,\mu)(\frac{m+1}{2})Z_{q}^+(\Phi,m, \xi_{\de,\mu}^{\underline{c}},w)\\
 &=\sum_{q=1}^{r_0}\sum_{\underline{c}\in \cS_q}\cB_{\underline{a},\underline{c}}(\de,\mu)(\frac{m+1}{2})\langle \tilde{Z}_q^+(\Phi, \chi_0^{-m}\otimes \pi_{\de,\mu},\xi_{\de,\mu}^{\underline{c}}),w\rangle\\
&=\langle\tilde{Z}^+(\Phi, |\chi_0|^{-m}\otimes \pi_{\de,\mu},A^{\de,\mu}\xi),w\rangle.
 \end{align*}
 \end{proof}

\me 
\subsection{$L$-functions and $\ep$-factors.}
\hfill

In this section we will always suppose that  $e=0$ or $4$. This implies that the groups $G$ and $P$ have both a unique open orbit in $V^\pm$ which we denote by $\Om^\pm$ and  $\cO^\pm$ respectively.

 Then  $\Om^\pm\simeq G/H$ where  $H$ is the centralizer of $I^\pm$ in  $G$.\me

We fix a unitary character  $\de=(\de_{0},\ldots,\de_{k})\in\widehat{F^*}^{k+1}$ and  $\mu\in\bU$. As before we denote by  $(\pi_{\de,\mu},I_{\de,\mu})$ the minimal spherical principal series. As $e=0$ or $4$, the group $S_e$ is equal to  $F^*$ and hence   $\Sr_e=\{1\}$. Therefore, in what follows, we will omit the indexes in this space.  The space  $(I_{\de,\mu}^*)^H$ is $1$-dimensional  and we fix  a non zero linear form  $\xi$ in this space.\me

 The aim of this section is to prove an analogue of  (\cite{G-J} Theorem 3.3.3)) for the minimal spherical principal series.

\begin{definition}\label{def-facteur-eulerien} \hfill

\begin{enumerate}\item An Euler factor \index{Euler factor} is a function $L$ on  $\C$ of the form  $\displaystyle L(s)=\frac{1}  {P(q^{-s})}$ where  $P\in\C[T]$ is a polynomial such that  $P(0)=1$.  

\item We denote by  $\cE^\pm$ the set of Eulor factors  $L^\pm(z)$ such that for all $\Phi^\pm\in\cS(V^\pm)$ and all  $w\in I_{\de,\mu}$ the quotient  
$\displaystyle \frac{Z^\pm(\Phi^\pm,z+\frac{1}{2}(m-1),\xi,w)}{L^\pm(z)}$ is a polynomial in the variables  $q^{-z}$ and  $q^z$.
\end{enumerate}
\end{definition}
\me
  We know from Theorem \ref{thm-Zeta-mero}, that there exist polynomials  $\cR^\pm (\de,\mu,z)$ in the variable  $q^{-z}$, which are products of polynomials of the form  $(1-cq^{-Nz})$ ($c\in\C$ and $N\in\N^*$)  such that, for all  $w\in I_{\de,\mu}$ and for all $\Phi^\pm\in\cS(V^\pm)$, the product  $\cR^\pm(\de,\mu,z) Z^\pm(\Phi^\pm, z,\xi,w)$ is a polynomial in the variables  $q^{-z}$ and  $q^z$.\me

Hence  $\cE^\pm\neq \emptyset$.
\begin{lemme}\label{lem-ideal-poly} \hfill

  {\rm(1)} Let  $L_0^+(z)=P_0(q^{-z})^{-1}\in\cE^+$. We denote by   $\cJ_{P_0}$ the set of Laurent polynomials  $P\in\C[T,T^{-1}]$ such that there exists  finite families  $(\Phi_j)_{j\in J}$ in  $\cS(V^+)$  and  $(w_j)_{j\in J}$ in  $ I_{\de,\mu}$   such that 
 $$\sum_{j\in J}Z^+(\Phi_j, z+\frac{1}{2}(m-1),\xi,w_j)= \frac{P(q^{-z},q^z)}{P_{0}(q^{-z})}.$$
 Then  $\cJ_{P_0}$ is an ideal of  $\C[T,T^{-1}]$.
 
 {\rm(2)}  Let  $L_0^-(z)=Q_0(q^{-z})^{-1}\in\cE^-$. We denote by $\cJ_{Q_0}$ the set of Laurent polynomials $Q\in\C[T,T^{-1}]$ such that there exists  finite families  $(\Psi_j)_{j\in J}$ in  $\cS(V^+)$ and  $(w_j)_{j\in J}$ in  $ I_{\de,\mu}$  such that 
 $$\displaystyle \sum_{j\in J}Z^-(\Psi_j, z+\frac{1}{2}(m-1),\xi,w_j)= \frac{Q(q^{-z},q^z)}{Q_{0}(q^{-z})}.$$ 
 Then  $\cJ_{Q_0}$ is an ideal of  $\C[T,T^{-1}]$.

\end{lemme}
\begin{proof} We only prove assertion $(1)$, the proof of (2) is similar. As $e=0$ or $4$, one has  $\chi_0(G)=F^*$ (Theorem \ref{thm-orbites-e04}), and  therefore there exists $g\in G$  such $|\chi_0(g)|=q^{-r}$ (for any $r\in \Z$). 

Let  $\Phi\in\cS(V^+)$ and  $w\in I_{\de,\mu}$.  Define:
$$\Phi_r(X)=\Phi(g^{-1}X)|\chi_{0}(g)|^{-\frac{1}{2}(m-1)} \text{ and } w_r=\pi_{\de,\mu}(g)w$$. 

As $m=1+\frac{kd}{2}$ in our case, we know from  Theorem \ref{thm-Zeta-mero} (1)(a), that for  $\mu=\mu_0\la_0+\ldots +\mu_k\la_k\in\bU\cap \bC^+$ such that  ${\rm Re}(z-\mu_0)>1$,
the zeta function $Z^+(\Phi_r, z+\frac{1}{2}(m-1),\xi, w_r)$ is defined by an absolutely convergent integral. More precisely we have:

$$\begin{array}{lll} \displaystyle
Z^+(\Phi_r, z+\frac{1}{2}(m-1),\xi, w_r)&=&\displaystyle\int_{V^+}\Phi_r(X)|\Delta_{0}(X)|^{z+\frac{1}{2}(m-1)}\langle \pi_{\de,\mu}^*(X)\xi,w_r\rangle d^*X\cr
{}&=&\displaystyle\int_{V^+}\Phi(g^{-1}X)|\chi_{0}(g)|^{-\frac{1}{2}(m-1)}|\Delta_{0}(X)|^{z+\frac{1}{2}(m-1)}\langle \pi_{\de,\mu}^*(X)\xi,\pi_{\de,\mu}(g)w\rangle d^*X\cr
{}&=&\displaystyle\int_{V^+}\Phi(g^{-1}X)|\chi_{0}(g)|^{-\frac{1}{2}(m-1)}|\Delta_{0}(X)|^{z+\frac{1}{2}(m-1)}\langle \pi_{\de,\mu}^*(g^{-1}X)\xi,w\rangle d^*X\cr
{}&=&\displaystyle\int_{V^+}\Phi( X)|\chi_{0}(g)|^{z}|\Delta_{0}(X)|^{z+\frac{1}{2}(m-1)}\langle \pi_{\de,\mu}^*(X)\xi,w\rangle d^*X\cr
{}&=&\displaystyle q^{-zr}Z^+(\Phi, z+\frac{1}{2}(m-1),\xi, w).

\end{array}$$
By analytic continuation  we obtain the following equality of rational functions:
 $$Z^+(\Phi_r, z+\frac{1}{2}(m-1),\xi, w_r)=  q^{-rz}Z^+(\Phi, z+\frac{1}{2}(m-1),\xi, w)$$

This shows that if  $P(T,T^{-1})\in\cJ_{P_0}$ then, for all $r\in\Z$,      $T^rP(T,T^{-1})\in \cJ_{P_0}$. Hence $\cJ_{P_0}$ is an ideal (the fact that $\cJ_{P_0}$ is stable under addition is obvious from the definition).

\end{proof}
\begin{prop}\label{prop-L(z,pi)}  \hfill

\begin{enumerate}\item There exist unique Euler factors   $L^\pm(\pi_{\de,\mu},z)$  \index{Lpideltamu@$L^\pm(\pi_{\de,\mu},z)$}(called $L$-functions)  satisfying the two following conditions: 

(a) For all $\Phi^\pm\in \cS(V^\pm)$  and  $ w \in  I_{\de,\mu}$, the quotients   $$\dfrac{Z^\pm(\Phi^\pm,z+\frac{m-1}{2},\xi,w)}{L^\pm(\pi_{\de,\mu},z)} $$ are polynomials in the variables  $q^z$ and  $q^{-z}$,\me

(b) There exists two  finite families   $(\Phi_i^\pm, v_i)_{i\in I}$ in  $\cS(V^\pm)\times I_{\de,\mu}$  such that 
$$\sum_{i\in I} \dfrac{Z^\pm(\Phi_i^\pm,z+\frac{m-1}{2},\xi,v_i)}{L^\pm(\pi_{\de,\mu},z)}=1.$$

\item  Moreover, if  $L^\pm\in\cE^\pm$,     then the quotients  $\dfrac{L^\pm(\pi_{\de,\mu},z)}{L^\pm(z)}$  are polynomials in  $q^{-z}$.
\end{enumerate}

\end{prop}
\begin{proof}\hfill

We only prove the result for $L^+(\pi_{\de,\mu},z)$.  The proof for  $L^-(\pi_{\de,\mu},z)$ is similar.\me

 Let  $L_0(z)=P_0(q^{-z})^{-1}\in\cE^+$. From Lemma \ref{lem-ideal-poly}, the set  $\cJ_{P_0}$ is an ideal of the principal ideal domain $\C[T,T^{-1}]$. Hence there exists a polynomial   $R_0\in \C[T]$ such that   $\cJ_{P_0}=R_0(T)\C[T,T^{-1}]$ and we can suppose that either  $R_0=0$ or $R_0(0)=1$.\me

We show first that  $R_0\neq 0$. 

 Let $w\in I_{\de,\mu}$ such that  $\langle \xi,w\rangle\neq 0$. As  $w$ is right invariant under an open compact subgroup,   there exists an open compact neighborhood $\cV$ of $I^+$ in  $\Om^+$ such that for all  $X\in \cV$,   $ \langle \pi_{\de,\mu}^*(X)\xi,w\rangle= \langle \xi,w\rangle\neq 0$. Let   $\Phi\in \cS(V^+)$ be the characteristic function of  $\cV$. Then for all  $z\in\C$,  we have  $Z^+(\Phi, z+\frac{1}{2}(m-1),\xi,w)=  \langle \xi,w\rangle\int_\cV |\De_0(X)|^{z+\frac{1}{2}(m-1)}d^*X\neq 0$. Hence  $\cJ_{P_0}\neq \{0\}$ and $R_0\neq 0$.\me

Therefore we can take   $R_0(0)=1$. \me

Let  $Q(T)$ be the HCF of $R_0$ and  $P_0$ normalized by the condition  $Q(0)=1$. There exist coprime polynomials  $R_1$  and  $P_1$ de $\C[T]$ such that  $R_0=R_1Q$, $P_0=P_1Q$  and  $P_1(0)=R_1(0)=1$.\me

If  $R_1\neq 1$ then  $R_1$ has a non zero root, and it exists  $z_1\in \C$ such that 
$$R_1(q^{-z_1})=0,\quad{\rm and }\quad P_1(q^{-z_1})\neq 0.$$
In that case,  for all $\Phi\in\cS(V^+)$ and  $w\in I_{\de,\mu}$ we would have  $Z^+(\Phi,z_1+\frac{1}{2}(m-1),\xi,w)=0$, and this is not possible  (take for example  $w$ and  $\Phi$ as above).\me

Therefore $R_1=1$, $Q=R_{0}$ and  $P_0=P_1R_0$. Hence for any  $P\in\C[T,T^{-1}]$, there exist finite families $(\Phi_i)_{i\in I}$ in $\cS(V^+)$ and  $(w_i)_{i\in I}$ in  $I_{\de,\mu}$ such that  
$$\sum_{i\in I} Z^+(\Phi_i, z+\frac{1}{2}(m-1), \xi, w_i)=\frac{R_0(T)P(T,T^{-1})}{P_0(T)}=\frac{P(T,T^{-1})}{P_1(T)}.$$
Define   $L^+(\pi_{\de,\mu},z)=\dfrac{1}{P_1(q^{-z})}$. Then from the discussion above we see that this Euler factor satisfies condition  (1) (a). Moreover if we take  $P=1$  in the equation above we obtain the condition  (1)(b). 

If  $\tilde{L}$ is another Euler factor satisfying conditions   (1) (a)  (1)(b),  then   $\dfrac{L^+(\pi_{\de,\mu},z)}{\tilde{L}(z)}\in\C[q^{-z}]$ and  $\dfrac{\tilde{L}(z)}{L^+(\pi_{\de,\mu},z)}\in\C[q^{-z}]$.  Hence   $\tilde{L}(z)=L(\pi_{\de,\mu},z)$ and the uniqueness is proved.

Let us show assertion $(2)$. From the construction  of $L(\pi_{\de,\mu},z)$ above we see that we can take $L^+=L_{0}$. Then $\dfrac{L(\pi_{\de,\mu},z)}{L_0(z)}=\dfrac{P_0(q^{-z})}{P_1(q^{-z})}= R_0(q^{-z})$.  \me

\end{proof}
\begin{definition}\label{def-Xi+-} \hfill

If $w\in I_{\de,\mu}$, $\Phi\in\cS(V^+)$ and  $\Psi\in\cS(V^-)$, we define
$$\Xi^+(\Phi,z,\xi,w)\index{ximajuscule@$\Xi^\pm(\Phi,z,\xi,w)$}=\frac{Z^+(\Phi, z+\frac{1}{2}(m-1), \xi,w)}{L^+(\pi_{\de,\mu},z)},$$
and  $$\Xi^-(\Psi,z,\xi,w)=\frac{Z^-(\Psi, z+\frac{1}{2}(m-1), \xi,w)}{L^-(\pi_{\de,\mu},z)}.$$
From  Proposition \ref{prop-L(z,pi)}, these two functions are polynomials in  $q^{-z}$ and  $q^z$.
\end{definition}

If  $k=0$ (that is in the case of Tate's theory) we denote by  $L_0(\de,z)$ the $L$-function associated to the character $t\longmapsto \delta(t)|t|^z$ (remember that here $\delta$ is a character of $F^*$). It is given by

\begin{equation}\label{eq-L0}L_0(\de,z)\index{L0@$L_0(\de,z)$}=\left\{\begin{array}{cl} (1-\de(\pi) q^{-z})^{-1} & {\rm if }\; \de\; \text{is unramified}\\ 1 & {\rm if }\; \de\; \text{is ramified},\end{array}\right.\end{equation}

(see \cite{B-H} (23.4.1))
 
We define now the $\ep$-factor $\ep_0(\de, z,\psi)$\index{epsilon0@$\ep_0(\de, z,\psi)$} by the equation:

\begin{equation}\label{eq-rho-L-ep}\de(-1)\rho(\de^{-1}, 1-z) =\rho(\de,z)^{-1}= \frac{L_0( \de^{-1},1-z)\ep_0(\de, z,\psi)}{L_0(\de,z)},\end{equation}

where Tate's $\rho$ function was defined in section \ref{par-explicite}. (This is the Definition given in (\cite{B-H}, \textsection 23.4 p.142), taking into account that our $\rho$ is the inverse of the $\gamma$ function of Bushnell and Henniart).

The function  $\ep_0(z,\de,\psi)$ is explicitly known (\cite{B-H} Theorem 23.4 p.144). It is always  of the form $c_0q^{-n_0z}$ where $c_0\in\C$ and  $n_0\in\Z$.

Moreover we know from (\cite{B-H} (23.4.2), p.142) that 
\begin{equation}\label{epsilon}\ep_0(\de, z,\psi)=\frac{\delta(-1)}{\ep_{0}(\delta^{-1},1-z,\psi)}
\end{equation}

 \me
 
 \begin{theorem}\label{thm-analogueGJ} \hfill

{\rm(1)} Let $\Phi\in\cS(V^+)$ and $\Psi\in\cS(V^-)$. The functions $\Xi^+(\Phi,z,\xi,w)$ and  $\Xi^-(\Psi,z,\xi,w)$ satisfy the following functional equations:
$$\Xi^-(\cF(\Phi),1-z,\xi,w)=\ep^+(\pi_{\de,\mu},z,\psi)\index{epsilon+-@$\ep^+(\pi_{\de,\mu},z,\psi)$} \Xi^+(\Phi,z,\xi,w), $$
$$\Xi^+(\cF(\Psi), 1-z,\xi, w)=\ep^-(\pi_{\de,\mu},z,\psi)\Xi^-(\Psi,z,\xi,w), ,$$
where  $$\ep^+(\pi_{\de,\mu},z,\psi)=\ga_\psi(q)^{\frac{k(k+1)}{2}}\frac{L^+(\pi_{\de,\mu},z)}{L^-(\pi_{\de,\mu},1-z)}\prod_{j=0}^{k}\frac{L_0(\de_j,1-(z-\mu_j))}{L_0(\de_j^{-1}, z-\mu_j)}\ep_0(\de_j^{-1}, z-\mu_j,\psi),$$
and 
$$\ep^-(\pi_{\de,\mu},z,\psi)  = \ga_{\psi}(q)^{-\frac{k(k+1)}{2}}\frac{L^-(\pi_{\de,\mu},z) }{L^+(\pi_{\de,\mu},1-z)}\prod_{j=0}^k\frac{L_0(\de_j^{-1}, 1-(\mu_j+z))}{L_0(\de_j, \mu_j+z)}\ep_0(\de_j , \mu_j+z,\psi).$$
{\rm(2)} The  $\ep$-factors $\ep^+(\pi_{\de,\mu}, z,\psi)$ and  $\ep^-(\pi_{\de,\mu},z,\psi)$ are monomials of the form $cq^{-nz}$ where $c\in\C$ and  $n\in\Z$, and they satisfy the relation:
$$\ep^+(\pi_{\de,\mu}, z,\psi)\ep^-(\pi_{\de,\mu},1-z,\psi)=(\de_0\ldots \de_k)(-1).$$
\end{theorem}
\me
\begin{proof} \hfill     

Let  $\Phi\in \cS(V^+)$. Corollary \ref{EF-globale-04} implies that
$$\Xi^-(\cF(\Phi),1-z,\xi,w)=\frac{1}{L^-(\pi_{\de,\mu},1-z) }Z^-(\cF(\Phi),1-z+\frac{1}{2}(m-1),\xi,w)$$
$$=\frac{d(\de,\mu,z) }{L^-(\pi_{\de,\mu},1-z)}Z^+(\Phi, z+\frac{1}{2}(m-1),\xi,w)=d(\de,\mu,z)\frac{L^+(\pi_{\de,\mu}, z)}{L^-(\pi_{\de,\mu},1-z)}
\Xi^+(\Phi,z,\xi,w),$$
where $d(\de,\mu,z)=\ga_\psi(q)^{\frac{k(k+1)}{2}}\prod_{j=0}^{k}\delta_{j}(-1) \rho(\de_j,\mu_j+1-z)$. 
We define then
\begin{equation}\label{eq-defep+}\ep^+(\pi_{\de,\mu}, z, \psi)=d(\de,\mu,z)\frac{L^+(\pi_{\de,\mu},z)}{L^-(\pi_{\de,\mu}, 1-z)}.\eeq
From  (\ref{eq-rho-L-ep}) and (\ref{epsilon}) we obtain
$$d(\de,\mu,z)= \ga_\psi(q)^{\frac{k(k+1)}{2}}\prod_{j=0}^{k}\ep_0(\de_j^{-1},z-\mu_j,\psi)\frac{L_0(\de_j,1-(z-\mu_j))}{L_0(\de_j^{-1}, z-\mu_j)},$$

and this proves the first functional equation.\me


Let now  $\Psi\in \cS(V^-)$. Applying Corollary  \ref{EF-globale-04} to the function $\Phi=\cF(\Psi)$, we get 
\begin{equation}\label{eq-FFPhi} Z^-(\cF\circ\cF(\Psi), \frac{m+1}{2}-z,\xi,w)=d(\de,\mu,z) Z^+(\cF(\Psi),z+\frac{m-1}{2},\xi,w).\end{equation}

As $\cF\circ\cF(\Psi)(Y)=\Psi(-Y)$,   Theorem \ref{thm-Zeta-mero} (2)(a) implies that for  $z\in\C$ and  $\mu\in\bU\cap \bC^+$ such  that ${\rm Re}(\mu_k-z)>0$, we have: 
$$Z^-(\cF\circ\cF(\Psi), \frac{m+1}{2}-z,\xi,w)=\int_{V^-}\Psi(-Y)|\nabla_0(Y)|^{\frac{m+1}{2}-z} \langle \pi^*(Y)\xi, w\rangle d^*Y.$$
Remember that there exists an element  $m_{-1}\in G$ which acts by $-1$ on $V^+$ and  $V^-$ and  trivially  on  $\go g$ (cf. Definition \ref{def-hx}). Hence  $m_{-1}$ is central in $G$ and from the definition we have  $\pi_{\de,\mu}(m_{-1})= \prod_{j=0}^k\de_j(-1)\rm{Id}_{I_{\delta,\mu}}$. Therefore  
\begin{equation}\label{eq-psi-Y}Z^-(\cF\circ\cF(\Psi), \frac{m+1}{2}-z,\xi,w)= (\prod_{j=0}^k\de_j(-1))Z^-(\Psi, \frac{m+1}{2}-z,\xi,w).\end{equation}
By analytic continuation, this equality between these two rational functions in the variables  $q^{-z}$ and  $q^z$ remains true for all  $\mu\in\bU$. From  (\ref{eq-FFPhi}), we get
$$Z^+(\cF(\Psi),z+\frac{m-1}{2},\xi,w)=\frac{(\de_0\ldots \de_k)(-1)}{d(\de,\mu,z)} Z^-(\Psi, \frac{m+1}{2}-z,\xi,w).$$
Making the change of variable  $z\to 1-z$, we get 
$$Z^+(\cF(\Psi),\frac{m+1}{2}-z,\xi,w)=\frac{(\de_0\ldots \de_k)(-1)}{d(\de,\mu,1-z)} Z^-(\Psi, z+\frac{m-1}{2},\xi,w),$$
and hence 
$$\Xi^+(\cF(\Psi), 1-z,\xi,w)=\frac{(\de_0\ldots \de_k)(-1)}{d(\de,\mu,1-z)} \frac{L^-(\pi_{\de,\mu},z)}{L^+(\pi_{\de,\mu},1-z)}\Xi^-(\Psi,z,\xi,w).$$
Therefore we set:
\begin{equation}\label{eq-defep-}\ep^-(\pi_{\de,\mu}, z, \psi)=\frac{(\de_0\ldots \de_k)(-1)}{d(\de,\mu,1-z)}\frac{L^-(\pi_{\de,\mu},z)}{L^+(\pi_{\de,\mu},1-z)}.\end{equation}
From the knowledge of  $d(\de,\mu,z)$ (Corollary  \ref{EF-globale-04}) and  (\ref{eq-rho-L-ep}), we obtain  
\begin{align*}\frac{(\de_0\ldots \de_k)(-1)}{d(\de,\mu,1-z)} &=\ga_{\psi}(q)^{-\frac{k(k+1)}{2}}\prod_{j=0}^k \rho(\de_j,\mu_j+z)^{-1}\\
{}&=\ga_{\psi}(q)^{-\frac{k(k+1)}{2}}\prod_{j=0}^k\frac{L_0(\de_j^{-1}, 1-(\mu_j+z))}{L_0(\de_j, \mu_j+z)}\epsilon_0(\de_j, \mu_j+z,\psi).
\end{align*}
This gives the second functional equation and ends the proof of assertion (1).\me

Let  $\Phi\in \cS(V^+)$. Define $\Psi(Y)=\cF(\Phi)(-Y)$,  then  $\Phi=\cF(\Psi)$. From  (\ref{eq-psi-Y}), we get  
\begin{equation}\label{eq-Xi-}\Xi^-(\cF(\Phi), 1-z,\xi,w)=(\de_0\ldots \de_k)(-1) \Xi^-(\Psi,1-z,\xi,w). \end{equation}Using the two functional equations and (\ref{eq-Xi-}) we obtain
$$\Xi^-(\cF(\Phi),1-z, \xi,w)=\ep^+(\pi_{\de,\mu},z,\psi)\Xi^+(\Phi,z,w)=\ep^+(\pi_{\de,\mu},z,\psi)\Xi^+(\cF(\Psi),z,w)$$
$$=\ep^+(\pi_{\de,\mu},z,\psi)\ep^-(\pi_{\de,\mu},1-z,\psi)\Xi^-(\Psi, 1-z,\xi,w)$$
$$=(\de_0\ldots \de_k)(-1)\ep^+(\pi_{\de,\mu},z,\psi)\ep^-(\pi_{\de,\mu},1-z,\psi)\Xi^-(\cF(\Phi),1-z,\xi,w),$$
which implies that 
\begin{equation}\label{eq-produit-ep}\ep^+(\pi_{\de,\mu},z,\psi)\ep^-(\pi_{\de,\mu},1-z,\psi)=(\de_0\ldots \de_k)(-1).\end{equation}
We know from  Proposition \ref{prop-L(z,pi)}, that there exist finite families  $(\Phi_i^\pm)_{i\in I}\in \cS(V^\pm)$  and  $(w_i^\pm)_{i\in I}\in I_{\de,\mu}$ such that   $\sum_{i\in I} \Xi^\pm(\Phi_i^\pm,z, \xi, w_i^\pm)=1$ and such that  $\Xi^\pm(\cF(\Phi_i^\pm), z, \xi, w_i^\pm)$ are polynomials in $q^{-z}$ and  $q^z$.  The two preceding functional equations imply then that  $\ep^\pm(\pi_{\de,\mu},z,\psi)$ are polynomials in $q^{-z}$ and $q^z$. Finally equation  (\ref{eq-produit-ep}) implies that there are in fact monomials in  $q^{-z}$. 
This ends the proof of the Theorem.
\end{proof}

\begin{rem}\label{rem-ep+ep-} The relation
$$\ep^+(\pi_{\de,\mu},z,\psi)\ep^-(\pi_{\de,\mu},1-z,\psi)=(\de_0\ldots \de_k)(-1)$$
 generalizes the  one for $GL(n,F) $ (see \cite{G-J} p.33 and \cite {B-H} (23.4.2) p. 142)\end{rem}

\begin{cor}\label{cor-calcul-L}\hfill. 

(1)  There exists a positive integer $d(\delta,\mu)$ and a polynomial  $Q^+_{\de_\mu}(T)=\prod_{r=1}^{d(\delta,\mu)}(1-a_r(\de_\mu) T)\in\C[T]$ such that  
$$L^+(\pi_{\de,\mu},z)=\frac{1}{Q^+_{\de_\mu}(q^{-z})}\prod_{j=0}^k L_0(\de_j^{-1}, z-\mu_j).$$
Then if we set  $Q^-_{\de_\mu}(T)=\prod_{r=1}^{d(\delta,\mu)}(1-a_r(\de_\mu)^{-1}q T)$ we have also
  $$ L^-(\pi_{\de,\mu}, z)=\frac{1}{Q^-_{\de_\mu}(q^{-z})}\prod_{j=0}^k L_0(\de_j, z+\mu_j).$$
(2) Let   $\de\in\widehat{\Or_F^*}^{k+1}$.   There exist roots of the unit  $u_r$ and rational numbers $p_r, p_{r,j}\in\Q$   ($r=1,\ldots,d(\delta,\mu)$ and  $j=0,\ldots, k$),  such that 
$$ a_r(\de_\mu)=u_rq^{-p_r-\sum_{j=0}^k p_{r,j}\mu_j}.$$
\end{cor}
\begin{proof} From the definition of Euler factor, there exist polynomials  $P^\pm_{\de_\mu}(T)\in\C[T]$ such that $P^\pm_{\de_\mu}(0)=1$ and
$\displaystyle L^\pm(\pi_{\de,\mu},z)=\frac{1}{P^\pm_{\de_\mu}(q^{-z})}$. 
Let us denote by
 $$P^+_{\de,\mu}(T)=\prod_{r=1}^N (1-a_rT),\quad{\rm and }\quad  P^-_{\de,\mu}(T)=\prod_{t=1}^{N' }(1-b_tT),\quad \;\; (a_r,b_t\in\C)$$
the decomposition into prime factors of these two polynomials.\me

As the factors $\ep_0(\de_j^{-1},z-\mu_j,\psi)$ and $\ep^+(\pi_{\de,\mu},z,\psi)$ are monomials in $q^{-z}$ (Theorem \ref{thm-analogueGJ} (2)), we obtain that there exists  $C\in\C$ and  $n\in\Z$ such that 
$$Cq^{-nz}=\frac{L^+(\pi_{\de,\mu},z)}{L^-(\pi_{\de,\mu},1-z)}\prod_{j=0}^k\frac{L_0(\de_j,1-(z-\mu_j))}{L_0(\de_j^{-1},z-\mu_j)}.$$

Here  $L_0(\chi,s)=P_0(\chi, q^{-s})^{-1}$ where $ P_0(\chi,T)\in\C[T]$ and  $P_0(\chi,T)=1$ if $\chi$ is ramified and  $P_0(\chi,T)=1-\chi(\pi)T$ if $\chi $ is non ramified (see (\ref{eq-L0})). \me

Then, we obtain easily 
\begin{equation}\label{eq-depart}CT^n P_{\de_\mu}^+(T)\prod_{j=0}^k P_0(\de_j, q^{-1-\mu_j} T^{-1})=P_{\de_\mu}^-(q^{-1}T^{-1})\prod_{j=0}^kP_0(\de_j^{-1}, q^{\mu_j}T)\end{equation}
And if all the $\de_j$'s are ramified, the preceding relation becomes
$$CT^n\prod_{r=1}^N (1-a_rT)=\prod_{t=1}^{N' }(1-b_t(qT)^{-1})=T^{-N'}\prod_{t=1}^{N'}(-b_t q^{-1})\prod_{t=1}^{N' }(1-b_t^{-1} qT).$$
This implies that  $N=N'$ and that  $\{a_1,\ldots,a_N\}=\{q b_1^{-1},\ldots, qb_N^{-1}\}$, and hence    assertion (1) is proved in this case by taking  $Q^\pm_{\de_\mu}(T)=P^\pm_{\de_\mu}(T)$.\me

Suppose now that at least one $\de_j$ is non ramified. Up to a change of the indexation of the $\de_j$'s, one can suppose that there exists  $k_0\geq 0$ such that  $\de_j$ is non ramified for  $0\leq j\leq k_0$ and  $\de_j$ is ramified for  $j>k_0$.

Equation (\ref{eq-depart}) becomes then
\begin{equation}\label{eq-nonramifie}\begin{array}{c}P_{\de_\mu}^-(q^{-1} T^{-1})\prod_{j=0}^{k_0}(1-\de_j^{-1}(\pi)q^{\mu_j}T)=CT^nP^+_{\de_\mu}(T)\prod_{j=0}^{k_0}(1-\de_j(\pi) q^{-1-\mu_j}T^{-1})\\
\\
=CT^{n-(k_0+1)}\left(\prod_{j=0}^{k_0}(-\de_j(\pi) q^{-1-\mu_j}\right)P^+_{\de_\mu}(T)\prod_{j=0}^{k_0}(1-\de_j(\pi)^{-1} q^{1+\mu_j}T).\end{array}\end{equation}
This shows that the polynomial  $\prod_{j=0}^{k_0}(1-\de_j^{-1}(\pi)q^{\mu_j}T)$ divides  $P^+_{\de_\mu}(T)$. We set   $$Q_{\de_\mu}^+(T)=\frac{P^+_{\de_\mu}(T)}{\prod_{j=0}^{k_0}(1-\de_j^{-1}(\pi)q^{\mu_j}T)}.$$
Relation  (\ref{eq-nonramifie}) can now be written as 
$$P_{\de_\mu}^-(q^{-1} T^{-1})=CT^nQ^+_{\de_\mu}(T)\prod_{j=0}^{k_0}(1-\de_j(\pi) q^{-1-\mu_j}T^{-1}).$$
Therefore   $ \prod_{j=0}^{k_0}(1-\de_j(\pi) q^{-\mu_j}T)$ divides  $P^-_{\de_\mu}(T)$ and we set 
$$Q_{\de_\mu}^-(T)=\frac{P^-_{\de_\mu}(T)}{\prod_{j=0}^{k_0}(1-\de_j(\pi) q^{-\mu_j}T)}.$$
Then 
$$CT^nQ^+_{\de_\mu}(T)=Q^-_{\de_\mu}(q^{-1}T^{-1}).$$
As before, we obtain then the asserted form for the functions $L^\pm(\pi_{\de,\mu},z)$.\me

Let us now show assertion  (2). 

From  Theorem \ref{thm-Sato-kp}, we know that for  $\om\in\widehat{\Or_F^*}^{k+1}$,  there exists a polynomial  \begin{equation}\label{eq-expressionR}R^+(\om,s)=\prod_{r=1}^{d_0}(1-q^{-N_r-\sum_{j=0}^k N_{r,j}s_j}),\quad N_r\in\N, N_{r,j}\in \N\end{equation}
such that, for all $\Phi\in\cS(V^+)$, the function  $R^+(\om,s)\cK^+(\Phi,\om,s)$ is a polynomial is the variables  $q^{s_j}$ and  $q^{-s_j}$. 

Remember that  $\om(\de)=(\om_0,\ldots ,\om_k)=(\de_0^{-1},\de_0\de_1^{-1},\ldots, \de_{k-1}\de_k^{-1})$ and that  $s(\mu)=(s_0,\ldots, s_k)$ is defined by the relations   $s_0+\ldots+s_j=\dfrac{d}{4}(k-2j)-\mu_j$ (see   Definition \ref{def-omdelta-smu}).   Theorem  \ref{thm-Zeta-mero} implies that for  $\Phi\in\cS(V^+)$   $w\in I_{\de,\mu}$ and  $\xi\in (I_{\de,\mu}^*)^H$, the function  $R^+(\om(\de),s(\mu)+z-m)Z^+(\Phi,z, \xi,w)$ is a polynomial in  $q^{-z}$ and  $q^z$. 

Set $\cP(T)=\prod_{r=1}^{d_0}(1-c_r(\de_\mu)T^{N_{r,0}}) $ with  $c_r(\de_\mu)=q^{-N_r+\frac{N_{r,0}}{2}(m+1)-\sum_{j=0}^k N_{r,j}s_j}$. Then  $\cP(q^{-z})=R^+(\om(\de),s(\mu)+z-\frac{1}{2}(m+1))$.

As $z+\frac{1}{2}(m-1)-m=z-\frac{1}{2}(m+1)$, it is easily seen that the function  $L^+(z)=\dfrac{1}{\cP(q^{-z})}$ is an Euler factor in  $\cE^+$ (see Definition \ref{def-facteur-eulerien} (2)).\me

Then from  Proposition \ref{prop-L(z,pi)}, 2), the polynomial  $P_{\de,\mu}^+$ divides  $\cP$ in  $\C[T]$. Hence, up to permutations of the families  $(N_r, N_{r,0},\ldots,N_{r,k})$, there exists  $r_0^+\leq d_0$, and for all $r\in\{1,\ldots ,r^+_0\}$, a finite family  $U^+_r$ of  $N_{r,0}$-th root  of unity  such that 
$$P_{\de,\mu}^+(T)=\prod_{r=1}^{r_0^+}\prod_{u\in U^+_r}(1-u \; b_r(\de_\mu)T),\;\;{\rm with }\;\; b_r(\de_\mu)=c_r(\de_\mu)^{1/N_{r,0}}=q^{-\frac{N_r}{N_{r,0}}+\frac{m+1}{2}-\sum_{j=0}^k \frac{N_{r,j}}{N_{r,0}}s_j}.$$
As $s_0=\dfrac{kd}{4}-\mu_0$ and  $s_j=\mu_{j-1}-\mu_j-\dfrac{d}{2}$, we obtain assertion (2).\end{proof}

 We will now describe how the $\ep$-factors  $\varepsilon^\pm(\pi_{\de,\mu}, z,\psi)$  depend on the additive character $\psi$. \me
 
Let $a\in F^*$  Remember that $m_a$ is the element in $G$ which acts by multiplication by  $a$ on $V^+$, by  multiplication by $a^{-1}$ on $V^-$ and trivially on  $\go g$ (Definition  \ref{def-hx}). Therefore $m_a$ is central in  $G$. Let  $\mu=\mu_0\la_0+\ldots \mu_k\la_k$ and define $\varpi(a)=(\de_0\ldots\de_k)(a)|a|^{\mu_0+\ldots +\mu_k}$. Then  $\pi_{\de,\mu}(m_a)=\varpi(a) {\rm Id}_{I_{\delta,\mu}}$.

 \begin{prop}\label{Prop-dependance-psi}. Let  $a\in F^*$. Denote by $\psi^a$ the character of $F$ given by $\psi^a(t)=\psi(at)$. Let   $(d^aX, d^aY)=(c_a dX, d_a dY)$,  with  $c_a, d_a>0$,  a pair of measures on $(V^+,V^-)$  which are dual for the Fourier transform $\cF^a$ defined by $\psi^a$. 
Then    $c_{a}d_{a}=|a|^{\dim V^+}$ and

$$\varepsilon^+(\pi_{\de,\mu},z, \psi^a)=\frac{\varpi(a)^{-1}|a|^{(k+1)(\frac{m-1}{2}+z)}}{ c_{a}}\ep^+(\pi_{\de,\mu}, z,\psi),$$
$$\ep^-(\pi_{\de,\mu}, z, \psi^a)=c_{a} \varpi(a) |a|^{-(k+1)(\frac{m+1}{2}-z)} \ep^-(\pi_{\de,\mu}, z, \psi).$$

If  $d_{a}=c_{a}= |a|^{\frac{\dim V^+}{2}}=|a|^{m\frac{(k+1)}{2}}$, the preceding formulas become   
$$\varepsilon^+(\pi_{\de,\mu}, z,\psi^a)=\varpi(a)^{-1} {|a|^{(k+1)(z-\frac{1}{2} )}} \varepsilon^+(\pi_{\de,\mu}, z,\psi),$$
$$\varepsilon^-(\pi_{\de,\mu}, z,\psi^a)= \varpi(a) |a|^{(k+1)(z-\frac{1}{2})} \ep^-(\pi_{\de,\mu}, z, \psi).$$

\end{prop}

\begin{proof} \hfill

From the definitions we have:
\begin{equation}\label{eq-Fourier-a}\cF^a(\Phi)(Y) =c_{a}\int_{V^+}\Phi(X)\psi(a\;b(X,Y) ) dX= c_{a}\cF(\Phi) (aY).
\end{equation}

It is easily seen  that  the measures  $c_{a}dX$ and  $d_{a}dY$ are  dual for $\cF^a$ if and only if $$c_{a}d_{a}=|a|^{\dim V^+}.$$
From  (\ref{eq-defep+}) and  (\ref{eq-defep-}), the functions  $\ep^\pm(\pi_{\de,\mu}, z,\psi)$ are uniquely determined by  $d(\de,\mu,z)$ which appears in the functional equation satisfied by  $Z^\pm(\Phi,z,\xi,w)$. The definition of these functions depends on the choice of the dual measures $(dX,dY)$ on  $V^+\times V^-$. Let   $Z^{a,\pm}(\Phi,z,\xi,w)$ be the new  zeta functions relative to  $(c_a dX, d_a dY)$. Then   $Z^{1,\pm}(\Phi,z,\xi,w)= Z^{\pm}(\Phi,z,\xi,w)$.\me
 
 Let  $\Phi\in\cS(V^+)$ and  $w\in I_{\de,\mu}$. Let also  $z\in \C$ such that ${\rm Re}(\mu_k-z)>0$. Then  $Z^-(\cF(\Phi), \frac{1}{2}(m+1)-z,\xi, w)$ is given by an integral  (Theorem \ref{thm-Zeta-mero} (2)(a)). Using  (\ref{eq-Fourier-a}) and the $G$-invariance of  $d^*Y$ under the central element $m_a$, we obtain: 
$$  Z^{a,-}(\cF^a(\Phi), \frac{m+1}{2}-z, \xi,w)= Z^{a,-}( c_{a}\cF(\Phi)(a\,  {\bf \cdot})), \frac{m+1}{2}-z, \xi,w)$$
$$= |a|^{\dim V^+}\int_{V^-}\cF(\Phi)(m_a^{-1}.Y)|\nabla_{0}(Y)|^{\frac{m+1}{2}-z} \langle\pi_{\de,\mu}^*(Y)\xi,w\rangle d^*Y$$
$$= |a|^{\dim V^+}\int_{V^-}\cF(\Phi)(Y)\frac{1}{|a|^{(k+1)\left(\frac{m+1}{2}-z\right)}}|\nabla_{0}(Y)|^{\frac{m+1}{2}-z} \langle\pi_{\de,\mu}^*(m_a.Y)\xi,w\rangle d^*Y$$
$$= |a|^{\dim V^+}\int_{V^-}\cF(\Phi)(Y)\frac{1}{|a|^{(k+1)\left(\frac{m+1}{2}-z\right)}}|\nabla_{0}(Y)|^{\frac{m+1}{2}-z} \langle\pi_{\de,\mu}^*(Y)\xi,\pi_{\de,\mu}(m_a^{-1})w\rangle d^*Y$$
$$= \varpi(a)^{-1}|a|^{\dim V^+}\int_{V^-}\cF(\Phi)(Y)\frac{1}{|a|^{(k+1)\left(\frac{m+1}{2}-z\right)}}|\nabla_{0}(Y)|^{\frac{m+1}{2}-z} \langle\pi_{\de,\mu}^*(Y)\xi,  w\rangle d^*Y$$
  \vskip 5pt

As $\displaystyle m=\frac{\dim V^+}{k+1}$ we have 

$$Z^{a,-}(\cF^a(\Phi), \frac{m+1}{2}-z, \xi,w)=\varpi(a)^{-1}|a|^{(k+1)(\frac{m-1}{2}+z)}  Z^-(\cF(\Phi),\frac{m+1}{2}-z, \xi,w).$$

Therefore  $ Z^{a,-}(\cF^a(\Phi), \frac{m+1}{2}-z, \xi,w)=d^a(\de,\mu,z) Z^{a,+}(\Phi, z+\frac{m-1}{2}, \xi,w)$ where
$$d^a(\de,\mu,z)=d(\de,\mu,z) \frac{\varpi(a)^{-1}|a|^{(k+1)(\frac{m-1}{2}+z)}}{c_{a}},$$
which implies that $$\ep^+(\pi_{\de,\mu}, z,\psi^a)=  \frac{\varpi(a)^{-1}|a|^{(k+1)(\frac{m-1}{2}+z)}}{c_{a}}\ep^+(\pi_{\de,\mu}, z,\psi).$$

Similarly, taking again  $z\in \C$ such that  ${\rm Re}(\mu_k-z)>0$, one has  
$$\cF^a(\Psi)(X)=d_a\int_{V^+} \Psi(Y) \psi(a\;b(X,Y))dY=d_a\cF(\Psi)(aX)$$
and 
$$Z^{a,+}(\cF^a(\Psi), \frac{m+1}{2}-z, \xi,w)= |a|^{\dim V^+} \int_{V^+} \cF(\Psi)(a X) |\De_0(X)|^{\frac{m+1}{2}-z}\langle\pi_{\de,\mu}^*(X)\xi,  w\rangle d^*X$$
$$= |a|^{\dim V^+}  \int_{V^+} \cF(\Psi)(X) |a|^{-(k+1)(\frac{m+1}{2}-z)}|\De_0(X)|^{\frac{m+1}{2}-z}\langle\pi_{\de,\mu}^*(m_a^{-1}.X)\xi,  w\rangle d^*X$$
$$= |a|^{\dim V^+}  \varpi(a) |a|^{-(k+1)(\frac{m+1}{2}-z)} Z^+(\cF(\Psi), \frac{m+1}{2}-z, \xi,w)$$
If we divide both sides of the preceding identity by $L(\pi_{\delta,\mu},1-z)$ we obtain 
\begin{align*}\Xi^{a,+}(\cF^a(\Psi),1-z,\xi,w)&= |a|^{\dim V^+}  \varpi(a) |a|^{-(k+1)(\frac{m+1}{2}-z)}\Xi^+(\cF(\Psi),1-z,\xi,w)\\
{}&= |a|^{\dim V^+}  \varpi(a) |a|^{-(k+1)(\frac{m+1}{2}-z)}\ep^-(\pi_{\delta,\mu},z,\psi)\Xi^-(\Psi,z,\xi,w)\\
{}&=\frac{1}{d_{a}}|a|^{\dim V^+}  \varpi(a) |a|^{-(k+1)(\frac{m+1}{2}-z)}\ep^-(\pi_{\delta,\mu},z,\psi)\Xi^{a,-}(\Psi,z,\xi,w).
 \end{align*}

And hence
 $$\ep^-(\pi_{\de,\mu}, z, \psi^a)=c_a \varpi(a) |a|^{-(k+1)(\frac{m+1}{2}-z)} \ep^-(\pi_{\de,\mu}, z, \psi).$$
 \end{proof}

      
   \vskip 5pt 
    
 \printindex

 \end{document}

%% file: 2zetaracinescopie.tex
\def\AI{
\hbox{\unitlength=0.5pt
\begin{picture}(250,30)(-50,0)
\put(10,40){\circle{10}}
\put(15,40){\line (1,0){30}}
\put(50,40){\circle{10}}
\put(60,40){\circle*{1}}
\put(65,40){\circle*{1}}
\put(70,40){\circle*{1}}
\put(75,40){\circle*{1}}
\put(80,40){\circle*{1}}
\put(85,40){\circle*{1}}
\put(90,40){\circle*{1}}
\put(95,40){\circle*{1}}
\put(100,40){\circle*{1}}
\put(110,40){\circle{10}}
\put(115,40){\line (1,-1){21}}
\put(140,15){\circle{10}}
\put(140,15){\circle{16}}
\put(115,-10){\line (1,1){21.5}}
\put(110,-10){\circle{10}}
\put(110,8){\vector(0,1){23}}
\put(110,22){\vector(0,-1){23}}
\put(100,-10){\circle*{1}}
\put(95,-10){\circle*{1}}
\put(90,-10){\circle*{1}}
\put(85,-10){\circle*{1}}
\put(80,-10){\circle*{1}}
\put(75,-10){\circle*{1}}
\put(70,-10){\circle*{1}}
\put(65,-10){\circle*{1}}
\put(60,-10){\circle*{1}}
\put(50,-10){\circle{10}}
\put(50,8){\vector(0,1){23}}
\put(50,22){\vector(0,-1){23}}
\put(15,-10){\line (1,0){30}}
\put(10,-10){\circle{10}}
\put(10,8){\vector(0,1){23}}
\put(10,22){\vector(0,-1){23}}
\end{picture}
}}

\def\AII{
\hbox{\unitlength=0.4pt
\begin{picture}(310,00)
\put(10,10){${\linethickness{1mm}\line(1,0){20}}$}
 \put(30,10){\line (1,0){17}}
 \put(52,10){\circle{10}}
 \put(57,10){\line (1,0){17}}
 \put(74,10){${\linethickness{1mm}\line(1,0){20}}$}
 \put(99,10){\circle*{1}}
\put(104,10){\circle*{1}}
\put(109,10){\circle*{1}}
\put(114,10){\circle*{1}}
\put(119,10){\circle*{1}}
\put(122,10){${\linethickness{1mm}\line(1,0){20}}$}
 \put(142,10){\line (1,0){15}}
 \put(164,10){\circle{10}}
 \put(164,10){\circle{16}}
 \put(171,10){\line (1,0){15}}
 \put(186,10){${\linethickness{1mm}\line(1,0){20}}$}
 \put(211,10){\circle*{1}}
\put(216,10){\circle*{1}}
\put(221,10){\circle*{1}}
\put(226,10){\circle*{1}}
\put(231,10){\circle*{1}}
\put(233,10){${\linethickness{1mm}\line(1,0){20}}$}
 \put(253,10){\line (1,0){17}}
 \put(275,10){\circle{10}}
  \put(280,10){\line (1,0){17}}
 \put(297,10){${\linethickness{1mm}\line(1,0){20}}$}
 \end{picture}
}
}

\def\AIII{
\hbox{\unitlength=0.5pt
\begin{picture}(250,30)
\put(10,10){\circle*{10}}
\put(15,10){\line (1,0){30}}
\put(50,10){\circle{10}}
\put(60,10){\circle*{1}}
\put(65,10){\circle*{1}}
\put(70,10){\circle*{1}}
\put(75,10){\circle*{1}}
\put(80,10){\circle*{1}}
\put(90,10){\circle*{10}}
\put(95,10){\line (1,0){30}}
\put(130,10){\circle{10}}
\put(130,10){\circle{16}}
\put(135,10){\line (1,0){30}}
\put(170,10){\circle*{10}}
\put(180,10){\circle*{1}}
\put(185,10){\circle*{1}}
\put(190,10){\circle*{1}}
\put(195,10){\circle*{1}}
\put(200,10){\circle*{1}}
\put(210,10){\circle{10}}
\put(215,10){\line (1,0){30}}
\put(250,10){\circle*{10}}
\end{picture}
}}

\def\BI{
\hbox{\unitlength=0.5pt
\begin{picture}(250,30)(-10,0)
\put(10,10){\circle{10}}
\put(10,10){\circle{16}}
\put(15,10){\line (1,0){30}}
\put(50,10){\circle{10}}
\put(55,10){\line (1,0){30}}
\put(90,10){\circle{10}}
\put(95,10){\line (1,0){30}}
\put(130,10){\circle{10}}
\put(135,10){\circle*{1}}
\put(140,10){\circle*{1}}
\put(145,10){\circle*{1}}
\put(150,10){\circle*{1}}
\put(155,10){\circle*{1}}
\put(160,10){\circle*{1}}
\put(165,10){\circle*{1}}
\put(170,10){\circle{10}}
\put(174,12){\line (1,0){41}}
\put(174,8){\line (1,0){41}}
\put(190,5.5){$>$}
\put(220,10){\circle*{10}}
\end{picture}
}}

\def\BII{
\hbox{\unitlength=0.5pt
\begin{picture}(250,30)(-10,0)
\put(10,10){\circle{10}}
\put(10,10){\circle{16}}
\put(15,10){\line (1,0){30}}
\put(50,10){\circle{10}}
\put(60,10){\circle*{1}}
\put(65,10){\circle*{1}}
\put(70,10){\circle*{1}}
\put(75,10){\circle*{1}}
\put(80,10){\circle*{1}}
\put(85,10){\circle*{1}}
\put(95,10){\circle{10}}
\put(100,10){\line (1,0){30}}
\put(135,10){\circle{10}}
\put(140,10){\circle*{1}}
\put(145,10){\circle*{1}}
\put(150,10){\circle*{1}}
\put(155,10){\circle*{1}}
\put(160,10){\circle*{1}}
\put(165,10){\circle*{1}}
\put(170,10){\circle*{1}}
\put(175,10){\circle{10}}
\put(179,12){\line (1,0){41}}
\put(179,8){\line (1,0){41}}
\put(195,5.5){$>$}
\put(225,10){\circle{10}}
\end{picture}
}}

\def\BIII{
\hbox{\unitlength=0.5pt
\begin{picture}(250,30)(-10,0)
\put(130,10){\circle{10}}
\put(130,10){\circle{16}}
\put(134,12){\line (1,0){41}}
\put(134,8){\line (1,0){41}}
\put(150,5.5){$>$}
\put(180,10){\circle*{10}}
\end{picture}
}}

\def\CI{
\hbox{\unitlength=0.5pt
\begin{picture}(250,30)
\put(10,10){\circle{10}}
\put(15,10){\line (1,0){30}}
\put(50,10){\circle{10}}
\put(55,10){\line (1,0){30}}
\put(90,10){\circle{10}}
\put(95,10){\circle*{1}}
\put(100,10){\circle*{1}}
\put(105,10){\circle*{1}}
\put(110,10){\circle*{1}}
\put(115,10){\circle*{1}}
\put(120,10){\circle*{1}}
\put(125,10){\circle*{1}}
\put(130,10){\circle*{1}}
\put(135,10){\circle*{1}}
\put(140,10){\circle{10}}
\put(145,10){\line (1,0){30}}
\put(180,10){\circle{10}}
\put(184,12){\line (1,0){41}}
\put(184,8){\line(1,0){41}}
\put(200,5.5){$<$}
\put(230,10){\circle{10}}
\put(230,10){\circle{16}}
\end{picture}
}}

\def\CIII{
\hbox{\unitlength=0.5pt
\begin{picture}(250,30)
\put(10,10){\circle*{10}}
\put(15,10){\line (1,0){30}}
\put(50,10){\circle{10}}
\put(55,10){\line (1,0){30}}
\put(90,10){\circle*{10}}
\put(95,10){\circle*{1}}
\put(100,10){\circle*{1}}
\put(105,10){\circle*{1}}
\put(110,10){\circle*{1}}
\put(115,10){\circle*{1}}
\put(120,10){\circle*{1}}
\put(125,10){\circle*{1}}
\put(130,10){\circle*{1}}
\put(135,10){\circle*{1}}
\put(140,10){\circle{10}}
\put(145,10){\line (1,0){30}}
\put(180,10){\circle*{10}}
\put(184,12){\line (1,0){41}}
\put(184,8){\line (1,0){41}}
\put(200,5.5){$<$}
\put(230,10){\circle{10}}
\put(230,10){\circle{16}}
\end{picture}
}}

\def\DI{
\hbox{\unitlength=0.5pt
\begin{picture}(240,40)(0,-10)
\put(10,10){\circle{10}}
\put(10,10){\circle{16}}
\put(15,10){\line (1,0){30}}
\put(50,10){\circle{10}}
\put(55,10){\line (1,0){30}}
\put(90,10){\circle*{10}}
\put(95,10){\line (1,0){30}}
\put(130,10){\circle*{10}}
\put(140,10){\circle*{1}}
\put(145,10){\circle*{1}}
\put(150,10){\circle*{1}}
\put(155,10){\circle*{1}}
\put(160,10){\circle*{1}}
\put(170,10){\circle*{10}}
\put(175,10){\line (1,0){30}}
\put(210,10){\circle*{10}}
\put(215,14){\line (1,1){20}}
\put(240,36){\circle*{10}}
\put(215,6){\line(1,-1){20}}
\put(240,-16){\circle*{10}}
\end{picture}
}}

\def\DIIA{
\hbox{\unitlength=0.5pt
\begin{picture}(240,40)(0,-10)
\put(10,10){\circle{10}}
\put(10,10){\circle{16}}
\put(15,10){\line (1,0){30}}
\put(50,10){\circle{10}}
\put(55,10){\line (1,0){30}}
\put(90,10){\circle{10}}
\put(95,10){\line (1,0){30}}
\put(130,10){\circle{10}}
\put(140,10){\circle*{1}}
\put(145,10){\circle*{1}}
\put(150,10){\circle*{1}}
\put(155,10){\circle*{1}}
\put(160,10){\circle*{1}}
\put(170,10){\circle{10}}
\put(175,10){\line (1,0){30}}
\put(210,10){\circle{10}}
\put(215,14){\line (1,1){20}}
\put(240,36){\circle{10}}
\put(215,6){\line (1,-1){20}}
\put(240,-16){\circle{10}}
\end{picture}
}}

\def\DIIB{
\hbox{\unitlength=0.5pt
\begin{picture}(240,40)(0,-10)
\put(10,10){\circle{10}}
\put(10,10){\circle{16}}
\put(15,10){\line (1,0){30}}
\put(50,10){\circle{10}}
\put(55,10){\line (1,0){30}}
\put(90,10){\circle{10}}
\put(95,10){\line (1,0){30}}
\put(130,10){\circle{10}}
\put(140,10){\circle*{1}}
\put(145,10){\circle*{1}}
\put(150,10){\circle*{1}}
\put(155,10){\circle*{1}}
\put(160,10){\circle*{1}}
\put(170,10){\circle{10}}
\put(175,10){\line (1,0){30}}
\put(210,10){\circle{10}}
\put(215,14){\line (1,1){20}}
\put(240,36){\circle{10}}
\put(215,6){\line (1,-1){20}}
\put(240,-16){\circle{10}}
\put(240,2){\vector(0,1){23}}
\put(240,16){\vector(0,-1){23}}
\end{picture}
}}

\def\DIIC{
\hbox{\unitlength=0.5pt
\begin{picture}(240,40)(0,-10)
\put(10,10){\circle{10}}
\put(10,10){\circle{16}}
\put(15,10){\line  (1,0){30}}
\put(50,10){\circle{10}}
\put(60,10){\circle*{1}}
\put(65,10){\circle*{1}}
\put(70,10){\circle*{1}}
\put(75,10){\circle*{1}}
\put(80,10){\circle*{1}}
\put(85,10){\circle*{1}}
\put(95,10){\circle{10}}
\put(100,10){\line  (1,0){30}}
\put(135,10){\circle{10}}
\put(140,10){\circle*{1}}
\put(145,10){\circle*{1}}
\put(150,10){\circle*{1}}
\put(155,10){\circle*{1}}
\put(160,10){\circle*{1}}
\put(165,10){\circle*{1}}
\put(170,10){\circle*{1}}
\put(175,10){\circle{10}}
\put(180,10){\line  (1,0){30}}
\put(215,10){\circle{10}}
\put(215,14){\line  (1,1){20}}
\put(240,36){\circle*{10}}
\put(215,6){\line  (1,-1){20}}
\put(240,-16){\circle*{10}}
\end{picture}
}}

\def\DIII{
\hbox{\unitlength=0.5pt
\begin{picture}(240,40)(0,-10)
\put(10,10){\circle{10}}
\put(10,10){\circle{16}}
\put(15,10){\line (1,0){30}}
\put(50,10){\circle*{10}}
\put(55,10){\line (1,0){30}}
\put(90,10){\circle*{10}}
\put(95,10){\line (1,0){30}}
\put(130,10){\circle*{10}}
\put(140,10){\circle*{1}}
\put(145,10){\circle*{1}}
\put(150,10){\circle*{1}}
\put(155,10){\circle*{1}}
\put(160,10){\circle*{1}}
\put(170,10){\circle*{10}}
\put(175,10){\line (1,0){30}}
\put(210,10){\circle*{10}}
\put(215,14){\line (1,1){20}}
\put(240,36){\circle*{10}}
\put(215,6){\line (1,-1){20}}
\put(240,-16){\circle*{10}}
\end{picture}
}}

\def\DIde{
\hbox{\unitlength=0.5pt
\begin{picture}(240,35)(0,-15)
\put(10,0){\circle*{10}}
\put(15,0){\line (1,0){30}}
\put(50,0){\circle{10}}
\put(55,0){\line (1,0){30}}
\put(90,0){\circle*{10}}
\put(95,0){\line (1,0){30}}
\put(130,0){\circle{10}}
\put(140,0){\circle*{1}}
\put(145,0){\circle*{1}}
\put(150,0){\circle*{1}}
\put(155,0){\circle*{1}}
\put(160,0){\circle*{1}}
\put(170,0){\circle*{10}}
\put(175,0){\line (1,0){30}}
\put(210,0){\circle{10}}
\put(215,4){\line (1,1){20}}
\put(240,26){\circle{10}}
\put(240,26){\circle{16}}
\put(215,-4){\line (1,-1){20}}
\put(240,-26){\circle*{10}}
\end{picture}
}}

\def\DIIde{
\hbox{\unitlength=0.5pt
\begin{picture}(240,35)(0,-15)
\put(10,0){\circle{10}}
\put(15,0){\line (1,0){30}}
\put(50,0){\circle{10}}
\put(55,0){\line (1,0){30}}
\put(90,0){\circle{10}}
\put(95,0){\line (1,0){30}}
\put(130,0){\circle{10}}
\put(140,0){\circle*{1}}
\put(145,0){\circle*{1}}
\put(150,0){\circle*{1}}
\put(155,0){\circle*{1}}
\put(160,0){\circle*{1}}
\put(170,0){\circle{10}}
\put(175,0){\line (1,0){30}}
\put(210,0){\circle{10}}
\put(215,4){\line (1,1){20}}
\put(240,26){\circle{10}}
\put(240,26){\circle{16}}
\put(215,-4){\line (1,-1){20}}
\put(240,-26){\circle{10}}
\end{picture}
}}

\def\EI{
\hbox{\unitlength=0.5pt
\begin{picture}(240,40)(-10,-45)
\put(10,0){\circle{10}}
\put(15,0){\line  (1,0){30}}
\put(50,0){\circle*{10}}
\put(55,0){\line  (1,0){30}}
\put(90,0){\circle*{10}}
\put(90,-5){\line  (0,-1){30}}
\put(90,-40){\circle*{10}}
\put(95,0){\line  (1,0){30}}
\put(130,0){\circle*{10}}
\put(135,0){\line  (1,0){30}}
\put(170,0){\circle{10}}
\put(175,0){\line (1,0){30}}
\put(210,0){\circle{10}}
\put(210,0){\circle{16}}
\end{picture}
}}

\def\EII{
\hbox{\unitlength=0.5pt
\begin{picture}(240,40)(-10,-35)
\put(10,10){\circle{10}}
\put(15,10){\line (1,0){30}}
\put(50,10){\circle{10}}
\put(55,10){\line (1,0){30}}
\put(90,10){\circle{10}}
\put(90,5){\line (0,-1){30}}
\put(90,-30){\circle{10}}
\put(95,10){\line (1,0){30}}
\put(130,10){\circle{10}}
\put(135,10){\line (1,0){30}}
\put(170,10){\circle{10}}
\put(175,10){\line (1,0){30}}
\put(210,10){\circle{10}}
\put(210,10){\circle{16}}
\end{picture}
}}

%% file: 2zetatablescor1.tex

\begin{landscape}
\vskip 20pt
\centerline{{\bf Table 1}\hskip 20pt Simple Regular Graded Lie Algebras over a $p$-adic field} 

\vskip 20pt
{\scriptsize
\label{tables1}
\begin{tabular}{|c|c|c|c|c|c|c|c|c|c|c|c|c|}
\hline
&&&&&&&&&&&&\\
&$\widetilde{\go g}$
&${\go g}'$
 &$V^+$
&$\widetilde {\mathcal R}$
&$\widetilde{\Sigma}$
  &{Satake-Tits diagram}
  &$ rank (=k+1)$ 
& $ \ell$
& $d$
&$e$
&Type 
&1-type\\
\hline
\hline 
(1)& $\go {sl}(2(k+1),D)$
& $\begin{matrix}\go {sl}(k+1,D)\\ \oplus \\ \go {sl}(k+1,D) \end{matrix}$ 
&  $ {\mathrm M}(k+1,D) $
& $\underset{n=(k+1)\delta}{ A_{2n-1}}$
& $A_{2k+1}$ 
& $\begin{matrix}\AII \\ \text{where} \hbox{\unitlength=0.4pt
\begin{picture}(30,30)
\put(10,10){${\linethickness{1mm}\line(1,0){20}}$}\end{picture}} = \hbox{\unitlength=0.4pt
\begin{picture}(200,00)
\put(10,10){\circle*{10}} \put(15,10){\line(1,0){20}}  \put(40,10){\circle*{10}}\put(50,10){\circle*{1}}\put(55,10){\circle*{1}}\put(60,10){\circle*{1}}\put(65,10){\circle*{1}}\put(70,10){\circle*{1}}\put(75,10){\circle*{10}}\put(80,10){\line(1,0){20}}\put(105,10){\circle*{10}} \put(30,-15){$\delta-1$} \put(120,10){ ($\delta\in \N^*)$} \end{picture}}   \end{matrix}$ 
&$\underset{( k+1=1)}{k+1\geq 2}$
&$\delta^2$
&$\underset{(--)}{2\delta^2}$
&$\underset{(--)}{0}$
&  I
&$(A,\delta)$\\
\hline
&&&&&&&&&&&&\\
(2)& $\go {su}(2n,E,H_{2n})$
& $\go {sl}(n,E)$ 
&  $ \mathrm {Herm}_{\sigma}(n,E) $
& $\underset{n\geqslant 1}{ A_{2n-1}}$
& $C_{n}$ 
&\AI 
&$n$
&$1$
&$2$
&$ 2$
&  II
&$(A,1)$\\ 
\hline
\hline
&&&&&&&&&&&&\\
(3)&$\!\!\go {o}(q_{(n+1,n)})$\!\!
&\!\! $\go {o}(q_{(n,n-1)})$\!\! 
&  $  F^{2n-1} $
& $\begin{matrix} \underset{n\geq 3}{ B_n}\\  
 \end{matrix}$
&\!  ${B_{n}}$ 
&\BII 
&$2$
&$1$
&$2n\!-\!3$
&$ 1$
& II
&$(A,1)$\\ 
\hline
&&&&&&&&&&&&\\
(4)&$\go {o}(q_{(n+2,n-1)})$
& $\go {o}(q_{(n+1,n-2)} )$ 
&  $  F^{2n-1} $
& $\underset{n\geqslant 3}{ B_n}$
&  $B_{n-1}$ 
&\BI 
&$2$
&$1$
&$2n\!-\!3$
&$  3$
& II
&$(A,1)$\\ 
\hline
&&&&&&&&&&&&\\
(5)&$\go {o}(q_{(4,1)})$
& $\go {o}(3)$ 
&  $  F^{3} $
& $ { B_2}$
&  $B_1=A_1$ 
&\BIII 
&$1$
&$3$
&$--$
&$ --$
&III
&$B$\\ 
\hline
\hline
&&&&&&&&&&&&\\
(6)&$\go {sp}(2n, F)$
& $\go {sl}(n,F)$ 
&  $ \mathrm {Sym}(n, F) $
& $\underset{n\geqslant 2}{ C_{n}}$
&  $C_{n}$ 
&\CI 
&$n$
&$1$
&$1$
&$ 1$
& II
&$(A,1)$\\ 
\hline
&&&&&&&&&&&&\\
(7)&$ \go{u}(2n, \H, H_{2n})$
& $\go {sl}(n,\bb H)$ 
& \!\!  
 $  \mathrm{SkewHerm (n,\H)}$
& ${ C_{2n}}$
&  $C_{n}$ 
&\CIII 
&$n$
&$3$
&$ 4$
&$ 4$
&  III
&$B$\\ 
\hline
 
\end{tabular}
\vskip 5pt 
\hskip 450pt {(continued next page)}

\pagebreak
\vskip 20pt
\centerline{{\bf Table 1}\,(continued)\hskip 20pt Simple Regular Graded Lie Algebras over a $p$-adic field} 

\vskip 20pt

\label{tables2}

\begin{tabular}{|c|c|c|c|c|c|c|c|c|c|c|c|c|}
\hline
&&&&&&&&&&&&\\
&$\widetilde{\go g}$
&${\go g}'$
 &$V^+$
&$\widetilde {\mathcal R}$
&$\widetilde{\Sigma}$
  &{Satake-Tits diagram}
  &$  rank (=k+1)$ 
& $ \ell$
& $d$
&$e$
&Type
&1-type \\
\hline
\hline 
&&&&&&&&&&&&\\
(8)&$\go {o}(q_{(m,m)})$
& $\go {o}(q_{(m\!-\!1,m-1)})$ 
&  $  F^{2m\!-\!2} $
& $\underset{m\geqslant 4}{ D_m}$
&  $ D_m$ 
&\DIIA 
&$2$
&$1$
&$2m\!-\!4$
&$0$
&  I
&$(A,1)$\\ 
\hline
&&&&&&&&&&&&\\
(9)&$\go {o}(q_{(m\!+1,m\!-\!1)})$
& $\go {o}(q_{(m,m\!-\!2)}$ 
&  $  F^{2m\!-\!2} $
& $\underset{m\geqslant 4}{ D_m}$
&  $ B_{m-1}$ 
&\DIIB 
&$2$
&$1$
&$2m\!-\!4$
&$2$
& II
&$(A,1)$\\ 
\hline
&&&&&&&&&&&&\\
(10)&$\go {o}(q_{(m+2,m-2)})$
& \!\!\!$\go {o}(q_{(m+ 1,m-3)})$\!\!\! 
&  $  F^{2m-2} $
& $\underset{m\geqslant 4}{ D_m}$
&  $ { B_{m-2}}$ 
&\DIIC 
&$2$
&$1$
&$2m\!-\!4$
&$4$
& I
&$(A,1)$\\ 
\hline
\hline
&&&&&&&&&&&&\\
(11)&\!\!\!$\go {o}(q_{(2n,2n)})$\!\!\!
& $\go {sl}(2n, F)$ 
& \!\! $\mathrm{Skew}(2n,F)\!\! $
& $\underset{n\geqslant 3}{ D_{2n}}$
&  $D_{2n}$ 
&\DIIde 
&$n$
&$1$
&$4$
&$0$
& I
&$(A,1)$\\ 
\hline
&&&&&&&&&&&&\\
(12)&$\go{u}(2n, \H, K_{2n})$
& $\go {sl}(n,\bb H)$ 
& $  \mathrm{Herm}(n,\bb H)$

& $\underset{n\geqslant 3}{ D_{2n}}$
&  $C_{n}$ 
&\DIde 
&$n$
&$1$
&$ 4$
&$ \bf 4$
&  I
&$(A,1)$\\ 
\hline
\hline
&&&&&&&&&&&&\\
(13)&$\text{split  } E_{7}$
& 
$\text{split  } E_{6}$

&  $ \mathrm{Herm}(3,\bb O_s) $
& $E_7$
&  $E_7$ 
&\EII 
&$3$
&$1$
&$8$
&$0$
&  I
&(A,1)\\ 
\hline

\end{tabular}

}
\end{landscape}